\documentclass[9pt,english]{article}
\usepackage{geometry}
\usepackage{setspace}
\usepackage{musicography}
\usepackage{empheq}
\usepackage[T1]{fontenc}
\usepackage[utf8]{inputenc}
\usepackage{babel}
\usepackage{varioref}
\usepackage{amsmath,nccmath}
\numberwithin{equation}{section}
\usepackage{amssymb}
\usepackage{amsthm}
\usepackage{mathrsfs} 
\usepackage{upgreek}
\usepackage{graphicx}
\usepackage{floatrow}
\usepackage[labelfont=bf]{caption}
\usepackage{subcaption}
\usepackage{appendix}
\usepackage{slashed}
\usepackage{svg}
\usepackage{bm}
\usepackage [autostyle, english = american]{csquotes}
\MakeOuterQuote{"}
\makeatletter
\usepackage{hyperref}
\usepackage[capitalize,nameinlink]{cleveref}
\usepackage{enumitem}

\usepackage[makeroom]{cancel}

\usepackage{authblk}
\usepackage{xcolor}
\usepackage{graphicx}

\graphicspath{{figures}}

\hypersetup{
	colorlinks,
	linkcolor={blue!60!black},
	citecolor={green!60!black},
	urlcolor={blue!80!black}
}
\usepackage{mathtools}
\usepackage{mathrsfs}
\usepackage{comment}

\usepackage{cases}
\usepackage{xargs}   

\usepackage[colorinlistoftodos,prependcaption,textsize=tiny]{todonotes}
\newcommandx{\typo}[2][1=]{\todo[linecolor=red,backgroundcolor=red!25,bordercolor=red,#1]{#2}}
\newcommandx{\change}[2][1=]{\todo[linecolor=blue,backgroundcolor=blue!25,bordercolor=blue,#1]{#2}}
\newcommandx{\question}[2][1=]{\todo[linecolor=green,backgroundcolor=blue!25,bordercolor=green,#1]{#2}}
\newcommandx{\answer}[1]{\todo[linecolor=pink,backgroundcolor=pink!25,bordercolor=pink]{#1}}
\newcommandx{\unsure}[2][1=]{\todo[linecolor=green,backgroundcolor=green!25,bordercolor=green,#1]{#2}}
\newcommandx{\improve}[2][1=]{\todo[linecolor=violet,backgroundcolor=violet!25,bordercolor=violet,#1]{#2}}
\newcommandx{\thiswillnotshow}[2][1=]{\todo[disable,#1]{#2}}
\allowdisplaybreaks

\newtheorem{thm}{Theorem}[section]
\newtheorem{conj}{Conjecture}[section]
\newtheorem{lemma}{Lemma}[section]
\newtheorem{defi}{Definition}[section]

\newtheorem{cor}{Corollary}[section]
\newtheorem{ass}{Assumption}[section]
\theoremstyle{remark}
\newtheorem{rem}{Remark}[section]
\newtheorem{obs}[rem]{Observation}
\theoremstyle{plain}

\newtheorem*{claim*}{Claim}
\newtheorem{prop}{Proposition}[section]

\newenvironment{nalign}{
	\begin{equation}
		\begin{aligned}
		}{
		\end{aligned}
	\end{equation}
	\ignorespacesafterend
}
\theoremstyle{remark}

\newcommand{\dd}{\mathop{}\!\mathrm{d}}
\renewcommand{\O}{\mathcal{O}}
\newcommand{\e}{\mathrm{e}}
\newcommand{\E}{\mathcal{E}}
\newcommand{\T}{\mathbb{T}}
\newcommand{\tT}{\tilde{\mathbb{T}}}
\renewcommand{\b}{\mathrm{b}}
\newcommand{\scrim}{\mathcal{I}^-}
\newcommand{\scrip}{\mathcal{I}^+}
\newcommand{\deformation}[1]{{}^{#1}\pi}
\newcommand{\tr}{\mathrm{tr}}
\newcommand{\Omegalin}{\stackrel{\scalebox{.6}{\mbox{\tiny (1)}}}{\Omega}}
\newcommand{\blin}{\stackrel{\scalebox{.6}{\mbox{\tiny (1)}}}{b}}
\newcommand{\glin}{\stackrel{\scalebox{.6}{\mbox{\tiny (1)}}}{\slashed{g}}} 
\usepackage{accents}    

\newcommand{\gs}{\accentset{\scalebox{.6}{\mbox{\tiny (1)}}}{\slashed{g}}}
\newcommand{\gsh}{\accentset{\scalebox{.6}{\mbox{\tiny (1)}}}{\hat{\slashed{g}}}}
\newcommand{\trg}{{\tr{\gs}}}

\newcommand{\xh}{\accentset{\scalebox{.6}{\mbox{\tiny (1)}}}{{\hat{\chi}}}}
\newcommand{\xhb}{\accentset{\scalebox{.6}{\mbox{\tiny (1)}}}{\underline{\hat{\chi}}}}

\newcommand{\al}{\accentset{\scalebox{.6}{\mbox{\tiny (1)}}}{{\alpha}}}

\newcommand{\alb}{\accentset{\scalebox{.6}{\mbox{\tiny (1)}}}{\underline{\alpha}}}

\newcommand{\Ps}{\accentset{\scalebox{.6}{\mbox{\tiny (1)}}}{{\Psi}}}
\newcommand{\Psb}{\accentset{\scalebox{.6}{\mbox{\tiny (1)}}}{\underline{\Psi}}}
\newcommand{\pv}{\partial_v}
\newcommand{\pu}{\partial_u}
\renewcommand{\sl}{\mathring{\slashed{\nabla}}}
\newcommand{\Dl}{\mathring{\slashed{\Delta}}}

\newcommand{\red}[1]{{\color{red}{#1}}}

\newcommand{\violet}[1]{{\color{violet}{#1}}}
\newcommand{\jpns}[1]{\langle #1 \rangle}

\title{
{Scattering, Polyhomogeneity and Asymptotics \\for Quasilinear Wave Equations From Past to Future Null Infinity}} 

\author[1]{Istvan Kadar\thanks{istvan.kadar@math.ethz.ch}}
\author[2,3]{Lionor Kehrberger\thanks{kehrberger@mis.mpg.de}}

\affil[1]{Princeton University, Department of Mathematics, Fine Hall, Washington Road, Princeton, NJ 08544, United States of America}
\affil[2]{Max Planck Institute for Mathematics in the Sciences,  Inselstraße 22, 04103 Leipzig, Germany}
\affil[3]{University of Leipzig, Institute for Theoretical Physics, Br\"uderstraße 16,  04103 Leipzig, Germany}
\date{\today} 

\makeatother
\usepackage[backend=biber,style=alphabetic,sorting=nty,doi=false,isbn=false,url=false,minalphanames=3]{biblatex}
\AtEveryBibitem{\clearfield{month}}
\AtEveryBibitem{\clearfield{day}}
\newcommand{\scri}{\mathcal{I}}
\newcommand{\p}[1]{\partial_{#1}}
\newcommand{\Hb}{H_{\mathrm{b}}}
\newcommand{\Hbt}{\tilde{H}_{\mathrm{b}}}
\newcommand{\A}[1]{\mathcal{A}_{\mathrm{#1}}}
\newcommand{\phg}{\mathrm{phg}}
\newcommand{\R}{\mathbb{R}}
\newcommand{\C}{\mathcal{C}}
\newcommand{\Cbar}{\underline{\mathcal{C}}}
\newcommand{\N}{\mathbb{N}}
\newcommand{\D}{\mathcal{D}}
\newcommand{\Deps}{\D^{\delta}}
\newcommand{\Dopen}{\underline{\D}}
\newcommand{\Dbold}{\boldsymbol{\D}}
\newcommand{\Ve}{\mathcal{V}_{\mathrm{e}}}
\newcommand{\Vb}{\mathcal{V}_{\mathrm{b}}}
\newcommand{\Vbt}{\tilde{\mathcal{V}}_{\mathrm{b}}}
\newcommand{\Lbar}{\underline{L}}

\newcommand{\incone}{\underline{\mathcal{C}}_{v_0}}
\newcommand{\outconeFar}{\C_{u_\infty}}
\newcommand{\outcone}[1]{\C_{u_{#1}}}
\newcommand{\norm}[1]{\left\lVert #1\right\rVert}
\newcommand{\abs}[1]{\left\lvert #1\right\rvert}
\newcommand{\mindex}[1]{\overline{(#1,0)}}
\newcommand{\cupdex}{\,\overline{\cup}\,}
\DeclareMathOperator{\Diff}{Diff}
\DeclareMathOperator{\divergence}{div}
\DeclareMathOperator{\supp}{supp}
\crefname{lemma}{Lemma}{Lemmata}
\crefname{thm}{Theorem}{Theorems}
\crefname{prop}{Proposition}{Propositions}
\crefname{conj}{Conjecture}{Conjectures}
\crefname{ass}{Assumption}{Assumptions}
\crefname{cor}{Corollary}{Corollary}
\crefname{rem}{Remark}{Remark}
\crefname{obs}{Observation}{Observations}
\crefname{defi}{Definition}{Definitions}
\crefname{equation}{}{}
\crefname{enumi}{}{}
\usepackage{soul}

\crefname{subfigure}{Figure}{Figure}
\Crefname{subfigure}{Figure}{Figure}

\begin{document}
	\pagenumbering{roman}
	
	\maketitle 
		\begin{abstract}
  We present a general construction of semiglobal scattering solutions to quasilinear wave equations in a neighbourhood of spacelike infinity including past and future null infinity, where the scattering data are posed on an ingoing null cone and along past null infinity.
  More precisely, we prove weighted, optimal-in-decay energy estimates and propagation of polyhomogeneity statements from past to future null infinity for these solutions,  we provide an algorithmic procedure how to compute the precise coefficients in the arising polyhomogeneous expansions, and we apply this procedure to various  examples.  
 As a corollary, our results directly imply the summability in the spherical harmonic number $\ell$ of the estimates proved for fixed spherical harmonic modes in the papers \cite{kehrberger_case_2022-1,kehrberger_case_2024} from the series  "The Case Against Smooth Null Infinity".
  
  Our (physical space) methods are based on weighted energy estimates near spacelike infinity similar to those of~\cite{hintz_microlocal_2023-1}, commutations with (modified) scaling vector fields to remove leading order terms in the relevant expansions, time inversions, as well as the Minkowskian conservation laws:
    \begin{equation*}
        \pu(r^{-2\ell}\pv(r^2\pv)^{\ell}(r\phi_{\ell}))=0,
    \end{equation*}
    which are satisfied if $\Box_\eta\phi=0$.

    Our scattering constructions apply to systems of equations as well and  go beyond the usual class of finite energy solutions.
    We use this to also derive a scattering theory and prove propagation of polyhomogeneity for the Einstein vacuum equations in a harmonic gauge. 
    In the process, we also need to introduce a novel ansatz accounting for the stronger-than-Schwarzschildean divergence of the light cones, which, in particular, extends existing exterior stability of Minkowski statements in harmonic gauge to allow for slowly decaying data as considered in \cite{bieri_extension_2010}.




	\end{abstract}
	\newpage
	\begingroup
	\hypersetup{linkcolor=black}
	\tableofcontents{}
	\endgroup
	\newpage
	\pagenumbering{arabic}
 
\part{Introduction and setup}
This work is divided into four parts, each having its own short preamble.

In this first part, consisting of \cref{sec:intro,sec:intro:example,sec:notation,sec:ODE_lemmas}, we first give motivation, historical background, and a detailed account of all the main results of this work, along with a rough description of some elements of their proofs, in \cref{sec:intro}. In particular, we hope that the busy reader will already gain a good understanding of our work by just reading \cref{sec:intro}.
We present a basic toy model showcasing some of our ideas in a simple context in \cref{sec:intro:example}.
In \cref{sec:notation}, we then set up our notation. We conclude with some elementary ODE lemmata in \cref{sec:ODE_lemmas}.
\section{Introduction}\label{sec:intro}
	The aim of this work is to  gain a systematic understanding of the asymptotic behaviour of solutions to  linear and nonlinear wave equations in a neighbourhood of the spacelike infinity of asymptotically flat spacetimes.
We immediately give a prosaic version of our main result: Motivated by the constructions in the series of papers "The Case Against Smooth Null Infinity" \cite{kehrberger_case_2022-1,kehrberger_case_2024-1,kehrberger_case_2024}, we pose initial data along an ingoing null cone~$\incone$, emanating from a sphere at past null infinity $\scrim$ and truncated at some $u_0\ll0$, and on the part of $\scrim$ to the future of the latter, denoted $\scrim_{v\geq v_0}$. See  \cref{fig:intro:main} below.
\begin{thm}[Main result, prosaic version]\label{thm:intro:main}
Solutions to a large class of (nonlinear) perturbations of the linear wave equation on Minkowski space arising from polyhomogeneous scattering data along $\incone\cup\scrim_{v\geq v_0}$ admit polyhomogeneous expansions towards $\scrim$, spacelike infinity~$I^0$ and future null infinity~$\scrip$; moreover, the coefficients in these latter expansions can be computed by a simple algorithmic procedure.

If, instead, the data only have a finite expansion plus error term, then the same is the case for the solution.
\end{thm}
 \begin{figure}[htbp]
 \centering
\includegraphics[width=0.2\textwidth]{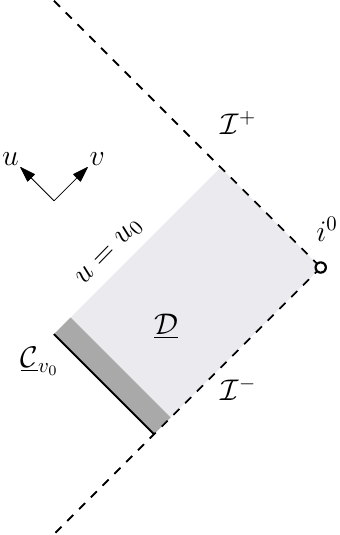}
\caption{Our scattering data are given on a null cone $\incone$ and on $\scrim_{v\geq v_0}$. Depicted is the Penrose diagram of $\Dopen$, the domain of dependence of $\incone\cup\scrim_{v\geq v_0}.$ The darker shaded region corresponds to an initial data slab and is only referred to in \cref{sec:EVE}.}
\label{fig:intro:main}
\end{figure}
There are three distinct aspects to the theorem: Firstly, the \textit{existence} of scattering solutions. 
Secondly, the \textit{propagation of polyhomogeneity} (plus potentially the propagation of error terms), where we recall that polyhomogeneity denotes asymptotic expansions into terms of the form $\rho_{\bullet}^{p}\log ^n \rho_{\bullet}$ with $(p,n)\in\R\times\N$ and with $\rho_{\bullet}$, for $\bullet\in\{-,0,+\}$, characterising the idealised boundaries $\scrim$, $I^0$ or~$\scrip$. 
Thirdly, the \textit{computation} of the actual coefficients in the \textit{asymptotic expansions}. See already \cref{sec:intro:sketch} for a more detailed breakdown.

 The "large class of perturbations" allows for the inclusion of various potential terms, semi-linearities (such as power nonlinearities or null forms), as well as quasilinear terms so long as these can, in a suitable sense, be treated perturbatively, which in turn depends on the initial data decay. 
 See already the discussion around \cref{eq:intro:wavegeneral} for more details. 
 In particular, we can treat the linearised Einstein vacuum equations in double null gauge around Schwarzschild introduced in \cite{dafermos_linear_2019} as well as the nonlinear Einstein vacuum equations in harmonic gauge near Minkowski as studied in \cite{Lindblad2010}. We note that for both of these latter systems, there is a wave part to which \cref{thm:intro:main} applies directly, and a transport part that is dealt with separately.

Furthermore, we also allow for the direct treatment of \textit{linear long-range potentials} such as the first-order terms that appear in the Teukolsky equations \cite{dafermos_linear_2019} or such as $c/r^2$-potentials \textit{for any} $c\in\mathbb R$ (this is in contrast to the restricted range of $c$ in the study of \textit{late-time dynamics}, see \cite{gajic_late-time_2022,hintz_linear_2023}).

 \subsection{Relevance to the series "The Case Against Smooth Null Infinity": The summing of the \texorpdfstring{$\ell$-modes}{spherical modes}}\label{sec:intro:relev}
With more general and self-contained context for this work being given  in \cref{sec:intro:motivation,sec:intro:mathcontext} below, let us first explain the original motivation of this work pertaining to the series of papers "The Case Against Smooth Null Infinity" \cite{kehrberger_case_2022,kehrberger_case_2022-1,kehrberger_case_2024-1,kehrberger_case_2024}, and the relevance of \cref{thm:intro:main} to that series: 
Consider the linear scalar wave equation on a Schwarzschild spacetime, 
\begin{equation}\label{eq:intro:waveSchwarzschild}
    \Box_g\phi=0, \qquad g=-\left(1-\frac{2M}{r}\right)\dd t^2+\left(1-\frac{2M}{r}\right)^{-1}\dd r^2+r^2(\dd \theta^2+\sin^2\theta\dd\varphi^2),\qquad M\in\mathbb{R}.
\end{equation}
Let $u,v$ denote the usual double null coordinates (defined in \cref{sec:Sch:wave}),  let, for some $u_0\ll 0$ and $v_0>0$,  $\incone:=\{u\leq u_0, v=v_0\}$, and let $\Dopen$ denote the domain of dependence of $\incone\cup \scrim_{v\geq v_0}$, where $\scrim_{v\geq v_0}$ denotes the set of limit points $\{u=-\infty,v\geq v_0\}$ and we henceforth suppress the ${}_{v\geq v_0}$ subscript. 
We say that a solution~$\phi$ has no incoming radiation from $\scrim$ for $v\geq v_0$ if the limit of $\norm{\pv\psi(u,\cdot)}_{L^2(\{v\geq v_0\}\times  S^2)}$ vanishes along $\scrim$, where we have introduced the notation
\begin{equation}
\psi=r\phi.
\end{equation}
Finally, we define $r_0(u):=r(u,v_0)$, i.e.~$r_0=r|_{\incone}$. Note that $\tfrac{r_0}{r} \sim 1$  as $\scrim$ and $I^0$ are approached ($u\to-\infty$ at fixed $v$ or fixed $u/v$, respectively), and $\tfrac{r_0}{r}\sim \tfrac{1}{r}$ as $\scrip$ is approached ($v\to\infty$ at fixed $u$).

In \cite{kehrberger_case_2022-1}, the following was shown (see \cref{thm:Schw:wave} for a precise version):
\begin{thm}[from \cite{kehrberger_case_2022-1}]\label{thm:intro:III}
    Let $\phi$ be the solution to \cref{eq:intro:waveSchwarzschild} arising from no incoming radiation and suitably regular initial data $r\phi|_{\incone}=C r^{-p}+\O(r^{-p-\epsilon})$ for some $p\in\mathbb N_{\geq0}$, $0\neq C(\omega)\in C^{\infty}(S^2)$ and $\epsilon>0$.
    Then the fixed angular mode projections  $\phi_\ell$ satisfy the following expansion throughout $\Dopen$:
   \begin{nalign}\label{eq:intro:Schwarzschild:firstexpansion}
			r\phi_{\ell}=\left(\sum_{n=0}^{p+1}f_n^{(\ell,p)}(r_0)\frac{r_0^n}{r^n}\right)+M c_{\log}^{(\ell,p)}\frac{\log (r/r_0)}{r^{p+1}}+{}^{(\ell)}\O\left(\frac{r_0^2}{r^{p+2}}+\frac{1}{r^{p+1+}}\right)
	\end{nalign}
  for some functions $f_n^{(\ell,p)}(r_0)\lesssim r_0^{-p}$ and some generically nonvanishing constants $c_{\log}^{(\ell,p)}$ (both of which are computed explicitly).
\end{thm}
In words, the projections of $\psi$ to fixed angular modes are polyhomogeneous plus error term. Now, the control of this error term achievable by the methods of \cite{kehrberger_case_2022-1} has an $\ell$-dependence that excludes any hope for summing this estimate.
\textit{However}, by \cref{thm:intro:main}, we know that the resummed solution is also polyhomogeneous (plus error term), and so we can, a posteriori, conclude (see \cref{thm:Schw:wave:summed} for the precise version):
\begin{thm}\label{thm:intro:IIIsummed}
    Under the assumptions of \cref{thm:intro:III}, \emph{if additionally it is assumed that $\epsilon>1$}, then \cref{eq:intro:Schwarzschild:firstexpansion} is summable and we have
      \begin{nalign}\label{eq:intro:Schwarzschild:firstexpansionsummed}
			r\phi=\left(\sum_{n=0}^{p+1}f_n^{(p)}(r_0, \theta, \varphi)\frac{r_0^n}{r^n}\right)+M c_{\log}^{(p)}(\theta,\varphi)\frac{\log (r/r_0)}{r^{p+1}}+\O\left(\frac{r_0^{2-}}{r^{p+2-}}+\frac{1}{r^{p+1+}}\right),
	\end{nalign}
 where the projections of $f_n^{(p)}(r_0, \theta, \varphi)$ and $c_{\log}^{(p)}(\theta,\varphi)$ are given by $f_n^{(\ell,p)}(r_0)$ and $c_{\log}^{(\ell,p)}$, respectively. 
\end{thm}

\cref{thm:intro:III,thm:intro:IIIsummed} above are just simple example statements. More generally, \textit{all  statements concerning data posed on $\incone\cup\scrim$ that were derived for fixed angular mode solutions in \cite{kehrberger_case_2022-1,kehrberger_case_2024} are, in fact, valid, for the resummed solution}, provided that the initial data assumptions are suitably strengthened. For instance, for \cref{thm:intro:IIIsummed}, it more generally suffices to specify an expansion up to order $\epsilon>1$ in order to deduce the result.
(See already \cref{rem:intro:error} for further commentary on the requirement $\epsilon>1$.)

In particular, the present work combined with \cite{kehrberger_case_2024} gives us a complete understanding of the asymptotics of linearised gravity around Schwarzschild around spacelike infinity, see already~\cref{thm:Schw:lingravity}.

\begin{rem}
 Let us emphasise two particular reasons why the summing of the $\ell$-modes in this context is difficult: Firstly, we are not trying to find some leading-order terms in certain asymptotic expansions, but we are instead trying to detect higher-order terms in such expansions, making the approach of simply subtracting the leading-order behaviour intractable. Secondly, within the physical motivation for the problem we are studying, the most relevant scenario is, just as in \cref{thm:intro:III}, that where all $\ell$-modes feature the same behaviour towards $\scrip$---it is therefore insufficient to simply prove the asymptotics for the first few $\ell$-modes and then prove an error estimate on remaining higher $\ell$-modes.
Both of these points should be compared to the complementary problem of proving asymptotics \textit{at late times}, cf.~\cite{angelopoulos_late-time_2018,gajic_late-time_2022,hintz_sharp_2022,luk_late_2024} and references therein.
\end{rem}

  \paragraph{Structure of the remainder of this introduction:}
The remainder of this section is structured as follows:
In \cref{sec:intro:motivation}, we motivate and put into historical perspective the problems studied in the series "The Case Against Smooth Null Infinity" and, in particular, explain where the difficulty of "summing the $\ell$-modes" comes from. We also give a high-level sketch of our approach towards proving \cref{thm:intro:main}.\footnote{We note that, were it just for the treatment of the linear problem discussed in \cref{sec:intro:relev}, this work would be much shorter; a significant amount of work goes into generalising all our results to nonlinear problems and proving \cref{thm:intro:main}. We therefore included in \cref{sec:intro:example} a small toy model discussion that already features the main ideas in just going from \cref{thm:intro:III} to \cref{thm:intro:IIIsummed}.}
Hoping that the reader will then have a better idea of the elements involved in the proof, we then, in \cref{sec:intro:mathcontext} give a thorough introduction to these elements along with an embedding into the relevant mathematical literature.
We will then state our main results in \cref{sec:intro:mainresults} and showcase a few applications of these results in \cref{sec:intro:applications}. We give a brief overview of further, miscellaneous results of independent interest proved in this work in \cref{sec:intro:misc}.
We conclude with an overview of the structure of the remainder of the work in \cref{sec:intro:structure}.

\subsection{Motivation, historical context and a first overview of the ideas} \label{sec:intro:motivation}
 Although the content of \cref{thm:intro:III,thm:intro:IIIsummed} may seem fairly technical, understanding and obtaining control over the asymptotic expansions of solutions to wave equations on asymptotically flat spacetimes actually has a high degree of physical relevance. 
 Specifically, in the context of gravitational radiation, understanding the first term in the expansion of the gravitational field (itself a solution to some modified wave equations, namely the Einstein equations) towards~$\scrip$ that is not regular in the conformal variable $1/r$  measures the roughness/smoothness of the associated conformal boundary, which then directly leads to a conserved quantity along $\scrip$ that is, in principle, measurable at late times, see \cite{gajic_relation_2022}.
 Moreover, making precise the notion of "asymptotically flat spacetimes" in general relativity \textit{requires} an understanding of the decay behaviour of gravitational radiation in a neighbourhood of $\scrim$, $I^0$ and $\scrip$.

The problems only briefly touched upon above possess, of course, a much deeper history going back in particular to Penrose's proposal of a smooth null infinity \cite{penrose_asymptotic_1963,Penrose65}, and we refer to the introduction of \cite{kehrberger_mathematical_2023} or \cite{kehrberger_case_2024-1} for a detailed general overview. In particular, it is argued in those references that the scattering problem, with data in the infinite past is, from a physical point of view, much more natural than the Cauchy problem, as one can more easily give physical meaning to assumptions on scattering data than on Cauchy data.

In contrast, in the present introduction, we shall shine light on the history of discussions about smooth null infinity from a purely mathematical perspective specifically relevant to the present work:
\subsubsection{The peeling conjecture in even-dimensional Minkowski spacetime}\label{sec:intro:peeling}
Penrose's idea of a smooth null infinity was quite directly motivated by the so-called peeling property~\cite{SeriesVIII}.
Loosely speaking, the content of the peeling property is that solutions to linearised gravity around Minkowski, if they have no incoming radiation from $\scrim$, can be expanded in integer powers of $1/r$ towards $\scrip$, the latter property also being referred to as \textit{conformal smoothness}.
Since linearised gravity around Minkowski is essentially described by wave equations, we restrict our discussion of the peeling property to the wave equation on Minkowski. 

To put into context how surprising the claim of the peeling property is, let us first consider more generally the scale-invariant wave equation on $\mathbb R^{n+1}$
\begin{equation}\label{eq:intro:scalewave}
    \Box_{\mathbb{R}^{n+1}}\phi+\frac{c}{r^2}\phi=0 \iff \pu\pv\psi=\frac{\Dl_{S^{n-1}}+c-\frac{(n-1)(n-3)}{4}}{r^2}\psi\,\,\,\text{for}\,\,\, \psi=r^{\frac{n-1}{2}}\phi,
\end{equation}
where $c\in\mathbb R$, and where $\Dl_{S^{n-1}}$ denotes the unit-sphere Laplacian with eigenvalues $-\ell(\ell+n-2)$, $\ell\in\mathbb N$.
Consider now initial data on $\incone$ satisfying $\psi|_{\incone}\sim r^{-p}$ for some $p\in\mathbb{R}$, and further pose trivial data at $\scrim$, corresponding to no incoming radiation (cf.~\cref{fig:intro:main}). Then the solution $\psi$ will still decay like $r^{-p}$ towards $I^0$, and it will have an expansion into powers of $1/r$ towards $\scrip$ \textit{up to order $r^{-p}$}. Then, at order $p$, it will for generic $c$ either have a nonvanishing $r^{-p}$-term if $p\notin\mathbb{Z}$, or a nonvanishing $r^{-p}\log r$-term if $p\in\mathbb Z$, thus no longer being regular in $1/r$. (See \cref{prop:even} for a precise statement.) 

However, peeling claims that such conformally irregular terms do not appear in odd spatial dimensions with $c=0$. We conjecture the following:
\begin{conj}\label{conj:intro}[The peeling property. See \cref{conj:even:peeling} for a precise version.]
    If $n$ is odd, then solutions to $\Box_{\mathbb{R}^{n+1}}\phi=0$ arising from trivial data on $\scrim$ and conormal\footnote{This refers to smoothness w.r.t.\ $u\pu$ and angular derivatives.} data on $\incone$ admit $1/r$-expansions to all orders towards~$\scrip$.
\end{conj}
Note, in particular, that, according to the conjecture, even if $p$ as above is negative, $\psi$ will not go like $r^{-p}$ towards~$\scrip$ but like $r^0$.
More generally, the conformal regularity towards $\scrip$ is conjecturally independent of the data along $\incone$! 

The condition of no incoming radiation can also be replaced with compactly supported radiation. Indeed, for trivial data along $\incone$ and smooth compactly supported data along $\scrim$ (i.e.~if the limit  $\pv\psi^{\scrim}$ is compactly supported), the following holds: If $n$ is odd, the corresponding solutions to $\Box_{\mathbb{R}^{n+1}}\phi=0$ are conformally smooth towards $\scrip$ (\cref{cor:app:improvedthmgeneral2}), whereas, if $n$ is even, then the rescaled $\psi$ generically blows up logarithmically towards $\scrip$ (\cref{cor:even}). If, instead, the limit of $\psi$ towards $\scrim$ is compactly supported, then a $\log$-term will appear at order $r^{-1}\log r$ towards~$\scrip$.

 Historically, peeling has never been thought of as in the conjectured form above; instead, it has been understood as a simple fact for projections $\phi_\ell$ to fixed spherical harmonics:
Indeed, in that case, if $n$ is odd, then the exact conservation laws (manifesting the strong Huygens' principle)
\begin{equation}\label{eq:intro:conslaw}
\Box_{\mathbb{R}^{n+1}}\phi=0\qquad\implies \qquad\pu(r^{-2\ell-n+1}(r^2\pv)^{\ell+\frac{n-1}{2}}\psi_{\ell})=0
\end{equation}
immediately imply the conjecture, since the vector field $r^2\pv$ precisely measures smoothness in the variable $1/r$, and the no incoming radiation condition combined with \cref{eq:intro:conslaw} ensures that $\pv(r^2\pv)^{\ell+\frac{n-3}{2}}\psi_{\ell}\equiv0$ everywhere, regardless of the data along $\incone$.\footnote{An alternative proof relies on the following representation formula, here stated in $3+1$ dimensions:
\begin{equation}\label{eq:intro:explicit_sol}
\psi_{\ell}=r^{\ell+1}\pu^{\ell}\left(\frac{f_{\mathrm{out}}(u)}{r^{\ell+1}}\right)+r^{\ell+1}\pv^{\ell}\left(\frac{f_{\mathrm{in}}(v)}{r^{\ell+1}}\right).
\end{equation}}

But what happens in the case where the data are supported on infinitely many spherical harmonics?
On the one hand, using the conservation laws \cref{eq:intro:conslaw}, we can easily construct counter-examples where initial data with, say, \textit{finite} angular regularity lead to solutions that are \textit{not} conformally smooth towards $\scrip$, simply because the relevant coefficients in the expansion can be designed to grow in $\ell$.
At the same time, the conservation laws, by themselves, are not suited for proving statements about the conformal smoothness of a solution even for data with infinite regularity, since they always require an $\ell$-dependent number of commutations, loosely speaking generating error terms of order~$\ell^\ell$.
For an approach to proving the conjecture in the class of solutions arising from conformally smooth data (where smoothness along $\incone$ is measured with respect to $r^2\pu$ and angular derivatives), see \cite{taujanskas_controlled_2023}. 
In the present work, we prove that conformal smoothness also holds in class of \textit{conormal} data that also admit a \textit{polyhomogeneous} expansion (\cref{lemma:app:minkowski_index_improvement}). 
Whether the conjecture holds for conormal data without expansions is still completely open,  cf.~\cref{conj:even:peeling} in \cref{sec:sharp} for further discussion.

\begin{rem}[Initial data assumptions: Conormality vs.~Conformal Smoothness]\label{rem:intro:conovsconf}
    In the context of our semi-global scattering problem, the data posed along~$\incone$ are thought to contain knowledge about some primordial history of the spacetime around past timelike infinity.
     The generalisation of results to data that are not conformally regular is necessary, since we expect that initial data along $\incone$ induced by $N$ infalling masses from the infinite past are not conformally smooth, cf.~\cite{WalkerWill79, Damour86,kehrberger_case_2024-1, kehrberger_case_2024}.
     For instance, in the context of $N$ infalling masses from the infinite past following approximately hyperbolic or parabolic orbits, the natural assumption to make is $\psi^{\incone}=C_{\mathrm{hyp}}+\tilde{C}_{\mathrm{hyp}}\frac{\log r_0}{r_0}+\dots$ or {$\psi^{\incone}={C}_{\mathrm{par}}r^{-2/3}+\dots$}, respectively, where $C_{\mathrm{hyp}}$ and $\tilde{C}_{\mathrm{hyp}}$ are supported on all angular modes, and where ${C}_{\mathrm{par}}$ is only supported on the lowest modes (which in the gravitational context correspond to $\ell=2$).\footnote{As an aside, the initial data assumed in the monumental \cite{christodoulou_global_1993} are suitable for describing parabolic, but not hyperbolic orbits. On the other hand, we note that the inclusion of non-integer rates of decay as required for parabolic orbits is not covered by the recent work on late-time asymptotics \cite{luk_late_2024}, though this inclusion would be of merely technical nature.}
     
     Notice that these decay rates are leading-order decay rates. While neither polyhomogeneity (expansions to arbitrarily high orders) nor conormality are explicitly predicted for the data along $\incone$ in the literature, we view it as a natural assumption based on its robust appearance for problems on asymptotically  flat spacetimes, cf.~\cite{hintz_stability_2020,hintz_microlocal_2023-1,hintz_linear_2023}. See also  \cite{kadar_construction_2024} as well as Open Problems~1 and 2 in \cite{kehrberger_mathematical_2023}.

\end{rem}

\subsubsection{The failure of peeling in the Schwarzschild spacetime}
Now, \textit{already at the level of fixed $\ell$-modes}, as proved by \cref{thm:intro:III}, the peeling property fails for slight perturbations of the geometry such as those induced by the Schwarzschild spacetime. 
Notice that the conformal irregularity produced by the Schwarzschild geometry appears one order later compared to the generic behaviour of \cref{eq:intro:scalewave}, namely at order~$r^{-p-1}$.
Indeed, from a suitable point of view, we can understand \cref{thm:intro:III} perturbatively around the Minkowski spacetime; in particular, the general framework of using the conservation laws \cref{eq:intro:conslaw} to compute asymptotic expansions can still be applied in Schwarzschild in some modified form.
In other words, restricting to fixed $\ell$-modes, the method for proving conformal smoothness in Minkowski can also be used to prove and compute the conformal irregularity in Schwarzschild or, in fact, more general perturbations.
At the same time, just as in Minkowski, the method does not allow for summing the fixed angular mode estimates: 
This is precisely where the present work comes in.

 \subsubsection{A high-level sketch of the proof of \texorpdfstring{\cref{thm:intro:main}}{the main theorem}}\label{sec:intro:sketch}
 We now list the main ingredients of the proof of \cref{thm:intro:main} and explain how to resolve the summing of the $\ell$-modes question raised above.
Solely for notational simplicity, let us work in $3+1$ dimensions. We will schematically write 
\begin{equation}\label{eq:intro:wave:P}
    \Box_\eta\phi=\mathcal{P}[\phi],\qquad \psi=r\phi
\end{equation}
where $\Box_\eta$ denotes the Minkowskian wave operator, and $\mathcal{P}[\phi]$ stands for various (possibly quasilinear) short-range perturbations that will be specified in detail later (see \cref{eq:intro:wavegeneral}).

\textbf{\textit{Step 1):} Scattering theory}\hypertarget{step1}{}

First, we require a semiglobal\footnote{We use the word semiglobal to distinguish from the global problem, which also includes the timelike infinities.} scattering theory for \cref{eq:intro:wave:P} with scattering data posed on $\incone$ and $\scrim$ as in \cref{fig:intro:main}, i.e.~well-posedness for solutions arising from such data.
This consists of proving weighted energy estimates (similar to those of~\cite{hintz_microlocal_2023-1}) for the inhomogeneous linear problem, closing nonlinear estimates, and, subsequently, performing standard limiting arguments. 
We note that, in the case of  \cref{eq:intro:waveSchwarzschild},  a standard $T$-energy estimate suffices to derive a (global) scattering theory  (see, for instance, \cite{nicolas_conformal_2016,dafermos_scattering_2018}). See \cref{intro:sec:scattering_theory} for further discussion.

\textbf{\textit{Step 2):} Persistence of polyhomogeneity and estimates on the error term}\hypertarget{step2}{}
 
 Having obtained a solution to the scattering problem, we now pose slightly more precise assumptions on our initial data, namely that the initial data $\psi^{\incone}$ are given by a polyhomogeneous expansion up to some error term lying in a weighted $L^2$ space. Solely to alleviate notation, we shall for now assume that there is no incoming radiation, but the ideas generalise.
 We write
\begin{equation}
\psi^{\incone}=\psi^{\incone}_{\mathrm{phg}}+\psi^{\incone}_{\Delta}, \text{ where } \psi^{\incone}_{\mathrm{phg}}\in \A{phg}^{\mathcal{E}}(\incone) \text{ and } \psi^{\incone}_{\Delta}\in \Hb^{p+\epsilon}({\incone}).
\end{equation} 
 The notation $\psi^{\incone}_{\mathrm{phg}}\in \A{phg}^{\mathcal{E}}(\incone) $ means that $\psi^{\incone}_{\mathrm{phg}}$ admits an expansion in terms of $r^{-q}\log^n r$ along $\incone$, where the set of $(q,n)$ appearing in this expansion is given by $\E\subset \R\times\N$.
On the other hand, $\Hb^{p+\epsilon}(\incone)$ denotes an $L^2$-based Sobolev space (where regularity is measured with respect to $\sl, u\pu$) that captures the decay $\psi^{\incone}_{\Delta}\lesssim r^{-p-\epsilon}$ in an $L^2$-sense.
Note that this corresponds to the assumptions in \cref{thm:intro:III} if $(p,0)=\min(\E)$. We now prove the following two statements:

\textbf{\textit{Step 2a):} Persistence of polyhomogeneity}\hypertarget{step2a}{}
Firstly, we prove that polyhomogeneous initial data lead to polyhomogeneous solutions; i.e.~if the data have an expansion along $\incone$, then the solution has an expansion towards past null {infinity}, spacelike {infinity}, and future null infinity. This step complements the results by \cite{hintz_stability_2020}, where polyhomogeneity was propagated from $I^0$ to $\scrip$.

\textbf{\textit{Step 2b):} Estimates on the error term}\hypertarget{step2b}{}
Secondly, we prove that if the initial data are just given by $\psi^{\incone}_{\Delta}$, then the solution admits an expansion towards either of the three boundaries up until some error term whose order is determined by $p+\epsilon$.

For nonlinear $\mathcal{P}[\phi]$, the polyhomogeneous solution $\psi_{\phg}$ acquired in \hyperlink{step2a}{Step 2a)} will enter as an inhomogeneity in \hyperlink{step2b}{Step~2b)}. See \cref{sec:intro:phg} for further discussion.

\textbf{\textit{Step 3):} Computation of the coefficients in the expansion}\hypertarget{step3}{} 

Knowing that the entire solution admits an expansion up until an error term at some fixed order, where all of this is understood in an $L^2$ sense, we can finally compute all the coefficients in these expansions by projecting onto fixed angular modes.
In the case of \cref{eq:intro:waveSchwarzschild}, we can thus simply cite the results of \cref{thm:intro:III} to obtain \cref{thm:intro:IIIsummed}. In the exact same way, we can also prove \cref{conj:intro} in the case of polyhomogeneous data without error.
Notice that the summing of the $\ell$-modes here comes up as an after-thought! We first construct general expansions in \hyperlink{step2}{Step 2)} without knowing the coefficients, and then compute the coefficients by projecting onto $\ell$-modes.

For more general short-range perturbations $\mathcal{P}[\phi]$, we can instead compute the coefficients in the polyhomogeneous expansions by iteratively using the conservation laws \cref{eq:intro:conslaw}: 
Roughly speaking, we can split up $\phi=\phi^{(0)}+\phi^{(1)}+\phi^{(2)}$, where $\phi^{(0)}$ comes from initial data $r^{-1}\psi^{\incone}$ and solves $\Box_\eta\phi^{(0)}=0$. The asymptotic expansion for $\phi^{(0)}$ is readily computed via \eqref{eq:intro:conslaw}.
Next, we insert the solution $\phi^{(0)}$ back into the equation and let $\phi^{(1)}$ solve
$\Box_\eta \phi=\mathcal{P}[\phi^{(0)}]$ with trivial data, an inhomogeneous Minkowskian wave equation. 
The asymptotic expansion for $\phi^{(1)}$ is then computed via the inhomogeneous version of \eqref{eq:intro:conslaw}, namely:
\begin{equation}\label{eq:intro:conslawf}
\Box_\eta\phi=f\implies \pu(r^{-2\ell}\pv(r^2\pv)^\ell\psi_\ell)=r^{-2\ell-2}(r^2\pv)^{\ell}(r^3f_{\ell}).
\end{equation}
One finally proves an error estimate on the difference $\phi^{(2)}$. Depending on the order of the desired asymptotics, one of course needs to iterate the procedure further. This proves \cref{thm:intro:main}.

\begin{rem}[Expanding around Minkowski]
   In practice, it is important to have a good understanding of the equation in question and to make a good choice for the expansion of $\phi$ "around the Minkowskian solution" in order to make the calculations accessible. See already \cref{sec:app,sec:Sch} for examples. 
For instance, in Schwarzschild, writing the equation \cref{eq:intro:waveSchwarzschild} in $u$, $v$ coordinates as
\begin{equation}\label{eq:intro:Schw_uv_coordinates}
\pu\pv\psi=\frac{(1-2M/r)\Dl\psi}{r^2}-\frac{(1-2M/r)2M\psi}{r^3}
\end{equation}
and then expanding around $M=0$ would result in very messy computations in the iterative procedure above due to the logarithmic terms in $r= v-u+\O(M\log |v-u|)$.
More precisely, letting $\psi^{(0)}$ solve $\pu\pv\psi^{(0)}=\frac{\Dl\psi^{(0)}}{(v-u)^2}$, then the next iterate would need to solve $\pu\pv\psi^{(1)}=O(\frac{\log (v-u)}{(v-u)^3})\Dl\psi^{(0)}+\dots$, with spurious logarithmic terms appearing.

A more appropriate choice of coordinates is given by $r_0=r(u,v_0),\,r=r(u,v)$, which casts the equation into the form
	\begin{nalign}\label{eq:intro:Schw:waveequationinr0r}
			-(D_0 \partial_{r_0}+D \partial_r)(D\partial_r\psi)=\frac{D\Dl\psi}{r^2}-\frac{2MD\psi}{r^3},\qquad D_0:=1-\frac{2M}{r_0},\qquad D=1-\frac{2M}{r}.
		\end{nalign}
        Indeed, if we set $M=0$ and define the first iterate as the solution to $-(\partial_{r_0}+\partial_r)\partial_r(\psi^{(0)})=\frac{\Dl\psi}{r^2}$, then we avoid these logarithmic terms in further expansions of the equations. See already the proof around \cref{eq:Schw:waveequationinr0r} in
\cref{sec:Sch} for further details.
\end{rem}
\begin{rem}[Linear long-range potentials]
We have already mentioned that we can also include linear long-range potentials of, say, the form  $\mathcal{P}_L[\phi]=c\phi/r^2+\mathcal{P}[\phi]$, with $\mathcal{P}$ denoting a short-range perturbation.
Indeed, in this case, \hyperlink{step1}{Step~1} and \hyperlink{step2}{Step~2} work as before. On the other hand, in \hyperlink{step3}{Step~3}, we cannot expand around the solution to $\Box_\eta\phi=0$, but we have to expand around $\Box_\eta\phi=\frac{c\phi}{r^2}$ instead. In particular, we can no longer use the conservation laws \cref{eq:intro:conslaw}. See already \cref{sec:even:ode} how to proceed in that case.
    \end{rem}
 
 \begin{rem}[Estimate on the error-term]\label{rem:intro:error} In \cref{thm:intro:III}, we require $\epsilon>0$, whereas, in \cref{thm:intro:IIIsummed}, we require $\epsilon>1$. 
 The reason for this loss in going from $\ell$-modes to the resummed solution is directly related to our inability to prove \cref{conj:intro}. The best estimate we are able to prove in \hyperlink{step2b}{Step 2b)} translates $p+\epsilon$-decay towards $\scrim$ into an expansion up to order $p+\epsilon$ towards $\scrip$, and since we want to resolve the expansion up to order $p+1$, we must take $\epsilon>1$.
 We emphasise that, in odd spacetime dimensions, where peeling does not hold even for fixed $\ell$-modes, the error estimate from \hyperlink{step2b}{Step 2b)}  is \emph{sharp}.
 See, again, \cref{sec:sharp} for more on this matter.
 \end{rem}

\subsection{Mathematical background and description of the framework}
\label{sec:intro:mathcontext}
We now give an introduction to the mathematical framework employed in this work and embed it into the literature. We first discuss general aspects of the scattering problem and sketch the origin of our main weighted energy estimates (\hyperlink{step1}{Step 1)} and \hyperlink{step3}{Step 2b)}), and then discuss polyhomogeneous expansions (\hyperlink{step2a}{Step 2a)}).

\subsubsection{Scattering constructions, conormality and energy estimates for perturbations of \texorpdfstring{$\Box_\eta \phi=f$}{Box phi =f}}\label{intro:sec:scattering_theory} 

The scattering map is the map from a state of a \emph{free} theory coming in from the infinite past, undergoing various interactions, and converging to another free theory at the infinite future.
We refer to \cite{reed_iii_1979} for further physical motivation and to the introductions of \cite{nicolas_conformal_2016, dafermos_scattering_2018,masaood_scattering_2022} for an overview of the history of the problem in the context of wave equations on black hole backgrounds. 
While the first scattering result in the context of wave equations of black hole spacetimes \cite{dimock_classical_1987} still made use of the Lax--Philips scattering framework \cite{lax_scattering_1964}, the more modern approach is based on Friedlander's framework \cite{friedlander_radiation_1980} of geometric scattering, where the scattering data (or the free theory) live at past or future null infinity, respectively. 

In this work, we are less interested in the properties of the scattering map, but rather of scattering solutions, i.e.~solutions arising from data at $\incone\cup\scrim$. 
There are various approaches to constructing scattering solutions. For instance, one can construct them via a limiting argument as done in \cite{christodoulou_formation_2009,dafermos_scattering_2013,dafermos_scattering_2018,masaood_scattering_2022}, or by taking a conformal perspective, see e.g.~\cite{nicolas_conformal_2016,taujanskas_controlled_2023,minucci_maxwell-scalar_2022}. 
As we mentioned in \cref{rem:intro:conovsconf}, the second approach requires a degree of conformal regularity too high for many physically relevant scenarios, see \cite{WalkerWill79, Damour86, christodoulou_global_2002} as well as \cite{kehrberger_case_2024,kehrberger_case_2024-1}.
We shall therefore follow the former approach, i.e., we shall prove uniform energy estimates for a sequence of solutions and show that this sequence obtains a limit that we may identitfy as the scattering solution. 
 
We emphasise that the present work (just like \cite{kehrberger_case_2024,kehrberger_case_2024-1,taujanskas_controlled_2023})  does not treat the global, but the \textit{semiglobal} scattering problem, where only part of the data are posed at $\scrim$ and the rest of the data are posed along an ingoing cone $\incone$, see \cref{fig:intro:main}.\footnote{Of course, prescribing trivial data along the incoming cone  corresponds to a solution that extends all the way to past timelike infinity as well.}

\paragraph{The behaviour of the free theory $\Box_\eta\phi=f$}
Since we want to construct scattering theories for a class of equations that are perturbations of the Minkowskian equation, we first discuss solutions to $\Box_\eta\phi=f$ along with what we take to be their crucial properties.
In what follows, we will use notation suited for $3+1$ dimensions, but everything applies to $n+1$ dimensions, with $n\geq3$, if $r\phi=\psi$ and $rf$ are replaced by $r^{(n-1)/2}\phi=\psi$ and $r^{(n-1)/2}f$, respectively.
Focusing our attention entirely on the region $\Dopen$ (cf.~\cref{fig:intro:main}), the first observation (going back to \cite{friedlander_radiation_1980}) is that, for $f=0$, finite $T$-energy solutions to the Cauchy problem attain a future radiation field, i.e.~$\pu(r\phi)$ converges to a limit towards $\scrip$; similarly, $\pv(r\phi)$ attains a limit towards $\scrim$.
Commuting with the symmetries of $\Box_\eta$ then shows that a suitable set of operators with respect to which to measure smoothness is
\begin{equation}\label{eq:intro:Vb}
\Vb=\{u\pu,v\pv,\sl,1\},
\end{equation}
where $\sl$ refers to the covariant derivative on the unit sphere. We shall refer to smoothness and boundedness w.r.t.~$\Vb$-commutations as \textit{conormality}.\footnote{We contrast conormality to \textit{conformal} regularity as in \cite{friedrich_existence_1986} or regularity with respect to $\partial_t$ and $\partial_x$ as in \cite{anderson_stability_2023}.}
Indeed, $\Box_\eta \phi=0$ has the property that if the data for $\psi=r\phi$ are conormal with respect to these vector fields, then so is the solution, at least below top order in derivatives.
The perturbations around $\Box_\eta\phi=0$ allowed in the present work are such that they respect this conormal structure.

\paragraph{Notation} For a more concrete discussion, we now introduce the coordinate functions
    \begin{equation}
        \rho_-=\frac{v}{r},\quad \rho_0=\frac{r}{-uv},\quad\rho_+=\frac{-u}{r}.
    \end{equation}
These coordinates characterise the idealised boundaries past null infinity $\scrim$ ($\rho_-\to0$), spacelike infinity $I^0$ ($\rho_0\to0$) and future null infinity $\scrip$ ($\rho_+\to0$), respectively. As we will see below, it is important to distinguish between the behaviours towards all three of these boundaries. It is therefore helpful to introduce $\D$ to be the compactification of~$\Dopen$ with these three boundaries, see~\cref{fig:intro:a} and \cref{def:notation:comp}.\footnote{The compactification and boundaries merely introduce appropriate coordinates for the analysis near infinity, and, while one never needs to make use of the boundaries ($\scrim$, $I^0$, $\scrip$), or equivalently, the compactification, the notation turns out to be very useful in building intuition, facilitating the bookkeeping and for diagrammatic depictions.}
Note that, in the compactified picture, conormality simply corresponds to regularity, all the way up to the boundaries, with respect to those smooth vector fields on $\D$ that have no normal components along the boundaries, namely with respect to angular vector fields and $\rho_{\bullet}\partial_{\rho_{\bullet}}$ for $\bullet\in\{-,0,+\}$. These are precisely the $\Vb$ vector fields from \cref{eq:intro:Vb}.

On $\D$, we work with the $L^2$-based Sobolev spaces (cf.~\cref{def:notation:Hb_spaces})
\begin{equation}
\Hb^{a_-,a_0,a_+; k}(\D)=\rho_-^{a_-}\rho_0^{a_0}\rho_+^{a_+}\Hb^{;k}(\D)=\{f\in L^2_{\mathrm{loc}}(\Dopen):\rho_-^{-a_-+3/2}\rho_0^{-a_++2}\rho_+^{-a_++3/2}\Vb^k f\in L^2(\Dopen) \}.
\end{equation}
Via the Sobolev embedding
\begin{equation}
    \Hb^{a_-,a_0,a_+;\infty}(\D)\subset \rho_-^{a_-}\rho_0^{a_0}\rho_+^{a_+}C^{\infty}(\D)\subset\Hb^{a_--\epsilon,a_0-\epsilon,a_+-\epsilon;\infty}(\D)
\end{equation}
for any $\epsilon>0$, the membership $f\in\Hb^{a_-,a_0,a_+; k}(\D)$ roughly means that $f$ is bounded in an $L^2$ sense by $\rho_-^{a_-}\rho_0^{a_0}\rho_+^{a_+}$ (cf.~\cref{rem:not:Hb_Linfty}), and that the same is true for up to $k$ vector fields from $\Vb$ acting on $f$. In other words, $f$ decays like $r^{-a_-}$ towards $\scrim$, like $r^{-a_0}$ towards $I^0$ and like $r^{-a_+}$ towards $\scrip$.
Often, we will work in the smooth category ($k=\infty$) and drop the $;\infty$ superscript.

\paragraph{Optimal decay rates for $\Box_\eta\phi=f$ and admissible weights} Observe the following mapping properties for the free wave operator (cf.~\cref{eq:intro:box_near_corner}):
\begin{equation}\label{eq:i:box_weights}
    \Box_\eta:\Hb^{a_-,a_0,a_+;k}(\D)\to  \Hb^{a_-+1,a_0+2,a_++1;k-2}(\D).
\end{equation}
The $(1,2,1)$-weight corresponds exactly to the factor $|u|^{-1}v^{-1}=\rho_-\rho_0^2\rho_+$, which in turn comes from the $\pu\pv$-part of the principle symbol of $\Box_\eta$. (Note already that the $\Dl/r^2$-term, on the other hand, comes with a $(2,2,2)$-weight.)
Similarly to \cite{hintz_microlocal_2023-1},\footnote{See already \cref{sec:en:previous} for comparison between the present work and \cite{hintz_microlocal_2023-1}.} the weighted energy estimate we will prove for $\Box_\eta^{-1}$ roughly gain the same decay rates and one order of regularity up to some subtleties. 
More concretely, we will show, using the vector field multipliers listed in \cref{eq:current:past_currents} and \cref{eq:current:future_current}, that for scattering data $\psi^{\incone}\in \Hb^{a_-}(\incone)$, $\pv\psi^{\scrim}=0$, and for inhomogeneity $rf\in\Hb^{a^f_-,a^f_0,a^f_+}(\D)$, solutions to $\Box_\eta\phi=f$ satisfy 
\begin{equation}\label{eq:intro:admissiblebelow}
r\phi=\psi\in\Hb^{a_-,a_0,a_+}(\D),\end{equation} where $\vec{a}$ ($=(a_-,a_0,a_+)$) and $\vec{a}^f$ are related to each other by the \textit{admissibility criteria} defined in~\cref{def:en:admissible:f} as well as \cref{def:scat:extended_short_range}:
Ignoring the gain in regularity, we call $\vec{a}$ admissible if $a_+<\min(0,a_0)\leq a_0<a_-\geq -1/2$, and we call $\vec{a}^f$ admissible w.r.t~$\vec{a}$ if $\vec{a}+(1,2,1)\leq \vec{a}^f$ (in accordance with \cref{eq:i:box_weights}),  and if $a_-^f>1$.
The condition $a_-\geq -1/2$ is simply the condition that the data along $\incone$ have finite $T$-energy; the condition $a_+<a_0<a_-$ captures the fact that we can never prove stronger decay than what we started with, and  the condition  $a_+<0$ comes from the existence of a radiation field towards future null infinity. Finally, the condition  $a_-^f>1$ is  related to the requirement that $\pv\psi$  attain a finite limit towards $\scrim$.
See \cref{sec:en:en} and \cref{prop:en:main} for further discussion and a precise statement in the context of the finite problem, along with the improvement \cref{cor:scat:enlarged_admissible_set}.

We finally note that the assumption of trivial data at $\scrim$ does not lose generality: Given non-trivial data at~$\scrim$, we can always subtract these from the solution to obtain a new solution with different inhomogeneity and trivial data at~$\scrim$.
\begin{figure}[htpb]
\centering
\begin{subfigure}{0.45\textwidth}
\centering
    \includegraphics[width=96pt]{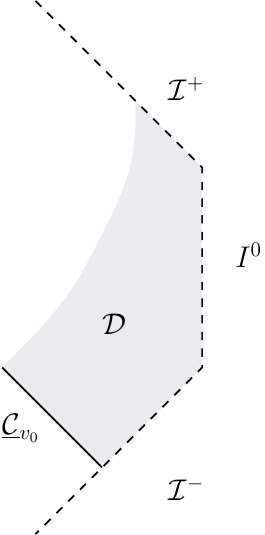}
    \caption{Compactification of $\Dopen$ w.r.t.~$\rho_-$, $\rho_0$, $\rho_+$.}
    \label{fig:intro:a}
\end{subfigure}
\hspace{0.05\textwidth}
\begin{subfigure}{0.45\textwidth}
\centering
 \includegraphics[width=72pt]{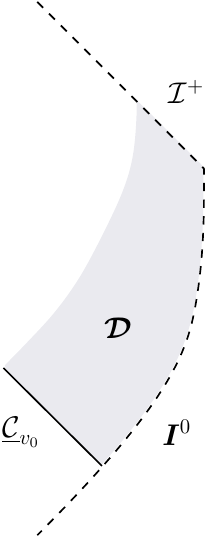}
    \caption{Compactification of $\Dopen$ w.r.t.~$\rho_-\rho_0:=\boldsymbol{\rho}_0$, $\rho_+$.}
    \label{fig:intro:b}
\end{subfigure}
\caption{The two different compactifications $\D$ and $\Dbold$. The latter compactification is specifically tailored to capturing the no incoming radiation condition. The boundary $\boldsymbol{I}^0$ corresponds to limit points $\boldsymbol{\rho}_0\to0$ (which includes both $\scrim$ and $I^0$ in the $\D$-picture). 
Solutions with (without) incoming radiation are regular w.r.t.~smooth vector fields that are tangent to the boundaries of $\D$ ($\Dbold$).}\label{fig:intro:Dbold}
\end{figure}
\paragraph{Geometrically capturing the no incoming radiation condition}
As already mentioned above, solutions to $\Box_\eta\phi=0$ are smooth w.r.t.~commutations with $\Vb$. In particular, a $\pv$-derivative only improves decay towards $I^0$ or $\scrip$, but not towards $\scrim$---it's a "bad" derivative towards $\scrim$.
However, for solutions arising from no incoming radiation, the $\pv$-derivative actually becomes a good derivative in the sense that such solutions remain smooth with respect to not just $v\pv$, but, in fact, with respect to $r\pv$ (although at a loss of regularity at top order).
The easiest way to see this is by simply integrating the equation $\pu\pv\psi=\Dl\psi/r^2$ from $\scrim$, and then inductively commuting with $r\pv$.

This simple equivalent characterisation of no incoming radiation was noticed already in \cite[Cor.~6.3]{baskin_asymptotics_2015}.
In our context of semi-global scattering constructions, we use it to streamline various proofs throughout our work that are more involved in the case of nontrivial incoming radiation,\footnote{Notice that this characterisation also places the same condition on $f$ when considering $\Box_\eta\phi=f$; in particular, the trick of subtracting the incoming radiation mentioned above will produce an inhomogeneity that is not smooth with respect to $r\pv$.} and we also use it in an essential way to prove various generalised scattering results, see already \cref{sec:intro:scat:nic}. 
We also mention that smoothness with respect to commutations with
\begin{equation}
\Vbt=\{u\pu,r\pv,\sl,1\}
\end{equation}
geometrically corresponds to being smooth on the compactification $\Dbold$ defined w.r.t.~the boundary defining functions  $\boldsymbol{\rho}_0=\rho_-\rho_0$ and $\rho_+$, see also \cref{fig:intro:b}.

\paragraph{Short-range perturbations of $\Box_\eta\phi=f$}
Having understood the Minkowskian $\Box_\eta\phi=f$, we next discuss the treatment of \textit{short-range perturbations} of the form
 \begin{equation}\label{eq:intro:wavegeneral}
     \Box_\eta \phi=P_g[\phi]+\mathcal{N}[\phi]+V[\phi]+f=\mathcal{P}[\phi]+f,
 \end{equation}
 where $P_g$, $\mathcal{N}$, $V$ denote quasilinear, semilinear and potential perturbations, respectively, and where $f$ denotes an inhomogeneity. 
 The operators $P_g$, $\mathcal{N}$ and $V$ are all such that for $k$ sufficiently large, their mapping property on $\Hb^{\vec{a};k}(\D)$ spaces gains an $\epsilon>0$ weight compared to the $(1,2,1)$-weight of $\Box_{\eta}$ as in \cref{eq:i:box_weights}.
 For the precise definition of allowed perturbations, see \cref{def:en:short_range} as well as \cref{def:scat:extended_short_range}. 
These definitions, in particular,  depend on the decay of the initial data.

We give an example to illustrate our definition of short-range perturbations: Consider the equation $\Box_\eta\phi=\phi^3$.
Assuming that $\psi=r\phi\in \Hb^{\vec{a}}(\D)$ for $\vec{a}$ admissible, then, treating $\phi^3$ as an inhomogeneity, the associated weight can be read off from  $r\phi^3=r^{-2}\psi^3\in \Hb^{3\vec{a}+\vec{2}}(\D)$.
The requirements that this be admissible (+$\epsilon$) with respect to $\vec{a}$, i.e.~$3\vec{a}+\vec{2}>\vec{a}+(1,2,1)$,  then translates into the requirements $a_->-1/2$, $a_0>0$ and $a_+>-1/2$, respectively.
Since the decay rates must satisfy $a_+<a_0<a_-$, we in fact require $a_->a_0$, so we can only solve this problem if $\psi^{\incone}$ decays like $r^{-a_-}$ for some $a_->0$.\footnote{It turns out that for nontrivial incoming radiation, we necessarily must have $a_-<0$. However, there are cases with special cancellations such that one has $a_0>a_-$ and is still able to treat the $\phi^3$-nonlinearity, see \cref{rem:poly:improved_decay_at_I0}.}

With this definition of short-range perturbations, we can extend our estimates for $\Box_\eta\phi=f$ to apply to \cref{eq:intro:wavegeneral} as well. See already \cref{sec:en:perturbations}.

For quasilinear perturbations, we of course need to be a bit more careful, as we have to put the quasilinear terms in \cref{eq:intro:wavegeneral} on the LHS in order to close energy estimates. This is achieved by essentially using the same multipliers but for the modified equations, see already \cref{sec:current:linearmetric}.

There is another subtlety related to quasilinear perturbations: To construct scattering solutions arising from data $\psi^{\incone}$ along $\incone$ and, say, trivial data along $\scrim$, one typically constructs a sequence of solutions arising from a sequence of compactly supported data $\chi_n \psi^{\incone}$ along $\incone$, where $\chi_n$ denotes a cutoff supported on $u\geq -n$.  
However, for quasilinear problems, the cutoff functions will generally deform the null cones. 
In order to separate this technical issue, we therefore treat general short-range quasilinear perturbations in two steps:
In the main body of this work, we restrict to quasilinear perturbations that leave in tact the Minkowskian null cones. For such perturbations, there is no difficulty with the limiting argument. 
Then, in \cref{sec:current:general_quasilinear_scattering}, we construct scattering solutions for \textit{general} quasilinear perturbations by performing Nash--Moser iteration at the level of scattering solutions for \textit{linearised} quasilinear equations.

\paragraph{Linear long-range potentials}
In addition to treating short-range nonlinear perturbations, we also treat linear \textit{long-range} potentials in this work: More precisely, we allow for $V$ in \cref{eq:intro:wavegeneral} to contain linear first-order terms that have weight $\rho_0^2$ towards~$I^0$. This includes, in particular, $1/r^2$ potentials or the first order terms appearing in the Teukolsky equations \cite{dafermos_linear_2019}. See already \cref{def:en:long_range}.
In contrast to short-range perturbations, we cannot treat these perturbatively; instead, we directly show in \cref{sec:en:longrange} that our core energy estimates still apply.

\paragraph{Comparison to related recent works and scattering in more general settings}
We conclude by commenting on two recent scattering results and comparing them to our work:

In \cite{yu_nontrivial_2022}, the author constructs scattering solutions to the equation $\Box_{\eta}\phi=(\partial_t\phi)^2$ in a neighbourhood of $\scrip$ including timelike and spacelike infinity. Now, $(\partial_t\phi)^2$ is not a short-range perturbation near $\scrip$ and so the author finds an \textit{alternative free} behaviour, and constructs scattering solutions around this alternative free behaviour. For such solutions, regularity with respect to $\Vb$ vector fields still persists.

In the present work, we can also treat such nonlinearities in the restricted region $\D\cap\{t\leq0\}$ provided that we prescribe trivial data along $\scrim$. The idea here is that for problems with no incoming radiation, $\partial_t$-derivatives become "good derivatives" near $\scrim$. See already \cref{sec:intro:scat:nic} and \cref{rem:en:examples}

We also mention \cite{lindblad_scattering_2023}: While that work is more generally concerned with the global structure of homogeneous scattering solutions, we here single out one aspect of their work: Amongst other things, the authors study scattering solutions for the equation $\Box_\eta\phi=f$ in a neighbourhood of $\scrip$, where $f=\mathcal{O}(r^{-\delta}\jpns{u}^{-q})$ for $\delta\leq2$ and $q\ll1$.
In general, such inhomogeneous terms ruin the possibility of prescribing the radiation field $\psi$ for solutions $\phi$, and thus, of constructing scattering solutions.
The authors in \cite{lindblad_scattering_2023} overcome this by assuming that $f$ is not merely in a Sobolev space, i.e.~error, but by also prescribing a leading order expansion such as $f=r^{-\delta}f_\delta(u)+\O(r^{-3}\jpns{u}^{-q})$, with $f_\delta(u)=\O(\jpns{u}^{-q})$.
This allows them to find scattering solution in a neighbourhood of $\scrip$ that, again, includes both timelike and spacelike infinity.
In this work, we prove a similar result for a large class of perturbations of $\Box_{\eta}$ in the restricted region $\D$, cf.~\cref{sec:intro:scat:phg} and \cref{thm:scat:weak_polyhom}.

\subsubsection{Polyhomogeneous expansions}\label{sec:intro:phg}
The last element required for this introduction is the concept of polyhomogeneity. 
Going back to the discussion in \cref{sec:intro:motivation}, a natural generalisation of having $1/r$ expansions towards $\scrip$ in the spirit of Bondi, Sachs and Penrose, is given by polyhomogeneous expansions, i.e.~expansions into $r^{-p}\log^n r$, $p\in\mathbb R$, $n\in\mathbb N$.
 The first reference in which the formal consistency of such expansions with the Einstein equations was tested is \cite{chrusciel_gravitational_1992}; see, however, references therein for less general earlier descriptions.
 Working with the framework introduced in \cite{melrose_atiyah-patodi-singer_1993}, this consistency was rigorously proved only much later in \cite{hintz_stability_2020}, where it was shown that evolutions of perturbations to the Minkowski initial data under the Einstein vacuum equations remain polyhomogeneous all the way up to future timelike infinity, provided the initial data are polyhomogeneous. 
 
 We quickly sketch the propagation of polyhomogeneity from $I^0$ to $\scrip$, which crucially relies on the following observation: Working in coordinates $\rho_0$, $\rho_+$, the twisted wave operator $r\Box_\eta (r^{-1} \cdot)$ takes the form:
 	\begin{equation}\label{eq:intro:box_near_corner}
			P_{\eta}:=r\Box_{\eta}r^{-1}=\rho_0^2\rho_+\bigg(-\rho_+\partial_{\rho_+}(\rho_+\partial_{\rho_+}-\rho_0\partial_{\rho_0})+\frac{\rho_+}{(1+\rho_+)^2}\Dl\bigg).
		\end{equation}	
The point is that the $\Dl$ term comes adorned with an extra decaying $\rho_+$-weight compared to the $-\rho_+\partial_{\rho_+}(\rho_+\partial_{\rho_+}-\rho_0\partial_{\rho_0})$, and can thus be treated perturbatively. Roughly speaking, one can thus simply integrate along the integral curves of $\rho_+\partial_+$ and of $\rho_+\partial_{\rho_+}-\rho_0\partial_{\rho_0}$ to iteratively construct a polyhomogeneous expansion towards $\rho_+=0$ term by term; see already \cref{lemma:prop:future_corner}.
(Of course, this procedure loses derivatives: The philosophy is always that one first constructs and proves energy estimates for sufficiently regular solutions, and then, in a second step, where regularity is no longer an issue, proves polyhomogeneity of those solutions.)

	In contrast, for the backwards problem (which is not addressed in \cite{hintz_stability_2020}), i.e.~going from $\scrip$ to $I^0$, or, equivalently, going from $\scrim$ to $I^0$, the angular term no longer comes with an extra decaying weight in the direction in which we want to construct an expansion and, thus, can no longer be treated perturbatively. 
		In the current work, we nevertheless prove the propagation of polyhomogeneity from $\scrim$ to $I^0$ via combinations of time integrals and suitable scaling commutations that remove the leading order terms in the expansions. 
  As before, we first treat the Minkowskian problem $\Box_\eta \phi=f$ (in \cref{sec:prop}) and then extend to short-range perturbations by perturbative arguments in \cref{sec:app}.
  See already \cref{sec:intro:example} for a sketch in the context of a simple model problem. 
 
 Before we can move on to stating our main results, we finally introduce the standard notation used to express polyhomogeneity following \cite{grieser_basics_2001}. 
 We use index sets $\mathcal{E}=\{(p,n)\}\subset\mathbb R\times \mathbb N$ to denote the set of terms $r^{-p}\log^{n} r $ appearing in the relevant expansions. By definition, the subsets $\mathcal{E}_{<c}:=\{(p_i,n_i)\in \E\,|\, p_i<c\}$ are required to be finite. 
 With the precise definition given in \cref{def:notation:index_sets,def:notation:polyhom}, the  notation
 $\psi^{\incone}\in \A{phg}^{\E_-}(\incone)$ essentially means that $\psi^{\incone}=\sum_{(p,n)\in\E^-} c^{p,n}(\theta,\varphi)\rho_-^{-p}\log^n \rho_-$ for a sequence of functions $c^{p,n}$ on $S^2$.
The slightly more complicated notation for polyhomogeneity towards multiple boundaries $\scrim$, $I^0$ and $\scrip$, namely $\psi\in\A{phg}^{\E_-,\E_0,\E_+}(\D)$, essentially means that $\psi$ admits an expansion of the form 
 \begin{equation}
 \chi\prod_{\bullet\in\{-,0\}}\left(\sum_{(p_\bullet,n_\bullet)\in\E_{\bullet}} c_\bullet^{p_\bullet,n_\bullet} (\theta,\varphi)\rho_\bullet^{p_\bullet}\log^{n_{\bullet}}\rho_{\bullet}\right)+(1-\chi)\prod_{\bullet\in\{0,+\}}\left(\sum_{(p_\bullet,n_\bullet)\in\E_{\bullet}} c_\bullet^{p_\bullet,n_\bullet} (\theta,\varphi)\rho_\bullet^{p_\bullet}\log^{n_{\bullet}}\rho_{\bullet}\right),
 \end{equation}
 where $\chi$ is a cut-off function localising to the past corner $t\leq 0$ and the $c_{\bullet}$ are some functions on the sphere.
We also use notation for situations where there is an expansion in only one direction and merely some decay in the other two directions: For instance, denoting by $\overline{(0,0)}$ the \textit{conformally smooth index set} given by $\overline{(0,0)}=\{(n,0)|n\in\mathbb N\}$, then the claim of \cref{conj:intro} is that solutions arising from arbitrary data along $\incone$ (so also data that do not admit an expansion) and no incoming radiation enjoy the membership $\psi\in\A{b,b,phg}^{a_-,a_0,\mindex0}(\D)$.
We already include an explanatory diagram in \cref{fig:intro:explain} below:
 \begin{figure}[htbp]
 \centering
\includegraphics[width=0.8\textwidth]{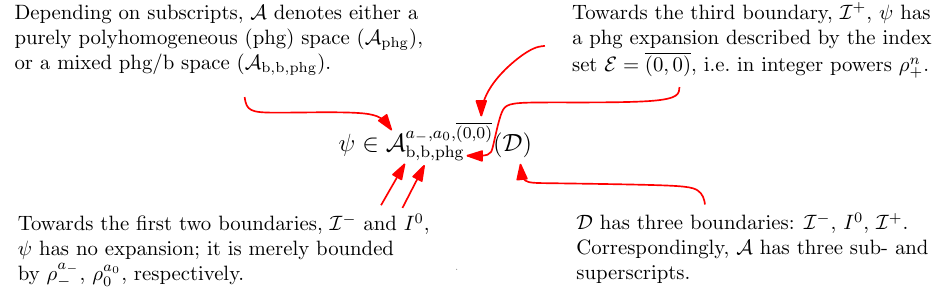}
\caption{A schematic guide to reading the notation $\psi\in\A{b,b,phg}^{a_-,a_0,\mindex0}(\D)$. In case $\psi$ is polyhomogeneous towards all boundaries, we simply write $\A{phg}$ instead of $\A{phg,phg,phg}$; in case $\psi$ is merely bounded towards all boundaries, we write $\Hb$ rather than $\A{b,b,b}$. In contrast, the number of superscripts always matches the number of boundaries. We also notice that, if we work on a space with only two boundaries, e.g.~$\D\cap\{t\leq 0\}$ or $\Dbold$, then we correspondingly only write two sub- and superscripts (e.g.~$\A{phg}^{\E_0,\E_+}(\Dbold)$). }
\label{fig:intro:explain}
\end{figure}

More precisely, restricting to $t\geq0$, the membership $\psi\in\A{b,b,phg}^{a_-,a_0,\mindex0}(\D)$ means that $\psi$ has an expansion $f^n(\omega,\rho_0)\rho_+^{n}$ around $\scrip\cap I^0$ for some $f^n(\omega,\rho_0)\in\Hb^{a_0}([0,\delta)_{\rho_0}\times S^2)$, $\delta<1$ sufficiently small.

	\subsection{Statement of the main results}\label{sec:intro:mainresults}
	Our main results  concern the perturbed wave equation \cref{eq:intro:wavegeneral} in $3\!+\!1$ dimensions (see however \cref{rem:intro:higherd} for the generalisation to other dimensions), for which we recall that the admissible class of nonlinearities, potentials and inhomogeneities is defined in \cref{def:en:admissible:f}, \cref{def:en:short_range} and \cref{def:scat:extended_short_range}.
We also recall anew that these definitions are such that we can essentially treat all terms on the RHS of \cref{eq:intro:wavegeneral} perturbatively (which, in turn, depends on the decay of the posed data), and such that the $\b$-smooth structure (smoothness with respect to the $\Vb$ vector fields) remains intact.

 We first introduce the notion of a scattering solution:
 \begin{defi}
     Given smooth data $\psi^{\incone}$ along $\incone$ and $\pv\psi^{\scrim}$ along $\scrim$, then we say that $\phi$ is the scattering solution arising from these data if $\phi$ solves \cref{eq:intro:wavegeneral}, if $r\phi|_{\incone}=\psi^{\incone}$, and if for all $v_1\geq v_0$: $$\lim_{u\to-\infty}\norm{\pv(r\phi)(u,\cdot)-\pv\psi^{\scrim}(\cdot)}_{L^2(\{v\in [v_0,v_1]\}\times S^2)}=0.$$
 \end{defi}
 We now state an abbreviated version of the first main result; for the precise versions see \cref{thm:scat:scat_general}, \cref{cor:scat:enlarged_admissible_set} (treating the case of no incoming radiation) and \cref{thm:scat:scat_incoming} (treating the case of incoming radiation).
    \begin{thm}[Scattering, cf.~\cref{thm:scat:scat_general,thm:scat:scat_incoming}]\label{thm:intro:scat}
  Let $\vec{a}$ be admissible with $a_-<0$, and let $\vec{a}^f$ be admissible w.r.t.~$\vec{a}$.
  Given conormal scattering data $\psi^{\incone}\in\Hb^{a_-}(\incone)$ and $v\pv\psi^{\scrim}\in \Hb^{a_0}(\incone)$, inhomogeneity $rf\in\Hb^{\vec{a}^f}(\D)$ as well as  short-range perturbations $P_g$, $\mathcal{N}$, $V$, then there exists a unique conormal scattering solution $\psi=r\phi\in\Hb^{\vec{a}}(\D)$ to \cref{eq:intro:wavegeneral} arising from these data, provided we restrict to a region bounded by sufficiently large negative retarded time.
More precisely, this solution satisfies $\psi\in\Hb^{\vec{a}}(\D)$.

If $v\pv\psi^{\scrim}=0$, then we can lift the assumption $a_-<0$ and still prove $ \psi\in\Hb^{\vec{a}}(\D)$.
 Furthermore, if we additionally place modified assumptions on the short range perturbations and the inhomogeneity, then the solution is also smooth with respect to $r\pv$, and we have $(r\pv)^k\psi\in\Hb^{\vec{a}}(\D)$ $\forall k\geq0$.
\end{thm}
\begin{rem}[The condition $a_-<0$]
    For problems with nontrivial data at $\scrim$,  $\psi$ will generically not decay towards~$\scrim$, which forces the condition $a_-<0$. This condition disappears for trivial data at $\scrim$.
\end{rem}

\begin{rem}[Uniqueness]
The uniqueness above is understood with respect to the class of solutions for which the $\Hb^{\vec{a};{5}}(\D)$-norm is finite, which is a stronger assumption than finite $T$-energy. We did not investigate whether this can be improved to finite energy solutions for general perturbations, but it can for linear problems.

\end{rem}

\cref{thm:intro:scat} tells us, in particular, that if the initial data are conormal, then so is the solution. The next statement we prove tells us that if the initial data are polyhomogeneous, then so is the solution, and that if the initial data feature certain decay, then the solution is polyhomogeneous up to error term.
	\begin{thm}[Persistence of polyhomogeneity, cf.~\cref{thm:app:general,thm:app:general2} for the precise versions]\label{thm:intro:prop_polyhom}
Under the assumptions of \cref{thm:intro:scat},
$\psi$ has an expansion towards $\scrip$ (cf.~\cref{fig:intro:explain} for a schematic explanation of the notation): 
\begin{equation}\label{eq:intro:thm:error}
\psi\in \A{b,b,phg}^{a_-,a_0,\mindex0}(\D)+\Hb^{a_-,a_0,\min(a_0,a_+^f-1)-\epsilon}(\D)
\end{equation} \emph{for any} $\epsilon>0$. In the case of no incoming radiation, \cref{eq:intro:thm:error} furthermore holds for any $a_-\geq -1/2$. In the case of non-trivial incoming radiation, we can also lift the assumption that $a_-<0$, resulting in  \cref{thm:app:general2}\cref{item:app:error2}.

If we additionally assume that the scattering data, $f$, as well as the coefficients of the nonlinearities and potential $P_g,\,\mathcal{N}, V$ are polyhomogeneous towards the boundaries $\scrim,\,I^0, \,\scrip$, then the solution is polyhomogeneous  as well.
\end{thm}
\begin{rem}[Long-range perturbations]
We extend \cref{thm:intro:scat,thm:intro:IIIsummed} to also admit linear long-range perturbations such as $1/r^2$-potentials or first order terms as appearing in the Teukolsky equations \cite{dafermos_linear_2019}, cf.~\cref{thm:scat:long} and \cref{thm:app3}. See also \cref{rem:Sch:Teukolsky_decay_rate} for further comments on the Teukolsky equations.
\end{rem}

\begin{rem}[Other spacetime dimensions]\label{rem:intro:higherd}
    Both \cref{thm:intro:scat} and \cref{thm:intro:prop_polyhom} apply in any spacetime dimension $\geq 4$ after changing the definition of the radiation field to $\psi=r^{(n-1)/2}\phi$ and replacing $rf\mapsto r^{(n-1)/2}f$. In fact, since we can also treat $1/r^2$-potentials with arbitrary sign, we can also extend the results to hold in spacetime dimension $2+1$. Finally, the results trivially apply to $1+1$-dimensions as well.
In general, the only parts of this work that apply exclusively to even spacetime dimensions are those of \cref{sec:app:app1}--\cref{sec:Sch}, as well as those of \cref{section:no incoming radiation Cauchy} and \cref{sec:EVE}, though the latter may be extended with a bit more care.
\end{rem}
 \begin{rem}[Optimality of the estimates and index sets]
Disregarding losses in regularity, the estimate on the error term \eqref{eq:intro:thm:error} is sharp in odd spacetime dimensions (cf.~\cref{prop:even}), but we do not know if it's sharp in even  spacetime dimensions, cf.~the discussion at the end of \cref{sec:intro:motivation} and around \cref{conj:intro}.
On the other hand, in the more precise version of \cref{thm:intro:prop_polyhom}, we actually prove that the solution admits expansions with respect to specific index sets. While the theorem itself makes no statements on the vanishing or nonvanishing of the coefficients in these index sets, we provide several examples where these index sets are, in fact, exhausted with nonvanishing coefficients; in this sense, the expansions are optimal (in any dimension).
 \end{rem}

 \subsection{Applications of the main results to some examples}\label{sec:intro:applications}
 \cref{thm:intro:prop_polyhom} tells us that scattering solutions have polyhomogeneous expansions up to some error term, but it does not give us the coefficients in these expansions. 
 The computation of the latter is achieved as explained in \hyperlink{step3}{Step 3)}; we here present the result for a few different applications:

The first application is to $\Box_\eta\phi=0$; this ties in with our discussion of the peeling property in \cref{sec:intro:motivation}:
\begin{cor}[Cf.~\cref{cor:app:minkowski_with_error}]\label{cor:intro:mink}
For any $a_+<a_-\geq -1/2$, let $\phi$ be the scattering solution to $\Box_\eta\phi=0$ arising from no incoming radiation and $\psi^{\incone}=\A{phg}^{\E_-}(\incone)+\Hb^{a_-}(\incone)$ for $\E_-$ an \emph{arbitrary} index set with $\min(\E_-)>-1/2$. Then, along any outgoing null cone $\C_u$, we have $\psi|_{\C_u}\in \A{phg}^{\mindex0}(\C_u)+\Hb^{a_+}(\C_u)$; in particular, the solution has expansions in integer powers of $1/r$ up to order $a_+$.

\end{cor}
Again, removing the $\Hb^{a_+}(\C_u)$-error term above would amount to proving \cref{conj:intro}.

	The second application is to the following equation, for some $M\in\mathbb{R}_{\neq 0}$:
	\begin{equation}\label{eq:intro:wave:potential}
	\Box_\eta \phi=-\frac{M\phi}{r^3}.
	\end{equation}
 Recalling the discussion at the beginning of \cref{sec:intro:motivation}, we shall henceforth only focus on the appearance of the first conformally irregular term towards $\scrip$ and restrict to data that are polyhomogeneous (without error term).

    \begin{cor}[Cf.~\cref{lem:app:potential}]\label{cor:intro:polyhom_potential}\label{cor:kotz}
		Let $p>-1/2$ and $0\neq C(\omega)\in C^{\infty}(S^2)$. 
    Given initial data $\psi^{\incone}\in\A{phg}^{\overline{(p,0)}}(\incone)$ (i.e.~with an expansion into powers of $r^{-(p+n)}$, $n\in\mathbb N$) such that $\lim_{u\to-\infty} r^p \psi^{\incone}=C(\omega)$, then there exists an explicitly computable $\tilde{C}(\omega)\neq0$ such that the scattering solution to \cref{eq:intro:wave:potential} with no incoming radiation satisfies along any outgoing null cone $\C_u$:
    \begin{equation}
        \psi|_{\C_u}=\Big(\sum_{i=0}^{\lceil p\rceil} f_i(u,\omega) r^{-i}\Big)+\tilde{C}(\omega) r^{-p-1}\log^{\sigma}r+\O(r^{-\lceil p+1\rceil}).
    \end{equation}
    Here, $\sigma=1$ if $p\in\mathbb{N}_{\geq1}$ and $\sigma=0$ otherwise.
	\end{cor}
	\begin{rem}
	It is interesting to compare this to the wave equation on Schwarzschild \cref{eq:intro:waveSchwarzschild}. Indeed, one might naively think that \cref{eq:intro:waveSchwarzschild} is morally equivalent to \cref{eq:intro:wave:potential}, but \cref{cor:intro:polyhom_potential} shows that this is not the case. For instance, in Schwarzschild, according to \cref{thm:intro:III}, there is also a nonvanishing logarithmic term at order $r^{-p-1}\log r$ if $p=0$.
	\end{rem}
One can nevertheless treat the Schwarzschild wave equation \eqref{eq:intro:waveSchwarzschild} similarly to the potential case above; this results in the proof of \cref{thm:intro:IIIsummed}. However, here, it is crucial to work in different coordinates than the $u$, $v$ coordinates. See already \cref{sec:Sch}, where this is discussed in detail and also generalised to the case of linearised gravity around Schwarzschild (see \cref{thm:Schw:lingravity}). 
	
	Finally, (and merely to demonstrate the utility of the algorithm), we treat an example semilinearity, namely
	\begin{equation}\label{eq:intro:wavepupv}
	\Box_\eta \phi=\pu\phi\pv\phi.
	\end{equation}

	\begin{cor}[Cf.~\cref{lem:app:semi}]\label{cor:intro:semilinear}
Under the assumptions of \cref{cor:kotz}, but with $\phi$ now solving \cref{eq:intro:wavepupv} instead of \cref{eq:intro:wave:potential}, there exists  $\tilde{C}(\omega)$ such that the solution satisfies along any $\C_u$
 \begin{equation}
        \psi|_{\C_u}=\Big(\sum_{i=0}^{\lceil 2p \rceil} f_i(u,\omega) r^{-i}\Big)+\tilde{C}(\omega) r^{-2p-1}\log^{\sigma}r+\O(r^{-\lceil 2p+1\rceil}).
    \end{equation}
    Here, $\sigma=1$ if $2p+1\in\mathbb{N}_{\geq2}$ and $\sigma=0$ otherwise. Furthermore, $\tilde{C}(\omega)\neq 0$ if $p\neq 0,1$.
	\end{cor}
	\begin{rem}
	In this example, the precise form of the index set and the order at which the first conformally irregular term appears depends very intricately on the initial data. In particular, if $p=0,1$, then we find that $\tilde{C}(\omega)=0$, i.e.~no conformally irregular term appears at order $2p+1$ at least for data supported on, say, $\ell\leq 10$.

    We also note that for the classical null form nonlinearity $\Box_\eta\phi=\pu\phi\pv\phi-r^{-2}|\sl\phi|^2$, one of course has conformal smoothness towards $\scrip$ (since the equation can be transformed to a linear equation).
	\end{rem}

The examples above are only treated in order to showcase in detail how to apply the algorithm of \hyperlink{step3}{Step 3)}. For now, the most important application remains that to \emph{linearised} gravity in \emph{double null gauge} around Schwarzschild treated in \cref{sec:Sch}, and to the Einstein vacuum equations in \emph{harmonic gauge}.
Since properly introducing either system would take up too much space for this introduction, we refer the reader to \cref{sec:Sch}, \cref{sec:EVE:equations} and to \cite{dafermos_linear_2019,kehrberger_case_2024}, \cite{lindblad_global_2005} respectively for background.
\subsection{Miscellaneous additional results}\label{sec:intro:misc}
We here list a few further results that are interesting in their own right, but slightly further detached from the discussion so far.

\subsubsection{Scattering results beyond finite energy with no incoming radiation}\label{sec:intro:scat:nic}
In \cref{thm:intro:scat}, we in particular required the data along $\incone$ to lie in $\Hb^{-1/2}(\incone)$, corresponding to finite $T$-energy. (This assumption is hidden in the requirement that $\vec{a}$ be admissible.)
We can, in fact, also consider data along~$\incone$ that decay much worse provided a suitable analogue of the no incoming radiation condition holds (\cref{def:scat:weak_decay}) and we specify a suitable amount of transversal derivatives at a finite sphere along $\incone$. 
The simple observation underlying this generalisation is that, given data with worse decay, we can apply sufficiently many time derivatives (which, by the no incoming radiation condition, improve decay) to the data to get back into the setting of \cref{thm:intro:scat} and construct a scattering solution $T^n\psi$. 
Then, we may integrate in time $n$ times to recover $\psi$ (this is why the extra transversal derivatives need to be specified). The precise statement is given in~\cref{thm:scat:weak_decay}.
This means, in particular, that the restrictions on $a_-$ and on $p$ in \cref{cor:intro:mink,cor:intro:polyhom_potential,cor:intro:semilinear} as well as \cref{thm:intro:III,thm:intro:IIIsummed} are, in fact, unnecessary.

\subsubsection{Scattering results beyond finite energy with polyhomogeneous data}\label{sec:intro:scat:phg}
Even if there is incoming radiation, we can still make sense of scattering solutions that don't have finite energy provided that the initial data and inhomogeneity are polyhomogeneous up to the order corresponding to finite energy.

We already give a brief sketch how this works: Firstly, if we consider the usual wave equation $\pu\pv\psi=r^{-2}\Dl\psi$ near~$\scrim$, then we see that the spherical Laplacian can be treated perturbatively, as integrating in the $\pu$ and $\pv$ direction only gives one extra weight towards $\scrim$, whereas the $\Dl$-term comes with two extra weights, cf.~\cref{eq:intro:box_near_corner}. 
Secondly, we note that the same perturbative treatment continues to be possible for short-range perturbations, as the RHS of $\pu\pv\psi$ will always have $1+\epsilon$ extra powers towards $\scrim$. 
Thus, the construction of scattering solutions is, to leading order, simply an ODE problem.

Consider then the ODE $\pu G= g$: The usual setting of the scattering problem corresponds to specifying the limit $G(u=-\infty)$, which produces a unique solution provided that $g$ is integrable. 
However, if $g$ is not integrable but instead has an explicit expansion, say, $g=1/u+O(|u|^{-1-\epsilon})$ for $\epsilon>0$, then this uniquely fixes the leading-order behaviour of $G$ (namely $\log u$).
To specify $G$ uniquely, we need only fix the order-1 term of $G$, i.e.~the term in the kernel of $\pu$. 

With these ideas, one can then iteratively construct the leading-order expansions to scattering problems for short-range perturbations provided that data, inhomogeneity and coefficients of the nonlinearities all have an expansion in order to finally get a remainder to which \cref{thm:intro:scat} can be applied again. We point out that the iterative construction of the leading-order expansion loses derivatives, so this only works for sufficiently regular problems.
 See~\cref{thm:scat:weak_polyhom} for details. This extension is important for \cref{sec:intro:EVE} below.

\subsubsection{The failure of peeling in odd spacetime dimensions}
That peeling is on less robust footing in odd spacetime dimensions has already been discussed in \cite{hollands_conformal_2004,godazgar_peeling_2012}. It turns out, however, that peeling fails for reasons orthogonal to those of the cited works: Already for compactly supported data along $\scrim$, the radiation field diverges logarithmically towards $\scrip$. Cf.~\cref{cor:even}.

\subsubsection{The antipodal matching condition}
The remaining subsections are concerned with even spacetime dimensions.

In \cite{masaood_scattering_2022}, it was shown that, for linearised gravity around Schwarzschild, Strominger's antipodal matching condition \cite{Strominger14,PI22,CNP22} can be reduced to the following statement for solutions to \cref{eq:intro:waveSchwarzschild}:
\begin{equation}\label{eq:intro:antipodal}
  P^{(\ell)}_{S^2}  \lim_{v\to\infty}\lim_{u\to-\infty}\psi=(-1)^{\ell}  P^{(\ell)}_{S^2}\lim_{u\to-\infty}\lim_{v\to\infty}\psi,
\end{equation}
which, in turn, was also proved in \cite{masaood_scattering_2022} based on the fixed $\ell$-result in \cite{kehrberger_case_2022}.
We show that \cref{eq:intro:antipodal} actually holds for a much larger class of equations in \cref{lemma:app:antipodal} (in even spacetime dimensions).

\subsubsection{The no incoming radiation condition on a Cauchy hypersurface}
The no incoming radiation condition is formulated most naturally as a condition on $\scrim$. If one nevertheless wants to work with Cauchy data, testing for no incoming radiation (without evolving the solution back to $\scrim$) is highly nontrivial. We provide an algorithmic procedure to do this, order by order, in \cref{section:no incoming radiation Cauchy}. 
Already for $\Box_\eta\phi=0$ (we note our algorithm also works on other spacetimes), testing for no incoming radiation at leading order requires testing a complicated Fourier multiplier relation between $\psi_\ell$ and $T\psi_\ell$ on data. See \cref{lemma:noinc:fourier_Minkowksi}.
We note that if an initial data set satisfies the no incoming radiation condition to some order, then one can immediately deduce better conformal regularity towards $\scrip$, which in turn allows to prove better decay towards $i^+$, cf.~\cite{gajic_relation_2022}. See also \cite[Section~1.5.3]{kehrberger_case_2024}.

\subsubsection{Detecting a smooth null infinity from Cauchy data}
As already discussed in \cref{sec:intro:peeling}, the no incoming radiation is historically strongly tied to the idea of peeling (conformal regularity) in \emph{Minkowski} space.
The question of finding conditions on Cauchy data to not give rise to logarithms towards future null infinity for $\Box_\eta$ has been entertained for a long time in the literature, see, for instance, \cite{ValienteKroon2012,Friedrich2013,gasperin_asymptotics_2024-1,fuentealba_logarithmic_2024}.
In \cref{cor:noinc}, we characterise Cauchy data with integer power expansions in $1/r$ that do not give rise to logarithmic terms towards $\scrip$ and thus recover \cite[Proposition 1]{gasperin_asymptotics_2024-1}.
The subset of data that do not give logarithms towards $\scrip$ up to a fixed order strictly contains the subset of data with no incoming radiation to that order in Minkowski.

We note, however, that our algorithms to characterise solution without incoming radiation are also valid beyond Minkowski (say, in Schwarzschild), where no incoming radiation, in fact, implies the failure of peeling \cite{christodoulou_global_2002,kehrberger_case_2022}.


\subsubsection{Scattering for the Einstein vacuum equations in harmonic gauge}\label{sec:intro:EVE}
We stated \cref{thm:intro:scat,thm:intro:prop_polyhom} for scalar equations, but they apply just as well for \textit{systems} of scalar equations, where each term can have a separate decay rate $\vec{a}_i$.
In particular, we  generalise our results to the Einstein vacuum equations for $g=\eta+h$ in harmonic gauge in \cref{sec:EVE} by treating the equations as a system of scalar equations for the components of the metric perturbation $h$.
We encounter the following  extra difficulties:
In the region $\D\cap\{t\leq 0\}$, we address the usual difficulty posed by the "logarithmic divergence of the radiation field" associated to the weak null condition by our scattering results beyond finite energy but with polyhomogeneous data from \cref{sec:intro:scat:phg}.
Another technical difficulty in constructing scattering solutions in $\D\cap\{t\leq0\}$ lies in posing characteristic data on some null cone $\incone$---in this work, we avoid this difficulty by assuming knowledge of the existence of a solution in a slab (see the darker shaded region in \cref{fig:intro:main}), for which we then show that it contains null cones.  Note that even if this existing solution is trivial, we can still generate nontrivial solutions via the scattering data along $\scrim$.

Concerning the future region, we recall from \cite{masaood_scattering_2022-1,kehrberger_case_2024-1} that scattering solutions arising from data at $\scrim$ will in general not be compatible with the decay class imposed in \cite{lindblad_global_2005,hintz_stability_2020}; instead, they will decay towards $I^0$ like $h\sim r^{-1}$(though we also construct examples of scattering solutions that decay faster towards $I^0$).
The existence of solutions  for this decay class is well known since \cite{bieri_extension_2010} and also within harmonic gauge as shown in \cite{ionescu_einstein-klein-gordon_2022}.

Here, we prove existence, conormality and polyhomogeneity of the solution in $\D\cap\{t>0\}$ under the assumption of yet weaker decay, namely $h\sim r^{-1/2}$.
In contrast to $\{t\leq 0\}$, this does not quite follow from \cref{thm:intro:scat} as the system now has to be coupled to a transport equation. 
This, however, is well understood since \cite{lindblad_global_2005}.The main novelty is that we need to find an ansatz that captures the stronger-than-Schwarzschildean divergence of the light cones.
In effect, we show a high regularity result on the exterior stability of Minkowski space with data similar to  those of \cite{bieri_extension_2010} in \emph{harmonic gauge}.
For the main result, see \cref{thm:EVE:main_scattering}.
\subsection{Structure of the remainder of this work}\label{sec:intro:structure}
    The remainder of this work is structured as follows:
    
\textbf{Remainder of Part I:}  

In \cref{sec:intro:example}, we showcase some of the concrete mathematical ideas of this paper using the simple model problem \cref{eq:intro:wave:potential}; this section can also be read as an extended introduction.
   
    The paper proper starts in \cref{sec:notation}, where we introduce the somewhat loaded notation we employ, along with the relevant mathematical concepts.
    
    In \cref{sec:ODE_lemmas}, we then prove a variety of basic ODE lemmata that are applied throughout the paper.
   
\textbf{Part II:}  

In \cref{sec:en}, we first prove robust energy estimates for $\Box_\eta\phi=f$. We then introduce the class of admissible quasilinear perturbations and prove that these estimates can still be proved in this larger class. In this section, we will restrict to quasilinear perturbations that do not affect the light cones.
    
    In \cref{sec:scat:scat}, we finally prove scattering for these equations \cref{eq:intro:wavegeneral}, namely \cref{thm:intro:scat}. We also prove the generalised scattering results mentioned in \cref{sec:intro:scat:nic,sec:intro:scat:phg}.
    
    In \cref{sec:prop}, we use the ODE lemmata of \cref{sec:ODE_lemmas} to prove the propagation of polyhomogeneity from $\scrim$ to $\scrip$ for the scalar Minkowskian wave equation $\Box_\eta\phi=f$.
   
\textbf{Part III:}  

In \cref{sec:app}, we then use this statement to prove the analogous statement for the larger class of wave equations \cref{eq:intro:wavegeneral}, namely \cref{thm:intro:prop_polyhom}. We also apply our results to the examples corresponding to the corollaries of \cref{sec:intro:applications}.
   
    In \cref{sec:Sch}, we discuss in detail the particular application of our result to the setting of \cref{eq:intro:waveSchwarzschild} as well as to that of linearised gravity around Schwarzschild. In particular, we will prove the statements of \cref{sec:intro:relev}.
    
    In \cref{sec:sharp}, we prove that peeling, even for fixed angular frequencies, is extremely specific to even spacetime dimensions, and formulate the precise version of \cref{conj:intro}, namely \cref{conj:even:peeling}, concerning the peeling property in Minkowski.
    This is also the only section of the paper where we compute asymptotics for non-short-range perturbations; in particular, we compute asymptotics for the scale invariant wave equation (with a $1/r^2$-potential).

    In \cref{section:no incoming radiation Cauchy}, we discuss versions of the no incoming radiation condition on a Cauchy hypersurface.

\textbf{Part IV:}   

In \cref{app:sec:current_computations}, we reprove the energy estimates of \cref{sec:en} within a geometric framework, allowing us to remove the previously made restriction to quasilinear perturbations that do not affect the light cones.
    
    Finally, in \cref{sec:EVE}, we then use these results to study scattering for the Einstein vacuum equations in harmonic gauge.

\subsection{Acknowledgements}
We would like to thank Mihalis Dafermos, Dejan Gajic, Gustav Holzegel, Jonathan Luk and Volker Schlue for helpful feedback on earlier versions of this manuscript. We also thank Hans Lindblad for helpful conversations regarding the scattering problem for the Einstein vacuum equations.

The first author was partially funded by the grant EPSRC 2436109.
    
 \section{Discussion of a toy model problem}\label{sec:intro:example}
Given the length of the paper, we here give a  brief discussion of a toy model problem, namely the linear wave equation on Minkowski with a $1/r^3$ potential.
 \begin{equation}\label{eq:intro:wave:example}
     \Box_\eta \phi= -\frac{M\phi}{r^3} \iff \pu\pv\psi=\frac{\Dl\psi}{r^2}+\frac{M\psi}{r^3}.
 \end{equation}
In what follows, we will commute with scaling, which generates an error term, so we will in fact have to consider \cref{eq:intro:wave:example} with inhomogeneity, 
 \begin{equation}\label{eq:intro:wave:examplef}
 \Box_\eta \phi=-\frac{M\phi}{r^3}-f \iff \pu\pv\psi=\frac{\Dl\psi}{r^2}+\frac{M\psi}{r^3}+rf.
 \end{equation}
 The data that we pose for \cref{eq:intro:wave:example} will consist of no incoming radiation, and $\psi^{\incone}= r^{-p}+\Hb^{p+\epsilon}(\incone)$ for some $p\in\mathbb R_{>-1/2}$ and some $\epsilon>0$ arbitrary; for details how to treat the case of incoming radiation, see already \cref{sec:prop:II}. 
Since we consider the case of no incoming radiation, we will, in particular, also demand that the decay of $f$ is preserved by the vector field $r\pv$.

 We shall now give a simple sketch how to prove \hyperlink{step1}{Steps 1)} to \hyperlink{step3}{3)} from \cref{sec:intro:motivation}. 
 Since energy estimates and propagation of polyhomogeneity are already available in going from $I^0$ to $\scrip$ (\cite{hintz_stability_2020}), we shall mostly restrict to the region $\D^-=\D\cap\{-\delta u\geq v\}$ for some $\delta>0$. This ensures that $|u|\sim |t| \sim r$ in $\D^-$.

 \textit{\hyperlink{step1}{Step 1)}} \textbf{Energy estimate and scattering construction.} 
 We begin by proving a basic energy estimate and scattering for \cref{eq:intro:wave:examplef}. In the spirit of a limiting argument, we first truncate $\psi^{\incone}$ and $f$ so that they are supported away from $\scrim$.
 Multiplying the second of \cref{eq:intro:wave:examplef} with $|t|^{2a+1}\partial_t\psi$, we then obtain an estimate of the form (integration over $S^2$ is implied)
 \begin{equation}\label{eq:intro:example:energyexplicit}
 \int_{\D^-} |t|^{2a} \left((\pv\psi)^2+(\pu\psi)^2+r^{-2}(\sl\psi)^2\right) \dd u \dd v\lesssim \int_{\D^-} |t|^{2a+2}(rf)^2 \dd u \dd v+\int_{\incone^-} |t|^{2a+1} \left((\pu\psi)^2+r^{-2}(\sl\psi)^2 \right)\dd u.
 \end{equation}
After also controlling zeroth order terms in standard fashion (e.g.~by multiplying the first of \cref{eq:intro:wave:examplef} with $r^2|t|^{2a+1}\partial_t\phi$), this energy estimate reads
 \begin{equation}\label{eq:intro:example:energyestimates}
 \norm{\psi}_{\rho_-^{a-0.5}\rho_0^{a}\Hb^{;1}(\D^-)}\lesssim \norm{rf}_{\rho_-^{a+1.5}\rho_0^{a+2}\Hb^{;0}(\D^-)}+\norm{\psi}_{\rho_-^{a}\Hb^{;1}(\incone)}.
 \end{equation}
 By suitable commutations with angular derivatives, Lorentz boosts, scaling and $\partial_t$, we obtain higher order estimates with arbitrary commutations from the set $\Vbt:=\{r\pv,r\pu,\sl\}$, where we recall that it is the no incoming radiation condition that allows us to gain $r$-decay from $\pv$ commutations.

We note that this estimate is not sharp. In particular, the lack of control near $\scrim$ can easily be improved (see \cref{prop:en:main}). This is, however, not necessary for this toy problem. 
By passing to the limit and performing standard limit arguments, we can now remove the support assumption on $\psi^{\incone}$ and $f$, and prove the analogue of \cref{thm:intro:scat} for this setting. 
We will from now on stop keeping track of regularity.

\textit{\hyperlink{step2a}{Step 2a)} (i)} \textbf{Propagation of polyhomogeneity up until $I^0$.}
The above estimate, in particular, shows that solutions to \cref{eq:intro:wave:example} which initially decay like $r^{-a}$ will also decay like $r^{-a}$ towards $I^0$ (in an $L^2$-sense). 
We now want to prove that $\psi$ has an expansion towards $I^0$. The idea is to \textit{chop off} the leading order term in the expansion of $\psi$ by commuting with the scaling vector field $S=u\pu+v\pv$. 

Let's first see what this does at the level of $\psi^{\incone}$, ignoring for now the error term (which cannot be chopped off):
Since $v\pv\psi|_{\incone}$ decays one power faster along $\incone$ than $\psi^{\incone}$ (as a consequence of the no incoming radiation condition), we obtain that $S\psi|_{\incone}=u\pu r^{-p}+\O(r^{-p-1})=-p r^{-p}+\O(r^{-p-1})$, and we thus get that $(S+p)\psi|_{\incone}=\O(r^{-p-1})\in \Hb^{p+1-}(\incone)$.
We can then keep peeling off higher order terms by further commutations; in this way, we  readily deduce that $(S+p+n)\cdots(S+p)\psi|_{\incone}\in\Hb^{p+n+1-}(\incone)$.

We next prove estimates for this commuted, better-decaying quantity: Observe that if $\psi$ solves \cref{eq:intro:wave:example}, then 
\begin{equation}\label{eq:intro:wavescalingcommuted}
\pu\pv (S+p)\psi=\frac{\Dl (S+p)\psi}{r^2} +\frac{M(S+p)r^{-1}\psi}{r^2}.
\end{equation} 
Using the already obtained membership $\psi\in\Hb^{p-1/2-,p-}(\D^-)$, we obtain, by definition of the $\Hb$ spaces, that $r^{-2}(S+p)r^{-1}\psi \in \rho^3_-\rho^3_0\Hb^{p-1/2-,p-}(\D^-)$.
We can thus apply the energy estimate  \cref{eq:intro:example:energyestimates} to \cref{eq:intro:wavescalingcommuted}, treating the error term $\frac{M(S+p)r^{-1}\psi}{r^2}$ as inhomogeneity, to find that $(S+p)\psi\in \Hb^{p+1/2-,p+1-}(\D^-)$. 
Proceeding inductively, we find that
\begin{equation}
(S+p+n)\cdots(S+p)\psi\in \Hb^{p+n+1/2-,p+1+n-}(\D^-).
\end{equation}

This statement, disregarding losses of regularity, is in fact already equivalent to polyhomogeneity near $I^0$: To see this we simply integrate along the integral curves of $S$, which, in coordinates $\rho_-,\,\rho_0$, is  given by $\rho_0\partial_0$, $n+1$ times from the interior of the spacetime (the boundary terms can be averaged away with a cut-off function), cf.~\cref{fig:scaling}.
By simple ODE analysis (cf. \cref{prop:ODE:ODE}), each of these integrations contributes one term in the expansion of $\psi$ towards $I^0$. 
This proves polyhomogeneity towards $I^0$.

We note that the same idea of  proving an expansion by proving faster decay for commuted quantities of the type $(S+p)\psi$ has been applied in \cite{hintz_stability_2020} in order to prove polyhomogeneity \textit{along} spacelike infinity. In a different, but somewhat related context, a similar idea has been applied in \cite{kehrberger_case_2022,kehrberger_case_2022-1} in order to prove estimates going from past timelike infinity to $\scrim$ (involving commutations with $u\partial_t+p$ rather than commutations with $S+p$).

\textit{\hyperlink{step2a}{Step 2a)} (ii)} \textbf{Polyhomogeneity towards $\scrip$.}
This is already contained in \cite{hintz_stability_2020}, cf.~the discussion around~\eqref{eq:intro:box_near_corner}.

\textit{\hyperlink{step2b}{Step 2b)}}\textbf{ Treatment of the error term.} 
We have at this point proved the first part of \cref{thm:intro:prop_polyhom}. In order to also treat the second part, we assume that the data are given by some error term $\psi^{\incone}\in \Hb^{p+\epsilon}(\incone)$ for some $\epsilon>0$. 
We opt to give a lazy proof of concept here:
The energy estimate (extended to all of $\D$), when also including the flux on ingoing cones in the LHS of \eqref{eq:intro:example:energyexplicit}, then shows that $\psi\lesssim |u|^{-p-\epsilon+}$ everywhere. 
Now, inserting this latter estimate into \eqref{eq:intro:wave:example} and integrating a suitable number of times from $\scrim$ gives an expansion in powers of $1/r$ towards $\scrip$, up to error term.

\textit{\hyperlink{step3}{Step 3)}}\textbf{ Computations of the coefficients.} 
We have at this point understood that the solution $\psi$ has a polyhomogeneous expansion up to some error term. The final step consists of computing the coefficients of these expansion. 
This proceeds as described in \hyperlink{step3}{Step 3)} in \cref{sec:intro:motivation}; that is, we let $\phi^{(0)}$ solve $\Box_\eta\phi^{(0)}$ with no incoming radiation and polyhomogeneous data, compute $\phi^{(0)}$ $\ell$-mode by $\ell$-mode using \eqref{eq:intro:conslawf} with $f=0$.
Since we have already proved the polyhomogeneity statements, we no longer need to worry about the summability of the arising expressions in $\ell$ at this point. 

We then let $\phi^{(1)}$ be the solution to $\Box_\eta\phi=-\frac{M\phi^{(0)}}{r^3}$ with trivial data, and compute it, $\ell$ by $\ell$, using again \eqref{eq:intro:conslawf}.
This gives rise to the first conformally irregular term predicted in \cref{cor:intro:polyhom_potential}. See also the proof of \cref{lem:app:potential} for more details on this computation.
 
	 \begin{figure}[htbp]
 \centering
\includegraphics[width=0.2\textwidth]{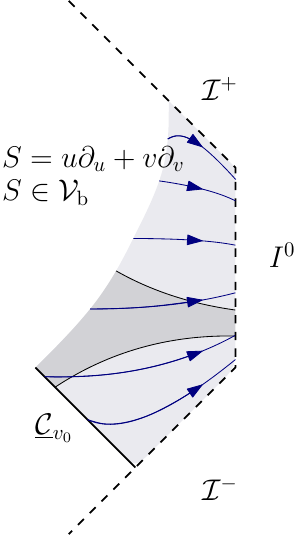}
\caption{Depicted are the integral curves of the scaling vector field. The shaded regions are $\D^+,\D^-$, see~\cref{fig:D-a}.}
\label{fig:scaling}
\end{figure}

  \section{Definitions, Preliminaries and Notation}\label{sec:notation}
	
	In this section, we introduce the analytic (\cref{sec:notation:analytic}) and geometric (\cref{sec:notation:geometric}) framework employed in this paper. 
For the convenience of the reader, we also proved a small glossary of the  symbols used in the paper at the end of the section.
	
	\subsection{Analytic preliminaries: Conormality and Polyhomogeneity}\label{sec:notation:analytic}
	
	In this subsection, we introduce the function spaces and other analytic tools used in the paper. We restrict ourselves to the minimum use of sophisticated spaces that suffice for our problem; for a more general discussion see \cite{grieser_basics_2001, hintz_stability_2020} and references therein.
	
	Let's fix a manifold with corners $X=[0,1)_{x_1}\times...\times[0,1)_{x_n}\times Y$ for some smooth manifold $Y$.
    We call $x_1$, as well as any other function $\bar{x}_1\in\C^{\infty}(X)$ satisfying $\{\bar{x}_1=0\}=\{x_1=0\}$ and $c^{-1}<\partial_{x_1}\bar{x}_1<c$ for $c>0$, a boundary defining function of the $\{x_1=0\}$-boundary.
    Next, we define the vector fields with respect to which we measure smoothness.
	\begin{defi}[$\b$-vector fields]\label{def:notation:b_vector_fields}
		Let 		
		\begin{equation}
			\mathcal{V}=\{x_i \partial_{x_i},Y_i\},
		\end{equation}
		where $Y_i$ are smooth vector fields on $Y$ spanning the tangent space at each point. Furthermore, we define $\Diff_{\b}^1(X)$ to consist of finite sums of vector fields from $\mathcal{V}$ with $\C^\infty(X)$ coefficients and multiplications by $\C^\infty(X)$ functions.
        Finally, let $\Diff^k_{\b}(X)$ denote the space of  finite sums of up to $k$-fold products of elements in $\Diff^1_{\b}(X)$.
	\end{defi}

	In application, we will mostly take $n\leq3$, and $Y=S^2$ (or $Y=S^2\times (0,1)$), with $Y_i$ then denoting the usual spherical derivatives.

    \begin{defi}[Multi-index notation]\label{defi:notation:multiindex}
        Let $\Gamma=\{\Gamma_1,...,\Gamma_m\}$ be a finite set of vector fields and $\norm{\cdot}_{H(X)}$ a norm on $X$.  
        For $\alpha=(\alpha_1,\alpha_2,...,\alpha_m)\in\N^m$ denoting multi-indices, for $|\alpha|=\sum_i\alpha_i$, and for $k\in\mathbb N$, we write
        \begin{equation}
            \norm{\Gamma^kf}_{H(X)}:=\sum_{\abs{\alpha}\leq k}\norm{{\Gamma^{\alpha}}f}_{H(X)}=\sum_{\abs{\alpha}\leq k}\norm{
            \prod_{i}{\Gamma_i^{\alpha_i}}f}_{H(X)}.
        \end{equation}
    \end{defi}

	\begin{defi}[The $\Hb$-norm]\label{def:notation:Hb_spaces}
		Given a measure $\dd y$ on $Y$, we define a naturally weighted $L^2$-norm on $X$, along with higher order variants (Sobolev spaces) and higher order variants with extra weights (weighted Sobolev spaces):
		\begin{nalign}
			\norm{f}_{L^2_{\b}(X)}^2&:=\int f^2 \frac{\dd x_1}{x_1}...\frac{\dd x_n}{x_n}\dd y,\\
			\norm{f}_{\Hb^{;k}(X)}&:=\norm{\mathcal{V}^kf}_{L^2_{\b}(X)},\\	
			\norm{f}_{\Hb^{\vec{a};k}(X)}:&=\norm{f}_{\Hb^{a_1,...,a_n;k}(X)}:=\norm{f{\prod_i x_i^{-a_i}}}_{\Hb^{;k}(X)}.
		\end{nalign}
        We define the corresponding spaces $L^2_{\b}(X)$, $\Hb^{;k}(X)$ and $\Hb^{\vec{a};k}(X)$ to be the completions of smooth, compactly supported functions on $X$ under the respective norms.
        We will occasionally also write $x_1^{a_1}...x_n^{a_n}\Hb^{;k}=\Hb^{a_1,...,a_n;k}$.
		
		Furthermore, we  use the notation $\Hb^{a-;k}([0,1))=\cap_{\epsilon>0}\Hb^{a-\epsilon;k}([0,1))$ as well as 
        $\Hb^{a+;k}([0,1))=\cup_{\epsilon>0}\Hb^{a+\epsilon;k}([0,1))$, and extend this notation appropriately to each component for manifolds with multiple boundaries.
        The spaces $\Hb^{a-;k}$ have no associated norms, but we will write $\norm{f}_{\Hb^{a-;k}}\lesssim A$ when for all $\epsilon>0$ there exists a constant $c(\epsilon)$ (possibly divergent as $\epsilon\to0$) such that $\norm{f}_{\Hb^{a-\epsilon;k}}\leq c(\epsilon) A$; similarly for $A\lesssim\norm{f}_{\Hb^{a+;k}}$.
        
        We  also write $\Hb^{-\infty;k}([0,1))=\cup_{a>0}\Hb^{-a;k}([0,1))$, and $\Hb^{\infty;k}([0,1))=\cap_{a>0}\Hb^{a;k}([0,1))$; and use analogous notation in the case of multiple boundaries.
 
		Finally, we say $f$ is conormal to order $k$ if $f$ is in $\Hb^{;k}$; if $k$ is infinite we will simply say $f$ is conormal.
	\end{defi}
 \begin{rem}
     In the course of the paper, we will encounter cutoff spaces in the context of limiting arguments. We will then, say, consider $X_{\epsilon}=(\epsilon,1)_{x_1}\times Y\subset X$, but still work with the measure inherited from the embedding $X_{\epsilon}\subset X$. For instance, we will write
		$$\norm{f}_{L^2_{\b}(X_{\epsilon})}^2:=\int_{X_{\epsilon}} f^2 \frac{\dd x_1}{x_1}\dd y.
		$$
    Similarly, we extend the weight on $X$ for hypersurfaces.
    For instance, for the hypersurface $\Sigma=\{x_1=x'\}$ for $x'>0$
    \begin{equation}
        \norm{f}_{H^{\vec{a};0}_b(\Sigma)}^2=\int f^2x_1^{-2a_1}...x_n^{-2a_n}\frac{\dd x_2}{x_2}...\frac{\dd x_n}{x_n}.
    \end{equation}
 \end{rem}

 \begin{rem}[Finite boundaries]\label{rem:not:finite_boundary}
     We will encounter manifolds with boundaries, where the boundaries are associated to some hypersurface inside $X$ (e.g. on the initial data slice in $\D$ below).
        For instance, considering $\{x_1\leq1/2\}\subset X$ has an extra boundary at $\{x_1=1/2\}$.
        To differentiate these from $\{x_1=0\}$-boundaries, we call the former a \emph{finite boundary} and the latter a \emph{$\b$-boundary}, respectively.
        Importantly, we will \textit{not} use a weighted norm towards finite boundaries. That is to say, near the finite boundary $x_1=1/2$, we take the norm with respect to $\dd x_1$ rather than $\dd x_1/(x_1-1/2)$. 
    This will always be clear from the context.
 \end{rem}

	\begin{rem}\label{rem:not:Hb_Linfty}
		Away from the boundary, the vector fields span the tangent space of $X$.
		Indeed, the $L^2_{\b}$-and the $\Hb$-spaces agree with the usual $L^2$-and $H^{;k}$-spaces on compact subsets of $\mathrm{int}(X)$.
		The normalisation for $L^2_{\b}$ is motivated by the observation that, for $\mathrm{dim}(X)=m$ and $k>\lceil m/2\rceil$,
		\begin{equation}\label{eq:not:Sobolev}
			\Hb^{\vec{a};k}(X)\subset x_{1}^{a_1}...x_n^{a_n}C^{k-\lceil m/2\rceil}(X)\subset \Hb^{\vec{a}-;k-\lceil m/2\rceil}(X).
		\end{equation}
	\end{rem}
	
	\begin{rem}\label{remark:notation:equivalence_of_norms}
		We note that the $\Hb^{;k}$-spaces only depend on the smooth structure of $X$.
		For instance, when $X=[0,1)_x\times Y$, for any other boundary defining function $\bar{x}(x):[0,1)\to[0,1)$, the spaces defined with respect to $\bar{x}$ are the same as $\Hb^{;k}$, and the norms are equivalent.
	\end{rem}

	Having introduced function spaces with respect to which we measure regularity (conormality) of functions, we now move on to introduce \textit{polyhomogeneity}. Loosely speaking, a function is polyhomogeneous if it has an expansion into $x_i^z\log(x_i)^k$.
In order to precisely capture such expansions, we first introduce index sets:
	\begin{defi}[Index sets]\label{def:notation:index_sets}
		A subset of $\mathcal{E}\subset\R\times\N$ is called an index set if it satisfies the following conditions
  \begin{enumerate}[label=\textbf{$\mathcal{E}$.\arabic*}]
			\item $(z,k)\in\mathcal{E}$ and $k\geq1$ implies $(z,k-1)\in\mathcal{E}$\label{item1:notation:index_sets},
			\item $(z,k)\in\mathcal{E}$ and $k\geq0$ implies $(z+1,k)\in\mathcal{E}$\label{item2:notation:index_sets},
			\item $\mathcal{E}_{\leq c}:=\{(z,k)\in\mathcal{E}| z\leq c\}$ is finite for all $c\in\R$\label{item3:notation:index_sets}.
		\end{enumerate}
		We define $\min(\mathcal{E})=(z,k)\in\mathcal{E}$  such that $\forall(z',k')\in\mathcal{E}, z\leq z' \text{ and } z=z'\Rightarrow k\geq k'$.
        For $\bullet\in\{\geq,>\}$, we use the notation $(z,k)\bullet (z',k')$ if $z>z'$ or $z=z'$ and $k'\bullet k$.
        We also write $(z,k)>z'$ whenever $z>z'$. In particular, $\min\E>z'$ means that $z>z'$ for all $(z,k)\in\E$.
	\end{defi}

	\begin{defi}[Polyhomogeneity]\label{def:notation:polyhom}
		Let $X=[0,1)_{x_1}\times Y$ be a manifold with boundary.
		Given $l\in\mathbb N$ and an index set $\mathcal{E}$, we define the corresponding \textbf{polyhomogeneous space} $\A{phg}^{\mathcal{E};l}(X)$ as follows: For $u\in x_1^{-\infty}\Hb^{;l}(X)$, we say $u\in\A{phg}^{\mathcal{E};l}(X)$ if there exist $v_{z,k}\in \Hb^{;l}(Y)$ for $(z,k)\in\mathcal{E}$ such that for all $c\in \R$ 
		\begin{equation}\label{eq:notation:polyhom_def_1}
			\begin{gathered}
				u-\sum_{(z,k)\in\mathcal{E}_{\leq c}} v_{z,k}x_1^{z}\log^kx_1\in x_1^c \Hb^{;l}(X).
			\end{gathered}
		\end{equation}
		
		At a corner, we define \textbf{mixed $\b-$polyhomogeneous spaces} as follows: Let $X=[0,1)_{x_1}\times[0,1)_{x_2}\times Y$ be a manifold with corners, $a_2\in\R$ and $\mathcal{E}_1$ an index set.
		For $u\in x_1^{-\infty}x_2^{-\infty}\Hb^{;l}(X)$, we say $u\in \A{phg,b}^{\mathcal{E}_1,a_2;l}(X)$ if there exists $v_{z,k}\in x_2^{a_2}\Hb^{;l}([0,1)_{x_2}\times Y)$ such that
		\begin{equation}
			u-\sum_{(z,k)\in\mathcal{E}_{\leq c}} v_{z,k}x_1^{z}\log^kx_1\in x_1^c x_2^{a_2}\Hb^{;l}(X).
		\end{equation}
		
		We define polyhomogeneous spaces at the corner for a manifold with corners $X=[0,1)_{x_1}\times[0,1)_{x_2}\times Y$.
		For $u\in x_1^{-\infty}x_2^{-\infty}\Hb^{;l}(X)$ we say $u\in\A{phg}^{\mathcal{E}_1,\mathcal{E}_2;l}(X)$ if there exists $v_{z,k}\in\A{phg}^{\mathcal{E}_2;l}([0,1)_{x_2}\times Y)$ such that
		\begin{equation}
			u-\sum_{(z,k)\in(\mathcal{E}_1)_{\leq c}} v_{z,k}x_1^{z}\log^kx_1\in \A{b,phg}^{c,\mathcal{E}_2;l}(X).
		\end{equation}
		
		We will mainly use the polyhomogeneous spaces with ${l}=\infty$, and we omit the superscript in this case, e.g. $\A{phg}^{\mathcal{E}}([0,1)):=\A{phg}^{\mathcal{E};\infty}([0,1))$.
		For a manifold with multiple boundaries, $B_i$, and corners of codimension 2, we define $\mathcal{A}_{\vec{\beta}}^{\vec{\alpha};l}$ for $\alpha_i\in\{\mathcal{E}_i,b_i\},\beta_i\in\{\mathrm{phg,b}\}$ if around each corner $B_i\cap B_j$  we have $\chi u\in\A{\beta_i,\beta_j}^{\alpha_i,\alpha_j;l}$, where $\chi$ is a smooth cutoff function on $X$ localising around $B_i\cap B_j$.

  A tacit convention used throughout the paper is that whenever we specify a function $f\in\A{phg}^{\E}(X)$, then $\E$ will be assumed to be an index set.
	\end{defi}

	\begin{rem}[Equivalent formulation]\label{rem:notation:equivalent_polyhoms}
     Note that we may give an alternative, more geometric characterisation\footnote{The equivalence follows from multiple applications of \cref{item:ode:1d}, see also (2.33) and (2.34) in \cite{hintz_stability_2020}.} of a polyhomogeneous function $u\in\A{phg}^{\mathcal{E}}([0,1)_{x}\times Y)$ (at infinite regularity) as
		\begin{equation}\label{eq:notation:polyhom_def_2}
			\Big(\prod_{(z,k)\in\mathcal{E}_{\leq c}}(x\partial_x+z) \Big)u\in x^c \Hb^{;\infty}([0,1)_{x}\times Y).
		\end{equation}
  We will make use of this equivalent formulation throughout the paper.
		From \cref{eq:notation:polyhom_def_2}, it is easy to see that the definition of $\mathcal{E}$ only depends on the smooth structure of $[0,1)_{x}\times Y$, that is: Given any other boundary defining function $\bar{x}$, the index sets with respect to $x$ and $\bar{x}$ coincide. 
		We will make use of this freedom and frequently change our coordinates to simplify computations, knowing that this will produce the same index set.
	\end{rem}
	
	\begin{rem}[Finite regularity]
		For $l<\infty$, the natural generalisation of the characterisation of polyhomogeneity \cref{eq:notation:polyhom_def_2} no longer coincides with \cref{def:notation:polyhom}. 
		The reason is that when we integrate \cref{eq:notation:polyhom_def_2} with right hand side in $x^c\Hb^{;l}$, the form we get in \cref{eq:notation:polyhom_def_1} has an error term in $x^c\Hb^{;l}$.
		In contrast,  differentiating \cref{eq:notation:polyhom_def_1} yields error terms $x^c \Hb^{;l-l'}$ in \cref{eq:notation:polyhom_def_2} for some $l'>0$.
Throughout the paper, we almost exclusively work with polyhomogeneous spaces at infinite regularity; we only provided the definitions at finite regularity to make clear that all our proofs still work at finite regularity.
\end{rem}

    To perform computations with polyhomogeneous functions, it is useful to introduce operations on the index sets as well.
	\begin{defi}[Index set operations]\label{def:notation:operations_index_sets}
        For $\mathcal{E}\subset\R\times\N$ satisfying \cref{item3:notation:index_sets}, we introduce
		\begin{equation}\overline{\mathcal{E}}=\bigcap_{\mathcal{E}\subset\mathcal{E}'}\mathcal{E}',\qquad \mathcal{E}' \text{ are index sets}.
		\end{equation}
		For $(z,k)\in\R\times\N$, we write $\overline{(z,k)}:=\overline{\{(z,k)\}}=\{(z',k')|z'-z\in\mathbb N, k'\in\mathbb N_{\leq k}\}$.
 
		For $\E,\mathcal{E}_1,\mathcal{E}_2\subset\R\times\N$, $z\in\R$ and $k\in\N$, we define 
		\begin{nalign}\label{index set operations}
			&\mathcal{E}_1\cupdex\mathcal{E}_2:=\overline{\{(z,k)| \exists(z,k_i)\in\mathcal{E}_i,\, k_1+k_2+1\geq k\}\cup\mathcal{E}_1\cup\mathcal{E}_2},\\
			&\mathcal{E}_1+\mathcal{E}_2:=\{(z,k)|\exists (z_i,k_i)\in\mathcal{E}_i,\, z_1+z_2=z,\,k_1+k_2=k\},\\
			&\mathcal{E}_{\bullet c}:=\{(z,k)\in\mathcal{E}:z\bullet c\},\quad \bullet\in\{\geq,\leq,<,>,=\},\\
            &\E\pm\tilde{z}:=\{(z\pm\tilde{z},k):(z,k)\in\E \},\qquad \E\cupdex z:=\E\cupdex\{(z,0)\}.
		\end{nalign}
	\end{defi}

    \begin{rem}
        Note that, for $\mathcal{E}_1,\E_2$  index sets, the resulting set $\E_1+\E_2$ is an index set.
        Note also that $\cupdex$ is both associative and commutative on index sets, but not in general: for $p,q\in\N_{>0}$
        \begin{nalign}
            0\cupdex\mindex{q}=\mindex{0},\qquad \mindex{0}\cupdex\mindex{q}=\mindex{0}\cup\overline{(q,1)},\\\mindex{p}\cupdex(0\cupdex\mindex{q})=\mindex{0}\cup\overline{(p,1)},\quad
            (\mindex{p}\cupdex 0)\cupdex\mindex{q}=\mindex{0}\cup\overline{(q,1)}.
        \end{nalign}
    \end{rem}
	
    To already highlight the importance of these notations, we provide two examples:
    \begin{itemize}
        \item For $f_i\in\A{phg}^{\E^i}([0,1)_x)$ with $i\in\{1,2\}$, we have $f_1\cdot f_2\in\A{phg}^{\E^1+\E^2}([0,1)_x)$.
        \item For $f=x\partial_x F\in\A{phg}^{\E}([0,1)_x)$ and $F(1/2)=0$, we have $F\in\A{phg}^{\E\cupdex 0}([0,1)_x)$.
    \end{itemize}
 
	\subsection{Geometric preliminaries: The compactification \texorpdfstring{$\D$}{} and the boundaries \texorpdfstring{$\mathcal{I}^\pm$ and $I^0$}{scriminus, i0 and scriplus}}\label{sec:notation:geometric}
	The majority of this paper takes place in Minkowski spacetime $(\mathbb{R}^{3,1},\eta)$ with the usual coordinates $t\in\mathbb R, \, x\in\mathbb R^3$, and metric $\eta=-\dd t^2+\dd x^2$.
	Unless otherwise stated, we use $\omega$ to denote coordinates on $S^2$, and we introduce the standard null coordinates
	\begin{equation}
		u=\frac{t-\abs{x}}{2},\qquad  v=\frac{t+\abs{x}}{2}.	
	\end{equation}

    The entirety of the paper is focused on a region close to spacelike infinity ($I_0$), and we set 
	\begin{nalign}\label{eq:notation:regions}
		&\Dopen_{u_0,v_0}^{u_\infty,v_\infty}=\{u\in(u_\infty,u_0],v\in[v_0,v_\infty)\}\\
		&\C_{u_1}^{v_0,v_\infty}=\{u=u_1,v\in[v_0,v_\infty)\},\quad \Cbar_{v_1}^{u_0,u_\infty}=\{v=v_1,u\in(u_\infty,u_0]\}.
	\end{nalign}
	We will always take $-u_0,v_0>0$.
	In most cases, we will set $u_\infty=-\infty,v_\infty=\infty$ and, in this case, omit these from the notation.
	When clear from context, we will also drop $u_0,v_0$ from the notation; e.g.~we write $\Dopen=\Dopen^{-\infty,\infty}_{u_0,v_0}$. Moreover, we write $\Dopen^{u_1}_{u_0}=\Dopen^{u^1,\infty}_{u_0,v_0}$, and $\Dopen^{v_1}_{v_0}=\Dopen^{-\infty,v_1}_{u_0,v_0}$; it will always be clear from our use of $u$ vs.~$v$-variable which of these we will mean. 
 
	Let us choose $\delta\in(0,1)$.
    We introduce the regions $X^-=X\cap \{-\delta u\geq v\},\, X^+=X\cap \{-\delta u\leq 2v\}$, where $X$ stands for either a hypersurface or a region of spacetime. See \cref{fig:D-a} for a Penrose diagrammatic depiction of the regions $\Dopen^-$ and $\Dopen^+$.
	
		   \begin{figure}[htpb]
\centering
\begin{subfigure}{0.45\textwidth}
\centering
    \includegraphics[width=120pt]{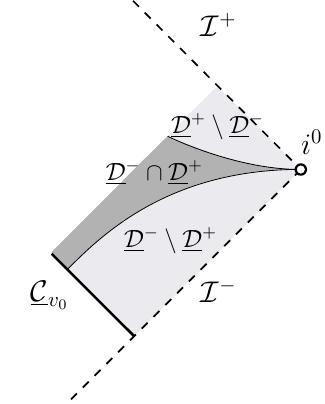}
    \caption{Penrose diagram.}
    \label{fig:D-a}
\end{subfigure}
\hspace{0.05\textwidth}
\begin{subfigure}{0.45\textwidth}
\centering
 \includegraphics[width=95pt]{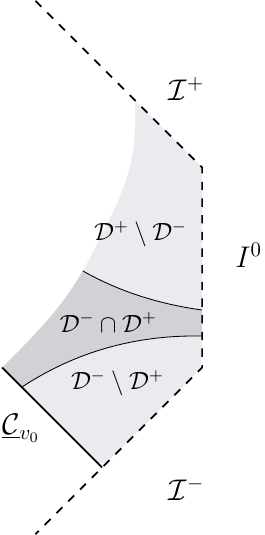}
    \caption{$\b$-compactification.}
    \label{fig:D-b}
\end{subfigure}
\caption{Depiction of the regions $\Dopen^{\pm}$ in a Penrose diagram and with respect to the compactification $\D$.}\label{fig:D}
\end{figure}
	We now introduce the compactifications of Minkowski spacetime on which we will prove regularity results. 
    Let's introduce the boundary defining functions $\rho=\rho_-\rho_0\rho_+=r^{-1}$, $\rho_\scri=\rho_-\rho_+$, where
	\begin{equation}\label{eq:notation:defining_functions}
		\rho_-=\frac{v}{r}, \quad \rho_0=\frac{r}{\abs{u}v}, \quad \rho_+=\frac{\abs{u}}{r}.
	\end{equation}
	Clearly, $\rho_\bullet\in C^\infty(\Dopen\to(0,1))$.
	\begin{defi}[Compactification]\label{def:notation:comp}
	     \begin{enumerate}
	         \item Let $\D$ be the manifold with corners defined by smoothly extending $\rho_\pm,\rho_0$ to 0, resulting in the $\b-$boundaries $\scri^\pm,\,I^0$, as well as the finite boundaries $\incone,\outcone{0}$. See \cref{fig:intro:a}.
          \item Let $\Dbold$ be the manifold with corners defined by smoothly extending $\boldsymbol{\rho_0}=\rho_0\rho_-$ and $\boldsymbol{\rho_+}=\rho_+$ to 0, resulting in the $\b-$boundaries $I_0,\,\scrip$, as well as the finite boundaries $\incone,\outcone{0}$. See \cref{fig:intro:b}.
	     \end{enumerate} 
	\end{defi}
    We give some detail for the construction of $\D$, but refer to Section 2.1 of \cite{hintz_stability_2020} for more:
    In $\Dopen^-$, we use the coordinates $\rho_0,\rho_-$ with ranges $(0,c)$ for some $c>0$ to define a manifold with boundaries by extending $\rho_0,\rho_-$ to 0 smoothly.
    We proceed similarly in $\Dopen^+$ with $\rho_0,\rho_+$ and use \cref{eq:notation:defining_functions} to patch the different coordinate charts together, thus defining the manifold $\D$.
    See \cref{fig:D-b} for a depiction of the regions $\D^-$ and $\D^+$ with respect to the compactification~$\D$.

	We emphasise that the compactifications will only be used for notational convenience, and all our estimates really take place in $\Dopen$. 
	However, referring to the different compactifications  is extremely convenient, as the norms we define are going to be equivalent\footnote{Let us also mention that this geometric viewpoint is the reason  to include \cref{item2:notation:index_sets} in the definition of an index set. For instance, the function $f(x)=x+x^3$ on $[0,1)$ has no $x^2$ part, but this information is disregarded by the inclusion $f\in\A{phg}^{\mindex{1}}([0,1))$, for the benefit of making the space $\A{phg}$ geometric.} for different choices of boundary functions; therefore, we can, depending on the situation, pick these coordinates to best suit our needs.\footnote{See Section 1.2 of \cite{hintz_lectures_2023} for more on this point in different contexts.}

    Instead of $\partial_{\rho_-}, \,\partial_{\rho_0}$ etc., we will simply write $\partial_-, \,\partial_0$, where the coordinate kept fixed will always be clear from context.
	For example, when working near the future corner $\scri^+\cap I^0\subset \D^+$ we use $\rho_+,\,\rho_0$ as coordinates and write $\partial_0=\partial_{\rho_0}|_{\rho_+}$.
    Similarly, a derivative of the form $\partial_u$ or $\partial_v$ is always to be understood in $(u,v)$-coordinates, and a derivative of the form $\partial_t$ or $\partial_r$ is always to be understood in $(t,r)$-coordinates unless specified otherwise.
	
	In the next lemma, we record some expressions for the usual vector fields (including scaling $S=t\partial_t+r\partial_r=u\pu+v\pv$ and Lorentz boosts $B_i=t\partial_{x_i}+x_i\partial_t$), volume forms and wave operators in $\rho_0,\rho_\pm$ coordinates:
    We introduce the notation $\Omega^i=\epsilon^{ijk}x_j\partial_k$ for spherical derivatives, where $\epsilon^{ijk}$ is the totally antisymmetric tensor in 3 dimensions, and we sum over repeated indices.
	
	\begin{lemma}\label{lemma:notation:coordinates}
		In $\D^-$, using coordinates $\rho_-=\frac{v}{-u}$, $\rho_0=\frac1v$ and $\hat{x}=x/r$ we have 
		\begin{nalign} \label{eq:notation:partial_derivatives}
  &\partial_r=\rho_0\big((1-\rho_-)\rho_-\partial_--\rho_0\partial_0\big),\qquad \partial_{x_i}=\hat{x}_i\partial_r+\frac{\rho_0\rho_-}{1+\rho_-}\epsilon^{ijk}\hat{x}_j\Omega^k,\qquad\partial_t=\rho_0\big((1+\rho_-)\rho_-\partial_--\rho_0\partial_0\big),\\
			&\pu=\rho_-^2\rho_0\p-,\qquad \pv=-\rho_0^2\p0+\rho_0\rho_-\p-,\qquad 2\partial_t=\partial_u+\partial_v,\\
			&S=-\rho_0\partial_0,\qquad B_i=\hat{x}_i(4\rho_-\partial_--2\rho_0\partial_0)-\frac{1-\rho_-}{1+\rho_-}\epsilon^{ijk}\hat{x}_j\Omega^k.
		\end{nalign}
		Analogous results hold in $\D^+$. In particular, using $\rho_0=\frac{1}{-u},\rho_+=\frac{-u}{v}$ coordinates, we have
		\begin{equation}
			\pu=-\rho_0\rho_+\p++\rho_0^2\p0,\qquad \pv=-\rho_0\rho_+^2\p+.
		\end{equation}
  
         Furthermore, we have the following identity for the volume form:
		\begin{align}\label{eq:notation:volumeform}
			\dd \mu=\dd u \dd v\dd g_{S^2} =\frac{1}{\rho_0^2\rho_\pm}\frac{\dd \rho_\pm}{\rho_{\pm}}\frac{\dd\rho_0}{\rho_0}\dd g_{S^2},
		\end{align}
        where $g_{S^2}$ is the standard metric on the unit sphere $S^2$ and $\dd g_{S^2}$ the associated volume form.
        
		Finally, we can express the twisted wave operator in $\D^\pm$ as 
		\begin{equation}\label{eq:notation:box_near_corner}
			P_{\eta}:=r\Box_{\eta}r^{-1}=-\partial_t^2+\partial_r^2+\frac{1}{r^2}\Dl=-\partial_v\partial_u+\frac{1}{r^2}\Dl=\rho_0^2\rho_\pm^2\bigg(-\partial_{\rho_\pm}(\rho_\pm\partial_{\rho_\pm}-\rho_0\partial_{\rho_0})+\frac{1}{(1+\rho_\pm)^2}\Dl\bigg).
		\end{equation}	
	\end{lemma}

	\begin{proof}
        This is a straight-forward computation.
        We record the most used computations for the convenience of the reader.
		We'll prove the expression at the future corner (i.e.~in coordinates $\rho_0,\rho_+$), the other case follows similarly. 
        First, we have		
		\begin{equation}
			\begin{gathered}
				\partial_u=-\frac{1}{v}\partial_++\frac{1}{u^2}\partial_0=-\rho_0\rho_+\partial_++\rho_0^2\partial_0,\qquad
				\partial_v=\frac{u}{v^2}\partial_+=-\rho_0\rho_+^2\partial_+.
			\end{gathered}
		\end{equation}
        From these, $2\partial_t=\pu+\pv$ and $2\partial_r=\pv-\pu$ follow.
        For $\partial_{x_i}$ we can verify the formula by showing $\partial_{x_i}x_j=\delta_{ij}$.
		  Using~\cref{eq:notation:partial_derivatives}, we get
		\begin{equation}
			\begin{gathered}
				-\partial_t^2+\partial_r^2=-\partial_u\partial_v=-\partial_u(\frac{u}{v^2}\partial_+)=-\frac{1}{v^2}\partial_+-\frac{u}{v^2}(-\rho_0\rho_+\partial_++\rho_0^2\partial_0)\partial_+\\
				=\rho_0^2\rho_+(-\rho_+\partial_+-\rho_+^2\partial_+^2+\rho_0\partial_0\rho_+\partial_+)=-\rho_0^2\rho_\scri(\rho_+\partial_+-\rho_0\partial_0)(\rho_+\partial_+).
			\end{gathered}
		\end{equation}
		Furthermore, $r=v-u=\rho_0^{-1}(\rho_+^{-1}+1)\implies \frac{1}{r^2}=\rho_0^2\rho_+^2\frac{1}{(1+\rho_+)^2}$. Substituting these into $P_{\eta}$ proves \eqref{eq:notation:box_near_corner}.

        The proof of \cref{eq:notation:volumeform} proceeds similarly straight-forwardly. 
	\end{proof}
	On $\D$, we will use the conormal spaces introduced before and write $\Hb^{a_-,a_0,a_+;k}(\D)$ for the space where the weights have subscripts matching those of the respective boundaries.
    The same applies to the polyhomogeneous spaces $\A{phg}$ and to the compactification $\Dbold$ (which only has two b-boundaries). 
    From \cref{lemma:notation:coordinates}, it follows directly that the vector fields $v\pv,\, u\pu,\,\Omega$ span $\Diff^1_{\b}(\D)$. We include $1$ in this set of vector fields, and write
    \begin{equation}
     \Vb=\{v\partial_v,u\partial_u,\Omega,1\},
    \end{equation}
    In a striking abuse of notation, we also write this as $\Vb=\{v\partial_v,u\partial_u,\sl,1\},$
    where $\sl$ denotes the covariant derivative on the unit sphere.
    We will apply the notation of \cref{defi:notation:multiindex} to $\Vb$ as well.
    From \cref{eq:notation:volumeform}, we see that the measure associated to $\Hb(\D)$, say, near the past corner, is given by \begin{equation}\label{eq:notation:volumeformD}\frac{\dd\rho_-}{\rho_-}\frac{\dd \rho_0}{\rho_0}\dd g_{S^2}\sim \frac{\dd u}{|u|}\frac{\dd v}{v}\dd g_{S^2}.\end{equation} We therefore define

 \begin{defi}[Measures]\label{def:not:measures}
        Let's write  $\dd\mu=r^2\dd t\dd r\dd g_{S^2}$, $\dd\tilde{\mu}=\dd t\dd r\dd g_{S^2}$ and $\dd\mu_{\b}=\frac{\dd u}{|u|}\frac{\dd v}{v}\dd g_{S^2}$ for the measure and weighted measures on Minkowski. 
	With slight abuse of notation, we write $\dd \mu$, $\dd\tilde{\mu}$, $ \dd\mu_{\b}$ for the measures $r^2\dd g_{S^2}\dd u$, $\dd g_{S^2}\dd u$, $\dd g_{S^2}\frac{\dd u}{|u|}$ induced on $\Cbar_v$, similarly for $\outcone{}$.
    \end{defi}

    \begin{rem}[Explicit form of norms]
        We note that \cref{def:notation:comp,def:notation:Hb_spaces} already define the norm $\norm{\cdot}_{\Hb^{\vec{a};k}(\D)}$.
        However, for
    concreteness, we give the following equivalence:
    \begin{multline}\label{not:eq:explicit_norm}
        \norm{f}_{\Hb^{a_-,a_0,a_+;k}(\D)}^2=\int_{\D} \dd g_{S^2}\frac{\dd\rho_0}{\rho_0}\frac{\dd\rho_-}{\rho_-}\chi \big(\rho_0^{-a_0}\rho_-^{-a_-}\Vb^k f\big)^2+\int_{\D} \dd g_{S^2}\frac{\dd\rho_0}{\rho_0}\frac{\dd\rho_+}{\rho_+}(1-\chi) \big(\rho_0^{-a_0}\rho_+^{-a_+}\Vb^k f\big)^2\\
        \sim\int_{\D}\dd g_{S^2}\frac{\dd v}{v}\frac{\dd u}{u}(\rho_-^{-a_-}\rho_0^{-a_0}\rho_+^{-a_+}\Vb^k f)^2.
    \end{multline}
    for a cutoff function $\chi\in\A{phg}^{\mindex{0}}(\D)$ localising to $\D^-$ and equal to 1 in a neighbourhood of $\scrim$.
    \end{rem}

We finally define the energy momentum tensors:
\begin{defi}[Energy momentum tensor]\label{def:notation:T}
We define the energy momentum tensor ($\T$) and its twisted analogue~($\tT$):
	\begin{nalign}\label{eq:notation:energyMomentumTensor}
		&\T[f]=\partial_\mu f\partial_\nu f-\frac{1}{2}\eta_{\mu\nu}\partial f\cdot\partial f,\\
		&\tT[f]=\tilde{\partial}_\nu f\tilde{\partial}_\mu f-\frac{1}{2}\eta_{\mu\nu}\big(\tilde{\partial}f\cdot\tilde{\partial}f+w f^2\big),\\
		&\tilde{\partial}(\cdot)=\beta\partial(\beta^{-1}\cdot ),\quad w=-\frac{\Box_\eta\beta}{\beta}.
	\end{nalign}
 \end{defi}
	A simple computation shows that $\divergence\T[f]=0$ for $f$ solving $\Box_\eta f=0$.
    For the rest of the paper, we make the choice $\beta=r^{-1}$ in $\tT$, so that we have $\divergence (T\cdot\tT[f])=0$ and $w=0$ for $T=\partial_t|_x$.

    \subsection{Spaces capturing the no incoming radiation condition}
    
	As already discussed in \cref{intro:sec:scattering_theory}, we will capture the \emph{no incoming radiation condition} by regularity with respect to $r\pv$ vector fields.
    We introduce the corresponding spaces and norms.
		
	\begin{defi}[$\tilde{H}_{\b}$ vector fields]\label{def:notation:Htilde}\label{def:not:Hbt}
		We write $\Hbt^{a_-,a_0,a_+;k}(\D)$ for the completion of $\C_c^\infty(\Dopen)$ with respect to the norm
		\begin{equation}
			\norm{f}_{\tilde{H}_{\b}^{\vec{a};k}(\D)}:=\norm{\Vbt^kf}_{\Hb^{\vec{a};0}(\D)},\qquad\Vbt=\{r\pv,u\pu,\Omega,1\}.
		\end{equation}
	\end{defi}	

    \begin{rem}
        In an analogous fashion to \cref{eq:notation:volumeformD}, it can be proved that $\Diff^1_{\b}(\Dbold)$ is spanned by the vector fields $\Vbt$, and has volume form $\frac{\dd v}{r}\frac{\dd u}{u}\dd g_{S^2}$.
         As the definition of $\Hbt(\D)$ involves the vector fields of $\Diff^1_\b(\Dbold)$, it might seem more natural to use function spaces on $\Dbold$, however, our estimates are better suited for the volume form of $\Hbt$ spaces.
        Nonetheless, we still have the following $L^\infty$ embeddings from Sobolev inequality, \cref{eq:not:Sobolev},
        \begin{nalign}
        	\Hbt^{a_0,a_0;3+k}(\D^-)\subset x_0^{a_0}x_-^{a_0}C^{k}(\Dbold^-)\subset \Hbt^{a_0-\epsilon,a_0-\epsilon;k}(\D^-),\\
        	\Hb^{a_0;3+k}(\Dbold^-)\subset x_0^{a_0}x_-^{a_0}C^{k}(\Dbold^-)\subset \Hb^{a_0-\epsilon;k}(\Dbold^-).
        \end{nalign}
        This in turn implies
        \begin{equation}\label{eq:not:inclusionDbold}
        	\Hbt^{a_0,a_0;\infty}(\D^-)\subset\Hb^{a_0-\epsilon;\infty}(\Dbold^-)\subset\Hbt^{a_0-2\epsilon,a_0-2\epsilon;\infty}(\D^-).
        \end{equation}
    \end{rem}

	\subsection{A list of symbols}\label{sec:notation:symbols}
	In the following list, we record the main symbols used in the paper:
	\begin{itemize}[itemsep=-1ex]
		\item $\D,\Dopen,\Dbold$ denote spacetime regions, $\C,\Cbar$ outgoing and ingoing null cones, and $\Sigma$ will denote aspacelike hypersurface.
		\item $\T,\tilde{\T}$ denote the (twisted) energy momentum tensors defined in \cref{def:notation:T}.
		
		\item $P_\eta=r\Box_{\eta} r^{-1}$ is the twisted wave operator.
        \item $\Vb=\{u\pu,v\pv,\sl,1\},\Vbt=\{u\pu,r\pv,\sl,1\},\Ve=\{u\pu,v\pv,\rho_\scri^{1/2}\sl,1\}$ denote sets of vector fields.
		\item $\Hb,\mathcal{A}$ are conormal and polyhomogeneous spaces. $\tilde{H}_b$ is the conormal space capturing the no incoming radiation condition. We will occasionally also make use of the usual $L^2$- and $H^{;k}$-spaces.
  \item $g_{S^2}$ (and $g^{-1}_{S^2}$) will denote the (inverse of the) usual metric on $S^2$, and $\Dl$ denotes the Laplacian on $S^2$.
  \item The measures $\dd \mu,\,\dd \tilde{\mu},\,\dd\mu_{\b}$ are as defined in \cref{def:not:measures}.
	
  \item $\nabla$ will always denote the Levi-Civita connection on Minkowski space, and, given a metric $g=\eta+h$, then $\bar{\Box}_h:=(g^{-1})^{\mu\nu}\nabla_\mu\nabla_\nu$.
  	\item We use $\bullet$ as local placeholder for a  elements of a list of symbols specified in the text.
      \item
  We write $A\lesssim B$ if there exists a uniform constant $C>0$ such that $A\leq CB$. Sometimes, we will write $A\lesssim_{a,b,c} B$ to make explicit what parameters $C=C(a,b,c)$ depends on.
  \item Whenever we write $\partial r\phi$ for a derivative operator $\partial$, we will mean $\partial r\phi=\partial(r\phi)=\partial\psi$.
	\end{itemize}

 \newpage
	
	\section{ODE Lemmata}\label{sec:ODE_lemmas}

	In this section, we collect some standard results about ODEs that we will need to eventually prove \cref{thm:intro:prop_polyhom}.
	In particular, \cref{rem:notation:equivalent_polyhoms} about the equivalence of the different formulations of polyhomogeneity is also a consequence of these results.
  Throughout this section $\E$ (with any number of sub- or superscripts) will denote an index set and $c,\tilde{c}\in\R$. Similarly, $a^{\bullet}$ will always denote some real number, and $k\in\mathbb{N}$.
  The first proposition below concerns the following ODE on $X=[0,1)_x$:
  \begin{equation}\label{eq:ode:ode1}
      g=(x\partial_x-c)G,\qquad c,\tilde{c}\in\mathbb R,\,g,G\in\Hb^{-\infty;0}(X).
  \end{equation}

	\begin{prop}\label{prop:ODE:ODE}
		\begin{enumerate}[label=\textbf{ODE.\alph*}]\itemindent=15pt
			\item \label{item:ode:1d}
   Let $g, G$ solve \cref{eq:ode:ode1} and specify $G(x')=G_0\in\R$ for some $x'\in(0,1/2)$. Then the following implications hold, with the memberships on the RHS of \cref{eq:ODE:1d} \emph{not uniform in $x'$} (cf.~\cref{rem:ODE:est}):
			\begin{subequations}\label{eq:ODE:1d}
				\begin{align}
					g\in \Hb^{\tilde{c};k}(X)\text{ with }c\neq\tilde{c}&\implies G\in \A{phg}^{\overline{(c,0)}}(X) +\Hb^{\tilde{c};k}(X),\label{eq:ODE:1d_error}
					\\
					g\in\A{phg}^{\mathcal{E};k}(X)&\implies G\in\A{phg}^{\mathcal{E}\cupdex{\{(c,0)\}};k}(X)\label{eq:ODE:1d_phg}.
				\end{align}
			\end{subequations}
			
            \item\label{item:ode:1d_boundary_cond} Let $g,G$ solve \cref{eq:ode:ode1}, and specify $\lim_{x\to0}x^{-c} G=G_0\in\mathbb R$. Then the following implications hold: 
			\begin{subequations}
				\begin{align}
					\tilde{c}>c \text{ and } g\in \Hb^{\tilde{c};k}(X) & \implies G\in \A{phg}^{\mindex{c};k}(X) +\Hb^{\tilde{c};k}(X),
					\\
				\min \E>c \text{ and }g\in\A{phg}^{\mathcal{E};k}(X)& \implies G\in\A{phg}^{\mathcal{E}\cup{\{(c,0)\}};k}(X).
				\end{align}
			\end{subequations}
    
            \item\label{item:ode:1d_boundary_cond_weak_decay} Let $g,G$ solve \cref{eq:ode:ode1}, and assume that 
            $g\in\A{phg}^{\mathcal{E};k}(X)+\Hb^{\tilde{c};k}(X)$ with $\tilde{c}>c$.
            Then, there exists some
            $\bar{G}\in\A{phg}^{\mathcal{E}\overline{\cup}{\{(c,0)\}};k}(X)$ such that $g-(x\partial_x-c)\bar{G}\in\Hb^{\tilde{c};k}(X)$. 
            Furthermore, for any such $\bar{G}$ and for any $G_0\in\R$, there exists a unique $G$ (solving \cref{eq:ode:ode1}) s.t.~$G-\bar{G}-x^cG_0\in\Hb^{c+;k}(X)$, and it satisfies
            $G\in\A{phg}^{\mathcal{E}\overline{\cup}{\{(c,0)\}};k}(X)+\Hb^{\tilde{c};k}(X)$.
            In fact, given $G_0\in\mathbb{R}$,  then $G$ is the unique solution to \cref{eq:ode:ode1} with
            \begin{equation}\label{eq:ode:poly_data_expansion}
              {G}=  \sum_{(z,k)\in(\E_{\leq c}\cupdex\{(c,0)\})\setminus\{(c,0)\}}\alpha_{z,k}x^z \log^k x +x^cG_0+\Hb^{c+;k}(X),
            \end{equation}
            where $\alpha_{z,k}\in\R$ for $(z,k)\neq(c,0)$ are uniquely determined by $g$.
\end{enumerate}
\end{prop}

\begin{rem}[{Initial data below leading order}]\label{rem:ODE:slow_decay}
     The difference between \cref{item:ode:1d} and \cref{item:ode:1d_boundary_cond} lies in the different choice of boundary conditions.
     In the former, we prescribe $G$ at some finite $x$, whereas in the latter, we dictate the leading order behaviour of $G$ and some subleading decay for $g$ and $G-x^cG_0$. 
     Similarly, in \cref{item:ode:1d_boundary_cond_weak_decay}, we prescribe the free behaviour of $G$ corresponding to the kernel of $x\partial_x-c$, which is hidden behind some leading order terms.
     We note that $\bar{G}$ as in \cref{item:ode:1d_boundary_cond_weak_decay} can be constructed iteratively:
     Let's set $\alpha_{z,k}:=0$ for $(z,k)\notin\E\cupdex\{(c,0)\}$,  $\alpha_{z,k}:=(z-c)^{-1}\beta_{z,k}-(k+1)\alpha_{z,k+1}$ for $z\neq c$ and $\alpha_{c,k+1}=:(k+1)^{-1}\beta_{c,k}$; then
     \begin{equation}\label{eq:ODE:G_prime_explicit}
         g-\sum_{(z,k)\in\E_{\leq c}} \beta_{z,k} x^z\log^k x \in\Hb^{c+;k}(X)\implies \bar{G}=\sum_{(z,k)\in(\E_{\leq c}\cupdex \{(c,0)\})\setminus\{(c,0)\}} \alpha_{z,k}x^z\log^k x .
     \end{equation}

     Finally, note that we could extend \cref{item:ode:1d_boundary_cond_weak_decay} for $g\in\Hb^{\tilde{c}}([0,1)_x)$ (without an expansion) even if $\tilde{c}\leq c$. Of course, in this case, there is no uniqueness statement of the form \cref{eq:ode:poly_data_expansion} and uniqueness would only hold for prescribed $\bar{G}$.\end{rem}

\begin{obs}[Uniqueness]\label{rem:ode:uniqueness}
    Note that while \cref{item:ode:1d,item:ode:1d_boundary_cond} both are geometrically unique, \cref{eq:ode:poly_data_expansion} of \cref{item:ode:1d_boundary_cond_weak_decay} is not:
    Fix a smooth diffeomorphism of $[0,1/2)$ $\Psi:x\mapsto\tilde{x}$ with $\Psi'(0)\neq0$.
    In, say, \cref{item:ode:1d}, the solution $\tilde{G}$ of the ODE $x(\tilde{x})\Psi'\big(x(\tilde{x})\big)\partial_{\tilde{x}} \tilde{G}=g$ and $\tilde{G}(\tilde{x}(x_0))=G_0$ is related to $G$ via the coordinate change $\tilde{x}=\Psi(x)$.
    However, for \cref{item:ode:1d_boundary_cond_weak_decay}, fixing the $x^0$ or $\tilde{x}^0$ part of $G$ or $\tilde{G}$ generally yields a different solution, provided that $\min(\E)<c-1$.\footnote{ Let $c=0$, $g(x)=x^{-a}$ for $a>0$, and specify the $x^0$ part of $G$ to be $C\in \R$. Consider the coordinate change $x=\tilde x+\tilde{x}^2$. For $a\in(0,1)$ specifying the term in $\tilde{x}^0$ or $x^0$ yields the same result at leading order and the rest is unique by \cref{item:ode:1d_boundary_cond}. For $a=1$, however, $-x^{-1}=-\tilde{x}^{-1}+1+o(\tilde{x})$, and so specifying the free term in $\tilde{x}$ is shifted by 1 compared to that in $x$.}
    Similar non-uniqueness will be present in \cref{corr:ODE:du_dv} \cref{item:ode:u-prop_boundary} .
\end{obs}

\begin{prop}\label{prop:ode:ode2}
            \begin{enumerate}[label=\textbf{ODE.d}]\itemindent=15pt
             \item \label{item:ode:2d} Let $X=[0,1)_{x_0}\times[0,1)_{x_+}$, let $g\in \Hb^{-\infty,-\infty;0}(X)$, and let $G$ solve
            \begin{equation}(x_0\partial_0-x_+\partial_+)G=g.\end{equation}
            Let $G|_{x_+=x'}=G_0\in\Hb^{-\infty;0}([0,1)_{x_0})$ for some $x'\in(0,1/2)$.
            Then the following implications hold for $\E_0=\E_0^{G_0}\cup\E_0^g$.
			\begin{subequations}
				\begin{align}
					G_0\in\Hb^{a^G_0;k}([0,1)_{x_0})\text{ and }g\in \Hb^{a_0,a_+;k}(X) \text{ with } a_0\neq a_+ &\implies G\in \Hb^{\min(a^G_0,a_0),\min(a^G_0,a_+,a_0);k}(X)\label{eq:ODE:2d_error},\\
					G_0\in\A{phg}^{\E_0^{G_0};k}([0,1)_{x_0})\text{ and }g\in \A{phg,b}^{\E_0^{g},a_+;k}(X) \text{ with } (a_+,\cdot)\notin\E_0 &\implies G\in\A{phg}^{\E_0,\mathcal{E}_0}(X)+\A{phg,b}^{\mathcal{E}_0,a_+;k}(X)\label{eq:ODE:2d_phg_error},\\
					G_0\in\A{phg}^{\E_0^{G_0};k}([0,1)_{x_0})\text{ and }g\in\A{phg}^{\mathcal{E}^g_0,\mathcal{E}^g_+;k}(X) &\implies G\in \A{phg}^{\mathcal{E}_0,(\mathcal{E}^g_0\overline{\cup}\mathcal{E}^g_+)\cup \E^G_0;k}(X)\label{eq:ODE:2d_phg}\\
                    G_0=0\text{ and }g\in \A{b,phg}^{a,\E;k}(X) \text{ with } (a,\cdot)\notin\E &\implies G\in\A{b,phg}^{a,\mathcal{E};k}(X)+\Hb^{a,a;k}(X)\label{eq:ODE:2d_error_phg}.
				\end{align}
			\end{subequations}
            
   
		\end{enumerate}
	
\end{prop}

\begin{rem}[Linearity]
    The estimates in \cref{prop:ODE:ODE,prop:ode:ode2} can be summed to also treat cases where $g,G_0$ is in a sum of the above spaces.
\end{rem}

\begin{rem}[Sharpness]\label{rem:ODE:sharp}
    We note that in \cref{eq:ODE:1d_error}, the case $c=\tilde{c}$ is not allowed: For instance, let $c=0$ and $g=\log^{-2}(x)\in\Hb^{0;\infty}([0,1/2))$, then $G=\log^{-1}(x)\notin\Hb^{0;0}([0,1/2))$.
    We also have that the index set $\E\cupdex\mindex{c}$ in \cref{eq:ODE:1d_phg} is sharp as demonstrated by taking $g=x^c\log^k x$.
    One similarly shows sharpness for all other restrictions and index sets.
\end{rem}

\begin{rem}[Product spaces]\label{rem:ODE:product_spaces}
		All results in \cref{prop:ODE:ODE,prop:ode:ode2} extend to product spaces. For instance, given a manifold with boundary, \cref{item:ode:1d} generalises to the setting where $X=[0,1)_x\times Y$, $g,G\in\Hb^{-\infty,a}(X)$, and we have 
		\begin{subequations}
			\begin{gather}
				g\in \Hb^{\tilde{c},a;k}(X)\implies G\in \A{phg,b}^{\overline{(c,0)},a;{k}}(X) +\Hb^{\tilde{c},a;k}(X),
				\\
				g\in\A{phg,b}^{\mathcal{E},a;{k}}(X)\implies G\in\A{phg,b}^{\mathcal{E}\bar{\cup}{\{(c,0)\}},a; {k}}(X).
			\end{gather}
		\end{subequations}
  The applicability of the above ODE results to product spaces is the reason why we refrained from recording improved regularity in, say, \cref{eq:ODE:1d_error}.
	\end{rem}
 \begin{rem}[Finite regularity]
     In the two propositions above, we kept track of the regularity so that it's clear that all results in the paper hold at finite regularity. To alleviate notation, we will write the proof as well as the remaining statements of this section at infinite regularity.
 \end{rem}

	\begin{proof}[Proof of \cref{prop:ODE:ODE,prop:ode:ode2}]
        Note that $(x\partial_x-c)G=x^c(x\partial_x)(x^{-c} G)$, therefore it suffices to treat the case $c=0$ in \cref{prop:ODE:ODE} and the other cases follow by twisting with $x^c$.

		Except for \cref{eq:ODE:2d_error_phg}, both \cref{item:ode:1d,item:ode:2d} are already proved in \cite{hintz_stability_2020} (Lemma 7.5~and 7.7, respectively) in the case of trivial initial data, i.e.~$\supp g,G\in\{x<1/2\}$ and $\supp g,G\in \{ x_+<1/2 \}$; the proof consists of writing down the explicit solutions to the relevant ODEs and then performing basic estimates.
		Adding initial data (i.e.~setting $G_0\neq0$) follows similarly; we write down the explicit representations for sake of completeness.
		For~\cref{eq:ODE:1d_error}, we write 
		\begin{equation}
			G(x)=\begin{cases}
				G_0+\int_{x'}^x\frac{\dd y}{y}g(y),  &\text{if }\tilde{c}< 0,\\
					G_0-\int_0^{x'}\frac{\dd y}{y}g(y) +\int_0^x\frac{\dd y}{y}g(y),  &\text{if }\tilde{c}>0.
					\end{cases}
		\end{equation}
		Swapping the order of integration in $\norm{G}_{\Hb^{0;0}}$ yields \cref{eq:ODE:1d_error}, cf.~\cite{hintz_stability_2020}. The proof of \cref{eq:ODE:1d_phg} follows similarly.
            
        To prove~\cref{item:ode:1d_boundary_cond}, again setting $c=0$, we consider $\tilde{G}=G-G_0$ and write 
        \begin{equation}
            \tilde{G}(x)=\int_0^{x}\frac{\dd y}{y}g(y).
        \end{equation}
        The bounds for $\tilde{G}$  then follow as in \cref{item:ode:1d} with $\tilde{c}>0$.

        To prove \cref{item:ode:1d_boundary_cond_weak_decay}, 
        we first specify $\bar{G}$ as in \cref{eq:ODE:G_prime_explicit}.
        Next, we notice that $(x\partial_x-c)(G-\bar{G})=g-(\partial_x-c)\bar{G}\in\Hb^{c+}(X)$.
        We apply \cref{item:ode:1d_boundary_cond} to get $G-G'\in\A{phg}^{\E\cup\{(0,0)\}}$.
        Adding back $\bar{G}$ yields the result.

        To prove \cref{eq:ODE:2d_error_phg}, we proceed by induction.
        We claim that there exist $G_n\in\A{b,phg}^{a,\E;k}(X)$ such that
        \begin{equation}\label{eq:ode:proof1}
        	(x_+\partial_+-x_0\partial_0)(G-G_n)\in\A{b,phg}^{a,\E_n;k}(X)
        \end{equation}
        where the index sets $\E_n$ satisfy $\E_{n+1}\subset \E_n$ and $\min(\E_{n+1})>\min(\E_n)$.
        The base case holds by assumption with $G_0=G$ and $\E_0=\E$.
        For the inductive step, we assume that \cref{eq:ode:proof1} holds for some $n\geq0$.
       Denoting $(p,m)=\min(\E_n)$, we use the definition of polyhomogeneity to write for some $g_{p,m}\in\Hb^{a;k}([0,1)_{x_0})$
        \begin{equation}
        	(x_+\partial_+-x_0\partial_0)(G-G_n)-x_+^p\log^mx_+g_{p,m}(x_0)\in\A{b,phg}^{a,\E_{n+1};k}(X).
        \end{equation}
        We use \cref{eq:ODE:1d_error} (and $p\neq a$) to conclude the existence of $\tilde{G}_{p,m}\in\Hb^{a;k}([0,1)_{x_0})$ solving $(p-x_0\partial_0)\tilde{G}_{p,m}=g_{p,m}$.
        Setting $G_{n+1}:=G_n+x_+^p\log^m x_+\tilde{G}_{p,m}$ yields the induction step.
   		Taking $n$ sufficiently large, we get $\min(\E_n)>a$.
   		Using \cref{eq:ODE:2d_error} for the remaining term yields \cref{eq:ODE:2d_error_phg}.
	\end{proof}

 We now apply these results to various relevant settings in $\D$:
	\begin{cor}\label{corr:ODE:du_dv}
		Let $g\in \A{phg}^{\vec{\E}^g}(\D)+\Hb^{\vec{a}^g}(\D)$. Then we have the following implications
		\begin{enumerate}[label={\alph*})]
			\item \label{item:ode:u-prop} Let  $g=\partial_u G$. If $\vec{a}^g=\infty$ and $G|_{\outcone{}}\in\A{phg}^{\E^{\outcone{}}}(\outcone{})$ for some $u$, then $ G\in\A{phg}^{\vec{\E}^G}(\D)$, where
			\begin{nalign}
				\E^G_-&=(\E^g_--1)\overline{\cup}\{(0,0)\},&\E_0^G&=\E^G_+\cup\big(\E^g_+\overline{\cup}(\E^g_0-1)\big),&\E_+^G&=\E^{\outcone{}}\cup\E^{g}_+.
			\end{nalign}
			
             If, instead, $\vec{\E}^g=\emptyset$ and $G|_{\outcone{}}\in\Hb^{a^{\outcone{}}}(\outcone{})$, then, for $a^g_0-1\neq a^g_+$ and $a_-^g\neq1$, $ G\in\A{phg,b,b}^{\overline{(0,0)},a_0^G,a_+^G}(\D)+\Hb^{\vec{a}^G}(\D)$, where
			\begin{nalign}
				a_-^G&=a_-^g-1,& a_0^G&=\min(a_0^g-1,a_+^G),& a_+^G&=\min(a^{\outcone{}},a^g_+).
			\end{nalign}
            
            \item \label{item:ode:u-prop_boundary} Let $g=\pu G$ and fix $\rho_-=v/u$. 
            If $\vec{a}^\infty=\infty$, then for all $G^{\scrim}\in\A{phg}^{\E^{\scri}}(\scrim)$ there exists a unique $G\in\A{phg}^{\vec{\E}^G}(\D)$ satisfying
            \begin{subequations}
                \begin{equation}
                    G|_{\D^-}=\sum_{(z,k)\in(\E^G_-\setminus \{(0,0)\})_{\leq0}}\rho_-^z\log^k(\rho_-) G_{z,k}(\rho_0,\omega)+G^{\scrim}(v,\omega)+\A{b,phg}^{0+,\E^G_0}(\D^-), \text{ where}
                \end{equation}
                \begin{nalign}
                    \E^G_-&=(\E^g_--1)\cupdex\{(0,0)\},&\E_0^G&=(\E^g_0-1)\cup\E^{\scri},&\E_+^G&=\E_0^G\cup\big((\E^g_0-1)\overline{\cup}\E^g_+\big);
                \end{nalign}
            \end{subequations}

            If instead, $\vec{\E}^g=\emptyset$, $a^g_->1$, $a_+^g\neq a^g_0$ and $G^{\scrim}\in\Hb^{a^{\scri};k}(\D)$, then$G\in\A{phg,b,b}^{\mindex{0},a^G_0,a^G_+}(\D)+\Hb^{\vec{a}^G;k}(\D)$ (and if $G^{\scrim}=0$, then $G\in \Hb^{\vec{a};k}(\D)$), where
            \begin{nalign}
				a_-^G=a^g_--1,\quad a_0^G=\min(a_0^g-1,a^{\scri}),\quad a_+^G=\min(a_+^g,a_0^G).
			\end{nalign}
        
			\item \label{item:ode:t-prop} Let $g=\partial_t G$. If $\vec{a}^g=\infty$,  and $G|_{\incone}\in\A{phg}^{\E^{\incone}}(\incone)$, then $ G\in\A{phg}^{\vec{\E}^{G}}(\D)$
			\begin{nalign}\label{eq:ODE:T-propagation}
				\E_-^G=\E_-^g\cup\E^{\incone},\quad \E_0^G=\E_-^G\cup\big(\E^g_-\overline{\cup}(\E^g_0-1)\big),\quad\E_+^G=\E_0^G\cup\big((\E^g_0-1)\overline{\cup}\E^g_+\big).
			\end{nalign}
			
			If, instead, $\vec{\E}^g=\emptyset$ and $G|_{\incone}\in\Hb^{a^{\incone}}(\incone)$, with $a_-^g\neq a_0^g-1$ and $a_0^g-1\neq a^g_+$, then $G\in\Hb^{\vec{a}^G;k}(\D)$
			\begin{nalign}\label{eq:ODE:T-propagation_error}
				a_-^G=\min(a^{\incone},a^g_-),\quad a_0^G=\min(a_0^g-1,a_-^G),\quad a_+^G=\min(a_+^g,a_0^G).
			\end{nalign}
		\end{enumerate}
	\end{cor}

    \begin{rem}
        The losses of different weights for $\partial_t$ and $\pu$ follow from the fact that $\partial_t\in\rho_0\Diff^1_{\b}(\D)$ and $\pu\in\rho_0\rho_-\Diff^1_{\b}(\D)$. See also \cref{fig:integralcurves}.
    \end{rem}

    \begin{rem}[Time reversed]
        All the statements in \cref{corr:ODE:du_dv} have straightforward time reversed versions with  $t\mapsto -t$, thus interchanging $\pu$ and $\pv$.
    \end{rem}

    \begin{rem}[Estimates vs inclusions]\label{rem:ODE:est}
    All statements above are stated as inclusions. Of course, these inclusions naturally come with estimates. For instance, in the setting of \cref{eq:ODE:T-propagation_error}, we have the estimate
    \begin{equation}
        \norm{G}_{\Hb^{\vec{a}_G;k}(\D)}\lesssim_k \norm{G}_{\Hb^{a^{\incone};k}(\incone)}+\norm{g}_{\Hb^{\vec{a}_g;k}(\D)} \quad\forall k.
    \end{equation}
As another example, the estimates corresponding to \cref{eq:ODE:1d} will depend on $x'$ and potentially degenerate as $x'\to0$.
    \end{rem}

	\begin{proof}
		We start with \cref{item:ode:u-prop}. 
		In a neighbourhood of $\scrip\cap I^0\subset \D^+$, we use coordinates $\rho_+=u/v,\rho_0=-1/u$ and write $\pu=\rho_0(\rho_+\partial_+-\rho_0\partial_0)$.
		\cref{item:ode:2d}, generalised to product spaces, and with the role of $x_0$ and $x_+$ interchanged since our data are at $x_0=\mathrm{const}$, concludes the result in $\D^+$. 

        We use a cutoff $\chi\in\A{phg}^{\mindex0,\mindex0}(\D^-)$ localising in the region $\D^-$ with $\supp\chi'\in \D^+\cap\D^-$, and write $G^-=\chi G$.
        Using the estimate in the region $\D^+$, we get that $\pu G^-$ is in the same function space as $\pu G$.
		Near the past corner, we use $\rho_-=-v/u,\rho_0=1/v$ to get $\pu=\rho_0\rho_-(\rho_-\partial_-)$.
		\cref{item:ode:1d} yields the result (again, generalised to product spaces).

        For \cref{item:ode:u-prop_boundary}, we use the same regions and coordinates as before.
        We start with $\D^-$ and apply \cref{item:ode:1d_boundary_cond,item:ode:1d_boundary_cond_weak_decay} to obtain the result near the past corner.
        In $\D^+$, we use \cref{item:ode:2d} to conclude the result.
  
		For \cref{item:ode:t-prop} we use 
		$\rho_-=v/r,\rho_0=1/v$ in the region $\D^-$ and write $\partial_t=\rho_0(\rho_-\partial_--\rho_0\partial_0)$. In $\D^+$, we use
		$\rho_+=-u/r,\rho_0=-1/u$ and write $\partial_t=\rho_0(\rho_+\partial_+-\rho_0\partial_0)$. 
		The result follows from consecutively applying \cref{item:ode:2d} first in $\D^-$, then in~$\D^+$.
	\end{proof}
 \begin{figure}[htpb]
\centering
\begin{subfigure}{0.45\textwidth}
\centering
    \includegraphics[width=120pt]{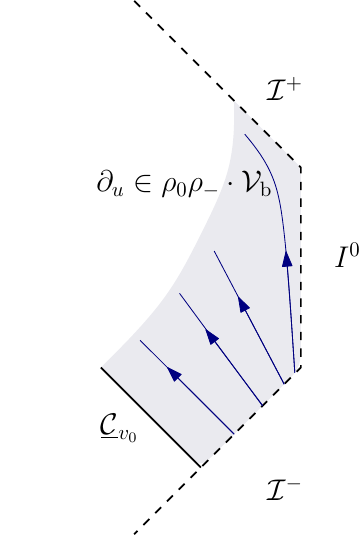}
    \caption{Integral curves of $\pu$.}
    \label{fig:integralcurves:a}
\end{subfigure}
\hspace{0.05\textwidth}
\begin{subfigure}{0.45\textwidth}
\centering
 \includegraphics[width=120pt]{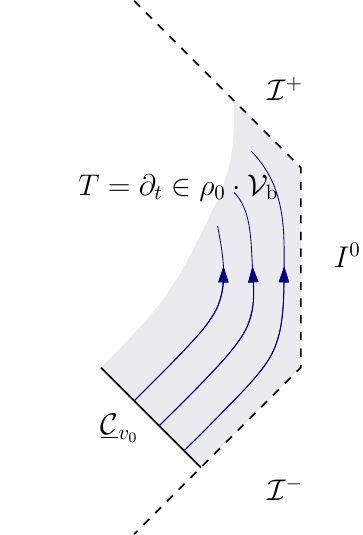}
    \caption{Integral curves of $\partial_t$.}
    \label{fig:integralcurves:b}
\end{subfigure}
\caption{Integral curves of the vector fields $\pu$ and $\partial_t$. Cf.~\cref{lemma:notation:coordinates}}\label{fig:integralcurves}
\end{figure}

	We can trivially extend the previous result for general coefficients
	\begin{cor}
		Let $(1-f),(1-h),g\in\A{phg}^{\vec{\E}^\bullet}(\D)$ for $\bullet\in\{f,h,g\}$ and $G\in H^{-\infty,-\infty,-\infty}(\D)$ with $\min(\vec{\E}^f,\vec{\E}^h)>0$.
        Then for $-u_0,v_0$ sufficient large, depending on $f,h$, we have
		\begin{enumerate}[label={\alph*})]
			\item Let $g=f\partial_u(hG)$ and $G|_{\outcone{}}\in\A{phg}^{\E^{\outcone{}}}(\outcone{})$, then $G\in \A{phg}^{\E^{\vec{G}}}(\D)$ for some $\vec{\E}^G$.
            \item Let $g=f\partial_u(hG)$ with $G|_{\outcone{}^{v_0,v_\infty}} \to G^{\scrim}\in\A{phg}(\scrim)$ as $u\to-\infty$ for all $v_\infty>v_0$.
            Then $G\in\A{phg}^{\vec{\E}^G}(\D)$ for some $\vec{\E}^G$.
		\end{enumerate} 
	\end{cor}
 \begin{proof}
     We use the condition on $f,h$ to localise so that $(1-f),(1-h)\in[1/2,3/2]$ in $\D$.
    Then $gf^{-1}$ is polyhomogeneous for some index set.
    We apply \cref{corr:ODE:du_dv} to conclude polyhomogeneity of $hG$, then divide by $h$ to get the same for $G$ with some other index set.
 \end{proof}
	
	\newpage
		\part{Energy estimates, scattering and polyhomogeneity}
    This second part of the paper consists of \cref{sec:en,sec:scat:scat,sec:prop}. These sections form the main building blocks for the rest of the paper.
    
    In \cref{sec:en}, we prove weighted energy estimates for the finite characteristic initial value problem for a large class of perturbations of the Minkowskian wave equation. In particular, we define our class of perturbations in this section. 
    In \cref{sec:scat:scat}, we then use these estimates to construct scattering solutions (with data at infinity) still satisfying the same estimates.
    In \cref{sec:prop}, we then prove various propagation of polyhomogeneity statements for solutions to the inhomogeneous Minkowskian wave equation $\Box_\eta\phi=f$.
		\section{Energy estimates for the finite problem}\label{sec:en}
	In this section, we prove the main energy estimates of the paper in the context of the finite characteristic initial value problem (IVP).
	More precisely, we consider
	\begin{equation}\label{eq:en:characteristicIVP}
		\begin{gathered}
			\Box_\eta \phi=f\qquad \text{in } \D^{u_{\infty},v_{\infty}}_{u_0,v_0},\\
			\psi|_{\Cbar_{v_0}}=\psi^{\Cbar_{v_0}}\qquad \psi|_{\C_{u_\infty}}=\psi^{\C_{u_\infty}},\qquad \text{where } \psi=r\phi,
		\end{gathered}
	\end{equation}
	for some $u_{\infty}<u_0\ll -1$ and $0<v_0<v_\infty$, and with the inhomogeneity $f$ later on replaced by suitable nonlinearities.
	Here, while $v_0>0$ is arbitrary, some of our estimates will require taking $\abs{u_0}$ sufficiently large, see e.g.~\cref{lemma:en:recovering_initial_data}.
	With the scattering problem in mind, the estimates we prove here are uniform in $u_\infty,v_\infty$. That is to say, for a sequence of functions $f_{u_\infty}$, labelled by $u_\infty$ and with $u_\infty\to-\infty$, we write $f_{u_\infty}\in\Hb^{\vec{a}}(\D^{u_{\infty},v_{\infty}}_{u_0,v_0})$ if there exist a constant $C$ independent $u_\infty$ such that $\norm{f_{u_\infty}}_{\Hb^{\vec{a}}(\D^{u_{\infty},v_{\infty}}_{u_0,v_0})}\leq C$ for all $u_{\infty}$.
   For the rest of this section, we will simply write $\D=\D^{u_{\infty},v_{\infty}}_{u_0,v_0}$, suppressing the dependence on $u_0,v_0,u_\infty,v_\infty$.
    Similarly, we will write $\C_{u}$ rather than $\C_{u_1}^{v_0,v_\infty}$ and $\Cbar_v$ rather than $\Cbar_{v}^{u_0,u_{\infty}}$.
   
 
 	We split up the proof of our energy estimates into the regions $\D^-$ and $\D^+$ (cf.~\cref{fig:D}). Notice that estimates in the region $\D^+$ have, in the context of the Cauchy problem, already been studied in great detail in previous works; we recall some of these in \cref{sec:en:previous}.
 	Indeed, for this region, we can simply use the estimates of \cite{hintz_stability_2020}, see~\cref{sec:en:previous}.
    On the other hand, we provide a self-contained physical space derivation for the required estimates in $\D^-$ in \cref{sec:en:en}.
 	We note already that these estimates overlap to some extent with those proved in the more recent \cite{hintz_microlocal_2023-1}, see already \cref{sec:en:previous} for further comparison.
 	
 	This section is structured as follows:
 	We first recall estimates from the literature in both regions $\D^\pm$ for \cref{eq:en:characteristicIVP} in \cref{sec:en:previous}. 
    Then, we study \cref{eq:en:characteristicIVP} with no incoming radiation and prove estimates on exact Minkowski space in \cref{sec:en:en}.
    The proof of these estimates follows integration by parts in $u,v$ coordinates, but we also provide a more geometric proof using currents  in \cref{app:sec:current_computations}. In \cref{sec:en:longrange}, we show that all these estimates can be extended to include what we call linear long-range potentials.
    Next, in \cref{sec:en:perturbations}, we introduce a class of (nonlinear)  short-range perturbations to \cref{eq:en:characteristicIVP}  and show that we can still close our estimates from \cref{sec:en:en}.
    We finally prove energy estimates for solutions with incoming radiation in \cref{sec:en:incoming}.

	\subsection{Previous estimates from the literature}\label{sec:en:previous}
We first recall an estimate from \cite{hintz_stability_2020} concerning the future region $\D^+$:

\begin{prop}\label{prop:en:futureestimate}
		Let $\psi^{\Cbar_{v_0}}=0=\psi^{C_{u_\infty}}$,  let $\supp f\subset\D^{+}\cap\{v>v_0+1\}\cap\{u>u_\infty+1\}$, and let $\phi$ solve \cref{eq:en:characteristicIVP}. Then, for $a_0$ arbitrary and for $a_+< \min(a_0,0)$, we have for any $k\in\mathbb N_{\geq0}$:
		\begin{nalign}\label{eq:en:futureestimate}
			\sup_{u\in (u_\infty,u_0)}\norm{(v\pv,\rho_+^{1/2}\sl)\Vb^kr\phi}_{\rho_0^{a_0}\rho_+^{a_+}\Hb^{;0}(\outcone{}^+)}+\sup_{v\in(v_0,v_\infty)}\norm{(u\pu,\rho_+^{1/2}\sl)\Vb^kr\phi}_{\rho_0^{a_0}\rho_+^{a_+}\Hb^{;0}(\Cbar_v^+)}\\
			+\norm{(1,u\pu,v\pv,\rho_+^{1/2}\sl)r\phi}_{\rho_0^{a_0}\rho_+^{a_+}\Hb^{;k}(\D)}\lesssim \norm{rf}_{\rho_0^{a_0+2}\rho_+^{a_++1}\Hb^{;k}(\D)}.
		\end{nalign}
	\end{prop}
	\begin{proof}
		This is the scalar analogue of \cite[Proposition~4.8]{hintz_stability_2020}.\footnote{Note that the relevant operator in Hintz--Vasy's $L_h u=f$ already contains weights in $L_h$, and so do their $u,f$.}
		It follows from a $v^{2a_++1}\abs{u}^{2(a_0-a_+)}(\pv+\frac{\abs{u}}{v}\pu)$ multiplier estimate for the twisted energy momentum tensor $\tilde{\T}$.
       We provide the computation in \cref{lemma:current:future}.
	\end{proof}

	In order to assist the reader in better parsing this estimate, we observe the following:
	\begin{itemize}
		\item  Firstly, $\b$-regularity is propagated by the estimate.
		This is a consequence of the symmetries, i.e.~good commutators, of $\Box_{\eta}$. 
More precisely, we actually have 1 order of regularity gain, though this is not $\b$-regularity, as $\sl$ derivatives lose half a power of decay ($\rho_+^{1/2}$) towards $\scrip$.
This loss for the angular derivative is related to choice of vector fields with respect to which $\Box$ is to leading order a nondegenerate quadratic form on $\D$. See   \cite{hintz_microlocal_2023-1} for more details.

		\item Ignoring the $\rho_+^{1/2}$ factor, we exactly lose 2 factors of decay at $I^0$ and 1 factor of decay at $\scri^+$, as expected from the fact that $\Box_\eta\in\rho_+^{-1}\rho_0^{-2}\Diff_b^2(\D^+)$, cf.~\cref{eq:notation:box_near_corner}. 
        That is, for $\phi\in\Hb^{a_0,a_+;k}(\D^+)$, we have $\Box_\eta\phi\in\Hb^{a_0-2,a_+-1;k-2}(\D^+)$.
	\end{itemize} 
 The specific gain of regularity with a loss in the angular derivatives will appear many times throughout the paper; we therefore introduce notation for it:
	\begin{defi}
		We use the shorthand $\Ve=\{1,v\pv,u\pu,\rho_\scri^{1/2}\Omega_i\}$. (Recall also $\rho_{\mathcal{I}}:=\rho_-\rho_+$.)
	\end{defi}
	
  For the past region $\D^-$, we quote the following, more recent statement from \cite{hintz_microlocal_2023-1}:

  \begin{prop}\label{prop:en:hintz_past}
		Let $\psi^{\Cbar_{v_0}}=0=\psi^{\C_{u_\infty}}$,  let $\supp f\in\D\cap\{u>u_\infty+1,v>v_0+1\}$ and let $\phi$ solve \cref{eq:en:characteristicIVP}. Then, for $a_->1/2$ and $a_0<a_-$, we have for any $k\in\mathbb N_{\geq0}$:
		\begin{nalign}\label{eq:en:hintz_past}
			\norm{\Ve r\phi}_{\rho_-^{a_-}\rho_0^{a_0}\Hb^{;k}(\D^{-})}\lesssim \norm{rf}_{\rho_-^{a_-+1}\rho_0^{a_0+2}\Hb^{;k}(\D^{-})}.
		\end{nalign}
	\end{prop}
	\begin{proof}
		This corresponds to \cite[Eq.~(6.6)]{hintz_microlocal_2023-1}, with $\tilde{s}=1$, $\tilde{\alpha}_\scri=a_--1/2$ and $\tilde{\alpha}_0=a_0-1$ (and after carefully deconstructing all the definitions in \cite{hintz_microlocal_2023-1}).
	\end{proof}

	The cited work of course includes a much more general statement than the one above; in particular, it is also valid for a large class of perturbations of the geometry such as Schwarzschild or Kerr. We do not go into details on this. (In particular, we do not introduce any additional structure associated to the geometry connected to the non-degeneracy of $g$ as a quadratic bilinear form up to the boundaries of $\D$. For more details, see~\cite{hintz_microlocal_2023-1}.)
 Let us for now just highlight that, \cref{prop:en:hintz_past} requires that the inhomogeneity decays at least like $rf\sim \rho_-^{3/2}$ towards $\scrim$. This is already improved to $rf\sim \rho_-^{1+}$ in \cite{hintz_microlocal_2023-1}; but stating it would, again, require importing extra notation concerning the regularity setting.
        	
	
	The authors only became aware of \cref{prop:en:hintz_past} when this section was already mostly written, and so we conclude by listing the main similarities and differences of the estimates of the next few subsections to \cref{prop:en:hintz_past}.
	\begin{itemize}
		\item In \cref{item:en:pastestimate_inhom} of \cref{prop:en:pastestimate}, we recover the estimate from \cref{prop:en:hintz_past} together with an extension (\cref{item:en:pastestimate_inhom_0th}) to the case where $rf\sim r^{-p}$ for $p\in(1,3/2)$.
		In the latter case, we either lose the optimal expected decay towards $I^0$---in particular, we can only show that $r\phi\sim r^{-p+3/2}$---or we lose the gain in $\Ve$ regularity. We further show improved estimates in the case of no incoming radiation in \cref{item:en:pastestimate_Hbt}.
		\item In \cref{lemma:en:energy_no_incoming,thm:scat:weak_polyhom}, we obtain estimates for the case when the inhomogeneity decays weakly ($rf\sim r^{-p}$ for $p<1$) towards $\scrim$ and either has no incoming radiation (which in this context means $\psi^{\C_{u_\infty}}=0$) or an expansion towards $\scrim$.
		\item We also include nontrivial characteristic initial data for \cref{eq:en:characteristicIVP}.
		\item Finally, in addition to extending all our results to \textit{short-range} perturbations, we show that the energy estimates also work for some long-range perturbations (e.g.~$r^{-2}$-potentials) of $\Box_{\eta}$ in \cref{sec:en:longrange}.
	\end{itemize}

	\subsection{Energy estimates I: The linear wave equation \texorpdfstring{$\Box_{\eta}\phi=f$}{} with no incoming radiation}\label{sec:en:en}
	
	In this section, we are concerned with deriving estimates for the finite problem \cref{eq:en:characteristicIVP} designed to later be applied to a sequence of solutions approaching a scattering solution with no incoming radiation. Therefore, we make the assumption that\footnote{We note already that, in the limiting problem, the no incoming radiation condition only imposes the vanishing of the limit of $\pv\psi$ towards $\scrim$, which is still consistent with nonzero or also diverging $\psi$.} 
	\begin{align}\label{eq:en:assumption0}
		\psi|_{\incone\cap\{u\in(u_\infty,u_\infty+1)\}}=f|_{\D\cap\{u\in(u_\infty,u_\infty+1)\}}=0,\quad\psi^{\Cbar_{v_0}}(u_{\infty})=0.
	\end{align}
	
	Let us now introduce some notation capturing the range of decay rates that we will consider in our estimates. The reader may wish to refer to the discussion below \cref{eq:intro:admissiblebelow} for basic intuition.
	\begin{defi}\emph{(Admissible weights)}\label{def:en:admissible:f}
		We say that $\vec{a}$ is an \emph{admissible weight triple}, in short \emph{admissible}, if $a_-\geq -1/2,a_+<a_0<a_-$, $a_+<0$ and $a_-\neq0$.\footnote{Note that the assumption $a_0<a_-$ may be relaxed for solutions with incoming radiation as discussed in \cref{rem:poly:improved_decay_at_I0}, but we will pursue no such improvements in the paper.}
  
Given $\vec{a}$ admissible, we say that $\vec{a}^f$ is an \emph{admissible inhomogeneous weight triple relative to $\vec{a}$}, in short \emph{admissible inhomogeneous}, if $a_-^f>1$ and

    \begin{subequations}\label{eq:en:admissible_inhom}
    	\begin{empheq}[left={a_-\leq a^f_--1,\quad a_+\leq a^f_+-1,\quad a_0\leq\empheqlbrace}]{alignat=2}
    		&\min(a^f_--1.5,a_0^f-2),& & \text{ if }a_-^f\leq3/2\label{eq:en:admissible_inhom1}\\
				&a_0^f-2,& & \text{ otherwise.}\label{eq:en:admissible_inhom2}
    	\end{empheq}
    \end{subequations}
		We say $\vec{a}^f$ is a \emph{strongly admissible inhomogeneous weight} if the inequalities  in \cref{eq:en:admissible_inhom} are strict and call $\delta=\sup\{\epsilon>0: \vec{a}^f-(\epsilon,\epsilon,\epsilon) \text{ is strongly admissible}\}$ the corresponding \emph{gap}.
        
        As a shorthand, we will simply say that the pair $\vec{a},\vec{a}^f$ is \textit{admissible} to mean that $\vec{a}$ is admissible and $\vec{a}^f$  is admissible inhomogeneous with respect to $\vec{a}$. Here, in the pair $\vec{a}$, $\vec{a}^{\bullet}$, the triple without the superscript will always be the one that's admissible; the other one will be the one that's admissible inhomogeneous.
	\end{defi}

    \begin{rem}\label{rem:en:strict_inequalities}
		In \cref{def:en:admissible:f}, we allowed for equality in \cref{eq:en:admissible_inhom} and did not allow for equality in $a_+<a_0<a_-$. Alternatively, one could swap these two cases and propagate sharp decay rates below top order regularity, we leave this to the reader.
		However, removing both of them is not possible, cf.~\cref{rem:ODE:sharp}. We don't study such small improvements further.
        Similarly, we made a choice not to allow $a_-=0$ for practical purposes as this case would have to be treated separately in some cases, e.g.~\cref{lemma:en:recovering_initial_data}. The reason is exactly the condition $c\neq \tilde c$ in \cref{eq:ODE:1d_error}.
	\end{rem}
 \begin{rem}
    The condition \cref{eq:en:admissible_inhom1} can always be ignored provided one is willing to work at a loss of derivatives. See already \cref{def:scat:extended_short_range}. 
 \end{rem}
\subsubsection{The main estimates}
 
	\begin{prop}[Past energy estimate]\label{prop:en:pastestimate}
		Let $\phi$ be the solution to \cref{eq:en:characteristicIVP} under the assumptions \cref{eq:en:assumption0}, and let $\vec{a}$ be admissible. Then, for any $k\in\mathbb N_{\geq0}$, we have the following estimates:
		\begin{enumerate}[label=\textbf{E.\arabic*}]
			\item \label{item:en:pastestimate_hom} Homogeneous estimate: For $f=0$, we have:
			\begin{equation}\label{eq:en:pastestimate_hom}
				\norm{\Ve \psi}_{\rho_-^{a_-}\rho_0^{a_0}\Hb^{;k}(\D^-)}\lesssim \norm{\Vb^k \psi}_{\rho_-^{a_-}\Hb^{;1}(\incone)}.
			\end{equation}
			\item \label{item:en:pastestimate_inhom} 
			Inhomogeneous estimate: For $\psi^{\incone}=0$ and admissible $\vec{a}^f$, we have:
			\begin{nalign}\label{eq:en:pastestimate_inhom}
				\sup_{u\in (u_\infty,u_0)}\norm{(v\pv,\rho_-^{1/2}\sl)\Vb^k\psi}_{\rho_-^{a_-}\rho_0^{a_0}\Hb^{;0}(\outcone{}^-)}+\sup_{v\in(v_0,v_\infty)}\norm{(u\pu,\rho_-^{1/2}\sl)\Vb^k\psi}_{\rho_-^{a_-}\rho_0^{a_0}\Hb^{;0}(\Cbar_v^-)}\\
				+\norm{\Ve \psi}_{\rho_-^{a_-}\rho_0^{a_0-}\Hb^{;k}(\D^-)}\lesssim \norm{rf}_{\rho_-^{a^f_-}\rho_0^{a^f_0}\Hb^{;k}(\D^-)}+\norm{\Vb^k \psi}_{\rho_-^{a_-}\Hb^{;1}(\incone)}.
			\end{nalign}
			
			\item \label{item:en:pastestimate_inhom_0th} Improved slow decay with loss of regularity: For $\psi^{\incone}=0$ and $\vec{a}^f$ satisfying \cref{eq:en:admissible_inhom2}, i.e.~dropping the extra restriction of $a_0\leq a^f_--1.5$ in \cref{eq:en:admissible_inhom1} for $a_-^f\in(1,3/2]$, we have:
            \begin{equation}\label{eq:en:pastestimate_inhom_0th}
				\norm{\psi}_{\rho_-^{a_-}\rho_0^{a_0}\Hb^{;k}(\D^-)}\lesssim \norm{rf}_{\rho_-^{a^f_-}\rho_0^{a^f_0}\Hb^{;k}(\D^-)}+\norm{\Vb^k \psi}_{\rho_-^{a_-}\Hb^{;1}(\incone)}.
			\end{equation}			
			\item No incoming radiation: \label{item:en:pastestimate_Hbt} \cref{eq:en:pastestimate_hom,eq:en:pastestimate_inhom,eq:en:pastestimate_inhom_0th} also hold upon changing at the same time spacetime norms from $\Hb(\D^-)$ to $\Hbt(\D^-)$, as well as the $\Vb$ to $\Vbt$ on the hypersurface norms, including those of the initial data. 
		\end{enumerate}
		
	\end{prop}
	
	\begin{rem}\label{rem:en:linearity_combine_estimates}
		Using linearity, we can merge the estimates in \cref{prop:en:pastestimate} to apply to  a solution that has both initial data and forcing.
        Furthermore, notice that in, say, \cref{eq:en:pastestimate_hom}, we estimate the LHS against $\Vb^k$ control for $\psi$ along~$\incone$. We can re-express this control directly in terms of $f$ and the data $\psi^{\incone}$ via integration of the equation along $\incone$. By nature of the characteristic IVP, this will lose derivatives. We do this in~\cref{lemma:en:recovering_initial_data}.
	\end{rem}
    
	\begin{rem}\label{rem:en:homogeneous_improvement}
		Let us note that, for the homogeneous case, we actually prove something stronger than \cref{eq:en:pastestimate_hom}; namely:
		\begin{nalign}\label{eq:en:pastestimate_hom:stronger}
			\norm{(1,\rho_-^{-1/2+}v\pv,\rho_-^{-1/2+}vr\pv r^{-1},u\pu,\sl)\psi}_{\rho_-^{a}\rho_0^{a-}\Hb^{;k}(\D^-)}\lesssim \norm{\Vb^k\psi}_{\rho_-^{a}\Hb^{;1}(\incone)}.
		\end{nalign}
		We do not know if this improvement plays any role for perturbations of $\Box_\eta$, and we don't make use of it in this paper.
	\end{rem}

Combining this estimate with \cref{prop:en:futureestimate} gives:
	\begin{prop}\label{prop:en:main}
        Let $\vec{a},\vec{a}^f$ be admissible.
		Under the assumptions of \cref{prop:en:pastestimate}, we have
		\begin{multline}\label{eq:model:estimate}
		\sup_{u\in (u_\infty,u_0)}\norm{(v\pv,\rho_{\scri}^{1/2}\sl)\Vb^k\psi}_{\rho_0^{a_0}\rho_+^{a_+}\Hb^{;0}(\outcone{}^+)}+\sup_{v\in(v_0,v_\infty)}\norm{(u\pu,\rho_{\scri}^{1/2}\sl)\Vb^k\psi}_{\rho_0^{a_0}\rho_+^{a_+}\Hb^{;0}(\Cbar_v^+)}	+\norm{\Ve  \psi}_{\Hb^{\vec{a};k}(\D)}
			\\
   \lesssim \norm{rf}_{\Hb^{\vec{a}^f;k}(\D)}+\norm{\Vb^k\psi}_{\rho_-^{a_-}\Hb^{;1}(\incone)}			
		\end{multline}
	\end{prop}
	\begin{proof}[Proof of \cref{prop:en:main}]
		We show that \cref{prop:en:futureestimate} and \cref{prop:en:pastestimate} combined imply \cref{eq:model:estimate}: We let $\chi\in\A{\phg}^{\vec{0}}(\D)$ be a cutoff  such that $\chi=1$ on $\D^-\setminus\D^+$ and such that $\chi=0$ on $\D^+\setminus\D^-$ (cf.~\cref{fig:D-b}).
		Taking $|u_0|$ sufficiently large, we  get that $\D^-\supset\incone$ such that for $\phi_{+}:=(1-\chi)\phi=:\phi-\phi_-$, $\phi_+$ vanishes near $\incone$.

We now apply \cref{prop:en:pastestimate} to $\phi_-$. Due to the support properties of $\phi_-$, this implies \cref{prop:en:main} for $\phi_-$. Similarly, we have the desired control on $\phi$ restricted to $\D^-$.
On the other hand, to also control $\phi_+$, we can write
		\begin{equation}
			\Box_\eta\phi_+=(1-\chi)f-\Box_\eta\chi\cdot \phi-2\partial\chi\cdot\partial\phi.
		\end{equation}
		Notice that the RHS is supported entirely in $\D^+$, so we can apply \cref{eq:en:futureestimate} to obtain

		\begin{equation}
			\norm{\Ve \psi_+}_{\rho_0^{a_0}\rho_+^{a_+}\Hb^k(\D^{+})}
			\lesssim \norm{rf}_{\rho_0^{a_0^f}\rho_+^{a_+^f}\Hb^{k}(\D^{+})}+\norm{\Ve \psi}_{\rho_0^{a_0}\Hb^{k}(\D^{-}\cap \D^{+})},
		\end{equation}
		where we have also used that $\partial\chi\in \A{b,\phg,b}^{\infty,\mindex{1},\infty}(\D)$ (and  that $\Box_{\eta}\chi\in\A{b,\phg,b}^{\infty,\mindex{2},\infty}(\D)$) is supported away from $\scrim$ and~$\scrip$). 
		Finally, we control the term $\norm{\Ve\psi}_{\rho_0^{a_0}\Hb^{k}(\D^{-}\cap \D^{+})}$ via the already obtained estimate restricted to $\D^-$.
	\end{proof}
	We next give the proof of \cref{prop:en:pastestimate}. In order to get the reader used to the notation, we give a lot of detail, and work with integration by parts. While this makes it easier to check computations line by line, the reader may prefer to alternatively directly jump to \cref{app:sec:current_computations:mink}, where a proof via currents is provided.
	\begin{proof}[Proof of \cref{prop:en:pastestimate}]
		Throughout the proof, we will use the weight functions $r\sim ~\abs{t}\sim ~\abs{u}$ interchangeably (recall that we work in $\D^-$. Furthermore, we recall from \cref{lemma:notation:coordinates} that $\frac{\dd \rho_-}{\rho_-}\frac{\dd \rho_0}{\rho_0}\sim \frac{\dd u}{|u|}\frac{\dd v}{v}$, and that $(\rho_0\partial_0 \psi)^2+(\rho_-\partial_-\psi)^2\sim (u\pu\psi)^2+(v\pv\psi)^2$.
		
		\textbf{Proof of \cref{item:en:pastestimate_hom}:} Let us start with $k=0$. We directly prove the stronger \cref{eq:en:pastestimate_hom:stronger}.
		We perform an energy estimate in~$\D^-$, which is bounded to the past by the two truncated null cones  $\C^-_{u_\infty},\Cbar^-_{v_0}$, and to the future by two (possibly empty) truncated null cones $\C^-_{u_1},\Cbar_{v_1}^-$ and the  spacelike $\{-u=\delta v\}\cap\D^-$. 

		Letting $V=v^{-\epsilon}|t|^{2a+1}\partial_t$, for $\epsilon>0$, we now apply $V$ as a multiplier for the twisted energy momentum tensor~$\tT$ and integrate in $\D^-$.
        Less geometrically, this corresponds to multiplying $P_{\eta}(r\phi)$ by $V(r\phi)$ and integrating by parts. 
		Since $V$ is timelike, the corresponding energy estimate will create positive contributions on the future boundaries, the contributions on the past boundaries being controlled by data. 
		We thus obtain:
		\begin{equation}\label{eq:en:pastlemma:proof1}
			\begin{gathered}
				\sup_{u_\infty\leq u\leq u_0}\int_{\outcone{}^-}\abs{t}^{2a+1}v^{-\epsilon}(\pv\psi)^2 \dd\tilde{\mu} +\sup_{v_\infty\geq v\geq v_0}\int_{\Cbar^-_{v}}\abs{t}^{2a+1}v^{-\epsilon}\Big((\pu\psi)^2+r^{-2}|\sl\psi|^2\Big)\dd\tilde{\mu}\\
                +
				(2a+1)\int_{\D^{-}} \abs{t}^{2a}v^{-\epsilon}\Big((\pv\psi)^2+(\pu\psi)^2+r^{-2}|\sl\psi|^2\Big)\dd\tilde{\mu}+
				\epsilon\int_{\D^{-}} \abs{t}^{2a+1}v^{-1-\epsilon}\Big((\pu\psi)^2+r^{-2}|\sl\psi|^2\Big)\dd\tilde{\mu}\\
				\leq 
				\int_{\Cbar_{v_0}} \abs{t}^{2a+1}\Big((\Lbar \psi)^2+r^{-2}|\sl\psi|^2\Big)\dd\tilde{\mu}
			\end{gathered}
		\end{equation}
		The bulk term on the LHS has a $|t|^{2a}$-weight near $\scrim$, and the estimate \cref{eq:en:pastlemma:proof1} is enough to deduce control over the $u\pu,v\pv,\sl$ terms on the LHS of \cref{eq:en:pastestimate_hom}, for $a=a_->-1/2,\epsilon<a_--a_0$.
        For $a_-=-1/2$ we lose control over $v\pv,\sl$ derivatives (we still control $u\pu$ derivatives).
		In order to also control these terms for $a_-=-1/2$, we integrate the boundary terms to get
		\begin{equation}\label{eq:en:pastlemma:proof2}
			\begin{gathered}
				\int_{\D^{-}}\frac{\abs{t}^{2a+1}}{v^{1+\epsilon}} r^{-2}|\sl\psi|^2\dd\tilde{\mu}\lesssim\sup_{v\in[v_0,v_\infty]}\int_{\Cbar^{-}_{v}}\abs{t}^{2a+1}r^{-2}|\sl\psi|^2\dd\tilde{\mu} \int_{v_0}^\infty \frac{\dd v}{v^{1+\epsilon}}\lesssim_\epsilon\text{RHS \cref{eq:en:pastlemma:proof1}}\\
                \int_{\D^{-}}\frac{\abs{t}^{2a+1}}{u^{1+\epsilon}} |\pv\psi|^2\dd\tilde{\mu}\lesssim\sup_{u\in[u_0,u_\infty]}\int_{\outcone{}^-}\abs{t}^{2a+1}|\pv\psi|^2\dd\tilde{\mu} \int_{u_0}^\infty \frac{\dd u}{u^{1+\epsilon}}\lesssim_\epsilon\text{RHS \cref{eq:en:pastlemma:proof1}}.
			\end{gathered}
		\end{equation}
		On the other hand, we now have, for $\rho_0=1/v$ and $\rho_-=v/|u|$, using that $|t|\sim r \sim |u|$:
		\begin{nalign}
			\text{LHS}\cref{eq:en:pastlemma:proof1}+\text{LHS}\cref{eq:en:pastlemma:proof2}
			& \gtrsim_{\epsilon,q} \int_{\D^{-}}\frac{|t|^{2a+1}}{v^{1+\epsilon}}\Big((\pu \psi)^2+r^{-2}|\sl\psi|^2\Big)+|t|^{2a-\epsilon}(\pv\psi)^2\dd \tilde{\mu}\\
			&\gtrsim \int_{\D^{-}}\frac{\rho_0^{1+\epsilon}}{(\rho_0\rho_-)^{2a+1}} \Big((\pu \psi)^2+(\rho_0\rho_-)^2|\sl\psi|^2\Big)       +\frac{1}{(\rho_0\rho_-)^{2a-\epsilon}}(\pv \psi)^2) \frac{\dd \mu_\b}{\rho_0^2\rho_-}\\
			&\gtrsim \int_{\D^{-}}\frac{1 }{\rho_0^{2a-\epsilon}\rho_-^{2a}}\Big(\frac{(\pu\psi)^2}{\rho_0^2\rho_-^2}+\rho_-^{-1+\epsilon}\frac{(\pv\psi)^2}{\rho_0^2}+|\sl\psi|^2
			\Big)\dd \mu_{\b}
		\end{nalign}
		Setting $a=a_-$ and $a-\epsilon=a_0$, this already concludes \cref{eq:en:pastestimate_hom} for differentiated quantities. 
  To obtain control over the 0th order term as well, we also run an energy estimate for $\T$ with the same multiplier, but with volume form now containing an $r^2$-weight. Again, less geometrically, this corresponds to multiplying $\Box_\eta\phi$ with $r^2V\phi$. The result is
		\begin{nalign}
			\int_{\D^{-}}\frac{r^2 }{\rho_0^{2a-\epsilon}\rho_-^{2a}}\Big(\frac{(\pu\phi)^2}{\rho_0^2\rho_-^2}+\rho_-^{-1}\frac{(\pv\phi)^2}{\rho_0^2}
			\Big)\dd \mu_{\b}\lesssim \norm{r\phi}_{\rho_-^{a}\Hb^{;1}(\incone)}^2.
		\end{nalign}
		
		\textit{Higher order estimates:}
		To obtain the corresponding estimates for commutations with $\sl$, we can simply commute with the angular Killing vector fields.
		To upgrade the result to general $k$, we need to additionally commute with the other Killing vector fields, i.e.~Lorentz boosts and scaling. Notice, however, that e.g.~commuting with $S=u\pu+v\pv$ requires control over transversal derivatives along data, which is why we provide the $\Vb^k$ commuted quantities on the RHS of \cref{eq:en:pastestimate_hom}.

		\textbf{Proof of \cref{item:en:pastestimate_inhom}, Case 1 ($a_0\leq \min(a_-^f-1.5,a_0^f-2)$):} We first prove this statement for $k=0$ and for $a_0\leq \min(a_-^f-1.5,a_0^f-2)$ and consider two different cases:
		
  \textit{Case 1a ($a_-^f$ dominant):}
	Consider the case where  $a_-^f-1.5\leq a_0^f-2$. We prove the borderline estimate where $a_-=a_-^f-1>0$; the full range of admissible $\vec{a}$ and $\vec{a}^f$ then follows a fortiori.
 
		We start with a $\tilde{\T}$ energy estimate with $\abs{u}^{2a_-}v^{-\epsilon}\partial_v$ multiplier\footnote{The $v^{-\epsilon}$ factor is necessary to control the bulk spherical derivatives. Alternatively, one could integrate the $\Cbar$ surface terms in $v$ to obtain the same control.} to get
		\begin{nalign}\label{eq:en:bulka-}
			\sup_u \norm{v\pv \psi}^2_{\Hb^{a_-,a_0;0}(\outcone{}^-)}+\sup_v\norm{\rho_-^{1/2}\sl\psi}^2_{\Hb^{a_-,a_0;0}(\Cbar_v^-)}+\norm{(v\pv,\rho_-^{1/2}\sl)\psi}^2_{\Hb^{a_-,a_0;0}(\D^-)}\lesssim-\int_{\D^-} rf\abs{u}^{2a_-}v^{-\epsilon}\pv \psi \dd\tilde\mu\\
			\lesssim c^{-1}\norm{rf}^2_{\Hb^{a^f_-,a^f_0;0}(\D^-)}+c\norm{v\pv \psi}^2_{\Hb^{a_-,a_0;0}(\D^-)}
		\end{nalign}
		Choosing $c$ sufficiently small concludes the $v\pv,\sl$ part of the estimate for $k=0$.
		Next, we use a $\frac{\abs{u}^{2a_-+1}}{v}v^{-\epsilon}\pu$ multiplier for $\tilde{\T}$ to obtain
		\begin{nalign}
			\sup_v \norm{u\pu \psi}^2_{\Hb^{a_-,a_0;0}(\Cbar_v^-)}+\sup_u\norm{\rho^{1/2}_-\sl \psi}^2_{\Hb^{a_-,a_0;0}(\outcone{}^-)}+\norm{u\pu \psi}^2_{\Hb^{a_-,a_0;0}(\D^-)}\lesssim -\int_{\D^-} rf\frac{\abs{u}^{2a_-+1}}{v}v^{-\epsilon}\pu \psi \dd\tilde{\mu} \\
			+\int_{\D^- } \frac{\abs{u}^{2a_-}}{vr^2}v^{-\epsilon}\abs{\sl \psi}^2 \dd\tilde \mu\lesssim\norm{\rho_-^{1/2}\sl \psi}_{\Hb^{a_-,a_0;0}(\D^-)}^2+ c^{-1}\norm{rf}^2_{\Hb^{a^f_-,a^f_0;0}(\D^-)}+c\norm{u\pu \psi}^2_{\Hb^{a_-,a_0;0}(\D^-)}.
		\end{nalign}
		This completes all but the zeroth order estimate. For the latter, we apply a $\abs{u}^{2a_-}v^{-\epsilon}\partial_t$ multiplier for $\T$ to get
		\begin{nalign}
			\sup_u \norm{rv\pv \phi}^2_{\Hb^{a_-,a_0;0}(\outcone{}^-)}+\norm{rv\pv\phi}^2_{\Hb^{a_-,a_0;0}(\D^-)}\lesssim-\int_{\D^-}f\abs{u}^{2a_-+1}v^{-\epsilon}\partial_t\phi \dd\mu \lesssim \norm{rf}_{\Hb^{a^f_-,a^f_0;0}(\D^-)}^2+\norm{(v\pv ,u\pu)\psi}_{\Hb^{a_-,a_0;0}(\D^-)}^2.
		\end{nalign}
		Finally, we apply a $V=\abs{u}^{2a_-+1}/v^{1+\epsilon}\pu$ multiplier for~$\T$ to get
  
		\begin{nalign}
			\sup_v\norm{ru\pu\phi}^2_{\Hb^{a_-,a_0;0}(\Cbar_v)}+\norm{ru\pu\phi}^2_{\Hb^{a_-,a_0;0}(\D^-)}\lesssim\int_{\D^-}(|f|+\frac{2}{r}|\pv\phi|)|V\phi| \dd\mu \\
			\lesssim c^{-1}\Big(\norm{rf}_{\Hb^{a_-^f,a_0^f;0}}^2+\norm{rv\pv\phi}^2_{\Hb^{a_-,a_0;0}(\D^-)}\Big)+c\norm{ru\pu\phi}^2_{\Hb^{a_-,a_0;0}(\D^-)}.
		\end{nalign}
		Combining with the previous estimate for $\psi=r\phi$, the above yields the zeroth order control.
		
		\textit{Case 1b ($a_0^f$ dominant):} In case $q=1+2(a_-^f-a_0^f)>0$, we can repeat the proof of Case 1a by adding a $v^{-q}$ factor to all the multipliers (i.e.~$v^{-q-\epsilon}\abs{u}^{2a_-}\pu,\,v^{-q-1-\epsilon}\abs{u}^{2a_-+1}\pu$). 
Since $v>0$ in $\D^-$, this only yields additional positive bulk terms for all estimates, while modifying the weight towards $I^0$ in all terms by $q/2$.

        \textit{Higher order estimates:} Commutation works the same way as for \cref{item:en:pastestimate_hom}. Notice that, while $\psi^{\incone}=0$, higher-order commuted quantities such as $\Vb^k\psi$ do not vanish along $\incone$. Thus, for commuted equations, we split up the solution into a part with no inhomogeneity but nontrivial data and apply \cref{item:en:pastestimate_hom}, and into a part with trivial data and nontrivial inhomogeneity, for which we proceed as above.
 Note that this is the reason why we have to include the extra term on the RHS of~\cref{eq:en:pastestimate_inhom}.

		\textbf{Proof of \cref{item:en:pastestimate_inhom}, Case 2 ($a^f_->3/2$):}
	Notice that while the previous proof made no restrictions on the range $a^f_->3/2$, we have claimed something stronger in the range $a_-^f>3/2,$ namely that $a_0$ can be taken to be $a_0^f-2$.

  \textit{Case 2a ($a_-^f+1\leq a_0^f$):}
		Consider first the case $a_-^f+1\leq a_0^f$.
		We use $V=\abs{u}^{2a_-+1}v^{-\epsilon}\pu$ multiplier 
  for $\tilde{\T}$, again set $a_-=a_-^f-1$ (and extend to smaller $a_-$ a fortiori), 
  and then exploit that $2a_-+1>2$ to get
		\begin{nalign}
			\sup_v \norm{u\pu \psi}^2_{\Hb^{a_-,a_0;0}(\Cbar_v^-)}+\sup_u\norm{\rho^{1/2}_-\sl \psi}^2_{\Hb^{a_-,a_0;0}(\outcone{}^-)}+\norm{(u\pu,\rho_-^{1/2}\sl)\psi}^2_{\Hb^{a_-,a_0;0}(\D^-)}\lesssim_{\vec{a}} -\int_{\D^-} rf\abs{u}^{2a_-+1}\pu \psi \dd\tilde{\mu}\\
			\lesssim c^{-1}\norm{rf}^2_{\Hb^{a^f_-,a^f_0;0}(\D^-)}+c\norm{u\pu \psi}^2_{\Hb^{a_-,a_0;0}(\D^-)}.
		\end{nalign}
		We continue with a $\abs{u}^{2a_-}v^{1-\epsilon}\pv$ estimate for $\tilde{\T}$ to get
		\begin{nalign}
			\sup_u \norm{v\pv \psi}^2_{\Hb^{a_-,a_0;0}(\outcone{})}+\sup_v\norm{\rho_-^{1/2}\sl \psi}^2_{\Hb^{a_-,a_0;0}(\Cbar_v)}+\norm{v\pv \psi}^2_{\Hb^{a_-,a_0;0}(\D^-)}\lesssim-\int_{\D^-} rf\abs{u}^{2a_-}v\pv \psi\dd\tilde\mu\\
			+\int_{\D^-}\frac{\abs{u}^{2a_-+2}}{vr^2}\abs{\sl \psi}^2 \dd\tilde\mu \lesssim\norm{\rho_-^{1/2}\sl \psi}_{\Hb^{a_-,a_0;0}}^2+ c^{-1}\norm{rf}^2_{\Hb^{a^f_-,a^f_0;0}(\D^-)}+c\norm{v\pv \psi}^2_{\Hb^{a_-,a_0;0}(\D^-)}
		\end{nalign}
		To obtain control over the 0th order terms, we use  a $V=\abs{u}^{2a_-+1}v^{-\epsilon}\pu$ multiplier for $\T$ to get
		\begin{nalign}
			\sup_v\norm{ru\pu \phi}^2_{\Hb^{a_-,a_0;0}(\Cbar_v^-)}+\norm{ru\pu \phi}^2_{\Hb^{a_-,a_0;0}(\D^-)}\lesssim \int_{\D^-}-fV\phi+\frac{2}{r}\pv\phi V\phi \dd\mu \\
			\leq \int_{\D^-} -fV\phi+\frac{2}{r^2}(\pv (r\phi)-\phi) V\phi \dd\mu\leq\int_{\D^-}  (r^2f+2\pv (r\phi))V\phi\dd\tilde{\mu}-\norm{\phi}^2_{\Hb^{a_0,a_-;0}(\outcone{0})}-\int_{\D^-}\abs{u}^{2a_-}\abs{\phi}^2 \dd\tilde{\mu} .
		\end{nalign}
		Bounding the first term on the right hand side with Cauchy--Schwarz and using the already obtained estimate for $\pv (r\phi)=\pv\psi$ yields the required estimate.
		
		\textit{Case 2b ($a_-^f+1>a_0^f$):} In the case $q=2+2(a_-^f-a_0^f)>0$, we introduce as before a $v^{-q}$-weight for all estimates and correct the norms towards $I^0$ by $q/2$.

  Control over higher order terms is achieved as before.

        \textbf{Proof of \cref{item:en:pastestimate_inhom_0th}:} 
        First, let us note that this estimate only improves over \cref{item:en:pastestimate_inhom} near $I^0$.
        Hence, using a cutoff function, we may assume without loss of generality that $\supp(\phi,f)\subset\{v>v_0+1\}$.
        Furthermore, it suffices to consider the case $a_-^f+0.5<a_0^f$, otherwise \cref{item:en:pastestimate_inhom} already yields the result.
        
		We consider the time integral of the inhomogeneity and of $\phi$
        \begin{equation}\label{eq:en:timeinverse}
            T^{-1}(rf)=\int_{2v_1-r}^t(rf)|_{(t',r,\omega)} \dd t',\qquad T^{-1}(r\phi)=\int_{2v_1-r}^t(r\phi)|_{(t',r,\omega)} \dd t'.
        \end{equation}
        Using \cref{item:ode:2d}, we get that $T^{-1}(rf)\in\Hb^{a_-^f,\min(a_0^f-1,a_-^f);k}$.
        Hence, for $T^{-1}(rf)$, we have that the weight near $I_0$ dominates, i.e.~it holds that $a_-^f+0.5>\min(a_0^f-1,a_-^f)$.
        Using that $\Box$ commutes with $T^{-1}$, we can apply \cref{item:en:pastestimate_inhom} to conclude that $\Ve T^{-1}(r\phi)\in\Hb^{a^f_--1,\min(a_0^f-1,a_-^f)-2-;k}$.
		Taking a time derivative improves the weight towards $I^0$ by 1 and thus implies the result (at a loss of one derivative).
		
		\textbf{Proof of \cref{item:en:pastestimate_Hbt}:} Note that we only need to improve $v\pv\mapsto r\pv$ commutations.
		We do this by commuting with $T$ and increasing $a_-,a_-^f$ by 1 each time and applying the previous estimates.
        For instance in \cref{item:en:pastestimate_hom}, we know that $\vec{a}$ admissible implies that $\vec{a}+1$ is also admissible, so we have for $j\leq k$
        \begin{equation}\label{eq:en:past_proof_T_commuted}
				\norm{\Ve T^jr\phi}_{\rho_-^{a_-+j}\rho_0^{a_0+j}\Hb^{;k-j}(\D^-)}\lesssim \sum_{i\leq k-j} \norm{T^{i+j}(r\phi)}_{\rho_-^{a_-+j}\Hb^{;k-i+1}(\incone)}.
		\end{equation}
		Summing \cref{eq:en:past_proof_T_commuted} with $0\leq j\leq k$ yields the required estimate, namely: 
		\begin{equation}
			\norm{\Ve r\phi}_{\rho_-^{a_-}\rho_0^{a_0}\Hbt^k(\D^-)}\lesssim\norm{\Vbt^k r\phi}_{\Hb^{a_-;1}(\incone)}.
		\end{equation}
	\end{proof}

  We also record the following basic lemma, which includes incoming radiation. It will later be used in the proof of uniqueness of scattering solutions.
    \begin{lemma}\label{lemma:en:past_estimate_with_incoming_data_L2}
        Let $\phi$ be the solution to \cref{eq:en:characteristicIVP}.
        Let $a_-<0$, $a_0<a_--1/2$ and let $\vec{a}^f$ satisfy $a^f_->1$, $a_0^f\geq a_0+2$.
        Then 
        \begin{equation}\label{eq:en:past_estimate_with_incoming_data_L2}
            \norm{\Ve\psi}_{\Hb^{\vec{a};0}(\D^-)}\lesssim \norm{rf}_{\Hb^{\vec{a}^f;0}(\D^-)}+\norm{(u\pu,\rho_-^{1/2}\sl,1)\psi^{\incone}}_{\Hb^{0;0}(\incone)}+\norm{(v\pv,\rho_-^{1/2}\sl,\frac{v}{r})\psi^{\outconeFar}}_{\Hb^{a_0;0}(\outconeFar)}.
        \end{equation}
    \end{lemma}
    \begin{proof}
        We use the multipliers $v^{-\epsilon}\partial_v,v^{-\epsilon}\partial_t,\abs{u}v^{-1-\epsilon}\pu$ as in the proof of \cref{item:en:pastestimate_inhom} (thus setting $a_-=0$).
        This no longer yields bulk control over $(\pv\psi)^2$ in the analogue of \cref{eq:en:bulka-}. We recover this bulk control by integration of the boundary terms, thereby loosing an arbitrary small decay factor and only recovering $a_-=0-$.
    \end{proof}

 \subsubsection{Extended estimates in the case of weakly decaying data}\label{sec:en:weakdecay}
    We next turn our attention to extending \cref{prop:en:pastestimate} to the case when $f,\psi^{\incone}$ decay slower than permitted by \cref{def:en:admissible:f}.
    In this case, we have to assume that $f$ is in a no incoming radiation space, i.e. $f\in\Hbt$.
    In fact, we can also construct scattering solutions in the case where $f$ is weakly decaying and has a polyhomogeneous expansion towards~$\scrim$ without having to assume no incoming radiation compatibility, but this needs no additional energy estimates and we'll prove it directly in~\cref{thm:scat:weak_polyhom}.

	\begin{lemma}[Energy estimate for weak decay and no incoming radiation]\label{lemma:en:energy_no_incoming}
  Let  $\phi$ be the solution to \cref{eq:en:characteristicIVP} under the assumption \cref{eq:en:assumption0}.
		Let $\vec{a}$ be admissible, but with the condition $a_-\geq-1/2$ eased to $a_->-m+1/2$ for some $m\in\N$.
		Let $\vec{a}^f=\vec{a}+(1,2)$, $rf\in\Hbt^{\vec{a}^f;k+m}(\D^-)$ and $\psi^{\incone}\in\Hb^{a_-;\infty}(\incone)$.
		Then we have 
		\begin{nalign}
			\sup_{u\in (u_\infty,u_0)}\norm{(v\pv,\rho_{\scri}^{1/2}\sl) \Vbt^k \psi}_{\Hbt^{\vec{a};0}(\outcone{})}+\sup_{v\in(v_0,v_\infty)}\norm{(u\pu,\rho_{\scri}^{1/2}\sl) \Vbt^k\psi}_{\Hbt^{\vec{a};0}(\Cbar_v)}+\sum_{n\leq m} \norm{(rT)^n\psi}_{\Hbt^{\vec{a};k}(\D^-)}\\
            \lesssim\norm{rf}_{\Hbt^{\vec{a}^f;k+m}(\D^-)}+\norm{\Vbt^{k+m}\psi|_{\incone}}_{\Hb^{a_-;1}(\incone)}
		\end{nalign}
	\end{lemma}
	\begin{proof}
		We note that $T^m\psi$ solves $P_\eta T^m\psi=T^mrf\in\Hbt^{\vec{a}+(1+m,2+m);k}(\D^-)$, so we can apply \cref{prop:en:pastestimate} to obtain control over $T^m\psi$.
		We integrate in time via \cref{corr:ODE:du_dv} to obtain control over $\psi$.
	\end{proof}

\subsection{Energy estimates II: Inclusion of long-range potentials}\label{sec:en:longrange}
In hopes of making the computations easier to follow, we proved \cref{prop:en:main} by  integration by parts    in $u$ and $v$. 
In \cref{app:sec:current_computations:mink}, we present a more geometric approach to proving \cref{prop:en:main} using currents and the divergence theorem. 
Using this more geometric approach, it is straight-forward to extend the energy estimates to hold for a class of modified equations as well (see already \cref{sec:current:non_short_range}), the prototype of which is $\Box_\eta\phi+c\phi/r^2=f$ for any $c$. 
More generally, we will study the problem
\begin{equation}\label{eq:en:longwave}
    \Box_\eta\phi -V_L\phi=f, \qquad \psi|_{\incone}=\psi^{\incone},\qquad \psi|_{\outconeFar}=\psi^{\outconeFar},
\end{equation}
where the class of admissible potentials $V_L$  is defined below:
\begin{defi}\label{def:en:long_range}  
Fix $a_-\geq -1/2$.
Let $\epsilon_{\pm}>0$, and let $f^L_{\pm}(\rho_{\pm},\omega)\in \Hb^{\epsilon_{\pm}-1}([0,1)_{\rho_{\pm}}\times S^2)$ such that $f_{\pm}\equiv 1$ on $\D_{\mp}\setminus \D_{\pm}$.
\begin{enumerate}[label=\textbf{PL}]
		\item \textbf{a)}  Then we call $V_L=\frac{f^L_{-}(\rho_-,\omega)f^L_{+}(\rho_+,\omega)}{r^2}\Ve$ an admissible long-range potential if $a_-+1+\epsilon_->1$.
  \item \textbf{b)} Define $f^L_{-,\b}$ such that $f^L_{-,\b}\Vb=f^L_{-}\Ve$. Then we call $V_L=\frac{f^L_{-}(\rho_-,\omega)f^L_{+}(\rho_+,\omega)}{r^2}{\Ve}$  a long-range potential compatible with no incoming radiation if  $f^L_{-,\b}(\rho_-,\omega)\in \A{phg}^{\mindex0}([0,1)_{\rho_-}\times S^2)$ and $a_-+1+\epsilon_->1$.
	\end{enumerate}
\end{defi}
Notice that long-range potential have weight $(1+\epsilon_-, 2, 1+\epsilon_+)$ (or $(2,2,1+\epsilon_+)$ in the case of no incoming radiation) towards the boundaries; thus, they are perturbative towards $\scrim$ and $\scrip$, but genuine modifications to $\Box_\eta$ towards $I^0$, (though below the principle symbol in differentiability).

We now show that the statements of the previous section carry over to long-range potentials:
\begin{cor}\label{cor:en:longrange1}
    \cref{prop:en:futureestimate}, \cref{prop:en:pastestimate} and, thus, \cref{prop:en:main} as well as \cref{lemma:en:energy_no_incoming} all also hold for \cref{eq:en:longwave} instead of \cref{eq:en:characteristicIVP}.
\end{cor}
\begin{proof}
For \cref{prop:en:futureestimate}, as well as for the $k=0$-part of \cref{item:en:pastestimate_inhom} of \cref{prop:en:pastestimate} and \cref{lemma:en:energy_no_incoming}, this follows by including a $w_{\pm}^{\bar{c}}$-weight in front of all multipliers of the relevant proofs, where $w_{+}=v/r$ and $w_-=|u|/r$, and where $\bar{c}$ is suitably large. The details are given in \cref{sec:current:non_short_range}; see, in particular, \cref{prop:current:long_range}. 

Since we will in this proof both refer to the exact statements \cref{item:en:pastestimate_hom,item:en:pastestimate_inhom} etc.~as well as their long-range analogues, we will use a subscript $L$ to denote the latter. For instance, we will refer to the long-range version of \cref{item:en:pastestimate_inhom} as $\cref{item:en:pastestimate_inhom}_L$.

    To prove the $k=0$ part of \cref{item:en:pastestimate_hom} for long-range potentials, we simply write $\phi=\phi_1+\phi_{\Delta}$, where only $\phi_1$ has nontrivial data, and where $\Box_\eta\phi_1=0$, $\Box_\eta\phi_{\Delta}-V_L\phi_{\Delta}=V_L\phi_1$. We apply \cref{item:en:pastestimate_hom} to $\phi_1$. We then treat $V_L\phi_1$ as error term and apply $\cref{item:en:pastestimate_inhom}_L$ with $k=0$ to $\phi_{\Delta}$. Now, in the case where $a_-+1+\epsilon_-\leq 3/2$, then, due to the case \cref{eq:en:admissible_inhom1}, $\cref{item:en:pastestimate_inhom}_L$ will only give us that $\Ve \psi_{\Delta}\in \Hb^{a_-,\min(a_-+\epsilon_--1/2-,a_0;0}(\D^-)$, which is not yet what we wish to prove. In order to fix this, we instead write $\phi_\Delta=\phi_{\Delta,1}+\phi_{\Delta,2}$, where $\phi_{\Delta,1}$ and $\phi_{\Delta,2}$ solve (each with trivial data)
    \begin{equation}\label{eq:savedhehe}
        \Box_\eta\phi_{\Delta,1}=V_L\phi_1,\qquad \Box_\eta\phi_{\Delta,2}-V_L\phi_{\Delta,2}=V_L\phi_{\Delta,1}.
    \end{equation}
We then apply \cref{item:en:pastestimate_inhom}, say, $n$ times, to the $\phi_{\Delta,1}$ to inductively get that $\Ve \psi_{\Delta,1}\in \Hb^{a_-,\min(a_-+n\epsilon_--1/2-,a_0;0}(\D^-)$ simply by reinserting this membership into \cref{eq:savedhehe} (note that this does not lose derivatives). Thus, for $n$ large enough, we get that $\Ve\psi_{\Delta,1}\in \Hb^{a_-,a_0;k}(\D^-)$. We then apply our already established $\cref{item:en:pastestimate_inhom}_l$ to  $\psi_{\Delta,2}$, which gives $\Ve\psi_{\Delta,2}\in \Hb^{a_-,a_0;0}(\D^-)$ and proves $\cref{item:en:pastestimate_hom}_L$ for $k=0$.

   We can then iteratively prove $\cref{item:en:pastestimate_inhom}_L$ and $\cref{item:en:pastestimate_hom}_L$ for higher $k$.
    
    In order to prove $\cref{item:en:pastestimate_inhom_0th}_L$, we note that while \cref{eq:en:longwave} no longer commutes with $T$, we can simply write $\phi=\sum_{1\leq i\leq n}\phi_i$, where for $n>i>1$
    \begin{equation}
        \Box_\eta \phi_1=f,\qquad \Box_\eta \phi_i=V_L\phi_{i-1},\qquad \Box_\eta\phi_n-V_L\phi_{n}=V_L\phi_{n-1},
    \end{equation}
    and where only $\phi_1$ has nontrivial data.
Applying \cref{item:en:pastestimate_inhom_0th} to $\phi_1$, we get that $V_L\psi_1\in \rho_-^{1+\epsilon_-}\rho_0^2\Hb^{a^f_{-}-1,\min(a_-^f-1,a_0^f-2)-;k}(\D^-)$.
Since we do not lose regularity at this step, we may iterate the procedure until we have sufficient decay towards $\scrim$ ($V_L\psi_{n-1}\in \rho_-^{2+\epsilon_-}\rho_0^2\Hb^{a^f_{-}-1,\min(a_-^f-1,a_0^f-2)-;k}(\D^-)$) and then apply $\cref{item:en:pastestimate_inhom}_L$ to $\phi_{n}$.

Finally, in order to prove $\cref{item:en:pastestimate_Hbt}_L$, we again note that while $T$ doesn't commute with the wave equation, 
\begin{equation}
    \Box_\eta T^n \phi+r^{-2}T^n(f^L_{-} \Ve\phi)= \Box_\eta T^n \phi+r^{-2}T^n(f^L_{-,\b} \Vb\phi)=T^nf, 
\end{equation}
we can make use of the condition that $f^L_{-,\b}\in \A{phg}^{\mindex0}([0,1)_{\rho_-}\times S^2)$, which implies that $T^m f^L_{-,\b}\in \A{phg}^{\mindex{m}}([0,1)_{\rho_-}\times S^2)$ for any $m\in\mathbb N$. (To see this, take the expressions \cref{eq:notation:partial_derivatives} and notice that $T\rho_-=|u|^{-1}(1+\rho_-)$.)
Thus, the proof of $\cref{item:en:pastestimate_Hbt}_L$ still goes through as before.
\end{proof}
Concerning the extended estimates in the case of weakly decaying data, we have
\begin{cor}\label{cor:en:longrange1.5}
Under the assumptions of \cref{lemma:en:energy_no_incoming}, but with $\phi$ a solution to \cref{eq:en:longwave} and with $rf\in \Hbt^{\vec{a}^f;k+m\cdot N}(\D^-)$ for $N=\lceil m/\epsilon_-\rceil+1$ instead, we have that
\begin{nalign}
			\sum_{n\leq m} \norm{(rT)^n\psi}_{\Hbt^{\vec{a};k}(\D^-)}\lesssim\norm{rf}_{\Hbt^{\vec{a}^f;k+mN}(\D^-)}+\norm{\Vb^{k+mN}\psi}_{\Hb^{a_-;1}(\incone)}.
		\end{nalign}
\end{cor}
\begin{proof}
    We write $\psi=\psi_0+(\sum_{i=1}^{N-1} \psi_i)+\psi_\Delta$, where we define $\psi_0$ to be the solution to $P_\eta\psi_0=f$ with data $\psi^{\incone}$, and, for $N> i\geq 1$:
   \begin{equation}
       P_\eta\psi_i=V_L\psi_{i-1}, \qquad P_\eta\psi_{\Delta}+V_L\psi_{\Delta}=V_L\psi_{N-1},
   \end{equation}
all with trivial data. 
We now first apply \cref{lemma:en:energy_no_incoming} to $\psi_0$ to obtain that $\Ve\psi_0\in \Hbt^{a_-,a_0;k+mN-m}(\D^-)$, which implies that $V_L\psi_1\in \Hbt^{a_-+1+\epsilon_-,a_0;k+mN-m}(\D^-)$. Inductively applying \cref{lemma:en:energy_no_incoming}, we obtain that for all $i\leq N-1$ 
\begin{equation}
   V_L\psi_i\in \Hbt^{a_-+1+i\epsilon_-;a_0+2;k+m (N-i)}(\D^-) \implies P_\eta \psi_{\Delta}-V_{L}\psi_{\Delta}\in \Hbt^{a_-+1+n\epsilon_-;a_0+2;k+m}(\D^-).
\end{equation}
But since $a_-+1+n\epsilon_-\geq a_-+1+m\geq 3/2$, we can now apply \cref{cor:en:longrange1} to deduce the result.

\end{proof}
\begin{rem}
    Note that the regularity here is not optimal, as we don't actually need to commute $m$ times in each step. 
\end{rem}

In summary, all results from \cref{sec:en:en} also hold for \cref{eq:en:longwave}. 
In the remainder of \cref{sec:en}, we will consider and prove estimates non-linear, \textit{short-range} perturbations of $\Box_\eta\phi=0$. Furthermore, we will prove scattering statements for such perturbations in \cref{sec:scat:scat}. Since these statements are all based on the estimates proved in \cref{sec:en:en}, all remaining statements of \cref{sec:en} and \cref{sec:scat:scat} extend to short-range perturbations of \cref{eq:en:longwave} as well. See already \cref{thm:scat:long}.

We conclude with two remarks:
\begin{obs}[Twisted long-range potentials]\label{rem:en:twistedlongrange}
Note that for $\phi$ solving
\begin{equation}\label{eq:long:twisted}
	(\Box_\eta+\frac{c_v}{r}\pv-\frac{c_u}{r}\pu)\phi=f,
\end{equation}
we can write $\bar{\phi}=\rho_+^{-c_u}\rho_-^{-c_v}\phi$, which satisfies
\begin{equation}
	(\Box_\eta+V)\bar{\phi}=\rho_+^{-c_u}\rho_-^{-c_v}f
\end{equation}
for some long-range potential $V$. In particular, we can still apply \cref{cor:en:longrange1,cor:en:longrange1.5} to $\bar{\phi}$ as above.
\end{obs}
\begin{rem}\label{rem:en:long_range}[Relevance of long-range potentials]
    	Let's consider for $c_u,c_v,c_0\in\R$, $c_{\Dl}\in\R_{\geq0}$ the operator
	\begin{equation}\label{eq:current:non-short-range-Qtilde}
		\tilde{Q}\phi:=\Big(\pu\pv+\frac{c_u}{r}\pu-\frac{c_v}{r}\pv+\frac{c_0}{r^2}-\frac{c_{\Dl}}{r^2}\Dl\Big)\phi=0.
	\end{equation}
	For $c_u=c_v=0$ and $c_0=0$ we recover the twisted wave operator $P_\eta$.
	Terms with $V=1/r^2$ come up naturally when studying the behaviour of charged scalar fields in the presence of nonzero total charge, see \cite{gajic_late-time_2024}.
	An immediate problem for the inclusion of such $V$, is that below a threshold value for $c_0$, the $T$ energy estimate for $\tilde{Q}$ is no longer coercive.
	This is indeed a problem when studying the solution in a neighbourhood of timelike infinity, see \cite{gajic_late-time_2024,hintz_linear_2023} for further discussion.
	In contrast, as shown by \cref{cor:en:longrange1}, there is never an issue with coercivity in the region~$\D$.
	
	The terms $c_u,c_v$ appear when studying the linearised vacuum Einstein equations in a double null gauge or a harmonic gauge, see \cite{dafermos_linear_2019,hintz_exterior_2023}. In fact, in the double null setting, these equations are called the Teukolsky equations; we will discuss these further in \cref{sec:Sch:ling}. 
	We remark, again, that when working in a neighbourhood of the timelike infinities, the specific first order terms in the Teukolsky equation (coming from $c_u$ and $c_v$) are  problematic; indeed, one of the core insights of \cite{dafermos_linear_2019} was that the Teukolsky equations can be transformed to equations without such terms. In contrast, as shown by the above, we can treat such first order terms \textit{directly} in the region 
$\D$.
\end{rem}

	\subsection{Energy estimates III: Perturbations of \texorpdfstring{$\Box_{\eta}\phi=f$}{the linear wave equation} with no incoming radiation}\label{sec:en:perturbations}
    In this section, we upgrade the energy estimates from \cref{sec:en:en} to short-range perturbations of $\Box_\eta$:
    \begin{equation}\label{eq:en:characteristicIVP_perturbed}	\begin{gathered}
			\Box_\eta \phi=\mathcal{P}[\phi]+f,\qquad \mathcal{P}[\phi]=V\phi +P_g[\phi]\phi+\mathcal{N}[\phi]\\
			\psi|_{\Cbar_{v_0}}=\psi^{\Cbar_{v_0}}\qquad \psi|_{\C_{u_\infty}}=\psi^{\C_{u_\infty}},
		\end{gathered}
	\end{equation}
 where $V$, $\mathcal{N}$, $P_g$ are potential,  semilinear and quasilinear perturbations, respectively. 
 We will define these perturbations in \cref{def:en:short_range} in \cref{sec:en:srdef} in such a manner that we can effectively treat them as inhomogeneities and such that they leave in tact the $\b$-smooth structure of the solutions. While this is straight-forward to define for linear perturbations, we note that, for nonlinear perturbations, the definition will depend on the decay of the initial data.
 
After proving some basic properties of these perturbations in \cref{sec:en:sh:properties}, we will prove estimates for the transversal derivatives along the initial cone in \cref{sec:en:sr:data}, and finally close the nonlinear energy estimates for \cref{eq:en:characteristicIVP_perturbed} in \cref{sec:en:sr:energy}.
 \subsubsection{Definition of short-range perturbations}\label{sec:en:srdef}
	 When referring to perturbations in general, or their coefficients, we will use $\mathcal{P}$ to denote  combinations of $V$, $\mathcal{N}$ and $P_g$-perturbations, all of which are defined below. We emphasise that $\mathcal{P}$ is always defined w.r.t.~an admissible $\vec{a}$.

	\begin{defi}\label{def:en:short_range}
		Fix admissible weight $\vec{a}=(a_-,a_0,a_+)$.
		\begin{enumerate}[label=\textbf{P\arabic*}]
			\item \label{item:en:potential_perturbation} \textbf{a)} We say that  $V=f_V\Vb$ for $f_V\in\rho_\scri^{\kappa/2}\Hb^{\vec{a}^V;\infty}(\D)$, where $\kappa=1$ for $\sl$ derivatives and $\kappa=0$ otherwise, is a \textbf{short-range potential}  perturbation if $\vec{a}+\vec{a}^V$ is a strongly admissible inhomogeneous weight relative to $\vec{a}$.\footnote{We could have equivalently written $V=f_V\Ve$ and avoided $\kappa$; the reason why we didn't will become transparent later.}
			
			\textbf{b)} Fix $\vec{a}^V=(a^V_0,a^V_0,a^V_+)$ such that $\vec{a}+\vec{a}^V-(1,0,0)$ is strongly admissible. 
            We say that $V$ is \textbf{compatible with the no incoming radiation condition} if $V=f_V\Vbt$ with $f_V\in\rho_+^{\kappa/2}\Hbt^{a^V_0,a^V_0,a^V_+;\infty}(\D)$, where $\kappa$ is as above.
            
			\item \label{item:en:semilinear_perturbation} \textbf{a)} Let $n^{\mathcal{N}}\in\N_{\geq 1}$. Let $\mathcal{N}[\phi]=\mathcal{N}[\phi,...,\phi]$ for $\mathcal{N}[\phi_1,...,\phi_n]:(\C^\infty(\D))^{n^{\mathcal{N}}}\to \C^\infty(\D)$  a finite sum of multi-linear functions of the form
			\begin{equation}\label{eq:en:short_range_nonlinear_form}
				\frac{1}{r}f_{\mathcal{N}} \prod_i \Vb(r\phi_i),\qquad f_{\mathcal{N}}\in\rho_\scri^{\kappa/2}\Hb^{a^\mathcal{N}_-,a^\mathcal{N}_0,a^\mathcal{N}_+;\infty}(\D),
			\end{equation}
			where $\kappa=0$ if none of the derivative terms contain $\sl$ and $\kappa=1$ otherwise.
			We say that $\mathcal{N}$ is a \textbf{short-range semilinear} perturbation relative to $\vec{a}$ if $n^{\mathcal{N}}\vec{a}+\vec{a}^{\mathcal{N}}$ is a strongly admissible inhomogeneous weight.
            We also allow for multiple $n_i^{\mathcal{N}}$ values, in which case $\mathcal{N}[\phi]=\mathcal{N}_1[\phi]+...+\mathcal{N}_m[\phi]$.
			
			\textbf{b)} Fix $\vec{a}^{\mathcal{N}}=(a^{\mathcal{N}}_0,a^{\mathcal{N}}_0,a^{\mathcal{N}}_+)$ such that $n^{\mathcal{N}}\vec{a}+\vec{a}^{\mathcal{N}}-(1,0,0)$ is strongly admissible. 
            We say $\mathcal{N}$ is \textbf{compatible with the no incoming radiation condition} if it is of the form (with $\kappa$ as above)
			\begin{equation}\label{eq:en:short_range_no_incoming_condition}
				\frac{1}{r}f_{\mathcal{N}}\prod_i (r\pv,u\pu,\sl,1)r\phi_i,\qquad f_{\mathcal{N}}\in\rho_+^{\kappa/2}\Hbt^{a^\mathcal{N}_0,a^\mathcal{N}_0,a^\mathcal{N}_+;\infty}(\D).
			\end{equation}
			\item \label{item:en:quasilinear_perturbation} \textbf{a)} Fix $n^{P_g}\in\mathbb N_{\geq1}$. Let
			\begin{equation}
				P_g[\phi]=\Omegalin[\phi]\partial_u\partial_v+\glin_{AB}[\phi]\sl^A\sl^B
			\end{equation}
			such that $(uv)^{-1}r^{-1}\Omegalin,r^{-1}\abs{\glin}$ are all of the form 
            \cref{eq:en:short_range_nonlinear_form}, where the corresponding $n^{\Omegalin}$ and $n^{\glin}$ are given by $n^{P_g}-1$, where $\kappa_{\glin}=2$ and  where, finally, $\kappa_{\Omegalin}=1$ if $f_{\Omegalin}$ contains $\sl$ term, 0 otherwise.
            If the corresponding prefactors satisfy $f_{P_g}\in\rho^{\kappa/2}\Hb^{\vec{a}^{P_g}}(\D)$ such that $n^{P_g}\vec{a}+\vec{a}^{P_g}$ is  strongly admissible, then we say that $P_g$ is a \textbf{short-range quasilinear} perturbation.
            Here, $\abs{\glin}$ denotes the standard Riemannian norm of $\glin$ on the unit sphere.
            
			\textbf{b)} We say that  $P_g$ is \textbf{compatible with  the no incoming radiation condition} if $(ru)^{-1}r^{-1}\Omegalin$ and $\rho_+^{-1}r^{-1}\abs{\glin_{AB}}$ are, instead, of the form \cref{eq:en:short_range_no_incoming_condition} with $\kappa=0$. 
		\end{enumerate}
    
		Let $n^V=1$. We define the \textbf{gap of $\mathcal{P}$} via:
		\begin{equation}
			\delta(\mathcal{P})=\sup\{ \epsilon>0: n^{\bullet}\vec{a}+\vec{a}^{\bullet}-(\epsilon,\epsilon,\epsilon),\text{ for } \bullet\in\{V,\mathcal{N},P_g\}, \text{ are short-range perturbations}\}.
		\end{equation}
	\end{defi}
 All the definitions naturally extend to perturbations with coefficients of finite $\Hb^{;k}$ regularity. The extension of the definition to finite regularity is required e.g.~in \cref{rem:en:regularity_of_coefficients}.

We also immediately note the following two points:
\begin{itemize}
			\item The definition in \cref{item:en:semilinear_perturbation} with $n^{\mathcal{N}}=1$ already includes potential perturbations as well.
			We nevertheless decided to keep $V$ separate because its treatment is strictly simpler.
			\item The commutator of $P_g$ with $r$, $[P_g,r]$, gives first and zeroth order terms that are short-range potential perturbations, c.f.~\cref{lemma:en:short_range_computations}.
            With this in mind, we shall freely commute with $r$ and, in estimates, treat $rP_g[\phi]$ and $P_g[r\phi]$ interchangeably.
		\end{itemize}

We provide some examples so the reader may more easily distill the definition:

\begin{rem}[Examples]\label{rem:en:examples}

	\begin{itemize}
		\item Let us start with the case of a potential perturbation.
		For $a_-\geq 1/2$, $a_0<a_-$ and $a_+<\min(0,a_0)$, we have  $V=\rho^{\epsilon}\rho_\scri\rho_0^2=(r^\epsilon u v)^{-1}$ is an allowed potential perturbation with $\vec{a}^V=(1,2,1)+\epsilon$:
        Indeed, we may compute that $\vec{a}+\vec{a}^V=(a_-+1,a_0+2,a_++1)+\epsilon$ is a strongly admissible weight.
        The example is not compatible with no incoming radiation because acting with $r\pv$ does not preserve the decay towards  $\scrim$.
		If we instead choose $a_-\leq 1/2$ and impose the stronger $a_+<a_0\leq a_--1/2$, then the same $V$ will still be short-range, but note that in this case, we lose an extra half decay order coming from the extra condition \cref{eq:en:admissible_inhom2}, see, however, \cref{rem:en:shortrange_admissible_weights}.
		
		\item For any $\epsilon>0$, the potential $V=r^{-2-\epsilon}(1,r\pv,\sl,r\pu)$ is a short-range potential perturbation compatible with no incoming radiation for any admissible $\vec{a}$. 
        Similarly, $P_g=f(\rho)(\rho^{2}\Dl,\pu\pv)$, with $\rho=1/r$, is a short-range quasilinear perturbation for $f\in\Hb^{\epsilon;\infty}([0,1)_{\rho})$ for any admissible~$\vec{a}$.
        In particular, the wave equation on a Schwarzschild background, \cref{eq:intro:Schw_uv_coordinates}, falls into this category.
		
		\item Consider $\mathcal{N}[\phi]=\pu\phi\pv\phi$.
		We write this as $\mathcal{N}[\phi]=r^{-1}\frac{1}{uvr}(u\pu \psi v\pv \psi+u\phi v\pv \psi-v\phi u\pu \psi-uv\phi^2)$.
		From this expression, we can read off $a_\pm^\mathcal{N}=2,a_0^{\mathcal{N}}=3$.
		A quick check shows that $\mathcal{N}$ is short-range for $\min(a_-,a_0)>-1/2$.
		\item For $(\partial_t\phi)^2=r^{-2}(\partial_t \psi)^2$, we compute that $a_\pm^\mathcal{N}=1,a_0^\mathcal{N}=3$.
		Since for no $a_+<0$ it holds that $na_++a_+^{\mathcal{N}}=2a_++1\leq a_++1$, this is never a short-range perturbation. However, for $(\frac{u}{r})^\epsilon(\partial_t\phi)^2$ and $\epsilon>0$, we have a short-range perturbation for $a_-,a_0>0$ and $a_+>-\epsilon$. (As a foreshadowing, we note already that the requirement $a_->0$ essentially excludes incoming radiation.)
	\end{itemize}
\end{rem}

    \begin{rem}[Compatibility with the no incoming radiation condition]\label{rem:en:no-incoming_vs_short-range}
        Being compatible with no incoming radiation implies short-range only for linear perturbations. Let us expand:
        
    	We restrict attention to the region $\D^-$.    
		A potential perturbation is compatible with no incoming radiation if $a^V_0>2$.
		In turn, we can write $V=\boldsymbol{\rho}_0^{a^V_0}\Vbt=\rho_0^{a^V_0}\rho_-^{a^V_0}\Vbt\in \Hb^{1+,2+}(\D^-)\Vb$.
		Therefore $V$ is also a short-range potential perturbation for any admissible weight $\vec{a}$ whenever $a_-\geq0$.
		
		In the case of nonlinearities, the situation is more intricate. 
		As in the potential case, it still holds that compatibility with no incoming radiation is equivalent to $n^\bullet a_0+a^\bullet_0>a_0+2$; however, it generally does not imply short-range.
        To understand this, take $\mathcal{N}=r^{-1-\alpha}(r\pv \psi)^2$. In the case of \cref{eq:en:short_range_no_incoming_condition}, we then have $a^{\mathcal{N}}_0=\alpha$, but in the case of \cref{eq:en:short_range_nonlinear_form}, we further have $a^{\mathcal{N}}_-=\alpha-2$.
		A short computation then implies that $\mathcal{N}$ is short-range if $a_0+\alpha>3$ and compatible with no incoming radiation if $a_0+\alpha>5/2$. Again, with \cref{rem:en:shortrange_admissible_weights} in mind, this can be improved to $a_0+\alpha>2$.
	\end{rem}
We conclude with the following important remarks: 
 \begin{rem}[Improvement of admissible weights]\label{rem:en:shortrange_admissible_weights}
    To keep the condition of a short-range perturbation accessible as a definition, we did not incorporate the improvement provided by \cref{eq:en:pastestimate_inhom_0th}.
    In particular, our definition of short-range perturbations also requires the bound \cref{eq:en:admissible_inhom1} to be satisfied. 
     Nevertheless, we can still construct scattering solutions if instead of admissibility of $\vec{a}^\bullet$, we merely require that $\vec{a}^\bullet+(n^\bullet-1)\vec{a}>(1,2,1)$ and that ${a}_-^{\bullet}+n^{\bullet}{a}_->1$, albeit with a loss of regularity.
    See already \cref{def:scat:extended_short_range} and \cref{cor:scat:enlarged_admissible_set}.
    In fact, in the case of first order (in derivatives) short-range perturbations, there as no loss of regularity, as we can argue exactly as in \cref{eq:savedhehe} to avoid the condition \cref{eq:en:admissible_inhom1}.
 \end{rem}
 
 \begin{rem}[Nonlinearities depending smoothly on $\phi$]\label{rem:en:smoothperturbations}
     All our results will also apply to nonlinearities that are not of fixed order, but only \textit{vanish} up to some fixed order. For instance, we can treat the nonlinearity $F(\phi)\phi^3$, with $F(\phi)$ depending smoothly on $\phi$, as cubic if $\vec{a}>-\vec{1}$ (this latter condition ensures that higher than cubic terms are subleading in decay). More generally, if $\vec{a}>\vec{-1}$, we can consider, instead of \cref{eq:en:short_range_nonlinear_form}, nonlinearities of the form
     \begin{equation}
         r^{n^{\mathcal{N}-1}}f_{\mathcal{N}}\cdot F(\Vb\phi),\qquad f_{\mathcal{N}}\in\rho_\scri^{\kappa/2}\Hb^{a^\mathcal{N}_-,a^\mathcal{N}_0,a^\mathcal{N}_+;\infty}(\D),
     \end{equation}
     where $F$ depends smoothly on $\Vb\phi$ and vanishes to order $n^{\mathcal{N}}$ at $\Vb\phi=0$. We will still call such nonlinearities "of order $n^{\mathcal{N}}$", and we can analogously extend the definitions in this way for \cref{eq:en:short_range_no_incoming_condition} and for \cref{item:en:quasilinear_perturbation}.
This will only become relevant in the discussion of the Einstein vacuum equations in \cref{sec:EVE}.
\end{rem}
    \begin{rem}[Systems of wave equations]
        The definition of admissible perturbation also extends to system of wave equations $\Box\phi_A=\mathcal{N}_A[\phi_B]$.
        In this case, each component has its own associated decay rate $\vec{a}_A$, and the forcing terms need to be strictly admissible when the respective decay rates are inserted into $\mathcal{N}_A$.
        Such generalisations would only make the notation harder to follow, thus we do not pursue them further. See, however, \cref{sec:EVE}.
    \end{rem}

\subsubsection{Basic properties of short-range perturbations}\label{sec:en:sh:properties}
    \begin{lemma}\label{lemma:en:short_range_computations}
        Fix admissible weight $\vec{a}$ and $\Ve (r\phi)\in\Hb^{\vec{a};k}(\D)$.
        Let $P_g,\mathcal{N}$ be short-range quasilinear and semilinear perturbations. Then the following holds:
       
           a) For $X\in\Diff^1_{\b}(\D)$ there exists short-range $\tilde{V},\tilde{\mathcal{N}}$ with $\Hb^{;k-1}(\D)$ coefficients  such that we have
            \begin{equation}
                [X,P_g[\phi]]=\begin{cases}
                    \tilde{V}\Diff^1_{\b}(\D) & \text{if } n^{P_g}=1,\\
                    (\tilde{V} \Diff^1_{\b}(\D) r\phi)\Diff^2_{\b}(\D) & \text{if } n^{P_g}=2,\\
                    (\Diff^1_{\b}(\D)r\tilde{\mathcal{N}}[\phi])\Diff^2_{\b}(\D)& \text{if } n^{P_g}\geq3.
                \end{cases}
            \end{equation}
          
           b) There exist perturbations $\tilde{V},\tilde{\mathcal{N}},\tilde{P}_g$ (with $k-1$ regular coefficients) that are short-range  with respect to $\vec{a}'$ for any $\vec{a}'\geq \vec{a}$ such that we can write 
            \begin{equation}
                \tilde{V}\tilde{\phi}+\tilde{\mathcal{N}}[\tilde{\phi}]+\tilde{P}_g[\tilde{\phi}]\tilde\phi=\mathcal{N}[\phi+\tilde{\phi}]-\mathcal{N}[\phi]+P_g[\phi+\tilde{\phi}](\phi+\tilde{\phi})-P_g[\phi+\tilde{\phi}](\phi)
            \end{equation}        
    \end{lemma}

    \begin{proof}
        \textit{a)} This is a direct computation. 

        \textit{b)} Let us do the expansion for $\mathcal{N}$, the result follows similarly for $P_g$.
        We write $n=n^{\mathcal{N}}$ and express $\mathcal{N}$ as a linear combination of terms
        \begin{equation}
            \frac{1}{r}f^{\mathcal{N}}(X_1 r\phi)(X_2r\phi)\cdots(X_nr\phi),
        \end{equation}
        for $X_i\in \Diff_\b^1(\D)$
        and analyse each term individually.
        Then, we can compute
        \begin{nalign}\label{eq:en:difference_of_N}
            \mathcal{N}[\phi+\tilde\phi]-\mathcal{N}[\phi]=\frac{1}{r}f^{\mathcal{N}}\sum_{\substack{\sigma\subset\{1,2,..,n\}\\ \sigma\neq\emptyset}}\prod_{i\in\sigma}(X_ir\tilde{\phi})\prod_{j\in\{1,2,\dots,n\}\setminus\sigma}(X_jr\phi).
        \end{nalign}
        Each of the terms in the sum is of the required form.
        Indeed, they take the form $\tilde{\mathcal{N}}_\sigma$ with $a^{\tilde{\mathcal{N}}_\sigma}_\bullet=a^{{\mathcal{N}}}_\bullet+(n-\abs{\sigma})a_\bullet$ (with $\bullet\in\{-,0,+\}$) and $n^{{\tilde{\mathcal{N}}_\sigma}}=\abs{\sigma}$.
        We check that the inequalities in \cref{def:en:short_range} hold.
    \end{proof}

We ultimately want to extend \cref{prop:en:main} to short-range perturbations, i.e.~to \cref{eq:en:characteristicIVP_perturbed}.
For this, we first record a product estimate:
	\begin{lemma}[Sobolev embeddings]\label{lemma:scat:sobolev}
		Let $f,g\in\C^{\infty}(\D)$. Then for $\vec{a}
^f+\vec{a}^g>\vec{a}$ we have 
        \begin{subequations}
            \begin{gather}
            \norm{fg}_{\Hb^{\vec{a};k}(\D)}\lesssim_{k,\vec{a}}\norm{f}_{\Hb^{\vec{a}^f;k}(\D)}\norm{g}_{\Hb^{\vec{a}^g;3}(\D)}+\norm{f}_{\Hb^{\vec{a}^f;3}(\D)}\norm{g}_{\Hb^{\vec{a}^g;k}(\D)},\\
            \norm{fg}_{\Hbt^{\vec{a};k}(\D)}\lesssim_{k,\vec{a}}\norm{f}_{\Hbt^{\vec{a}^f;k}(\D)}\norm{g}_{\Hb^{\vec{a}^g;3}(\D)}+\norm{f}_{\Hb^{\vec{a}^f;3}(\D)}\norm{g}_{\Hbt^{\vec{a}^g;k}(\D)}.
            \end{gather}
        \end{subequations}
        An analogous estimate to the first holds in $\Hb(\incone)$, with the 3 replaced by 2.
	\end{lemma}
	\begin{proof}
		This follows from product estimates on Sobolev spaces and Sobolev embedding, using also that the boundary defining functions are smooth.
	\end{proof}
    Next, we prove a lemma saying that we can solve a short-range perturbation problem via iteration:
	\begin{lemma}\label{cor:en:VN_estimates} Fix $\vec{a}$ admissible.

		a) Let $V,\mathcal{N}$ be short-range potential and semilinear perturbations, respectively.
		Let $\Ve (r\bar{\phi})\in\Hb^{\vec{a};k}(\D)$ with $k\geq 4$ and $\supp\bar{\phi}\in\{u>u_\infty+1\}$.
		Let $\phi$ be the solution to \cref{eq:en:characteristicIVP} with $f=V\bar{\phi}+\mathcal{N}[\bar{\phi}]$ and \cref{eq:en:assumption0}. 
  Then there exists $\vec{a}^{\mathcal{P}}$ strongly admissible such that 
  \begin{equation}\label{eq:en:cor_VN_estimate0}
  \norm{r\mathcal{N}[\bar{\phi}]+rV\bar{\phi}}_{\Hb^{\vec{a}^{\mathcal{P}};k}(\D)} \lesssim_{\mathcal{N},V}\norm{\Ve (r\bar{\phi})}_{\Hb^{\vec{a};k}(\D)}+\norm{\Ve (r\bar{\phi})}^{n^{\mathcal{N}}}_{\Hb^{\vec{a};k}(\D)}
  .\end{equation}
  In particular, we have for some $\epsilon(\mathcal{N},V)>0$ less than the gap of $\mathcal{N},V$ that 
		\begin{nalign}\label{eq:en:cor_VN_estimate_incoming}
			\norm{\Ve (r\phi)}_{\Hb^{\vec{a}+\epsilon;k}(\D)}\lesssim_{\mathcal{N},V}\norm{\Ve (r\bar{\phi})}_{\Hb^{\vec{a};k}(\D)}+\norm{\Ve (r\bar{\phi})}^{n^{\mathcal{N}}}_{\Hb^{\vec{a};k}(\D)}+\sum_{0\leq i\leq k} \norm{T^{i}(r\phi)}_{\rho_-^{a_-+\epsilon}\Hb^{;k-i+1}(\incone)}.
		\end{nalign}

		b)  Let $V,\mathcal{N}$ be  potential and semilinear perturbations compatible with the no incoming radiation condition. Let $\Ve(r \bar{\phi})\in\Hbt^{\vec{a};k}(\D)$ with $k\geq 4$ and $\supp\bar{\phi}\in\{u>u_\infty+1\}$.
		Let $\phi$ be the solution to \cref{eq:en:characteristicIVP} with $f=V\bar{\phi}+\mathcal{N}[\bar{\phi}]$ and $\psi^{\incone}=0,\psi^{\outconeFar}=0$. Then we have for some $\epsilon(\mathcal{N},V)>0$ less than the gap of $\mathcal N$, $V$:
		\begin{nalign}\label{eq:en:cor_VN_estimate_no_incoming}
			\norm{\Ve (r\phi)}_{\Hbt^{\vec{a}+\epsilon;k}(\D)}\lesssim_{\mathcal{N},V}\norm{\Ve (r\bar{\phi})}_{\Hbt^{\vec{a};k}(\D)}+\norm{\Ve (r\bar{\phi})}^{n^{\mathcal{N}}}_{\Hbt^{\vec{a};k}(\D)}+\sum_{0  \leq i\leq k} \norm{(rT)^{i}(r\phi)}_{\rho_-^{a_-+\epsilon}\Hb^{;k-i+1}(\incone)}.
		\end{nalign}
	\end{lemma}

\begin{obs}\label{rem:en:Pgestimates}
    \cref{cor:en:VN_estimates} also holds for quasilinear perturbations $P_g$ after performing the following changes: We only recover $k-1$ derivatives on the left hand side of \cref{eq:en:cor_VN_estimate0,eq:en:cor_VN_estimate_incoming,eq:en:cor_VN_estimate_no_incoming} and need to change $n^{\mathcal{N}}$ to $n^{P_g}$.
\end{obs}

	\begin{proof}
 We fix throughout the proof $\epsilon$ sufficiently small and $\vec{a}^f$ minimal with the property that it is admissible for $\vec{a}$.
		
 \textit{ a) }This follows from applying \cref{lemma:scat:sobolev,prop:en:main} together with the fine tuned definitions of \cref{def:en:short_range}.
        Let us expand: 
        
        We start with $V\bar{\phi}=f_V \Vb\bar{\phi}$:
        We use \cref{item:en:potential_perturbation,lemma:scat:sobolev} 
        \begin{nalign}
            \norm{rV\bar{\phi}}_{\Hb^{\vec{a}^f+\epsilon;k}(\D)}\lesssim\norm{f_V\Vb (r\bar{\phi})}_{\Hb^{\vec{a}^f+\epsilon;k}(\D)}\lesssim \norm{\Ve (r\bar{\phi})}_{\Hb^{\vec{a}^f+\epsilon-\vec{a}^V;k}(\D)}\lesssim \norm{{\Ve} (r\bar{\phi})}_{\Hb^{\vec{a};k}(\D)}.
        \end{nalign}

        For $\mathcal{N}[\bar{\phi}]$ we proceed similarly. We use \cref{item:en:semilinear_perturbation,lemma:scat:sobolev} to compute
        \begin{nalign}
            \norm{r\mathcal{N}[\bar{\phi}]}_{\Hb^{\vec{a}^f+\epsilon;k}(\D)}
            \lesssim\norm{f_{\mathcal{N}}(\Vb (r\bar{\phi}))^{n_\mathcal{N}}}_{\Hb^{\vec{a}^f+\epsilon;k}(\D)}
            \lesssim\norm{\rho_\scri^{\kappa/2}(\Vb (r\bar{\phi}))^{n_\mathcal{N}}}_{\Hb^{\vec{a}^f-\vec{a}^{\mathcal{N}}+\epsilon;k}(\D)}
            \lesssim\norm{\rho_\scri^{\kappa/2}(\Vb (r\bar{\phi}))^{n^\mathcal{N}}}_{\Hb^{n^{\mathcal{N}}\vec{a};k}(\D)}\\
            \lesssim \norm{\Vb (r\bar{\phi})}_{\Hb^{\vec{a};3}(\D)}^{n_\mathcal{N}-1}\norm{\Ve (r\bar{\phi})}_{\Hb^{\vec{a};k}(\D)}\lesssim\norm{\Ve (r\bar{\phi})}_{\Hb^{\vec{a};k}(\D)}^{n_\mathcal{N}},
        \end{nalign}
        where, in the last step, we have used that $k\geq4$ in order to convert the $\Vb$ into $\Ve$-control. 
        Using now \cref{prop:en:main}, the result \cref{eq:en:cor_VN_estimate_incoming} follows from \cref{eq:en:cor_VN_estimate0}.

		\textit{b)} We upgrade to $\Hbt$ spaces via \cref{item:en:pastestimate_Hbt}.
        Let us expand: 
        
        For $V\bar{\phi}$, we use \cref{item:en:potential_perturbation,lemma:scat:sobolev} and $\vec{a}^f+\epsilon<\vec{a}^V+\vec{a}-(1,0,0)$ to  compute
        \begin{nalign}
            \norm{rV\bar{\phi}}_{\Hbt^{\vec{a}^f+\epsilon;k}(\D)}\lesssim\norm{f_V\Vbt (r\bar{\phi})}_{\Hbt^{\vec{a}^f+\epsilon;k}(\D)}\lesssim \norm{\rho_+^{\kappa/2}\Vbt (r\bar{\phi})}_{\Hbt^{\vec{a}^f+\epsilon-\vec{a}^V;k}(\D)}\lesssim \norm{\Ve (r\bar{\phi})}_{\Hbt^{\vec{a}^f+\epsilon-\vec{a}^V+(1,0,0);k}(\D)}\lesssim \norm{\Ve (r\bar{\phi})}_{\Hbt^{\vec{a};k}(\D)}.
        \end{nalign}

        For $\mathcal{N}[\bar{\phi}]$, we similarly use $\vec{a}^f+\epsilon-\vec{a}^{\mathcal{N}}+(1,0,0)<n_\mathcal{N}\vec{a}$ and compute
        \begin{nalign}
            \norm{r\mathcal{N}[\bar{\phi}]}_{\Hbt^{\vec{a}^f+\epsilon;k}(\D)}\lesssim\norm{f_{\mathcal{N}}(\Vbt (r\bar{\phi}))^{n_\mathcal{N}}}_{\Hbt^{\vec{a}^f+\epsilon;k}(\D)}\lesssim\norm{\rho_+^{\kappa/2}(\Vbt (r\bar{\phi}))^{n_\mathcal{N}}}_{\Hbt^{\vec{a}^f-\vec{a}^{\mathcal{N}}+\epsilon;k}(\D)}\\
            \lesssim \norm{\Vbt (r\bar{\phi})}_{\Hbt^{\vec{a};3}(\D)}^{n_\mathcal{N}-1}\norm{\rho_+^{\kappa/2}\Vbt (r\bar{\phi})}_{\Hbt^{\vec{a}^f+\epsilon-\vec{a}^{\mathcal{N}}-(n^{\mathcal{N}}-1)\vec{a};k}(\D)}\\
            \lesssim \norm{\Vbt (r\bar{\phi})}_{\Hb^{\vec{a};3}(\D)}^{n_\mathcal{N}-1}\norm{\Ve (r\bar{\phi})}_{\Hbt^{\vec{a}^f+(1,0,0)+\epsilon-\vec{a}^{\mathcal{N}}-(n^{\mathcal{N}}-1)\vec{a};k}(\D)}\lesssim\norm{\Ve (r\bar{\phi})}_{\Hbt^{\vec{a};k}(\D)}^{n^{\mathcal{N}}}.
        \end{nalign}
        Using now  the \cref{item:en:pastestimate_Hbt} version of \cref{prop:en:main}, the result follows.
	\end{proof}

	\subsubsection{Estimating transversal derivatives along the initial cone}\label{sec:en:sr:data}
    	Before we turn to spacetime estimates using short-range perturbations, let us show how to recover transversal derivatives on the light cone for solutions of \cref{eq:en:characteristicIVP_perturbed}. Since we exclusively work on $\incone$ in this section, we of course don't need the full strength of \cref{def:en:short_range}, but we will continue to use its notation.
	\begin{lemma}\label{lemma:en:recovering_initial_data}
    Let $k\geq2$, $a_->0$, and let $\phi$ solve \cref{eq:en:characteristicIVP_perturbed} with $ \psi^{\outconeFar}=0$ and $\psi^{\incone}\in \Hb^{a_-;2k+2}(\incone)$.

		a) Fix $\mathcal{N},V,P_g$ perturbations such that $n^\bullet a_-+a^\bullet_->a_-+1$ for $\bullet\in\{\mathcal{N},V,P_g\}$, where $n^{\bullet},a^{\bullet}_-$ are as in \cref{def:en:short_range}. Also fix inhomogeneity $f$ such that $\mathcal{X}_{\incone}^{2,2}\leq 1$ as defined below. Then, for $\abs{u_0}$ sufficiently large depending only on $\{V,\mathcal{N},P_g\}$ and $j\leq k$
		\begin{equation}\label{eq:scat:data_incoming}
			\norm{T^j\pv \psi|_{\incone}}_{\Hb^{a_-;2(k-j)}(\incone)}\lesssim\norm{\psi^{\incone}}_{\Hb^{a_-;2k+2}(\incone)}+\sum_{j'\leq j}\norm{T^{j'}(rf)}_{\Hb^{a_-+1;2(k-j')}(\incone)}:=\mathcal{X}_{\incone}^{j,k}.
		\end{equation}
		b) Fix $\mathcal{N},V,P_g,f$ compatible with no incoming radiation, i.e. $n^\bullet a_-+a^\bullet_->a_-+2$ and components in $\Hbt^{\vec{a}^\bullet}(\D)$ as in \cref{def:en:short_range}. Assume that $f$ satisfies $\tilde{\mathcal{X}}_{\incone}^{2,2}\leq1$. Then
		\begin{equation}\label{eq:scat:data_no_incoming}
			\norm{(rT)^jr\pv \psi|_{\incone}}_{\Hb^{a_-;2(k-j)}(\incone)}\lesssim\norm{\psi^{\incone}}_{\Hb^{a_-;2k+2}(\incone)}+\sum_{j'\leq j}\norm{(rT)^{j'}(rf)}_{\Hb^{a_-+2;2(k-j')}(\incone)}:=\tilde{\mathcal{X}}_{\incone}^{j,k}.
		\end{equation}
        c) For $a<0$, both \cref{eq:scat:data_incoming,eq:scat:data_no_incoming} still hold if we also we include in the terms $\mathcal{X}^{j,k}_{\incone},\tilde{\mathcal{X}}^{j,k}_{\incone}$ the following norm control on the sphere $S=\incone\cap\{u=u_1\}$ for an arbitrary $-\infty<u_1\leq u_0$:
        \begin{equation}
            \norm{T^j\partial_v\psi|_{S}}_{\rho_-^{a_-}H^{2(k-j)}(S)},\qquad \norm{(rT)^j\partial_v\psi|_{S}}_{\rho_-^{a_-}H^{2(k-j)}(S)}
        \end{equation}
        respectively.\footnote{Note that the estimate, while uniform in $u_\infty$, is \emph{not uniform in $u_1$.} This means that for slow decay, we can not specify transversal data at infinity.} 
	\end{lemma}

    \begin{rem}\label{rem:en:slab_data_recovery}
        Note that if instead of working with data on a $\incone$, we specified the full solution in a spacetime slab $\D^{u_\infty,v_0+1}_{u_0,v_0}$, we would not face the loss of regularity provided by the lemma.
        We further note that, for trivial data posed on $\incone$ and nontrivial incoming radiation along $\outconeFar$, the analogous loss in derivatives will disappear in the limit $\outconeFar\to\scrim$.
        Finally, note that had we not assumed an explicit bound on the norm of the data, then, due to the nonlinear terms, $u_0$ would also need to depend on the size of the data.
    \end{rem}
 
	\begin{proof}
		\textit{a)} In this proof, we will exclusively work on $\incone$ and all spaces are to be understood on this cone.
		We can write the wave equation along the null cone $\incone$ as
		\begin{equation}
			\pu\pv \psi=\frac{\Dl \psi}{r^2}-rf-rV\phi-r\mathcal{N}[\phi]-rP_g[\phi]\phi.
		\end{equation}
		We first prove the existence of $u_0$ such that \cref{eq:scat:data_incoming} holds for $j=0$, $k=2$ uniformly in $u_\infty$.
        Then, we will prove that \cref{eq:scat:data_incoming} holds with the same $u_0$ and $j\leq k\geq2$. 
        By \cref{item:ode:1d_boundary_cond}, using $\phi|_{\outconeFar}=0$  and the short-range nature of the nonlinearity (see below for details), we can estimate for some $\epsilon(\mathcal{N},V,P_g)>0$

		\begin{nalign}\label{eq:en:dataestimateproof1}
			\norm{\pv \psi}_{\Hb^{a_-;2k}}+\norm{u\pu\pv \psi}_{\Hb^{a_-;2k}}\lesssim \norm{r\pu\pv \psi}_{\Hb^{a_-;2k}}\lesssim\norm{\Dl \psi}_{\Hb^{a_--1;2k}}
			+\norm{rf}_{\Hb^{a_-+1;2k}}\\+(\norm{\psi}_{\Hb^{a_--\epsilon;2k+1}}
			+\norm{\pv \psi}_{\Hb^{a_--\epsilon;2k}})+(\norm{\pv \psi}_{\Hb^{a_--\epsilon;2k}}+\norm{\psi}_{\Hb^{a_--\epsilon;2k+1}})^{n_\mathcal{N}}\\
			+(\norm{\pv \psi}_{\Hb^{a_--\epsilon;2k}}+\norm{\psi}_{\Hb^{a_--\epsilon;2k+1}})^{n^{P_g}-1}{\norm{r(\abs{u}^{-1}\Dl,u\pu\pv)\phi}_{\Hb^{a_--\epsilon;2k}}}
		\end{nalign}
        Here, the nonlinear parts are estimated as 
         \begin{subequations}
        	\begin{gather}
        		\begin{multlined}
        			\norm{r\mathcal{N}[\phi]}_{\Hb^{a_-+1;2k}}=\norm{f_{\mathcal{N}}(\Vb \psi)^{n_\mathcal{N}}}_{\Hb^{a_-+1;2k}}\lesssim \norm{(\Vb \psi)^{n_\mathcal{N}}}_{\Hb^{a_-+1-a^{\mathcal{N}}_-;2k}}\lesssim\Big(\norm{\psi}_{\Hb^{a_--\epsilon;2k+1}}+\norm{\pv \psi}_{\Hb^{a_--\epsilon;2k}}\Big)^{n^{\mathcal{N}}}
        		\end{multlined}\\
        		\begin{gathered}
        			\norm{r(\Omegalin[\phi]\pu\pv,\glin_{AB}[\phi]\sl^A\sl^B)\phi}_{\Hb^{a_-+1;2k}}\lesssim\norm{(u^{-1}\Omegalin[\phi],r\glin[\phi])}_{\Hb^{a_-(n^{P_g}-1)+a_-^{P_g}-\epsilon;2k}}\norm{r(u\pu\pv,r^{-1}\Dl)\phi}_{\Hb^{a_--\epsilon;2k}}\\
        			\lesssim(\norm{\pv \psi}_{\Hb^{a_--\epsilon;2k}}+\norm{\psi}_{\Hb^{a_--\epsilon;2k+1}})^{n^{P_g}-1}\norm{r(u\pu\pv,r^{-1}\Dl)\phi}_{\Hb^{a_--\epsilon;2k}}
        		\end{gathered}
        	\end{gather}
        \end{subequations}
		Estimate \cref{eq:en:dataestimateproof1} takes the form $x\lesssim 1+|u_0|^{-\epsilon}B x^n$; thus, taking $\abs{u_0}$ sufficiently large shows that $\pv \psi\in\Hb^{a_-;2k}$ uniformly in $u_\infty$.

		To get $k\geq3$, we apply the same procedure to get the result for some potentially larger $-u'_0$ depending on $k$.
        To extend to the rest of the region, we use that higher norms are propagated linearly, i.e.~we have
        \begin{multline}
            \norm{\pv \psi}_{\Hb^{a_-;2k}}+\norm{r\pu \pv \psi}_{\Hb^{a_-;2k}}\lesssim \norm{\Dl \psi^{\incone}}_{\Hb^{a_--1;2k}}
			+\norm{rf}_{\Hb^{a_-+1;2k}}\\+(\norm{\psi}_{\Hb^{a_--\epsilon;2k+1}}
			+\norm{\pv \psi}_{\Hb^{a_--\epsilon;2k}})+(\norm{\pv \psi}_{\Hb^{a_--\epsilon;4}}+\norm{\psi}_{\Hb^{a_--\epsilon;5}})^{n_\mathcal{N}-1}(\norm{\pv \psi}_{\Hb^{a_--\epsilon;2k}}+\norm{\psi}_{\Hb^{a_--\epsilon;2k+1}})\\
			+
   (\norm{\pv \psi}_{\Hb^{a_--\epsilon;4}}+\norm{\psi}_{\Hb^{a_--\epsilon;5}})^{n^{P_g}-1}{\norm{r(u\pu\pv,r^{-1}\Dl)\phi}_{\Hb^{a_--\epsilon;2k}}}.
        \end{multline}
		We extend to $[u'_0,u_0]$ by Gr\"onwall inequality.
		
		To obtain higher order commutation with $T$, we inductively prove the estimates
		\begin{nalign}
			\norm{T^jr\mathcal{N}[\phi]}_{\Hb^{a_-+\epsilon;2(k-j)}}\lesssim \norm{T^{j}\pv \psi}_{\Hb^{a_--\epsilon;2(k-j)}}w^{n^{\mathcal{N}}-1}+w^{n^{\mathcal{N}}},\\
			\norm{T^jrP_g\phi}_{\Hb^{a_-;k}}\lesssim\norm{T^j(\pv,u\pu\pv)\psi}_{\Hb^{a_--\epsilon;k}}w^{n^{P_g}-1}+w^{n^{P_g}}, \\
			\text{where    }\quad w=\sum_{j'\leq j}\norm{\pv^{j'} \psi}_{\Hb^{a_--\epsilon;{2(k-j')+2}}}
		\end{nalign}
		Therefore, for higher $T^j$ commutations, we only have linear in $T^j\pv$ terms, with the lower order terms being controlled from previous iterates.
		\cref{eq:scat:data_incoming} follows by induction in $j$ for fixed $k$. 
		
		\textit{b)}
        We again start with the base case $j=0,k=2$, but include an extra weight for $\pv \psi$:
        \begin{nalign}
			\norm{\pv \psi}_{\Hb^{a_-+1;2k+1}}+\norm{r\pu\pv \psi}_{\Hb^{a_-+1;2k}}\lesssim\norm{\Dl \psi^{\incone}}_{\Hb^{a_-;2k}}
			+\norm{rf}_{\Hb^{a_-+2;2k}}\\+(\norm{\psi}_{\Hb^{a_--\epsilon;2k+1}}
			+\norm{r\pv \psi}_{\Hb^{a_--\epsilon;2k}})+(\norm{r\pv \psi}_{\Hb^{a_--\epsilon;2k}}+\norm{\psi}_{\Hb^{a_--\epsilon;2k+1}})^{n_\mathcal{N}}\\+
            (\norm{r\pv \psi}_{\Hb^{a_--\epsilon;2k}}+\norm{\psi}_{\Hb^{a_--\epsilon;2k+1}})^{n^{P_g}-1}\norm{(\Dl,ur\pu\pv)\psi}_{\Hb^{a_--\epsilon;2k}}
		\end{nalign}
        Using that $a^\bullet_-+n^\bullet a_->2+a_-$, the nonlinear terms are bounded as
        	\begin{subequations}
		\begin{gather}
			\begin{multlined}
				\norm{r\mathcal{N}[\phi]}_{\Hb^{a_-+2;2k}}=\norm{f_{\mathcal{N}}(\Vbt \psi)^{n^{\mathcal{N}}}}_{\Hb^{a_-+2;2k}}\lesssim\norm{(\Vbt \psi)^{n^{\mathcal{N}}}}_{\Hb^{a_-+2-a^{\mathcal{N}};2k}}
				\lesssim\Big(\norm{ \psi}_{\Hb^{a_--\epsilon;2k+1}}+\norm{r\pv\psi}_{\Hb^{a_--\epsilon;2k}}\Big)^{n^{\mathcal{N}}}
			\end{multlined}\\
			\begin{multlined}
				\norm{r^2(\Omegalin\pu\pv,\glin_{AB}\sl^A\sl^B)\psi}_{\Hb^{a_-;2k}}\lesssim\norm{r^2(r^{-2}\Omegalin,\glin)}_{\Hb^{a_-(n^{P_g}-1)-\epsilon;2k}}\norm{(ur\pu\pv,\Dl)\psi}_{\Hb^{a_--\epsilon;2k}}\\
				\lesssim(\norm{r\pv \psi}_{\Hb^{a_--\epsilon;2k}}+\norm{\psi}_{\Hb^{a_--\epsilon;2k+1}})^{n^{P_g}-1}\norm{(ur\pu\pv,\Dl)\psi}_{\Hb^{a_--\epsilon;2k}}.
			\end{multlined}
		\end{gather}
	\end{subequations} 
        For $u_0$ sufficiently large, we already conclude \cref{eq:scat:data_no_incoming} with $k=2,j=0$.
        We increase the value of $k$ exactly as before.
        
        To obtain \cref{eq:scat:data_no_incoming} with $j\geq1$, we proceed as before, but include an $r$-weight with each $T$-commutation and use that $f_\bullet\in\Hbt^{a^\bullet}(\D)$, for $\bullet\in\{V,\mathcal{N},P_g\}$, does not lose weights for such commutations.

       \textit{ c)} This follows in the same way as the previous ones, but we need to use \cref{item:ode:1d}, since the integration constant is specified at a finite value.
        This is required to obtain estimates that are uniform in $u_\infty$. 
	\end{proof}

    \subsubsection{The main energy estimates}\label{sec:en:sr:energy}
    We now turn to the main energy estimates of the paper.
    We first close the semilinear estimate:
 
	\begin{prop}\label{cor:en:perturbed_energy_estimate}
		 Fix $\vec{a}$, $\vec{a}^f$ admissible, $k\in\N_{\geq4}$, and let $\phi$ solve the characteristic initial value problem \cref{eq:en:characteristicIVP_perturbed} with~\cref{eq:en:assumption0}.

   a) Let $V,\mathcal{N}$ be short-range perturbations for $\vec{a}$, and assume that $\mathcal{X}^4\leq1$, where $\mathcal{X}^k$ is defined below:
		Then for $|u_0|$ sufficiently large depending only on $V,\mathcal{N}$, we have:
		
		\begin{nalign}\label{eq:en:perturbed_energy_estimate}
			\sup_{u\in (u_\infty,u_0)}\norm{(v\pv,\rho_-^{1/2}\sl)\Vb^k\psi}_{\Hb^{\vec{a};0}(\outcone{})}+\sup_{v\in(v_0,v_\infty)}\norm{(u\pu,\rho_-^{1/2}\sl)\Vb^k\psi}_{\Hb^{\vec{a};0}(\Cbar_v)}+\norm{ \Ve \psi}_{\Hb^{\vec{a};k}(\D)}\\
			\lesssim_{\mathcal{N},V,k} \norm{rf}_{\Hb^{\vec{a}^f;k}(\D)}+\sum_{j\leq k}\norm{T^{j}(rf)}_{\rho_-^{a_-^f}\Hb^{;2(k-j)}(\incone)}
			+ \norm{\psi^{\incone}}_{\rho_-^{a_-}\Hb^{;2k+2}(\incone)}=:\mathcal{X}^k.
		\end{nalign}
  b) Let  $V,\mathcal{N}$ instead be compatible with the no incoming radiation condition, and assume  that $\tilde{\mathcal{X}}^4\leq1$ (as defined below). 
  
  Then, for $|u_0|$ sufficiently large depending only on $V,\mathcal{N}$, we have:
		\begin{nalign}\label{eq:en:perturbed_energy_estimate_no_incoming}
			\sup_{u\in (u_\infty,u_0)}\norm{(v\pv,\rho_-^{1/2}\sl)\Vbt^k \psi}_{\Hbt^{\vec{a};0}(\outcone{})}+\sup_{v\in(v_0,v_\infty)}\norm{(u\pu,\rho_-^{1/2}\sl) \Vbt^k \psi}_{\Hbt^{\vec{a};0}(\Cbar_v)}+\norm{ \Ve\psi}_{\Hbt^{\vec{a};k}(\D)}\\
			\lesssim_{\mathcal{N},V,k} \norm{rf}_{\Hbt^{\vec{a}^f;k}(\D)}+\sum_{j'\leq k}\norm{(rT)^{j}(rf)}_{\Hb^{a_-+2;2(k-j)(\incone)}}
			+ \norm{\psi^{\incone}}_{\Hb^{a_-;2k+2}(\incone)}=:\tilde{\mathcal{X}}^k
		\end{nalign}
        
	\end{prop}
 
	\begin{proof}
		\textit{a) Step 1 (base case):}
		Let us first prove the result for $k=4$.
		We apply consecutively  \cref{prop:en:main} to infer that for some admissible inhomogeneous $\vec{a}^{\mathcal{P}}$  and for  
        $\delta=\delta(V,\mathcal{N})$ sufficiently small (in particular smaller than the gap of $V,\mathcal{N}$):
		\begin{nalign}\label{eq:en:perturbed_step1}
      	\norm{\Ve  \psi}_{\Hb^{\vec{a};k}(\D)}\lesssim \sum_{i\leq k} \norm{T^{i}\psi}_{\Hb^{a_-;k-i}(\incone)}+\norm{rf}_{\Hb^{\vec{a}^{f};k}(\D)}+\norm{r\mathcal{N}[\phi]+rV\phi}_{\Hb^{\vec{a}^{\mathcal{P}};k}(\D)}\\
      	\lesssim \sum_{i\leq k} \norm{T^{i}\psi}_{\Hb^{a_-;k-i}(\incone)}+\norm{rf}_{\Hb^{\vec{a}^f;k}(\D)}+\abs{u_0}^{-\delta}\norm{r\mathcal{N}[\phi]+rV\phi}_{\Hb^{\vec{a}^{\mathcal{P}}+\delta;k}(\D)}\\
      	\lesssim \sum_{j\leq k}\norm{T^{j}rf}_{\rho_-^{a_-^f}\Hb^{;2(k-j)}(\incone)}
			+ \norm{\psi^{\incone}}_{\rho_-^{a_-}\Hb^{;2k+2}(\incone)}+\norm{rf}_{\Hb^{\vec{a}^f;k}(\D)}+\abs{u_0}^{-\delta}(\norm{\Ve \psi}^{n^{\mathcal{N}}}_{\Hb^{\vec{a};k}(\D)}+\norm{\Ve \psi}_{\Hb^{\vec{a};k}(\D)}),
      \end{nalign}
      where, in the last line, we have used \cref{cor:en:VN_estimates} to estimate the perturbations, as well as \cref{lemma:en:recovering_initial_data} to estimate the terms along $\incone$.

		Using the boundedness assumption on $\psi^{\incone},f$, we can take $\abs{u_0}$ sufficiently large such that the last two terms on the right hand side can be absorbed into the left hand side.

		\textit{a) Step 2 (commuted estimates):}
		We obtain higher regularity in the same spacetime region in two steps. First, we commute with the symmetries of $\Box_\eta$ and use the previously obtained estimates to obtain \cref{eq:en:perturbed_energy_estimate} in a region $\abs{u_0}>R'(k)$ for $R'(k)$ possibly larger than $-u_0$.
  
		Second, we extend the result to the remaining \textit{compact in $u$} region by a local energy estimate and Gr{\"o}nwall's inequality with weights only towards $\scrip$. More precisely, we use \cref{lemma:scat:sobolev} to write for $k\geq5$
        \begin{nalign}\label{eq:en:perturbed_step10000}
          	\norm{\Ve  \psi}_{\Hb^{\vec{a};k}(\D)}
          	\lesssim \sum_{i\leq k} \norm{T^{i}\psi}_{\Hb^{a_-;k-i+1}(\incone)}+\norm{rf}_{\Hb^{\vec{a}^f;k}(\D)}+\abs{u_0}^{-\delta}(\norm{\Ve \phi}^{n^{\mathcal{N}}-1}_{\Hb^{\vec{a};4}(\D)}+1)\norm{\Ve \phi}_{\Hb^{\vec{a};k}(\D)}.
        \end{nalign}
        and then apply Gr{\"o}nwall.

		\textit{b):} In order to prove \cref{eq:en:perturbed_energy_estimate_no_incoming}, we instead use the \textit{b)} parts of \cref{cor:en:VN_estimates,lemma:en:recovering_initial_data}.
	\end{proof}

	\begin{prop}\label{lemma:en:metric_perturbed_energy_estimate}
		Let $f,\psi^{\incone},\phi^{\outconeFar},\vec{a}^f,\vec{a},V,\mathcal{N},k$ be as in \cref{cor:en:perturbed_energy_estimate}, and let $P_g$ be a short-range quasilinear perturbation. 
  Finally, let $\phi$ be a solution to \cref{eq:en:characteristicIVP_perturbed} with \cref{eq:en:assumption0}. 
		Then there exists $R$ depending on $V,\mathcal{N},P_g$ sufficiently large such that for $-u_0>R$ and $\D=\D^{u_\infty,v_\infty}_{u_0,v_0}$ the estimate \cref{eq:en:perturbed_energy_estimate} holds. 
        Provided  $\mathcal{P}$ is compatible with the no incoming radiation condition, then \cref{eq:en:perturbed_energy_estimate_no_incoming}  holds.
	\end{prop}
	\begin{proof}		
		Whereas in \cref{cor:en:perturbed_energy_estimate} we were able to treat the perturbation as error terms in the bulk estimates, for metric perturbations, we need to use the boundary terms in \cref{prop:en:pastestimate} as well.
		Indeed, the lemma follows from applying the same vectorfield multipliers as in \cref{prop:en:pastestimate} and \cref{prop:en:futureestimate}, while keeping track of the boundary terms as well.
		Therefore, we will only concentrate on the terms coming from $P_g$, since if we can bound these, the $\mathcal{N},V$ terms are treated as before.
        Furthermore, to avoid boundary terms arising on $\partial \D^-$, we only work out energy estimates in the causal past of the constant $u,v$ spheres on $\D^-$ bounded by null cones to the future.
        This implies a uniform boundedness of the boundary terms. 
        Integrating these yields the required bulk estimate as well.
		
		Let us showcase how to achieve this, by proving \cref{item:en:pastestimate_inhom} with $P_g$ error, the rest of the estimates follow similarly.
		More specifically, let us shows that in the case $a_-^f+0.5\leq a_0^f$, when we used the multipliers $W_1=v^{-\epsilon}\abs{u}^{2a_-}\pv$ and $W_2=v^{-\epsilon-1}\abs{u}^{2a_-+1}\pu$, all error terms from $P_g$ are bounded.
		We will also use the observation form the proof of \cref{prop:en:pastestimate}, that instead of using $v^{-\epsilon}$ weight, we can integrate the boundary terms with a weight $v^{1+\epsilon}$ to recover the same bulk estimates.
		Let $2\delta>0$ be the gap associated to $\mathcal{N},V,P_g$.
		
		\textit{Step 1 (energy estimate):}
		We use $W_1=\abs{u}^{2a_-}\pv$ multiplier in the domain $\D^-=\D^{u_0,v_0,-}_{u_\infty,v_\infty}$ for $\tilde\T$ and get error terms from $P_g$ given by the following
		\begin{subequations}\label{eq:en:error_metric_pert}
			\begin{equation}
				2W_1r\phi\Omegalin\pu\pv r\phi=\pu\Big(\Omegalin \abs{u}^{2a_-}(\pv r\phi)^2\Big)-(\pv r\phi)^2\pu(\Omegalin\abs{u}^{2a_-})
			\end{equation}
			\begin{nalign}\label{eq:en:error3_metric_pert}
				W_1r\phi\glin_{AB}\sl^A\sl^Br\phi=\sl^A(W_1r\phi \glin_{AB}\sl^Br\phi)-\frac{1}{2}\glin_{AB}W_1(\sl^Ar\phi\sl^Br\phi)-W_1r\phi\sl^Br\phi\sl^A\glin_{AB}\\
				=\sl^A(W_1r\phi \glin_{AB}\sl^Br\phi)-\pv (\frac{\abs{u}^{2a_-}}{2}\glin_{AB}\sl^Ar\phi\sl^Br\phi)-W_1r\phi\sl^Br\phi\sl^A\glin_{AB}+\sl^Ar\phi\sl^Br\phi \pv(\frac{\abs{u}^{2a_-}}{2}\glin_{AB}).
			\end{nalign}
		\end{subequations}
		We introduce $w=(\rho_-\rho_0)^{-\delta}\Vb\big((\Omegalin)^2+(r^2\glin)^2\big)^{1/2}$ and bound the integral in $\D^-$ for all terms from \cref{eq:en:error_metric_pert} by
		\begin{nalign}\label{eq:en:error_metric_pert_bound}
			R^{-\delta}\norm{w}_{L^{\infty}(\D^-)}\Big(\sup_u\norm{(v\pv,\rho_-^{1/2}\sl)r\phi}_{\Hb^{\vec{a};0}(\outcone{})}+\sup_v\norm{(u\pu,\boxed{\rho_-^{1/2}\sl}) r\phi}_{\Hb^{\vec{a};0}(\Cbar_{v}^-)}+\norm{\Ve r\phi}_{\Hb^{\vec{a};0}(\D^-)}\Big)^2,
		\end{nalign}
        where the boxed term is not necessary and only included for future convenience.
        Importantly, we can use the short-range assumption \cref{item:en:quasilinear_perturbation} to conclude that for $\delta$ sufficiently small we have 
        \begin{multline}
            \norm{r^{\delta}\Omegalin}_{L^\infty(\D^-)}\lesssim\norm{r^{\delta}uvf^{\Omega}(\Vb \psi)^{n^{P_g}-1}}_{L^\infty(\D^-)}\lesssim \norm{\Vb \psi}^{n^{P_g}-1}_{\Hb^{\vec{a};3}(\D^-)}\norm{f^\Omega}_{\Hb^{\delta+(1,2)-(n^{P_g}-1)\vec{a};3}(\D^-)}\\
            \lesssim \norm{\Vb \psi}^{n^{P_g}-1}_{\Hb^{\vec{a};3}(\D^-)}\norm{f^\Omega}_{\Hb^{\vec{a}^{P_g};3}(\D^-)}\lesssim \norm{\Vb \psi}^{n^{P_g}-1}_{\Hb^{\vec{a};3}(\D^-)}.
        \end{multline}
        A similar estimate also holds for $\slashed{g}$, so we have
        \begin{equation}\label{eq:en:w_estimate}
            \norm{w}_{L^\infty(\D^-)}\lesssim\norm{\Vb \psi}^{n^{P_g}-1}_{\Hb^{\vec{a};3}(\D^-)}
        \end{equation}
        
        We complement this with the $W_2=\frac{\abs{u}^{2a_-+1}}{v}\pu$ estimate. The corresponding error terms are
		\begin{subequations}
			\begin{equation}
				2W_2r\phi\Omegalin\pu\pv r\phi=\pv\Big(\Omegalin \abs{u}^{2a_-+1}v^{-1}(\pu r\phi)^2\Big)-(\pu r\phi)^2\pv(\Omegalin\abs{u}^{2a_-+1}v^{-1})
			\end{equation}
			\begin{nalign}
				W_2r\phi\glin_{AB}\sl^A\sl^Br\phi=\sl^A(W_2r\phi \glin_{AB}\sl^Br\phi)-\frac{1}{2}\glin_{AB}W_2(\sl^Ar\phi\sl^Br\phi)-W_2r\phi\sl^Br\phi\sl^A\glin_{AB}\\
				=\sl^A(W_2r\phi \glin_{AB}\sl^Br\phi)-\pu (\frac{\abs{u}^{2a_-+1}}{2v}\glin_{AB}\sl^Ar\phi\sl^Br\phi)-W_2r\phi\sl^Br\phi\sl^A\glin+\sl^Ar\phi\sl^Br\phi \pu(\frac{\abs{u}^{2a_-+1}}{2v}\glin_{AB}).
			\end{nalign}
		\end{subequations}
		As before, all error terms after integration can be bounded by \cref{eq:en:error_metric_pert_bound}. Summing the two estimates yields for any strictly admissible $\vec{a}^{\mathcal{P}}$
		\begin{nalign}\label{eq:en:quasilinear_proof_estimate1}
			\sup_u \norm{(v\pv,\rho_-^{1/2}\sl) r\phi}^2_{\Hb^{\vec{a};0}(\outcone{})}+\sup_v\norm{(u\pu,\rho_-^{1/2}\sl) r\phi}^2_{\Hb^{\vec{a};0}(\Cbar_v)}+\norm{\Ve r\phi}^2_{\Hb^{\vec{a};0}(\D^-)}\\
			\lesssim
			\norm{f}^2_{\Hb^{\vec{a}^f;0}}+\norm{\mathcal{N}[\phi]}^2_{\Hb^{\vec{a}^{\mathcal{P}};0}(\D^-)}
			+\norm{V r\phi}_{\Hb^{\vec{a}^{\mathcal{P}};0}(\D^-)}^2+R^{-\delta}\norm{w}_{L^{\infty}(\D)}\Big(\sup_u\norm{(v\pv,\rho_-^{1/2}\sl)r\phi}_{\Hb^{\vec{a};0}(\outcone{})}\\
			+\sup_v\norm{(u\pu,\rho_-\sl) r\phi}_{\Hb^{\vec{a};0}(\Cbar_{v})}+\norm{\Ve r\phi}_{\Hb^{\vec{a};0}(\D^-)}\Big)^2+(\mathcal{X}^0)^2,
		\end{nalign}
		where we used $\mathcal{X}^0$ from \cref{eq:en:perturbed_energy_estimate}.
  
		\textit{Step 2:}
		Up to this point, we still have error terms involving $\phi$ on the right hand side of \cref{eq:en:quasilinear_proof_estimate1}.
		To close the nonlinear parts of the estimate, i.e.~absorb these nonlinearities from the right hand side of \cref{eq:en:quasilinear_proof_estimate1}, we first work at $k=4$ regularity.
        We recall from \cref{lemma:en:short_range_computations}
		\begin{equation}
			[\Diff^1_{\b}(\D),P_g]=(\Diff^1_{\b}(\D)f_{P_g}[\phi])\Diff_{\b}^2(\D)
		\end{equation}
		and conclude that error terms arising from commuting $P_g$ with the symmetries of $\Box_\eta$ can be included in the potential perturbations.
        Hence, for any $k\geq0$ and stictly admissible $\vec{a}^{\mathcal{P}}$ we get
		\begin{nalign}
			\sup_u \norm{(v\pv,\rho_-^{1/2}\sl)  \Vb^kr\phi}^2_{\Hb^{\vec{a};0}(\outcone{})}+\sup_v\norm{(u\pu,\rho_-^{1/2}\sl)\Vb^k r\phi}^2_{\Hb^{\vec{a};0}(\Cbar_v)}+\norm{(u\pu,v\pv,\rho_-^{1/2}\sl)r\phi}^2_{\Hb^{\vec{a};k}(\D^-)}\\
			\lesssim
			\norm{f}^2_{\Hb^{\vec{a}^f;k}}+\norm{\mathcal{N}[\phi]}^2_{\Hb^{\vec{a}^{\mathcal{P}};k}(\D^-)}
			+\norm{\Ve  r\phi}_{\Hb^{\vec{a}^{\mathcal{P}};k}(\D^-)}^2+R^{-\delta}\norm{w}_{L^{\infty}(\D^-)}\Big(\sup_u\norm{(v\pv,\rho_-^{1/2}\sl)r\phi}_{\Hb^{\vec{a};k}(\outcone{})}\\
			+\sup_v\norm{(u\pu,\rho_-\sl) \Vb^k r\phi}_{\Hb^{\vec{a};0}(\Cbar_{v})}+\norm{\Ve r\phi}_{\Hb^{\vec{a};k}(\D^-)}\Big)^2+(\mathcal{X}^{k})^2.
		\end{nalign}
		Using \cref{cor:en:VN_estimates} for $V,\mathcal{N}$,\cref{eq:en:w_estimate} for $w$ and $R$ sufficiently large we can close the energy estimates to obtain \cref{eq:en:perturbed_energy_estimate}.
		
		\textit{Step 3 (higher regularity):} To obtain further commuted estimates, we proceed as in the last step of \cref{cor:en:perturbed_energy_estimate}.
	\end{proof}

 Similarly to \cref{lemma:en:past_estimate_with_incoming_data_L2}, we also have the following linear statement (with $\mathcal{N}=0$):
    \begin{lemma}\label{lemma:en:past_estimate_L2_perturbed}
        Let $\vec{a},\vec{a}^f$ be as in \cref{lemma:en:past_estimate_with_incoming_data_L2}.
        Let $P^g,V$ be linear short-range perturbations, with $n^{P_g}=1$.
        There exists $-u_0$ sufficiently large such that the solution $\phi$ to \cref{eq:en:characteristicIVP_perturbed} satisfies \cref{eq:en:past_estimate_with_incoming_data_L2}.
    \end{lemma}
    \begin{proof}
        The proof follows from the proof of \cref{lemma:en:metric_perturbed_energy_estimate} in the same way in which \cref{lemma:en:past_estimate_with_incoming_data_L2} followed from \cref{prop:en:pastestimate}.
    \end{proof}
    	
	\begin{obs}[Regularity of coefficients]\label{rem:en:regularity_of_coefficients}
		The assumption that $V,\mathcal{N}$ are in $\Hb^{\vec{a};\infty}$ can be weakened to $V,\mathcal{N}$ having the same regularity as $\phi$, as we can then still control the products in the same way as for the nonlinear-in-$\phi$ terms.
		In particular, assuming that $f_{\mathcal{N}},f_V\in\Hb^{\vec{a};k}(\D)$ is sufficient for to proving \cref{cor:en:perturbed_energy_estimate} at regularity $k$.
		
		Let us also note that when $\mathcal{N}=0$ and $P_g$ does not depend on $\phi$, we can take $k\geq0$, as we need not close a nonlinear estimate. However, even in this case, we still do require the coefficients to be at least $k\geq4$ regular.
	\end{obs}
    
    We finally also elevate \cref{lemma:en:energy_no_incoming} to the perturbed setting.
    The purpose of this lemma is to treat data and perturbations~$\mathcal{P}$ that decay weakly or even grow.
    We therefore continue to use the notation of \cref{def:en:short_range}, but weaken the requirement on the weights as specified below:
	\begin{prop}\label{lemma:en:no_incoming_perturbed_en_estimate}
		Let $\vec{a},\vec{a}^f,m,f,\psi^{\incone}$ be as in \cref{lemma:en:energy_no_incoming}.
		Let $k\in\N_{\geq4}$ and $\mathcal{P}$ be perturbations compatible with the no incoming radiation condition with the requirements on the weights weakened to  $n^{\bullet}\vec{a}+\vec{a}^{\bullet}>\vec{a}+(2,2,1)$ for $\bullet\in\{V,\mathcal{N},P_g\}$.
		Then there exists $-u_0(f,\psi^{\incone},V,\mathcal{N},P_g)$ sufficiently large that for $\phi$ the solution to \cref{eq:en:characteristicIVP_perturbed} in $\D=\D^{u_\infty,v_\infty}_{u_0,v_0}$ satisfies
		\begin{nalign}\label{eq:en:no_incoming_estimate}
			\sup_{u\in (u_\infty,u_0)}\norm{(v\pv,\rho_{\scri}^{1/2}\sl) \Vbt^k \psi}_{\Hbt^{\vec{a};0}(\outcone{})}+\sup_{v\in(v_0,v_\infty)}\norm{(u\pu,\rho_{\scri}^{1/2}\sl) \Vbt^k\psi}_{\Hbt^{\vec{a};0}(\Cbar_v)}+\norm{ \Ve \psi}_{\Hbt^{\vec{a};k}(\D)}\\
			\lesssim_{\mathcal{N},V,j,\gamma,k} \norm{rf}_{\Hbt^{\vec{a}^f;k+m}(\D)}
			+ \norm{\Vbt^{\red{2k+2m+2}}\psi|_{\incone}}_{\rho_-^{a_-}\Hb^{0}(\incone)}
		\end{nalign}
	\end{prop}
	\begin{proof}
		Let us work in $\D^-$.
		Let $2\delta$ be the gap of the perturbation $\mathcal{P}$.
		We commute the equation $m$ times to get $\Box T^n\phi=T^n f+T^n\mathcal{P}\phi$.
		We estimate for $j<m$ using \cref{corr:ODE:du_dv}
		\begin{nalign}\label{eq:en:Tj_from_Tm}
			\norm{T^j\psi}_{\Hb^{a_-+j,a_0+j;s}(\D^-)}\lesssim\norm{T^j\psi|_{\incone}}_{\Hb^{a_-+j;s}(\incone)}+\norm{T^{j+1}\psi}_{\Hb^{a_-+j,a_0+j+1;s}(\D^-)}\\\lesssim
            \sum_{j\leq j'<m}\norm{T^{j'}\psi|_{\incone}}_{\Hb^{a_-+j';s}(\incone)}+\norm{T^{m}\psi}_{\Hb^{a_-+j,a_0+m;s}(\D^-)}.
		\end{nalign}
		We use \cref{eq:en:Tj_from_Tm} and the no incoming radiation control of $V,\mathcal{N},P_g$, to control the perturbations for $s\geq4$
		\begin{nalign}
			\norm{T^mV\psi}_{\Hb^{\vec{a}^f+\vec{m};s}(\D^-)}&\lesssim\sum_{n\leq m}\norm{\Ve T^n\psi}_{\Hb^{\vec{a}+\vec{n}-\vec{\delta};s}(\D^-)}\lesssim \abs{u_0}^{-\delta}\Big(\norm{\Ve T^m\psi}_{\Hb^{\vec{a}+\vec{m};s}(\D^-)}+\text{LHS}\cref{eq:en:no_incoming_estimate}\Big)\\
			\norm{rT^m\mathcal{N}[\phi]}_{\Hb^{\vec{a}^f+\vec{m};s}(\D^-)}&\lesssim\abs{u_0}^{-\delta}\sum_{n_1+...+n^{n^{\mathcal{N}}}\leq m}\norm{\Ve T^{n_1}\psi}_{\Hb^{\vec{a}+\vec{n}_1;s}(\D^-)}\cdots\norm{\Ve T^{n_{n^\mathcal{N}}}\psi}_{\Hb^{\vec{a}+{\vec{n}}_{n^\mathcal{N}};s}(\D^-)}\\
			&\lesssim\abs{u_0}^{-\delta}\Big(\norm{\Ve T^m\psi}_{\Hb^{\vec{a}+\vec{m};s}(\D^-)}+\text{LHS}\cref{eq:en:no_incoming_estimate}\Big)^{n^{\mathcal{N}}}		\\
			\norm{[rT^m,P_g]\phi}_{\Hb^{\vec{a}^f+\vec{m};s}(\D^-)}&\lesssim\abs{u_0}^{-\delta}\sum_{n_1+...+n_{n^{P_g}}\leq m}\norm{\Ve T^{n_1}\psi}_{\Hb^{\vec{a}+\vec{n}_1;s}(\D^-)}\cdots\norm{\Ve T^{n^{P_g}}\psi}_{\Hb^{\vec{a}+\vec{n}_{n^{P_g}};s}(\D^-)}\\
			&\lesssim\abs{u_0}^{-\delta}\Big(\norm{\Ve T^m\psi}_{\Hb^{\vec{a}+\vec{m};s}(\D^-)}+\text{LHS}\cref{eq:en:no_incoming_estimate}\Big)^{n^{P_g}}	.
		\end{nalign}
		For $P_g=0$, we simply estimate $T^m\psi$ using \cref{lemma:en:energy_no_incoming} and closing a bootstrap as in \cref{cor:en:perturbed_energy_estimate}. For $P_g\neq0$, we only need to worry about the $P_gT^m\psi$ term, and the corresponding boundary terms are estimated as in \cref{lemma:en:metric_perturbed_energy_estimate}. Higher order estimates follow as in \cref{cor:en:perturbed_energy_estimate}.
	\end{proof}
	
    \subsection{Energy estimates IV: The linear wave equation \texorpdfstring{$\Box_{\eta}\phi=0$}{} with nontrivial incoming radiation}\label{sec:en:incoming}
 
	We finish this section with an estimate treating non-trivial incoming radiation $\psi^{\outconeFar
 }\neq 0$.
 In practice, we will not be using this estimate, as we can always absorb the incoming radiation into an inhomogeneity by subtracting it from the solution, cf.~\cref{thm:scat:scat_incoming}.
 This trick, however, loses regularity; the estimate below can be used to remove this loss of regularity, see also \cref{rem:scat:derivative_loss_incoming}.
 Let us also recall that in \cref{lemma:en:energy_no_incoming,lemma:en:no_incoming_perturbed_en_estimate}, we already allowed for nontrivial incoming radiation; indeed, we will use the latter results to argue for uniqueness of scattering solutions in \cref{sec:scat:scat}.

 
Throughout this section, we impose the additional assumption that $v_\infty-v_0<\abs{u_\infty}^{1/2}\gg 1$. Since we would later on take the limit $u_{\infty}\to-\infty$, this does not lose generality.
 

\begin{prop}\label{prop:scat:energy_incoming_rad}
	Let $\phi$ solve \cref{eq:en:characteristicIVP} with $\psi^{\incone}=0=f$, and with $r\phi^{\outconeFar}=\psi^{\outconeFar}\in\Hb^{a_0;\infty}$ satisfying $\psi^{\C_{u_\infty}}(v,\omega)=0$ for $v\in[v_0,v_0+1)$.\footnote{This assumption's purpose is to avoid having to deal with boundary terms when integrating for transversal derivatives of $\psi$ along $\outconeFar$.}
	Let $a_0\in(-\infty,1/2).$
	
Then we have for any $k\geq0$:
	
	\begin{multline}\label{eq:en:incoming_prop}
		\norm{\psi}_{\Hb^{0-,\min(0,a_0)-;k}(\D^-)}+\norm{(v\pv,\sl) \psi}_{\Hb^{0-,a_0-;k}(\D^-)}+\norm{u\pu \psi}_{\Hb^{1/2;,a_0-;k}(\D^-)}\\
		\lesssim \norm{(v\pv,\sl) \psi^{\outconeFar}}_{\Hb^{a_0;k}(\outconeFar)}+
		\frac{1}{\abs{u_\infty}}\norm{\psi^{\outconeFar}}_{\Hb^{a_0+1;2k}(\outconeFar)}
	\end{multline}
	\end{prop}
	\begin{proof}
		\emph{Step 1): Weight towards $I^0$.} We will only show the result for $a_0=1/2-$. 
		For any other $a_0$, one multiplies all vector field multipliers in the proof with $v^{-(1-2a_0)}$, and observes that only already controlled or positive terms appear.
		
		\emph{Step 2): Multiplier.} We first prove the estimate for $k=0$.
		We apply the conformal vector field multiplier $V=u^2\pu+v^2\pv$ for $\tT$ to obtain:
		\begin{equation}
			\sup_u\norm{(v\pv,\sl)\psi}_{\rho_0^{1/2}H^{;0}(\outcone{})}+\sup_v\norm{u\pu\psi}_{\rho_0^{1/2}\rho_-^{1/2}\Hb^{;0}(\Cbar_{v})}\lesssim \norm{(v\pv,\sl) \psi^{\outconeFar}}_{\Hb^{1/2;0}(\outconeFar)}
		\end{equation}
		Integrating the boundary terms already yields \cref{eq:en:incoming_prop} without the 0th order term.
		Notice, that we can use \cref{corr:ODE:du_dv} to recover control over $\psi$ from $\pv\psi$:
		\begin{equation}
			v\pv\psi\in\Hb^{0-,1/2-;0}(\D^-) \text{ and }\psi|_{\incone}=0\implies \psi\in\Hb^{0-,0-;0}(\D^-).
		\end{equation}
		
		\emph{Step 3): Commutators.}
		For higher $k$, we commute with rotations ($\Omega$), boosts (${B}_i$) and scaling $S$.  $\Omega$-commutators introduce no loss in regularity. 
		Commuting the equation with scaling ($S=u\partial_u+v\partial_v$) and boost  vector fields, on the other hand, gives rise to  loss of derivatives.
		Let's focus on scaling, the boosts behave similarly.
		
		First, note that we can estimate the $\Hb$ norm of $\partial_u \psi$ as follows. Using that $(u\pu)^n\psi=0$ at $\incone\cap\outconeFar$, we have for all $a\in\mathbb R$ that
		\begin{nalign}
			\partial_u\psi|_{\C_{u_{\infty}}}=\int_{v_0}^v\frac{1}{r^2}\Dl \psi\implies \norm{u_\infty\partial_u\psi|_{\C_{u_{\infty}}}}_{\Hb^{a;k}(\C_{u_\infty})}\lesssim_a\frac{v_\infty-v_0}{\abs{u_\infty}^{3/2}}\norm{\Dl \psi}_{\Hb^{a;k}(\C_{u_\infty})}\lesssim\abs{u_\infty}^{-1}\norm{\psi}_{\Hb^{a;k+2}(\C_{u_\infty})}\\
			\implies \norm{(u\partial_u)^j\psi|_{\C_{u_{\infty}}}}_{\Hb^{a;k}(\C_{u_\infty})}\lesssim_{a,j}\abs{u_\infty}^{-1}\norm{\psi}_{\Hb^{a;2j+k}(\C_{u_\infty})}.
		\end{nalign}
		
		We may thus estimate the boundary term for $S^k\psi$ as follows:
		\begin{nalign}
			S^kr\phi|_{\C_{u_{\infty}}}=(u\partial_u+v\partial_v)^k\psi|_{\C_{u_{\infty}}}=(v\partial_v)^k\psi|_{\C_{u_{\infty}}}+\sum_{\substack{k_1+k_2=k\\k_2>0}}c_{k_1,k_2}(v\partial_v)^{k_1}(u\partial_u)^{k_2}\psi|_{\outconeFar}\\\implies
			\norm{(\partial_v,r^{-1},r^{-1}\sl)S^k\psi|_{\outconeFar}}_{\Hb^{-1/2;0}(\outconeFar)}\lesssim_{k} \norm{\psi^{\C_{u_\infty}}}_{\Hb^{-1/2;k+1}(\outconeFar)}+\abs{u_\infty}^{-1}\norm{\psi^{\C_{u_\infty}}}_{\Hb^{-1/2;2k+1}(\outconeFar)}.
		\end{nalign}
		On the other hand, since  $\phi$  vanishes identically for $v\in[v_0,v_0+1)$, commutations do not generate any boundary terms along $\incone$.		
		
	\end{proof}

    \begin{rem}[Optimal decay with incoming radiation]\label{rem:en:optimatl decay from ccdata}
		Note that \cref{prop:scat:energy_incoming_rad} only gives weak decay towards $I_0$ even for compactly supported incoming radiation.
        In fact, one cannot prove arbitrary fast decay towards $I_0$ even for compactly supported $\pv\psi^{\outconeFar}$: If $\pv\psi^{\outconeFar}=h(v,\omega)$ along $\scrim$, the solution will generically decay like $\psi\sim 1$ towards $I^0$. (For instance, if $\ell=0$, then a solution is given by $\psi_{\ell=0}=\int h_{\ell=0} \dd v$, and if $\ell=1$, then $\psi_{\ell=1}=\int h_{\ell=1}-\frac2r\int\int h_{\ell=1}$ and so on.) See already \cref{cor:prop:solution_with_radiation_scri}.
        
	\end{rem}

	\section{Scattering theory for perturbations of \texorpdfstring{$\Box_{\eta}\phi=0$}{the linear wave equation}}\label{sec:scat:scat}

	We now elevate the (uniform in $u_\infty$) results for the finite characteristic initial value problem from \cref{sec:en} to the scattering problem. 
    We first prove a scattering statement without incoming radiation for admissible short-range perturbations in \cref{sec:scat:noincoming}, and we explain how to remove the extra condition \cref{eq:en:admissible_inhom1} in the class of short-range perturbations in \cref{cor:scat:enlarged_admissible_set}.
    We treat the case of incoming radiation in \cref{sec:scat:incoming}.
    We then provide a scattering theory in the case where the inhomogeneity decays weakly (or grows) in \cref{sec:scat:weak_decay}.
    We finally treat the case of long-range perturbations in \cref{sec:scat:longrange}. %
  Henceforth, we write $\D$ to denote $\D^{-\infty,\infty}_{u_0,v_0}$.

    Before we begin, we give the following definition and well-posedness result:
 \begin{defi}\label{defi:scat:scat_general}
		Let $\vec{a},\vec{a}^f$ be admissible and $P_g,V,\mathcal{N}$ be short-range perturbations.
		Fix scattering data $\psi^{\incone}\in \Hb^{a_-;k+1}(\incone)$, $v\pv \psi^{\scrim}\in\Hb^{a_0;k}(\scrim)$ and $rf\in\Hb^{\vec{a}^f;k}(\D)$, for $k\in \mathbb{Z}_{\geq {0}}$.
        We will write $\psi^{\scrim}:=\int_{v_0}^v\pv\psi^{\scrim}$ (notice that this will generally not be the limit of $\psi$ towards $\scrim$).
        We call $\phi$ the scattering solution to
		\begin{equation}\label{eq:scat:scattering_def}
			\Box \phi=f+\mathcal{P}[\phi],\qquad \mathcal{P}[\phi]=(P_g[\phi]+V)\phi+\mathcal{N}[\phi]
		\end{equation}
        \textbf{with incoming radiation} $v\pv\psi^{\scrim}\in\Hb^{a_0;k}(\scrim)$ if \cref{eq:scat:scattering_def} holds and 
        \begin{equation}\label{eq:scat:scat_with_inc_condition}
			\forall v_\infty>v_0\text{ we have }\lim_{u\to-\infty} \norm{\pv(r\phi)-\pv(\psi^{\scrim})}_{L^2(C^{v_{\infty}}_{u})}\to 0,\qquad r\phi|_{\incone}=\psi^{\incone}.
		\end{equation}
		When $\pv\psi^{\scrim}=0$, we say that $\phi$ is the scattering solution to \cref{eq:scat:scattering_def} with \textbf{no incoming radiation}.
	\end{defi}
 
	We also require the following standard result concerning the characteristic IVP for quasilinear waves: 
 
	\begin{thm}[Local characteristic solution (\cite{rendall_reduction_1990})]\label{thm:scat:local}
		Let $\psi^{\incone}\in H^{\infty}(\incone),\,\psi^{\outconeFar}\in H^{\infty}(\incone),\,f\in H^{\infty}(\D)$ such that $\psi^{\incone}$ and $\psi^{\outconeFar}$ satisfy the corner condition.
		Then the problem 
		\begin{nalign}\label{eq:scat:cutoff_scattering}
			&\Box\phi=f+\mathcal{P}[\phi]\\
			&\psi|_{\incone}=\psi^{\incone},\qquad \psi|_{\outconeFar}=\psi^{\outconeFar},
		\end{nalign}
        for $\mathcal{P}$ as in \cref{eq:scat:scattering_def},
		has a unique solution $\phi$ in $\D^{u_\infty,v_0+\delta v}_{u_{\infty}+\delta u,v_0}$ for $\delta u,\delta v$ sufficiently small depending only on the size of the norms $\norm{(\mathrm{V}^3\phi)|_{\incone}}_{L^2(\incone)},\norm{(\mathrm{V}^3\phi)|_{\outconeFar}}_{L^2(\outconeFar)},\norm{f}_{H^2(\D^-)}$, where $V=\{\partial_{x_i},\partial_t\}$. Moreover, $\phi|_{u=u_1}$ for all $u_1\in(u_\infty,u_\infty+\delta u)$ satisfies the corner condition.
	\end{thm}
	
	\subsection{Scattering theory with no incoming radiation}
	\label{sec:scat:noincoming}
	
	We now use the \textit{a priori} uniform bounds proved in the previous section together with \cref{thm:scat:local} result to infer a scattering theory for \cref{eq:scat:scattering_def}.
	\begin{thm}\label{thm:scat:scat_general}
        Fix admissible weight $\vec{a}$ and corresponding inhomogeneous weight $\vec{a}^f$. Let $k\in\N_{\geq 4}$  and 
        \begin{equation}\label{eq:scat:k-data}
            \psi^{\incone}\in\Hb^{a_-;2k+2}(\incone), \,rf\in\Hb^{\vec{a}^f;k}(\D),\,T^j(rf)\in\Hb^{a_{-}^f;2(k-j)}(\incone),\qquad \text{for } j\leq k.
        \end{equation}
		Let $\mathcal{X}^4\leq 1$ as defined in \cref{eq:en:perturbed_energy_estimate}. Let $\mathcal{N},V,P_g$ be short-range perturbations.
        \begin{enumerate}[label=\alph*)]
            \item There exists $-u_0$ sufficiently large, depending only on the perturbations $\mathcal{P}$, such that there exists a unique scattering solution $r\phi\in \Hb^{a_-,a_0,a_+;k}(\D^{-\infty,\infty}_{u_0,v_0})$ to \cref{eq:scat:scattering_def} with no incoming radiation. Here, the uniqueness is understood with respect to the class of solutions satisfying the \emph{a priori} bound $\norm{\Ve r\phi}_{\Hb^{a_--;4}(\D\cap[v_0,v_{\infty}))}<\infty$ for any $v_{\infty}<\infty$. 
            \item For linear perturbations, i.e.~$\mathcal{N}=0,n^{P_g}=1$, 
            $\phi\in H^1_{\mathrm{loc}}(\D^-)$ is sufficient for uniqueness.
            \item\label{item:} Moreover, for any $k\geq4$, we have that \cref{eq:scat:k-data} implies $
			(\Ve  r\phi)\in\Hb^{\vec{a};k}(\D)$.
            \item Provided that $f,\,V,\,\mathcal{N},\,P_g$ also satisfy the no incoming radiation condition, we have the  following improved estimate:
		\begin{nalign}\label{eq:scat:higher_no_incoming_estimate}
			\psi^{\Cbar_{v_0}}\in{\rho_-^{a_-}\Hb^{2k+2}(\Cbar_{v_0})},\, rf\in\Hbt^{\vec{a}^f;k}(\D),\,(rT)^j(rf)\in \Hb^{a_-^f+1;2(k-j)}(\incone)\implies 
			(\Ve  r\phi)\in\Hbt^{\vec{a};k}(\D).
		\end{nalign}
        \end{enumerate}
		
	\end{thm}

	\begin{proof}
		Without loss of generality, we take $\psi^{\incone},f$ smooth, otherwise the result follows using approximations.
        By taking $u_0$ sufficiently large and by using \cref{eq:scat:data_incoming}, we may assume that $\norm{\Vb^k\phi|_{\incone}}_{\Hb^{a_-;1}(\incone)}\leq1$ for $k=4$. 
        We now first construct a sequence of finite solutions satisfying uniform bounds, then show that this sequence attain a limit, and finally show that this limit is the unique solution.
    
        We also take $-u_0$ to be large enough so that \cref{lemma:en:metric_perturbed_energy_estimate} applies. (Strictly speaking, we pose the requirements on the sequence $\phi_{u_{\infty}}$ introduced below.)

		\textit{a)} \textit{Step 1: Uniform bounds.} 
        Let $\tilde{\chi}$ be a smooth cutoff satisfying $\tilde{\chi}(x)|_{x<0}=1,\tilde{\chi}(x)|_{x>2}=0$. Then, for any $u_{\infty}<u_0$, $\chi_{u_\infty}:=\tilde\chi(4\abs{u}/\abs{u_\infty}-1)\in\A{\phg}^{\vec{0}}(\D)$ localises to $u>u_\infty/4$, and we let $\phi^{u_\infty}$ be the maximal local characteristic solution to 
		\begin{nalign}\label{eq:scat:finite_scattering}
			&\Box\phi^{u_\infty}=\chi_{u_\infty}f+\mathcal{P}[\phi^{u_\infty}],\\
			&\psi^{u_\infty}|_{\incone}=\psi^{\incone}\chi_{u_\infty},\qquad \psi^{u_\infty}|_{\outconeFar}=0.
		\end{nalign}
        We claim that $\phi^{u_\infty}$ is defined on $\D^{u_\infty,\infty}_{u_0,v_0}$.
        
        We first iterate in the $u$-direction:
        By \cref{thm:scat:local}, we know that a solution exists in $\D^{u_\infty,v_0+\epsilon}_{u_\infty+\epsilon,v_0}$ for $\epsilon$ sufficiently small.
        If $u_\infty+\epsilon<u_0$, we use \cref{lemma:en:metric_perturbed_energy_estimate}  to get 
        \begin{equation}\label{eq:scat:local_bootstrap}
			\sup_{u\in (u_\infty,u_0)}\norm{(v\pv,\rho_{\scri}^{1/2}\sl)\Vb^k \psi^{u_{\infty}}}_{\Hb^{\vec{a};0}(\outcone{}^{v_0,v_0+\epsilon})}+\norm{\Ve \psi^{u_\infty}}_{\Hb^{\vec{a};k}(\mathcal{D}^{u_\infty,v_0+\epsilon}_{u_\infty+\epsilon,v_0})}\lesssim 1.
		\end{equation}
        Eq.~\cref{eq:scat:local_bootstrap} is sufficient to establish existence up to $u_\infty+2\epsilon$ (provided $u_{\infty}+2\epsilon<u_0$).
        Indeed, we can iterate the above to conclude that $\phi^{u_\infty}$ exists in $\D^{u_\infty,v_0+\epsilon}_{u_0,v_0}$.
        We proceed similarly for the induction in $v$.
        Assume we already know that $\phi^{u_\infty}$ exists in $\D_1\cup\D_2$ for $\D_1=\D^{u_\infty,v_1}_{u_0,v_0}$ and $\D_2=\D^{u_1,v_1+\epsilon}_{u_0,v_1}$ for $v_1>v_0$ and $u_1\in(u_\infty,u_0)$.
        Using energy estimate from \cref{lemma:en:metric_perturbed_energy_estimate} in both regions, we get
        \begin{equation}
            \sup_{u\in (u_\infty,u_1)}\norm{(v\pv,\rho_{\scri}^{1/2}\sl)\Vb^k r\phi^{u_{\infty}}}_{\Hb^{\vec{a};0}(\outcone{}^{v_0,v_1+\epsilon})},\sup_{v\in(v_0,v_1)}\norm{(u\pu,\rho_{\scri}^{1/2}\sl) \Vb^k r\phi}_{\Hb^{\vec{a};0}(\Cbar_v)}\lesssim 1.
        \end{equation}
        Using \cref{thm:scat:local}, we can extend $u_1\mapsto u_1+\epsilon$ as long as $u_1+\epsilon\leq u_0$.
        Similarly, we can also extend the solution to $\D^{u_\infty,v_1+\epsilon}_{u_0,v_0}$ and start anew.

        A posteriori, we can apply \cref{prop:en:pastestimate} in the region $\D^{u_\infty,v_1}_{u_1,v_0}$ for $u_1<u_0$ with $k=0$ and \cref{eq:scat:local_bootstrap} to get that for some $\delta>0$ and for any fixed $v_1<\infty$, we have
        \begin{equation}\label{eq:scat:L2_estimate}
            \norm{\pv \psi^{u_\infty}}_{L^2(\outcone{1}^{v_1})}\lesssim_{v_1} \norm{\psi^{\incone}}_{\Hb^{a_-;1}(\incone^{u_\infty,u_1})}+\norm{\mathcal{P}[\phi^{u_\infty}]+\chi_{u_\infty}f}_{\Hb^{a_-+1+2\delta;0}(\D^{u_\infty,v_1}_{u_1,v_0})}\lesssim\norm{\psi^{\incone}}_{\Hb^{a_-;1}(\incone^{u_\infty,u_1})}+ \abs{u_1}^{-\delta}
        \end{equation}
        where we did not indicate the weights toward $I^0$ and $\scrip$, as the estimate may depend on $v_1$ anyway.
        Furthermore, we used $L^2(\C^{v_1}_{u_1})$, without any extra weight (e.g. $L^2_\b$) for the same reason. We will later on use this estimate to check that \cref{eq:scat:scat_with_inc_condition} holds.
  
		\textit{Step 2:  Taking the limit.}
        Since $f$ and $\psi^{\incone}$ smooth, and, by \cref{lemma:en:metric_perturbed_energy_estimate}, regularity is propagated, we know that \cref{eq:scat:local_bootstrap} holds for $k\mapsto k+1$ and the right hand side replaced by $<\infty$.
		We now prove the convergence of $\phi^{u_\infty}$ in $\Hb^{\vec{a};k}(\D^{-\infty,\infty}_{u_0,v_0})$ for any $k\geq 4$.
        Let us set $\psi^{1,2}=\psi^{u_1}-\psi^{u_2}$ for $u_2<u_1<u_0$ arbitrary, where the $\psi^{u_i}$ are defined as in \cref{eq:scat:finite_scattering} for $u\geq u_i$ and where we extend $\phi^{u_i}|_{u<u_i}=0$ for $i\in\{1,2\}$.  Note that $\psi^{1,2}$ is the solution to the characteristic initial value problem
		\begin{nalign}\label{eq:scat:scatproofdifference1}
			&(\Box-V-P_g[\phi_1])\phi^{1,2}=f^{1,2}+\mathcal{N}[\phi^{u_1}]-\mathcal{N}[\phi^{u_2}]+(P_g[\phi_1]-P_g[\phi_2])\phi_2,\qquad f^{1,2}:=\big(\chi_{u_1}-\chi_{u_2}\big)f\\
			&\psi^{1,2}|_{\incone}=\psi^{\incone}\big(\chi_{u_1}-\chi_{u_2}\big),\qquad r\phi|_{\outcone{2}}=0.
		\end{nalign}
		Using multi-linearity of $\mathcal{N},P_g$, there exists first order differential operators $(V_{\mathcal{N}},V_{P_g})$ with coefficients depending on up to two derivatives of $\phi^{u_1},\phi^{u_2}$ such that we can write 
        \begin{nalign}\label{eq:scat:linear_splitting}
            \mathcal{N}[\phi^{u_1}]-\mathcal{N}[\phi^{u_2}]=V_{\mathcal{N}}\phi^{1,2}\\
            (P_g[\phi_1]-P_g[\phi_2])\phi_2)=V_{P_g}\phi^{1,2}.
        \end{nalign}
        Using the uniform estimates \cref{eq:scat:local_bootstrap} with $k+1$ for $\phi^{u_1},\phi^{u_2}$ and \cref{rem:en:regularity_of_coefficients}, we will show that $V_{\mathcal{N}},V_{P_g}$ are short range, i.e. for some strongly admissible $\vec{a}^{\mathcal{P}}$ and $\delta>0$
		\begin{equation}\label{eq:scat:estimate3}
			\norm{r(\mathcal{N}[\phi^{u_1}]-\mathcal{N}[\phi^{u_2}],r(P_g[\phi_1]-P_g[\phi_2])\phi_2)}_{\Hb^{\vec{a}^{\mathcal{P}(\D)}+(\delta,\delta,\delta),k}}\lesssim\norm{\Ve r\phi^{1,2}}_{\Hb^{\vec{a};k}(\D)}.
		\end{equation}
        We prove this for $\mathcal{N}$, for $P_g$ one proceeds similarly.
        Let's write $\mathcal{N}$ as a linear combination of terms 
        \begin{equation}
            \frac{1}{r}f_{\mathcal{N}}(X_1r\phi)(X_2r\phi)\cdots(X_n\phi)
        \end{equation}
        for $X_i\in\Diff^1_{\b}(\D)$ and analyse each individually.
        Then, we can express
        \begin{nalign}
            \mathcal{N}[\phi^{u_1}]-\mathcal{N}[\phi^{u_2}]=\frac{1}{r}f_{\mathcal{N}}\Big((X_1r\phi^{1,2})(X_2r\phi^{u_1})\cdots(X_nr\phi^{u_1})+(X_1r\phi^{u_2})(X_2r\phi^{1,2})(X_3r\phi^{u_1})\cdots(X_nr\phi^{u_1})\\
            +\dots+(X_1r\phi^{u_2})(X_2r\phi^{u_2})\cdots(X_{n-1}r\phi^{u_2})(X_nr\phi^{1,2})\Big)
        \end{nalign}
        Using the apriori control from \cref{eq:scat:local_bootstrap} and the short range nature of $\mathcal{N}$ yields \cref{eq:scat:estimate3}.
  
		We also note that $\norm{rf^{1,2}}_{\Hb^{\vec{a}^f;k}},\norm{\Vb^k r\phi^{1,2}|_{\incone}}_{\Hb^{a_-;1}}\to0$ as $u_1,u_2\to-\infty$.
		Therefore, using \cref{lemma:en:metric_perturbed_energy_estimate} again, we get that $\phi^{1,2}$ is a Cauchy sequence in $\Hb^{a;k}$, thus $\phi^{u_\infty}$ converges to a limit $\phi$. Vanishing of the radiation field for $\phi$ follows from \cref{eq:scat:L2_estimate}.

		\textit{Step 3: Local uniqueness.} 
        We restrict attention to $\D^{-\infty,v_1}_{u_0,v_0}$ for this step, as uniqueness for any $v_1<\infty$ implies uniqueness in the full region.
        We also drop weights towards $I^0,\scrip$ and allow the inequalities to depend on $v_1$.
        Assume that $\phi^1,\phi^2$ both are solutions to the scattering problem. 
        Then $\phi^{1,2}=\phi^1-\phi^2$ solves
		\begin{nalign}
			(\Box-V-P_g[\phi^1])\phi^{1,2}=\mathcal{N}[\phi^1]-\mathcal{N}[\phi^2]+(P_g[\phi^1]-P_g[\phi^2])\phi^2,\quad r\phi|_{\incone}=0, \quad \norm{\partial_vr\phi^{1,2}}_{L^2(\outconeFar)}\to0.
		\end{nalign}
		Using the control over $\phi^1,\phi^2$ provided by $\Ve \psi^i\in\Hb^{a;4}(\D\cap[v_0,v_\infty))$, we get that $\phi^{1,2}$ satisfies a linear equation with some short range $\bar{V},\bar{P}_g$. In particular
		\begin{equation}
			\norm{r\big(\mathcal{N}[\phi^{1}]-\mathcal{N}[\phi^2],(P_g[\phi^1]-P_g[\phi^2])\phi_2\big)}_{\Hb^{a+1+\delta;0}(\D^{u_\infty,v_1}_{u_1,v_0})}\lesssim\norm{\Ve r\phi^{1,2}}_{\Hb^{a;0}(\D^{u_\infty,v_1}_{u_1,v_0})}.
		\end{equation}
         Using \cref{lemma:en:past_estimate_L2_perturbed} and \cref{rem:en:regularity_of_coefficients}, there exists $-u_1$ sufficiently large, that we already have
		\begin{equation}\label{eq:scat:proof_incoming}
			\norm{\Ve r\phi^{1,2}}_{\Hb^{a;0}(\D^{u_\infty,v_1}_{u_1,v_0})}\lesssim\norm{\pv r\phi^{1,2}|_{\outconeFar}}_{L^2(\outconeFar)}+\frac{1}{\abs{u_\infty}^{1/2}}\norm{\sl r\phi^{1,2}|_{\outconeFar}}_{L^2(\outconeFar)}.
		\end{equation}
		We conclude that $r\phi^{1,2}=0$ in $\D^{-\infty,v_1}_{u_1,v_0}$. Extending the uniqueness to $\D^{-\infty,v_1}_{u_0,v_0}$ follows from local well-posedness.
        Note, that for a linear problem, we need not use $\Ve\psi^i\in\Hb^{a;4}(\D\cap[v_0,v_\infty))$ to obtain short range perturbed wav equation for $\phi^{1,2}$, therefore the assumption is not necessary.

		\textit{c)}
		We construct a solution $\phi'$, by including the higher regularity norms in the space in which we take the limit and use that the region in which \cref{eq:en:perturbed_energy_estimate} holds does not depend on $k$.
		As we have uniqueness, we conclude that the $\phi=\phi'$. 

        \textit{d)}
		To get \cref{eq:scat:higher_no_incoming_estimate}, we instead use \cref{eq:en:perturbed_energy_estimate_no_incoming}.
	\end{proof}

    \begin{rem}[Scattering theory for no incoming radiation compatible problems that are not short-range]
        As we already discussed in \cref{rem:en:no-incoming_vs_short-range}, compatibility with no incoming radiation for $\mathcal{N}$ does not imply short-range.
        \cref{thm:scat:scat_general} only yields a scattering theory for short-range pertubations and yields $\Hbt$ regularity when $\mathcal{P}$ is in the intersection with no incoming radiation and short-range.
        It is not difficult to see that the estimates from \cref{lemma:en:metric_perturbed_energy_estimate,lemma:en:recovering_initial_data} allow one to mimic the proof of \cref{thm:scat:scat_general} to obtain scattering result for only no incoming radiation $\mathcal{P}$.
        We pursue no such improvements here.
    \end{rem}

    As already mentioned in \cref{rem:en:shortrange_admissible_weights}, we can enlarge the perturbations for which scattering solutions exist in $\D$:

    \begin{defi}\label{def:scat:extended_short_range}
        Use the notation of \cref{def:en:short_range}.
        Fix admissible weight $\vec{a}$.
        We say that $\vec{a}^f$ is extended admissible if
        $\vec{a}^f\geq\vec{a}+(1,2,1)$ with $a_-^f>1$.
        Similarly, we say the perturbations $\bullet\in\{P,\mathcal{N},V\}$ are extended short range, if they satisfy
        \begin{equation}\label{eq:scat:minimal_scattering_assumption}
            (1,2,1)+\vec{a}<\vec{a}^{\bullet,p},\quad a^{\bullet,p}_->1,\quad\text{where }
            \vec{a}^{\bullet,p}=\vec{a}^\bullet+n^\bullet\vec{a}.
        \end{equation}
        The \emph{gap} for these is $\delta=\min(\vec{a}^{\bullet,p}-(1,2,1)-\vec{a})$.
    \end{defi}

    \begin{cor}\label{cor:scat:enlarged_admissible_set}
        Let $\vec{a}$ be admissible, $\vec{a}^f$ extended admissible and $\mathcal{P}$ be extended short range perturbations with weights~$\vec{a}^{\bullet,p}$ as in \cref{eq:scat:minimal_scattering_assumption}.
        There exists $k$ sufficiently large, depending on $\vec{a},\vec{a}^f,\vec{a}^{\bullet,p},\delta$, such that the following holds:
        Let
        $\psi^{\incone}\in\Hb^{a_-;2k+2}(\incone), rf\in\Hb^{\vec{a}^f;k}(\D),T^j(rf)\in\Hb^{a_{-}^f;2(k-j)}(\incone)$ for $j\leq k$ such that $\mathcal{X}^{k}\leq 1$ as defined in \cref{eq:en:perturbed_energy_estimate}. Then there exists $-u_0$ sufficiently large, depending only on $\mathcal{P}$, such that there exists a unique scattering solution $\Ve r\phi\in\Hb^{\vec{a};4}$ to \cref{eq:scat:scattering_def}, where uniqueness holds in the same class as in \cref{thm:scat:scat_general}.
    \end{cor}

    \begin{rem}[Example]\label{rem:scat:enlarged}
        As a simple example of a perturbation not covered by \cref{thm:scat:scat_incoming} yet covered by \cref{cor:scat:enlarged_admissible_set} (restricting to $\D^-$), consider any semilinear perturbation with $rf=0$ and $a_-=1/4,\,n^{\mathcal{N}}=2,\,(a^{\mathcal{N}}_-,a^{\mathcal{N}}_0)=(3/4+\epsilon,2-\epsilon)$ with $\epsilon=0.01$.
        Clearly, this is extended short-range. However, it is not short-range for any $a_0$, since this would simultaneously require $a_0+2\leq 2a_0+(2-\epsilon)$ and $a_0\leq 1/4+2(3/4+\epsilon)$ (cf.~\cref{eq:en:admissible_inhom1}).
        
    \end{rem}
    \begin{rem}[No loss of derivatives for semi-linear perturbations]
        For short-range perturbations that are semilinear, then $k$ in \cref{cor:scat:enlarged_admissible_set} is exactly as in \cref{thm:scat:scat_general}, since we can remove the requirement \cref{eq:en:admissible_inhom1} as in \cref{eq:savedhehe}, cf.~\cref{rem:en:shortrange_admissible_weights}.
    \end{rem}
    \begin{proof}
        \emph{Existence:}
        We begin with constructing an ansatz such that we can apply \cref{thm:scat:scat_general} to the remainder term.
        Let $1>\delta>0$ be such that $\delta+(1,2,1)+\vec{a}<\vec{a}^{\bullet,p}$.
        Let's define $\psi^{(0)}\in \Hb^{\vec{a};k}(\D)$ via $-\pu\pv\psi^{(0)}=rf$ with data that of $\psi$ using \cref{corr:ODE:du_dv}.
        
        For $j\geq0$, we then inductively define $\psi^{(j+1)}$  with trivial data as solutions to
        \begin{equation}\label{eq:scat:proof1_enlarged}
            -\pu\pv\psi^{(j+1)}=-r^{-2}\Dl\psi^{(j)}+r\mathcal{P}[\phi^{(\leq j)}]-r\mathcal{P}[\phi^{(\leq j-1)}],\quad\text{where }\phi^{(\leq j)}:=\sum_{i=0}^j\phi^{(i)}\quad \text{and }\mathcal{P}[\phi^{(\leq -1)}]:=0.
        \end{equation}
        We inductively claim that $\psi^{(j+1)}\in\Hb^{a_-+\delta j,a_0,a_+;k-2j}(\D)$ and that $r\mathcal{P}[\phi^{(\leq j)}]-r\mathcal{P}[\phi^{(\leq j-1)}]\in\Hb^{a_-+\delta j+1,a_0+2,a_++1;k-2j}(\D)$.
        For $j=0$, this holds by an application of \cref{corr:ODE:du_dv}, since $r\mathcal{P}[\phi^{0}]\in \Hb^{\vec{a}+(1,2,1)+\vec{\delta}}(\D)$. 
        For higher $j$, we proceed inductively, in each step splitting the RHS according to $\phi^{(j)}=\phi^{(\leq j)}-\phi^{(\leq j-1)}$, and then decomposing e.g.~$r\mathcal{N}[\phi^{(\leq j)}]-r\mathcal{N}[\phi^{(\leq j-1)}]$ linearly in $\phi^{(k)}$ as in \cref{eq:scat:linear_splitting}. 
        
        Finally, we fix $N$ such that $a_-+N\delta+1+\delta>3/2$ (thus avoiding the condition \cref{eq:en:admissible_inhom}), write $\psi=\psi^{\Delta}-\psi^{(\leq N+1)}$, and note that $\psi^{\Delta}$ satisfies:
       \begin{nalign}\label{eq:scat:proof2_enlarged}
            P_\eta\psi^{\Delta}&=-\frac{\Dl \psi^{(N+1)}}{r^2}+
           r\mathcal{P}[\phi]-r\mathcal{P}[\phi^{(\leq N)}]\\&=-\frac{\Dl \psi^{(N+1)}}{r^2}+
           (r\mathcal{P}[\phi^{\Delta}+\phi^{(\leq N+1)}]-r\mathcal{P}[\phi^{(\leq N+1)}])+(r\mathcal{P}[\phi^{(\leq N+1)}]-r\mathcal{P}[\phi^{(\leq N)}]).
        \end{nalign}  
       The last term on the RHS is already controlled, and using \cref{lemma:en:short_range_computations} for the middle term, we may apply \cref{thm:scat:scat_general} to conclude that there exists a unique scattering solution $\psi^{\Delta}\in\Hb^{(1/2+,a_0,a_+);k-2N-2}(\D)$. 
        
        \emph{Uniqueness:}  For uniqueness, we proceed exactly as before.
        
    \end{proof}
	
	\subsection{Scattering theory with nontrivial incoming radiation}\label{sec:scat:incoming}
	We now prove the analogue of \cref{thm:scat:scat_incoming} in the case of incoming radiation:
	\begin{thm}\label{thm:scat:scat_incoming} Let $k\geq 4$, fix admissible weights $\vec{a}, \vec{a}^f$ with $a_-<0$ and prescribe scattering data and inhomogeneity $\psi^{\incone},f$ as in \cref{thm:scat:scat_general}.
 Furthermore, specify $\pv\psi^{\scrim}$ such that $\psi^{\scrim}\in \Hb^{a_0;k+2}(\scrim)$ satisfies $\norm{\psi^{\scrim}}_{\Hb^{a_0;4+2}(\scrim)}\leq 1$.
		Let $\mathcal{N},V,P_g$ be short-range perturbations. Then we have:
    \begin{enumerate}[label=\alph*)]
	\item 
		There exists $-u_0$ sufficiently large, depending only on $\mathcal{P}$, such that there exists a unique scattering solution $\Ve \psi\in \Hb^{a_-,a_0,a_+;k}(\D_{u_0,v_0})$ to \cref{eq:scat:scattering_def} with \cref{eq:scat:scat_with_inc_condition}. 
        Here, uniqueness is as in \cref{thm:scat:scat_general}.
        
		\item  Moreover, for any $k\geq 4$, if \cref{eq:scat:k-data} holds and $\psi^{\scrim}\in \rho_0^{a_-}\Hb^{k+2}(\scrim)$, then $\Ve\psi\in \Hb^{\vec{a};k}(\D)$. 
 
        \item As in \cref{cor:scat:enlarged_admissible_set}, we can weaken the requirement on $\vec{a}^f$ to \emph{extended admissibility} and the requirements on $V,\mathcal{N},P_g$ to merely be \emph{extended short-range perturbations}, provided the initial data are sufficiently regular.
  \end{enumerate}
	\end{thm}
    \begin{rem}\label{rem:scat:derivative_loss_incoming}
       While \cref{thm:scat:scat_incoming} requires $k+1$ extra derivatives on $\incone$ (namely $2k+2$, though recall \cref{rem:en:slab_data_recovery}), it only requires $1$ extra derivative along $\scrim$.
        We further expect that this latter loss of one derivative is merely a consequence of our 2-step proof below, and that it can be removed by using an upgraded version of \cref{prop:scat:energy_incoming_rad}.
    \end{rem}
	\begin{proof}
        \textit{a)} \emph{Existence.}
        The idea is to subtract the incoming radiation from the solution so that we can get a new solution with no incoming radiation, where the actual incoming radiation is prescribed in the inhomogeneity:
        We write $\psi=\psi^{(0)}+\psi^{(1)}$, with $\psi^{(0)}(u,v,\omega)=\psi^{\scrim}(v,\omega)\chi$ for a cutoff function localising to $\D^-$.
        Clearly $\norm{\psi^{(0)}}_{\Hb^{\vec{a};k+2}}\lesssim1$ for $k=4$.
        Then, we use \cref{lemma:en:short_range_computations} to get that $\psi^{(1)}$ satisfies a wave equation with short range perturbations depending on $\psi^{(0)}$.
        We now apply \cref{thm:scat:scat_general} to get a scattering solution for $\psi^{(1)}$.
        
        \emph{Uniqueness:}
        This follows  as in the proof of \cref{thm:scat:scat_general}, since the difference $\phi^{1,2}$ of two scattering solutions $\phi^1,\phi^2$ has trivial data on $\scrim$.

        The statements \textit{b)} and  \textit{c)} then follow exactly as in the case of no incoming radiation.
	\end{proof}

    \subsection{Scattering theory with weakly decaying or growing data}\label{sec:scat:weak_decay}
    We now discuss two cases where scattering solutions still exist even if the initial data decay weakly or grow.
        \subsubsection{Scattering with weak decay and "no incoming radiation"}
		If we  impose trivial data at $\scrim$, then we can still make sense of scattering solutions even if the data along $\incone$ diverge, although at a loss of regularity. Of course, in this setting, the previous version of no incoming radiation is  no longer useful as $\pv\psi$ will generally not even attain a limit at $\scrim$ anymore. We therefore introduce the following 
		\begin{defi}\label{def:scat:weak_decay}
        
            Fix $m\in\N$, $\vec{a}$ with $a_+<a_0<a_-$, $-m+1/2<a_-$ and $\vec{a}^f\geq\vec{a}+(1,2,1)$.
            Let $\{V,\mathcal{N},P_g\}$ be perturbations with weights as in \cref{lemma:en:no_incoming_perturbed_en_estimate}.
            Given smooth data $\psi^{\incone}\in \Hb^{a_-}(\incone)$, inhomogeneity $rf\in\Hbt^{\vec{a}^f}(\D)$, and also prescribing for all $j\leq m-1$ and a fixed $u_1<u_0$ the functions $\psi^{\incone,u_1,j}\in H^{\infty}(S^2)$, we call $\phi$ the scattering solution to \cref{eq:scat:scattering_def} with \textbf{no incoming radiation (at order $m$)} corresponding to these data if $\phi$ solves \cref{eq:scat:scattering_def} and if for all $1\leq j\leq m-1$:
            	\begin{equation}
				T^j\psi|_{\incone\cap\outcone{1}}=\psi^{\incone,u_1,j},\quad \psi|_{\incone}=\psi^{\incone},\quad\text{and if } \forall v_\infty>v_0\text{ we have }\lim_{u\to-\infty} \norm{\pv^m(r\phi)}_{L^2(C^{v_{\infty}}_{u})}\to 0.
			\end{equation}
		\end{defi}

        \begin{rem}\label{rem:scat:slab_scattering_with_weak_decay}
            For \cref{def:scat:weak_decay}, we need to specify transversal derivatives of $\phi$ at some fixed finite sphere to be able to construct unique scattering solutions admitting a given data. Note that this is exactly what happens in the case of the negative spin Teukolsky equations, see \cite{kehrberger_case_2024} and \cref{sec:Sch} for further details.
            Note that if we specified $\phi$ in an initial slab  $v\in(v_0,v_0+\delta)$ instead of along $\incone$, then we wouldn't need to specify these extra derivatives.
        \end{rem}
  
		\begin{thm}\label{thm:scat:weak_decay}
			Fix  $k,m,\vec{a},\vec{a}^f,f,m,\psi^{\incone},\mathcal{P},u_0$ as in \cref{lemma:en:no_incoming_perturbed_en_estimate}. Let $u_1<u_0$ arbitrary. Given $\psi^{\incone}\in\Hb^{a_-;2(k+m)+2}(\incone)$, $rf\in\Hbt^{\vec{a}^f;k+m}(\D)$, $(rT)^j(rf)\in\Hb^{a_-+2;2(k+m-j)}(\incone)$ for $j\leq k+m$ and, finally, $\rho_-^{-2(k+m+1-j)}\psi^{\incone,u_1,j}\in H^{2(k+m-j)}(S^2)$, there exists a unique scattering solution $\psi\in\Hbt^{\vec{a};k}(\D)$ with no  incoming radiation at order $m$ to \cref{eq:scat:scattering_def}.
            
            Here, uniqueness is understood with respect to functions with bounded $\norm{\Ve T^m r\phi}_{\Hb^{a_-;4}(\D\cap[v,v_\infty])}$ for any $v_\infty<\infty$.
		\end{thm}
		\begin{proof}
       \emph{Transversal derivatives:} First, we note that the specification of the transversal derivatives at some fixed sphere allows us to appeal to the estimates of \cref{lemma:en:recovering_initial_data}, case \textit{c)}, and, in particular, uniquely determine the transversal derivatives of $\psi$ along all of $\incone$. 

       \emph{Existence:} We then prove the existence of $\psi$ in the same way as we proved  \cref{thm:scat:scat_general}, but using the estimates of \cref{lemma:en:no_incoming_perturbed_en_estimate} instead of \cref{lemma:en:metric_perturbed_energy_estimate}.
			
      \emph{Uniqueness:}
      Suppose we have two solutions $\phi^1$ and $\phi^2$. 
      Starting with the a priori bounds on $T^m\psi^i$, and the uniquely determined data for $T^{m-j}\psi$ along all of $\incone$, we get corresponding bounds for $T^{m-j}\psi^i$ in all of $\D$ via integration using \cref{corr:ODE:du_dv}\cref{item:ode:t-prop}.
      We can then consider the equation satisfied by $\bar\phi:=T^m(\phi^1-\phi^2)$. 
      Using the just established bounds on  $\phi^{1}$ and $\phi^{2}$ (cf.~the bounds in the proof of \cref{lemma:en:no_incoming_perturbed_en_estimate}), we deduce $\bar{\phi}=0$ just as in \textit{Step 3)} of \cref{thm:scat:scat_general}.
   It follows that $\phi^1=\phi^2$.
		\end{proof}

        \subsubsection{Scattering with weakly decaying polyhomogeneous data}
        We can also extend our scattering theory to setups \textit{with} incoming radiation; however, in this case we require the data, the inhomogeneity and the nonlinearities to all have an expansion. We restrict the discussion to $\D^-$:

        \begin{defi}\label{def:scat:weak_polyhom}
            Fix arbitrary index set $\vec{\E}^f$, fix $k\in\N_{\geq1}$, and fix admissible weights $\vec{a}$, $\vec{a}^f$ with $a_->0$.
            Fix inhomogeneity $rf\in\A{phg,b}^{\E^f_-,a^f_0;k}(\D^-)+\Hb^{\vec{a}^f;k}(\D^-)$, as well as data $\psi^{\incone}\in\A{phg}^{\E^f_--1;k}(\incone)+\Hb^{a_-;k}(\incone)$ and $\psi^{\scrim}\in\Hb^{a_0;k}(\scrim)$.
           Finally, fix coordinates $\rho_-=v/|u|, v, \omega$. Then we say that $\phi$ is the scattering solution to \cref{eq:scat:scattering_def} with weakly decaying polyhomogeneous data if there exists some index set $\E_-$ such that
            \begin{equation}\label{eq:scat:polyhom_cond1}
                \pv \psi\in\A{phg,b}^{\E_-,a_0;k}(\D^-)+\Hb^{a_-,a_0;k}(\D^-)
            \end{equation}
            and such that there exist functions $\psi_{z,k}(v,\omega)$ such that
            \begin{equation}\label{eq:scat:polyhom_cond2}
                \pv \psi=\sum_{(z,k)\in(\E_-\setminus (0,0))_{\leq 0}} \rho_-^{z}\log^k\rho_- \pv\psi_{z,k}(v,\omega)+\pv\psi^{\scrim}(v,\omega)+\Hb^{0+,a_0;k}(\D^-).
            \end{equation}
        \end{defi}

        \begin{rem}[Coordinate dependence]
            Similarly as in \cref{rem:ode:uniqueness}, we note that the above definition, in particular~\cref{eq:scat:polyhom_cond2}, depends on the coordinate $\rho_-$ used.
            Using an alternative coordinate, e.g.~$\rho_-=v/r$ may yield a different solution.
            Just as in \cref{rem:ode:uniqueness}, this coordinate dependence is absent if we require $\min(\E_-)>-1$.
        \end{rem}

\begin{thm}\label{thm:scat:weak_polyhom}
    Let $\E_-^f$ be an arbitrary index set.
    Let $\vec{a}'$, $\vec{a}^f$ be admissible with $a'_->0$. 
    Let $rf\in\A{phg,b}^{\E^f_-,a^f_0;N}(\D^-)+\Hb^{\vec{a}^f;k}(\D^-)$, $\psi^{\incone}\in\A{phg}^{\E^f_--1;N}(\incone)+\Hb^{a_-';k}(\incone)$, and let $\psi^{\scrim}\in \Hb^{a_0';k}(\scrim)$, with $N$ and $k$ specified below.

    Define now $a_-=\min(\min(\E^f_-)-1-,a_-'), a_0=\min(a_0',a_--)$. With notation as in \cref{def:en:short_range} and for some $\delta>0$, let $\mathcal{P}$ be perturbations satisfying 
     \begin{equation}\label{eq:scat:weak_decay_gap_condition}
            n^\bullet \vec{a}+\vec{a}^\bullet>\vec{a}+(1,2)+(\delta,\delta), \quad \bullet\in\{V,\mathcal{N},P_g\}
        \end{equation}
        and with coefficients $f_\bullet$ all polyhomogeneous on $\D^-$.
        If we require that $k\geq 10$, that $\min{\E_-^f}>\min(3/2,a'_-+1)-n\delta$ for some $n\in\mathbb N$, and that $N>k+2n+2$, then there exists $-u_0$ sufficiently large and an index set $\E_-$ such that there exist a unique scattering solution $\phi$ to~\cref{eq:scat:scattering_def} satisfying \cref{eq:scat:polyhom_cond1,eq:scat:polyhom_cond2}.

\end{thm}
\begin{proof}
    \emph{Existence:} We recycle the iteration used in the proof of \cref{cor:scat:enlarged_admissible_set}. We define $\psi^{\incone}_{\phg}$ to be the polyhomogeneous part of the data such that $\psi^{\incone}-\psi^{\incone}_{\phg}=:\psi^{\incone}_{\Delta}\in \Hb^{a'_-;k}(\incone)$, similarly for $f=f_{\phg}+f_{\Delta}$. We then define $\psi_0$ to solve
     \begin{equation}
            -\pu\pv(\psi^{(0)})=rf_{\phg},\quad \psi^{(0)}|_{\incone}=\psi^{\incone}_{\phg}.
    \end{equation}
    We solve this by first applying \cref{item:ode:u-prop_boundary} and then \cref{item:ode:u-prop} (with $u$ and $v$ interchanged) of \cref{corr:ODE:du_dv}. 
    We specify the extra requirements on the order one part of $\pv\psi^{(0)}$ to vanish, which is $G^{\scrim}=0$ in \cref{item:ode:u-prop_boundary}.\footnote{This choice is arbitrary. Note also that this step is only possible since we specified $f_{\phg}$ to have an expansion.}
    This gives $\psi^{(0)}\in \A{phg,b}^{(\E_-^f-1)\cupdex\mindex0,a_0;N}(\D^-)$. (This requires a very minor adaptation of \cref{corr:ODE:du_dv} to include mixed spaces $\A{phg,b}$.) 

        Inductively, we now define for $j\geq0$ $\psi^{(j+1)}$ as in \cref{eq:scat:proof1_enlarged}, with trivial data corresponding to the vanishing of the order one part of $\psi^{(j+1)}$,  and prove that
        \begin{equation}
            \psi^{(j+1)}\in \A{phg,b}^{\E^{f_k}_--1,a_0;N-2j-2}(\D^-), \quad rf^{(j)}:=-r^{-2}\Dl\psi^{(j)}+r\mathcal{P}[\phi^{(\leq j)}]-r\mathcal{P}[\phi^{(\leq j-1)}]\in \A{phg,b}^{\E^{f_k}, a_0+2; N-2j-2}(\D^-),
        \end{equation}
        for some index sets $\E_-^{f_k}$ satisfying $\min(\E_-^{f_k})>3/2-(n-j)\delta$.
        Indeed, by the same arguments as in the proof of \cref{cor:scat:enlarged_admissible_set}, the claim for $rf^{(j)}$ follows from the inductively assumed claim for $\psi^{(j)}$, the polyhomogeneity of the $f_{\bullet}$, and, importantly, the gap condition \cref{eq:scat:weak_decay_gap_condition}.
        The claim for $\psi^{(j+1)}$, on the other hand, then follows by an application of \cref{corr:ODE:du_dv} as for $\psi^{(0)}$.

        Next, we again consider the difference $\psi=\psi^{\Delta}-\psi^{(\leq n+1)}$, and note that $\psi^{\Delta}$ satisfies 
        \begin{nalign}
            P_\eta\psi^{\Delta}=rf_{\Delta}-\frac{\Dl \psi^{(n+1)}}{r^2}+
           (r\mathcal{P}[\phi^{\Delta}+\phi^{(\leq n+1)}]-r\mathcal{P}[\phi^{(\leq n+1)}])+(r\mathcal{P}[\phi^{(\leq n+1)}]-r\mathcal{P}[\phi^{(\leq n)}])
        \end{nalign}  
with scattering data $\psi_{\Delta}^{\incone}$ and $\psi^{\scrim}$. We apply the same argument as in \cref{eq:scat:proof2_enlarged} obtain that the nonlinearity on the RHS can be treated as a short-range perturbation with respect to $\vec{a}'$.  The existence of a scattering solution $\psi$ then follows from \cref{thm:scat:scat_incoming}.

\emph{Uniqueness:}
Note that $\psi_{z,k}$ in \cref{eq:scat:polyhom_cond2} are uniquely determined by the data $rf,\psi^{\incone}$ and can be constructed via a series expansion as done above.
Therefore, for two solutions $\psi^{1},\psi^{2}$, the difference satisfies $\psi^{1,2}=\psi^1-\psi^2\in\Hb^{0+,a_0;k}(\D^-)$.
The rest of the proof proceeds as in \cref{thm:scat:scat_general}.
\end{proof}

        \subsection{Scattering theory with long-range potentials}\label{sec:scat:longrange}
         As we explained in \cref{sec:en:longrange}, we can extend all previous results to include long-range modifications of $\Box_\eta$:
     \begin{thm}
        [Scattering with long-range potential]\label{thm:scat:long} 
        \cref{thm:scat:scat_general,thm:scat:scat_incoming,thm:scat:weak_decay,thm:scat:weak_polyhom} all also hold if $\mathcal{P}[\phi]$ in \cref{eq:scat:scattering_def} is replaced by $\mathcal{P}_L[\phi]=\mathcal{P}[\phi]+V_L\phi$ for some \emph{long-range} potential $V_L$ that, depending on the context, also is compatible with the no incoming radiation condition or, in the case of \cref{thm:scat:weak_polyhom}, polyhomogeneous.
     \end{thm}
     \begin{proof}
     This is a consequence of the fact that all energy estimates of \cref{sec:en:perturbations} still hold when including a long-range potential by the results of \cref{sec:en:longrange}. Thus, we can replicate all the previous proofs of  \cref{thm:scat:scat_general,thm:scat:scat_incoming,thm:scat:weak_decay,thm:scat:weak_polyhom}.
     \end{proof}
     We conclude this section with an observation regarding the decay of solutions to (twisted) long-range potential problems.
        \begin{obs}[Mixed slow decaying and radiative solutions]\label{rem:en:good_derivatives_noincoming}
        
        First consider the problem $\Box_\eta\phi=0$ with 
        data $r\phi|_{\incone}\in\Hb^{a_-;\infty}(\incone)$ and $v\pv r\phi_{\scrim}\in\Hb^{a_0;\infty}(\scrim)$ \textit{without the restriction $a_-\geq-1/2$}.
        
        Combining now \cref{item:en:pastestimate_inhom_0th,lemma:en:energy_no_incoming,sec:scat:incoming} (or rather, their scattering versions) gives by linearity that the corresponding solution satisfies $r\phi\in\Hb^{\vec{a}_{\mathrm{rad}};\infty}(\D^-)+\Hbt^{\vec{a};\infty}(\D^-)$ for $a_{\mathrm{rad},-}=0-$.
        In particular, the solution can be split into a \emph{radiative} term that is merely smooth with respect to $v\pv$ and has explicit $\phi\sim r^{-1}$ decay, and a \emph{non-radiative} term that is smooth with respect to $r\pv$.

        Next, consider now a twisted long-range potential perturbation such as \cref{eq:long:twisted}. As explained in \cref{rem:en:twistedlongrange}, in order to treat such problems, say, in $\D^-$, we need to twist with $\rho_-^{-c_v}$. 
        By the above, we can then still solve such problems even for weakly decaying data $\rho^{-c_v}\psi^{\incone}$ by splitting the solution into a radiative and a non-radiative part to get $\rho_-^{-c_v}r\phi\in\Hb^{\vec{a}_{\mathrm{rad}};\infty}(\D^-)+\Hbt^{\vec{a};\infty}(\D^-)$.
        In particular, the radiative part always decays like $\phi\sim r^{-(1+c_v)}$.
       See already 
         \cref{rem:Sch:Teukolsky_decay_rate} for the application of this to observation to the Teukolsky equation.
         See also \cref{lemma:EVE:weak_data_at_I0} for a similar observation in $\D^+$.
    \end{obs}
	
	
	\newpage
	
    	\section{Propagation of polyhomogeneity for \texorpdfstring{$\Box_\eta \phi= f$}{Minkowskian wave equation}}\label{sec:prop}

	We now, temporarily, turn our attention back to the Minkowskian wave equation and prove that for polyhomogeneous inhomogeneity and scattering data (+error), scattering solutions to 
	\begin{equation}\label{eq:prop:wave}
		\Box_\eta\phi=f
	\end{equation}
	are also polyhomogeneous ($+$error). The results proved in this section provide, in particular, a complete and sharp picture of the asymptotics of solutions to \cref{eq:prop:wave} in the case of nontrivial incoming radiation.\footnote{We note that we always assume our incoming radiation $\pv\psi^{\scrim}$ to be supported away from $v_0$.}
    
    In the subsequent \cref{sec:app}, we will then use this result to infer a similar result for short-range perturbations of \cref{eq:prop:wave}.

  This section is structured as follows: 
    We first recall the relevant results (along with a small improvement of the index set) from \cite{hintz_stability_2020} concerning the propagation of polyhomogeneity in the future region $\D^+$ in \cref{sec:prop:recall}.
    We then split the discussion into two parts for the  region $\D^-$.
	We first treat the case with no incoming radiation in \cref{sec:propagation:past_corner}. We then treat the slightly more complicated case with nontrivial incoming radiation  in \cref{sec:prop:II}. 
    Finally, in \cref{sec:prop:longrange}, we explain how to extend the results in the case where a long-range potential $V_L\phi$ is included in \cref{eq:prop:wave}.

    \begin{rem}\label{rem:prop:finite_loss}
	For the rest of the paper, we will work exclusively at infinite regularity $\Hb^{;\infty}$ and suppress the ${}^{;k}$-superscript. This only serves the purpose of simplifying the paper's already heavy notation. Nevertheless, recovering each finite term in the expansions in our results only requires a finite amount of regularity, for instance one can replace \cref{eq:prop:futuresemiindex,eq:prop:future_pol} below with
    \begin{align}\label{eq:prop:finite_reg}
        \chi\psi=0, rf\in\Hb^{a_0+2,a_++1;k}(\D^+)\quad&\implies \psi\in \A{b,phg}^{a_0,\mindex{0};k'}(\D^+)+\Hb^{a_0,\min(a_0,a_+)-;k'}(\D^+)\\
        \chi\psi=0, rf\in\A{phg}^{\E^f_0,\E^f_+;k}(\D^+)&\implies \psi\in\A{phg}^{\E_0,\E_+;k-k'}	(\D^+)+\Hb^{a_0,a_+;k-k'},
    \end{align}
    for $k$ sufficiently large depending on $a_0,a_+,\E_0,\E_+$ and $k-k'\sim_{a_0,a_+,\E_0,\E_+}1$.
    For the first one, in the case $a_0,a_+=1$ we can take $k'=4$, while for the second one with $\E^f_0=\overline{(2,1)},\E^f_0=\overline{(1,0)}$ and $a_0=a_+=1$ we may take $k'=10$.
    \end{rem}

	\subsection{Propagation of polyhomogeneity near the future corner \texorpdfstring{$I^0\cap\scrip$}{}}\label{sec:prop:recall}
	We recall (a slight extension of) the  propagation of polyhomogeneity statements in $\D^+$ from \cite{hintz_stability_2020}, and we provide an improvement over it in \cref{lemma:prop:improvement_at_future}.
	
	\begin{lemma}[\cite{hintz_stability_2020}]\label{lemma:prop:future_corner}
		Let $\phi$ be a solution to \cref{eq:prop:wave} and let $\chi$ be a smooth cutoff supported in $\D^+$ away from $\scrip$ and equal to 1 around a compact subset of $I_0\cap\D^-$.
        Let $a_0,a_+\in\mathbb R$ as well as $\E_0$ and $\E_+$ be index sets. 
		Then we have
		\begin{subequations}\label{eq:prop:future}
			\begin{align}
				f=0,\, \chi\psi\in\A{phg}^{\E^f_0-2}(\D^+\cap\D^-)+\Hb^{a_0}(\D^+\cap\D^-)&\implies\psi\in\A{phg}^{\E_0,\E_+}(\D^+)+ \A{b,phg}^{a_0, \mindex0}(\D^+)+\Hb^{a_0,a_0-}(\D^+)\label{eq:prop:future_data},\\
				\chi\psi=0, rf\in\Hb^{a_0+2,a_++1}(\D^+)&\implies \psi\in \A{b,phg}^{a_0,\mindex{0}}(\D^+)+\Hb^{a_0,\min(a_0,a_+)-}(\D^+)\label{eq:prop:futuresemiindex},\\
				\chi\psi=0, rf\in\A{b,phg}^{a_0+2,\E^f_+}(\D^+)&\implies \psi\in\A{b,phg}^{a_0,\E_+}(\D^+)+\Hb^{a_0,a_0-}(\D^+),\label{eq:prop:futureerrorindex}\\
				\chi\psi=0, rf\in\A{phg,b}^{\E^f_0,a_++1}(\D^+)&\implies \psi\in\A{phg}^{\E_0,\E_+}(\D^+)+\A{phg,b}^{\E_0,a_+-}(\D^+),\label{eq:prop:futureindexerror}\\
				\chi\psi=0, rf\in\A{phg}^{\E^f_0,\E^f_+}(\D^+)&\implies \psi\in\A{phg}^{\E_0,\E_+}	(\D^+),\label{eq:prop:future_pol}
			\end{align}
		\end{subequations}
		where $\E_0:=\E^f_0-2$, and $\E_+$ is minimal with the property that
		\begin{equation}\label{eq:prop:suboptimal_future}
			\E_+\supset \big(\E_0\cupdex(\E^f_+-1)\big)\cupdex0,\big(\E_0\cupdex(\E_++1)\big)\cupdex0.
		\end{equation}
		Here, we take $\E^f_0-2,\E^f_+-1=\emptyset$ whenever they do not appear in the assumptions of the respective version of \cref{eq:prop:future}.
	\end{lemma}

 \begin{lemma}\label{lemma:prop:improvement_at_future}
		For all cases \cref{eq:prop:future} in \cref{lemma:prop:future_corner}, the future index set $\E_+$ given in \cref{eq:prop:suboptimal_future} can be improved to the smaller index set
        \begin{equation}\label{eq:prop:optimal_future}
            \E^\phi_+=\big(0\overline{\cup}(\E^f_+-1)\big)\overline{\cup}(\E_0^f-2).
        \end{equation}
	\end{lemma}
	\begin{proof}[Proof of \cref{lemma:prop:future_corner}]
 \cref{eq:prop:future_data} without error term and \cref{eq:prop:future_pol} are proved in a more general setting in section 7 of \cite{hintz_stability_2020}. 
We first prove \cref{eq:prop:futuresemiindex}: From \cref{prop:en:futureestimate}, we already know that $\psi\in\Hb^{a_0,\min(a_0,a_+,0)-}(\D^+)$. 
As already explained in the introduction, we now write $P_\eta=\rho_0^2\rho_+^2\bigg(-\partial_{\rho_+}(\rho_+\partial_{\rho_+}-\rho_0\partial_{\rho_0})+\frac{1}{(1+\rho_+)^2}\Dl\bigg)$ and treat the Laplacian as error term; that is, we write for $n=1$:
\begin{equation}
    \rho_+\partial_+(\rho_+\partial_{\rho_+}-\rho_0\partial_{\rho_0})\psi=-\rho_0^{-2}\rho_+^{-1}(rf)+\frac{\rho_+\Dl\psi}{(1+\rho_+)^2}\in \Hb^{a_0,a_+}(\D^+)+\Hb^{a_0,\min(a_0,a_+,0)+n-}(\D^+).
\end{equation}
Applying first \cref{eq:ODE:2d_error} from \cref{prop:ode:ode2} and then \cref{eq:ODE:1d_error} from \cref{prop:ODE:ODE} then gives $\psi\in \A{b,phg}^{a_0,\mindex0}(\D^+)+\Hb^{a_0,\min(a_0,a_+)-}(\D^+)+\Hb^{a_0,\min(a_0,a_+,0)+n-}(\D^-)$. The result now follows by induction in $n$.

\cref{eq:prop:futureerrorindex}, as well as the error part of \cref{eq:prop:future_data}, follow analogously. Notice that this produces exactly the index set \cref{eq:prop:suboptimal_future}.

Finally, in order to prove \cref{eq:prop:futureindexerror}, we note that in \cite{hintz_stability_2020}, it is already proved that $\psi\in \A{phg,b}^{\E_0,a'}(\D^+)$ for some $a'$. The improvement near the future corner then follows as before.
\end{proof}
 
	\begin{proof}[Proof of \cref{lemma:prop:improvement_at_future}]
		We only consider the case \cref{eq:prop:future_pol}, the other cases following analogously.
		We will make crucial use of a $T$-shifted version of the scaling vector field (see \cref{fig:scaling}):  
		 Define, for any $u_1< u_0$, for any index set $\E$ and for $c\geq \min\E$:
		\begin{equation}\label{eq:prop:Su1}
			S_{u_1}=(u-u_1)\pu+(v-u_1)\pv,\qquad \mathcal{S}^{\mathcal{E}}_{u_1,c}:=\prod_{(z,k)\in\mathcal{E}_{\leq c}}(S_{u_1}+z),
		\end{equation}
		and write $\mathcal{S}_{u_1,c}=\mathcal{S}_{u_1,c}^{\mathcal{E}^{\phi}_0}$ for $\E^\phi_0=\E^f_0-2$.
	We note the following commutation property of $S_{u_1}=S-u_1T$ with the twisted wave operator $P_\eta=r\Box r^{-1}$:
		\begin{equation}\label{eq:prop:wave:commute}
			[S_{u_1},r^2P_\eta]=0\implies [\mathcal{S}_c,r^2P_\eta]=0
			\implies P_\eta\mathcal{S}_{u_1,c}\psi=r^{-2}\mathcal{S}_c r^2rf.
		\end{equation}
  We furthermore have, writing $\mathcal{S}_c\psi=\psi_{c,u_1}$ and using that $(S+z)r \circ =r(S+z+1)\circ$:
  \begin{equation}\label{eq:prop:waveimproved}
   P_\eta\psi_{c,u_1}=
   \Big(\prod_{(z,k)\in(\E^\phi_{0})_{\leq c}+2}(S_{u_1}+z)\Big)(rf)=\mathcal{S}_{c+2}^{\E^f_0}(rf):=rf_c\in \A{phg}^{\E^{f_c}_0,\E^{f_c}_+}(\D^+),
  \end{equation}
  where we used that $\E^\phi_0+2=\E^f_0$, and 
  where $\E^{f_c}_0=(\E^f_{0})_{>c+2}$ and $\E^{f_c}_+=\E^f_+$. 
Indeed, working with boundary defining functions $\rho_{0,u_1}=-(u-u_1)^{-1}$, $\rho_{+,u_1}=-(u-u_1)v^{-1}$, then (suppressing the $u_1$ dependence) $S_{u_1}=-\rho_{0,u_1}\partial_{\rho_{0,u_1}}|_{\rho_{+,u_1}}$, and thus $\mathcal{S}_{u_1,c+2}^{\E^f_0}(rf)$ removes all terms up to order $c+2$ in $rf$ towards $I^0$, while keeping the index set towards $\scrip$ unchanged. Here, we are using the equivalent definition of polyhomogeneity from \cref{rem:notation:equivalent_polyhoms}.

		
		We now apply \cref{lemma:prop:future_corner} to \cref{eq:prop:waveimproved} to obtain that $\psi_{c,u_1}\in\A{b,phg}^{c+,(\E^f_+-1)\cupdex0}(\D^+)+\Hb^{c+,c+}(\D^+)$, where we treated all terms beyond order $c$ (coming from $\E_0^{f_c}-2$) as $\Hb$ errors.\footnote{\label{footnote:prop:improvedfuture}More specifically, we used that, for $\E_+$ minimal such that $	\E_+\supset \big((\E^{f_c}_0-2)\overline{\cup}(\E^f_+-1)\big)\overline{\cup}0,\big((\E^{f_c}_0-2)\overline{\cup}(\E_++1)\big)\overline{\cup}0$, if $(z,k)\in\E^+$ and $z\leq c$, then $(z,k)\in (\E^f_+-1)\cupdex0 $.}
		On the cone $\outcone{1}$, since $S_{u_1}|_{\outcone{1}}=(v-u_1)\pv$,  we can now apply \cref{item:ode:1d} (as in \cref{rem:notation:equivalent_polyhoms}) to get that $\psi|_{\outcone{1}}\in \A{\phg}^{\big(0\overline{\cup}(\E^f_+-1)\big)\overline{\cup}(\E^f_0-2)}(\outcone{1})+\Hb^{c+}(\outcone{1})$, where we used that~$\E^\phi_0=\E^f_0-2$.

We would like to now extend this improved expansion into all of $\D^+$ by integrating from $u_1$ along the integral curves of $S_{u_1}$. Since $S_{u_1}$ degenerates at $u=u_1$, we therefore consider a new $u_2\in(u_1,u_0]$ and, by the same considerations as above, obtain that $\psi_{c,u_2}\in\A{phg}^{c+,(\E^f_+-1)\cupdex{0}}(\D^+)+\Hb^{c+,c+}(\D^+)$.
		As before, working in coordinates $\rho_{0,u_2}, \rho_{+,u_2}$ we then have that $S_{u_2}=-\rho_{0,u_2}\partial_{\rho_{0,u_2}}|_{\rho_{+,u_2}}$, so the integral curves of $S_{u_2}$ intersect
	 $\outcone{1}$ and $I^0$ once.
		Therefore, we can integrate from $\outcone{1}$ towards $I^0$ using \cref{item:ode:1d} (using that $\psi|_{\outcone{1}}\in \A{phg}^{\E^{\phi}_+}(\D^+)+\Hb^{c+}(\D^+)$) to obtain that we can restrict  $(\E_+^\phi)_{\leq c}\subset\big(0\overline{\cup}(\E^f_+-1)\big)\overline{\cup}(\E^f_0-2)$.
		Taking $c\to\infty$ yields the result.
	\end{proof}
 \begin{rem}[Explicitly computing the index sets for examples and commenting on the improvement.]
      Let us work out  \cref{eq:prop:suboptimal_future}  and \cref{eq:prop:optimal_future} in a straightforward case.
		Consider \cref{eq:prop:future_data} with $a_0=\infty$ and $\E^{f}_0-2=\overline{(p,0)}=\E_0$. 
		Then the future index sets read, respectively:
		
			\begin{equation}\label{eq:app:bad_index_on_minkowski}
				\text{for }p\notin\mathbb N:\qquad \E_+=\mindex{0}\cup\{(n+p,m):n\in\N, m\in\N_{\leq n}  \}, \quad \text{vs.} \quad \E^+_{\phi}=\mindex0\cup \mindex{p}.
			\end{equation}
			
			\begin{equation}
				\text{for }p\in\mathbb N:\qquad\E_+=\big\{(n,k):n=m, k\leq\max(0,m-p+1) \text{ where } m\in\N_{\geq0} \big\}, \quad \text{vs.} \quad \E^+_{\phi}=\mindex0\cupdex\mindex{p}.
			\end{equation}

Thus, while for, say, $p\notin\mathbb{N}$ the index set $\E_+$ from \cref{eq:prop:suboptimal_future} contains conformally irregular terms $(n+p,m)$ for all $m\leq n$, $\E^\phi_+$ from \cref{eq:prop:optimal_future} only has conformally irregular terms $(n+p,0)$.

 \end{rem}

\begin{rem}[First conformal irregularity]\label{rem:prop:suboptimal}
		While \cref{lemma:prop:improvement_at_future} gives a significantly smaller index set near $\scrip$ than \cref{lemma:prop:future_corner}, the first conformally irregular part (i.e.~an element not of the form $(n,0)$ for $n\in\mathbb N$) appears at the same order for both. 
  In fact, this first conformally irregular part of the index set is sharp for all the above statements.
		For instance, for $p\in\N$ there exists $rf\in\A{phg}^{\mindex{p+2}}(\D^+\cap\D^-)$ with $\supp f\in\D^+\cap\D^-$ such that solutions to $\Box\phi=f$ satisfy $r\phi\in\A{phg}^{\mindex{p},\E^+}(\D^+)$ with the coefficient of $(p,1)\in\E^+$ nonvanishing, cf.~\cref{rem:ODE:sharp}.
  For  $f=0$, we may use \cref{eq:intro:explicit_sol} to exhibit sharp examples. 
		
Similarly, the first conformally irregular term of the index sets in $\D^-$ that we will construct below (\cref{lemma:prop:inhomogeneous_no_rad,prop:prop:inhomogeneous_rad,cor:prop:solution_with_radiation_scri}) is also sharp. 	
However, an \textit{extraordinary} cancellation happens in the global $\D$ in the case of no incoming radiation in \textit{even spacetime dimensions}, which makes the index sets in \cref{prop:prop:prop_no_incoming} suboptimal, see already \cref{lemma:app:minkowski_index_improvement}.

We emphasise, however, that in applications, in order to find the exact structure of the index sets, we will always resort to explicit computations, see \cref{sec:app}. In particular, the explicit forms of the index sets given in \cref{lemma:prop:future_corner} or \cref{lemma:prop:improvement_at_future} are never relevant for actual computations, we just need to know that the solution is polyhomogeneous with respect so \textit{some} index set.
\end{rem}

	\subsection{Propagation of polyhomogeneity near the past corner \texorpdfstring{$\scrim\cap I^0$}{}, I: No incoming radiation}\label{sec:propagation:past_corner}
	We will now prove the analogue of \cref{lemma:prop:future_corner} in the past region $\D^-$, that is, we will propagate polyhomogeneity from initial data along $\incone$ and a polyhomogeneous inhomogeneity. 
We will first work with the no incoming radiation condition and, in particular, also assume that $f\in\Hb(\Dbold)+\A{phg}(\Dbold)$, i.e.~compatible with the no incoming radiation condition. 
More precisely, we assume:
\begin{ass}\label{ass1}
   Let for some real numbers $a^{\incone}, \vec{a}^f$ and index sets $\E^{\incone}, \vec{\E}^f$:
   \begin{align}
       \psi^{\incone}\in\A{phg}^{\E^{\incone}}(\incone)+\Hb^{a^{\incone}}(\incone) \text{ with } \min(\E^{\incone})>-1/2\leq a^{\incone},\\   rf\in\A{phg}^{\vec{\E}^f}(\Dbold)+ \Hb^{a^f_0,a^f_+}(\Dbold) \text{ with } \min(\E_0^f),a_0^f>1,
   \end{align}
   and let, for some $\vec{a}'$, $\psi\in \Hb^{\vec{a}'}(\Dbold)$ be the scattering solution to \cref{eq:prop:wave} with no incoming radiation arising from \cref{thm:scat:scat_general}.
\end{ass}
Notice that have now switched from the $\Hbt(\D)$-spaces to the $\Hb(\Dbold)$ spaces for more convenient notation. We remind the reader of the inclusion \cref{eq:not:inclusionDbold}.
\begin{rem}\label{rem:prop:ass1}
    Since we work in no incoming radiation spaces, we can in fact drop the restrictions on $\min(\E^{\incone})$, $a^{\incone}$, $a_0^f$ and $\min(\E_0^f)$, provided that we specify additional transversal derivatives along $\incone$; cf.~\cref{thm:scat:weak_decay}.
\end{rem}
	
	\begin{lemma}\label{lemma:prop:inhomogeneous_no_rad}
		Under \cref{ass1}, we have ${\psi\in\A{phg}^{\mathcal{E}_0^{\phi}}}(\Dbold ^-)+\Hb^{a_0}(\Dbold^-)$ for $\mathcal{E}_0^{\phi}=\mathcal{E}^{\incone}\cup(\mathcal{E}^{f}_0-2)$ and any $a_0<\min(a_0^f-2,a^{\incone})$.
	\end{lemma}
	
	\begin{proof}
Similarly to \cref{eq:prop:Su1}, we define\footnote{The $v_0$-shift from $S$ to $S_{v_0}$ uses a time translation symmetry of $\Box_\eta$ and is equivalent to assuming $v=0$ on the incoming cone.} $S_{v_0}:=(u-v_0)\pu+(v-v_0)\pv$ and
		\begin{equation}\label{eq:prop:S_c_definition}
			\mathcal{S}^{\mathcal{E}}_c:=\prod_{(z,k)\in\mathcal{E}_{\leq c}}(S_{v_0}+z),
		\end{equation}
		and write $\mathcal{S}_c=\mathcal{S}_c^{\mathcal{E}_0^{\phi}}$.
		In coordinates $\rho=r^{-1}$, $\tau=(t-v_0)/r$ we get that $S_{v_0}=-\rho\partial_{\rho}|_{\tau}$; thus, using as before the equivalent definition of polyhomogeneity from \cref{rem:notation:equivalent_polyhoms}, it suffices to show that $\mathcal{S}_c \psi\in\Hb^{c}(\Dbold^-)+\Hb^{a_0}(\Dbold^-)$ for arbitrary $c$ to prove polyhomogeneity.

		We show this by first computing the initial data for $\mathcal{S}_c\psi$.
		Decomposing the product into its individual terms, we note that all terms with $\pv$-derivative vanish on $\incone$, as they are multiplied by $(v-v_0)$.
		On $\incone$, we may use the coordinate $\tilde{u}=u-v_0$ so that  we get $\mathcal{S}_c\psi|_{\Cbar}\in \A{phg}^{(\mathcal{E}_0^{\phi})_{>c}}(\Cbar)+\Hb^{a^{\incone}}(\incone)$ (using that $\E^\phi_0\supset \E^{\incone}$).
		
		Furthermore, since $\E^\phi_0\supset \E^f_0-2$, we also have that $r^{-2}\mathcal{S}_cr^2rf\in \Hb^{\min(c+2,a_0^f)}(\Dbold )$, cf.~\cref{eq:prop:wave:commute}.
		Using the improved initial data $\mathcal{S}_c\psi|_{\incone}$, the improvement on the inhomogeneity and that $\mathcal{S}_c\psi$ solves the wave equation \cref{eq:prop:wave:commute}, we can apply \cref{thm:scat:scat_general} to conclude that $\mathcal{S}_c\psi\in\Hb^{\min(a_0,c)}(\Dbold^-)$ for any $a_0<\min(a^{\incone},a_0^f-2)$.
		As this holds for arbitrary $c$, the result follows.
	\end{proof}
	
	We can combine this with \cref{lemma:prop:improvement_at_future} to obtain
	\begin{prop}\label{prop:prop:prop_no_incoming}
  Under \cref{ass1}, we have $\psi\in\A{phg}^{\vec{\E}^\phi}(\Dbold )+\A{b,phg}^{a_0,\mindex{0}}(\Dbold)+\Hb^{a_0,\min(a_+^f-1,a_0)-}(\Dbold )$ for any $a_0<\min(a^{\incone},a^f_0-2)$, $\E^\phi_0=(\E^f_0-2)\cup\E^{\incone}$ and $\E^\phi_+=\Big(\big(0\overline{\cup}(\E^f_+-1)\big)\cupdex(\E^f_0-2)\Big)\cup\big(\mindex{0}\cupdex\E^{\incone}\big)$.
	\end{prop}
	\begin{proof}
		This follows from patching together \cref{lemma:prop:inhomogeneous_no_rad,lemma:prop:future_corner,lemma:prop:improvement_at_future}, via a cutoff $\chi$ supported in $\D^+\cap\D^-$.
        In the corner $\D^+$, we furthermore split up the solution into a part with data and no inhomogeneity, and into a part with inhomogeneity and no data.
	\end{proof}

	\subsection{Propagation of polyhomogeneity near the past corner \texorpdfstring{$\scrim\cap I^0$}{}, II: Incoming radiation}\label{sec:prop:II}
	The goal of this section is to drop the assumption of no incoming radiation. 
We do this by first weakening the assumption  $f\in\A{phg}(\Dbold)$ in favour of the weaker $f\in\A{phg}(\D)$, while still assuming no incoming radiation.
	We can then no longer peel off the leading order terms in $f$ at $\scrim$ and $I^0$ at the same time using scaling. Instead, we first propagate polyhomogeneity along $\scrim$ in \cref{lemma:prop:data}, and then, in \cref{lemma:prop:scri_to_I0}, propagate polyhomogeneity towards $I^0$ using time integrals as in \cite{angelopoulos_asymptotics_2018}.\footnote{Similarly to \cite{hintz_linear_2023}, one could alternatively use the Mellin transform and contour shifting to obtain a similar result.} 
 We can then finally extend our results to solutions with nontrivial incoming radiation by treating the incoming radiation as an inhomogeneity, see already \cref{cor:prop:solution_with_radiation_scri}.
 We again state our assumptions in one place:
	\begin{ass}\label{ass2}
   Let $\psi^{\incone}$ as in \cref{ass1},  let $\vec{\E}^f$ be index sets, and let $\vec{a}^f$ be real numbers. Let 
   \begin{align}\label{eq:prop:ass2}
  rf\in\A{phg}^{\vec{\E}^f}(\D)+ \Hb^{a_-^f,a^f_0,a^f_+}(\D) \text{ with } \min(\E_-^f),a_-^f>1,
   \end{align}
   and let, for some $\vec{a}'$, $\psi\in \Hb^{\vec{a}'}(\D)$ be the scattering solution to \cref{eq:prop:wave} with no incoming radiation arising from \cref{cor:scat:enlarged_admissible_set}.
  
   Furthermore, define $a_-=\min(a^{\incone}, a^f-1)$, $a_0=\min(a_--,a^f_0-2)$ and $a_+=\min(a_0-,a^f_+-1)$ as well as the index sets
   \begin{equation}\label{eq:prop:assumption:indexsets}
       \E^\phi_-=\E^{\incone}\cup(\E^f_--1),\quad \E^\phi_0=\big((\E^f_--1)\overline{\cup}(\E^f_0-2)\big)\cup\E^{\incone}, \quad \E^\phi_+=(0\cupdex(\E^f_+-1))\cupdex\E^\phi_0. 
   \end{equation}
\end{ass}
\begin{rem}\label{rem:prop:restrictiondropping}
  As in \cref{rem:prop:ass1}, we can actually drop the assumptions on $\min(\E^{f}_-), \min(\E^{\incone})$, cf.~\cref{thm:scat:weak_polyhom}, but we can no longer drop the assumptions on $a_-^f$ and $a^{\incone}$.
\end{rem}

	\begin{lemma}[Propagation along $\scrim$]\label{lemma:prop:data}
		Under \cref{ass2}, we have $\psi\in \A{phg,b}^{\mathcal{E}_-^\phi,a_0'}(\D^{-})+\Hb^{a_-,a_0'}(\D^{-})$.
  
  If we replace \eqref{eq:prop:ass2} by $rf\in \A{phg,b}^{\E_-^f,a_0^f}(\D^-)$, then $\psi\in \A{phg,b}^{\mathcal{E}_-^\phi,a_0'}(\D^{-})$.
	\end{lemma}
	\begin{rem} 
		The index set $\E^f_-$ is shifted by 1 in this lemma, whereas it was shifted by 2 in \cref{lemma:prop:inhomogeneous_no_rad}.
		Thus, the index set near $\scrim$ is better than in  \cref{lemma:prop:inhomogeneous_no_rad}.
		The difference follows from the different scaling behaviour of $\Box_\eta$ towards $\scrim,I^0$.
	\end{rem}
 \begin{proof}
     We want to mimic the proof of \cref{lemma:prop:inhomogeneous_no_rad} in order to propagate polyhomogeneity along $\scrim$: 
     Working with coordinates $\rho_-=v/r,\,\rho_0=v^{-1}$, we get that $rvP_\eta=-(\rho_-\partial_-)(\rho_-\partial_--\rho_0\partial_0)+\rho_-\Dl$.
    We define
     \begin{equation}\label{eq:prop:S:with:upu}
			\mathcal{S}^\mathcal{E}_{ c}=\prod_{(z,k)\in\mathcal{E}_{\leq c}}(\rho_-\partial_{\rho_-}-z),\qquad\dot{\mathcal{S}}^\mathcal{E}_{c}=\prod_{(z,k)\in\mathcal{E}_{\leq c}\setminus\mathcal{E}_{\leq c-1}}(\rho_-\partial_{\rho_-}-z),
		\end{equation}
  use the short-hand $\mathcal{S}_c:=\mathcal{S}^{\mathcal{E}^{\phi}_-}$ and the convention that for $c<\min\E$, $\mathcal{S}_{\leq c}^{\E}=1$,
        and observe the following commutation
        \begin{equation}\label{eq:prop:commute}
            rv\pu\pv \mathcal{S}_c\psi=\rho_-\Dl \mathcal{S}_{c-1}\psi-\mathcal{S}_c (r^2vf).
        \end{equation}

          Having established the wave equation satisfied by $\mathcal{S}_c \psi$, we now improve initial data and inhomogeneity:
For the initial data, we note that $\mathcal{S}_c$ only contains tangential derivatives on $\incone$ and, since $\E^\phi_-\supset \E^{\incone}$, we immediately get that $\mathcal{S}_c\psi|_{\incone}\in \Hb^{\min(c,a^{\incone})}(\incone)$. 
Similarly, we have the improvement $\mathcal{S}_{c}^{\E^{\phi}_-}(vr^2f)=rv\mathcal{S}_{c+1}^{\E^{\phi}_-+1}(rf)\in rv\Hb^{\min(c+1+,a_-^f),a_0'+2}(\D^-)$, using that $\E^{\phi}_-\supset \E^f_--1$. 

We now claim that, for arbitrarily large $c$,
		\begin{equation}\label{eq:prop:data_claim}
			\mathcal{S}_{c}\psi\in\Hb^{\min(c,a^{\incone}), a_0'}(\D^-).
		\end{equation} 
		For $c<a'_-$, the above holds by assumption of the lemma.
		 Inductively, we assume that for some fixed $c$, $\mathcal{S}_{c-1}\psi\in \Hb^{\min(c-1,a^{\incone}),a_0'}(\D^-)$. 
  By \cref{eq:prop:commute}, we get $\pu\pv\mathcal{S}_{c}\psi\in \Hb^{\min(c+1,a^{\incone}),a_0'+2}(\D^-)$. 
		Applying now \cref{corr:ODE:du_dv} to \cref{eq:prop:commute}, we get that $\mathcal{S}_c\psi\in\Hb^{\min(c,a^{\incone}),a_0'}(\D^-)$, which completes the proof.

  The case where $f$ is in a mixed $\A{phg,b}$ space follows analogously.
     
 \end{proof}

    \begin{lemma}[Propagation towards $I^0$]\label{lemma:prop:scri_to_I0}
        Under \cref{ass2}, we have ${\psi\in\A{phg}^{\vec{\E}^{\phi}}}(\D^-)+\Hb^{\vec{a}}(\D^-)$.

        If we replace \cref{eq:prop:ass2} with $rf\in\A{phg,b}^{\E^f_-,a_0^f}(\D^-)$, then we have $\psi\in\A{phg}^{\vec{\E}^\phi}(\D^-)+\A{phg,b}^{\E^\phi_-,a_0}(\D^-)$, where $\vec{\E}^\phi$ is as in \cref{eq:prop:assumption:indexsets} but with  $\E^f_0-2$ replaced by $\emptyset$.
    \end{lemma}

    \begin{proof}
        By linearity, we can set $\psi^{\incone}=0$ and assume that $\supp f\subset\{v>v_0+1\}$. More precisely, we can split $\phi=\phi_1+\phi_2+\phi_3$ such that $\Box_\eta\phi_1=0$ with data $\psi^{\incone}$ (and apply \cref{lemma:prop:inhomogeneous_no_rad}), such that $\Box_\eta\phi_2=\chi f$, $\Box_\eta\phi_3=(1-\chi)f$, both with trivial data, where $\chi$ is a cutoff supported in $v>v_0+1$ (and identically 1 for $v\geq v_0+2$).
        For $\phi_3$, we can then apply \cref{lemma:prop:data} to get back into the setting of \cref{lemma:prop:inhomogeneous_no_rad}.

		Furthermore, we can use \cref{thm:scat:scat_general}, \cref{cor:scat:enlarged_admissible_set} and \cref{lemma:prop:future_corner} to deal with the $\Hb$ error part of $f$ and $\psi^{\incone}$, so we here assume that $\psi^{\incone}$ and $rf$ are polyhomogeneous (we treat the mixed case in the end).
            
		Fix $k\geq 1$: Using that $\psi$ and $f$ are supported away from $v\leq v_0+1$,  and the fact that $[T,P_\eta]=0$, we get that $P_\eta T^{-k}\psi=T^{-k}rf$, where
		\begin{equation}\label{eq:prop:T-k_integral}
			T^{-k}g:=\underbrace{T^{-1}T^{-1}...T^{-1}}_{k\text{ times}}g=\int_{2v_0-r}^t(g)|_{(t',r,\omega)} \frac{(t-t')^{k-1}}{(k-1)!}\dd t',\qquad g\in\{rf,\psi\}
		\end{equation}
        where the formula above follows by simple integration by parts.
       Iteratively using \cref{corr:ODE:du_dv}\cref{item:ode:t-prop} (with trivial data), we get $T^{-k}rf\in \A{phg,b}^{\E^f_-,\E'_k}(\D^-)$ for some index set $\E'_k$. Note that the index set towards $\scrim$ remains the same. On the other hand, in order to get better control over the resulting index set $\E'_k$, we instead expand the bracket in \cref{eq:prop:T-k_integral}, set $k_1+k_2=k-1$ and apply \cref{corr:ODE:du_dv}\cref{item:ode:t-prop} \textit{once} to each term:
		\begin{nalign}\label{eq:prop:tintegral:termbyterm}
			t^{k_1}\int_{2v_0-r}^t rf(t',r,\omega)(t')^{k_2}\dd t'&\in (\rho_0\rho_-)^{-k_1}\mathcal{A}_{\phg}^{\E^f_--k_2,(\E^f_--k_2)\cupdex(\E^f_0-k_2-1)}(\D^-)=\mathcal{A}_{\phg}^{\E^f_--k+1,\Big(\E^f_-\cupdex(\E^f_0-1)\Big)-k+1}(\D^-).
		\end{nalign}
		Taking intersections of the index sets, we can thus take $T^{-k}rf\in \A{phg,b}^{\E^f_-,\E'_k}(\D^-)$ for $\E'_k=\Big(\E^f_-\overline{\cup}(\E^f_0-1)\Big)-k+1$. (Note that the index set towards $\scrim$ resulting from \cref{eq:prop:tintegral:termbyterm} is bad, but we already know that we can take it to be $\E^f_-$ by iterative application of \cref{corr:ODE:du_dv}\cref{item:ode:t-prop}.)
		Notice already that $\E^\phi_0=\E'_k+k-2$, with $\E^\phi_0$ as in the statement of the proposition.
		
		Commuting the wave equation with $\mathcal{S}^{\E'_k-2}_c$ from \cref{eq:prop:S_c_definition}, we get the following equation with trivial data:\footnote{Note that since the index set $\E'_k$ comes with a $-k$, commuting with $\mathcal{S}_c$ for, say,  $c=0$ already implies many commutations.}
		\begin{equation}
			P_\eta \mathcal{S}^{\E'_k-2}_c T^{-k}\psi =r^{-2}\mathcal{S}^{\E'_k-2}_c r^2 T^{-k}rf\in\A{phg}^{\E^f_-,(\E'_k)_{>c}}(\D^-).
		\end{equation}
		We use \cref{thm:scat:scat_general} to conclude that $\mathcal{S}_c^{\E'_k-2} T^{-k}\psi\in \Hb^{a,a}(\D^-)$, for $a=\min(\min(\E^f_--1)-,c-2)$.
		Using \cref{lemma:prop:data} we may improve this to $\mathcal{S}_c^{\E'_k-2} T^{-k}\psi\in \A{phg,b}^{\E^\phi_-,a}(\D^-)$, with $\E^{\phi}_-$ as defined in the proposition.
		
 Finally, we commute with time derivatives, using that $[T,S]=T$ and thus $T^k\mathcal{S}^{\E}_c=\mathcal{S}^{\E+k}_{c+k}T^k$:
		\begin{equation}\label{eq:prop:incoming_improved_rphi}
			T^k\mathcal{S}_c^{\E'_k-2}T^{-k}\psi=\mathcal{S}_{c+k}^{\E'_k-2+k}\psi=\mathcal{S}_{c+k}^{\E^\phi_0}\psi\in \A{phg,b}^{\E^\phi_-,a+k}(\D^-),
		\end{equation}
        Since we can take $k\in\N_{\geq0}$ and $c$ arbitrary, \cref{eq:prop:incoming_improved_rphi} implies $\mathcal{S}^{\E^\phi_0}_{k}\psi\in\A{phg,b}^{\E^\phi_-,k-2}(\D^-)$, which, in turn, implies $\mathcal{S}^{\E^\phi_0}_{k-2}\psi\in\A{phg,b}^{\E^\phi_-,k-2}(\D^-)$; thus, the result follows.

        For $rf\in\A{phg,b}^{\E^f_-,a_0^f}(\D^-)$, then instead of \cref{eq:prop:incoming_improved_rphi} we get by the same chain of arguments:
        \begin{equation}
            \mathcal{S}^{\E^{\phi}_0}_{c+k}\psi\in\A{phg,b}^{\E^\phi_-,a+k}(\D^-)+\A{phg,b}^{\E^\phi_-,a_0}(\D^-),
        \end{equation}
        so the second claim also follows.
    \end{proof}

We thus deduce	\begin{prop}\label{prop:prop:inhomogeneous_rad}
	Under \cref{ass2}, we have ${\psi\in\A{phg}^{\vec{\E}^{\phi}}}(\D)+\A{b,b,phg}^{a_-,a_0,\mindex0}(\D)+\Hb^{\vec{a}}(\D)$. 
	\end{prop}
	
	\begin{proof}
        This follows from \cref{lemma:prop:scri_to_I0,lemma:prop:improvement_at_future} via a cutoff function.
	\end{proof}
	
	We finally treat incoming radiation. By linearity, we can set $\psi^{\incone}$ and $f$ to vanish:
	
	\begin{cor}\label{cor:prop:solution_with_radiation_scri}
		Let $\pv\psi^{\scrim}(v,\omega)$  be such that $\psi^{\scrim}:=\int_{v_0}^v \pv\psi^{\scrim} \dd v'$ satisfies $\psi^{\scrim}\in\A{phg}^{\E^{\scrim}}(\scrim)+\Hb^{a_0}(\scrim)$ for $a_0\in\mathbb{R}$. Let further $\psi^{\incone}=rf=0$, and let $\psi$ be the corresponding scattering solution from \cref{thm:scat:scat_incoming}. Then $\psi\in\A{phg}^{\mindex{0},\E^{\scrim}\overline{\cup}\mindex{1},\mindex0\cupdex (\E^{\scrim}\overline{\cup}\mindex{1}) }(\D)+\A{phg,b,phg}^{\mindex0,a_0-,\mindex0}(\D)+\A{phg,b,b}^{\mindex0,a_0-,a_0-}(\D)$. 
	\end{cor}
 \begin{rem}[Optimality of the index set for error term]
  In the case where $\psi^{\scrim}\in\Hb^{a_0}(\scrim)$, we will later show the stronger statement that \textit{in even spacetime dimensions}
    $\psi\in \A{phg}^{\mindex 0,\mindex1,\mindex 0}(\D)+\A{phg,b,phg}^{\mindex0,a_0-,\mindex0}(\D)+\A{phg,b,b}^{\mindex0,a_0-,a_0-}(\D) $ (without the index set $\mindex0\cupdex\mindex1$ near $\scrip$). In words, there are no logarithms towards $\scrip$ before order $a_0$. See already \cref{cor:app:improvedthmgeneral2}.
On the other hand, we will also show that in general dimensions, the index set given by \cref{cor:prop:solution_with_radiation_scri} is sharp. See already \cref{prop:even}.
 \end{rem}

 \begin{obs}[Improved decay towards $I^0$]\label{rem:poly:improved_decay_at_I0}
    Notice that, in \cref{cor:prop:solution_with_radiation_scri}, for compactly supported $\pv \psi^{\scrim}$, we get that the solution $\psi$ towards $I^0$ has an expansion with coefficients in $\mindex{0}$.
    Each of these terms can be computed explicitly from $\psi^{\scrim}$ as a weighted integral, for instance the first element with decay $r^0$ vanishes if $\int \pv\psi^{\scrim}=0$.
    This means that for specific choices of incoming radiation, we may relax the assumption $a_0<a_-$, provided that certain integrals over $\psi^{\scrim}$ vanish.
    We do not incorporate such improvements in our general scattering result.
 \end{obs}
	
	\begin{proof}
		Let's write $\psi(u,v,\omega)=\psi^{\scrim}(v,\omega)+\psi'(u,v,\omega)$.
    We first consider the case where $\psi^{\scrim}\in\A{phg}^{\E^{\scrim}}(\scrim)$.		
  Then $P_{\eta}\psi'(u,v,\omega)=r^{-2}\Dl \psi^{\scrim}(v,\omega)\in\A{phg}^{\mindex{2},\E^{\scrim}+2}(\D^-)$, and $\psi'$ has no incoming radiation. 
  We thus get $\psi'\in \A{phg}^{\mindex1,\E^{\scrim}\cupdex\mindex{1}}(\D^-)$ from \cref{prop:prop:inhomogeneous_rad}. The result follows by applying a cutoff localising to $\D^+$ and using \cref{eq:prop:future_data} along with \cref{eq:prop:optimal_future} for $P_\eta\psi=0$.

  On the other hand, if $\psi^{\scrim}\in \Hb^{a_0}(\scrim)$, then we get that $P_\eta\psi'\in \A{phg,b}^{\mindex2,a_0+2}(\D^-)$ with $\psi'$ having trivial data and no incoming radiation.
  \cref{lemma:prop:scri_to_I0} then implies $\psi'\in \A{phg}^{\mindex1,\mindex 1}(\D^-)+\Hb^{\mindex1, a_0}(\D^-)$, which, in turn, gives $\psi\in \A{phg}^{\mindex0,\mindex 1}(\D^-)+\Hb^{\mindex0, a_0}(\D^-)$. 
  The result in $\D^+$ follows using a cut-off function and \cref{eq:prop:future_data}.
  \end{proof}

For later reference, we also give a more precise characterisation of where the terms in the expansion towards $I^0$ come from in the case where the data along $\scrim$ do not have an expansion:
\begin{cor}\label{cor:prop:expansiontowardsI0}
    Under the assumptions of \cref{cor:prop:solution_with_radiation_scri} with $\E^{\scrim}=\emptyset$, we can write $\psi=\tilde{\psi}+\psi_{\Delta}$, where 
    i) both $\tilde\psi,\psi_\Delta$ solve $P_\eta\tilde\psi=P_\eta\psi_\Delta=0$; ii)
    $\psi_{\Delta}\in \A{phg,b}^{\mindex{0},a_0-}(\D^-)$; 
    iii) $\tilde\psi$ has no incoming radiation;
    iv) $\tilde{\psi}|_{\incone}=-\psi_\Delta|_{\incone}\in \A{phg}^{\mindex{0}}(\incone)$. 
\end{cor}

  \begin{proof}
The proof is essentially the same as in \cref{lemma:prop:scri_to_I0}:
Let $n=\lfloor a_0\rfloor$. We construct $\tilde\psi=\sum_{1\leq i \leq n}\psi^{(i)}$ by iteration. 
Assuming that $a_0>1$ (otherwise there is nothing to show), we define $T^{-1}\psi$ to be the scattering solution to $P_\eta(T^{-1}\psi)=0$ with trivial data on $\incone$ and with data 
\begin{equation}
    T^{-1}\psi^{\scrim}:=\int_{v_0}^v \psi^{\scrim} \dd v=\underbrace{\int_{v_0}^{\infty} \psi^{\scrim} \dd v}_{:=T^{-1}\psi^{\scrim}_1}-\underbrace{\int_{v}^{\infty} \psi^{\scrim} \dd v}_{:=T^{-1}\psi^{\scrim}_2} \in \A{phg}^{\mindex0}(\scrim)+\Hb^{a_0-1-}(\scrim),
\end{equation}
where we used \cref{eq:ODE:1d_error}.
 Observe that $T(T^{-1}\psi)=\psi$ by uniqueness.
 We similarly define $T^{-1}\psi_i$ to be the solution to $P_\eta(T^{-1}\psi_i)=0$ with scattering data $T^{-1}\psi^{\scrim}_i $ and $T^{-1}\psi^{\incone}_i\equiv T^{-1}\psi^{\scrim}_i(v=v_0)$, and we then define $\psi^{(1)}:=TT^{-1}\psi_1\in  \A{phg}^{\mindex0,\mindex1}(\D^-)$; this clearly has no incoming radiation.

 In turn, for $T^{-1}\psi_2$, we may apply\footnote{Note that $T^{-1}\psi_2$ also has constant data on $\incone$, but this poses no problem.} \cref{cor:prop:solution_with_radiation_scri} to obtain 
 $
     T^{-1}\psi_2\in\A{phg}^{\mindex0,\mindex1}(\D^-)+ \A{phg,b}^{\mindex0,a_0-1}(\D^-).
$

Assuming that $a_0>2$, we next consider the time-inverse of $T^{-1}\psi_2$, again split up $T^{-2}\psi_2$ into a part with constant data at $\scrim$ and into a part sitting in $\Hb^{a_0-2}(\scrim)$ and repeat the argument. The result follows inductively.
	\end{proof}
 \begin{obs}\label{obsidobsi}
    A closer inspection of the proof reveals that the first coefficients in the expansion of $\psi$ towards $I^0$ vanish (so then $\psi\in\A{phg,b}^{\mindex0,a_0}(\D^-))$ if the integrals $\int_{v_0}^\infty(v'-v_0)^{k}\psi^{\scrim}\dd v'$ vanish for all $0\leq k< n-1$. 

    In other words, for $k'<n-1$, if it holds that for all $k\leq k'$ we have $\int_{v_0}^\infty(v'-v_0)^{k}\psi^{\scrim}\dd v'=0$, then we have $\psi\in\A{phg,b}^{\mindex{0},a_0-}(\D^-)+\A{phg}^{\mindex{0},\mindex{k'+2}}(\D^-)$. 
\end{obs}

 \subsection{Propagation of polyhomogeneity for linear long-range potentials}\label{sec:prop:longrange}
 We recall from \cref{thm:scat:long} that all scattering statements and energy estimates remain valid if one includes linear long-range potentials as in \cref{def:en:long_range}. 
 In this subsection, we prove that the propagation of polyhomogeneity statements also remain valid, with some minor modifications. Since the $\Hb$-estimates can be dealt with exactly as before, we will restrict to the case of pure polyhomogeneity, that is, we always assume polyhomogeneous data $\psi^{\incone}\in\A{phg}^{\E^{\incone}}(\incone)$ and polyhomogeneous inhomogeneity $rf\in \A{phg}^{\vec{\E}^f}(\D)$ with the usual restrictions (though recall \cref{rem:prop:restrictiondropping}) that $\min(\E^{\incone})>-1/2$, $\min(\E_-^f)>1$.
 
 We will, for $V_L$ an admissible long-range potential as in \cref{def:en:long_range}, consider the equation
 \begin{equation}\label{eq:prop:longrangewave}
    P_{\eta,L}\psi:=P_{\eta}\psi-V_L\psi=P_{\eta}\psi-\frac{f^L_{-}(\rho_-,\omega)f^L_{+}(\rho_+,\omega)}{r^2}\Ve\psi=rf,\qquad f^L_{\pm,\b}\in \A{phg}^{\E^{L}_{\pm}-1}([0,1)\times S^2),
\end{equation}
where $f^L_{\pm,\b}$ are defined via $f^L_{\pm,\b}\Vb=f^L_{\pm}\Ve$.
Recall that $V_L$ being a long-range potential implies that $\min(\E^L_{\pm})>0$ (and also $\min\E^{L}_{\pm}>1/2$ in case $V_L$ contains angular derivatives), that $f_{\mp}\equiv 1$ on $\D_{\pm}\setminus \D_{\mp}$.
Recall also that $V_L$ is compatible with the no incoming radiation condition if $\E^L_-=\mindex1$.

 We now first list the corresponding changes to the statements of the previous subsections, and then sketch the relevant modifications of the proofs.
In all the statements below, we assume $\psi$ to be a solution to \cref{eq:prop:longrangewave} rather than to $P_{\eta}\psi=rf$, with $f^L_{\pm}$ as above, and with additionally $\E^L_-=\mindex1$ in the case of no incoming radiation.
\begin{lemma}\label{lemma:prop:futurecorner:long}
    \cref{lemma:prop:future_corner} still holds but with $\E_+$ now being minimal with the property that
    \begin{equation}
        \E_+\supset (\E_0\cupdex(\E_+^f-1))\cupdex 0, (\E_0\cupdex(\E_++1))\cupdex 0, (\E_0\cupdex(\E_++\E_+^L))\cupdex 0
    \end{equation}
\end{lemma}
\begin{lemma}\label{lemma:prop:inhomogeneous_no_rad:long}
    \cref{lemma:prop:inhomogeneous_no_rad} still holds without changes.
\end{lemma}
\begin{lemma}\label{lemma:prop:data:long}
    \cref{lemma:prop:data} still holds with $\E_-^\phi$ now given by the minimal index set satisfying $\E_-^\phi\supset\E^{\incone}\cup(\E_-^f-1), \E^\phi_-+\E^L_-$.
\end{lemma}
\begin{lemma}\label{lemma:prop:timeintegral:long}
    Restricting to $\D^-$, \cref{lemma:prop:scri_to_I0} and \cref{cor:prop:solution_with_radiation_scri} still hold with $\E_-^{\phi}$ modified as in \cref{lemma:prop:data:long}.
\end{lemma}
Notice that the proof of \cref{lemma:prop:improvement_at_future} in general does not apply anymore\footnote{The property used in \cref{footnote:prop:improvedfuture} no longer holds if $\E_+$ is taken to be as in \cref{lemma:prop:futurecorner:long}.
This may be remedied using the extra assumption $\E^L_+\subset\N$, see \cref{lemma:app:linear_improvement} for similar improvement}.
Thus, in summary, we can still deduce propagation of polyhomogeneity statements in all of $\D$, but the future index sets (towards $\scrip$) appearing in \cref{prop:prop:prop_no_incoming,prop:prop:inhomogeneous_rad} in general needs to be replaced with the worse one given by from \cref{lemma:prop:futurecorner:long}.
\begin{obs}[Applicability to scale-invariant wave equation]
   In the particular case where $\E^L_{\pm}=\mindex1$, then the proof of \cref{lemma:prop:improvement_at_future} still works, and both \cref{prop:prop:prop_no_incoming,prop:prop:inhomogeneous_rad} hold without modifications. 
In particular, the propositions apply without changes to the scale invariant wave equation $P_{\eta}\psi+\frac{c\psi}{r^2}=rf$ or to the suitably twisted Teukolsky equations.
\end{obs}

\begin{proof}[Proof of \cref{lemma:prop:futurecorner:long,lemma:prop:inhomogeneous_no_rad:long,lemma:prop:data:long,lemma:prop:timeintegral:long}]
   \cref{lemma:prop:futurecorner:long}: We first note that the scaling vector field $S=u\pu+v\pv$ still commutes with the equation since $S\rho_{\pm}=0$.
Thus, we can still propagate polyhomogeneity along spacelike infinity as before. 
In order to propagate from $I^0$ towards $\scrip$, we note that the long-rang term can again be treated perturbatively, i.e.:
\begin{equation}
    \rho_+\partial_+(\rho_+\partial_{\rho_+}-\rho_0\partial_{\rho_0})\psi=-\rho_0^{-2}\rho_+^{-1}(rf)+\frac{\rho_+\Dl\psi}{(1+\rho_+)^2}-\frac{f^L_{+,\b}(\rho_+,\omega)\rho_+\Vb\psi}{(1+\rho_+)^2}.
\end{equation}
In the same way in which the Laplacian is responsible for the inclusion $\E_+\supset (\E_0\cupdex(\E_++1))\cupdex 0$, the $f^L_{+}$-term gives rise to the inclusion $\E_+\supset (\E_0\cupdex(\E^++\E^L_+))\cupdex 0$.
This concludes the proof of \cref{lemma:prop:futurecorner:long}.
For the rest of the proof, we work in $\D^-$ and drop $f^L_{+}$.

The proof of \cref{lemma:prop:inhomogeneous_no_rad} no longer directly applies, as the perturbation does not have time translation symmetry, so we avoid shifting $S$ by $v_0 T$ as done in \cref{eq:prop:S_c_definition} and instead define:
\begin{equation}\label{eq:prop:pooprp}
			\mathcal{S}^{\mathcal{E}}_c:=\prod_{(z,k)\in\mathcal{E}_{\leq c}}(u\pu+v\pv+z), \mathcal{S}_c=\mathcal{S}_c^{\E^{\phi}_0}.
\end{equation}
We now need to show that  $\mathcal{S}_c\psi|_{\incone}\in\A{phg}^{(\E^\phi_0)_{>c}}(\incone)+\Hb^{a^{\incone}}(\incone)$.
We prove this by induction: The statement clearly holds for $c=-1/2$ by assumption.
Using the induction hypothesis and the no incoming radiation estimate \cref{eq:scat:data_no_incoming}, we also have that $\pv \mathcal{S}_c\psi|_{\incone}\in\Hb^{\min(a^{\incone},c)+1}(\incone)$.
Therefore, we have
\begin{equation}
	\mathcal{S}^{\E}_{c+1/2} \psi|_{\incone}=\prod_{(z,k)\in(\mathcal{E}^{\phi}_0)_{\leq c+1/2}\setminus(\mathcal{E}^{\phi}_0)_{\leq c} }(u\pu+z) \mathcal{S}^{\E}_{c} \psi|_{\incone}+\Hb^{\min(a^{\incone},c)+1}(\incone).
\end{equation}
With the initial data improved, the rest of the proof follows as in \cref{lemma:prop:inhomogeneous_no_rad}.

In order to prove \cref{lemma:prop:data:long}, we note the commutation property
  \begin{equation}\label{eq:prop:commute:long}
      rv\pu\pv\mathcal{S}_{c}\psi=\rho_- \Diff^N_{\b}\Dl\mathcal{S}_{c-1}\psi-\mathcal{S}_c(\rho_-f^L_{-,\b}\Vb\psi)-\mathcal{S}_c (r^2vf),
  \end{equation}
  where $\mathcal{S}_c=\mathcal{S}_c^{\E^{\phi}_-}$ is as defined in \cref{eq:prop:S:with:upu}. 
  Notice that $\rho_-f^L_{-,\b}\in\A{phg}^{\E^{L}_{-},\mindex1}(\D^-)$. 
  We then inductively assume that $\psi\in \A{phg,b}^{\E^\phi_-,a_0'}(\D^-)+\Hb^{c_n-, a_0'}(\D^-)$, where $c_n:=n\min(1,\min(\E^L_-))$ for $n\in\mathbb N$. 
The assumption implies that
\begin{equation}
   \mathcal{S}_{c_{n+1}-1}\psi\in \Hb^{c_{n+1}-1, a_0'}(\D^-), \qquad \mathcal{S}_{c_{n+1}}\underbrace{(\rho_-f^L_{-,\b}\Vb\psi)}_{\mathrlap{\in \A{phg,b}^{\E^\phi_-,a_0'}(\D^-)+\Hb^{\min(\E^{L}_-)+c_n-, a_0'}(\D^-)}}\in\Hb^{\min (\E^{L}_-)+c_{n}-, a_0'}(\D^-),
\end{equation}
and thus, inserting these memberships into \cref{eq:prop:commute:long}, we get
\begin{equation}
    rvP_{\eta}\mathcal{S}_{c_{n+1}}\psi\in \Hb^{c_{n+1},a_0'}(\D^-)+S_{c_{n+1}}(r^2vf).
\end{equation}
Dividing by $rv$ and applying \cref{corr:ODE:du_dv} gives the result. 

Finally, we note that the time integral proof of \cref{lemma:prop:scri_to_I0} is also affected by the inclusion of the $f^L_{-}r^{-2}\Ve\psi$-term. We thus avoid directly taking the time integral. 
Instead, we proceed as in the proof of \cref{cor:en:longrange1.5}  and write $\psi=\psi^{(0)}+...+\psi^{(N)}+\psi^{\Delta}$ where, for $k\in\{0,...,N-1\}$, we define
\begin{equation}
    P_\eta\psi^{(0)}=rf,\quad P_\eta\psi^{(k+1)}=f^L_{-}r^{-2} \Ve\psi^{(k)},\quad P_\eta\psi^{\Delta}_N=f^L_{-}r^{-2} \Ve\psi^{(N)}+f^L_{-}r^{-2}\Ve\psi^{\Delta}_N,
\end{equation}
where only $\psi^{(0)}$ has nontrivial data.
Using \cref{lemma:prop:scri_to_I0} iteratively, we get that $\psi^{(N)}\in\A{phg}^{{\E}_-^\phi+N\E^L_-,\E_0^\phi}(\D^-)$ for ${\E}_-^{\phi}$ and~$\E_0^\phi$ as in \cref{lemma:prop:scri_to_I0}.
Commuting with $\mathcal{S}_c^{\E_0^\phi}$ (as defined in \cref{eq:prop:pooprp}) , we then obtain $P_{\eta,L} \mathcal{S}^{\E_0^\phi}_c\psi^{\Delta}_{N}\in \Hb^{\min({\E}^\phi_-+N\E^{L}_-),c+2}(\D^-)$.
Since this holds for arbitrary $c$ and $N$, the result follows after an application of \cref{thm:scat:long}.
\end{proof}

	\newpage
	\part{Applications and computation of explicit asymptotics}
In this part, consisting of \cref{sec:app,sec:Sch,sec:sharp,section:no incoming radiation Cauchy}, we apply our building blocks from the previous part to prove various statements.
In \cref{sec:app}, we first prove that propagation of polyhomogeneity in fact holds for general perturbations of $\Box_\eta\phi=0$, i.e.~we prove \cref{thm:intro:prop_polyhom}. We apply this to various example equations, for which we moreover explicitly compute the asymptotics, i.e.~the coefficients in the polyhomogeneous expansions.
In \cref{sec:Sch}, we deduce results on the explicit asymptotics for linearised gravity around Schwarzschild.
In \cref{sec:sharp}, we compute explicit asymptotics in the case of a long-range $1/r^2$-potential, which is, in particular relevant for wave equations in other dimensions.
Finally, in \cref{section:no incoming radiation Cauchy}, we investigate formulations of no incoming radiation on Cauchy hypersurfaces.

\section{Propagation of polyhomogeneity for perturbations of \texorpdfstring{$\Box_\eta\phi=0$}{the linear wave equation} and applications}\label{sec:app}

 In \cref{sec:en,sec:scat:scat}, we proved energy estimates and scattering results for perturbations of the form $\Box_{\eta}\phi=\mathcal{P}[\phi]+f$.
Then, in \cref{sec:prop}, we proved that solutions to $\Box_\eta \phi=f$ satisfy propagation of polyhomogeneity results (cf.~\cref{prop:prop:prop_no_incoming,prop:prop:inhomogeneous_rad}).
In the present section, by treating the $\mathcal{P}[\phi]$ term perturbatively, we extend \cref{prop:prop:prop_no_incoming,prop:prop:inhomogeneous_rad} to $\Box_{\eta}\phi=\mathcal{P}[\phi]+f$, thus proving one of the main results of the paper,  \cref{thm:intro:prop_polyhom}, namely that solutions to \cref{eq:scat:scattering_def} arising from polyhomogeneous data (+error) remain polyhomogeneous (+error).
We first prove this in the case of no incoming radiation (\cref{thm:app:general}), and then in the case of nontrivial incoming radiation (\cref{thm:app:general2}).
A similar result, namely \cref{thm:app3}, holds when long-range potentials are included.
We will continue to work in infinite regularity spaces.

Using the algorithm already explained in the introduction, we then provide three examples where we showcase how to explicitly compute the coefficients in the polyhomogeneous expansions of the relevant solutions in \cref{sec:app:app1}--\cref{sec:app:nonlinearity}.
We conclude with a proof of the so-called antipodal matching condition in \cref{sec:app:antipodal}.

\begin{rem}
    We note that, while the results of \cref{sec:app:main} all hold in any dimension (with modifications as in~\cref{rem:intro:higherd}), the results from \cref{sec:app:app1} onwards are specific to \textit{even} spacetime dimensions.
\end{rem}

	\subsection{Propagation of polyhomogeneity for solutions to general perturbations of \texorpdfstring{$\Box_{\eta}\phi=0$}{the linear wave equation}}\label{sec:app:main}
 
Let us already mention that the following theorem is further improved in \cref{cor:app:sharp_index_sets} (in even spacetime dimensions):
	\begin{thm}\label{thm:app:general}
		Let $\vec{a}$, $\vec{a}^f=(a^f_0,a^f_0,a^f_+)$ be admissible, let $\mathcal{N},V,P_g$ be short-range perturbations compatible with no incoming radiation, and let $rf\in \Hb^{a_0^f,a_+^f}(\Dbold)$.  
        Finally, let  $\psi\in\Hb^{a_0,a_+}(\Dbold)$ be the corresponding scattering solution to~\cref{eq:scat:scattering_def} with no incoming radiation and $\psi^{\incone}\in\Hb^{a_-}(\incone)$.
		
		\begin{enumerate}[label={\alph*})] 
         \item\label{item:app:error} We have $\psi\in\A{b,phg}^{\bar{a}_0,\mindex{0}}(\Dbold)+\Hb^{\bar{a}_0,\bar{a}_+}(\Dbold)$ for any $\bar{a}_0< \min(a^f_0-2,a_-)$ and any $\bar{a}_+<\min(a_+^f-1,\bar{a}_0)$. 
			\item \label{item:app:poly} Make the additional assumptions that
			\begin{equation}
				rf\in\A{phg}^{\vec{\E}^f}(\Dbold ),\quad \psi^{\incone}\in\A{phg}^{\vec{\E}^{\incone}}(\incone),\quad f_{\bullet}\in\A{phg}^{\vec{\E}^{\bullet}}(\Dbold ),\qquad\bullet\in\{V,P_g,\mathcal{N}\}.
			\end{equation}
			Furthermore, set $\vec{\E}^{\mathfrak{F},\bullet}:=n^\bullet\vec{\E}^{\phi}+\vec{\E}^\bullet$ for $\bullet\in\{V,\mathcal{N},P_g\}$, and use the notation $\vec{\E}^{\mathfrak{F},f}=\vec{\E}^f$.\footnote{Note that due to our notation, the $f_{\bullet}$ for $\bullet\in\{V,\mathcal{N},P_g\}$ stand at the same footing as $rf$.}
			Then, $\psi\in\A{phg}^{\vec{\E}^\phi}(\Dbold )$ where $\E_0^\phi,\E_+^\phi$ are minimal with the property
			\begin{equation}\label{eq:app:general_index_set}
				\begin{aligned}
					&\E^\phi_0\supset \E^{\incone},\E^{\mathfrak{F},\bullet}_0-2\\
					&
     { \E^\phi_+\supset\mindex0\cupdex\E^\phi_0, \big(0\overline{\cup}\big(\E^{\mathfrak{F},\bullet}_+-1)		)\overline{\cup}	(\E^{\mathfrak{F},\bullet}_0-2)			}\end{aligned}\qquad\bullet\in\{f,V,\mathcal{N},P_g\}.
			\end{equation}
			
		\end{enumerate}
	\end{thm}
	\begin{rem}
		We can of course also include multiple nonlinear orders in \cref{thm:app:general} (or \cref{thm:app:general2}), i.e.~we can use multiple values for $n^{P_g},n^{\mathcal{N}}$.
	\end{rem}
	\begin{proof}
		\textit{Proof of \cref{item:app:error}, Step 1) $I^0$:}
  Starting with $\psi\in\Hb^{a_0,a_+}(\Dbold)$, we apply \cref{cor:en:VN_estimates} (together with \cref{rem:en:Pgestimates}) to deduce that $r\mathcal{P}[\phi]\in\Hb^{a_0+2+\delta+,a_++1+\delta+}(\Dbold)$ for some sufficiently small $\delta$ less than the gap of $\{V,\mathcal{N},P_g\}$.
  We can thus apply \cref{prop:prop:prop_no_incoming} to \cref{eq:scat:scattering_def}, ignoring the improvement near $\scrip$, to infer that $\psi \in \Hb^{\min(a_0+\delta,\bar{a}_0),a_+}(\Dbold)$.
  We iterate the procedure of applying \cref{cor:en:VN_estimates} and \cref{prop:prop:prop_no_incoming} inductively to obtain that $\psi \in \Hb^{\min(a_0+n\delta,\bar{a}_0),a_+}(\Dbold)$ for any $n\in\mathbb N$.

  \textit{Proof of \cref{item:app:error}, Step 2) $\scrip$:}
We perform an induction over 
\begin{equation}
			r\phi\in\A{b,pgh}^{\bar{a}_0,\mindex{0}}(\Dbold)+\Hb^{\bar{a}_0,\min(a_+ +n\delta,\bar{a}_+)}(\Dbold).
		\end{equation}
		We already proved the $n=0$ base case. 
		As before, the induction step then follows from applying \cref{cor:en:VN_estimates} (with \cref{rem:en:Pgestimates}) and \cref{prop:prop:prop_no_incoming}. This proves the result.

  \textit{Proof of \cref{item:app:poly}, Step 1) $I^0$:}
  We again use \cref{cor:en:VN_estimates} and \cref{prop:prop:prop_no_incoming} and start with
  \begin{equation}
  \psi\in \Hb^{a_0,a_+}(\Dbold)\implies    r\Box\phi=r\mathcal{P}[\phi]+rf\in \Hb^{a_0+2+\delta+,a_++1+\delta+}(\Dbold)+\A{phg}^{\E^f}(\Dbold)\implies \psi \in \A{phg,b}^{\E^{\phi}_0,a_+}(\Dbold)+\Hb^{a_0+\delta,a_+}(\Dbold),
  \end{equation}
  where we used that $\E^{\phi_0}\supset \E^{\incone}\cup (\E^{f}-2)$ and again ignored the improvement towards $\scrip$ in the last step.
  We now prove inductively that $\psi \in \A{phg,b}^{\E^{\phi}_0,a_+}(\Dbold)+\Hb^{a_0+n\delta,a_+}(\Dbold)$ for any $n\in\mathbb N$:
  Having already established this for $n=0$, we prove the inductive step as follows:
\begin{nalign}
			\psi\in\A{phg,b}^{\E^\phi_0,a_+}(\Dbold)+\Hb^{a_0+n\delta,a_+}(\Dbold)
            &\implies rP_g[\phi]\phi,r\mathcal{N}[\phi],rV\phi\in\A{phg,b}^{\E_0^\phi+2,a_++1+\delta+}(\Dbold)+\Hb^{a_0+n\delta+2+\delta+,a_++1+\delta+}(\Dbold)\\
			&\implies \psi\in\A{phg,b}^{\E^\phi_0,a_+}(\Dbold)+\Hb^{a_0+(n+1)\delta,a_+}(\Dbold),
		\end{nalign}
		where we used \cref{cor:en:VN_estimates} in the first implication and \cref{prop:prop:prop_no_incoming} in the second, and we moreover used in the first step that the inclusions \cref{eq:app:general_index_set} (e.g.~$\E^{\mathcal{N}}_0+n^{\mathcal{N}}\E^{\phi}_0\subset\E^{\phi}_0+2$) hold:
		For instance, we have
		\begin{nalign}
			r\mathcal{N}[\phi]=f_{\mathcal{N}}\prod_i \Diff^1_b(\D)\psi\in \A{\phg,\b}^{\E_0^{\mathcal{N}}+n^{\mathcal{N}}\E^\phi_0,a_++1+\delta+}(\Dbold)+\Hb^{a_0+n\delta+2+\delta+,a_++1+\delta+}(\Dbold)\\
			\subset \A{\phg,\b}^{\E^\phi_0+2,a_++1+\delta+}(\Dbold)+\Hb^{a_0+n\delta+2+\delta+,a_++1+\delta+}(\Dbold).
		\end{nalign}

		\textit{Proof of \cref{item:app:poly}, Step 2) $\scrip$:} 
        We proceed similarly towards $\scrip$:
		By induction, we claim that
		\begin{equation}
			\psi\in\A{phg}^{\mathcal{E}^\phi_0,\mathcal{E}^\phi_+}(\Dbold)+\A{phg,b}^{\mathcal{E}^\phi_0,a_+ +n\delta}(\Dbold).
		\end{equation}
		We already proved the $n=0$ base case.
		For the induction step, we again use \cref{cor:en:VN_estimates}, \cref{prop:prop:prop_no_incoming} and the inclusions of the index sets given in \cref{eq:app:general_index_set}.
		Indeed, we treat each term in the perturbation $\mathcal{P}$ as an additional inhomogeneity, and include an extra term of the form $\big((\E^{\mathfrak{F},\bullet}_+-1)\cupdex0\big)\cupdex (\E^{\mathfrak{F},\bullet}_0-2)$ for each of them.		This proves the result.
	\end{proof}
	We now prove the analogous result in the case of nontrivial incoming radiation. We therefore switch back to the compactification $\D$ and remove the assumption that $\mathcal{P}$ be compatible with the no incoming radiation condition. We again note that we will improve the theorem in \cref{cor:app:improvedthmgeneral2} in the case of even spacetime dimensions.
	\begin{thm}\label{thm:app:general2}
		Let $\vec{a}$ be admissible and let $\vec{a}^f$ be extended admissible and let
     $\mathcal{P}=\{V,\mathcal{N},P_g\}$ be extended short-range perturbations as in \cref{def:scat:extended_short_range}.
        Let $\psi\in \Hb^{\vec{a}'}(\D)$ denote the scattering solution to \cref{eq:scat:scattering_def}, \cref{eq:scat:scat_with_inc_condition} with data $\psi^{\incone}\in\Hb^{a_-}(\incone)$ and $\psi^{\scrim}\in\Hb^{a_0}(\scrim)$.
  Then
		\begin{enumerate}[label={\alph*})] 
  \item\label{item:app:error2}  $\psi\in \A{phg}^{\mindex0, \mindex1, \mindex0\cupdex\mindex1}(\D)+\A{phg,b,phg}^{\mindex0,a_0-,\mindex0}(\D)+\A{phg,b,b}^{\mindex0,a_0-,a_0-}(\D)+   \A{b,b,phg}^{a_-,a_0,\mindex{0}}(\D)+\Hb^{a_-,a_0,\min(a^f_+-1,a_0)-}(\D)$.
  \item\label{item:app:poly2} If we make the additional, stronger assumptions that $f$,  the coefficients of $\mathcal{P}$ and the initial data $\psi^{\incone},\psi^{\scri}$ are polyhomogeneous, then so is $\psi$:
			Assume that
			\begin{equation}
				rf\in\A{phg}^{\vec{\E}^f}(\D),\quad \psi^{\incone}\in\A{phg}^{\vec{\E}^{\incone}}(\incone),\quad\psi^{\scrim}\in\A{phg}^{\E^{\scrim}}(\scrim),\quad f_{\bullet}\in\A{phg}^{\vec{\E}^{\bullet}}(\Dbold ),\qquad\bullet\in\{V,P_g,\mathcal{N}\}.
			\end{equation}
			Then, using the same notation as in \cref{thm:app:general} for $\vec{\E}^{\mathfrak{F},\bullet}$, we have $\psi\in\A{phg}^{\vec{\E}^{\phi}}$, where

			\begin{equation}
				\begin{aligned}
					&\E^\phi_-\supset\mindex{0},\E^{\incone},(\E_-^{\mathfrak{F},\bullet}-1),\\
					&\E^\phi_0\supset\E^\phi_-, \mindex{1}\cupdex\E^{\scrim},(\E^{\mathfrak{F},\bullet}_--1)\cupdex(\E^{\mathfrak{F},\bullet}_0-2)\\
					&{\E^\phi_+\supset {\mindex0}\cupdex\E^\phi_0,\, \big(0\cupdex(\E^{\mathfrak{F},\bullet}_+-1)\big)\cupdex(\E^{\mathfrak{F},\bullet}_0-2),}
				\end{aligned}
			\end{equation}
			where we take $\bullet\in\{f,V,\mathcal{N},P_g\}$ and with the inclusion $\E^{\phi}_-\supset\mindex{0}$ only being necessary if $\psi^{\scrim}\neq0$.
			\end{enumerate}
		
	\end{thm}
	\begin{proof}
		The proof is analogous to that of \cref{thm:app:general}, using \cref{prop:prop:inhomogeneous_rad} and \cref{cor:prop:solution_with_radiation_scri} instead of \cref{prop:prop:prop_no_incoming}.
	\end{proof}

	Consider the index set \cref{eq:app:general_index_set}: Just as for the suboptimal index set for $\Box_{\eta}\phi=f$ in \cref{eq:app:bad_index_on_minkowski}, provided a single logarithmic term appears for the future index set, this produces $\log^k$ terms for arbitrary $k$ at higher orders.
	Just as for the improvement in \cref{lemma:prop:improvement_at_future} for the wave equation on Minkowski space, we can remove some of these logarithmic terms in the case of \textit{linear} equations with a short-range potential perturbation $\Box_\eta\phi=f+V\phi$.
	
	\begin{lemma}[Improvement for linear problems]\label{lemma:app:linear_improvement}
		Under the assumptions of \cref{thm:app:general}\cref{item:app:poly}, assuming also that $\mathcal{N}=0=P_g$ and that $V$ is of the form\footnote{This requirement can also be phrased as polyhomogeneity on yet another compactification, or that both $f^V(x,-t),f^V(x,t)\in\A{phg}(\Dbold)$.}
		\begin{equation}
			V=f^V(1/r,t/r,x/r)\in\A{phg}^{\E^V}([0,1)\times(-3/2,3/2)\times S^2),
		\end{equation} 
		we can improve the future index set in \cref{thm:app:general} to $\E^\phi_+=\E^\phi_0\overline{\cup}\E^{\phi_\infty}$, where $\E^{\phi_\infty}$ is the minimal index set with the property
		\begin{equation}\label{eq:app:linear_improvement}\E^{\phi_\infty}\supset0\cupdex\big(\E^{\phi_\infty}+\E^V_+-1\big),0\cupdex(\E_+^f-1).
		\end{equation}
	\end{lemma}
	\begin{rem}
 	One application of \cref{lemma:app:linear_improvement} is to show that, for conformally smooth $f$ ,$V$, $\psi^{\incone}$, $\log^k$ terms towards $\scrip$ only appear with $k\leq 2$.
		Indeed, in this case, \cref{eq:app:linear_improvement} simplifies to\begin{equation}\label{eq:app:linear_improvement_conformal_V}
			\E^{\phi_\infty}=0\cupdex (\E^f_+-1)\big.
		\end{equation}
		A similar improvement also holds for linear $P_g$ perturbations. 
		Since, in order to understand the exact index sets, we need to resort to explicit computations anyways, we see no point in pursuing such improvements further.
\end{rem}
	\begin{proof}
    We set $\rho=r^{-1}$.
		We only consider potential perturbations with $f^V=\rho^{2+\delta}\log^k\rho f^V_\delta(t/r,\omega)$ such that $\E^V=\overline{(2+\delta,k)}$, $\delta>0$. 
  The inclusion of multiple such terms follows straightforwardly. We will simply repeat the ideas of the proof of \cref{lemma:prop:improvement_at_future}; that is, we will first commute with scaling to peel of sufficiently many orders towards $I^0$ to then infer a better index set for the scaling (minus $T$) commuted quantity along a single null cone towards $\scrip$, from which the result then follows via backwards integration.
		
\textit{Step 1): Commutations with scaling and improvement of the scaling commuted $\psi$:}
		Let's define $\mathcal{S}_c^{\E},\dot{\mathcal{S}}^{\E}_c$, for $S=u\pu+v\pv$ (as in \cref{lemma:prop:improvement_at_future}, we will later  actually want to consider $S-u_1 T$ for some $u_1$)
		\begin{equation}
\mathcal{S}^{\mathcal{E}}_c=\prod_{(z,k)\in\mathcal{E}_{\leq c}}(S+z),\qquad\dot{\mathcal{S}}^{\mathcal{E}}_c=\prod_{(z,k)\in\mathcal{E}_{\leq c}\setminus\mathcal{E}_{\leq c-\delta}}(S+z),
		\end{equation}
		where $\mathcal{S}^\emptyset=1$. 
We inductively claim that, for $c_n=\bar{a}_++n\delta$, with $\vec{\bar{a}}$ as in \cref{thm:app:general}:
\begin{equation}
\label{eq:app:future_improvement_induction:}
			\mathcal{S}^{\E_0^\phi}_{c_n}\psi\in\A{phg}^{(\E_0^{\phi})_{>{c_n}},\E^{\phi_\infty}}(\Dbold)+(\Hb^{c_n-,c_n-}(\Dbold)\cap\Hb^{\bar{a}_0,\bar{a}_+}(\Dbold)).
		\end{equation}
		This holds for $n=0$ by \cref{thm:app:general}.
		For the inductive step, we write
		\begin{nalign}
        P_\eta\mathcal{S}^{\E_0^\phi}_{c_{n+1}}\psi&=r^{-2}\mathcal{S}^{\E_0^\phi}_{c_{n+1}} r^2f^V\psi+r^{-2}\mathcal{S}^{\E_0^\phi}_{c_{n+1}}r^2 rf
       \\
       &\in \A{phg}^{2+(\E_0^{\phi})_{>c_{n+1}}, \E^{\phi_{\infty}}+\E^V}(\Dbold)+(\Hb^{c_n+2+\delta-,c_n+2+\delta-}(\Dbold)\cap\Hb^{\bar{a}_0,\bar{a}_+}(\Dbold))+\A{phg}^{2+(\E_0^{\phi})_{>c_{n+1}}, \E_+^f}(\Dbold).
		\end{nalign}
  Here, we used that $\E_0^{\phi}\supset \E_0^{\phi}+\E^V-2$, $\E^f_0-2$ as well as the induction assumption \cref{eq:app:future_improvement_induction:}.
  We may now apply \cref{prop:prop:prop_no_incoming} along with the inclusion \cref{eq:app:linear_improvement} (cf.~the proof of \cref{lemma:prop:improvement_at_future}) to deduce that \cref{eq:app:future_improvement_induction:} holds with $n$ replaced by $n+1$.

  \textit{Step 2): Improvement for $\psi$:} 
  As in \cref{lemma:prop:improvement_at_future}, we observe that \cref{eq:app:future_improvement_induction:} also holds if $S$ is replaced by $S-u_1T$. We thus deduce that $\psi|_{\outcone{1}}\in\A{phg}^{\E^\phi_0\cupdex\E^{\phi_\infty}}(\outcone{1})+\Hb^{c_n-}(\outcone{1})$. 
  Then, we integrate along the integral curves of $S-u_2T$ backward in time (for $u_2$ as in the proof of \cref{lemma:prop:improvement_at_future}), in turn implying that we can in fact restrict the future index set to satisfy $(\E^\phi_+)_{<c_n}\subset\E^\phi_0\overline{\cup}\E^{\phi_\infty}$.
		Taking the limit $n\to\infty$ proves the lemma.
  \end{proof}
  
  We conclude this subsection with the following generalisation of the previous theorems to include long-range potentials:
  \begin{thm}\label{thm:app3}
      Let $\vec{a}$ be admissible, and let $V_L$ be a long-range potential (compatible with no incoming radiation). 
      \begin{enumerate}[label={\alph*})] 
          \item  Then \cref{thm:app:general2}\cref{item:app:error2} (\cref{thm:app:general}\cref{item:app:error}) continues to hold with $\mathcal{P}[\phi]$ in \cref{eq:scat:scattering_def} replaced by $\mathcal{P}_L[\phi]=\mathcal{P}[\phi]+V_L\phi$.
\item Furthermore, if $V_L$ is assumed to be polyhomogeneous (and compatible with no incoming radiation), then \cref{thm:app:general2}\cref{item:app:poly2} (\cref{thm:app:general}\cref{item:app:poly}) continue to hold with modified index sets $\vec{\E}^{\phi}$.
        \end{enumerate}
\end{thm}
\begin{proof}
    The proof is analogous to the previous proofs, perturbing now around the results of \cref{sec:prop:longrange}.
\end{proof}		

	\subsection{Application 1: The linear wave equation on Minkowski \texorpdfstring{$\Box_\eta\phi=0$}{}}\label{sec:app:app1}

	As we have already mentioned in \cref{rem:prop:suboptimal}, the index sets provided by \cref{thm:app:general} are in general much larger than the actual index set exhausted by solutions to specific equations and arising from specific data.
	In order to understand the structure of the optimal index set in specific examples, we descend to fixed angular frequencies.
  In the remainder of this section, everything is specific to even spacetime dimensions (for explicit computations in more general settings, see \cref{sec:sharp}).
  
 As the simplest example, we first study $\Box_\eta\phi=0$ with no incoming radiation:
	\begin{lemma}[Minkowskian improvement]\label{lemma:app:minkowski_index_improvement}
		Let $\vec a$ be admissible, and let $\psi^{\incone}\in\Hb^{a_-}(\incone)$ be supported on a fixed spherical harmonic $\ell\in\mathbb N$, and let $\phi\in\Hb^{a_0,a_+}(\Dbold)$ be the scattering solution to $\Box_\eta\phi=0$ with no incoming radiation. Then, in fact,
		\begin{equation}\label{eq:app:minkowskilemma1}
			\psi\in\A{b,phg}^{a_0,\overline{(0,0)}}(\Dbold ).
		\end{equation}
	\end{lemma}	
	\begin{proof}
		Using the conservation laws $\pu(r^{-2\ell}\pv(r^2\pv)^{\ell}\psi)=0$ already mentioned in \cref{eq:intro:conslawf} for fixed $\ell$-mode solutions, we immediately have that $\psi|_{\outcone{}}\in\A{phg}^{\overline{(0,0)}}(\outcone{})$ for any $u\leq u_0$.
		We now want to use this, together with $\psi\in\Hb^{a_0,a_+}(\Dbold)$, to prove $\psi\in\A{b,phg}^{a_0,\overline{(0,0)}}(\Dbold )$. 
		
		Let us write $\Psi_0=r\phi$ and define, for $n\geq1$,
		\begin{equation}
			\Psi_n=\prod_{m=0}^{n-1} (r\pv+m)\psi.
		\end{equation}
  As before, this quantity will peel off the leading order terms in the expansion of $\psi$; in particular, we have $\Psi_n|_{\outcone{}}\in\A{phg}^{\mindex{n}}(\outcone{})$ for any $u$.
		We claim that the quantities $\Psi_n$ satisfy:
		\begin{equation}\label{eq:app:Psin_eq}
			\underbrace{\Big(\pu\pv+\frac{2n}{r}\partial_t-\frac{\Dl}{r^2}\Big)}_{:= Q_{2n+1}}\Psi_n=\Big(\pu\pv-\frac{2n}{r}\partial_r-\frac{\Dl}{r^2}+2nr^{-1}\pv\Big)\Psi_n=0
		\end{equation}
		 The case $n=0$ holds trivially. We compute the induction step as
		\begin{nalign}
			Q_{2n+1}r\pv=r\pv Q_{2n+1}-\Big(\frac{2\Dl}{r^2}-\frac{2n}{r}\partial_t-\pu\pv+\pv\pv\Big)
			=r\pv Q_{2n+1}+\Big(2\pu\pv+\frac{2n}{r}\partial_t-\frac{2\Dl}{r^2}\Big)-\frac{2}{r}\partial_tr\pv\\
			\implies  Q_{2n+3}(r\pv+n)=r\pv Q_{2n+1}+n Q_{2n+1}+\Big(2\pu\pv+\frac{2n}{r}\partial_t-\frac{2\Dl}{r^2}\Big)
			=(r\pv+n+2) Q_{2n+1}.
		\end{nalign}
  This proves \cref{eq:app:Psin_eq}. In particular, we thus have
  \begin{equation}
      \pu\pv(r^n\Psi_n)=r^n\Big(\frac{-n(n-1)+\ell(\ell+1)}{r^2}-\frac{2n}{r}\pv\Big)\Psi_n.
  \end{equation}
  We will now treat the RHS as an error term and inductively integrate the ODE twice using \cref{corr:ODE:du_dv}\cref{item:ode:u-prop}: We claim that 
  \begin{equation}\label{eq:app:inductionPsin}
      \Psi_n\in\Hb^{a_0,a_++n}(\Dbold).
      \end{equation}
      This trivially holds for $n=0$. Assuming it to hold for $n-1$, we then get
      \begin{equation}
          \pu\pv(r^n\Psi_n)=r^n\Big(\frac{-n(n-1)+\ell(\ell+1)}{r^2}-\frac{2n}{r}\pv\Big)(r\pv+n-1)\Psi_{n-1}\in\Hb^{a_0+2-n,a_++1}(\Dbold).
      \end{equation}
		Integrating this from $\outcone{}$ backwards in $u$, with initial data $\pv(r^n\Psi_n)|_{\outcone{}}\in\A{phg}^{\mindex{1}}(\outcone{})\in\Hb^{1-}(\outcone{})$,\footnote{This is the part of the proof that fails in even spatial dimensions: We need $\psi$ to be conformally regular on \textit{some} outgoing cone.} using \cref{corr:ODE:du_dv}\cref{item:ode:u-prop}\footnote{Note that \cref{corr:ODE:du_dv}\cref{item:ode:u-prop} applies to $\Hb(\D)$ spaces, but since we only want to improve the behaviour in $\D^+$, it still applies.}, we get that $\pv(r^n\Psi_n)\in \Hb^{a_0+1-n,a_++1}(\Dbold)$. 
Another integration, this time in $v$ from $\incone$, then gives $r^n\Psi_n\in \Hb^{a_0-n,a_+}(\Dbold)$, which implies \cref{eq:app:inductionPsin}. This, in turn, directly implies \cref{eq:app:minkowskilemma1}. 

	\end{proof}
	
	\begin{cor}\label{cor:app:minkowski_with_error}
		Let $a_-\geq -1/2$, $\min(\E^{\incone})>-1/2$ and let $\psi^{\incone}\in\A{phg}^{\E^{\incone}}(\incone)+\Hb^{a_-}(\incone)$. Let $\phi$ be the corresponding scattering solution to $\Box_\eta\phi=0$ with no incoming radiation. Then for any $a_+<a_0<a_-$, we have 
		\begin{equation}\label{eq:app:bad_minkowski}
			\psi\in\A{phg}^{\E^{\incone},\overline{(0,0)}}(\Dbold)+\A{b,phg}^{{a_0},\overline{(0,0)}}(\Dbold)+\Hb^{{a_0,a_+}}(\Dbold ).
		\end{equation}
	\end{cor}
	\begin{proof}
		The existence of an expansion like \cref{eq:app:bad_minkowski} with some index set $\E^\phi_+$ in place of the first $\overline{(0,0)}$ follows from \cref{thm:app:general} (by linearity).
		We improve the index set a posteriori by \cref{lemma:app:minkowski_index_improvement}.
	\end{proof}
	
	\begin{rem}\label{rem:app:Minkowski_improvement}
		It is easy to see that \cref{eq:app:bad_minkowski} cannot hold without the error term for data of \textit{finite regularity}, see already \cref{lemma:finiteregularitypeeling}.
		On the other hand, even at infinite regularity, it is unclear to the authors whether \cref{eq:app:bad_minkowski} holds without $\Hb^{a_0,a_+}(\Dbold)$ error term.
		For further details, see \cref{sec:sharp} and \cref{conj:even:peeling}.
	\end{rem}
	
	Using this Minkowskian improvement, we can actually improve both \cref{thm:app:general,thm:app:general2} near $\scrip$. (See already the proof of \cref{lem:app:potential} below to see how \cref{cor:app:sharp_index_sets} improves over \cref{thm:app:general}.)
	\begin{cor}\label{cor:app:sharp_index_sets}
		Under the assumptions of \cref{thm:app:general}\cref{item:app:poly}, we can improve the index set $\E^\phi_+$ from \cref{eq:app:general_index_set}: More precisely, we can write $\E^\phi_+=\mindex{0}\cup \E^{\phi^{(1)}}_+$ and correspondingly $\phi=\phi^{(0)}+\phi^{(1)}$ such that $r\phi^{(0)}\in\A{phg}^{\E^{\incone},\mindex{0}}(\Dbold)$, $r\phi^{(1)}\in\A{phg}^{\E_0^{\phi^{(1)}},\E_+^{\phi^{(1)}}}(\Dbold)$ and such that $\E^{\phi^{(1)}}_0$, $\E^{\phi^{(1)}}_+$ are minimal with the property 
		\begin{equation}\label{eq:app:index_set_of_iterate}
            \E^{\phi^{(1)}}_0\supset \E_0^{\mathfrak{F},\bullet}-2,\qquad
			\E^{\phi^{(1)}}_+\supset \big(0\overline{\cup}(\E^{\mathfrak{F},\bullet}_+-1)\big)\overline{\cup}(\E^{\mathfrak{F},\bullet}_0-2),\qquad \bullet\in\{f,V,P_g,\mathcal{N}\},
		\end{equation}
        where the $\E_{0,+}^{\mathfrak{F},\bullet}$ are defined as in \cref{thm:app:general}.
	\end{cor}
	\begin{proof}
		This follows from writing $\phi=\phi^{(0)}+\phi^{(1)}$, where $\phi^{(0)}$ solves $\Box_{\eta}\phi^{(0)}=0$ with data $\psi^{\incone}$. Indeed, \cref{cor:app:minkowski_with_error} applies to $\phi^{(0)}$, giving $r\phi^{(0)}\in\A{phg}^{\E^{\incone},\mindex{0}}(\Dbold)$. 
		The difference $\phi^{(1)}$ then has trivial initial data 
        and satisfies $\Box_{\eta}\phi^{(1)}=f+\mathcal{P}[\phi^{(0)}+\phi^{(1)}]$. 
        Following the same proof as for \cref{thm:app:general}, we can inductively prove $\mathcal{P}[\phi^{(0)}+\phi^{(1)}]\in\sum_{\bullet\in\{f,V,P_g,\mathcal{N}\}}\A{phg}^{\vec{\E}^{\mathfrak{F},\bullet}}(\Dbold)$ and the claimed \cref{eq:app:index_set_of_iterate}.
  We give the details for $\mathcal{N}$ and leave the rest to the reader: 
  
  We have $\Box_{\eta}\phi^{(1)}=f+\mathcal{N}[\phi^{(0)}+\phi^{(1)}]+\dots$. Doing computations analogous to \cref{eq:en:difference_of_N}, we find, for instance, that 
  \begin{equation}\label{eq:app:diffofnonls}
      \mathcal{N}[\phi^{0}+\phi^{(1)}]=\mathcal{N}[\phi^{1}]+\mathcal{N}[\phi^{(0)}]+\sum_{i=1}^{n^{\mathcal{N}}-1} \mathcal{N}_i[\phi^{(1)}],
  \end{equation}
where the $\mathcal{N}_i$ are semi-linear perturbations of order $n^{\mathcal{N}}-i$ with the corresponding $f_{\mathcal{N}_i}\in \A{phg}^{\E_0^{\mathcal{N}}+i\E^{\incone},\E_+^{\mathcal{N}}+i\mindex0}$. 
In particular, we thus have $\vec{\E}^{\mathfrak{F},\mathcal{N}_i}\subset \vec{\E}^{\mathfrak{F},\mathcal{N}}$ for all $i=1,\dots n-1$. 

The inhomogeneous term $\mathcal{N}[\phi^{(0)}]$ as well as the quasilinear and potential perturbations are dealt with similarly. 
Thus, applying \cref{thm:app:general} to $\phi^{(1)}$ (with multiple nonlinear orders) gives the result.
	\end{proof}

The improvement of \cref{thm:app:general2} is a bit more subtle:
\begin{cor}\label{cor:app:improvedthmgeneral2}
    Under the assumptions of \cref{thm:app:general2}\cref{item:app:error2}, we have the stronger inclusion:
    \begin{equation}\label{eq:app:minkowski_improvement_incoming}
        \psi\in \A{phg}^{\mindex0, \mindex1, \mindex0}(\D)+\A{phg,b,phg}^{\mindex0,a_0-,\mindex0}(\D)+\A{phg,b,b}^{\mindex0,a_0-,a_0-}(\D)+   \A{b,b,phg}^{a_-,a_0,\mindex{0}}(\D)+\Hb^{a_-,a_0,\min(a^f_+-1,a_0)-}(\D).
    \end{equation}
    In particular, there are no logarithms towards $\scrip$ below order $\min(a_0,a^f_+-1)$.
\end{cor}
\begin{proof}
 It suffices to show this improvement for \cref{cor:prop:solution_with_radiation_scri}. 
    For this, we can simply make use of \cref{cor:prop:expansiontowardsI0} and write $\psi=\psi_\Delta+\tilde\psi$, where $\psi_{\Delta}\in \A{phg,b}^{\mindex0,a_0}(\D^-)$ and $\tilde\psi$ has no incoming radiation. The result then follows by applying \cref{cor:app:minkowski_with_error} to $\tilde\psi$.

  An alternative proof uses the conservation laws directly:  We prove the stronger statement:
    Let $rh\in\Hb^{a^h_-,a_0+2}(\D^-)$ with $\supp h\subset\D^-,\, a^h_->1$ and let $\phi$ be the scattering solution to $\Box\phi=h$ with trivial data at $\scrim,\incone$.
    Then \cref{eq:app:minkowski_improvement_incoming} holds with $a^f_+=\infty$.
    (The corollary then follows after writing $\psi(u,v,\omega)=\chi\psi^{\scrim}(v,\omega)+\bar{\psi}(u,v,\omega)$, for some cutoff $\chi$ localising to $\D^-$).

    We now prove this stronger statement. Since we know that polyhomogeneity already holds, it suffices to show \cref{eq:app:minkowski_improvement_incoming} on fixed $\ell$ modes.
    Using \cref{eq:intro:conslawf} and \cref{corr:ODE:du_dv}, we already have
    \begin{equation}
        r^{-2\ell}\pv (r^2\pv)^\ell \psi_\ell\in\Hb^{a^h_--1,a_0+1+\ell,a_0+1+\ell}(\D)+\A{phg,b,b}^{\mindex0, a_0+1+\ell,a_0+1+\ell}(\D).
    \end{equation}
    In particular, on any outgoing cone, we get
    \begin{equation}
        \pv(r^2\pv)^\ell\psi_\ell|_{\outcone{}}\in\Hb^{a_0+1-\ell}(\outcone{}).
    \end{equation}
    Integrating via \cref{item:ode:1d} $\ell+1$ times, we get $\psi_{\ell}|_{\outcone{}}\in\A{phg}^{\mindex{0}}(\outcone{})+\Hb^{a_0-}(\D)$.
    Extending to all of $\D^+$ follows as in \cref{lemma:app:minkowski_index_improvement}.

\end{proof}

	\subsection{Application 2: The linear wave equation on Minkowski with a potential \texorpdfstring{$\Box_\eta\phi=V\phi$}{}}\label{sec:app:potential}
	
	The remaining examples will all have conformally smooth initial data in order to make the statements easier to read, and to better understand the creation of the "first conformally irregular" term towards $\scrip$. 
	
	\begin{lemma}\label{lem:app:potential}
		Let $V=-M/r^3$, $\mathcal{N}=P_g=f=0$ and $\E^{\incone}=\overline{(p,0)}$ for $p> -1/2$ in \cref{thm:app:general}, i.e.~let $\phi$ be the scattering solution to $\Box_\eta \phi=-M\phi/r^3$ with no incoming radiation and data $\psi^{\incone}\in\A{phg}^{\E^{\incone}}(\incone)$, where $\E^{\incone}=\overline{(p,0)}$, with $M\neq0$ and $p>-1/2$. Then we have
		\begin{equation}
			r\phi\in\A{phg}^{\E_0,\E_+}(\Dbold),
		\end{equation}
		where $\E_0=\overline{(p,0)}$, and where $\E_+=\mindex{p+1}\cupdex \mindex 0$.
  
		Furthermore, the coefficient of the first conformally irregular term in the expansion of $\E_+$, $(p+1,0)$ or $(p+1,1)$, respectively, is generically nonvanishing: In this sense, the index set is sharp. For initial data supported on $\ell>0$, however, this coefficient always vanishes.
	\end{lemma}
	\begin{rem}
		One might think that the linear wave equation on Schwarzschild as studied in \cite{kehrberger_case_2022} can be completely understood by the model problem above. However, \cref{lem:app:potential} shows that there are qualitative differences between the two. In Schwarzschild, initial data supported on angular modes $\psi_{\ell}$ always produce a conformally irregular term at order $p+1$ if $\ell\neq p$.
		In particular, the fact  that the $(1,1)$ coefficient in the potential case above always vanish if $\ell>0$ (for $p=0$)  is in stark contrast to the behaviour on Schwarzschild. 
	\end{rem}
	\begin{proof}
		Let us note that $\E_0=\mindex{p}$ follows from \cref{eq:app:general_index_set}.
  We next apply \cref{lemma:app:linear_improvement} with $\E^V=\mindex3$ to obtain that $\E^{\phi_{\infty}}$ as given in \cref{eq:app:linear_improvement_conformal_V} equals $\mindex0$. Thus, a direct application of \cref{lemma:app:linear_improvement} produces the future index set $\E_+=\mindex{p}\cupdex\mindex0$.
  At the same time, applying instead \cref{cor:app:sharp_index_sets}, we find that $\E_+$ can be further restricted to $\E_+\supset (0\cupdex(\E_++2))\cupdex\mindex{p+1}$; thus, there is no conformally irregular term at order $p$ and $\E_+=\mindex0\cupdex\mindex{p+1}$.
  We will now show that the actual first conformally irregular term appearing at order $p+1$ is generically nonvanishing.
  
	For this, we use slight modifications of methods introduced in \cite{kehrberger_case_2022}. First, we recall from the introduction, see~\cref{eq:intro:conslawf}, the following identities for fixed $\ell$-modes:
		\begin{equation}\label{eq:app:conslaw}
			\pu\pv\psi-\frac{\Dl\psi }{r^2}=\frac{M\phi}{r^2}\implies \pu(r^{-2\ell}\pv(r^2\pv)^{\ell}\psi_{\ell})=\frac{M}{r^{2+{2\ell}}}(r^2\pv)^{\ell}(\psi_{\ell}/r)=\frac{M}{r^{3+2\ell}}(r^2\pv)^{\ell}\psi_{\ell}-\frac{\ell \cdot M}{r^{2\ell+2}}(r^2\pv)^{\ell-1}\psi_\ell.
		\end{equation}
		More generally, we have
		\begin{equation}\label{eq:app:conslawN}
			\pu(r^{-2N}\pv(r^2\pv)^N\psi_{\ell})=r^{-2N-2}(N(N+1)-\ell(\ell+1))(r^2\pv)^{N}\psi_{\ell}+\frac{M}{r^{3+2N}}(r^2\pv)^{N}\psi_{\ell}-\frac{N \cdot M}{r^{2N+2}}(r^2\pv)^{N-1}\psi_\ell.
		\end{equation}
		As in \cite{kehrberger_case_2022}, \cref{eq:app:conslaw} can be used to show that, to leading order, $(r^2\pv)^{\ell}\psi_\ell\sim(r^2\pv)^{\ell}\psi_\ell|_{\incone}$.
		To make the "$\sim$" more precise, we again define $\phi^{(i)}$ for $i=0,1,2$ to be the solution to 
  \begin{equation}
      \Box_\eta\phi^{(0)}=0,\quad \Box_\eta\phi^{(1)}=V\phi^{(0)}, \quad \phi^{(2)}=\phi-\phi^{(0)}-\phi^{(1)},
  \end{equation}
  where only $\phi^{(0)}$ has nontrivial data.
		We then integrate \cref{eq:app:conslaw} in the $u$ and $v$ direction to get $(r^2\pv)^{\ell}\psi_{\ell}^{(0)}=(r^2\pv)^{\ell}\psi_{\ell}^{(0)}|_{\incone}$.
	By \cref{lemma:app:minkowski_index_improvement}, $\phi^{(0)}$ produces no conformally irregular terms. 
 We now analyse $\phi^{(1)}$, and leave it to the reader to check that $\phi^{(2)}$ produces no conformally irregular terms below order $p+2$. 
		
		For simplicity, let us assume that $\psi^{\incone}=r_0^{-p}$, where $r_0=r|_{\incone}= |u|$. In the case $\ell=0$, we can insert $\psi_{\ell=0}^{(0)}=\psi_{\ell=0}^{(0)}|_{\incone}$ back into \cref{eq:app:conslaw} to obtain that
		\begin{equation}
			\pv\psi_{\ell=0}^{(1)}= \int_{-\infty}^u \frac{Mr_0^{-p}}{r^3}\dd u';
		\end{equation}
		this integral produces a conformal irregularity towards $\scrip$ at order $r^{-2-p}$ if $p\notin\mathbb N$, and at order $r^{-2-p}\log r$ if $p\in\mathbb N_{>0}$ (see Appendix B of \cite{kehrberger_case_2024}). Thus, for $\ell=0$, the index set $\E^+$ is saturated, and it remains to prove the last sentence of \cref{lem:app:potential}:
		
		For general $\ell>0$, we do a similar computation. Integrating \cref{eq:app:conslawN} (for $\psi^{(0)}$) $N$ times from $\scrim$ along $\incone$, we find that $(r^2\pv)^N\psi^{(0)}_\ell|_{\incone}=\frac{(-1)^N p!}{(N+p)!}\frac{(\ell+N)!}{(\ell-N)!}r_0^{N-p}$.
		Thus, writing 
		\begin{equation}
			(r^2\pv)^{\ell}\psi^{(0)}_{\ell}= (r^2\pv)^{\ell}\psi^{(0)}_{\ell}|_{\incone}=C_{\ell}r_0^{\ell-p},
		\end{equation}
		for some $C_{\ell}\neq0$, we find that, after integrating once from $\incone$:
		\begin{equation}
			(r^2\pv)^{\ell-1}\psi^{(0)}_{\ell}=\left(\frac{1}{r_0}-\frac{1}{r}\right)C_{\ell} r_0^{\ell-p}-\frac{\ell+p}{2\ell}C_{\ell}r_0^{\ell-p-1}.
		\end{equation}
		Now, we insert this back into the RHS of \cref{eq:app:conslaw}:
		\begin{nalign}\label{eq:app:potential_cancellation}
			\frac{1}{MC_{\ell}}\pu(r^{-2\ell}\pv(r^2\pv)^{\ell}\psi^{(1)}_{\ell})&= \frac{r_0^{\ell-p}}{r^{2\ell+3}}-\frac{\ell }{r^{2\ell+2}}\left(\frac{1}{r_0}-\frac1r\right)r_0^{\ell-p}+\frac{(\ell+p)}{2}\frac{r_0^{\ell-p-1}}{r^{2\ell+2}}\\
			&=\frac{r_0^{\ell-p}}{r^{2\ell+3}}(1+\ell)-\frac{r_0^{\ell-p-1}}{r^{2\ell+2}}\frac{\ell-p}{2}=\pu\left(\frac{r_0^{\ell-p}}{2 r^{2\ell+2}}\right).
		\end{nalign}
        Integrating along the $u$ direction, we see that the RHS of \cref{eq:app:potential_cancellation} has surprisingly fast ($\rho_+^{2l+2}$) decay towards $\scrip$.
		Thus, after $l+1$ integration in the $v$ direction, we see that for $\ell>0$ we never have a conformally irregular term at leading order, so $\psi^{(1)}$ is also conformally regular.
	\end{proof}

	\subsection{Application 3: The semilinear wave equation on Minkowski \texorpdfstring{$\Box_\eta\phi=\pu\phi\pv\phi$}{}}\label{sec:app:nonlinearity}
	
	Let us now turn our attention to a specific choice of nonlinearity, purely for showcasing the method.
	
	\begin{lemma}\label{lem:app:semi}
		a) Let $V,f=0$, $\mathcal{N}[\phi]=\pu \phi\pv\phi$ and $\psi^{\incone}\in\A{phg}^{\mindex{p}}(\incone)$ for $p> -1/2$. Let $r\phi$ be the corresponding scattering solution with no incoming radiation. Then $r\phi\in\A{phg}^{\E_0,\E_+}(\Dbold)$ for
		\begin{equation}\label{eq:app:semilinear_index_set}
			(\E_+)_{\leq 2p+1}=\mindex{0}_{\leq2p+1}\cup \begin{cases}
				\{(2p+1,0)\},\qquad 2p+1\notin \N\\
				\{(2p+1,1)\}, \qquad 2p+1\in\N.
			\end{cases}
		\end{equation}
		
		b) Moreover, the index set \cref{eq:app:semilinear_index_set} is saturated for $p\notin\{0,1\}$. 
	\end{lemma}
	\begin{rem}
		Note that we could also consider the case where $\Box_{\eta}\phi=-M\phi/r^3+\pu\phi\pv\phi$. Redoing the proofs of \cref{lem:app:potential} and \cref{lem:app:semi} then shows that if $p<0$, then the nonlinearity is responsible for the first conformally irregular term, whereas if $p>0$, then the potential dominates and leads for the first conformally irregular term towards $\scrip$.
	\end{rem}
 \begin{rem}
     The cases $p=0,1$ are exceptional in that we show that for data supported on low $\ell$-modes, the corresponding conformally irregular terms vanish. The proof can likely be extended to show that this holds for data supported on all $\ell$-modes, but we did not further pursue this. In any case, this shows that the precise analysis of the asymptotic expansions towards $\scrip$ is, in general, a quite subtle issue.
 \end{rem}
	\begin{proof}
		\textit{a)} Let us first note that $\E^{\mathcal{N}}_0=\mindex{3},\E^{\mathcal{N}}_+=\mindex{2}$.
		From \cref{eq:app:general_index_set}, it follows that $\E_0=\mindex{p}$.
		For $\E_+$, we use \cref{cor:app:sharp_index_sets} to get that $\E_+$ is minimal with
		\begin{equation}
			\E_+\supset \big(0\cupdex(2\E_++2-1)\big)\cupdex(2\E_0+3-2)=\big(0\cupdex(2\E_++1)\big)\cupdex\mindex{2p+1}.
		\end{equation}
		The first conformally irregular  thus lives at order $2p+1$.
		
\textit{b)} We now prove that the index set is saturated for $p\notin\{0,1\}$: We write $\phi=\phi^{(0)}+\phi^{(1)}+\phi^{(2)}$, where $\phi^{(i)}$ all have no incoming radiation and only  $\phi^{(0)}$ has nontrivial data $r\phi^{(0)}|_{\incone}=\psi^{\incone}$:
		\begin{equation}
			\Box\phi^{(0)}=0,\quad \Box\phi^{(1)}=\mathcal{N}[\phi^{(0)}],\quad\Box\phi^{(2)}=\mathcal{N}[\phi]-\mathcal{N}[\phi^{(0)}].
		\end{equation}
		A similar analysis as in part \textit{a)} shows that the index set towards $\scrip$ for $\phi^{(0)},\phi^{(2)}$ does not contain conformally irregular terms below order $2p+2$, i.e.~no terms $(2p+1,1)$ for $2p+1\in\N$ or $(2p+1,0)$ when $2p+1\notin\N$.
		Therefore, it suffices to compute the first conformally irregular term in the expansion of $\phi^{(1)}$.
		For $p\notin\{0,1\}$, we take $\psi^{\incone}=r^{-p}$ supported  on $\ell=0$. We then have $r\phi^{(0)}=\abs{u}^{-p}$, and $r\mathcal{N}[\phi^{(0)}]=r^{-3}\abs{u}^{-2p-1}(pr-\abs{u})$.
		Integrating now $\pu\pv\psi^{(1)}=r\mathcal{N}[\phi^{(0)}]$, we see that both the $pr$ and $\abs{u}$ term gives leading order contribution leading to $2p+1$ behaviour towards $\scrip$, however their leading terms exactly cancel: Indeed, we have
		\begin{equation}\label{eq:app:simple}
			pr^{-2}|u|^{-2p-1}=\pu(\tfrac12 r^{-2}|u|^{-2p})+r^{-3}|u|^{-2p} \implies r\mathcal{N}[\phi^{(0)}]=\frac12\pu(r^{-2}|u|^{-2p}).
		\end{equation}
		Thus, initial data supported on $\ell=0$ do not lead to conformally irregular terms for $\phi^{(1)}$. Alternatively, this also follows from the observation that when restricted to $l=0$ modes, the nonlinearity is the same as for $\Box\phi=-\partial\phi\cdot\partial\phi$, which gives conformal smoothness (since it can be transformed to the standard linear wave equation).
		
		By the same reasoning,  proving sharpness of the index set for $\pu\phi\pv\phi$ nonlinearity is equivalent to treating the nonlinearity $r^{-2}|\sl\phi|^2$.
		Let's therefore work with this nonlinearity instead, and assume that the initial data are supported on $\ell=1$. We will show that the $\ell=0$ projection (denoted by  $P_{S^2}^{(0)}$) of $\phi^{(1)}$ then saturates the index set:
		If the initial data are supported on $\ell=1$, then $\psi^{(0)}=Y_1 r^2\pu(\abs{u}^{-p+1}r^{-2})=r^{-1}\abs{u}^{-p}((p-1)r+2\abs{u})$, so that 
		\begin{equation}\label{eq:app:leading_forcing}
			-\pu\pv P_{S^2}^{(0)}\psi^{(1)}=	P_{S^2}^{(0)}rr^{-2}\abs{\sl\phi^{(0)}}^2=P_{S^2}^{(0)}r^{-3}\abs{\sl\psi^{(0)}}^2=r^{-5}\abs{u}^{-2p}((p-1)r+2\abs{u})^2.
		\end{equation}
		Simple integration by parts as in \cref{eq:app:simple} then give that, assuming that $p\notin \{1/2,1\}$:
		\begin{multline}\label{eq:app:proof_nonlin_int}
			\int_{-\infty}^{u}r^{-5}\abs{u}^{-2p}((p-1)r+2\abs{u})^2 \dd u'=\int_{-\infty}^{u}4\cdot\frac{\abs{u}^{2-2p}}{r^5}+4(p-1)\frac{\abs{u}^{1-2p}}{r^4}+(p-1)^2\frac{\abs{u}^{-2p}}{r^3} \dd u'\\
			=|u|^{-2p-2}f(|u|/r)+\int_{-\infty}^{u}\frac{\abs{u}^{-2p+2}}{r^5}\underbrace{\Big(4-4(p-1)\frac{4}{2p-2}+(p-1)^2\frac{12}{(2p-2)(2p-1)}\Big)}_{\neq 0 \text{ if } p\neq -1}\dd u',
		\end{multline}
		where $f(|u|/r)=\O(|u|^{3}/r^3)$ is conformally smooth near $\scrip$ (smooth in $|u|/r$), and the integral gives a conformally irregular term at order $(2p+2,0)$ if $p\notin\mathbb N$, and a conformally irregular term at order $(2p+2,1)$ if $p\notin\mathbb{N}_{\leq 1}$. (See Appendix B of \cite{kehrberger_case_2024} for the computation of these integrals.)
		
		In the case where $p=1/2$, one similarly checks that the first two terms on the RHS of the first line of \cref{eq:app:proof_nonlin_int} give conformally smooth contributions, whereas the third term gives a $\log (r/|u|)/r^2$-contribution.
		Integrating these results for $\pv\psi^{(1)}$ in $v$ gives the result.
		
		Notice that for $p=1$, then \cref{eq:app:leading_forcing} is simply an integral over $1/r^5$, i.e.~conformally smooth. If $p=0$, then the integral also remains conformally smooth. 
		Thus, our ansatz of initial data supported on $\ell=1$ does not produce conformal irregularities at order $p+1$ for $p=0,1$. We conjecture that the same is true for data supported on any $\ell$, but we do not show this here.

	\end{proof}
 \subsection{Application 4: The antipodal matching condition for perturbations of \texorpdfstring{$\Box_\eta\phi=0$}{the Minkowskian wave equation}}\label{sec:app:antipodal}
Given a function $\psi$ on $\D$, it is said to satisfy the antipodal matching condition if the limits below exist and
\begin{equation}\label{eq:app:antipodal}
  P^{(\ell)}_{S^2}  \lim_{v\to\infty}\lim_{u\to-\infty}\psi=(-1)^{\ell}  P^{(\ell)}_{S^2}\lim_{u\to-\infty}\lim_{v\to\infty}\psi.
\end{equation}
We will first show that this condition is satisfied by solutions to $P_\eta\psi=0$, and then deduce that it is also satisfied by solutions to \cref{eq:scat:scattering_def}.
\begin{lemma}\label{lemma:app:antipodal}
   Let $\phi$ be the scattering solution to $\Box_\eta\phi=0$ arising from data $\psi^{\incone}\in\A{phg}^{\mindex0}(\incone)+ \Hb^{0+}(\incone)$ and $v\pv\psi^{\scrim}\in \Hb^{0+}(\scrim)$. 
   Then $\psi$ satisfies \cref{eq:app:antipodal}.
\end{lemma}
\begin{proof}
    We write $\psi=\psi^{(0)}+\psi^{(1)}+\psi^{(2)}$, where each $\psi^{(i)}$ solves $P_\eta\psi^{(i)}=0$ such that $\psi^{(0)}$ has data $\psi^{(0),\incone}=\lim_{u\to-\infty}\psi^{\incone}$, $\pv\psi^{(0),\scrim}=0$, $\psi^{(1)}$ has data $\psi^{(1),\incone}=\psi^{\incone}-\lim_{u\to-\infty}\psi^{\incone}\in\Hb^{0+}(\incone)$ and no incoming radiation, and $\psi^{(2)}$ has trivial data on $\incone$ and incoming radiation $\pv\psi^{(2),\scrim}=\pv\psi^{\scrim}$.

    For $\psi^{(1)}$, we may directly apply \cref{prop:prop:prop_no_incoming} to infer that both limits in \cref{eq:app:antipodal} vanish.

    For $\psi^{(2)}$, we are in the setting of \cref{cor:prop:solution_with_radiation_scri} with $\E^{\scrim}=\mindex0$ and $a_0>0$.
    (Note that error term has no contribution to either of the limits in \cref{eq:app:antipodal}.)
    In particular, the limits in \eqref{eq:app:antipodal} are well-defined and in fact commute with the projection operator; we can thus explicitly compute them by decomposing into $\ell$-modes and integrating \cref{eq:intro:conslaw}. 
    This gives the explicit formulae\footnote{It is easiest to check this by directly confirming that $\pu\pv\psi=-\ell(\ell+1)\psi/r^2$ holds.}
    \begin{equation}
        \psi^{(2)}_{\ell}=\sum_{n=0}^{\ell}\frac{(-1)^n}{n!r^n}\frac{(\ell+n)!}{(\ell-n)!}\underbrace{\int_{v_0}^v\cdots\int_{v_0}^{v^{(n)}} \pv\psi^{\scrim}(v^{(n)},\omega) \dd v^{(n)}\cdots \dd v^{(0)}}_{n+1 \text{ integrals}}.
    \end{equation}
Now, we note that
\begin{equation}
    \int_{v_0}^v\cdots\int_{v_0}^{v^{(n)}} \pv\psi^{\scrim}(v^{(n)},\omega) \dd v^{(n)}\cdots \dd v^{(0)}=\int_{v_0}^{\infty}\pv\psi^{\scrim}(v,\omega)\dd v\cdot\frac{v^n}{n!}+\Hb^{-n+}(\scrim).
\end{equation}
Using that $\sum_{n=0}^{\ell}\frac{(-1)^n}{(n!)^2}\frac{(\ell+n)!}{(\ell-n)!}$ is the zeroth coefficient of $(1-x)^{\ell}/x^{\ell}$, it follows that
\begin{align}
\lim_{v\to\infty}\psi^{(2)}_{\ell}=\sum_{n=0}^{\ell}\frac{(-1)^n}{(n!)^2}\frac{(\ell+n)!}{(\ell-n)!}\int_{v_0}^{\infty}\pv\psi^{\scrim}\dd v=(-1)^{\ell}\int_{v_0}^{\infty}\pv\psi^{\scrim}\dd v;
\end{align}
in particular, $\psi^{(2)}$ satisfies \cref{eq:app:antipodal}.

Finally, for $\psi^{(0)}$, we are in the setting of \cref{prop:prop:prop_no_incoming}, and it again suffices to prove the result for fixed $\ell$-modes. This has been done already in Equation (12.4) of \cite{kehrberger_case_2024}; see also \cref{rem:even:antipodal} in the present paper for a proof.
\end{proof}
\begin{cor}
  For $\delta>0$, let $\vec{a}=(-\delta/8,-\delta/4,-\delta/2)$ admissible, and let $\vec{a}^f=\vec{a}+(1,2,1)$ be a corresponding inhomogeneous weight. Let $rf\in \Hb^{\vec{a}^f+\vec{\delta}}(\D)$, and let $\mathcal{P}=\{V,\mathcal{N},P_g\}$ be extended short-range perturbations with gap $\geq\delta$.
  Finally, let $\psi\in \Hb^{\vec{a}}(\D)$  denote the scattering solution to \cref{eq:scat:scattering_def}, \cref{eq:scat:scat_with_inc_condition} with data $\psi^{\incone}\in\A{phg}^{\mindex0}(\incone)+\Hb^{0+}(\incone)$ and $v\pv\psi^{\scrim}\in\Hb^{0+}(\scrim)$.
  Then $\psi$ satisfies \cref{eq:app:antipodal}.
\end{cor}
\begin{proof}
    As in the proof of \cref{cor:app:sharp_index_sets}, we write $\phi=\phi^{(0)}+\phi^{(1)}$, where $\phi^{(0)}$ solves $\Box_\eta\phi^{(0)}=0$ with data $\psi^{(0),\incone}=\psi^{\incone}$ and $v\pv\psi^{(0),\scrim}=v\pv\psi^{\scrim}\in\Hb^{0+}(\scrim)$. By \cref{lemma:app:antipodal}, $\psi^{(0)}$ satisfies \cref{eq:app:antipodal}.
    We then note that the difference satisfies $\Box_\eta\phi^{(1)}=f+\mathcal{P}[\phi^{(0)}+\phi^{(1)}]$. 
    Decomposing the nonlinearity as in \cref{eq:app:diffofnonls}, we may now apply \cref{thm:app:general2} to deduce that both limits in \cref{eq:app:antipodal} vanish for $\psi^{(1)}$. This completes the proof.

\end{proof}
	
	\newpage

	\section{Wave equations on Schwarzschild and the summing of the \texorpdfstring{$\ell$}{l}-modes}\label{sec:Sch}

In this section, we discuss another application of \cref{thm:app:general}, namely to linearised gravity on Schwarzschild. 
More specifically, we will now use \cref{thm:app:general}, in conjunction with the ODE lemmata of \cref{sec:ODE_lemmas} to  upgrade the theorems obtained in the previous papers of the series \cite{kehrberger_case_2022-1,kehrberger_case_2024}, which were obtained for \textit{fixed} angular mode solutions, to solutions supported on \textit{all} angular frequencies (cf.~\cref{thm:intro:IIIsummed} from the introduction).
	 We will first do this in the case of the linear wave equation on Schwarzschild (\cref{sec:Sch:wave}), and then in the case of linearised gravity around Schwarzschild (\cref{sec:Sch:ling}).
Furthermore, we will briefly discuss a more suitable choice of coordinates in which we can construct optimal index sets in \cref{sec:Schw:bettercoords}.

	\subsection{The linear wave equation on Schwarzschild}\label{sec:Sch:wave}

    We begin by introducing the Schwarzschild manifold:
    Recall first the definition of the set $\Dopen=\{u\in(u_\infty,u_0],v\in[v_0,v_\infty)\}$ for some $u_0<0<v_0$ from \cref{eq:notation:regions}.
    We can equip $\Dopen$ with a Lorentzian metric $g$---called the Schwarzschild metric---as follows:
    \begin{equation}
        g=-4D(r)\dd u\dd v+r_M(r_\star)^2g_{S^2},\qquad D(r)=1-\frac{2M}{r_M},
    \end{equation}
    where $r_\star:=u+v$, and where $r_M(r_{\star})$ is implicitly defined via $r_\star=r_M+2M\log\Big(\frac{r_M-2M}{M}\Big)$---notice that $r_{\ast}$ is precisely the Minkowskian $r$, and $r_M$ is only equal to the Minkowskian $r$ up logarithmic corrections.
    We will from now on suppress the ${}_M$ subscript in $r_M$.
  For later reference, let us already define the functions $r_0(u):=r(u,1)$, and $D_0(u):=D(r_0)$.
 %
Next, we define the compactification $\D$ with respect to $u,v,r_\star$ coordinates (as opposed to $u,v,r$) as in \cref{sec:notation:geometric} (cf.~\cref{eq:notation:defining_functions}); that is, we define it w.r.t~$\rho_-=v/r_{\ast}$ etc., similarly for $\Dbold$.

We now regard $g$ as a perturbation around $\eta$ in coordinates $(u,v,\omega)$ by writing $g=\eta+8M/r \dd u \dd v+ (r_M^2-r_\ast^2) \dd \omega$.
Then, if we we consider the Schwarzschild wave equation 
\begin{equation}\label{eq:Schw:linear_on_Schw}
\Box_g\phi=0 \iff \pu\pv\psi=\frac{D(r) \Dl\psi}{r^2}-\frac{2MD(r)\psi}{r^3}
\end{equation} and write it as $\Box_\eta\phi=(\Box_g-\Box_\eta)\phi$, we see that $g$ in the coordinates $(u,v,\omega)$ is a short range perturbation of $\Box_{\eta}$ on $\D$ by \cref{rem:en:examples}.

	We now recall from  \cite{kehrberger_case_2022-1,kehrberger_case_2024} the following fixed $\ell$-mode results  (cf.~Theorem 10.1 of \cite{kehrberger_case_2024}, with $s=0$):
	\begin{thm}[\cite{kehrberger_case_2022-1,kehrberger_case_2024}]\label{thm:Schw:wave}
		Let $p>-1/2$, let $C(\omega)\in C^{\infty}(S^2)$, and let $\phi$ be the scattering solution to \cref{eq:Schw:linear_on_Schw} with no incoming radiation from $\scrim$ (cf.~\cref{eq:scat:data_no_incoming}) with $\psi^{\incone}-\frac{C(\omega)}{r_0^p}\in \Hb^{p+\epsilon}(\incone)$ for some $\epsilon>0$.
		
		If $p\in\mathbb N$, then projections to fixed angular modes $\phi_\ell\cdot Y_{\ell}$ (we suppress the azimuthal number $m$ in the following) have the following expansion in $\D$ for any $N>p$:
		\begin{nalign}\label{eq:Schw:exp:pinN}
			\psi_{\ell}=\left(\sum_{n=0}^{N}\big(f_n^{(\ell,p)}(r_0)+g_n^{(\ell,p,\epsilon)}(r_0)\big)\frac{r_0^n}{r^n}\right)+M c_{\log}^{(\ell,p)}\frac{\log (r/r_0)}{r^{p+1}}+\underbrace{\O\left(\frac{r_0^{N+1-p}}{r^{N+1}}+\frac{Mr_0\log (r/r_0)}{r^{p+2}}+\frac{1}{r^{p+1+\epsilon-}}\right)}_{\in \A{b,phg}^{p-,\overline{(N+1,0)}}(\Dbold)+\A{b,phg}^{p+1-,\overline{(p+2,1)}}(\Dbold)+\Hb^{p+1+\epsilon-,p+1+\epsilon-}(\Dbold)}
		\end{nalign}
		for some constants $c_{\log}^{(\ell,p)}\in\mathbb R.$\footnote{Here, and below, the notation $\mathcal{O}$ simply means that the remainder is bounded in an $L^\infty$ sense by the function in the brackets.
        We included these for the convenience of the reader familiar with previous works on the subject.}
		If $\ell\geq p-1$, then the coefficients $c_{\log}^{(\ell,p)}$ are given explicitly by $c_{\log}^{(\ell,p)}= (-1)^{\ell+p} C_{\ell}\frac{(\ell+p+1)(\ell-p)}{p+1}$.
		
		Similarly, if $p\notin\mathbb N$, then we have for any $N>\lceil p\rceil$:
		\begin{nalign}\label{eq:Schw:exp:pnotinN}
			\psi_{\ell}=\left(\sum_{n=0}^{N}\big(f_n^{(\ell,p)}(r_0)+g_n^{(\ell,p,\epsilon)}(r_0)\big)\frac{r_0^{n}}{r^n}\right)+
			c^{(\ell,p)}\frac{M} {r^{p+1}}
			+\underbrace{\O\left(\frac{r_0^{N+1 -p}}{r^{N+1}}+\frac{Mr_0}{r^{p+2}}+ \frac{1}{r^{p+1+\epsilon-}}\right)}_{ \in \A{b,phg}^{p-,\overline{(N+1,0)}}(\Dbold)+\A{b,phg}^{p+1-,\overline{(p+2,0)}}(\Dbold)+\Hb^{p+1+\epsilon-,p+1+\epsilon-}(\Dbold)}
		\end{nalign}
		for some constants $c^{(\ell,p)}\in\mathbb R$. If $\ell\geq\lceil p-1\rceil$,  then $c^{(\ell,p)}=(-1)^{\ell+1}C_\ell\left((\ell+p+1)(\ell-p)p!(-p-2)!\right)$.
		In both cases, the functions $f_n^{(\ell,p,\epsilon)}$, $g_n^{(\ell,p,\epsilon)}$ are given by
		\begin{equation}
			f_n^{(\ell,p,\epsilon)}(r_0)=C_\ell S_{\ell,p,\ell,n,0}r_0^{-p}+M\cdot C_\ell S'_{\ell,p,n}r_0^{-p-1}+\O(M^2\cdot r_0^{-p-2})\in\A{phg}^{\mindex{p}}(\incone),\qquad g_n^{(\ell,p,\epsilon)}\in\Hb^{p+\epsilon}(\incone)
		\end{equation}
		where the $S_{\ell,p,\ell,n,0}$, given by \cite[(10.64)]{kehrberger_case_2024} and vanish for $n\geq \ell$, and the $S'_{\ell,p,n}$ are coefficients that we do not explicitly provide here (see  \cite[\S14]{kehrberger_case_2024}). 

  Similar expressions hold for initial data of the type $\psi^{\incone}=C(\omega) \frac{\log^q r}{r^p}$.
	\end{thm}
 \begin{rem}[Comparison to \cite{kehrberger_case_2024}]
     \cref{thm:Schw:wave} includes three points that are not stated explicitly in \cite[Theorem~10.1]{kehrberger_case_2024}: 
     Firstly, that the sum in \cref{eq:Schw:exp:pinN} can be taken until $N$ arbitrarily large. Secondly, that regardless of how large $\epsilon$ is taken, there will only be $\log^k$-terms with $k\leq 1$ in \cref{eq:Schw:exp:pinN}. While the first point is trivial, the second was not explicitly stated or proved in \cite{kehrberger_case_2024} and requires a more detailed understanding of the proof itself: It relies on the fact that when plugging in the leading order asymptotics in order to obtain refined asymptotics as described in \cite[ \S10.3]{kehrberger_case_2024}, one will at worst encounter integrals of the form $\int_{-\infty}^u\frac{\log r/r_0}{r^N}$, which, for $N\geq 1$, do not produce further log-terms.
We don't give further details here because we will give an independent and self-contained proof \cref{sec:Schw:bettercoords}.
Finally, \cref{thm:Schw:wave} states that the expressions for $c^{(\ell,p)}$ are valid for $\ell\geq p-1$ rather than for $\ell-p\geq 0$ as in \cite{kehrberger_case_2024}. This extension is contained in  \cite[\S9]{kehrberger_case_2022-1} (the extension to the general case could be dealt with via time inversion as in \cite{kehrberger_case_2022}).
 \end{rem}
	\begin{rem}[Lossy error term towards $\scrim$]\label{rem:Schw:lossy}
It may at first sight be confusing that the error \cref{eq:Schw:exp:pinN}, \cref{eq:Schw:exp:pnotinN} above decays like $r^{-p+}$ towards $\scrim$ despite the data along $\incone$ having $r^{-p-\epsilon}$-error.
This apparent loss can, in fact, only be remedied by letting the sum in e.g.~\cref{eq:Schw:exp:pinN} go at least until $N$ for $N\geq\ell$ (It doesn't come from the error term in the data, but rather from the $r_0^{-p}$-term.)
\end{rem}
\begin{rem}[Higher-order asymptotics]\label{rem:Schw:higher}
		In the same context, we note that, in order to compute higher order asymptotic expansions, it is necessary to fully resolve the semi-global asymptotics of fixed $\ell$-modes (i.e.~to always consider $N\geq\ell$) and to then plug those into the higher-order iteration scheme. In particular, \cref{eq:Schw:exp:pinN} with $N<\ell$ \textit{cannot} be used to compute the precise asymptotics of the next iterate; in fact, already computing the $\log$-term in \cref{eq:Schw:exp:pinN} essentially requires resolving the sum completely ($N\geq\ell$).
  This may be compared to \cref{section:no incoming radiation Cauchy}, where we consider a different perspective and understand the creation of these terms via certain Fourier multipliers.
	\end{rem}	
	\begin{proof}[Sketch of the proof]
		In order to make this section more self-contained, we provide a short sketch of how to prove \cref{thm:Schw:wave} that aligns with the general spirit of this paper. The existence of the scattering solution of course follows from \cref{thm:scat:scat_general}, so we only focus on the derivation of the asymptotic expansions \cref{eq:Schw:exp:pinN,eq:Schw:exp:pnotinN}:
        
		The idea is to expand the Schwarzschild wave equation around its Minkowskian counter-part, in a similar fashion to how we treated \cref{lem:app:potential}.
  A good choice of coordinates to do this is $(r_0,r)$ as defined above, since they capture the logarithmic divergence of the light cones in Schwarzschild compared to Minkowski. 
		In these coordinates, the wave equation reads
		\begin{nalign}\label{eq:Schw:waveequationinr0r}
			-(D_0 \partial_{r_0}+D \partial_r)(D\partial_r(r\phi))=\frac{D\Dl(r\phi)}{r^2}-\frac{2MD(r\phi)}{r^3}
		\end{nalign}
The iteration takes place in powers of the decay of the iterates towards $I^0$. Since $M$ is the only length-scale in this problem, we can write this iteration in terms of powers of $M$ (though this has nothing to do with $M$ being a small parameter, it is simply convenient notation).

		In a first step, we thus set $M=0$ in this equation and compute the Minkowskian solution $\phi_{M^0}$ with data $r\phi_{M^0}|_{\incone}=C(\omega)r_0^{-p}$; we interpret this solution to be of order $\O(M^0)$. This gives us precisely the $S_{\ell,p,\ell,n,0}$ terms.
		
		In a second step, we insert $\phi_{M^0}+ \phi_{M^1}$ into \cref{eq:Schw:waveequationinr0r}, regard $\phi_{M^1}$ to be of order $M^1$ and only keep terms of order $M^1$. The result is a Minkowskian wave equation for $\phi_{M^1}$ with inhomogeneity sourced by $\phi_{M^0}$:
		\begin{equation}\label{eq:pm:M}
			-(\partial_{r_0}+\partial_{r})(\partial_{r}(r\phi_{M^1}))-\frac{1}{r^2}\Dl (r\phi_{M^1})=2M
			\left(-\left(\frac{1}{r_0}+\frac{1}{r}\right)\partial_{r_0}\partial_{r}
			-\frac{2}{r}\partial_{r}^2+\frac{1}{r^{2}}\partial_{r}
			-\frac{1}{r^{3}}\Dl-\frac{1}{r^{3}}
			\right)(r\phi_{M^0})
		\end{equation}
		We pose trivial data for this equation. Then, $\phi_{M^1}$ gives us precisely those terms in \cref{eq:Schw:exp:pinN}, \cref{eq:Schw:exp:pnotinN} that are multiplied by $M$; in particular, it contains the conformal irregularity. 
		We conclude by proving an error term estimate for the equation satisfied by $\phi_{\Delta}=\phi-\phi_{M^0}-\phi_{M^1}$, the Schwarzschildean wave equation with inhomogeneity sourced by $\phi_{M^0}$ and $\phi_{M^1}$, with data $r\phi_{\Delta}|_{\incone}=\O(r^{-p-\epsilon})$.
		
		Note that we could carry on with the procedure of expanding into powers of $M$ to higher orders to obtain more precise asymptotics, however, this will produce no higher order $\log^k$-terms with $k>1$!
	\end{proof}
    
Now, the goal of this subsection is to upgrade this fixed $\ell$-mode result to a statement on the full solution. To this end, we note that our \cref{thm:app:general} immediately implies the following corollary:
	\begin{cor}\label{cor:Schw:phg}
		1) Let $\phi_{\mathrm{phg}}$ be the scattering solution to \cref{eq:Schw:linear_on_Schw} with no incoming radiation and with $\psi_{\phg}|_{\incone}=\psi_{\mathrm{phg}}^{\incone}\in \A{phg}^{\mathcal{E}^{\incone}}(\incone)$, with  $\min(\E^{\incone})>-1/2$.
		Then $\psi_{\mathrm{phg}} \in \A{phg,phg}^{\mathcal{E}_0^{\phi},\mathcal{E}^{\phi}_+}(\Dbold )$
   for some index sets $\mathcal{E}^{\phi}_0,\mathcal{E}^{\phi}_+$.
		
		2) Let $\phi_{\Delta}$ be the scattering solution to the linear wave equation on Schwarzschild with no incoming radiation and with $\psi_{\Delta}^{\incone}\in \Hb^{p+\epsilon}(\incone)$, with $p+\epsilon\geq-1/2$.
		Then $\psi_{\Delta}\in \A{b,phg}^{p+\epsilon-, \overline{(0,0)}}(\Dbold )+ \Hb^{p+\epsilon-,p+\epsilon-}(\Dbold)$.
	\end{cor}
 \begin{rem}
     As before, the assumptions on $p+\epsilon$, $\min(\E^{\incone})$ can be removed since we are working in no incoming radiation spaces.
 \end{rem}
	\begin{proof}
		1) This follows from \cref{thm:app:general} \cref{item:app:poly}, since $\Box_g$ is a short range perturbation of $\Box_\eta$.
	
		2) We use \cref{thm:app:general} \cref{item:app:error} instead.
	\end{proof}
	To conclude: Theorem~\ref{thm:Schw:wave} tells us the coefficients of the asymptotic expansions of fixed angular mode solutions, whereas Corollary~\ref{cor:Schw:phg} tells us that the full solution has a polyhomogeneous expansion towards $\scrip$ up to error term. Together, the two results give the following:
	\begin{thm}\label{thm:Schw:wave:summed}
  Under the assumptions of \cref{thm:Schw:wave}, if $p\in\mathbb N$, then $\phi$ satisfies throughout $\tilde{\D}$ for any $N>p$:
		\begin{equation}\label{eq:Schw:expansionfull}
			\psi=\sum_{n=0}^{N}f_n^{(p,\epsilon)}(r_0,\omega)\frac{r_0^n}{r^n}+c_{\log}^{(p)}(\omega)\frac{M \log r/r_0}{r^{p+1}}+\underbrace{{\O}\left(\frac{r_0^{N+1-p}}{r^{N+1}}+\frac{Mr_0\log (r/r_0)}{r^{p+2}}+\frac{1}{r^{p+\epsilon-}}\right)}_{\in \A{b,phg}^{p-,\overline{(N+1,0)}}(\Dbold)+\A{b,phg}^{p+1-,\overline{(p+2,1)}}(\Dbold)+\Hb^{(p+\epsilon-,p+\epsilon-)}(\Dbold) }.
		\end{equation}
		Here, the functions $f_n^{(p,\epsilon)}(r_0,\omega)$ and $c_{\log}^{(p)}(\omega)$ satisfy
		\begin{equation}\label{eq:Schw:thm:wave:projections}
			\langle f_n^{(p,\epsilon)},Y_{\ell}\rangle=f_n^{(\ell,p,\epsilon)}+\O(r_0^{-p-\epsilon+}) \in \A{phg}^{\overline{(p,0)}}(\incone)+\Hb^{p+\epsilon-}(\incone), \qquad \langle c_{\log}^{(p)}(\omega),Y_{\ell}\rangle =c_{\log}^{(\ell,p)},
		\end{equation}
		where $f^{(l,p,\epsilon)}_n,c_{\log}^{(\ell,p)}$ are defined in \cref{thm:Schw:wave}. An analogous statement holds for $p\notin\mathbb N$.
	\end{thm}
	\begin{proof}
We write the initial data as $\psi^{\incone}=\psi_{\phg}^{\incone}+\psi_{\Delta}^{\incone}\in \A{phg}^{\mindex{p}}(\incone)+\Hb^{p+\epsilon}(\incone)$ and apply \cref{cor:Schw:phg}, thus obtaining an expansion of the form \cref{eq:Schw:expansionfull} (with potentially larger index sets, i.e.~more $\log$-terms~etc.)
We then compute the coefficients in these expansions (which hold in an $L^2$ sense) by projecting onto fixed angular modes, for which we can apply \cref{eq:Schw:exp:pinN}. In particular, this then proves \cref{eq:Schw:thm:wave:projections}.
	\end{proof}
	\begin{rem}
		Notice that in \cref{eq:Schw:expansionfull}, the error term is only of order $r^{-p-\epsilon}$ towards $\scrip$, which is 1 power worse compared to the fixed $\ell$-mode result \cref{eq:Schw:exp:pinN}. In this sense, there seems to be a certain loss going from fixed $\ell$-modes to the resummed solution, see \cref{conj:even:peeling} for further commentary.
  On the other hand, notice that if we specified an expansion for $\psi^{\incone}\in\A{phg}^{\E}(\incone)$ for some index set with $\min\E=(p,0)$, then this will only introduce error terms beyond order $p+1$ towards $\scrip$.
	\end{rem}

    \subsection{Aside: Construction of the sharp index set}\label{sec:Schw:bettercoords}
    \newcommand{\pvbar}{\partial_{\bar{v}}}
    \newcommand{\pubar}{\partial_{\bar{u}}}

    When working with coordinates $u,v$ and corresponding smooth structure on $\D$, we observe that the wave equation~\cref{eq:Schw:linear_on_Schw} has a perturbation $(r^2-r_\star^2)\Dl$, whose coefficient contains a logarithmic term.
    This introduces many extra logarithmic terms in the non specified index sets in \cref{cor:Schw:phg}.
    Of course, as we have repeatedly stated before, this is not directly relevant since we have an algorithmic procedure to compute the actual coefficients in the index sets. Nevertheless, we now showcase that with a better choice of coordinates, one can immediately arrive at the sharp index set in \cref{cor:Schw:phg}.
    Furthermore, we can also show that there's only a single logarithmic term appearing, that is the expansion does not contain terms of the form $r^{-p}\log^k r$ with $k\geq2$.
   
We introduce coordinates $\bar{u}=-r_0, \bar{v}=r+\bar{u}$ and write the Schwarzschild metric as
		\begin{nalign}
			&g&&=-\underbrace{\Big(D(2D_0^{-1}-D^{-1})^2-D^{-1}\Big)}_{=\frac{-8MD_0^{-2}\bar{v}}{rr_0}\sim\rho_-^2\rho_0}\dd \bar{u}^2-4D_0^{-1}\dd\bar{u}\dd\bar{v}+r^2\dd g_{S^2},&&
            &g^{-1}&&=-\frac{2M\bar{v}}{rr_0}\partial_{\bar{v}}^2-D_0\partial_{\bar{v}}\partial_{\bar{u}}+r^{-2}g_{S^2}^{-1}.
		\end{nalign}
    For this metric, $\{\bar{v}=v_0\}=\incone$ as well as the $\bar{u}=$~constant hypersurfaces are null.
    We modify \cref{def:notation:comp} (with $u,v$ in \cref{eq:notation:defining_functions} replaced by $\bar{u},\bar{v}$), to define the compactifications $\D^{\mathrm{S}},\Dbold^{\mathrm{S}}$ of $\Dopen$.
    This corresponds to modifying the smooth structure of $\D, \Dbold$ at subleading level.\footnote{See \cite{hintz_stability_2020} for a similar change in a neighbourhood of timelike infinity.}
    We have that the wave operator takes the form
    \newcommand{\DiffS}{\Diff^2_{\b}(\Dbold^{\mathrm{S}})}
      \begin{nalign}\label{eq:Schw:box_in_uvbar}
       \bar{ \Box}_g&=-\frac{2M\bar{v}}{rr_0}\partial_{\bar{v}}^2-D_0\partial_{\bar{v}}\partial_{\bar{u}}+r^{-2}\Dl+\frac{D_0}{r}(\partial_{\bar{v}}-\partial_{\bar{u}})-\frac{2M(r+\bar{v})}{r_0r^2}\pvbar=D_0\bar{\Box}_\eta-D_0(\bar{\Box}_{\eta}-D_0^{-1}\bar{\Box}_g),\\
       \bar{\Box}_{\eta}&=-\pubar\pvbar+\frac1r(\pvbar-\pubar)+r^{-2}\Dl,\\
       \bar{\Box}_{\eta}-D_0^{-1}\bar\Box_{g}&=(1-D_0^{-1})r^{-2}\Dl+\frac{2M}{\bar{v}rr_0D_0}\bar{v}^2\pvbar^2+\frac{2M(r+\bar{v})}{r_0r^2}\pvbar\in \frac{1}{r^2r_0}\DiffS,
    \end{nalign}
    where $\DiffS$ denotes linear combinations of $1,\Dl,\bar{v}\pvbar, \bar{u}\pubar$ with coefficients in $\A{phg}^{\mindex0,\mindex0}(\Dbold^{\mathrm{S}})$.
    
    Now, the appearance of $\pvbar^2$ term makes the operator unfit for the application of \cref{thm:scat:scat_general} or \cref{thm:app:general}.
    Furthermore, the appearance of $\bar{v}$ factors make it non compatible with the no incoming radiation condition.
    Instead of including these specific terms in the notationally already heavy theorems, we prove a sharp result on the index set by hand following \cref{lemma:app:linear_improvement}, treating $\bar{\Box}_g$ as a perturbation around Minkowski.
    
    \begin{lemma}\label{lemma:Schw:no_logk_terms}
        Let $\phi_{\mathrm{phg}}$ be as in \cref{cor:Schw:phg} 1).
        Then we can take the index sets to be $\E^\phi_0=\E^{\incone}$ and $\E^\phi_+=(\E^{\incone}+1)\cupdex\mindex0$. 
    \end{lemma}
    \begin{rem}
        This is essentially the same result as \cref{lem:app:potential} with a specific linear choice of $P_g\neq0$.
    \end{rem}
    \begin{proof}
        We prove that $\E^\phi_+=\E^{\incone}\cupdex \mindex{0}$.
        The improvement to $(\E^{\incone}+1)\cupdex\mindex{0}$ comes from splitting the solution to first and second iterate (i.e.~$\phi=\phi^{(0)}+\phi^{(1)}$ with $\phi^{(0)}$ solving $\bar\Box_{\eta}\phi^{(0)}=0$) and using \cref{lemma:app:minkowski_index_improvement}.
        
       Using $u,v$ coordinates, we already have that $\psi\in\Hb^{\vec{a}}(\Dbold)$ and therefore $\psi\in\Hb^{\vec{a}}(\Dbold^{\mathrm{S}})$.

        \textit{Step 1: Improvement in past region.}
        We write $\bar\Box_g\phi=0$ as $\bar{\Box}_\eta\phi=(\bar{\Box}_\eta-\bar{\Box}_g)\phi$ with 
        \begin{equation}
            \bar{\Box}_\eta-D_0^{-1}\bar{\Box}_g=f^{\mathrm{Sch}}\DiffS,\qquad f^{\mathrm{Sch}}\in\A{phg}^{{\mindex{3},\mindex{2}}}(\Dbold^{\mathrm{S}})
        \end{equation}
        We can iteratively peel the leading order terms off of $\phi$ in $\Dbold^{\mathrm{S};-}$:
            Indeed, using $\psi\in\Hb^{a_0}(\Dbold^{\mathrm{S};-})$ we get $\bar\Box_\eta\phi\in\Hb^{a_0+3}(\Dbold^{\mathrm{S};-})$ and then \cref{prop:prop:prop_no_incoming} implies $\psi\in\A{phg}^{\E^\phi_0}(\Dbold^{\mathrm{S};-})+\Hb^{a_0+1}(\Dbold^{\mathrm{S};-})$.
        Using this, we immediately get that $r(D_0^{-1}\bar{\Box}_g-\bar{\Box}_\eta)\phi\in\A{phg}^{\E^\phi_0+3}(\Dbold^{\mathrm{S};-})+\Hb^{a_0+4}(\Dbold^{\mathrm{S};-})$.
        By induction, we get that, for all $n\in\mathbb N$,
        \begin{equation}
            \psi\in\A{phg}^{\E^\phi_0}(\Dbold^{\mathrm{S};-})+\Hb^{a_0+n}(\Dbold^{\mathrm{S};-}).
        \end{equation}
        This can be propagated all the way to the future corner to get $\phi\in\A{phg,b}^{\E^\phi_0,a_+}(\Dbold^\mathrm{S})$ for some $a_+$.
        
        \textit{Step 2: Improvement around the future corner for scaling commuted quantity.} 
       Around the future corner $\Dbold^{\mathrm{S};+}$, we again write $\Box_g\phi=0$ as
        \begin{equation}
            \bar{\Box}_\eta\phi=(\bar{\Box}_\eta-D_0^{-1}{\Box}_g)\phi=
            \frac{f^{\mathrm{Sch}}(r,\bar{u}/\bar{v})}{r^2(r_0-2M)}(r\pvbar,\sl,1)^2\phi
        \end{equation}
         with coefficients  $f^{\mathrm{Sch}}(r,\bar{u}/\bar{v})\in \A{phg}^{\mindex{0}}(\R\times(-1,1))$ smooth in the second paramater and polyhomogeneous in the first.\footnote{Here, we use the $1/r$ compactification of $\R$.}

        Let's define $\mathcal{S}^{\mathcal{E}}_c$ as in \cref{eq:prop:S_c_definition}, with $S=\bar{u}\pubar+\bar{v}\pvbar$.
        We commute the wave operator to get that
        \begin{equation}\label{eq:Schw:S_commuted_wave}
            \bar{\Box}_\eta \mathcal{S}^{\E^{\phi}_0}_c\phi=\frac{1}{r^2(r_0-2M)}\mathcal{S}^{\E^{\phi}_0-1}_{c-1} f^{\mathrm{Sch}} (r\pvbar,\sl,1)^2\phi.
        \end{equation}
        
       Let us observe that $r\phi\in\A{phg,b}^{\E^\phi_0,a_+}(\Dbold^{\mathrm{S};+})$ already implies $r(\bar{\Box}_\eta-D_0^{-1}\bar{\Box}_g)\phi\in \A{phg,b}^{\E^\phi_0+3,a_++2}(\Dbold^{\mathrm{S};+})$.
        Therefore, we get that for $p\leq a_+$, $r\bar{\Box}_\eta \mathcal{S}^{\E^{\phi}_0}_{p}\phi\in\A{phg,b}^{\E^{\phi}_0+3,a_++2}(\Dbold^{\mathrm{S};+})$.
        Using \cref{eq:prop:futuresemiindex} of \cref{lemma:prop:future_corner} we get 
        \begin{equation}
         r   \mathcal{S}^{\E^{\phi}_0}_{p}\phi\in\A{phg}^{(\E_{0}^\phi)_{>p},\mindex{0}}(\Dbold^{\mathrm{S};+})+\A{phg,b}^{\E_{>p}^\phi,\min(p,a_+)+1}(\Dbold^{\mathrm{S};+}).
        \end{equation}
        By induction, we claim
        \begin{equation}\label{eq:Sch:improved_induction_future_corner}
        r    \mathcal{S}^{\E^{\phi}_0}_{p+n}\phi\in\A{phg}^{(\E_{0}^\phi)_{>p+n},\mindex{0}}(\Dbold^{\mathrm{S};+})+\A{phg,b}^{\E_{{>p+n}}^\phi,\min(p,a_+)+n}(\Dbold^{\mathrm{S};+}).
        \end{equation}
        Indeed, assuming \cref{eq:Sch:improved_induction_future_corner} holds for $n\leq n_0$, we get
        \begin{nalign}
           r \bar{\Box}\mathcal{S}^{\E^{\phi}_0}_{p+n_0+1}\phi=\frac{1}{rr_0}\mathcal{S}^{\E^\phi_0-1}_{p+n_0}\DiffS\phi\in\A{phg}^{\E_{>p+n_0}^\phi+3,\mindex{2}}(\Dbold^{\mathrm{S};+})+\A{phg,b}^{\E_{{>p+n_0}}^\phi+3,\min(p,a_+)+n_0+2}(\Dbold^{\mathrm{S};+})\\
            \implies 
            r\mathcal{S}^{\E^{\phi}_0}_{p+n_0+1}\phi\in\A{phg}^{\E_{>p+n_0}^\phi+1,\mindex{0}}(\Dbold^{\mathrm{S};+})+\A{phg}^{\E_{{>p+n_0}}^\phi+1,\min(p,a_+)+n_0+1}(\Dbold^{\mathrm{S};+}).
        \end{nalign}

        \textit{Step 3:} We may integrate along the vectorfields $S$ as in \cref{lemma:prop:improvement_at_future} to obtain the result.
        That is, we use shifted commutator $(\bar{u}-u_\bullet)\partial_{\bar{u}}+(\bar{v}-u_\bullet)\partial_{\bar{v}}$ for $\bullet\in\{1,2\}$.
        First, we get that $\phi|_{\outcone{1}}\in\A{phg}^{\E^\phi_0\cupdex\mindex0}(\outcone{1})$ along a single outgoing cone, then we integrate back from a later cone $\C_{u_2}$ towards $I_0$ to obtain the result.
    \end{proof}

	\subsection{Linearised gravity on Schwarzschild in a double null gauge}\label{sec:Sch:ling}
	We now present how to apply our propagation of polyhomogeneity results in the context of linearised gravity around Schwarzschild in a double null gauge \cite{dafermos_linear_2019}.
    From the perspective of the present paper, linearised gravity essentially consists of wave equations (to which we directly apply our results as in \cref{sec:Sch:wave}), and of transport equations which we treat as ODEs, using the results of \cref{sec:ODE_lemmas}.
With that said, we now import the minimal set of prerequisites to be able to reasonably discuss linearised gravity.

	Recall first the extremal components of the linearised Weyl curvature tensor $\al$ and $\alb$. In order to not have to introduce the notation of projected covariant derivatives etc., we will always talk about the components of a 2-tensor evaluated in an orthonormal basis; with the only exception to this rule being angular derivatives. For instance, when writing $\al$, we will always mean the component $\al(e_A,e_B)$, where $e_A, e_B\in\{\frac{1}{r}\partial_{\theta},\frac{1}{r\sin\theta}\partial_{\varphi}\}$.
	Similarly, when we write $\pu\al$, we mean the partial derivative of the component of $\al$. On the other hand, when we write $\Dl \al$, we mean the components of the Laplacian acting on the 2-tensor $\al$.
The purpose of this is that we can treat all equations as scalar equations.
For instance, $\al$ and $\alb$ satisfy the Teukolsky equations: 
\begin{nalign}\label{eq:Schw:teuk}
    \pu\left(\frac{D^2}{r^4}\pv(D^{-1}r^5\al)\right)=\frac{D^2(\Dl+2)\al}{r}-\frac{30MD^2\al }{r^2},&& \pu\left(\frac{r^4}{D^2}\pv(rD\alb)\right)=r^3(\Dl-2)\alb-6MD^2\alb.
\end{nalign}

	We also recall the definition of the Regge--Wheeler quantities:
	\begin{align}\label{eq:Schw:psal}
		\Ps:=(D^{-1}r^2\pu)^2(r\Omega^2\al),&& \Psb:=(D^{-1}r^2\pv)^2(r\Omega^2\alb); 
	\end{align}
	these both satisfy the Regge-Wheeler equation
	\begin{equation}
		\pu\pv \Psi=D\frac{(\Dl-4) \Psi}{r^2}+\frac{6MD\Psi}{r^3},
	\end{equation}
	to which all our results (in particular, \cref{cor:Schw:phg}) apply. (Note that $\Dl-4$ has eigenvalues $-\ell(\ell+1)$ when acting on 2-tensors.)

 \begin{rem}[Treating the Teukolsky equations as twisted long-range potentials]\label{rem:Sch:Teukolsky_decay_rate}
      The way the Teukolsky equations are treated in \cite{dafermos_linear_2019} is by transforming them into the Regge--Wheeler equations as above.
We can, however, also treat them directly. For, the Teukolsky equations \cref{eq:Schw:teuk} are  (short-range potential perturbations of)  twisted long-range potential modifications of the usual wave operator $\Box_\eta$ as discussed in \cref{rem:en:twistedlongrange}.

    More precisely, in $\D^-$, it is the quantities $rD\al$ and $r^5D^{-1}\alb$ that satisfy (non-twisted) long-range potential  equations.
        By \cref{rem:en:good_derivatives_noincoming}, then solutions to \cref{eq:Schw:teuk} with nontrivial incoming radiation and with $\al^{\incone},\,\alb^{\incone}\in \Hb^{a_-}(\incone)$ (for $a_-$ arbitrary) will satisfy $\al\in\Hb^{1-,a_0;\infty}(\D^-)+\Hbt^{a_-,a_0;\infty}(\D^-)$, $\alb\in\Hb^{5-,a_0;\infty}(\D^-)+\Hbt^{a_-,a_0;\infty}(\D^-)$.
Note that, unless the data for $\alb^{\incone}$ decay fast, then, in order to construct the $\Hbt^{a_-,a_0;\infty}(\D^-)$ of the solution, we need to commute with time derivatives (albeit not using as much structure) as the commutations in \cref{eq:Schw:psal}.

        On the other hand, in $\D^+$, the equations \cref{eq:Schw:teuk} are already of the correct form: By \cref{lemma:EVE:weak_data_at_I0} near the future corner, we then obtain  $\al\in\Hb^{a_0,5-;\infty}(\D^+)+\{1,r\pu\}^{-\infty}\Hb^{a_0,a_+;\infty}(\D^+)$, $\alb\in\Hb^{a_0,1-;\infty}(\D^+)+\{1,r\pu\}^{-\infty}\Hb^{a_0,a_+;\infty}(\D^+)$, where $\{1,r\pu\}^{-\infty}\Hb^{a_0,a_+;\infty}(\D^+)=\{f\in \Hb^{a_0,a_+;\infty}(\D^+):(r\pu)^k f\in\Hb^{a_0,a_+;\infty}(\D^+)\; \forall k\}$ and $a_0\in\R$ corresponds to the initial decay towards $I^0$.
   
    \end{rem}

	The next definition we import is that of the in- and outgoing shear. We have
	\begin{align}\label{eq:Schw:chidefinitions}
		\pu\left(\frac{r^2\xhb}{\sqrt{D}}\right)= - r^2\alb,&& {\pv}\left(\frac{r^2\xh}{\sqrt{D}}\right):=- r^2\al.
	\end{align}
 Of course, the full specification of $\xhb$ and $\xh$ also requires specified initial values somewhere. 

 Now, the full system of linearised gravity consists of many other quantities and equations and, subsequently, the formulation of the scattering problem for them is a bit more involved. It consists of the specification of seed scattering data along $\incone\cup\scrim$, from which one can then derive all other quantities along $\incone\cup\scrim$. Cf.~\cite{kehrberger_case_2024} for details. In the context of this section, we will occasionally formulate statements with as many assumptions on data as we need, knowing that they can ultimately recovered from seed data in the context of the full system. We will also formulate all statements assuming already that the solution exists (which follows from \cite{kehrberger_case_2024}). While we thus do not define the full system; when we talk about solutions to linearised gravity, we will, in particular, always mean \cref{eq:Schw:teuk}--\cref{eq:Schw:chidefinitions} to be satisfied.

 In this spirit, we first define what it means for a solution to have no incoming radiation from $\scrim$ (cf.~\cite[\S8]{kehrberger_case_2024}):
 \begin{defi}[No incoming radiation for linearised gravity]\label{def:schw:nir}
		Solutions to the system of linearised gravity with no incoming radiation from $\scrim$ satisfy  $\lim_{u\to-\infty}\norm{\Big((\Omega^{-2}r^2\pu)^i(r\Omega^2\al),\pv(r^2\xhb)\Big)(u,v)}_{L^2(\mathbb S^2)}=0$ for all $v<\infty$ and for $i=0,1$. They also satisfy $\lim_{u\to-\infty}\norm{\pv\Ps,\pv\Psb,r\xh}_{L^2(\outcone{}^{v_0,v_\infty})}=0$ for any $v_\infty>v_0$.
  For smooth solutions, the above equalities remain valid for commutations with $\pv,\sl$.
	\end{defi}
We finally note that, while we avoid working with the entire system, the quantities $\Ps, \Psb, \al, \alb, \xh, \xhb$ already capture the essential physical degrees of freedom of the problem. In particular, the essential part of the seed data is given by $\xhb|_{\incone}$, $\xh|_{\scrim}$ and by the specification of $\pv\alb$ along any sphere along $\incone$. The specification of $\pv\alb$ should be compared to the scattering problem with weak decay, cf.~\cref{thm:scat:weak_decay}.

 Now, analogously to \cref{cor:Schw:phg}, the main ingredient for upgrading the fixed $\ell$-mode results of \cite{kehrberger_case_2024} will be the following statement, which is a corollary of \cref{thm:app:general} and our ODE lemmata:
    	\begin{lemma}
    	    \label{cor:Schw:lingravphg}
Consider solutions to the system of linearised gravity with no incoming radiation. 

	1) Assume $\Ps|_{\incone},\Psb|_{\incone},r^2\xh|_{\incone}, r^2\xhb|_{\incone},r^3\pv(rD\alb)|_{\incone},r^3\alb|_{\incone}\in \A{phg}^{\mathcal{E}^{\bullet}_-}(\incone)$, with $\min \mathcal{E}_-^{\bullet}>-1/2$.
		
        Then  $\Ps,\Psb,\al,\alb,\xh,\xhb\in\A{phg}^{\vec{\E}^\bullet}(\D)$ for some index sets.

2) Assume $\Ps|_{\incone},\Psb|_{\incone},r^2\xh|_{\incone}, r^2\xhb|_{\incone},r^3\pv(rD\alb)|_{\incone},r^3\alb|_{\incone}\in \Hb^{p}(\incone)$, with $p\geq-1/2$.
		Then  $\Ps,\Psb,r^2\xh\in\A{b,b,phg}^{p-,p-,\mindex{0}}(\D)+\Hb^{p-,p-,p-}(\D)$, $r\alb\in\A{b,b,phg}^{p+2-,p+2-,\mindex{0}}(\D)+\Hb^{p+2-,p+2-,p+2-}(\D)$,
		${r^3\al\in\A{b,b,phg}^{p-,p-,\mindex{2}}(\D)+\Hb^{p-,p-,p-}(\D)}$
        and  $r\xhb\in\A{b,b,phg}^{p+1-,p+1-,\mindex{0}}(\D)+\Hb^{p+1-,p+1-,p+1-}(\D)$.
	\end{lemma}

 \begin{rem}[Peeling of $\al$ and $\alb$]
    For $p$ sufficiently large, then \cref{cor:Schw:lingravphg} gives the leading order decay behaviour for $\al\sim r^{-5},\alb\sim r^{-1}$ towards $\scrip$ as predicted by the peeling property for gravitational radiation. (The reader should of course compare this with \cref{rem:Sch:Teukolsky_decay_rate})
    For the Einstein vacuum equations, these are the rates required by a smooth null infinity.
 \end{rem}

	\begin{proof}
		\textit{1)} For $\Ps,\Psb$ we apply \cref{cor:Schw:phg} to get the result.
        More precisely, \cref{cor:Schw:phg} gives polyhomogeneity on~$\Dbold$, and, a fortiori, we get it on $\D$.
		
		For $\al$, we apply \cref{corr:ODE:du_dv}\cref{item:ode:u-prop_boundary} to the first of \cref{eq:Schw:psal} (with boundary terms vanishing by  \cref{def:schw:nir}) to get the result.
		
	     For $\xhb$, we cannot directly apply \cref{corr:ODE:du_dv}\cref{item:ode:u-prop_boundary}, as the decay may be too slow. Instead, we apply \cref{item:ode:1d_boundary_cond_weak_decay} to the first of \cref{eq:Schw:chidefinitions} to get that $\xhb\in\A{phg}^{\E^{\xhb}_-, \E^{\xhb}_+}(\D^-)$ for some index sets.
        The future corner is dealt with as in \cref{corr:ODE:du_dv}\cref{item:ode:u-prop_boundary}.
        An alternative approach is to commute \cref{eq:Schw:chidefinitions} with $\pv$; we shall follow this approach for the proof of 2).

        For $\alb,\xh$ the result follows by \cref{corr:ODE:du_dv}\cref{item:ode:u-prop}, with $u$ and $v$ interchanged, applied to the second of \cref{eq:Schw:psal} and of \cref{eq:Schw:chidefinitions}, respectively.

    \textit{2)} 
        For $\Ps,\Psb$, we use \cref{cor:Schw:phg} to get the result.
        
        For $\al$, first start with the weaker $\Ps\in\Hb^{p-,p-,\min(0,p)-}(\D)$. 
        We apply \cref{corr:ODE:du_dv}\cref{item:ode:u-prop_boundary} twice to obtain that $r\al\in \Hb^{2+p-,2+p-,\min(4,2+p)-}(\D)$.
        Similarly, we can use the mixed estimate \cref{eq:ODE:2d_error_phg} near the future corner to get $r^3\al\in\A{b,b,phg}^{p-,p-,{\mindex{2}}}(\D)+\Hb^{p-,p-,p-}(\D)$.
        From this, the required estimate for $\xh$ follows directly as in 1), i.e.~by applying \cref{corr:ODE:du_dv}\cref{item:ode:u-prop} in $\D^-$ and \cref{eq:ODE:1d_phg} in $\D^+$.

        For $\alb$, we integrate the second of \cref{eq:Schw:psal} from $\incone$, using \cref{corr:ODE:du_dv}\cref{item:ode:u-prop} in $\D^-$ and \cref{eq:ODE:2d_error_phg} in $\D^+$ to first get that
\begin{equation}
    r^2\pv(rD\alb)\in \A{b,b,phg}^{p+1-,p+1-,\mindex0}(\D)+\Hb^{p+1-,p+1-,p+1-}(\D).
\end{equation}
The same ODE integrations then give
\begin{equation}
    r\alb\in \A{b,b,phg}^{p+2-,p+2-,\mindex0}(\D)+\Hb^{p+2-,p+2-,p+2-}(\D).
\end{equation}

Finally, we need to establish the result for $\xhb$:
For $p<0$, we cannot integrate \cref{eq:Schw:chidefinitions} from $\scrim$ directly, so we consider instead
\begin{equation}
    \pu\pv(D^{-\frac12}r^2\xhb)=\pv(r^2\alb)=\frac{r}{D}\pv(rD\alb)+\pv(\frac{r}{D})rD\alb\in \A{b,b,phg}^{p+2-,p+2-,\mindex0}(\D)+\Hb^{p+2-,p+2-,p+2-}(\D).
\end{equation}
Now, we first establish the result in $\D^-$ by integration: first in $u$ via \cref{corr:ODE:du_dv}\cref{item:ode:u-prop_boundary}  using that $p+1>0$ and the vanishing of $\pv\xhb$ at $\scrim$; then in $v$ via \cref{corr:ODE:du_dv}\cref{item:ode:u-prop}  using that we also have data for $\xhb$ on $\incone$.
We then prove the result in $\D^+$ by using a cutoff function, and applying \cref{eq:ODE:2d_error_phg} to the first of \cref{eq:Schw:chidefinitions}.
    \end{proof}
    \begin{obs}\label{rem:Sch:DtoDbold}
        In \cref{cor:Schw:lingravphg}, we proved all results in $\Hb(\D)$ and $\A{phg}(\D)$ spaces. 
We can upgrade them to $\Hb(\Dbold)$ and $\A{phg}(\Dbold)$ spaces as follows: 
First, note that it suffices to show this improvement in $\D^-$. Then, note that we already proved the improved statements for $\Ps, \Psb$. 
Furthermore, we can directly apply the results of \cref{sec:prop:longrange} to the (suitably twisted, cf.~\cref{rem:Sch:Teukolsky_decay_rate}) Teukolsky equations for $\al$ and $\alb$ to improve the statements for those as well. Note that in $\D^-$, the quantity $r^5 D\alb$ features slow initial data decay, and we therefore need to appeal to our slow initial data decay results, cf.~\cref{rem:prop:restrictiondropping}.
Finally, we get the improvements for $\xh$ and $\xhb$ by applying \cref{lemma:prop:inhomogeneous_no_rad} to $\pv\pu(D^{-1/2}r^2\xhb)=-\pv (r^2\alb)$ and $\pu\pv(D^{-1/2}r^2\xh)=-\pu(r^2\alb)$.
    \end{obs}

	We now use this result to infer that the results of \cite{kehrberger_case_2024} are, in fact, valid for solutions supported on all angular frequencies. 
	For this, it will be convenient to recall (and slightly modify) the following definition:
\newcommand{\Bb}{\underline{\mathscr{B}}}
\newcommand{\Aa}{\mathscr{A}}
\newcommand{\xhbs}{\accentset{\scalebox{.6}{\mbox{\tiny (1)}}}{{\hat{\underline{\upchi}}}}_{\scrim}}
\begin{defi}[$N$-body data. Cf.~Def.~1.4 of \cite{kehrberger_case_2024}]\label{def:Schw:data}
		A Bondi normalised seed scattering data set (for the system of linearised gravity around Schwarzschild) is said to describe the exterior of $N$ infalling masses coming in from the infinite past and following approximately hyperbolic orbits if it has no incoming radiation and the following three conditions are satisfied
		\begin{enumerate}[label=(\Roman*)]
			\item\label{item:nir:I}  $\alb|_{\incone}=\alb_{\incone}\in \A{phg}^{\mathcal{E}}(\incone)+\Hb^{3+\epsilon}(\incone) $, where $\min\mathcal{E}=(4,0)$ and $\epsilon>1$. We write $\alb_{\incone}=-6\Bb(\omega)r_0^{-4}+\dots$ for $\Bb\neq0$.
			\item\label{item:nir:II} $\lim_{u\to-\infty}r^2\xhb|_{\incone} = \xhbs(\omega)\neq 0.$
			\item\label{item:nir:III}  $\lim_{u\to-\infty}r^3\al|_{\incone}=\Aa(\omega)\neq0.$
		\end{enumerate}
        If moreover $\Bb$ vanishes, then we say that the seed scattering data set describes the exterior of a spacetime that is isometric to Minkowski near $i^-$ and has compactly supported incoming gravitational radiation.
	\end{defi}
 
 \begin{rem}\label{rem:Schw:Nbodydataexplained}
     Comparing to \cref{cor:Schw:lingravphg}, one might again think that the definition above should also specify, say, data for $\xh$ along $\incone$; however, such information can be recovered from the points above in the context of the full system.
     In particular, within the full system, one can deduce from \cref{def:Schw:data} that the induced full data along $\incone$ can be written as $\mathrm{data}=\mathrm{data}_{\phg}+\mathrm{data}_{\Delta}$, with $\mathrm{data}_{\phg}$ satisfying the assumptions of \cref{cor:Schw:lingravphg} \textit{1)}, and with $\mathrm{data}_{\Delta}$ satisfying the assumptions of \cref{cor:Schw:lingravphg} \textit{2)} with $p=\epsilon$. Here, $\mathrm{data}$ is short-hand for the tuple $\{\Ps|_{\incone},\Psb|_{\incone},\xh|_{\incone},\xhb|_{\incone},\pv(rD\alb)|_{\incone},\alb|_{\incone}\}$. We emphasise that deducing this truly requires the full system of linearised gravity and is not possible from just the information given in the present section. Cf.~\S5 of \cite{kehrberger_case_2024}. 
 \end{rem}
	We then infer from  \cite[Theorem~1.2]{kehrberger_case_2024}  and \cref{cor:Schw:lingravphg} together with \cref{rem:Sch:DtoDbold} the following:
\begin{thm}\label{thm:Schw:lingravity}
		Consider solutions to the system of linearised gravity arising from scattering data as in Definition~\ref{def:Schw:data}. Then the following expansions are valid throughout $\Dopen$ for any $N\geq 3$: \begin{align}\label{eq:Schw:thmlingrav:al}
			r^5D^{-1}\al&=\sum_{n=0}^{N-4} {A}_{n}^{(\mathcal{E},\epsilon)}(r_0,\omega)\frac{r_0^n}{r^n}+M{a}(\omega)\cdot r+ \underbrace{\O\left(\frac{r_0^{N-1}}{r^{N-3}}+M({r_0\log r_0/r+\log^2 r_0/r})+r^{2-\epsilon-}\right)}_{\in\Hb^{-2-,0-}(\Dbold)+\Hb^{-2+\epsilon-,-2+\epsilon-}(\Dbold)},\\
   rD\alb&=\sum_{n=0}^N \underline{A}_{n}^{(\mathcal{E},\epsilon)}(r_0,\omega)\frac{r_0^n}{r^n}+M\underline{a}_{\log}(\omega)\frac{\log(r/r_0)}{r^3}
   +\underbrace{\O\left(\frac{r_0^{N-1}}{r^{N+1}}+M\frac{r_0 \log r_0/r +\log^2 r_0/r}{r^4}+\frac{1}{r^{2+\epsilon-}}\right)}_{\in\Hb^{2-,4-}(\Dbold)+\Hb^{2+\epsilon-,\violet{2}+\epsilon-}(\Dbold)}
   ,\\
    D^{-1/2}r^2\xh&=\lim_{v\to\infty}r^2\xh+M\cdot \frac{x_{\log}(\omega)}{r} +\underbrace{\O\left(\frac{r_0^2}{r^2}+M\frac{r_0\log r_0/r+\log^2 r_0/r}{r^2}+r^{-\epsilon+}\right)}_{\in\Hb^{0-,2-}(\Dbold)+\Hb^{\epsilon-,\epsilon-}(\Dbold)},\\
   	rD^{1/2}\xhb&=\sum_{n=0}^{N-1} \underline{X}_n^{(\mathcal{E},\epsilon)}(r_0,\omega)\frac{r_0^n}{r^n}+M\cdot \underline{x}_{\log}(\omega) \frac{\log r/r_0}{r^2}+\underbrace{\O\left(r_0^{N-1}/r^N+M\frac{r_0\log r/r_0+\log^2 r/r_0}{r^3}+r^{-1-\epsilon+}\right)}_{\in \Hb^{2-,3-}(\Dbold)+\Hb^{1+\epsilon-,1+\epsilon}(\Dbold)},
		\end{align} 
  where
\begin{equation}
    r_0^{-2}{A}_{n}^{(\mathcal{E},\epsilon)},\, r_0^{2} \underline{A}_{n}^{(\mathcal{E},\epsilon)},\, r_0\underline{X}_n^{(\mathcal{E},\epsilon)},\, \lim_{v\to\infty}r^2\xh \in \A{phg}^{\E'}(\incone)+\Hb^{\epsilon}(\incone),\qquad \text{with  }\min(\E')=(0,0),
\end{equation}
  and where the leading order behaviour of the projections onto fixed angular frequencies of ${A}_{n}^{(\mathcal{E},\epsilon)}$ and $ a $, of $\underline{A}_{n}^{(\mathcal{E},\epsilon)}$ and $\underline{a}_{\log}$, of $\underline{X}_n^{(\mathcal{E},\epsilon)}$ and $\underline{x}_{\log}$, and of $\lim r^2\xh$ and $x_{\log}$ are given by the results of Theorem~11.1, Theorem~13.1, Proposition~16.1 and Proposition~16.2 of \cite{kehrberger_case_2024}, respectively.
	\end{thm}
 \begin{rem}
     Notice that we only specified the leading-order decay towards $\scrim$ and $I^0$ in order to simplify the notation. Cf.~\cref{rem:Schw:lossy,rem:Schw:higher}.
 \end{rem}
 \begin{proof}
     The idea is identical to that of \cref{thm:Schw:wave:summed}. First, we split up the seed data from \cref{def:schw:nir} into $r\alb_{\phg}^{\incone} \in \A{phg}^{\E}(\incone)$ in \cref{item:nir:I}, and with the limits \cref{item:nir:II}, \cref{item:nir:III}.
     For the solutions arising from these data, we can apply \cref{cor:Schw:lingravphg} 1), cf.~\cref{rem:Schw:Nbodydataexplained}.

     The other part of the seed data is taken as $\alb_{\Delta}^{\incone}\in H^{3+\epsilon}(\incone)$, with the limits in \cref{item:nir:II}, \cref{item:nir:III} vanishing.
     As shown in \S5 of \cite{kehrberger_case_2024} (cf.~\cref{rem:Schw:Nbodydataexplained}), one can deduce using the full system of linearised gravity that
     \begin{equation}
         r^3\al|_{\incone}, r^2\xhb|_{\incone}r^3\alb_{\Delta}^{\incone}\in H^{\epsilon}(\incone)\quad\implies\quad\Ps|_{\incone},\Psb|_{\incone},r^2\xh|_{\incone}, r^2\xhb|_{\incone},r^3\pv(rD\alb),r^3\alb|_{\incone}\in \Hb^{\epsilon}(\incone)
     \end{equation}
     We can then apply \cref{cor:Schw:lingravphg} 2) with $p=\epsilon$. 

The result now follows from angular mode projections, using the results of \cite{kehrberger_case_2024}.
 \end{proof}

	The statements above are enough to conclude that almost all other statements  from \cite[Thm~1.2]{kehrberger_case_2024} hold for the resummed solution, with the only exceptions being the statements about the last three limits ($r\gsh$, $r\trg$ and~$r\blin$) in \cite[Thm.~1.2, Item \textbf{(VI)}]{kehrberger_case_2024}.\footnote{Recall from \cite{kehrberger_case_2024} that, in order to also treat these last three limits, one needs to better resolve the structure of $\xhb$ near $\scrim$, cf.~(16.16) and Proposition 17.3 of \cite{kehrberger_case_2024}. }
 We can easily include this in our results by extending \cref{cor:Schw:lingravphg} to also include a polyhomogeneity statement for $\gsh$ (which satisfies $\pu\gsh=2\Omega\xhb$); we leave this to the reader.
 Finally, \cref{thm:Schw:lingravity} is stated for initial data with infinite regularity for the sake of easier notation. Since our proofs of \cref{sec:app} and \cref{cor:Schw:lingravphg} only lose finite regularity, it clearly also holds for initial data with finite regularity.

	\newpage
	\section{The specificity of peeling to even spacetime dimensions and asymptotics for the scale-invariant wave equation}\label{sec:sharp}
In the present section, we prove that peeling is a property specific to even spacetime dimensions by computing the coefficients in the expansions for the linear wave equation in general dimensions.
\textit{This section is thus the only part in the paper where we compute precise asymptotics for \emph{long-range} potential perturbations, where we no longer have access to the conservation laws.}
More precisely, we show in \cref{prop:even} that \cref{lemma:app:minkowski_index_improvement} fails entirely in odd spacetime dimensions, cf.~the discussion in \cref{sec:intro:motivation}.
In particular, we can infer that scattering solutions to the linearised Einstein vacuum equations around Minkowski with no incoming radiation are not conformally smooth towards $\scrip$, see \cref{rem:even:linearised_grav}.

    To avoid confusion with the results of e.g.~\cite{hollands_conformal_2004,godazgar_peeling_2012}, let us first discuss the notion of conformal smoothness in this context, as the reason why we don't have conformal smoothness is orthogonal do that in the cited works. For simplicity, we first restrict our attention to the behaviour of solutions, arising from compactly supported Cauchy data, in neighbourhood of a compact part of $\scrip$, $\D^{+,\mathrm{comp}}=\{u\in[u_1,u_2]\}$.
    Firstly, nonzero solutions to $\Box_{\mathbb{R}^{n+1}}\phi=0$ in Minkowski space with $n+1$ spacetime dimensions have nonzero radiation fields, with $\phi\sim r^{-(n-1)/2}$ in $\D^{+,\mathrm{comp}}$.
    Therefore, if $n\in2\mathbb N$, generic solutions trivially fail to be smooth with respect to $r^2\pv$ vector fields (as would be required by conformal smoothness). 

One could still hope for conformal smoothness in the sense that $\psi=r^{(n-1)/2}\phi$ is smooth with respect to $r^2\pv$. Indeed, for compactly supported data for the linear wave equation, this can easily be proved. 
Including nonlinearities, however, one immediately sees that this breaks down, see \cite{godazgar_peeling_2012}.
One could then try to instead measure smoothness with respect to $\sqrt{r^{-1}}$, i.e.~regularity with respect to the vector field $r^{3/2}\pv$. 
Indeed, using a change of variables $U=u,\, V=v^{-1/2}$, and then performing a simple local existence argument, it is easy to see that this smoothness holds for generic compactly supported data in $\D^{+,\mathrm{comp}}$ in a more robust way:
For instance, solutions to $\Box_{\mathbb{R}^{n+1}}\phi=V(r)\phi$, with $V(r)$ smooth with respect to $r^{3/2}\pv$ and $V(r)=\O(r^{-2})$, as well as solutions to $\Box_{\mathbb{R}^{n+1}}\phi=F[\phi]$, for $F\in\{\phi^2,\phi^3,(\partial_t\phi)^2\}$ for $n\geq4$, can be shown to respect this smoothness with respect to $r^{3/2}\pv$ (in $\D^{+,\mathrm{comp}}$).

The entire discussion above concerned solutions arising from compactly supported Cauchy data. The content of the present section is that for solutions with no incoming radiation and nontrivial behaviour towards $\scrim$, none of the above notions of smoothness apply. 
The failure of peeling proved in this section is therefore completely orthogonal to that discussed in \cite{hollands_conformal_2004, godazgar_peeling_2012}, see already \cref{prop:even}.
Since the proof extends to general $c/r^2$-potentials, we more generally prove this failure of peeling for the scale invariant linear wave equation in any dimension.

In \cref{sec:even:ode}, we prove the statement above by writing down certain ODE's satisfied by fixed angular mode solutions with specificied leading-order decay behaviour.
In \cref{sec:even:conj}, we make further comments on the difference between even and odd spacetime dimensions and  formulate \cref{conj:even:peeling} on the peeling property in even spacetime dimensions.

 \subsection{ODE analysis for fixed angular frequency solutions}\label{sec:even:ode}
In odd spacetime dimensions, we no longer have access to the conservation laws \cref{eq:intro:conslaw}, which so far constituted our main tool for computing solutions. We therefore resort to ODE analysis similar to that performed e.g.~in \cite{taujanskas_controlled_2023}.
We reiterate that generally, in the understanding of precise behaviour of $\phi$ in $\D$, such ODE analysis is merely the secondary step after performing the necessary propagation of polyhomogeneity statements as done in \cref{sec:prop} and \cref{thm:app:general}. (Of course, since the main point of this section is to show that peeling does not hold in any form in even space dimensions, this secondary step is actually enough.)
We also recall that the energy estimates and, therefore, the propagation of polyhomogeneity statements hold for arbitrary $1/r^2$ potentials as shown in \cref{thm:app3}.

 Let us first recall the following property of $\Box_{\mathbb R^{n+1}}$:
\begin{equation}
    r^2 r^{\frac{n-1}{2}}\Box_{\R^{n+1}}r^{-\frac{n-1}{2}}=-r^2\pu\pv+\Dl_{S^{n-1}}-\frac{(n-1)(n-3)}{4}.
\end{equation}
Thus, defining $\psi:= r^{\frac{n-1}{2}}\phi$, and recalling that the eigenvalues of $\Dl_{S^{n-1}}$ are given by $-\ell(\ell+n-2)$, $\ell\in\mathbb N$, we get
\begin{equation}\label{eq:even:wave}
    \Box_{\R^{n+1}}\phi+\frac{c\phi}{r^2}=0\implies (-r^2\pu\pv +\tilde{c})\psi_\ell=0, \text{     where       }    \tilde{c}=c-\ell(\ell+n-2)-\frac{(n-1)(n-3)}{4}.
\end{equation}
For later convenience and for comparison with the 3+1-dimensional case, we will often write $\tilde c=-l(l+1)$.
We note already that in even spacetime dimensions, where $n=2N+1$ for some $N\in\mathbb N$, we have
\begin{equation}
    \tilde{c}=c-(\ell(\ell+2N-1)-N(N-1)=c-(\ell+N)(\ell+N-1),
\end{equation}
from which it follows that all even spacetime dimensions are equivalent in the sense that an $\ell$-mode in $n=2N+1$ spatial dimensions behaves like an $(\ell+N-1)$-mode in $3$ spatial dimensions.

The main result of this section is the observation that  unless $n$ is odd and $c=0$, fixed angular frequency solutions to \cref{eq:even:wave} with no incoming radiation from $\scrim$ are in general not conformally smooth towards $\scrip$. More precisely, we prove
\begin{prop}\label{prop:even}
    Let $n\geq 1$. Unless $\tilde{c}$ in \cref{eq:even:wave} is of the form $\tilde{c}=-l(l+1)$ for $l\in\mathbb N$, solutions $\psi_{\ell}$ to \cref{eq:even:wave} arising from no incoming radiation and initial data $\psi_{\ell}|_{\incone}=r^{-p}$ for $p>-1$ are not conformally smooth towards $\scrip$, with the first nonzero term that is conformally irregular towards $\scrip$ appearing at order $r^{-p}$ if $p\notin\mathbb N$, and at order $r^{-p}\log r$ if $p\in\mathbb N$.
\end{prop}
\begin{cor}\label{cor:even}
    Unless $\tilde{c}$ in \cref{eq:even:wave} is of the form $\tilde{c}=-l(l+1)$ for $l\in\mathbb N$, then solutions to \cref{eq:even:wave} arising from compactly supported scattering data along $\scrim$ generically diverge logarithmically towards $\scrip$, $\lim_{v\to\infty}\psi/(\log r)\neq 0$.
\end{cor}
\begin{proof}[Proof of \cref{cor:even}]
It is an easy computation that for $\pv\psi^{\scrim}_{\ell}$ compactly supported along $\scrim$, there exists a null cone $\Cbar$ to the future of the support where the assumptions of \cref{prop:even} are satisfied with $p=0$. More precisely, we have $\psi_{\ell}|_{\Cbar}=\int_{v_0}^{\infty}\pv\psi^{\scrim}_\ell\dd v +O(r^{-1})$.
\end{proof}
\begin{rem}
    Notice that the first conformal irregularity appearing for $\tilde{c}\neq -l(l+1)$ for any $l\in\mathbb N$ is precisely the first one not excluded by \cref{thm:app:general} (though \cref{thm:app:general} of course didn't tell us that the coefficient of that irregularity is nonvanishing). See also \cref{rem:prop:suboptimal}.
\end{rem}
\begin{rem}\label{rem:even:linearised_grav}
    For linearised gravity around Minkowski in higher dimensions, certain components of the Weyl curvature still satisfy a Regge--Wheeler equation of the type \eqref{eq:even:wave}. For instance, using the notation of \cite{collingbourne_2022}, the Weyl tensor quantity~$r^{\frac{n-1}{2}}\accentset{\scalebox{.6}{\mbox{\tiny (1)}}}{{\varsigma}}$ satisfies \eqref{eq:even:wave} with $c=0$ when replacing $\phi$, as one swiftly derives from the system of equations presented in Section~2.10 therein.
 Thus, the failure of peeling exposed by \cref{prop:even} immediately extends to linearised gravity around Minkowski.
\end{rem}
\cref{prop:even} already shows the property of fixed angular frequency solutions to $\Box_{\mathbb{R}^{2N+1,1}}\phi=0$ to be conformally smooth towards $\scrip$ is truly special. We comment further on this in \cref{sec:even:conj}.
\begin{proof}
    We split the proof into two parts: First, we show the statement for $p\notin\mathbb N$, for which we will make use of an ODE motivated specifically by our understanding of the structure of solutions with no incoming radiation in $3+1$ dimensions. This ODE is derived by making the ansatz $\psi=|u|^{-p} \tilde\psi(u/r)$.
    While we could also use the same ODE to show the statement for $p\in\mathbb N$, it will be more convenient to use a more symmetric ODE for this case, derived by making the ansatz $\psi=r^{-p}\bar\psi(t/r)$.

\paragraph{The ODE geared towards conservation laws and the case $p\notin\mathbb N$}
\newcommand{\tpsi}{\tilde{\psi}}
Motivated by our understanding of solutions in $3+1$-dimensions, we make the following ansatz for solutions to \cref{eq:even:wave}: Let $\psi=|u|^{-p}\tilde\psi(x)$ for $x=u/r$.
It is then straight-forward to derive that \cref{eq:even:wave} implies the following ODE for $\tilde\psi(x)$:
\begin{equation}\label{eq:even:ODE1}
    -(x+x^2)\tpsi''+(p-1-2x)\tpsi'-\tilde{c}\tpsi=0.
\end{equation}
This ODE has regular singular points at $x=-1$ (corresponding to $\scrim$) and at $x=0$ (corresponding to $\scrip$). 

\textbf{Near $x=-1$,} making the ansatz $\tpsi=\sum_{k=0}^{\infty}a_k (1+x)^{R+k}$, we obtain the recurrence relations, writing $\tilde{c}=-l(l+1)$:
\begin{equation}\label{eq:even:rec1}
    a_{k+1}((R+k+1)(R+k+p+1))=a_k((R+k)(R+k+1)-l(l+1)) \implies \text{  indicial roots are  }R\in\{0,-p\}.
\end{equation}
\newcommand{\nir}{\tpsi^{-}_{\slashed{\nwarrow}}}
\newcommand{\ir}{\tpsi^{-}_{\nwarrow}}
\newcommand{\nor}{\tpsi^{+}_{\slashed{\nearrow}}}
\newcommand{\orr}{\tpsi^{+}_{\nearrow}}
We now assume that $-1<p\notin\mathbb N$. We note that the solution corresponding to $R=0$ has no incoming radiation; we denote it by $\nir$. Similarly, the solution corresponding to $R=-p$ has nontrivial incoming radiation, we denote it by $\ir$. Indeed: 

\begin{nalign}\label{eq:even:defofnirandir}
		R=-p\implies \tpsi\sim(1+x)^{-p}=(v/r)^{-p}\implies&\psi\sim |u|^{-p}(v/r)^{-p} &\implies  \pv \psi|_{\scrim}\neq0, &\quad \text{ incoming radiation, denote by }\ir .\\
		R=0\implies \tpsi\sim(1+x)^{0}\implies&\psi\sim |u|^{-p} &\implies \pv\psi|_{\scrim}=0,&\quad \text{ no incoming radiation, }\nir.
	\end{nalign}
Notice that $\psi=|u|^{-p}\ir$ is conformally regular towards $\scrim$, whereas $\psi=|u|^{-p}\nir$ is not.

Similarly, \textbf{near $x=0$,} making the ansatz\footnote{Note that the solution with $R=-p$ will then be complex as $x<0$, the bothered reader may feel free to take the real part. Also note that for $p\in(-1,-1/2)$, this is not a finite energy solution, but since we do explicit computations, we can still deal with this.} $\tpsi=\sum_{k=0}^{\infty}b_k x^{r+k}$, we get
\begin{equation}\label{eq:even:rec2}
     - b_{k+1}((R+k+1)(R+k-p+1))=b_k((R+k)(R+k+1)-l(l+1)) \implies \text{  indicial roots are  }R\in\{0,p\}.
\end{equation}
By similar considerations to the above, the case $R=p$ now corresponds to a solution with no outgoing radiation ($\pu\psi|_{\scrip}=0)$, and we denote in this case $\tpsi$ by $\nor$. Notice that $\psi=|u|^{-p}\nor$ is conformally irregular towards $\scrip$.
The case $R=p$, for which we will write $\orr$, corresponds to nontrivial outgoing radiation and is conformally regular towards~$\scrip$.

We now want to understand if solutions with no incoming radiation are conformally irregular towards $\scrip$: Since~$\nor$ is conformally irregular towards $\scrip$, this means that we need to understand the coefficient $c_2(p,l)$ in  $\nir=c_1(p,l) \orr+c_2(p,l)\nor$. 
Notice that if $l\in\mathbb N$, then the recurrence relation \cref{eq:even:rec1} for $R=0$ terminates at $k=\ell$, and so the solution is a polynomial, and it remains conformally regular towards $\scrip$, i.e.~$c_2(p,l)=0$ for any $l\in\mathbb N$. We claim that these are the only cases where this coefficient vanishes. 

In order to show this latter claim, we explicitly solve the recurrence relations \cref{eq:even:rec1} and \cref{eq:even:rec2}. We compute that for any $l\in\R$ (after suitable normalisation):
\begin{nalign}\label{eq:even:hypGeom_representation}
    x=-1,R=0\implies \tpsi=\nir&={}_2F_1(-l,1+l;1+p;1+x),\\
    x=-1,R=-p\implies \tpsi=\ir&=(1+x)^{-p}\cdot {}_2F_1(-p-l,1+l-p;1-p;1+x),\\
      x=0,R=0\implies \tpsi=\orr&={}_2F_1(-l,1+l;1-p;-x),\\
      x=0,R=p\implies \tpsi=\nor&=x^{p}\cdot {}_2F_1(p-l,1+l+p;1+p;-x).
\end{nalign}
One now easily shows (or looks up Eq.~07.23.17.0058.01  in \cite{Mathematica} or Eq.~15.8.4 in \cite{DLMF}) that 
\begin{nalign}
    \nir&={}_2F_1(-l,1+l;1+p;1+x)\\
    &=\frac{\Gamma(1+p)\Gamma(p)}{\Gamma(1+p+l)\Gamma(p-l)}\,{}_2F_1(-l,1+l;1-p;-x)
    + \frac{\Gamma(1+p)\Gamma(-p)}{\Gamma(-l)\Gamma(1+l)}x^{p}\cdot {}_2F_1(p-l,1+l+p;1+p;-x)\\
    &=c_1(p,l) \orr+c_2(p,l)\nor.
\end{nalign}
Clearly, $c_2(p,l)=\frac{\Gamma(1+p)\Gamma(-p)}{\Gamma(-l)\Gamma(1+l)}=0$ if and only if $l\in\mathbb N$. This proves the result for $p\notin\mathbb N$.\footnote{For the reader who'd rather not do computations as explicit as the above, we note that it is straight-forward to derive from the recurrence relations \cref{eq:even:rec1} in the case $r=0$ that $\lim_{x\to1}\nir=\infty$ if $p<0$ and $l\notin\mathbb N$, implying $c_2(p,l)\neq 0$. 
For $p>0$ and $p\notin\mathbb{N}$, it is then straight-forward to extend this result by using time integrals.}

\begin{obs}[Integer value of $p$]\label{rem:even:hypGeom_rep_integer}
    Note that the formulae for $\nir,\nor$ in  \cref{eq:even:hypGeom_representation} are also valid for $p\in\N$.
    We will make use of this fact in \cref{section:no incoming radiation Cauchy}.
    
    For $\ir,\orr$, the formulae do not extend to $p\in\N$ in general, as the meromorphic in $p$ function has poles there.
    See \S 15.2 \cite{DLMF} for more details on how to renormalise the function in this case and \S 14.3, \S 15.9 \cite{DLMF} for connection to Legendre polynomials.
   
\end{obs}
\begin{obs}[Antipodal matching]\label{rem:even:antipodal}
One may check from \cref{eq:even:hypGeom_representation} that if $p=0$, then $\lim_{x\to-1}\nir =1=(-1)^{\ell}\lim_{x\to 0} \nir$. 
\end{obs}

\paragraph{The ODE at spacelike infinity and the case $p\in\mathbb N$}
We now treat the case where $p\in\mathbb N$. We use the opportunity to do this using a slightly modified ODE that treats $\scrim$ and $\scrip$ symmetrically: 
We first introduce coordinates $\rho=1/r$ and $\tau=t/r$. In these coordinates, we have $\pu=\rho^2\partial_\rho+\rho(1+\tau)\partial_\tau$, $\pv=-\rho^2\partial_\rho+\rho(1-\tau)\partial_\tau$, and \cref{eq:even:wave} reads 
\begin{equation}
  \left(  -(1-\tau)^2\partial_\tau^2+2\partial_\tau+2\tau\partial_\tau\rho\partial_\rho+(\rho\partial_\rho)^2+\rho\partial_\rho+\tilde{c}\right) \psi_{\ell}=0.
\end{equation}
Since all $\partial_\rho$ derivatives come with exactly one power of $\rho$, and $\rho\partial_\rho r^{-p}=p r^{-p}$, we can make the ansatz $\psi=r^{-p}\bar\psi(\tau,\omega)$ to reduce the above to the following ODE (we drop the $\ell$-subscript):
\begin{equation}\label{eq:sharp:ode_at_i0}
 (-((1-\tau^2)\partial_\tau-2\tau-2\tau p)\partial_\tau+p(p+1)+\tilde{c})\bar\psi=0.
\end{equation}
\renewcommand{\nir}{\bar{\psi}^{-}_{\slashed{\nwarrow}}}
\renewcommand{\ir}{\bar{\psi}^{-}_{\nwarrow}}
\renewcommand{\nor}{\bar{\psi}^{+}_{\slashed{\nearrow}}}
\renewcommand{\orr}{\bar{\psi}^{+}_{\nearrow}}
The indicial roots of \cref{eq:sharp:ode_at_i0} at the regular singular points $\tau=\pm1$ are $R\in\{0,-p\}$.
\textbf{We assume from now that $p\in\mathbb N_{\geq0}$ and that $l\notin\mathbb N$:}
As in the previous section, the solution with no incoming radiation at $\tau=-1$ corresponds to $R=0$; the solution then being given by $\bar\psi=\sum_{k\geq0} c_k(1+\tau)^k=\nir$. 
The other solution is given by
\begin{equation}\label{eq:even:dummy}
    \bar\psi=c_{\log} \log(1+\tau)\nir+ (1+\tau)^{-p}\sum_{k\geq0} d_k(1+\tau)^k=\ir.
\end{equation}
By computing the recurrence relations coming from the ansatz $\bar\psi=\sum c_k (1+\tau)^k$,
\begin{equation}\label{eq:even:recursion}
		c_{k+1}2(k+1)(k+1+p)=c_k\Big(p(p+1)-l(l+1)+k(k+1+2p)\Big)=c_k\Big((p-l)(p+l+1)+k(k+1+2p)\Big),
	\end{equation}
 we directly see that the constant $c_{\log}$ in \cref{eq:even:dummy} is nonzero (again, for $l\notin\mathbb N$) since the RHS in \cref{eq:even:recursion} has roots $k=l-p,-(1+l+p)$ and thus never vanishes.
 Therefore, unlike in the $p\notin\N$ case, $\ir$ is the conformally irregular solution.

 Similarly, near $\scrip$, we have a solution of the form $\bar\psi=\sum_{k\geq0}  e_k (1-\tau)^k=\nor$,
 and another solution of the form $\bar\psi=c'_{\log}\log(1-\tau) \nor+(1-\tau)^{-p} \sum_{k\geq0}  f_k (1-\tau)^k=\orr$, where $c'_{\log}\neq 0$ by the same argument as above.

 Thus, as the limit $\lim_{\tau\to1}\nor$ is finite, in order to show that $c_2(p,l)\neq 0$ in $\nir=c_1(p,l)\nor +c_2(p,l)\orr$, it suffices to show that 
 \begin{equation}\label{eq:even:divergence}
     \lim_{\tau\to 1} \nir=\infty.
 \end{equation}
 For this, we simply take the recurrence relation \cref{eq:even:recursion} and evaluate at $(1+\tau)=2$:
 \begin{equation}\label{eq:even:recursion relation}
    \left. \frac{(1+\tau)^{k+1}c_{k+1}}{ (1+\tau)^kc_k}\right|_{\tau=1}=\frac{p(p+1)-l(l+1)+k(k+1+2p)}{(k+1)(k+1+p)}=1+\frac{-1+p}{k}+\O(1/k^2).
 \end{equation}
It follows that
\begin{equation}\label{eq:even:computing_coefficients}
    2^kc_k\sim \prod^k(1+\frac{p-1}{m})\sim e^{\sum^k\log(1+\frac{p-1}{m})}\sim e^{\sum^k\frac{p-1}{m}}\sim e^{(p-1)\log k}\sim k^{p-1}.
\end{equation}
Thus, the series $\sum_{k=0}^\infty 2^kc_k$ diverges if $p\geq0$, proving \cref{eq:even:divergence} and thus that $c_2(p,l)\neq 0$.
\end{proof}

\begin{rem}\label{rem:even:regular}

\newcommand{\tpsi}{\tilde{\psi}}
\newcommand{\nir}{\tpsi^{-}_{\slashed{\nwarrow}}}
\newcommand{\ir}{\tpsi^{-}_{\nwarrow}}
\newcommand{\nor}{\tpsi^{+}_{\slashed{\nearrow}}}
\newcommand{\orr}{\tpsi^{+}_{\nearrow}}
    We note that, for $l\in\mathbb{N}$, $p\in\mathbb{N}$, \eqref{eq:even:dummy} is still a solution, but the constant $c_{\log}$ will now be vanishing if and only if $\abs{p}>\ell$ since the RHS of \cref{eq:even:recursion} has a zero at $k=l-p,-(1+p+l)$.
Thus, if $p>l$, $p,l\in\mathbb{N}$, then both solutions will be conformally regular towards $\scrip$!
To give examples, let $p=1$.
For $l=0<p$, we have $\nir=1/u,\nor=1/v$ (both conformally regular).
In contrast,  for $l=1$, we have $\nor=\nir=1/r$ and $\ir=\orr=t/(uv)-2\log(u/v)/r$.
For a discussion concerning the same ODEs, see also \cite{gasperin_asymptotics_2024-1}.
\end{rem}

\subsection{A peeling conjecture for \texorpdfstring{$\Box_{\R^{2N+1,1}}\phi=0$}{even spacetime dimensions} and the contrast to odd spacetime dimensions}\label{sec:even:conj}
\cref{prop:even} shows that if $n$ is odd, then fixed angular frequency solutions $\psi_{\ell}$ to $\Box_{\R^{n+1}}\phi=0$ arising from data of the form $\psi_{\ell}^{\incone}=r^{-p}$ and no incoming radiation are conformally smooth towards $\scrip$, i.e.~$\psi_{\ell}|_{\C_u}\in \A{phg}^{\overline{(0,0)}}(\C_u)$ for all $u\leq u_0$. 
Using the exact conservation laws, or the explicit representation formulae for fixed angular mode solutions, the same is true for data $\psi_{\ell}^{\incone}\in \Hb^{a;\infty}(\incone)$ for any $a\geq-1/2$; indeed, this is what we showed in \cref{lemma:app:minkowski_index_improvement}.

We have seen in \cref{prop:even} that the natural attempt to generalise this to odd spacetime dimensions or to the inclusions of $1/r^2$-potentials fails. 
A related question is whether, even in 3+1 dimensions, the property of fixed angular modes to always be conformally smooth towards $\scrip$ if there is no incoming radiation carries over to the resummed solution. 
A first observation with regards to this question is that this is certainly not the case if the initial data are of finite angular regularity.
Indeed, in that case, it is easy to construct explicit solutions where the higher-order terms in the expansions of fixed angular mode solutions are no longer summable in $\ell$ (using for instance the expression (12.3) combined with (10.64) of \cite{kehrberger_case_2024}):
\begin{lemma}\label{lemma:finiteregularitypeeling}
 Let $N\in\mathbb N$ and let $\vec{a}$ be admissible.   Let $\psi^{\incone}\in\Hb^{a_-;1}(\incone)$, and let $\psi\in\Hb^{\vec{a}}(\Dbold)$ be the corresponding scattering solution to $\Box_{\mathbb{R}^{2N+1,1}}\phi=0$ with no incoming radiation.
    Then, the spherical projections $\psi_{\ell}$ satisfy for any $K\in\mathbb N$
    \begin{equation}
        \psi_{\ell}|_{\outcone{}}\in
        \A{phg}^{\mindex{0}}(\outcone{})\implies\psi_{\ell}=\sum_{k\leq K}\psi_{\ell,k}r^{-k}+\Hb^{K;\infty}(\outcone{}).
    \end{equation}
    However, in general, $\psi_{\ell,k}\notin l_{\ell}^2(S^2)$ for $k$ suitably large (depending on $a_0$).
\end{lemma}
At the same time, we know from \cref{thm:app:general} combined with \cref{lemma:app:minkowski_index_improvement} that if the initial data for $\psi$ are polyhomogeneous and of infinite regularity, then the solution is conformally smooth towards $\scrip$.
The case still shrouded in mystery is when the initial data do \textit{not} have an expansion but are of \textit{infinite} regularity; this is the content of the conjecture below. 
(We note that a task slightly simpler than (but related to) proving the conjecture below would be to show that solutions arising from data $\psi^{\incone}\in \Hb^{a_-;\infty}(\incone)$ with $a_-<0$ have a finite radiation field towards $\scrip$, i.e. a finite limit $\lim_{v\to\infty}\psi$ towards~$\scrip$.)
\begin{conj}\label{conj:even:peeling}
    a) Let $N\in\mathbb N_{\geq1}$, and let $\phi$ be the solution to $\Box_{\R^{2N+1,1}}\phi=0$ with no incoming radiation and data $\psi^{\incone}\in \Hb^{a_0;\infty}(\incone)$ for $a_0\geq -1/2$. Then, along any outgoing null cone $\C_u$ with $u\leq u_0$, we have
    \begin{equation}
        \psi|_{\C_u}\in \A{phg}^{\overline{(0,0)}}(\C_u).
    \end{equation}
    b) Now assume that $\psi^{\incone}\in\Hb^{a_0;k}(\incone)$ for $a_0\in[-1/2,0)$ and some $k\in\mathbb N_{\geq1}$, then, for any $u\leq u_0$:
  \begin{equation}\label{eq:peeling:conj_no_inclusion}
        \psi|_{\outcone{}}\in \A{phg}^{\mindex{0};0}(\C_u)+\Hb^{\lfloor\frac{k-1}{2}\rfloor;0}(\C_u),
    \end{equation}
    but, in general:
    \begin{equation}\label{eq:peeling:conj_inclusion}
        \psi|_{\outcone{}}\notin \A{phg}^{\mindex{0};0}(\C_u)+\Hb^{\lceil \frac{k}{2}\rceil;0}(\C_u),
    \end{equation}
    where we used the limited regularity notation for the polyhomogeneous space from \cref{def:notation:polyhom}.\footnote{Note that the coefficients $S_{\ell,a_0,\ell,n,0}$ defined in (10.64) of \cite{kehrberger_case_2024} (and appearing in the expansion $\psi_{\ell}=r_0^{-a_0}\sum_{n\geq0} S_{\ell,a_0,\ell,n,0}\frac{r_0^n}{r^n}\dots$) asymptotically grow like $\ell^{2(n-a_0)}$ as $\ell\to\infty$, which is consistent with \cref{eq:peeling:conj_no_inclusion,eq:peeling:conj_inclusion} .}
\end{conj}

\paragraph{A compactified depiction of the conjecture}
It is instructive to consider the content of the conjecture in a compactified setting. 
The following picture is valid for many different choices of compactifications (of the null coordinates); we here choose the simplest possible one:
    We perform a a coordinate change $\mu=-u^{-1},\nu=-v^{-1}$:
    \begin{equation}
        (\pu\pv-r^{-2}\Dl)r\phi=\mu^2\nu^2(\partial_\nu\partial_\mu-(\mu-\nu)^{-2}\Dl)r\phi=0.
    \end{equation}
    We denote $\nu(v_0)=\nu_0$. 
    Thus, the solution from the conjecture above is defined for $\nu\geq \nu_0$, $\mu\geq 0$. We may extend it to the past of $\mu=0$ by zero; this is consistent with no incoming radiation but introduces a singularity as we cross $\mu=0$.
    We extend the solution to the past of the lightcone $\mu=0$ by $r\phi=0$.
    We may think of this extended solution as arising from initial data along $\nu=\nu_0$ (terminating at $\mu=\mu_0$). 
    Then, these data are smooth away from $\mu=0$, and conormal as $\mu=0$ is approached from $\mu>0$, and smooth as $\mu=0$ is approached from $\mu<0$.  
By standard propagation of singularities arguments, it follows that for any $\mu'>0$, $r\phi|_{\mu=\mu'}$ is smooth away from $\nu=0$ and is conormal at $\nu=0$ from \textit{both} $\nu>0$ and $\nu<0$.
Now, in this picture, \cref{lemma:app:minkowski_index_improvement} combined with \cref{thm:app:general}  shows that if the data are polyhomogeneous for $\mu\geq 0$, then for any $\mu'>0$, $r\phi|_{\mu=\mu'}$ is not only conormal for $\nu\leq 0$, but, indeed, smooth. 
\cref{prop:even}, on the one hand, shows that this property only holds in even spacetime dimensions (and fails in odd spacetime dimensions). 
On the other hand, \cref{conj:even:peeling} posits that, in even spacetime dimensions, this property generalises to weakening the assumption of polyhomogeneity to conormality. See also \cref{fig:conj}.
 \begin{figure}[htbp]
 \centering
\includegraphics[width=0.54\textwidth]{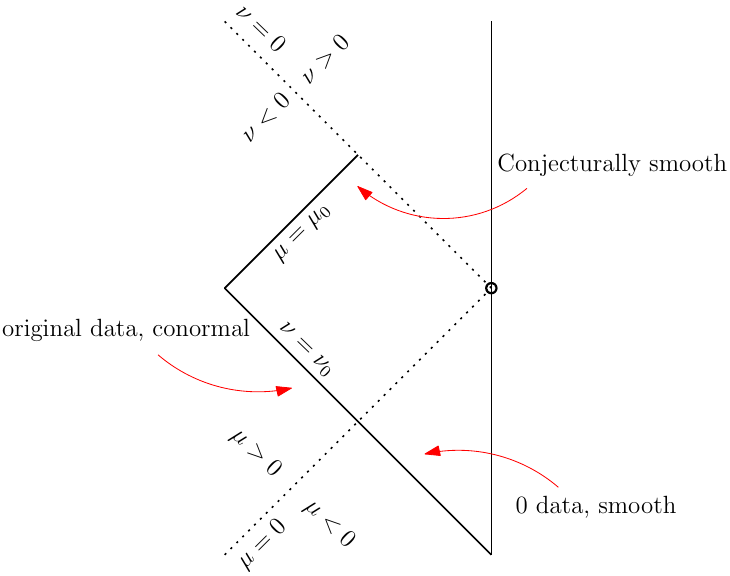}
\caption{Depiction of the compactification w.r.t.~$\mu,\,\nu$. In even spacetime dimensions, if the data at $\nu=\nu_0$ are polyhomogeneous for $\mu\geq0$, then the solution will be smooth towards $\nu=0.$ We conjecture that the same holds for data that are only conormal. In odd spacetime dimensions, this fails completely. }
\label{fig:conj}
\end{figure}

 
	\newpage
	\section{The no incoming radiation condition on Cauchy data}\label{section:no incoming radiation Cauchy}
    Let us recall that the entire discussion of the peeling property started exactly with the aim of excluding incoming radiation from $\scrim$, cf.~\cref{sec:intro:peeling}. (In particular, there was never a reason to expect the peeling property to hold in the presence of nontrivial incoming radiation!)
    In the context of the wave equation, the no incoming radiation condition is given exactly by $\pv\psi^{\scrim}=0$ as a condition \textit{on past null infinity}.
    In the present section, we want to discuss equivalent and approximate formulations of the no incoming radiation as a condition on \textit{Cauchy data} $(\psi_0,\psi_1)$ on a Cauchy hypersurface $\Sigma$ (such that the arising solution $\phi$ to the relevant equation $\Box_\eta\phi=\mathcal{P}[\phi]$ satisfies $r\phi|_{\Sigma}=\psi_0$, $T(r\phi)|_{\Sigma}=\psi_1$). 
    
    Let us begin by formulating the two main requirements we ask such of alternative formulations of the no incoming radiation condition.

        \begin{enumerate}[label=({\Alph*})]
        \item\label{item:noinc:A} {\textbf{Simplicity and practicability:}} We want to be able to determine if $\phi$ satisfies (some version of) the no incoming radiation condition by only checking computationally simple properties of the Cauchy data $\psi_0,\psi_1$, as opposed to constructing the solution $\phi$ in $\D^-$.
        \item\label{item:noinc:B} {\textbf{Nontrivial implications:}} We require that if $\psi_0,\psi_1$ satisfy the condition, one can prove improved regularity of $\phi$ towards $\scrip$.\footnote{This regularity in turn can have consequences for the decay towards $i^+$ as discussed in \cite{gajic_relation_2022}.}
    \end{enumerate}
    Concerning \cref{item:noinc:A}, we, in particular, want to avoid making the obvious choice of \cref{defi:noinc:basic}\cref{item:noinc:defa} below.
    \begin{defi}[No incoming radiation]\label{defi:noinc:basic}
        Smooth initial data $(\psi_0,\psi_1)$ to $\Box_\eta\phi=\mathcal{P}[\phi]$ satisfy 
        \begin{enumerate}[label={\alph*})]
            \item \label{item:noinc:defa} the \emph{no incoming radiation condition} if the corresponding solution $\phi$ satisfies $v\pv\psi|_{\scrim}=0$.
            \item \label{item:noinc:defb} the \emph{no incoming radiation condition to order $q$} if the corresponding solution $\phi$ satisfies $v\pv\psi|_{\scrim}\in\Hb^{q}(\scrim)$.
        \end{enumerate}
    \end{defi}
    On the other hand, for general $\mathcal{P}[\phi]$ (and already for the linear wave equation on Schwarzschild), we cannot hope to practicably distinguish initial data resulting in no incoming radiation from very fast decaying incoming radiation. 
Thus, in general, we can, at, best hope to establish an equivalent formulation of \cref{defi:noinc:basic}\cref{item:noinc:defb}.
    
    Concerning \cref{item:noinc:B}, we recall from \cref{cor:app:sharp_index_sets} that polyhomogeneous solutions satisfying the no incoming radiation condition have improved index sets, i.e.~improved conformal regularity, towards $\scrip$.

    In \cref{sec:noinc:mink} below, we will first focus on an alternative formulation of \cref{defi:noinc:basic} for $\Box_\eta\phi=0$ in Minkowski. In the process, we will also show that solutions arising from data on $t=0$ that are smooth in $1/r$ generically contain nonvanishing log-terms towards $\scrip$. 
    In \cref{sec:noinc:Schw}, we will then generalise this to Schwarzschild. 
These two cases already containing the main ideas, we leave it to the interested reader to attempt further generalisations.

    \subsection{The no incoming radiation condition on Minkowski}\label{sec:noinc:mink}
    For concreteness, we study the Cauchy problem (with smooth data)
    \begin{equation}\label{eq:noinc:Cauchymink}
        \Box_\eta\phi=0,\qquad \Sigma=\{t=0\}\cap\D,\qquad (\psi|_{\Sigma},T\psi|_{\Sigma})=(\psi_0,\psi_1)\in \Hb^{-1/2}(\Sigma)\times \Hb^{1/2}(\Sigma).
    \end{equation}
In Minkowski space, the no incoming radiation condition \cref{defi:noinc:basic}\cref{item:noinc:defa} for \textit{finite-energy} solutions to $\Box_\eta\phi=0$ on fixed $\ell$-modes is equivalent to $\pv(r^2\pv)^\ell(r\phi_\ell)\equiv 0$, as follows from the conservation laws \cref{eq:intro:conslaw}.\footnote{The finite-energy requirement is so that we exclude solutions with $(\psi_{\ell})|_{\scrim}=\sum_{i=1}^{\ell}c_iv^i$ along $\scrim$. Of course, we can remove the assumption if we work instead with our generalised condition of no incoming radiation from \cref{def:scat:weak_decay}.}

    Similarly, \cref{defi:noinc:basic}\cref{item:noinc:defb} is equivalent $\pv(r^2\pv)^{\ell}(\psi_{\ell})|_{\Sigma}\in\Hb^{q+1-\ell}(\Sigma)$:

    \begin{lemma}\label{lemma:noinc:mink1}
        Initial data $(\psi_0,\psi_1)$ to \cref{eq:noinc:Cauchymink} satisfy the no incoming radiation condition to order $q$ if, for all $\ell\in\mathbb N$, the arising solution satisfies $(\pv(r^2\pv)^{\ell}P_{\ell}^{S^2}\psi)|_{\Sigma}\in\Hb^{q+1-\ell}(\Sigma)$.
    \end{lemma}

    Concerning \cref{item:noinc:B}, we note that in view of the exact conservation law \cref{eq:intro:conslaw}, \cref{lemma:noinc:mink1} also allows to conclude better regularity towards $\scrip$, namely $\psi_{\ell}|_{\outcone{1}}\in\A{phg}^{\mindex{0}}(\outcone{1})+\Hb^{q}(\outcone{1})$ for any $u\leq u_0$.
   If the initial data are polyhomogeneous, then one can conclude the same for the summed solution $\psi$, cf.~\cref{cor:app:minkowski_with_error}.

However, with regards to \cref{item:noinc:A}, \cref{lemma:noinc:mink1} still leaves us with having to do a fair amount of computations in order to express $(\pv(r^2\pv)^{\ell}\psi_{\ell})|_{\Sigma}$ in terms of $(\psi_0,\psi_1)$. 
In the \emph{very special} case where the initial data have an expansion with no logarithms, we can use our results from \cref{sec:sharp} to circumvent these computations.\footnote{We emphasise once again that as discussed in \cref{rem:intro:conovsconf}, such an expansion is \emph{not} expected to hold for physical scenario, such as the hyperbolic orbit of $N$ infalling masses. 
We leave the generalisation of \cref{lemma:noinc:fourier_Minkowksi} for more general polyhomogeneous data to the reader.}
Note already that we comment on a more geometric version of \cref{eq:noinc:Fourier_multiplier_on_data} below in \cref{rem:noinc:1}.
 \begin{lemma}\label{lemma:noinc:fourier_Minkowksi}
    Assume that the initial data to \cref{eq:noinc:Cauchymink} satisfy $\psi_0,r\cdot \psi_1\in \A{phg}^{\E}(\Sigma)$ for some index set $\E$ containing no $(p,k)$ with $k>0$ and $\min(\E)>-1/2$.
    We write $(\psi_0,\psi_1)=\sum (r^{-p}\psi_0^{(p)}(\omega),r^{-p-1}\psi_1^{(p)}(\omega))$ for the expansion.
    Then there exist $k\in\mathbb{N}$ and Fourier multiplier operators of order $k$, $Q^p_0,Q^p_1:H^{s}(S^2)\to H^{s-k}(S^2)$ for any $s\in\mathbb{N}$, such that $(\psi_0,\psi_1)$ satisfies the no
    incoming radiation condition to order $q$ if and only if $Q^p_0\psi^{(p)}_0+Q^p_1\psi^{(p)}_1=0$ for all $p\leq q$.
    We have an explicit form for $Q^p_0\psi^{(p)}_0+Q^p_1\psi^{(p)}_1=0$ as
    \begin{equation}\label{eq:noinc:Fourier_multiplier_on_data}
        P^{S^2}_\ell\psi_0^{(p)}\underbrace{\Big(p\cdot {}_2F_1[-\ell,1+\ell,1+p,1/2]-\frac{\ell(1+\ell)}{4(p+1)}{}_2F_1[1-\ell,2+\ell,2+p,1/2]
        \Big)}_{:= \mathrm{LHS}(p,\ell)}=P^{S^2}_\ell\psi_1^{(p)}\underbrace{{}_2F_1[-\ell,1+\ell,1+p,1/2]}_{:=\mathrm{RHS}(p,\ell)}.
        \end{equation}
    \end{lemma}\label{lemma:Cauchy:fourier_Minkowksi}
    Before the proof, let us quickly discuss the difference between \cref{lemma:noinc:mink1,lemma:noinc:fourier_Minkowksi}.
  While the first applies for a larger class of data, only restricted by regularity and decay assumptions, it is unclear weather \cref{item:noinc:A} holds in this case. (Of course, one can rewrite $\pv(r^2\pv)\psi$ in terms of combinations of $\partial_r^n$ derivatives of $\psi_0$ and $\psi_1$ along $\Sigma$ for $n\leq \ell+1$.)
    In contrast, the second lemma gives fixed $k\in\N$ (independent of $l,p$) order Fourier multipliers such that we can immediately test the condition in the case of \textit{data with $r^{-p}$ expansions}. 
    
    \begin{proof}
        Let us first prove the existence of the Fourier multiplier operators.
        For $\psi_0^{(p)},\psi_1^{(p)}\in\Hb^{k}(S^2)$ and $k$ sufficiently large, we can use \cref{thm:scat:scat_general,thm:app:general} to conclude that
        \begin{equation}
            v\pv\psi|_{\scrim}=\sum_{p}v^{-p}\psi_-^{(p)}+\Hb^{q_b;0}(\scrim)
        \end{equation}
        for $\psi_-^{(p)}\in H^{0}(S^2)$.
        Using linearity and that $\Box_\eta$ commutes with spherical projections, it follows that $\psi_-^{(p)}=Q^p_0\psi^{(p)}_0+Q^p_1\psi^{(p)}_1$ for some multipliers.
        Indeed, note that $\psi_-^{(p)}$ cannot depend on $\psi_i^{(q)}$ for $q\neq p$ since $\Box_\eta$ preserves scaling towards~$I^0$.
    
        \newcommand{\tpsi}{\tilde{\psi}}
         \newcommand{\nir}{\tpsi^{-}_{\slashed{\nwarrow}}}
        \newcommand{\ir}{\tpsi^{-}_{\nwarrow}}
        Next, in order to prove the more specific \cref{eq:noinc:Fourier_multiplier_on_data}, it suffices to consider a single term at fixed order $p\leq q$ in the expansion and project to fixed $\ell$-mode.  We therefore consider $\psi$ to have initial data $r^{-p}P^{S^2}_\ell(\psi_0^{(p)}(\omega),r^{-1}\psi_1^{(p)}(\omega))$, and we will show that \cref{eq:noinc:Fourier_multiplier_on_data} is equivalent to $\psi$ not having incoming radiation in the sense of \cref{defi:noinc:basic}\cref{item:noinc:defa}.
        By \cref{sec:even:ode}, we may write the general form of the solution as $\psi=\abs{u}^{-p}(a \nir(u/r)+b \ir(u/r))$, where $\nir$ and $\ir$ are as defined in \eqref{eq:even:defofnirandir}.
        
        Let us first prove that no incoming radiation implies \cref{eq:noinc:Fourier_multiplier_on_data}:
         Enforcing the no incoming radiation condition implies that $b=0$, and without loss of generality we set $a=1$.
        Now, using \cref{rem:even:hypGeom_rep_integer} and $b=0$, we know that the exact representation from (the first of) \cref{eq:even:hypGeom_representation} is valid for $\nir$.
        We next differentiate $\psi=\abs{u}^{-p}\nir(u/r)$ in $t$ (recall $2T=\pu+\pv$) to get $2T\psi=p|u|^{-p-1}\nir+r^{-1}|u|^{-p}{\nir}'$.
        Using now (15.5.1) of \cite{DLMF}:
        \begin{equation}
            \partial_d \big({}_2F_1(a,b,c,d)\big)=\frac{ab}{c}{}_2F_1(a+1,b+1,c+1,d),
        \end{equation}
        then \cref{eq:noinc:Fourier_multiplier_on_data} follows from multiplying (note that $u/r=-1/2$ at $t=0$)
        \begin{equation}\label{eq:noinc:proof}
       \frac12     \psi_0^{(p)}\cdot \Big(p|u|^{-p-1}\nir+r^{-1}|u|^{-p}{\nir}'\left.\Big)\right|_{\frac{u}{r}=-\frac{1}{2}}
       =r^{-1}\psi_1^{(p)}\cdot \abs{u}^{-p}\left.\nir\right|_{\frac{u}{r}=-\frac{1}{2}},
        \end{equation}
        and then multiplying \cref{eq:noinc:proof} by $|u|^{p+1}$, using again that $|u|/r=1/2$.

        Next, we show that \cref{eq:noinc:Fourier_multiplier_on_data} implies that $b=0$
        in the representation $\psi=\abs{u}^{-p}(a \nir+b \ir)$.
        Under these assumptions, then \cref{eq:noinc:proof} holds with $a$ replaced by $b$ and with $\nir$ replaced by $\ir$.
        Unless $\partial_t\nir|_{t=0}=\nir|_{t=0}=0$, which would force $\nir=0$ and thus $b=0$ by uniqueness, this then produces another linear relation between $\psi_0^{(p)}$ and $\psi_1^{(p)}$. 
        If this linear relation differs from \cref{eq:noinc:Fourier_multiplier_on_data}, then $b=0$. But if not, then by linearity, the arising solution is a constant multiple of $\nir$ and so $b=0$. This completes the proof.
    \end{proof}

    \begin{cor}\label{cor:noinc}
        Consider initial data to \cref{eq:noinc:Cauchymink} satisfying $\psi_0,r\psi_1\in \A{phg}^{\mindex0}(\Sigma)$, then the solution will generically be conformally irregular towards $\scrip$. 
More precisely, let $(\psi_0,\psi_1)=(r^{-p}\psi_0^{(p)}(\omega),r^{-p-1} \psi_1^{(p)}(\omega))$
for $p\in\mathbb{N}$.
Then the solution $\psi$ has no logarithmic terms towards $\scrip$ if and only if  \cref{eq:noinc:Fourier_multiplier_on_data} holds for all $\ell\geq p$ (otherwise, the first logarithmic term towards $\scrip$ will appear at order $r^{-p}\log r$).
    \end{cor}
    \begin{proof}
    \newcommand{\nir}{\bar{\psi}^{-}_{\slashed{\nwarrow}}}
\newcommand{\ir}{\bar{\psi}^{-}_{\nwarrow}}
Let $\psi$ be the solution arising from $(\psi_0,\psi_1)=(r^{-p}\psi_0^{(p)}(\omega),r^{-p-1} \psi_1^{(p)}(\omega))$, $p\in\mathbb{N}$.
        We know from \cref{lemma:noinc:fourier_Minkowksi} that if \cref{eq:noinc:Fourier_multiplier_on_data} does not hold, then the solution has nontrivial incoming radiation. In turn, this means that when writing $\psi=a\nir+b\ir$ for $a,b\in\mathbb{R}$ and $\nir,\ir$ as in \cref{eq:even:dummy}, then $b\neq0$.
On the other hand, we know from \cref{rem:even:regular} that $\ir$ is conformally irregular towards $\scrip$ if and only if $\ell\geq p$, with the conformally irregular term appearing at order $r^{-p}\log r$.
\end{proof}

\begin{rem}[Better understanding the condition \cref{eq:noinc:Fourier_multiplier_on_data}]\label{rem:noinc:1}
    The conditions \cref{eq:noinc:Fourier_multiplier_on_data} are slightly larger than one might expect, since either of  $\mathrm{RHS}(p,\ell)$ and $\mathrm{LHS}(p,\ell)$ can vanish for certain $p,\ell$. 
    Understanding this better allows us to give a more geometric characterisation of \cref{eq:noinc:Fourier_multiplier_on_data}, see already \cref{rem:noinc:2} below.
    First, it is instructive to compute the first few values of the ratios $\mathrm{RHS}(p,\ell)/\mathrm{LHS}(p,\ell)$:
    \begin{table}[htpb]
	\centering
$\begin{array}{c|ccccccccccc}
	p\text{  vs.  } \ell&\ell=0	&\ell=1 	&\ell=2	&\ell=3	&\ell=4	&\ell=5	&\ell=6	&\ell=7	&\ell=8	&\ell=9&\ell=10\\
	\hline
	p=0&\infty & 0 & \infty & 0 & \infty & 0 & \infty & 0 & \infty & 0 & \infty \\
	p=1&	1 & 2 & 0 & 2 & 0 & 2 & 0 & 2 & 0 & 2 & 0 \\
	p=2&	\frac{1}{2} & \frac{4}{7} & 1 & 0 & 1 & 0 & 1 & 0 & 1 & 0 & 1 \\
	p=3&	\frac{1}{3} & \frac{6}{17} & \frac{16}{39} & \frac{2}{3} & 0 & \frac{2}{3} & 0 & \frac{2}{3} & 0 & \frac{2}{3} & 0 \\
	p=4&	\frac{1}{4} & \frac{8}{31} & \frac{5}{18} & \frac{32}{99} & \frac{1}{2} & 0 & \frac{1}{2} & 0 & \frac{1}{2} & 0 & \frac{1}{2} \\
	p=5&	\frac{1}{5} & \frac{10}{49} & \frac{16}{75} & \frac{70}{303} & \frac{256}{955} & \frac{2}{5} & 0 & \frac{2}{5} & 0 & \frac{2}{5} & 0 \\
\end{array}
$
\caption{A table giving the first few values of the ratio $\mathrm{RHS}(p,\ell)/\mathrm{LHS}(p,\ell)$.}\label{table}
	\end{table}

\textbf{For $p=0$}, we can show that $\mathrm{RHS}(0,\ell)=0$ for $\ell$ odd, and that $\mathrm{LHS}(0,\ell)=0$ for $\ell$ even; in other words, the condition \cref{eq:noinc:Fourier_multiplier_on_data} is satisfied so long as $\psi^{(0)}_0$ is an even function and $\psi^{(0)}_1$ is an odd function on the sphere, see already \cref{rem:noinc:2} below.

\textbf{More generally, for $p>0$,} we expect the following picture, which would be a consequence of\footnote{The statement \cref{eq:noinc:guess} is merely motivated by numerical computations. In fact, the base cases $n=0$ can be proved directly, and we expect that case of general $n$ should then follow from standard recurrence properties of the hypergeometric functions, but we did not attempt a proof.} 
\begin{align}\label{eq:noinc:guess}
    \mathrm{RHS}(p,p+2n+1)=0 \,\,\forall n\in\mathbb{N}_{\geq0},&& \frac{\mathrm{RHS}(p,p+2n)}{\mathrm{LHS}(p,p+2n)}=\frac{2}{p}\,\,\forall n\in\mathbb{N}_{\geq0}:
\end{align}

1) For $\ell<p$, there are \textit{nontrivial constraints} between $P_\ell^{S^2}\psi_0^{(p)}$ and $P_\ell^{S^2}\psi_1^{(p)}$. (Cf.~the $\ell<p$ triangle in \cref{table}.)

2a) \textbf{If $p\leq \ell$ is even:} Then, by the first of \cref{eq:noinc:guess}, $\psi_0^{(p)}$ must be an even function on $S^2$, the odd part of $\psi_1^{(p)}$ is arbitrary, and we require that
\begin{equation}
    \text{the even part of  } (\psi_0^{(p)}-\frac2p \psi_1^{(p)}) \text{ vanishes.}
\end{equation}

2b) \textbf{If $p\leq \ell$ is odd:} Then, by the first of \cref{eq:noinc:guess}, $\psi_0^{(p)}$ must be an odd function on $S^2$, the even part of $\psi_1^{(p)}$ is arbitrary, and we require that
\begin{equation}
    \text{the odd part of  } (\psi_0^{(p)}-\frac2p \psi_1^{(p)}) \text{ vanishes.}
\end{equation}
We note that the combination $(\psi_0^{(p)}-\frac2p \psi_1^{(p)})$ corresponds exactly to the restriction of $2(\partial_r+2T)\psi= (3\pv+\pu)\psi$ to $\{t=0\}$, but we have, at the moment, no interpretation for this. See, however, also \cite[Proposition 1]{gasperin_asymptotics_2024-1}.
\end{rem}

\begin{rem}[Comparison to the results of \cite{lindblad_scattering_2023}]\label{rem:noinc:2}
    In the special case $p=0$, the work \cite{lindblad_scattering_2023} already gave a nice geometric characterisation of initial data satisfying the no incoming radiation condition to leading order. Indeed, this question is equivalent to demanding that the initial data (given by $M$ and $N$ in their notation) lie in the kernel of the maps $\mathcal{F}[N],\, \mathcal{G}[M]$ in their (1.29). 
    Now, the kernel of $\mathcal{G}$ is given by even functions, and the kernel of $\mathcal{G}$ is given by odd functions on $S^2$.
In other words, for $p=0$, initial data satisfy the no incoming radiation condition if $\psi_0^{(p)}$ is even and $\psi_1^{(p)}$ is odd.
Similarly, the condition for $p=1$ from \cref{rem:noinc:1} may also be extracted from \cite{lindblad_scattering_2023}.
\end{rem}
\begin{rem}[Better understanding \cref{cor:noinc}]
As explained in \cref{rem:noinc:1}, we expect  \cref{eq:noinc:Fourier_multiplier_on_data} to have a nice geometric interpretation concerning the even and odd parts of $\psi_0^{(p)}$ and $\psi^{(p)}_1$ in addition to some nontrivial constraints for $\ell<p$. However, since the case $\ell<p$ is excluded from the condition in \cref{cor:noinc}, we expect that a more geometric characterisation of initial data not leaving to $\log$-terms towards $\scrip$ is given simply by the conditions 2a), 2b) above.
We also note that the result of \cref{cor:noinc} is essentially identical to  \cite[Proposition 1]{gasperin_asymptotics_2024-1}, where a coordinate system is used under which the $\ell\geq p$ conditions take a simpler form.
\end{rem}
\subsection{The no incoming radiation condition on Schwarzschild}\label{sec:noinc:Schw}
For perturbations of \cref{eq:noinc:Cauchymink}, understanding no incoming radiation on a Cauchy level is harder as one no longer has access to the exact conservation law $\pv(r^2\pv)^{\ell}\psi_\ell=0$, for instance. 
Nevertheless, we can still define alternative formulations of the no incoming radiation condition by perturbing around \cref{eq:noinc:Cauchymink}. We illustrate the ideas for the wave equation in Schwarzschild. 
As we have already motivated in \cref{sec:Schw:bettercoords}, this is significantly easier when working in $\bar{u}$, $\bar{v}$ coordinates, as the perturbation around Minkowski in those coordinates does not exhibit any spurious logarithms (in other words, the coordinates capture the logarithmic divergence of the light cones).
The Cauchy problem we study then is
 \begin{equation}\label{eq:noinc:CauchySchw}
        \bar{\Box}_\eta\phi=(\bar{\Box}_{\eta}-D_0^{-1}\bar\Box_g)\phi,\qquad \Sigma=\{\bar{u}=-\bar{v}\}\cap\D,\qquad (\psi|_{\Sigma},\tfrac12(\pubar+\pvbar)\psi|_{\Sigma})=(\psi_0,\psi_1)\in \Hb^{1/2}(\Sigma)\times \Hb^{-1/2}(\Sigma),
    \end{equation}
    where $\bar{\Box}_\eta,\,\bar{\Box}_g$, $D_0$, $\bar{u}$ and $\bar{v}$ are defined in \cref{sec:Schw:bettercoords}.
    We emphasise that $\Sigma$ is logarithmically shifted compared to the hypersurface $t=0$!\footnote{We may also deduce from \cref{lemma:noinc:fourier_Sch,rem:noince:Sch_l=0}, that solutions with no incoming radiation are \emph{not} conformally smooth on $\{t=0\}$ hypersurface, and that a \emph{no-incoming} radiation data ansatz on $\{t=0\}$, similar to \cref{eq:noinc:data_p+1_noinc}, \emph{must} contain logarithmic terms.}

For simpler notation, we assume that $(\psi_0,\psi_1)=\sum_{n=0}(r^{-p-n}\psi_0^{(p+n)},r^{-p-n-1}\psi_1^{(p+n)})$ for $p>-1/2$ and $n\in\mathbb{N}$.
In order to test \cref{defi:noinc:basic}\cref{item:noinc:defb}, we can then proceed as follows:
\begin{enumerate}
    \item In order to check that there is no incoming radiation to order $q$ for $q<p+1$, it suffices to simply treat the Minkowskian problem $\bar\Box_\eta\phi=0$ and apply \cref{lemma:noinc:mink1} (or \cref{lemma:noinc:fourier_Minkowksi}, i.e.~test \cref{eq:noinc:Fourier_multiplier_on_data}) to $(\psi_0^{(p)},\psi_1^{(p)})$. Let us denote the solution to $\bar{\Box}_\eta\phi=0$ arising from the data at order $p$ as $\phi^{(p)}$.
    \item In order to check that there is no incoming radiation to order $q$ for $p+1\leq q<p+2$, we consider a corrective term $\phi^{(p+1)}_{\mathrm{cor}}$ defined as the solution to
\begin{equation}\label{eq:noinc:stepb-1}
            \bar{\Box}_\eta\phi^{(p+1)}_{\mathrm{cor}}=(\bar{\Box}_\eta-D_0^{-1}\bar{\Box}_g)\phi^{(p)}
        \end{equation}
with trivial data on $\scrim\cup\incone$.
We may then check that there is no incoming radiation to order $q$ by applying \cref{lemma:noinc:mink1} to $\phi-\phi^{(p+1)}_{\mathrm{cor}}$; that is, we can check whether
\begin{equation}\label{eq:noinc:stepb}
     \pvbar(r^2\pvbar)^{\ell}\big(P^{S^2}_\ell(\psi-\psi^{(p+1)}_{\mathrm{cor}})\big)|_{\Sigma}\in\Hb^{q+1-\ell}(\Sigma).
\end{equation}
     
\item For $p+2\leq q <p+3$, we can construct a corrective term $\phi^{(p+2)}_{\mathrm{cor}}$ solving $\bar{\Box}_{\eta}\phi_{\mathrm{cor}}^{(p+2)}=(\bar{\Box}_\eta-D_0^{-1}\bar{\Box}_g)(\phi_{\mathrm{cor}}^{(p+1)})$, and then test for no incoming radiation to order $q$ by checking that $\pvbar(r^2\pvbar)^{\ell}(P^{S^2}_\ell(\psi-\psi_{\mathrm{cor}}^{(p+1)}-\psi_{\mathrm{cor}}^{(p+2)})))\in \Hb^{q+1-\ell}(\Sigma)$.
\item We can iteratively treat higher orders.
\end{enumerate}
We give an example below (\cref{rem:noince:Sch_l=0}) showcasing this algorithm in action.

Similarly as in \cref{lemma:noinc:fourier_Minkowksi}, we can prove the following more abstract result (without explicit expressions):
\begin{lemma}\label{lemma:noinc:fourier_Sch}
    Let $p\in\mathbb{R}$ and $Q^p_0, Q^p_1$ be as in \cref{lemma:noinc:fourier_Minkowksi}.
    Assume that the initial data to \cref{eq:noinc:CauchySchw} are
    \begin{equation}\label{eq:noinc:data_p+1_noinc}
        (\psi_0,\psi_1)=(r^{-p}\psi_0^{(p)}+r^{-p-1}\psi_0^{(p+1,0)},r^{-p-1}\psi_1^{(p)}+r^{-p-2}\psi_1^{(p+1,0)}).
    \end{equation}
    There exists $k\in\mathbb{N}$ and a Fourier multiplier operators of order $k$, $Q^{p,(1)}_0, Q^{p,(1)}_1:H^{s}(S^2)\to H^{s-k}(S^2)$ such that the data $(\psi_0,\psi_1)$ satisfy the no
    incoming radiation condition to order $p+2-$ if and only if $Q^{p}_0\psi^{(p)}_0+Q^{p}_1\psi^{(p)}_1=0$ and
    \begin{equation}\label{eq:noinc:saved}
        Q^{p+1}_0\psi^{(p+1)}_0+Q^{p+1}_1\psi^{(p+1)}_1+Q^{p,(1)}_0\psi^{(p)}_0+Q^{p,(1)}_1\psi^{(p)}_1=0
    \end{equation}
  
\end{lemma}
\begin{proof}
    The proof is essentially the same as in \cref{lemma:noinc:fourier_Minkowksi}:
    From \cref{thm:scat:scat_general,thm:app:general}, together with an argument as in \cref{lemma:Schw:no_logk_terms}, we already have that $v\pv \psi|_{\scrim}=v^{-p}\psi^{(p)}_-+v^{-p-1}\psi^{(p+1)}_-+\Hb^{p+1+;0}(\scrim)$ for some $\psi^{(p)}_-,\psi^{(p+1)}_-\in H^{0}(S^2)$.
    As before, requiring $\psi^{(p)}_-=0$ yields $Q^{p}_0\psi^{(p)}_0+Q^{p}_1\psi^{(p)}_1=0$ by \cref{lemma:noinc:fourier_Minkowksi}.
    Similarly, we can compute the contribution of faster decaying part of the data to get that $\psi^{(p+1)}_-=0$ implies \cref{eq:noinc:saved}.
\end{proof}

In order for the reader to get an idea how the algorithm above works in action, we provide a simple lemma for the case $\ell=0$:
\begin{lemma}\label{rem:noince:Sch_l=0}
   Fix $p\in\R_{\geq -1/2}$ and data as in \cref{lemma:noinc:fourier_Sch} supported on $\ell=0$.
   Then the explicit version of \cref{eq:noinc:saved} is given by 
   \begin{equation}\label{eq:noinc:l=0_multiplier}
       (\psi_1^{(p+1)}-(p+1)\psi_0^{(p+1)})-\psi_0^{(p)}M\Big(1-2p+2p(1+p)\big(\mathbf{H}(\frac{1+p}{2})-\mathbf{H}(\frac{p}{2})\big)\Big)=0
   \end{equation}
   where $\mathbf{H}(p)=\int_0^1\frac{1-x^p}{1-x}\dd x$ is the harmonic number function.
\end{lemma}
\begin{proof}
First, recall the form of the Schwarzschildean wave equation in $\bar{u}, \bar{v}$ coordinates:
\begin{equation}
    (D_0\pubar+(D_0-D)\pvbar)\pvbar \psi-\frac{2M}{r^2}\pvbar\psi=\frac{\Dl\psi}{r^2}-\frac{2M\psi}{r^3}.
\end{equation}

 We follow the procedure outlined above:
\begin{enumerate}
    \item At order $q<p+1$, we simply need to test that the $r^{-p}$-part of $\pvbar\psi|_{\Sigma}$ vanishes; this is precisely the condition that $\psi_1^{(p)}-p\psi_0^{(p)}=0$.
   By integrating $\pvbar\psi^{(p)}$ from $\Sigma$ (where $r=2|\bar{u}|$), we then compute the corresponding Minkowskian solution at order $r^{-p}$ to be $\psi^{(p)}=2^{-p}|\bar{u}|^{-p}\psi_0^{(p)}$.
   \item The correction term $\psi_{\mathrm{cor}}^{(p+1)}$ is defined via the equation 
   \begin{equation}
       \pubar\pvbar \psi_{\mathrm{cor}}^{(p+1)}=-\frac{2M\psi^{(p)}}{D_0r^3}=-\frac{2M|2\bar{u}|^{-p}\psi^{(p)}_0}{D_0(\bar{v}-\bar{u})^3},
   \end{equation}
   where we used that $\pvbar$ and angular derivatives annihilate $\psi^{(p)}$.
Integrating with \cite{Mathematica} in $\bar{u}$ from past null infinity gives, modulo $\mathcal{O}(r^{-(p+3)})$
\begin{equation}
    \pvbar\psi_{\mathrm{cor}}^{(p+1)}|_{\Sigma}=\int_{-\infty}^{\bar{u}} \frac{2M|2\bar{u}'|^{-p}\psi^{(p)}_0}{(\bar{v}-\bar{u})^3}\dd \bar{u}'= \psi^{(p)}_0\frac{2M2^{-p}}{8|\bar{u}|^{p+2}}\Big(1-2p+2p(1+p)\big(\mathbf{H}(\frac{1+p}{2})-\mathbf{H}(\frac{p}{2})\big)\Big).
\end{equation}
The condition \cref{eq:noinc:stepb} now requires that $(\pv\psi-\pv\psi_{\mathrm{cor}}^{(p+1)})|_{\Sigma}$ vanishes modulo $\mathcal{O}(r^{-(p+3)})$; this is exactly\cref{eq:noinc:l=0_multiplier}.
\end{enumerate}

\end{proof}
We emphasise that, in contrast to the Minkowskian computation, in Schwarzschild, testing for no incoming radiation \textit{does not} exclude logs towards $\scrip$---the opposite is the case: For $p\in\mathbb{N}$, nontrivial initial data as in \cref{rem:noince:Sch_l=0} satisfying \cref{eq:noinc:l=0_multiplier} \textit{will} generate log-terms towards $\scrip$ at order $r^{-p-1}\log r$.
\begin{rem}[Practicability of the algorithm]
    Let us comment on the computations required in order to follow the steps above, i.e.~let us comment on \cref{item:noinc:A}.
    The term $\phi^{(p)}$ can be computed directly using \cref{eq:intro:conslaw}, as an alternative to the ODE method used in \cref{lemma:noinc:fourier_Minkowksi}.
    The computation of $\phi_{\mathrm{cor}}^{(p+1)}$ via \cref{eq:intro:conslawf} or the ODE method, however, already gets significantly more involved, cf.~\cite[Proposition~14.4]{kehrberger_case_2024}.\footnote{In \cite{kehrberger_case_2024}, the analogue of $\phi_{\mathrm{cor}}^{(p+1)}$ for the Teukolsky equations is computed up to an undetermined constant restricted to $p\leq \ell$. There, this was important for leading order asymptotic statements about the full system of linearised gravity in the presence of cancellations, cf.~\cite[Eq.~(1.19)]{kehrberger_case_2024}.}
    (Note however that,  in order to test \cref{eq:noinc:stepb}, one does not need to compute $\phi^{(p)}$ itself, but it suffices to compute $(r^2\pv)^{\ell-i}(\psi^{(p)}_{\ell})$ for $i=0,1$ as follows from \cref{eq:Schw:box_in_uvbar}, cf.~\cref{eq:app:conslaw}.)

    To avoid confusion, we stress that even though using the conservation laws \cref{eq:intro:conslaw,eq:intro:conslawf}  requires computing unbounded amounts of derivatives in $v$, all our conditions apply to initial data of finite regularity; indeed, this is exactly the statement that the unspecified higher order Fourier multipliers $Q_i^{p,(1)}$ are of finite order.

\end{rem}

\begin{rem}[Usability of the algorithm.]
  With the remark above already having addressed \cref{item:noinc:A}, let us now explain how the algorithm above addresses \cref{item:noinc:B}:
    Recall again that, for initial data to \cref{eq:noinc:CauchySchw} of the form $(\psi_0,\psi_1)=(r^{-p}\psi_0^{(p)},r^{-p-1}\psi_1^{(p)})$ with $p\in\mathbb{N}$ satisfying the no incoming radiation condition to order $q>p+1$, an exceptional cancellation occurs and $\psi$ exhibits its first logarithm towards $\scrip$ at order $r^{-p-1}\log r$ instead of $r^{-p}\log r$, similarly for $p\notin\mathbb{N}$.
Thus, in order to see this exceptional improvement by $1$, it suffices to test \cref{eq:noinc:stepb} or, equivalently, \cref{eq:noinc:saved}. 
\end{rem}



 	\appendix	

	
    
    \part{Appendix: General quasilinearities and the Einstein vacuum equations}
In this last part of our work, we extend our results from \cref{sec:en,sec:scat:scat} to general quasilinear perturbations (in \cref{app:sec:current_computations}), and we use these extensions to develop a scattering theory for the Einstein vacuum equations in harmonic gauge (in \cref{sec:EVE}).
We follow a more geometric approach to improve the estimates of \cref{sec:en} to general metric and quasilinear perturbations, but the ideas and multiplier are essentially the same as there.

\section{Scattering theory for general quasilinear perturbations}\label{app:sec:current_computations}
\newcommand{\Tmod}{\accentset{\scalebox{.6}{\mbox{\tiny (1)}}}{{\T}}}
The purpose of this section is to put the estimates of \cref{sec:en} on a more geometric footing, using compatible currents and the divergence theorem, and to then finally deduce scattering also for general quasilinear short-range perturbations that \textit{do} affect the light cones.

This section is structured as follows: In \cref{app:sec:current_computations:mink}, we reprove the energy estimates of \cref{sec:en:en} for $\Box_\eta\phi=f$ in a geometric way using currents defined out of the energy momentum tensor $\T$. 
In \cref{sec:current:non_short_range}, we then deduce that the core energy estimates also still apply to long-range potentials as defined in \cref{def:en:long_range}.

In \cref{sec:current:linearmetric}, we then prove estimates for \textit{linear short-range metric perturbations}, i.e.~for metric perturbations that do not depend on $\phi$. The crucial ingredient here is that, given such a metric perturbation $g$, we define a pseudo-geometric energy momentum tensor $\Tmod$ for which we can essentially still prove the same divergence estimates.

In \cref{sec:current:linearscattering}, we elevate the energy estimates to scattering results. Since the metric perturbation does not depend on $\phi$, there is no difficulty in fixing the light-cone yet.

Finally, in \cref{sec:current:general_quasilinear_scattering}, we allow for quasilinear metric perturbations (i.e.~depending on $\phi$) by a Nash--Moser iteration, where, at each step, we solve a \textit{linear} semi-global scattering problem.
\subsection{Geometric energy estimates for \texorpdfstring{$\Box_\eta\phi=0$}{the linear wave equation} using compatible currents}\label{app:sec:current_computations:mink}
In this subsection, we give a more geometric proof of \cref{prop:en:pastestimate} using the energy momentum tensors defined in \cref{def:notation:T}. 
Let us first compute $\T$ and $\tT$ in $u,v,\omega_a$ coordinates, where $\omega_a$ are some coordinates on the unit sphere.
We will reserve the Latin indices $a,b,c\dots$ appearing as components of tensors in spherical directions. 
For convenience, we recall the Minkowski metric as $\eta=-2\dd u \otimes\dd v-2\dd v \otimes\dd u+r^{2}g_{S^2}$, and we use  $g_{S^2}$, $g_{S^2}^{-1}$ to denote the standard metric and  inverse metric on the unit sphere.
We begin by computing the components of the energy momentum tensors of \cref{def:notation:T}:
\begin{lemma}
	For $\tT,\T$ as in \cref{eq:notation:energyMomentumTensor}, we have
	\begin{equation}\label{eq:current:T_indices}
		\begin{aligned}
			4\T^{vv}&=\T_{uu}=(\pu\phi)^2,\\
			4\T^{uu}&=\T_{vv}=(\pv\phi)^2,\\
			4\T^{uv}&=\T_{uv}=r^{-2}\abs{\sl\phi}^2,\\
			r^4\T^{\mu\nu}(g_{S^2})_{\mu\nu}&=\T_{\mu\nu}(g_{S^2}^{-1})^{\mu\nu}=r^2\pu\phi\pv\phi,
		\end{aligned}\qquad
		\begin{aligned}
			4\tT^{vv}&=\tT_{uu}=r^{-2}(\pu r\phi)^2,\\
			4\tT^{uu}&=\tT_{vv}=r^{-2}(\pv r\phi)^2,\\
			4\tT^{uv}&=\tT_{uv}=r^{-2}\abs{\sl\phi}^2,\\
			r^4\tT^{\mu\nu}(g_{S^2})_{\mu\nu}&=\tT_{\mu\nu}(g_{S^2}^{-1})^{\mu\nu}=\pu (r\phi)\pv (r\phi).
		\end{aligned}
	\end{equation}
\end{lemma}
\begin{proof}
	These relations are the consequences of direct computations.
\end{proof}

\begin{defi}
	For a vector field $V$ and a metric $g$, we define its \emph{deformation tensor} $\deformation{V,g}=\tfrac12\mathcal{L}_V g$ for $\mathcal{L}_V$ the Lie derivative along $V$. In particular, we have $\deformation{V,\eta}=\frac12(\nabla_\mu V_\nu+\nabla_\nu V_\mu)$.
\end{defi}

In what is to follow, we will consider vector fields of the form 
$V=f^u(u,v)\pu+f^v(u,v)\pv$. For such $V$, and for metric $\eta$, we swiftly compute:
\begin{equation}\label{eq:current:divergence_computation}
	\deformation{V,\eta}_{uu}=-2\pu f^v,\quad \deformation{V,\eta}_{vv}-2\pv f^u,\quad \deformation{V,\eta}_{uv}=-(\pv f^v+\pu f^u),\quad\deformation{V,\eta}_{ab}=(g_{S^2})_{ab}V(r^{2}).
\end{equation}

Hence, we have the following computations for the currents:
\begin{lemma}
	For $\Box_\eta\phi=f$ and $J=V\cdot \T[\phi]$ and $\tilde{J}=V\cdot\tT[\phi]$, we have
	\begin{subequations}\label{eq:current:divergence_general}
		\begin{align}
			2\divergence(J)&=-\pv f^u(\pu\phi)^2-\pu f^v(\pv\phi)^2+r^{-2}V(r^{2})\pu\phi\pv\phi-r^{-2}(\pu f^u+\pv f^v)\abs{\sl\phi}^2+2V(\phi)f\\
			2\divergence(\tilde{J})&=-r^{-2}\pv f^u(\pu r\phi)^2-r^{-2}\pu f^v(\pv r\phi)^2+\Big(r^{-4}V(r^{2})-r^{-2}(\pu f^u+\pv f^v)\Big)\abs{\sl r\phi}^2+2r^{-1}V(r\phi)f
		\end{align}
	\end{subequations}
\end{lemma}
\begin{proof}
	Let us start with the $J$ divergence.
	It is a standard computation that $\nabla^\mu\T_{\mu\nu}[\phi]=\nabla_\nu\phi \Box \phi$.
	Then $\divergence(J)$ follows from combining \cref{eq:current:T_indices,eq:current:divergence_computation}.
	
	For $\tT[\phi]$ as defined in \cref{eq:notation:energyMomentumTensor}, we use the divergence computation from \cite[Proposition~3]{holzegel_boundedness_2014}: 
	\begin{equation}\label{not:eq:twisted_conservation}
		\begin{gathered}
			\nabla^\mu\tilde{\T}_{\mu\nu}[\phi]= S_\nu[\phi]+f\tilde{\partial}_\nu\phi,\\
			S_\nu[\phi]=\frac{\beta^{-1}\partial_\nu \Box_\eta\beta}{2\beta}\phi^2-\frac{\beta^{-1}\partial_\nu\beta^2}{2\beta}\tilde{\partial}_\sigma\phi\cdot\tilde{\partial}^\sigma\phi,
		\end{gathered}
	\end{equation}
	where $\beta=r^{-1}$, and thus $\Box_\eta\beta=0$.
	We compute the extra term from $S_\mu[\phi]$ to give
	\begin{nalign}
		V^\mu\nabla^\nu\tT_{\mu\nu}[\phi]=-\frac12V(r^{-2})\partial r\phi\cdot\partial r\phi=\frac12r^{-4}V(r^2)\big(\abs{\sl\phi}^2-\pu r\phi\pv r\phi\big).
	\end{nalign}
	Combining this with \cref{eq:current:T_indices,eq:current:divergence_computation} yields the result.
\end{proof}

Therefore, the more geometric proof of \cref{prop:en:pastestimate} follows from the following explicit computations:
\begin{lemma}[Past currents]\label{lemma:current:past}
	Let $\vec{a},\vec{a}^f$ be admissible weights.
	Fix $\epsilon=2(a_--a_0)>0$, $w=\abs{u}/r$ and $\bar{c}>0$. There exist sufficiently small constants $\delta$, $c$ and $C$ (independent of $\bar{c}$) such that the currents 
	\begin{nalign}\label{eq:current:past_currents}
		J_1&=V_1\cdot(\tT[\phi]+\T[\phi]),\qquad V_1=v^{-\epsilon}\abs{u}^{2a_-+1}(\pv+\pu),\\
		J_2&=V_2\cdot(\tT[\phi]+\T[\phi]),\qquad V_2=w^{\bar{c}}v^{-\epsilon-1}\abs{u}^{2a_-}\big(v\pv+c\abs{u}\pu\big),\\
		J_3&=V_3\cdot(\tT[\phi]+\delta\T[\phi]),\qquad V_3=w^{\bar{c}}v^{-\epsilon}\abs{u}^{2a_-}\big(\abs{u}\pu+cv\pv\big),
	\end{nalign} 
	satisfy the following divergence estimates  in $\D^-$ provided that $\Box_\eta\phi=f$:
	\begin{nalign}\label{eq:current:past_coercivity}
		\divergence(J_1)&\geq C \frac{\abs{u}^{2a_--1}}{r^2v^{\epsilon+1}}\abs{\Ve r\phi}^2-C^{-1}\frac{|u|^{2a_--1 }}{r^2v^{1+\epsilon}}|rf|^2|u|^3v,&a_->-1/2,\\
		\divergence(J_2)&\geq Cw^{\bar{c}}\frac{\abs{u}^{2a_--1}}{r^2v^{\epsilon+2}}\big(1+\frac{\bar{c}v}{\abs{u}}\big)\abs{\Ve r\phi}^2-C^{-1}\frac{\abs{u}^{2a_--1}}{r^2v^{\epsilon+2}}\abs{rf}^2u^2v^2w^{\bar{c}},&a_->0,\\
		\divergence(J_3)&\geq Cw^{\bar{c}}\frac{\abs{u}^{2a_--1}}{r^2v^{\epsilon+1}}\big(1+\frac{\bar{c}v}{\abs{u}}\big)\abs{\Ve r\phi}^2+C\frac{1}{r^2}\pu\Big(w^{\bar{c}}|u|^{2a_-}v^{\epsilon}\phi^2\Big)-C^{-1}\frac{\abs{u}^{2a_--1}}{r^2v^{\epsilon+1}}\abs{rf}^2u^2v^2w^{\bar{c}},&a_->1/2.
	\end{nalign}
	In particular, for $a_->0$, $f\in\C^\infty_c(\Dopen^-)$ and for solutions  to $\Box_\eta\phi=f$ with trivial data, we have the uniform estimate $\norm{\Ve r\phi}_{\Hb^{\vec{a};0}(\D^-)}\lesssim \norm{rf}_{\Hb^{\vec{a}^f;0}(\D^-)}$.
\end{lemma}
\begin{obs}\label{rem:current:weight_change_divergence}
	Notice that for both $J_2$ and $J_3$, the weight controlling $\Ve\phi$ (say, $w^{\bar{c}}\frac{\abs{u}^{2a_--1}}{v^{\epsilon+2}}$, for $J_2$) loses a factor of $u\cdot v$ compared to the weight in front of the   bracket in the definition of $V_2,V_3$ (say, $w^{\bar{c}}v^{-\epsilon-1}\abs{u}^{2a_-}$).
	As already seen in \cref{prop:en:main}, the weights are much less clear for $J_1$.
\end{obs}
\begin{obs}\label{rem:current:weights_towards_incone}
	Introducing a weight $q_k:=1+(v-v_0)^{-k}$ (with $k>0$) in the multiplier in \cref{eq:current:past_currents} yields extra bulk terms with good sign (since $\pu q_k=0$ and $\pv q_k<0$), we can therefore always introduce such weights for free. This will be relevant for estimates that capture the vanishing of quantities near $v_0$ to order $k$.
\end{obs}
\begin{proof}
	Recall that, in $\D^-$, $|u|\sim r$.
	Let us start with the computation for $J_1$ (for which we require $a_-\geq -1/2$):
	Using \cref{eq:current:divergence_general}, we compute
	\begin{equation}
		\begin{multlined}
			2\divergence(J_1)=\frac{\abs{u}^{2a_--1}}{r^2v^{\epsilon+1}}\Big(\epsilon(u\pu r\phi)^2+\frac{\abs{u}}{v}(2a_-+1)(v\pv r\phi)^2\Big)+2\frac{\abs{u}^{2a_-+1}}{r^4v^{\epsilon+1}}\abs{\sl r\phi}^2\Big(\epsilon+(2a_-+1)\frac{v}{\abs{u}}\Big)\\
			+\frac{\abs{u}^{2a_--1}}{v^{\epsilon+1}}\Big(\epsilon(u\pu \phi)^2+(2a_-+1)\frac{\abs{u}}{v}(v\pv \phi)^2\Big)+2V\phi\cdot f+\frac{2}{r}V\psi \cdot f.
		\end{multlined}
	\end{equation}
	The result follows by controlling the $f$-term with Young's inequality. In fact, we actually have the stronger control of \cref{rem:en:homogeneous_improvement}.
	
	We next show the result for $J_2$, for which we require $a_->0$:
	First, for $f=0$ and $\bar{c}=0$, we use \cref{eq:current:divergence_general} to compute
	\begin{nalign}
		\begin{multlined}
			2	\divergence(J_2)=\frac{\abs{u}^{2a_--1}}{r^2v^{\epsilon+2}}\Big((1+\epsilon)c(u\pu r\phi)^2+2a_-(v\pv r\phi)^2\Big)+\frac{\abs{u}^{2a_-}}{r^2v^{\epsilon+1}}\abs{\sl \phi}^2\Big(\frac{2}{r}(v-c|u|)+2(c(2a_-+1)+\epsilon)\Big)\\
			+\frac{\abs{u}^{2a_--1}}{v^{\epsilon+2}}\Big((1+\epsilon)c(u\pu \phi)^2+2a_-(v\pv \phi)^2\Big) 
			+\frac{2|u|^{2a_-}}{rv^{\epsilon}}\Big(1-\frac{\abs{u}c}{v}\Big)\pu\phi\pv\phi
		\end{multlined}\\
		\begin{multlined}
			=\frac{\abs{u}^{2a_--1}}{r^2v^{\epsilon+2}}\Big(\big((1+\epsilon)c(u\pu r\phi)^2+2a_-(v\pv r\phi)^2\big)+r^2\big((1+\epsilon)c(u\pu \phi)^2+2a_-(v\pv \phi)^2\big)\Big)\\
			+\frac{\abs{u}^{2a_-}}{v^{\epsilon+2} r^3}\frac{v}{r}\abs{\sl r \phi}^2\Big(\frac{2v}{r}-\frac{2\abs{u}c}{r}+2(c(2a_-+1)+\epsilon)\Big)
			+\frac{\abs{u}^{2a_--1}}{v^{\epsilon+2}}\Big(\frac{2v}{r}-\frac{2\abs{u}c}{r}\Big)|u|\pu\phi v\pv\phi
		\end{multlined}
	\end{nalign}
	By taking $c<2a_->0$, we already have the desired control. If we now consider the case where $f\neq0$, then we  get the additional terms:
	\begin{equation}
		V(\phi) f+r^{-1}V(r\phi)f\lesssim\abs{rf}\frac{\abs{u}^{2a_-}}{r^2v^{\epsilon+1}}\abs{(v\pv,u\pu,1)r\phi}\lesssim \frac{1}{\epsilon'}\frac{\abs{u}^{2a_--1}}{r^2v^{\epsilon+2}}\abs{rf}^2u^2v^2+\epsilon'\frac{\abs{u}^{2a_--1}}{r^2v^{\epsilon+2}}\abs{(v\pv,u\pu,1)r\phi}^2.
	\end{equation}
	
	Lastly, for $\bar{c}\neq 0$, we first compute $\pu w=-\frac{v}{r^2}<0$ and $\pv w=-\frac{\abs{u}}{r^2}<0$.
	Therefore, in the computations above, we simply need to perform the following replacements to conclude the claimed result:
	\begin{nalign}
		(1+\epsilon)c\mapsto (1+\epsilon)c+\bar{c}\frac{v}{r},\quad 2a_-
		\mapsto 2a_-+\bar{c}\frac{v}{r},\quad c(2a_-+1)+\epsilon\mapsto c(2a_-+1)+\epsilon+2\bar{c}\frac{v}{r}.
	\end{nalign}
	
	Finally, we provide the proof for $J_3$, for which we require $a_->1/2$:
	We again set $f=0=\bar{c}$ and first compute
	\begin{multline}\label{eq:currents:proofJtilde3}
		2\divergence(V_3\cdot\tT[\phi])=\frac{|u|^{2a_--1}}{r^2v^{\epsilon+1}}\Big( \epsilon(u\pu\psi)^2 +2ca_- (v\pv\psi)^2  \Big)+\frac{|\sl\psi|^2}{r^4}\frac{|u|^{2a_-}}{v^{\epsilon}}\Big((2a_-+1)-c(1-\epsilon)+\frac{2}{r}(cv-|u|) \Big).
	\end{multline}
	For $c\ll 2a_--1>0$, this already gives the desired control of all first order terms. 
	We next compute 
	\begin{multline}\label{eq:currents:proofJ3}
		2\delta\divergence (V_3\cdot\T[\phi])=\delta\frac{|u|^{2a_--1}}{v^{\epsilon+1}}\Big( \epsilon(u\pu\phi)^2 +2ca_- (v\pv\phi)^2  \Big)+\delta\frac{|\sl\psi|^2}{r^4}\frac{|u|^{2a_-}}{v^{\epsilon}}\Big((2a_-+1)-(1-\epsilon) \Big)\\
		+\delta\frac{2}{r}\pu\phi\pv\phi \Big(cv^{1-\epsilon}|u|^{2a_-}-v^{-\epsilon}|u|^{2a_-+1}\Big).
	\end{multline}
	For $c$ sufficiently small, it suffices to control
	\begin{equation}
		-\frac{2\delta |u|^{2a_-+1} }{rv^{\epsilon}}\pu\phi\pv\phi=-\frac{2\delta |u|^{2a_-+1} }{r^2v^{\epsilon}}\pu\phi(\pv\psi-\phi)
		=r^{-2}\pu\Big(\delta |u|^{2a_-+1}v^{-\epsilon} \phi^2\Big)+\delta(2a_-+1)\frac{|u|^{2a_-}}{r^2v^{\epsilon}}\phi^2
		-2\delta \frac{|u|^{2a_-+1}}{r^2v^{\epsilon}}\pu\phi\pv\psi
	\end{equation}
	The first two terms have a good sign. For the third, we estimate
	\begin{equation}
		-2\delta \frac{|u|^{2a_-+1}}{r^2v^{\epsilon}}\pu\phi\pv\psi\leq   \delta^{3/2}\frac{|u|^{2a_--1}}{v^{\epsilon+1}}(u\pu\phi)^2+\delta^{1/2}\frac{|u|^{2a_--1}}{r^2v^{\epsilon+1}}(v\pv\psi)^2.
	\end{equation}
	For $\delta\ll \min(\epsilon, 2ca_-)$, we can absorb the two terms into the LHS of \cref{eq:currents:proofJtilde3}, \cref{eq:currents:proofJ3}, respectively.
	
	Including $\bar{c}$ and $f$ works exactly as for $J_2$.
\end{proof}

We similarly provide a current for the future region:
\begin{lemma}[Future current]\label{lemma:current:future}
	Let $\vec{a},\vec{a}^f$ be admissible weights.
	Fix $\epsilon=2(a_0-a_+)>0$ and $w=r/v$. Then, there exist sufficiently large constants $c$, $\bar{c}$ such that the current
	\begin{equation}\label{eq:current:future_current}
		J_+=V_+\cdot(\T[\phi]+\tT[\phi]),\qquad V_+=w^{\bar{c}}v^{2a_+}\abs{u}^{\epsilon}(v\pv+c\abs{u}\pu)
	\end{equation}
	has the following coercitivity properties for some $C(\vec{a})$ not depending on $\bar{c}$, provided that $\Box_\eta \phi=f$:
	\begin{equation}\label{eq:current:future_coercivity}
		\divergence(J_+)\geq C \abs{u}^{\epsilon-1}v^{2a_+-1}w^{\bar{c}}\big(1+\bar{c}\frac{\abs{u}}{v}\big)\abs{\Ve\phi}^2-C^{-1}\abs{u}^{\epsilon-1}v^{2a_+-1}w^{\bar{c}}\abs{f}^2u^2v^2.
	\end{equation}
	In particular, we have $\norm{r\phi}_{\Hb^{\vec{a}}(\D^+)}\lesssim \norm{rf}_{\Hb^{\vec{a}^f}(\D^+)}$ for solutions with trivial data at $t=0$.
\end{lemma}
\begin{proof}
	We again set $f=0$ (but keep $\bar{c}>0$) and first compute, using that $\pv w^{\bar{c}}=-\frac{|u|}{vr}\bar{c}w^{\bar{c}}$ and $\pu w^{\bar{c}}=-\frac{1}{r}\bar{c}w^{\bar{c}}$:
	\begin{multline}
		2\divergence(V_+\tT[\phi])=\frac{v^{2a_+-1}|u|^{\epsilon-1}w^{\bar{c}}}{r^2}\Big((v\pv\psi)^2\big(\epsilon+\bar{c}\tfrac{|u|}{r}\big)+(u\pu\psi)^2\cdot c\big(-2a_++\bar{c}\tfrac{|u|}{r}\big)\Big)\\
		+\frac{v^{2a_+}|u|^{\epsilon}w^{\bar{c}}}{r^4}|\sl\psi|^2\Big(\big(2\tfrac{v}{r}-2c\tfrac{|u|}{r}\big)+\big(c(1+\epsilon)+c\bar{c}\tfrac{|u|}{r}+\bar{c}\tfrac{|u|}{r}-2a_+-1\big)\Big).
	\end{multline}
	Now, since $a_+<0$, this yields the  claimed control over all first order terms provided that, say, $\bar{c}>2$ and $c\geq1$.
	We move on to the non-twisted current: We have
	\begin{multline}\label{eq:current:futureproof1} 2\divergence(V_+\T[\phi])=v^{2a_+-1}|u|^{\epsilon-1}w^{\bar{c}}\Big((v\pv\phi)^2\big(\epsilon+\bar{c}\tfrac{|u|}{r}\big)+ c(u\pu\phi)^2\big(-2a_++\bar{c}\tfrac{|u|}{r}\big)\Big)
		+\frac{2v^{2a_+}|u|^{\epsilon-1}w^{\bar{c}}}{r}\Big(1-c\frac{|u|}{v}\Big)|u|\pu\phi v\pv\phi.
	\end{multline}
	We now apply Young's inequality for some $\delta>0$ to deal with the last term:
	\begin{multline}
		\frac{2v^{2a_+}|u|^{\epsilon-1}w^{\bar{c}}}{r}\Big(1-c\frac{|u|}{v}\Big)|u|\pu\phi v\pv\phi\\
		\leq w^{\bar{c}}v^{2a_+-1}|u|^{\epsilon-1}\Big( \delta (v\pv\phi)^2+\delta^{-1}(u\pu\phi)^2     \Big)+w^{\bar{c}}\frac{|u|}{v} v^{2a_+-1}|u|^{\epsilon-1}\Big( c (v\pv\phi)^2+c(u\pu\phi)^2     \Big).
	\end{multline}
	We can absorb the first term on the RHS into the RHS of \cref{eq:current:futureproof1} provided that $\delta<\epsilon$, and provided that $(-2a_+\delta)^{-1}<c$. We can similarly absorb the second term as long as $\bar{c}\gg c$.
	
	Including $f$ is as before.
\end{proof}

We quickly showcase how to use the above current computations to prove the energy estimates of \cref{sec:en:en}:
\begin{proof}[Proof of \cref{prop:en:pastestimate}, \cref{item:en:pastestimate_hom} and \cref{item:en:pastestimate_inhom}]
	In order to prove \cref{item:en:pastestimate_hom}, we apply the divergence theorem to $\divergence(J_1)$ from \cref{lemma:current:past}. Noting that we have 
	\begin{equation}
		\int_{\D^-} \divergence(J_1)\dd \mu \gtrsim  \int_{\D^-} \frac{|u|^{2a_-}}{v^{\epsilon}}|\Ve\psi|^2 \frac{\dd \tilde{\mu}}{|u| v}= \norm{\Ve\psi}_{\Hb^{a_-,a_0;0}(\D^-)},
	\end{equation}
	as well as 
	\begin{equation}
		\int_{\incone} J_1\cdot(\pv) \dd \mu=\int_{\incone} v^{-\epsilon}|u|^{2a_-+1}\big((r^{-1}u\pu\psi)^2+(u\pu\phi)^2+r^{-2}|\sl\phi|^2\big)\dd \mu \lesssim \norm{\psi}_{\Hb^{a_-;1}(\incone)},
	\end{equation}
	the result follows for $k=0$. Higher $k$ follow from commuting.
	
	Similarly, in order to prove \cref{item:en:pastestimate_inhom}, we either apply the divergence theorem to $\divergence(J_2)$ in the case $a_->0$, yielding
	\begin{equation}
		\norm{\Ve\psi}_{\Hb^{a_-,a_--\frac{\epsilon}{2}-\frac12;0}(\D^-)}\lesssim \norm{rf}_{\Hb^{a_-+1,a_-+\frac32-\frac{\epsilon}{2};0}(\D^-)},
	\end{equation}
	or, in the case $a_->1/2$, to $\divergence(J_3)$, yielding
	\begin{equation}
		\norm{\Ve\psi}_{\Hb^{a_-,a_--\frac{\epsilon}{2};0}(\D^-)}\lesssim \norm{rf}_{\Hb^{a_-+1,a_-+2-\frac{\epsilon}{2};0}(\D^-)}.
	\end{equation}
	We may, of course, also include boundary terms on the LHS. For higher $k$, we commute and split the problem into a homogeneous problem with nontrivial data, and an inhomogeneous problem with trivial data.
\end{proof}
\begin{proof}[Proof of \cref{prop:en:futureestimate}]
	Similarly to the above, we apply the divergence theorem to the current $J$ from \cref{lemma:current:future}.
	
\end{proof}
As another quick application, we show that \cref{lemma:current:past} together with \cref{lemma:current:future} already yields a coercive energy estimate for any short-range linear perturbation:
\begin{proof}[Proof of \cref{cor:en:perturbed_energy_estimate} for $k=0$ and $\mathcal{N}=0$]
	We restrict attention to the case when $a_->1/2$, so that $a^f_-\geq3/2$ and admissibility simplifies to $\vec{a}+(1,2,1)<\vec{a}^f$.
	(Otherwise, we would simply also have to use the current $J_2$).
	
	Let us use the notation $\Cbar[J]$ for the flux of the current $J$ over a hypersurface $\Cbar$. Notice that the incoming flux of $J_3$ is controlled by initial data:
	\begin{equation}\label{eq:current:proof_flux_equiv}
		\incone[J_3]\sim \norm{\psi}_{\Hb^{a_0;1}(\incone)}.
	\end{equation}
	
	Next, we apply the divergence theorem for $J_3$ in $\D^-$ to obtain
	\begin{multline}
		\norm{\Ve\psi}_{\Hb^{\vec{a};1}(\D^-)}\lesssim \norm{\psi}_{\Hb^{a_0;1}(\incone)}+\norm{\abs{rf}+\abs{rV\phi}}_{\Hb^{\vec{a};1}(\D^-)}\lesssim \norm{\psi}_{\Hb^{a_0;1}(\incone)}+\norm{rf}_{\Hb^{\vec{a};1}(\D^-)}+\abs{u_0}^{-\delta}\norm{\Ve\psi}_{\Hb^{\vec{a};1}(\D^-)}.
	\end{multline}
	Taking $\abs{u_0}$ suitably large yields the required bulk control.
	The boundary terms follow in much the same way, also using comparability statements such as \cref{eq:current:proof_flux_equiv}.
	
	In the future region $\D^+$, we use a cut-off function to remove the data and apply the current $J_+$ from \cref{lemma:current:future} to obtain the result.
\end{proof}

\subsection{Proof of energy estimates for linear long-range potentials}\label{sec:current:non_short_range}
In this subsection, we showcase how to use the extra constant $\bar{c}$ in \cref{lemma:current:past,lemma:current:future} to extend the core energy estimates to the case where long-range potentials are included, cf.~\cref{def:en:long_range} and \cref{eq:en:longwave}.
\begin{prop}\label{prop:current:long_range}
	Let $\phi$ solve \cref{eq:en:longwave} instead of $\Box_\eta\phi=f$. Then \cref{prop:en:futureestimate}, the $k=0$-part of \cref{item:en:pastestimate_inhom} of \cref{prop:en:pastestimate}, as well as  \cref{lemma:en:energy_no_incoming} all still hold for \cref{eq:en:longwave}.
\end{prop}
\begin{rem}
	Importantly, \cref{item:en:pastestimate_hom} does not generalise directly!
\end{rem}
\begin{proof}
	We prove the result in $\D^-$, as the proof in $\D^+$ proceeds analogously.
	Fix $f_L,\epsilon$ as in \cref{def:en:long_range}.
	For simpler notation, we assume that $a_->1/2$. (Otherwise we may simply use $J_2$ instead.) Then, applying the divergence theorem to $J_3$ from \cref{lemma:current:past}, we obtain that there exists a constant $C$ sufficiently small and independent of $\bar{c}$ such that 
	\begin{equation}
		\norm{\Ve\psi\sqrt{(1+\bar{c}\tfrac{v   }{|u|})}}_{\Hb^{a_-,a_0;0}(\D^-)}\leq C^{-2}\norm{rf}_{\Hb^{a_-+1,a_0+2;0}(\D^-)}+C^{-2}\norm{f_L\Ve\psi\tfrac{v^{\epsilon/2}}{|u|^{\epsilon/2}}}_{\Hb^{a_-,a_0;0}(\D^-)}.
	\end{equation}
	If $\epsilon\geq 1$, we can absorb the second term on the RHS into the LHS by picking $\bar{c}$ sufficiently large. 
	
	On the other hand, if $\epsilon\in(0,1)$, we proceed as follows:
	In the region where $0\leq v/|u|\leq \delta$, then $v^{\epsilon}/|u|^{\epsilon}\leq \delta^{\epsilon}$ can be absorbed into the LHS for sufficiently small $\delta$. 
	On the other hand, in the region where $\delta\leq v/|u|\leq 1$, then we can control $\bar{c}v/|u|\geq \bar{c} \delta^{1-\epsilon}\cdot v^{\epsilon}/|u|^{\epsilon}$ by picking $\bar{c}$ sufficiently large. 
\end{proof}

\subsection{Energy estimates for linear metric perturbations}\label{sec:current:linearmetric}
In the previous subsection, we constructed currents for the Minkowski metric. 
In the present subsection, we show how to construct the currents for linear metric perturbations (with quasilinear problems in mind).
We will work geometrically on \emph{exact Minkowski} space, and treat all objects as tensors on Minkowski space with corresponding metric $\eta$.
Raising and lowering of indices is therefore done with $\eta$.
This Minkowskianisation of the geometry manifests our ability to treat short-range quasilinear problems in $\D$ perturbatively around $\eta$.

{Let $H^{\mu\nu}$} be a smooth symmetric $(2,0)$-tensor field on $\Dopen$, corresponding to the perturbed (inverse) metric $\mathbf{g}^{\mu\nu}:=\eta^{\mu\nu}+H^{\mu\nu}$ and consider
\begin{equation}\label{eq:current:quasilinear_eq}
	\tilde{\Box}_H\phi:=\nabla_\mu (\mathbf{g}^{\mu\nu}\nabla_\nu\phi)=f,\qquad r\phi|_{\incone}=\psi^{\incone},\qquad r\phi|_{u\leq u_\infty}=0
\end{equation}
where $\nabla$ is the covariant derivative with respect to $\eta$.
Central to our computations is the definition of the following \textbf{modified energy momentum tensor} adapted to $H$: 
\begin{equation}\label{eq:current:tmod}
	\Tmod[\phi]^{\mu\nu}:=\mathbf{g}^{\mu'\mu}\mathbf{g}^{\nu\nu'}\nabla_{\mu'}\phi\nabla_{\nu'}\phi-\frac{\mathbf{g}^{\mu\nu}}{2}(\mathbf{g}^{\alpha\beta}\nabla_\alpha\phi\nabla_\beta\phi)
\end{equation}
We compute that, if \cref{eq:current:quasilinear_eq} holds, then it satisfies
\begin{equation}\label{eq:current:Tmod_divergence}\nabla_\mu\Tmod[\phi]^{\mu\nu}=f\mathbf{g}^{\nu\nu'}\nabla_{\nu'}\phi+\nabla_\mu(\mathbf{g}^{\nu\nu'})\mathbf{g}^{\mu\mu'}\nabla_{\mu'}\phi\nabla_{\nu'}\phi-\frac{1}{2}\nabla_{\mu}(\mathbf{g}^{\nu\mu}\mathbf{g}^{\alpha\beta})\nabla_\alpha\phi\nabla_\beta\phi.
\end{equation}

Similarly as for the wave equation on Minkowski, we can use an $r$ twist to derive the equation
\begin{equation}
	r\nabla_\nu \big(\mathbf{g}^{\nu\mu}\frac{1}{r^2}\nabla_\mu \psi\big)+\psi r^{-1} W=f,
\end{equation}
where $W=r\tilde{\Box}r^{-1}$ and $\psi=r\phi$. 
From this, it follows that the \textbf{modified twisted energy momentum tensor}, defined via
\newcommand{\tTmod}{\accentset{\scalebox{.6}{\mbox{\tiny (1)}}}{{\tT}}}
\begin{equation}\tTmod[\phi]^{\mu\nu}:=r^{-2}\mathbf{g}^{\mu'\mu}\mathbf{g}^{\nu'\nu}\nabla_{\mu'}\psi\nabla_{\nu'}\psi-\frac{\mathbf{g}^{\mu\nu}}{2r^2}(\mathbf{g}^{\alpha\beta}\nabla_\alpha \psi\nabla_\beta \psi-W\psi^2),
\end{equation}
satisfies
\begin{equation}\label{eq:app:tTmod_divergence}
	\nabla_\mu\tTmod[\phi]^{\mu\nu}=\frac{f}{r}\mathbf{g}^{\nu\nu'}\nabla_{\nu'} \psi+\frac{1}{r^2}\nabla_\mu(\mathbf{g}^{\nu\nu'})\mathbf{g}^{\mu\mu'}\nabla_{\mu'} \psi\nabla_{\nu'} \psi-\frac{1}{2}\nabla_\mu\frac{\mathbf{g}^{\mu\nu}\mathbf{g}^{\alpha\beta}}{r^2}\nabla_\alpha \psi\nabla_\beta \psi
	-\frac{\nabla_\mu (r^{-2}W\mathbf{g}^{\mu\nu}) }{2}\psi^2.
\end{equation}

Now, using \cref{eq:current:Tmod_divergence,eq:app:tTmod_divergence}, under the assumption that $H^{\mu\nu}$ is conormal, we can derive the required decay of each component of $H^{\mu\nu}$, so as to keep the divergence of the currents from \cref{lemma:current:past,lemma:current:future} coercive. This is done below:
\begin{defi}
	We define $\mathrm{V}$ to be a vector space of tensor fields on $\D$ spanned by $uv$ times
	\begin{equation}\label{eq:current:admissible_h}
		\frac{v}{u}\partial_v\otimes\partial_v,\quad\frac{u}{v}\partial_u\otimes\pu,\quad\pu\otimes\pv,\quad \frac{\sqrt{v}}{r\sqrt{u}}b^a\pv\otimes \partial_a,\quad \frac{\sqrt{u}}{r\sqrt{v}}b^a\pu\otimes \partial_a,\quad r^{-2}\slashed{H}^{ab}\partial_a\otimes \partial_b
	\end{equation}
	where $b,\slashed{H}$ are tensors on Minkowski space that are tangent to the spheres of constant $u,v$ and have $\norm{b,\slashed{H}}_{H^k(S^2_{u,v})}\lesssim_k1$  uniformly in $u,v$.\footnote{Note, that
		$uv\mathrm{V}\subset S^2T_{\mathrm{e}}(\D)$ where the latter space is the symmetric two tensors over the edge tangent bundle as discussed in \cite{hintz_microlocal_2023-1}, and $uv\mathrm{V}$ spans $S^2T_{\mathrm{e}}(\D)$ over $C^\infty(\D)$.} (Equivalently, the elements of $\mathrm{V}$ are spanned by $v\pv\otimes v\pv$, $u\pu\otimes v\pv$ etc.)
\end{defi}

\begin{lemma}\label{lemma:current:perturbative_h}
	Let $\vec{a}^H-(1,2,1)>\delta>0$ componentwise, $f_H\in\Hb^{\vec{a}^H;\infty}(\D)$, and let $H$ be a $f_H$-linear combination of  elements from $\mathrm{V}$. Assume that $\tilde{\Box}_H\phi:=\nabla_\mu (\mathbf{g}^{\mu\nu}\nabla_\nu\phi)=f$. 
	Provided that $\norm{f_H}_{\Hb^{\vec{a}^H;4}(\D)}\leq\epsilon$ is sufficiently small {(alternatively, let $-u_0$ be sufficiently large)}, the following holds:
	For $J_2,J_3$ ($J_+$) as defined in \cref{eq:current:past_currents} (\cref{eq:current:future_current}) with $\T,\tT$ replaced by $\Tmod,\tTmod$, respectively, then \cref{eq:current:past_coercivity} (\cref{eq:current:future_coercivity}) still holds.
\end{lemma}
\begin{proof}
	In this proof, we exclusively use $u,v,\omega_a$ coordinates, with $\omega_a$ denoting coordinates on the unit sphere.
	Correspondingly, we write $X_\mu$ to be a number corresponding to the value of the tensor evaluated in this coordinate system.
	
	It is important to note that we can already use the Minkowskian control provided by $\eta^{\mu\nu}$ part of \cref{eq:app:tTmod_divergence,eq:current:Tmod_divergence}, so all we need to do is bound the contribution from $H$ in these latter two equations against the Minkowskian contribution.
	To this end, let us introduce the weight functions $w_\e,w_\b$ (defined over indices), measuring $\Ve,\Vb$ decay respectively:
	\begin{equation}
		w_{\e}(\mu)=\begin{cases}
			u\quad\text{if }\mu=u,\\
			v\quad\text{if }\mu=v,\\
			\rho_\scri^{1/2}\quad\text{if }\mu=a,
		\end{cases}\qquad
		w_{\b}(\mu)=\begin{cases}
			u\quad\text{if }\mu=u,\\
			v\quad\text{if }\mu=v,\\
			1\quad\text{if }\mu=a.
		\end{cases}
	\end{equation}
	Let us also introduce the vector field multiplier $V=\abs{u}\pu+v\pv$.
	Then we have the following inequalities
	\begin{nalign}\label{eq:app:proof_weights}
		\abs{\nabla_\mu\phi}\lesssim \frac{\abs{\Ve\phi}}{w_e(\mu)},\quad \abs{H^{\mu\nu}}\lesssim \norm{f_H}_{\Hb^{\vec{a}^H;4}(\D)}\frac{w_\e(\mu)w_\e(\nu)}{uv},\quad \abs{V_\mu}\lesssim \frac{uv}{w_\b(\mu)},\quad \abs{V^\mu}\lesssim w_\b(\mu).
	\end{nalign}
	We also note that since $V$ has no angular components, so $V^\mu\nabla_\mu\in C^\infty(\D)\Ve$ as an operator acting on tensors of arbitrary rank.\footnote{Alternatively, we can write $V^\mu\nabla_\mu\in \Diff^1_{\mathrm{e}}({}^{\mathrm{sc}}T^{(n,m)}\R^{3+1},{}^{\mathrm{sc}}T^{(n,m)}\R^{3+1})$, where the scattering tangent bundle (${}^{\mathrm{sc}}T\R^{3+1}$) containing powers of $\dd x,\dd t$ up to the boundary and edge differential operators ($\Diff_{\mathrm{e}}$) spanned over $C^{\infty}(\D)$ by $\Ve$ is introduced in \cite[Section~2]{hintz_stability_2020}.}
	
	We leave the analogous treatment of $J_3$ and $J_+$ to the reader, and focus solely on the treatment of $J_2$, that is, we will treat the error term coming from $H$ in $\nabla_\nu \big((J_2)_\mu\Tmod[\phi]^{\mu\nu}\big)=\Tmod[\phi]^{\mu\nu}\nabla_\nu(J_2)_\mu+(J_2)_\mu\nabla_\nu\Tmod[\phi]^{\mu\nu}$.
	Yet more specifically, we only treat the terms       $\nabla_\mu(H^{\mu\mu'})\nabla_{\mu'}\phi\nabla_\nu\phi$ and $\nabla_\nu (H^{\alpha\beta})\nabla_\alpha\phi\nabla_\beta\phi$ from \cref{eq:current:Tmod_divergence}, as the others are strictly easier to control:
	Compared to the weight in front of the vectorfield $V=(\abs{u}\pu+v\pv)$, as already observed in \cref{rem:current:weight_change_divergence}, we loose a $uv$ weight for the control of $\Ve\phi$.
	Hence, we compute, using \cref{eq:app:proof_weights}, that 
	\begin{nalign}
		uv\abs{V^\mu \nabla_\nu(H^{\nu\sigma})\nabla_\sigma\phi\nabla_\mu\phi}\leq \abs{\Ve\phi}^2\abs{\Vb H^{\nu\sigma}}\frac{uv\cancel{w_\b(\mu)}}{w_\e(\sigma)w_\e(\nu)\cancel{w_\b(\mu)}}&\lesssim \norm{f_H}_{\Hb^{\vec{a}^H;4}(\D)}\abs{\Ve\phi}^2,\\
		uv\abs{V^\nu\nabla_\nu(H^{\alpha\beta})\nabla_\alpha\phi\nabla_\beta\phi}\leq\abs{\Ve\phi}^2\abs{\Vb H^{\alpha\beta}} \frac{uv}{w_\e(\alpha)w_\e(\beta)}&\lesssim \norm{f_H}_{\Hb^{\vec{a}^H;4}(\D)}\abs{\Ve\phi}^2.
	\end{nalign}
	Therefore, all error terms can be bounded by the bulk terms for $-u_0$ sufficiently small.
\end{proof}

Having now obtained the necessary bulk control of the divergences, we next to establish control of the fluxes. To this end, we construct safely spacelike hypersurfaces in order to appeal to the following

\begin{lemma}[Positivity]\label{lemma:current:positivity}
	Let $X,Y$ be covectors that are null- or timelike for $\mathbf{g}$, i.e.~$\mathbf{g}(X,X)\leq 0\geq \mathbf{g}(Y,Y)$. 
	Then $X_\mu Y_\nu\Tmod[\phi]^{\mu\nu}\geq0$.
\end{lemma}	
\begin{proof}
	Since, for a function $\phi$, we have that $\nabla_\mu\phi=(\dd\phi)_{\mu}$ irrespectively of what connection is used, we have that $\Tmod^{\mu\nu}[\phi]$, as a $(2,0)$ tensor, is the same as the usual energy momentum tensor corresponding to $g$.
	Hence, the result follows from the analogous result for the energy momentum tensor for a general metric.
\end{proof}

\begin{lemma}[Spacelike hypersurfaces]\label{lemma:current:hypersurfaces}
	Let $H,\,\delta,\,\epsilon,\,\vec{a}^H$ be as in \red{\cref{lemma:current:perturbative_h}}.
	
	\emph{a)} Then the following is a  family of spacelike hypersurfaces w.r.t~$\mathbf{g}$: 
	\begin{itemize}
		\item 
		for $\epsilon$ sufficiently small (or $-u_1$ sufficiently large) $\Sigma_{u_1}:=\{(v-v_0+1)^{\delta/10}(u-2u_1)=-u_1\}\cap\D$;
		\item for  $\epsilon$ sufficiently small (or $v_2$ sufficiently large) $\bar{\Sigma}_{v_2}:=\{(v-2v_2)u^{\delta/10}=-v_2u_0^{\delta/10}\}\cap\D$
	\end{itemize}
	
	\emph{b)} For  $\epsilon$ sufficiently small (or $-u_0$ sufficiently large), $V^{\musFlat{}}_2,V^{\musFlat{}}_3$ from \cref{eq:current:past_currents} and $V^{\musFlat{}}$ from \cref{eq:current:future_current} are timelike w.r.t~$\mathbf{g}$.
	Here, the musical ismorphism $V^{\musFlat{}}:=\eta(V,\cdot)$ is defined with respect to $\eta$.
\end{lemma}
\begin{proof}
	a) 
	Without loss of generality, we set $v_0=1$ and use $\alpha=\delta/10$.
	We simply compute that there exists some $C>0$ depending on $H$ such that for a defining function of $\Sigma_{u_1}$, $s=v^{d}(u-2u_1)$, we have
	\begin{nalign}
		\mathbf{g}(\dd s,\dd s)=\mathbf{g}(v^{\alpha}\dd u+\alpha v^{\alpha-1}(u-2u_1)\dd v,v^{\alpha}\dd u+\alpha v^{\alpha-1}(u-2u_1)\dd v)\\
		=-4\alpha v^{2\alpha-1}(u-2u_1)+H(v^{\alpha}\dd u+\alpha v^{\alpha-1}(u-2u_1)\dd v,v^{\alpha}\dd u+\alpha v^{\alpha-1}(u-2u_1)\dd v)\\
		\geq 4\alpha u_1v^{\alpha-1}+\epsilon Cr^{-\delta}\alpha^2\frac{\abs{u_1}^2}{\abs{u}v}+\epsilon Cr^{-\delta}v^{2\alpha-1}\abs{u}\leq\abs{u_1}\big(-\alpha v^{\alpha-1}+\epsilon Cr^{-\delta}v^{-1}+\epsilon Cr^{-\delta}v^{\alpha-1}\big).
	\end{nalign}
	Taking $\epsilon$ sufficiently small (or $-u_1$, and correspondingly $r$ large enough), the last term is negative, which yields the result.
	The same computation with $s=(v-2v_2)u^{\alpha}$ and $u,v$ interchanged shows the second result.
	
	b) For this part, it suffices to check that $W=cv\pv+\abs{u}\pu$ for $c\in\R_{>0}$ is timelike:
	\begin{equation}
		\mathbf{g}(cv\dd u+\abs{u}\dd v,cv\dd u+\abs{u}\dd v)\geq-2cv\abs{u}+\epsilon Cr^{-\delta}\abs{u}v.
	\end{equation}
	The result follows by taking $\epsilon$ sufficiently small.
\end{proof}

Therefore, we have the following energy estimate for $\tilde{\Box}_H$:

\begin{prop}\label{prop:current:metric_perturbed_estimate}
	Let $\vec{a},\vec{a}^f$ be admissible with $a_->0$.
	Let $H$ be as in \cref{lemma:current:perturbative_h} with the additional assumption that $\vec{a}^H+\vec{a}$ is strongly admissible\footnote{Note that this assumption is required for a similar reason as in \cref{rem:en:shortrange_admissible_weights}.
		In particular, it can be lifted to $\vec{a}^H>(1,2,1)$, though with a loss of derivatives.} with gap $\delta$. Furthermore, \emph{assume that $\incone$ is null w.r.t~$\mathbf{g}=\eta^{-1}+H$.}
	Let $\psi^{\incone}\in\Hb^{a_-;2k+2}(\incone), rf\in\Hb^{\vec{a}^f;k}(\Deps),T^jf\in\Hb^{a_{-}^f+1;2(k-j)}(\incone)$ for $j\leq k$ satisfy \cref{eq:en:assumption0}.
	Then, for $\norm{f_H}_{\Hb^{\vec{a}^H;4}(\D)}$ sufficiently small, the solution  $\phi$ to \cref{eq:current:quasilinear_eq} satisfies
	\begin{equation}\label{eq:current:linear_quasilinear_estimate}
		\norm{ \Ve \psi}_{\Hb^{\vec{a};k}(\Deps)}
		\lesssim_{H,k} \norm{rf}_{\Hb^{\vec{a}^f;k}(\Deps)}+\sum_{j\leq k}\norm{T^{j}rf}_{\rho_-^{a_-^f}\Hb^{;2(k-j)}(\incone)}
		+ \norm{\psi^{\incone}}_{\rho_-^{a_-}\Hb^{;2k+2}(\incone)},
	\end{equation}
	where $\Deps=\D\cap\{(v-v_0+1)^{\delta/10}(u-2u_0)\leq-u_0\}$.
	
	If in addition we assume that $rf\in\Hbt^{\vec{a}^f;k}(\Deps)$ and $f_H\in\Hbt^{\vec{a}^H;\infty}(\D)$, we get
	\begin{equation}\label{eq:current:linear_quasilinear_estimate_no_incoming}
		\norm{ \Ve \psi}_{\Hbt^{\vec{a};k}(\Deps)}
		\lesssim_{H,k} \norm{rf}_{\Hbt^{\vec{a}^f;k}(\Deps)}+\sum_{j\leq k}\norm{T^{j}rf}_{\rho_-^{a_-^f+1}\Hb^{;2(k-j)}(\incone)}
		+ \norm{\psi^{\incone}}_{\rho_-^{a_-}\Hb^{;2k+2}(\incone)},
	\end{equation}
\end{prop}
\begin{proof}
	Let us only consider the case $k=0$, as we can then treat higher $k$ in the usual way by commuting with the symmetries of $\eta$ and recovering transversal derivatives on $\incone$ as in \cref{lemma:en:recovering_initial_data}, using that $\incone$ is null for $g$.

	We use the divergence theorem with the currents from \cref{lemma:current:past,lemma:current:future} in the finite regions with borders $\incone,\Sigma_{u_0},\Sigma_{u_\infty},\bar{\Sigma}_{v_2}$ for $v_2\to\infty$.
	Using \cref{lemma:current:positivity,lemma:current:hypersurfaces}
	we can neglect the boundary terms on $\Sigma_{u_0}$ and $\bar{\Sigma}_{v_2}$.         
	Using \cref{lemma:current:perturbative_h}, we get bulk control of $\phi$ in terms of $f$ and boundary terms from $\incone$, yielding the result.
\end{proof}
\subsection{Scattering for linear metric perturbations}\label{sec:current:linearscattering}
We can use \cref{prop:current:metric_perturbed_estimate} to construct scattering solutions on more general backgrounds for linear equations, following the steps in \cref{sec:scat:scat}:
\begin{thm}\label{thm:current:metric_scattering}
	Let $\vec{a},\vec{a}^f$ be admissible with $a_->0$.
	Let $\vec{a}^H+\vec{a}$ be strongly admissible.
	Let $\psi^{\incone},rf$ and $H$ be as in \cref{prop:current:metric_perturbed_estimate}.
	Then, for $\norm{f_H}_{\Hb^{\vec{a}^H;4}(\D)}$ sufficiently small, there exists a scattering solution to
	
	\begin{equation}\label{eq:current:quasilinear_no_incoming}
		\tilde{\Box}_H\phi=f,\quad \psi|_{\incone}=\psi^{\incone},\quad \pv\psi\in\Hb^{0+;0}((\D^{v_\infty}_{v_0})^-)\quad \forall v_\infty>v_0,
	\end{equation}
	in $\Deps$ satisfying \cref{eq:current:linear_quasilinear_estimate}.
	The solution is unique in the class of functions satisfying $\Ve\psi\in\Hb^{(0+,a_0+1);0}(\D^-)$.
\end{thm}
\begin{rem}
	Notice that, in \cref{eq:current:quasilinear_no_incoming}, we used a spacetime estimate to capture no incoming radiation condition.
	Indeed, for the case when the null cones are those of Minkowski, it follows easily that \cref{eq:current:quasilinear_no_incoming} with one extra derivative implies the weaker inclusion in \cref{eq:scat:data_incoming}.
\end{rem}

\begin{proof}
	The proof of existence is the same as in \cref{thm:scat:scat_general}, just we need to use \cref{prop:current:metric_perturbed_estimate}.
	
	For uniqueness, we use the
	energy estimate from \cref{prop:current:metric_perturbed_estimate} with a decay $a_-'\in(0,a_-)$, and a cutoff supported in $\D^{u_\infty,v_0}_{2u_\infty,v_\infty}$.
	We take $u_\infty\to-\infty$.
	Using the a priori bound on $\psi$ the result follows.
\end{proof}

In \cref{thm:current:metric_scattering}, we can also allow $\vec{a}^f\geq(1,2,1)+\vec{a}$ instead of admissibility at extra losses of regularity, and treat the case $a_-\in[-1/2,0)$ with extra iterations as in \cref{cor:scat:enlarged_admissible_set,thm:scat:weak_polyhom}. Similarly, we can also include semilinear and potential perturbations.
We leave the details to the reader.

We finish this subsection by providing a few examples of metrics that yield a wave equation with principle symbol that is, or is not, of the form \cref{eq:current:admissible_h}.
\begin{enumerate}[label=\textnormal{Ex.~\arabic*)}]
	\item Any metric that is of the form $\eta+h$, where $h$ is a $r^{-\epsilon}\Hb^{0,0,0;k}(D)$ linear combination of
	\begin{equation}\label{eq:current:admissible_metric}
		\frac{u}{v}\dd v^2,\quad \frac{v}{u}\dd u^2,\quad \dd u\dd v,\quad r\frac{\sqrt{u}}{\sqrt{v}}b^a\dd v\rho_a,\quad r\frac{\sqrt{v}}{\sqrt{u}}b^a\dd u \rho_a,\quad r^2\slashed{h}^{ab} \rho_a\rho_b
	\end{equation}
	where $b,\slashed{h}$ are as in \cref{lemma:current:perturbative_h} and $\rho_a$ is any uniformly bounded one form on $S^2$ ($\norm{\rho_a}_{H^k(S_{u,v}^2)}\lesssim_k1$) gives an example of an $H$ as in \cref{lemma:current:perturbative_h}.
	We note that this class is the same class as considered in  (3.13) of \cite{hintz_microlocal_2023-1}.
	In particular, note that the Kerr metric in suitable coordinates falls into this category.
	\item\label{item:current:metric_coordiante_change}
	\newcommand{\pvb}{\partial_{\bar{v}}}
	Let's restrict attention to $\D^-$.
	Fix $f\in\Hb^{a_-;\infty}(\Dbold^-)$ for some $a_-\in(0,1)$, and consider 
	\begin{equation}\label{eq:current:admissible_metric_bad_coordinates}
		g=\eta+f\dd u^2.
	\end{equation}
	Clearly, $g$ is not of the form \cref{eq:current:admissible_metric}, but we can remedy this. This is important for applications to the Einstein vacuum equations, so let's explain how:\footnote{Alternatively, we could find an eikonal function for $g$ as in \cref{lemma:EVE:eikonal}, but we find the prescription below to be more streamlined though less geometric.}
	First, we find $F\in\Hb^{a_--1;\infty}(\Dbold^-)$ satisfying $\pu F=f$.\footnote{We provide some details: Notice that $\pu\pv F=\pv f$ and $\pu F|_{\incone}=f|_{\incone}$ imply $\pu F=f$. Therefore, solving $\pu\pv F=\pv f$ with no incoming radiation for $F$ as in \cref{thm:scat:weak_decay} yields a possible $F$.}
	Therefore, in $\bar{v}=v+F,u$ coordinates, we can write $\pv|_u=(1+\pv F)\pvb|_u,\,\pu|_v=\pu|_{\bar{v}}+f\pvb|_u$ and get that $ F\in\Hb^{a_--1;\infty}(\Dbold_{u,\bar{v}}^-)$, where we introduced a new compactification with respect to $u,\bar{v}$ coordinates.
	Writing the metric in these coordinates, we get $\dd v=(1-\partial_{\bar{v}} F)\dd \bar{v}-\pu|_{\bar{v}} F \dd u=(1-\pv F)\dd \bar{v}-(f-f\pvb F)\dd u$, and thus:
	\begin{multline}
		g=\dd u\big((1-\partial_{\bar{v}} F)\dd \bar{v}-(f-f\pvb F)\dd u\big)+f\dd u^2+(\bar{v}-u-F)^2\slashed{g}_{S^2}\\=\eta_{u\bar{v}}+f\pvb F\dd u^2
		-\partial_{\bar{v}} F\dd u\dd\bar{v}+F\big(F+2(\bar{v}-u)\big)\slashed{g}_{S^2}.
	\end{multline}
	Clearly, the last two terms in the metric are of the form \cref{eq:current:admissible_metric}, while the $\dd u^2$ term is improved to $\Hb^{2a_-;\infty}(\Dbold^-)$.
	Iterating the procedure a finite number of times yields a metric of the form \cref{eq:current:admissible_metric}, hence we get that \cref{eq:current:admissible_metric_bad_coordinates} is, after a change of coordinates, a short range perturbation of $\eta$ in the sense of falling into the category \cref{eq:current:admissible_metric}. 
	Indeed, the need to modify the coordinates is due to the bending of the lightcones, as constant $v$ hypersurfaces are no longer a sufficiently good approximations to incoming lightcones and do not terminate in the correct~$\scrim$.
\end{enumerate}

\subsection{Scattering for general quasilinear perturbations}\label{sec:current:general_quasilinear_scattering}
We finally elevate the results to allow for metric perturbations that depend on the solution. For this, it is necessary to first study the precise dependence on $\norm{f_H}_{\Hb^{\vec{\epsilon};k}}(\D)$ in the estimate  of \cref{prop:current:metric_perturbed_estimate}.
\begin{prop}\label{prop:current:tame_estimate}
	a) Fix $k\in\N_{\geq10}$ and $\vec{a},\vec{a}^f,\vec{a}^H,f_H,\delta$ as in \cref{thm:current:metric_scattering}.
	Let $\phi$ be a solution to \cref{eq:current:quasilinear_eq} under the assumption \cref{eq:en:assumption0} with $\norm{f_H}_{\Hb^{\vec{a}^H;4}(\Deps)}$  sufficiently small and $\norm{f_H}_{\Hb^{\vec{a}^H;k}(\Deps)}\leq 1$, it holds that
	\begin{equation}\label{eq:current:tame}
		\norm{\Ve \psi}_{\Hb^{\vec{a},k}(\Deps)}\lesssim_k \norm{rf}_{\Hb^{\vec{a}^f;k}(\Deps)}+\norm{rf}_{\Hb^5(\Deps)}\norm{f_H}_{\Hb^{\vec{\epsilon};k}(\Deps)}.
	\end{equation}
	
	b) \cref{thm:current:metric_scattering} holds with explicit dependence on the norm of $f_H$ as in \cref{eq:current:tame}.
\end{prop}
\begin{proof}
	A similar statement is proved in Proposition 3.32 \cite{hintz_exterior_2023}, so we will be brief.
	First, one gets from \cref{prop:current:metric_perturbed_estimate} that, for $k=5$,
	\begin{equation}\label{eq:current:proof_tame}
		\norm{\Ve \psi}_{\Hb^{\vec{a},k}(\Deps)}\lesssim_{\norm{f_H}_{\Hb^{\vec{\epsilon};8}(\Deps)}} \norm{rf}_{\Hb^{\vec{a}^f;k}(\Deps)}.
	\end{equation}

	Next, one simply keeps track of the derivatives acting on $f_H$ after commuting $\tilde{\Box}_H$ in \cref{prop:current:metric_perturbed_estimate} using \cref{lemma:scat:sobolev}. We have, for any $k\geq5$,
	\begin{multline}\label{app:en:proof1}
		\norm{(\Ve^2\psi)(\Vb f_H)}_{\Hb^{\vec{a}^{\tilde{f}};k-1}(\Deps)}\leq \norm{\Vb f_H}_{\Hb^{\vec{a}^H-\vec{\epsilon};3}(\Deps)}\norm{\Ve^2\psi}_{\Hb^{\vec{a};k-1}(\Deps)}\\+
		C_k\norm{ f_H}_{\Hb^{\vec{a}^{H};k}(\Deps)}\norm{\psi}_{\Hb^{\vec{a};5}(\Deps)}+C_k\norm{\Vb f_H}_{\Hb^{\vec{a}^H-\vec{\epsilon};4}(\Deps)}\norm{\Ve^2\psi}_{\Hb^{\vec{a};k-2}(\Deps)}
	\end{multline}
	for some strongly admissible $\vec{a}^{\tilde{f}}$ and $\epsilon>0$.
	Indeed, this follows by expanding the norm definition from \cref{not:eq:explicit_norm} and using Sobolev embedding.
	
	Finally, we use induction on $k$ to prove \cref{eq:current:tame}.
	For $k\leq5$, \cref{eq:current:proof_tame} already suffices.
	Next we treat the induction step and assume that \cref{eq:current:tame} already holds for $k'\leq k-1$.
	In \cref{app:en:proof1}, the first term on the right hand side can be absorbed to the left of \cref{eq:current:tame} by assumption on $f_H$;
	the second term can be estimated via \cref{eq:current:proof_tame}; the third is already controlled by induction on $k'$. This proves $a)$. 
	
	The proof of $b)$ is then immediate.
\end{proof}

We have finally gathered all the necessary ingredients to prove scattering for general quasilinear perturbations of the form $\tilde{\Box}_H$. Let us first provide definitions of admissible metric perturbations and scattering solutions, extending \cref{def:en:short_range}:
\begin{defi}\label{def:current:quasilinear}  
	Fix $\vec{a}$ admissible.
	\begin{enumerate}[label=\textbf{P4}]
		\item \textbf{}  	Let $H[\phi]$ be a linear combination of tensors of $\mathrm{V}$  (see \cref{eq:current:admissible_h}) with coefficients of the form
		\begin{equation}\label{eq:current:admissible_h_quasi}
			f_H\prod_{i=1}^{n^H-1}\Diff_{\b}(r\phi),\qquad f_H\in\Hb^{\vec{a}^H;\infty}(\D),
		\end{equation}
		with \emph{weight} $\vec{a}^H$ and \emph{degree} $n^H$ satisfying that $\vec{a}^H+n^H\vec{a}$ is strongly admissible.
		We call $H^{\mu\nu}\nabla_\mu\nabla_\nu$ an \emph{admissible quasilinear perturbation of general type} and the corresponding $H^{\mu\nu}$ an \emph{admissible metric perturbation}.
		We call $\delta=\min(\vec{a}^H+(n^H-1)\vec{a})$ \emph{the gap} of $H$.
		
		We say that $H[\phi]$ is compatible with the no incoming radiation condition if $f_H\in\Hbt^{\vec{a}^H;\infty}(\D)$ and if the only differentiated terms in \cref{eq:current:admissible_h_quasi} are of the form $\Vbt(r\phi)$.
	\end{enumerate}
\end{defi}
We also  generalise our notion of a scattering solution:
\begin{defi}
	Fix an admissible quasilinear perturbation of general type $H$, short-range semilinear perturbations $\mathcal{N}$, and linear perturbations $A$ (which may include both long-range and  short-range potentials, cf.~\cref{def:en:long_range} and \cref{def:en:short_range}).
	For $\psi^{\incone}\in\Hb^{a_-;4}(\incone)$, we call $r\phi\in\Hb^{\vec{a};k}(\D)$ for $k\geq 1$ a scattering solution with no incoming radiation to
	\begin{equation}\label{eq:current:quasi_box}
		\tilde\Box_H\phi+\mathcal{N}[\phi]+A[\phi]=f
	\end{equation}
	if $\phi$ solves \cref{eq:current:quasi_box}, and if, for any $v_{\infty}<\infty$, 
	\begin{equation}
		r\phi|_{\incone}=\psi^{\incone},\qquad \pv\psi\in\Hb^{0+;0}(\D^{v_{\infty}}_{v_0}).
	\end{equation}
\end{defi}

\begin{thm}\label{thm:current:quasilinear_scattering}
	Let's fix $\vec{a},\vec{a}^f$ admissible.
	Let $k\in\N_{\geq 298}$, and assume $\psi^{\incone}\in\Hb^{a_-;2k+2}(\incone)$, $ rf\in\Hb^{\vec{a}^f;k}(\D)$, $T^jf\in\Hb^{a_{-}^f+1;2(k-j)}(\incone)$ for all $j\leq k$.
	Let $H$ be an admissible quasilinear perturbation of general type such that $H^{vv}[\phi]|_{\incone}=0$,\footnote{This condition only depends on $\phi|_{\incone}$ and not $T\phi|_{\incone}$. Indeed, we can recover $T\phi$ by \cref{lemma:en:recovering_initial_data} under the  \emph{assumption} that $H^{vv}|_{\incone}=0$. Then, we use this to evaluate whether $H^{vv}|_{\incone}=0$.} let $\mathcal{N}$ be semilinear short-range perturbations and $A$ long-range and short-range potentials.
	
	Then there exists $-u_0$ sufficiently large such that a scattering solution to \cref{eq:current:quasi_box} with no incoming radiation exists and satisfies $\Ve r\phi\in\Hb^{\vec{a};k}(\Deps)$.
	Finally, the solution is unique in the space $\Ve r\phi\in\Hb^{(a_-,-\infty,-\infty);5}(\D)$.
	
	If in addition $\mathcal{N},A,H$ are compatible with no incoming radiation then $rf\in\Hbt^{\vec{a}^f;k}(\D)\implies r\phi\in\Hbt^{\vec{a};k}(\D^\delta)$.
	
	If $\mathcal{N},A,H,f$ as well as the data are polyhomogeneous (or have polyhomogeneous coefficients), then the solution $\phi$ is also polyhomogeneous on $\D$.
\end{thm}

\begin{proof}
	We only prove scattering, as the no incoming radiation estimate follows by replacing \cref{eq:current:linear_quasilinear_estimate} by \cref{eq:current:linear_quasilinear_estimate_no_incoming} and polyhomogeneity follows as in \cref{thm:app:general}.
	
	\emph{Step 1: Removing the data.}
	Note that all $T^j\psi$ derivatives of $\phi$ can be recovered on $\incone$ as in \cref{lemma:en:recovering_initial_data} from $\psi^{\incone}$ using that $H^{vv}=0$ along $\incone$.
	These satisfy $T^j\psi|_{\incone}\in\Hb^{a_-;2(k+1-j)}(\incone)$ for $j\leq k+1$.
	
	Let $\psi^{(0)}\in\Hb^{\vec{a};k_0}(\D)$ be any function satisfying $T^j\psi^{(0)}|_{\incone}=T^j\psi|_{\incone}$ for $j\leq k+1$ and $\supp\psi^{(0)}\subset\{v<v_0+1\}$.
	Writing $\psi^{(1)}=\psi-\psi^{(0)}$, at a loss of regularity, it suffices to prove \cref{thm:current:quasilinear_scattering} with $T^j\phi|_{\incone}=0$ $\forall j\leq k_0$.
	Equivalently, we are looking for $r\phi\in  q_{k_0}\Hb^{\vec{a}}(\D)$ where $ q_{k_0}=1+(v-v_0)^{-k_0}$.

	\emph{Step 2: Iteration.}
	Fix $k_0\geq k$.
	Let $L_{\Phi_1}=(\tilde\Box_{H[\Phi_1]}+V[\Phi_1])$ be the linearisation of \cref{eq:current:quasi_box} around $\Phi_1$.
	Similarly as in \cref{lemma:en:short_range_computations}, we get that for $\Phi_1\in  q_{k_0}\Hb^{\vec{a};k}(\Deps)$, $P_g[\Phi_1]:=\tilde{\Box}_{H[\Phi_1]},V[\Phi_1]$ are short range metric and long range potential perturbations for $\vec{a}$ with coefficients of $k-1$ regularity.
	Therefore, applying \cref{prop:current:tame_estimate} together with the extra weights from \cref{rem:current:weights_towards_incone}, we get for $10\leq j\leq k-1$, and any admissible inhomogeneous weight $\vec{a}^\mathfrak{f}$, the scattering solution with no data for $L_{\Phi_1}\Phi=\mathfrak{f}\in q_{k_0}\Hb^{\vec{a}^{\mathfrak{f}};j}(\Deps)$ satisfies 
	\begin{equation}\label{eq:app:nash_step1}
		\norm{ q_{k_0}\Ve r\Phi}_{\Hb^{\vec{a};j}(\Deps)}\lesssim \norm{ q_{k_0}r\mathfrak{f}}_{\Hb^{\vec{a}^f;j}(\Deps)}+\norm{ q_{k_0}r\Phi_1}_{\Hb^{\vec{a};j+1}(\Deps)} \norm{ q_{k_0}r\mathfrak{f}}_{\Hb^{\vec{a}^\mathfrak{f};5}(\Deps)}.
	\end{equation}
	Now, we can use a Nash-Moser iteration scheme of \cite{raymond_simple_1989} with a loss of derivative $d=3$ as in \cite[Corollary~3.36]{hintz_exterior_2023}  to get that there exists a scattering solution $\Ve r\phi\in\Hb^{\vec{a};d}(\D^\delta)$ provided that $rf\in\Hb^{\vec{a}^f;k}(\D^\delta)$
	whenever $\norm{r\Phi_0}_{\Hb^{\vec{a};2d}(\Deps)}$ is sufficiently small depending on $\norm{r\Phi_0}_{\Hb^{\vec{a};k}(\Deps)}$ with $k\geq 297 = 16d^2 + 43d + 24$ (and on the implicit constant in \cref{eq:app:nash_step1}, which in turn depends on the size of the coefficients of the perturbations and can be made small by choosing $u_0$ sufficiently large), where the first iterate $\Phi_0$ is defined via
	\begin{equation}
		\Box_\eta\Phi_0=f,\quad r\phi|_{\incone}=0,\quad \pv r\Phi_0\in\Hb^{0+,a_0;k}(\D^-).
	\end{equation}
	
	{\emph{Step 3: Regularity.}}
	Let us now improve the regularity of $r\phi\in\Hb^{\vec{a};d}(\D^\delta)$ to $r\phi\in\Hb^{\vec{a};k}(\D^\delta)$.
	This follows the same strategy as the propagation of regularity for quasilinear wave equations, where we use \cref{thm:current:metric_scattering} to propagate \emph{linearly} higher order norms on an already sufficiently regular background.
	
	Commuting \cref{eq:current:quasi_box} with the symmetries of $\tilde{\Box}_0$ up to three orders and schematically writing $\Vb^3\phi$ for this, we obtain an equation of the form $L_{\phi}\Vb^3\phi+\mathrm{V}\Ve\Diff_{\b}^3\phi+f_{\phi}=\Vb^3f$,
	where $\mathrm{V}$ is a short-range admissible linear perturbation depending on $\Diff_{\b}^2\phi$ and $\norm{f_\phi}_{\Hb^{\vec{a}+(1,2,1)+\delta;k}}\lesssim \norm{\Vb\Diff^2_{\b}\phi}_{\Hb^{\vec{a};k}}^2$ for all $k\geq2$.
	We apply the linear result of \cref{prop:current:tame_estimate} to $\Vb^3\phi$ with $k=d-2$ to obtain for $j=1$
	\begin{equation}\label{eq:currernt:improved_reg}
		\norm{\Ve r\phi}_{\Hb^{\vec{a};d+j}(\D^\delta)}\lesssim \norm{rf}_{\Hb^{\vec{a}^f;d+j}(\D^\delta)}+\norm{\Ve r\phi}_{\Hb^{\vec{a};d+j-1}}^2+\norm{\Ve r\phi}_{\Hb^{\vec{a};d}(\D^\delta)}\norm{rf}_{\Hb^{\vec{a}^f;5}(\D^\delta)},
	\end{equation}
	on a region $\D^\delta\cap\{u<u_0\}$ for some possibly large $-u_0$.
	Commuting with $\Vb^{j+2}$ and inducting on the regularity index, we obtain \cref{eq:currernt:improved_reg} for all $j\geq 1$ in $\D^\delta\cap\{u<u_j\}$ with $-u_j\gg_j 1$.
	Notice that each time, we only apply \cref{prop:current:tame_estimate} with $k=d-2$, so we only need $d$ regularity on the coefficients of $L_\phi$.
	For $j=k-d$, we combine this with the a priori control on $\norm{\Ve r\phi}_{\Hb^{\vec{a};d}}$ in terms of $\norm{rf}_{\Hb^{\vec{a}^f;k}(\D^\delta)}$ from Step 2 to close the estimate at the same order.
	
	Regularity for $u>u_j$ follows from standard propagation of regularity for quasilinear wave equations.
	
	\emph{Step 4: Uniqueness.}
	Consider two scattering solutions with no incoming radiation $\phi_1,\phi_2$.
	By assumption $\Ve \phi_i \in \Hb^{\delta;5}(\D^{v_1}_{v_0})$ for some $\delta>0$.
	Fix a smooth cutoff $\chi(x)$ with $\chi|_{x>1}=1$ and $\chi|_{x<1/2}=0$.
	Then $\bar\phi_{u_\infty}=\chi(u/u_\infty)(\phi_1-\phi_2)$ satisfies a linear equation with short range linear and quasilinear perturbation of general type and inhomogeneity uniformly bounded in $\Hb^{\delta}(\D^{v_1}_{v_0})$.
	Using \cref{prop:current:metric_perturbed_estimate} with $a_-=\delta/2, k=0$ and sending $u_\infty\to-\infty$ yields uniqueness in $\D^{v_1}_{v_0}$ for all $v_1<\infty$.
\end{proof}

\begin{rem}[Proof by bootstrap]
	\cref{thm:current:quasilinear_scattering} can be proved with fewer derivatives by using a bootstrap  as for \cref{thm:scat:scat_general}. 
	However, for quasilinear problems, the characteristic IVP is no longer suitable (as it would not be robust enough for a limiting argument), so one instead has to use the hypersurfaces from \cref{lemma:current:hypersurfaces} to set up a sequence of finite characteristic-spacelike IVPs and then take the limit. 
	We leave the details to the interested reader.
\end{rem}
    \newpage
\section{Analysis of the Einstein vacuum equations in harmonic gauge}\label{sec:EVE}
The aim of this section is to construct scattering solutions to the Einstein vacuum equations (EVE) in a neighbourhood of spacelike infinity.
Harmonic gauge is the most natural choice for the framework of our paper, as the EVE then  immediately fall into our general setting in the region $\D^-$, and need only a slight modification in $\D^+$.
We briefly recall the usual harmonic gauge fixing procedure:
The Einstein vacuum equations with zero cosmological constant read
\begin{equation}\label{eq:EVE:EVE}\tag{EVE}
    \mathrm{Ric}(g)=0
\end{equation}
for a Lorentzian metric $g$.
To obtain a hyperbolic system, we introduce a gauge term, following \cite{lindblad_global_2005,hintz_stability_2020}:
Introduce the gauge 1-form
\begin{equation}\label{eq:EVE:upsilon}
    \Upsilon_\nu:=(g\eta^{-1}\delta_gG_g\eta)_\nu=g_{\mu\nu}g^{\kappa\lambda}(\Gamma(g)^\nu_{\kappa\lambda}-\Gamma(\eta)^\nu_{\kappa\lambda}),
\end{equation}
where $\eta$ is the Minkowski metric, and where $\delta_g$, $G_g=1-\frac12 g\tr_g$ are the divergence and the trace reversal  operators with respect to $g$.
Fix a (1,2) tensor field $\mathrm{X}$ depending only on $\Upsilon,g,\nabla_\eta g$ ($\nabla_\eta$ denoting the Minkowski covariant derivative).
Write $\Upsilon\lrcorner \mathrm{X}$ for the contraction in its first index and $\delta_g^\star$ for the symmetric gradient with respect to $g$.
Then, we obtain the reduced Einstein vacuum equations
\begin{equation}\label{eq:EVE:rEVE}\tag{rEVE}
    \mathrm{Ric}(g)-\delta_g^\star \Upsilon-\Upsilon\lrcorner \mathrm{X}=0.
\end{equation}
An important fact about \cref{eq:EVE:rEVE} is, that using the Bianchi identity $\delta_gG_g\mathrm{Ric}(g)=0$ along with $\mathrm{Ric}(g)=0$, we obtain that $\Upsilon$ satisfies $\Box_g\Upsilon=-\delta_g G_g(\Upsilon\lrcorner \mathrm{X})$, where $\Box_g$ is the covariant wave operator associated to $g$.
Therefore, a solution to \cref{eq:EVE:rEVE} that has $\Upsilon=0$ also yields a solution to \cref{eq:EVE:EVE}, furthermore $\Upsilon=0$ is propagated by \cref{eq:EVE:rEVE}.
Solutions with $\Upsilon=0$ are said to be in harmonic gauge.
The naiv $\mathrm{X}=0$ choice for metric perturbations $g=\eta+h$ in rectangular coordinates $h_{\mu\nu}=h(\partial_\mu,\partial_\nu)$ yields
\begin{equation}
    \Box_{\eta+h}h_\mu\nu=Q_{\mu\nu}[\partial h],
\end{equation}
where the right hand side contains quadratic in $\partial h$ terms with smooth in $h$ coefficients, see \cite{lindblad_global_2005}.
An important, yet somewhat subtle, achievement of \cite{lindblad_global_2005} is to choose an appropriate $\mathrm{X}$ such that the resulting wave equation for $g$ has semilinear terms of a nice form (in terms of null conditions), see already \cref{eq:EVE:g_original}.

It is well known that one cannot straight-forwardly associate much physical significance to results in harmonic gauge (though one can, for instance, prove the Bondi mass loss formula), so the point of this section is merely to conclude that the solution exists and has a polyhomogeneous expansion towards all regions, provided the data also have one. In particular, we do not concern ourselves with determining the coefficients in the asymptotic expansions (see, however, \cref{sec:EVE:fast_decaying_scattering}).

Let us give a rough version of the main result:
\begin{thm}[See \cref{cor:EVE:final_form} for precise decay rates]\label{thm:EVE:main_scattering}
    Let $k\geq300,a_0\in(-1/2,0]$  and let $\mathfrak{h}_{\mathcal{UU}}\in\Hb^{a_0;k}(\scrim_{v>v_0+1})$ be scattering data for \cref{eq:EVE:EVE} in harmonic gauge satisfying \cref{eq:EVE:restriction_on_scattering_data}, and assume that a suitably decaying solution---at least $rh_{\mathcal{UU}}\in \Hb^{a_-}(\mathcal{D}^{v_0+1}_{v_0})$ with $a_->-1/2$---in harmonic gauge already exists in the slab $\mathcal{D}^{v_0+1}_{v_0}$ (see \cref{ass:EVE:data}).
    Then, there exists $-u_0$ sufficiently large such that \cref{eq:EVE:EVE} admits a solution $g=\eta+h$ in the domain of dependence of $\incone,\scrim_{v>v_0}$, with $rh\in \Hb^{\left(a_-,a_0,a_0\right)-;k}(\D)$.
    Moreover, provided that $\mathfrak{h}$ and the solution in the initial slab $\D^{v_0+1}_{v_0}$ is polyhomogeneous, the solution is polyhomogeneous on a compactification associated to the blow-up of $\scrim,\scrip$ (associated to $g$).
\end{thm}
Let us already remark that the regularity $\Hb^{\left(a_-,a_0,a_0\right)-;k}(\D)$ and polyhomogeneity statement only holds if one is working in coordinates that correctly capture the deviation from the null geometry.
The necessary requirement is that the metric deviation is found to be strongly admissible as in \cref{def:current:quasilinear}, for which we need to find coordinates $u$ and $v$ that are approximately eikonal for $\eta+h$.
This requirement is verified in \cref{eq:EVE:solution_past} for~$\D^-$, similarly in \cref{eq:EVE:future_h1_size,eq:EVE:proof_bootstrap} for $\D^+$.
\begin{proof}
   We first construct an eikonal function in \cref{lemma:EVE:eikonal} within $\D^{v_0+1}_{v_0}$ using \cref{ass:EVE:data}, and then prove the result restricted to $\D^-$ in \cref{thm:EVE:scatteringpast} as an application of \cref{thm:current:quasilinear_scattering}. 
   We extend the result to the future in \cref{prop:EVE:ansatz} and \cref{lemma:EVE:future_nonlinear}.
\end{proof}

\paragraph{Structure of this section:} 
We first recall the form of EVE in harmonic gauge in \cref{sec:EVE:equations} from \cite{lindblad_global_2005} which, in particular, captures the well known weak null structure of the equations first noticed in \cite{lindblad_weak_2003}.
For well chosen decay rates, the system in $\D^-$ is a short-range quasilinear perturbation of $\Box_\eta$ with a long-range potential (the long-range potential only arising from frame projections).
The corresponding scattering theory in $\D^-$ is worked out in \cref{sec:EVE:ass,sec:EVE:eikonal,sec:EVE:proof}. 
We also give an example of a scattering solution with fast decay towards Schwarzschild in \cref{sec:EVE:fast_decaying_scattering}.

In $\D^+$, studied in \cref{sec:EVE:future}, the equations are no longer short-range perturbations of $\Box_\eta$, and we have to use the transport character of the gauge condition for one of the metric components, an element of the analysis well known since \cite{Lindblad2010}. 
The only subtlety in this region is that the harmonic gauge set up with respect to coordinates in a neighbourhood of $I^0$ suffers from \emph{light bending} stronger than that caused by the Schwarzschildean term.\footnote{For the same issue, see also \cite{leflochNonlinearStabilitySelfgravitating2024}.}
We resolve this issue by first constructing an ansatz correcting this bending; this is similar to \cref{item:current:metric_coordiante_change}.

\textit{We stress that, throughout the entire section,  just as in \cref{app:sec:current_computations}, all geometric operators such as covariant derivatives and musical isomorphism are performed with respect to the Minkowski metric $\eta$ unless otherwise noted.
This is to emphasise the fundamentally perturbative nature of the problem.}

\subsection{The form of the equations in \texorpdfstring{$\D^+$ and in $\D^-$}{D+ and D-}}\label{sec:EVE:equations}

Throughout this section, we will assume that $h$ (always to be thought of as in $g=\eta+h$) is a symmetric 2-tensor field on Minkowski space, and that in rectangular frame, the components $h_{\mu\nu}$ corresponding to  satisfy
\begin{equation}\label{eq:EVE:assumption_h}
    \norm{rh_{\mu\nu}}_{\Hb^{a',a',a';4}(\D)}\leq\epsilon
\end{equation}
for some $a'>-1/2$ and $\epsilon$ small  enough.\footnote{This requirement is to ensure that the only failure of $\eta+h$ to be a short range metric perturbation comes from $h_{LL}$ near $\scrim$ and from $h_{\Lbar\Lbar}$ near $\scrim$,respectively.
Indeed, the aim of this section is to show how to find $rh$ in the given class with extra decay for $h_{LL}$ or $h_{\Lbar\Lbar}$.}
Of course, this assumption will later be recovered via a bootstrap in our dynamical construction of solutions.

Let us first consider the form of the equations in $\D^+$. 
We introduce the labels $\Lbar,L,1,2$ corresponding to contraction with incoming null ($\pu$), outgoing null ($\pv$) and spherical vector fields ($S_1,S_2$)  which are defined  on the sphere $\{u=u_0,v=v_0\}$ and then extended to all of $\D$ via  parallel transport.\footnote{Of course, to locally trivialise the tensor bundle (and corresponding sections) over the sphere, we  need to use two coordinate charts: 
For instance, we can use the spherical coordinates around the $z$ axis $\theta,\phi,r$ and take $rS_1=\partial_\phi,\,rS_2=\partial_\theta$ in $\abs{\cos\theta}<9/10$ and, we can repeat the analogous procedure to construct vector fields via spherical coordinates around the  $y$ axis  in $\abs{y/r}<0.9$. From here on, we suppress this issue of multiple charts.}

As customary, we group the above labels into $\mathcal{T}=\{L,1,2\}$, $\mathcal{U}=\{\Lbar,L,1,2\}$, corresponding to \emph{good} (tangential to outgoing null hypersurfaces) and \emph{bad} (arbitrary) directions. 
For a tensor field $Y$ of rank $(0,2)$, we will then use $Y_{\mathcal{UU}}$ to refer to the \textit{components} of $Y$ in the above basis, and we will consider equations for these components as \emph{scalar} equations. 
Now, \cref{eq:EVE:rEVE} being a geometric equation for a symmetric $(0,2)$ tensor field $g=\eta+h$, in harmonic gauge, it can be written as a system of  equations for 10 \emph{scalar} functions, $h_{\mathcal{UU}}$.
Henceforth, we rid ourselves of most geometry and simply study the corresponding scalar system and write
\begin{itemize}
    \item $h$ for the collection of scalars  $h_{\mathcal{UU}}$ unless otherwise specified, such as in the proof of \cref{prop:EVE:ansatz};
    \item $F[h_{\mathcal{UT}}]$ for functions that depend only on the scalar quantities $h_{\mathcal{UT}}$;
    \item $\Box_\eta h_{\mathcal{UT}}=F_{\mathcal{UT}}$ when the scalar equation $\Box_\eta h_{\mu\nu}=F_{\mu\nu}$ holds for any labels $\mu\nu\in\mathcal{UT}$.
\end{itemize}
Bringing some of the results of \cite{lindblad_global_2005,Lindblad2010} into a form suitable for our paper's general framework, we then have:

\begin{prop}[Form of \cref{eq:EVE:rEVE} in $\D^+$]\label{thm:EVE:form_of_EVE}
	Assume \cref{eq:EVE:assumption_h}, and let $g=\eta+h$. There exists a choice of $\mathrm{X}$, vanishing when $h=0$, such that the Einstein vacuum equations in harmonic gauge \cref{eq:EVE:rEVE} in $\D^+$ can be written in the form 
	\begin{subequations}\label{eq:EVE:EVE_schematic}
		\begin{align}
			\bar{\Box}_{h}h_{\mathcal{T}\mathcal{U}}:=g^{\mu\nu}[h](\nabla_{\mu}\nabla_\nu h_{\mathcal{T}\mathcal{U}})&=\mathfrak{F}_{\mathcal{TU}}[h]=A_{\mathcal{T}\mathcal{U}}[h]+Q_{\mathcal{T}\mathcal{U}}[h]+G_{\mathcal{T}\mathcal{U}}[h]\label{eq:EVE:eve_L},\\
			\bar{\Box}_hh_{\Lbar\Lbar}&=\mathfrak{F}_{\Lbar\Lbar}[h]=A_{\Lbar\Lbar}[h]+Q_{\Lbar\Lbar}[h]+G_{\Lbar\Lbar}[h]+F_{\Lbar\Lbar}[h_{\mathcal{U}\mathcal{T}}]\label{eq:EVE:eve_Lbar},\\
			\partial_u h_{LL}&=A_{\mathrm{t}} [h]+G_{\mathrm{t}}[h],\label{eq:EVE:eve_Ltransport}
		\end{align}
	\end{subequations}
	
	where\footnote{We will from now on use the notation from \cref{rem:en:smoothperturbations}!} 
	\begin{enumerate}[itemsep=-1ex,label={\alph*)}]
		\item all the terms are formed of $\partial_t,\partial_x$ derivatives and coefficients that are smooth in coordinates $\{r^{-1},x/r,t/r\}$;
		\item $A_{\mathcal{UU}}$ are long-range potentials with weight $(2,2)$ arising from projecting onto the frame $\mathcal{U}$;
		\item the nonlinearities $Q_{\mathcal{UU}},G_{\mathcal{UU}}$  (using the notation of \cref{def:en:short_range}) have orders $n^Q=2$, $n^G=3$ and weights $(a^{Q}_0,a^{Q}_+)=(3,3/2)$, $(a^{G}_0,a^{G}_+)=(4,2)$ respectively;
		\item $Q$ is a quadratic function of $(\partial,r^{-1}) h$ with bounded coefficients;
		\item $G$ is a smooth function of $h$ vanishing at $h=0$ times $((\partial,r^{-1}) h)^2$ with bounded coefficients;
		\item the nonlinearity $G_{\mathrm{t}}$ is a smooth product of $h$ and $\{\pu,\pv,r^{-1}\sl\}h$;  $A_t$ is a linear term of the form $A_{\mathrm{t}}[h]=r^{-1}\Diff^1_{\b}[h]$;
		\item \cref{eq:EVE:eve_Ltransport} is the contraction of the gauge 1-form $\Upsilon_\mu$ with $\partial_t-\partial_r$ where $\Upsilon_\mu$ (defined in \cref{eq:EVE:upsilon}) solves the covariant wave equation on 1-forms: $\Box_g\Upsilon=0$; 
		\item $F$ depends on $\pu h_{\mathcal{U}\mathcal{T}}$ quadratically with bounded coefficients, that is $n^F=2$ and $(a^F_0,a^F_+)=(3,1)$;
		\item the metric is given by $g=\eta+h$ and the inverse is given by
		\begin{equation}\label{eq:EVE:g_original}
			g^{-1}=\eta^{-1}+H[h]=\eta-\frac14h_{LL}\pu\otimes\pu+\bar{H}_1[h]+\bar{H}_2[h],
		\end{equation}
		where $\bar{H}_1,\bar{H}_2$ are function of $h_{\mathcal{UU}}$ vanishing to order one and two respectively with bounded coefficients; moreover, they are metric perturbations (see \cref{def:current:quasilinear}) with orders {$n^{\bar{H}_1}=2$, $n^{\bar{H}_2}=3$} and weights $(3,3/2)$, $(4,2)$.
	\end{enumerate}	
\end{prop}
\begin{rem}
    Notice that the splitting into \cref{eq:EVE:eve_L} and \cref{eq:EVE:eve_Lbar} corresponds exactly to the usual presentation of the weak null condition, with the "bad" term in \cref{eq:EVE:eve_Lbar}, namely $F_{\bar{L}\bar{L}}$, depending only on the components $h_{\mathcal{TU}}$, which themselves satisfy a "good" equation, namely \cref{eq:EVE:eve_L}.
\end{rem}
\begin{rem}
    The definition of $\bar{\Box}_h$ differs from that of $\tilde{\Box}_H$ in \cref{eq:current:quasilinear_eq}. This difference only generates semilinear perturbations, as $g^{\mu\nu}$ does not depend on derivatives of $h$.
\end{rem}

\begin{rem}\label{rem:EVE:projectionlongrange}
    The appearance of the long-range potential terms in \cref{eq:EVE:EVE_schematic} is due to the fact that we work with the equations satisfied by the \textit{projections} of $h$ onto the frame $\mathcal{U}$, thus reducing tensorial equations to scalar equations. 
    
    This allows us to sidestep various issues present in \cite{Lindblad2010,keir_weak_2018}.
    We can make this simplification because unlike \cite{Lindblad2010,keir_weak_2018}, we only work in a neighbourhood of spacelike infinity, where long-range potentials are easy to deal with, see \cref{rem:en:long_range}.
\end{rem}
\begin{proof}
    Although \cite{lindblad_global_2005} do not use the wording as in~\cref{eq:EVE:rEVE}, it is straight-forward to find the choice of $\mathrm{X}$ they implicitly use:
    They assume that $\Upsilon=0$, and then manipulate \cref{eq:EVE:EVE} to arrive at \cref{eq:EVE:rEVE} whilst dropping all terms related to~$\Upsilon$.
    In particular, they drop $\partial_\mu\Upsilon_\nu$ from their (3.11) as well as the $\mu\nu$ symmetrization of the following terms below:
    \begin{nalign}
        \Upsilon_{\beta'}g^{\beta\beta'}\partial_{(\mu} g_{|\beta|\nu)}+\frac{1}{2}g^{\alpha\alpha'}\Upsilon_{(\nu}\partial_{\mu)}g_{\alpha\alpha'},\quad \text{dropped in \cite[Eq.~(3.14)]{lindblad_global_2005}};\\ 
        -\frac{1}{2}g^{\beta\beta'}\partial_{\beta'}g_{\beta(\mu}\Upsilon_{\nu)}-\frac{1}{4}g^{\alpha\alpha'}\Upsilon_{(\mu}\partial_{\nu)}g_{\alpha\alpha'},\quad \text{dropped in \cite[Eq.~(3.14)]{lindblad_global_2005}};\\ 
        \partial_{(\mu}\Upsilon_{\nu)}+\Upsilon_{\beta'}g^{\beta\beta'}\partial_{(\mu} g_{|\beta|\nu)}-\Upsilon_\mu\Upsilon_\nu, \quad \text{dropped in \cite[Eq.~(3.17)]{lindblad_global_2005}}.\\ 
    \end{nalign}
    In particular, the term in the last line (dropped in (3.17)) term is exactly equal to $\delta_g^{\star}\Upsilon+\Upsilon\lrcorner \mathrm{X}$ and thus defines $\mathrm{X}$.

	Moving on to \textit{a)}, we note that, while the regularity of the coefficients is not explicitly classified as we write it here in \cite{lindblad_global_2005}, it can easily be inferred from \cite[Lemma~3.1]{lindblad_global_2005}.
	In particular, in the rectangular frame $\{\dd t, \dd x\}$, the coefficients are constant and the projections onto the $\mathcal{U}$ basis satisfy the required smoothness. 
	
	The nonlinearities in \cref{eq:EVE:eve_Lbar} follow from \cite[Proposition~3.1]{Lindblad2010}---again, the weights are not explicitly written there, but can easily be deduced from the explicit form of the nonlinearity in \cite[Eq.~(3.9)]{Lindblad2010} and the computation in \cite[Section~4]{Lindblad2010}. 
 See also \cite[Lemmata~4.2 and 9.6]{Lindblad2010}, where the improved weights of the $Q$-terms in \cref{eq:EVE:eve_L} are phrased as inequalities. 
	
	The origin of the long-range terms $A_{\mathcal{UU}}$  (cf.~\cref{rem:EVE:projectionlongrange}) lies in the fact that projecting onto the basis used in \cite[Eq.~(3.7)]{Lindblad2010} does not commute with $\bar\Box_h$: Indeed, fix $X,Y$ to be vector fields from the set $\{L,\Lbar\}$ or unit size spherical derivatives. 
	Then, we have for some $A_{\mu\nu}\in\rho_+^2\rho_0^2\Diff^1_{\b}(\D^+)$, possibly changing from line to line, (here, we treat $h_{\mu\nu}$ as a tensor!)
	\begin{nalign}
		X^\mu Y^\nu(\bar{\Box}_{0}h)_{\mu\nu}&=\bar{\Box}_{0}(X^\mu Y^\nu h_{\mu\nu})-2\eta^{\rho\sigma}\nabla_\rho (Y^\nu X^\mu)\nabla_\sigma h_{\mu\nu}-h_{\mu\nu}\bar{\Box}_0 (X^\mu Y^\nu)\\&=
		\bar{\Box}_{0}(X^\mu Y^\nu h_{\mu\nu})-2\nabla_{L}(X^\mu Y^\nu)\pu h_{\mu\nu}+A_{\mu\nu}[h]=\bar{\Box}_{0}(X^\mu Y^\nu h_{\mu\nu})+A_{\mu\nu}[h],
	\end{nalign}
	where, in the second equality, we used (in $u,v,\omega_a$ basis) that only a $\pu$ derivative acting on $h$ can give a term that decays too slowly towards $\scrip$ to be a long-range potential and, in the third, we used that the vector fields $X,Y$ are paralell propagated along $\Lbar$.
	The terms from $\bar{\Box}_{h}-\bar{\Box}_0$ are all short-range following from the assumption that $h$ decays and using rectangular coordinates.
	
	The transport equation \cref{eq:EVE:eve_Ltransport} follows from contracting $\Upsilon$ as in \cref{eq:EVE:upsilon} with $L$, see \cite[Lemma~8.1]{Lindblad2010}.
	Note, as for the nonlinearities, this statement in \cite{Lindblad2010} appears as an inequality, but the equality follows from the equations \cite[Eq.~(8.2),(8.3)]{Lindblad2010}, see also \cite[Section~2.6]{keir_weak_2018}.
	For later reference, we provide explicit expressions for the linear parts of the constraint equations from \cite[Eq.~(2.80)]{keir_weak_2018}:\footnote{Let's recall Keir's notation: In Keir's (2.80), the expression $(\Lbar h)_{LL}$ means $(\pu (h_{\mu\nu}))L^\mu L^\nu$ evaluated in rectangular coordinates. Since $\pu L^\mu=0$ in rectangular coordinates, we then have $(\Lbar h)_{LL}=\pu h_{LL}$. On the other hand, Keir's term $\slashed{g}^{\mu\nu}(r^{-1}\sl_\mu h)_{L\nu}$ only equals $r^{-1}g_{S^2}^{AB}\sl_A h_{LB}$ to leading order, where, by $\sl_A h_{LB}$, we mean the $A$-component of the covariant derivative acting on the scalar $h_{LB}$.} 
	\begin{subequations}\label{eq:eve:constraint_linear}
		\begin{align}\label{eq:EVE:constraints:puhLL}
			\pu h_{LL}&=-2r^{-1}g_{S^2}^{AB}\sl_A h_{LB}-r^{-1}g_{S^2}^{AB}\pv h_{AB}+A_{\mathrm{t},0}[h]+G_{\mathrm{t}}[h],\\\label{eq:EVE:constraints:pvhLbarLbvar}
			\pu g_{S^2}^{AB}h_{AB}&=-\pv h_{\Lbar\Lbar}-2r^{-1}g_{S^2}^{AB}\sl_A h_{\Lbar B}+A_{\mathrm{t},0}[h]+G_{\mathrm{t}}[h],\\
			\pu h_{L B}&=-\pv h_{\Lbar A}-2r^{-1}g_{S^2}^{BC}\sl_C h_{AB}-r^{-1}g_{S^2}^{BC}\sl_A h_{BC}-r^{-1}\sl_A h_{L\Lbar}+A_{\mathrm{t},0}[h]+G_{\mathrm{t}}[h],\label{eq:EVE:constraints:puhLA}
		\end{align} 
	\end{subequations}
	where, just as before, $A_{\mathrm{t},0}[h]$ come from commuting the derivatives with the projections onto $\mathcal{U}$ and  depends linearly on $h$ with coefficients that are smooth in $\{r^{-1},x/r,t/r\}$ and weight $(1,1,1)$.
    Importantly, we see that $\partial_u$-derivatives of certain components of the metric perturbation are related to $\partial_v$-derivatives of other components.

	Finally, \cref{eq:EVE:g_original} follows from inverting the metric $g=\eta+h$, see \cite[Section~3]{Lindblad2010}:
    \begin{equation}
        (g^{-1})^{\mu\nu}=\eta^{\mu\nu}-h^{\mu\nu}+h^{\mu\sigma}h_{\sigma}{}^{\nu}+\mathcal{O}(h^3)=\eta^{\mu\nu}-\eta^{\mu\mu'}\eta^{\nu\nu'}\big(h_{\mu'\nu'}-\eta^{\sigma\sigma'}h_{\mu'\sigma}h_{\nu'\sigma}\big)+\mathcal{O}(h^3).
    \end{equation}
The reason why we record the $-\frac12 h_{LL}$ term in \cref{eq:EVE:g_original} separately is that it is not a short-range metric perturbation under just the assumption \cref{eq:EVE:assumption_h}, cf.~\cref{def:current:quasilinear}.
\end{proof}

Next, we analogously introduce the form of the equations useful in $\D^-$.
Here, the labels differentiating \emph{good} and \emph{bad} terms are $\mathcal{T}=\{\Lbar,1,2\}$, $\mathcal{U}=\{\Lbar,L,1,2\}$, i.e.~$L$ and $\Lbar$ are interchanged.

\begin{prop}\label{thm:EVE:form_of_EVE_past}
        Assume \cref{eq:EVE:assumption_h}.
	   Then \cref{eq:EVE:rEVE} in harmonic gauge, with $g=\eta+h$, can be written in the form

\begin{subequations}\label{eq:EVE:EVE_schematic_past}
		\begin{align}
			\bar{\Box}_{h}h_{\mathcal{T}\mathcal{U}}:=g^{\mu\nu}[h]\nabla_{\mu}\nabla_\nu h_{\mathcal{T}\mathcal{U}}&=\mathfrak{F}_{\mathcal{TU}}[h]=A_{\mathcal{T}\mathcal{U}}[h]+Q_{\mathcal{T}\mathcal{U}}[h]+G_{\mathcal{T}\mathcal{U}}[h]\label{eq:EVE:past_TU}\\
			\bar{\Box}_{h}h_{LL}&=\mathfrak{F}_{LL}[h]=A_{LL}[h]+Q_{LL}[h]+G_{LL}[h]+F_{LL}[h_{\mathcal{T}\mathcal{U}}]\label{eq:EVE:past_LL}
		\end{align}
	\end{subequations}
	where the nonlinearities $\mathfrak{F}_{UU}$ satisfy the time inversed statements of \cref{thm:EVE:form_of_EVE}, namely:
	\begin{enumerate}[itemsep=-1ex,label={\alph*)}]
		\item all the terms are formed of $\partial_t,\partial_x$ derivatives and coefficients that are smooth in coordinates $\{r^{-1},x/r,t/r\}$;
		\item $A_{\mathcal{UU}}$ are long-range potentials with weight $(2,2)$ arising from the frame projections;
		\item the nonlinearities $Q_{\mathcal{UU}},G_{\mathcal{UU}}$  have orders $n^Q=2$, $n^G=3$ and weights $(a^{Q}_-,a^{Q}_0)=(3/2,3)$, $(a^{G}_-,a^{G}_0)=(2,4)$ respectively;
		\item $Q$ is a quadratic function of $(\partial,r^{-1}) h$ with bounded coefficients;\label{item:EVE:Q_past}
		\item $G$ is a smooth function of $h$ vanishing at $h=0$ times $((\partial,r^{-1}) h)^2$ with bounded coefficients;\label{item:EVE:G_past}
		\item $F$ depends on $\pv h_{\mathcal{U}\mathcal{T}}$ quadratically with bounded coefficients, that is $n^F=2$ and $(a^F_-,a^F_0)=(1,3)$;\label{item:EVE:F_past}
		\item the metric takes the form $g=\eta+h$ and the inverse is given by 
		\begin{equation}\label{eq:EVE:g_past}
			g^{-1}=\eta^{-1}+H=\eta^{-1}-\frac14 h_{\Lbar\Lbar} \pv\otimes\pv +\bar{H}_1[h]+\bar{H}_2[h],
		\end{equation}
        where $\bar{H}_1,\bar{H}_2$ are function of $h_{\mathcal{UU}}$ vanishing to order one and two respectively with bounded coefficients, moreover they are metric perturbations (see \cref{def:current:quasilinear}) with order {$n^{\bar{H}_1}=2$, $n^{\bar{H}_2}=3$} and weights $(3/2,3)$, $(2,4)$.\label{item:EVE:H_past}
    \end{enumerate}
\end{prop}
\begin{proof}
	This is simply a time reversed version of \cref{thm:EVE:form_of_EVE}.
\end{proof}

\subsection{Construction of scattering solutions in \texorpdfstring{$\D^-$}{D-}}\label{sec:EVE:past}
\subsubsection{Assumptions on the scattering data}\label{sec:EVE:ass}
In this subsection, we consider the scattering problem for the Einstein vacuum equations in the region $\D^-$, provided that a solution is already given in some bounded advanced time region $\D^{v_0+1}_{v_0}$.
We will use the following sets of labels $\mathcal{T}=\{\Lbar,1,2\}$, $\mathcal{U}=\{\Lbar,L,1,2\}$.
As is clear from \cref{eq:EVE:past_LL}, the component $h_{LL}$ will have a forcing ($F_{LL}[h_{\mathcal{TU}}]$) that is not admissible whenever there is incoming radiation for the components $h_{\mathcal{TU}}$.
To remedy this, we need to define scattering as done for weakly decaying polyhomogeneous forcing in \cref{def:scat:weak_polyhom}:
\begin{defi}\label{def:EVE:data}
    Fix $\mathfrak{h}_{\mathcal{UU}}\in\Hb^{a_0;1}(\scrim_{v>v_0+1})$.
    Let's define $\mathfrak{h}^{(0,1)}_{LL}=\int\dd v (rF_{LL}[r^{-1}\mathfrak{h}_{\mathcal{UT}}])|_{\scrim}$ for $F_{LL}$ as in \cref{eq:EVE:EVE_schematic_past}.
    We call $h$ a scattering solution to \cref{eq:EVE:EVE_schematic_past} with incoming radiation $\mathfrak{h}_{\mathcal{UU}}$ if for all $v_\infty>v_0$ and $k=0$
    \begin{equation}\label{eq:EVE:scattering_def}
		(\pv rh_{\mathcal{T}\mathcal{U}}-\pv\mathfrak{h}_{\mathcal{T}\mathcal{U}})|_{\D^{v_\infty}_{v_0+1}}\in\Hb^{0+;k}(\D^{v_\infty}_{v_0+1}),\qquad \pv rh_{LL}|_{\D^{v_\infty}_{v_0+1}}=\log(-u/v) \pv \mathfrak{h}^{(0,1)}_{LL}+\pv\mathfrak{h}_{LL}(v,\omega)+\Hb^{0+;k}(\D^{v_\infty}_{v_0+1}).
	\end{equation}

    We say that the scattering data $\mathfrak{h}_{\mathcal{UU}}$ is compatible with the harmonic gauge condition if
	\begin{equation}\label{eq:EVE:restriction_on_scattering_data}
		(\pv \mathfrak{h}_{\Lbar\Lbar})|_{\scrim}=(\pv \mathfrak{h}_{\Lbar A})|_{\scrim}=(\pv \eta^{\mu\nu}\mathfrak{h}_{\mu\nu})|_{\scrim}=0.
	\end{equation}
\end{defi}

\begin{rem}[Compactly supported incoming radiation]\label{rem:appB:compactsupport}
    For $\supp\mathfrak{h}_{\mathcal{UU}}\subset \scrim_{v_0,v_1}$ for $v_1<\infty$ we say that the solution has compactly supported incoming radiation.
    Note that this does not imply $\supp \mathfrak{h}_{LL}^{0,1}$ is compact.
\end{rem}

The extra condition \cref{eq:EVE:restriction_on_scattering_data} is because we need to impose vanishing scattering data for the harmonic gauge potential~$\Upsilon$, see \cref{eq:eve:constraint_linear}:
Indeed, provided that there exists a solution $rh\in \Hb^{-1+}(\D^{v_\infty}_{v_0})$ satisfying that $\Upsilon[h]=0$, we must have \cref{eq:EVE:restriction_on_scattering_data}.

The physically relevant assumption in $\D^{v_0+1}_{v_0}$ comes from Post-Newtonian theory, see Section 2 of~\cite{kehrberger_case_2024} and \cref{rem:EVE:PN} below for more details:
\begin{ass}\label{ass:EVE:data}
	We say that $h$ is \emph{compatible} with the far field region of $N$-body scattering if for some collection $\{a^{\mathcal{UU}}_-\}$ with $\min\{a^{\mathcal{UU}}_-\}>-1/4$ and with 
	$a^{\Lbar\Lbar}_->1/2$, it holds that ${rh_{\mathcal{UU}}|_{\D^{v_0+1}_{v_0}}\in\Hbt^{a^{\mathcal{UU}}_-;k}(\D^{v_0+1}_{v_0})}$.
\end{ass}

\begin{rem}[Construction of harmonic coordinates in $\D^-$ such that \cref{ass:EVE:data} holds]\label{rem:EVE:PN}
   At the \emph{linear level}, the outgoing radiation (= no incoming radiation) solutions of systems with approximately Newtonian multipole moments are discussed in \cite{thorne_multipole_1980}. 
   Inserting then the Newtonian multipole moments of a hyperbolic $N$-body scattering event in the infinity past (see \cite[Section 2.2]{kehrberger_case_2024-1}) into \cite[Eq.~(8.13)]{thorne_multipole_1980}  yields that the asymptotics for $h$ in the region $X=[0,1)_{1/r}\times (-3/2,0)_{t/r}\times S^2$ are of the form
    \begin{equation}\label{eq:EVE:thorn_asymp}
        h_{\mathcal{UU}}\in\A{phg}^{\E}(X)\subset \A{phg}^{\mindex1 \cup \overline{(2,1)}}(X)+\Hb^{2+}(X).
    \end{equation}
    
    Now, these expressions are seemingly worse than the desired $rh_{\bar{L}\bar{L}}\in \Hb^{1/2+}(\D^{v_0+1}_{v_0})$.
    Nevertheless, assuming that there exists a solution to \cref{eq:EVE:EVE} satisfying \cref{eq:EVE:thorn_asymp}, we here provide a sketch how to derive \cref{ass:EVE:data}:\footnote{The construction we present here gives harmonic coordinates in the exterior of an incoming lightcone. Alternatively, it is also possible to find such harmonic coordinates in an open neighbourhood of $\scrim$ within $X$, but this requires solving an initial boundary value problem to find $F$ and $\tilde{x}_\mu$ with no incoming radiation through all of $\scrim$.}
    We first apply a coordinate transformation $\bar{v}=v+F$ where $\pu F=h_{LL}$ and $F\in\A{phg}^{(0,1)}(X)$ as in \cref{item:current:metric_coordiante_change}.
    This gives
    \begin{equation}
        h_{\bar{\Lbar}\bar{\Lbar}}\in\Hb^{2-}(X),
        \qquad h_{\mathcal{\bar{U}\bar{T}}}\in\Hb^{1-}(X).
    \end{equation}
    (Here, the vector field $\bar{\Lbar}=\pu|_{\bar{v}}$ is defined via $\partial_{u}$ in the new coordinates.)
    In the new rectangular coordinates $\bar{x}_\mu$ corresponding to $\bar{v}, u, \omega$,\footnote{That is, $\bar{x}_0=\bar{t}=\bar{u}+\bar{v} $ etc.} we have $\Box_g \bar{x}_\mu\in\Hb^{2-}(X)$, where $\Box_g$ is the covariant wave operator corresponding to $g=\eta+h$. The problem with these nicer behaved coordinates, however, is that they are not harmonic.

    To remedy this, let us introduce $\bar{\D}^-$ for the compactification with respect to $\bar{x}_\mu$.
    Using \cref{lemma:EVE:eikonal}, we know that there exists an eikonal function $v+w$ with $w\in\Hbt^{1/2-}(\bar{\D}^{v_0+1}_{v_0-1})$ and corresponding lightcone $\underline{\tilde{\mathcal{C}}}=\{v+w=v_1\}$ for $\eta+h$ in $\{v\in(v_0-1,v_0+1),u\ll-1\}$.
    Next, we consider the solution to the scattering problem with no incoming radiation to 
    \begin{equation}
        \Box_g (\tilde{x}_\mu-\bar{x}_\mu)=-\Box_g\bar{x}_\mu,\quad (\tilde{x}_\mu-\bar{x}_\mu)|_{\underline{\tilde{\mathcal{C}}}}=0.
    \end{equation}
    Using \cref{thm:scat:scat_general}, we conclude that $\tilde{x}_\mu-\bar{x}_\mu\in\rho_-\Hbt^{0-,0-}(\bar{\D}^-)$.
    In particular, we get that $\tilde{v}-\bar{v}\in\rho^-\Hbt^{0-,0-}(\bar\D^-)$.
    Therefore, the $\Ve$ derivatives corresponding to $\bar{x}_\mu$ change by an error of  $\rho_-\Hbt^{0-,0-}(\bar\D^-)\Ve$ when transforming to $\tilde{x}_\mu$ coordinates: $\tilde{v}\partial_{\tilde{v}}-\bar{v}\partial_{\bar{v}}\in\rho_-\Hbt^{0-,0-}(\bar{\D}^-)\bar{\mathcal{V}}_{\mathrm{e}}$ and similarly for other vector fields in $\tilde{\mathcal{V}}_{\mathrm{e}}$.
    In conclusion, we get that \cref{eq:EVE:thorn_asymp} in $\tilde{x}_\mu$ coordinates is given by
    \begin{equation}
        h_{\tilde{\Lbar}\tilde{\Lbar}}\in\Hb^{2-,1-}(\bar\D^-),
        \qquad h_{\mathcal{\tilde{U}\tilde{T}}}\in\Hb^{1-,1-}(\bar\D^-).
    \end{equation}
\end{rem}

Alternatively to using \cref{ass:EVE:data}, one may recover these rates on an initial null cone $\incone$, using the assumptions of \cite{kehrberger_case_2024} and the methods of \cite{luk_local_2012}.
As this recovery is tangential to the main interest in the paper and would already rely on an \emph{assumption} on the initial data on an incoming cone, we do not pursue this direction further.
\subsubsection{Construction of eikonal function}\label{sec:EVE:eikonal}
 The decay rates in \cref{ass:EVE:data} on the metric is sufficient to fix the null geometry close to $\scrim$. To avoid confusion, we stress that the lemma below is far from sharp in terms of regularity and decay:

\begin{lemma}[Construction of eikonal function]\label{lemma:EVE:eikonal}
    Fix a collection of weights $\{a^{\mathcal{UU}}_-\}$ with $\min \{a^{\mathcal{UU}}_-\}>-1/4$ and $a^{\Lbar\Lbar}_->1/2$. Let $k\geq5$ and ${rh_{\mathcal{UU}}|_{\D^{v_0+1}_{v_0}}\in\Hb^{a^{\mathcal{UU}}_-;k}(\D^{v_0+1}_{v_0})}$ and $\norm{rh_{\mathcal{UU}}}_{\Hb^{a^{\mathcal{UU}}_-;5}(\D_{v_0}^{v_0+1})}$ sufficiently small.
    
    a)
    Then there exists $w\in\Hb^{1/4;k-1}(\D^{v_0+1}_{v_0})$ such that $v+w$ is an eikonal function for $\eta+h$.
    Thus, for all $v_1\in(v_0+1/4,v_0+3/4)$ and $u_0\ll-1$ there exist a null hypersurface $\Cbar\subset\{\abs{v-v_1}\lesssim \abs{u}^{-1/4+}\}$.
    
    b)
    Assuming additionally that ${rh_{\mathcal{UU}}|_{\D^{v_0+1}_{v_0}}\in\Hbt^{a^{\mathcal{UU}}_-;k}(\D^{v_0+1}_{v_0})}$ yields $w\in\Hbt^{1/4-;k-1}(\D^{v_0+1}_{v_0})$.
\end{lemma}
\begin{proof}
	\emph{a) }
    We find $v+w$, a solution to the eikonal equation, $g\big(\dd(v+w),\dd(v+w)\big)=0$, in $\D'=\D^{v_0+7/8}_{v_0+1/8}$.
    We work in coordinates $u,v,\omega_a$.
    We write (partial derivatives of the) eikonal equation as a nonlinear transport equation for the derivatives $W_\mu:=\partial_\mu w$
    \begin{equation}\label{eq:EVE:eikonal}
        \big(2g^{\nu v}+2g^{\nu\mu}W_\mu\big)\partial_\nu W_\sigma=-\partial_\sigma g^{vv}-2\partial_\sigma g^{v\mu}W_\mu-\partial_\sigma g^{\mu\nu}W_\mu W_\nu.
    \end{equation}
    The linear part of the transport term is $\tilde{X}/2=g^{uv}\pu+g^{vv}\pv+g^{av}\partial_{\omega_a}$.
    Using that $g^{vv}\in\Hb^{3/2+;k}(\D'), g^{av}\in\Hb^{7/4+;k}(\D')$ and $g^{uv}+1/2\in\Hb^{3/4+;k}(\D')$, it follows that the integral curves of $\tilde{X}$ terminate at $\scrim$.
    We also note that the inhomogeneity satisfies $u^{\delta^{\sigma}_{u}}\partial_\sigma g^{vv}\in\Hb^{3/2-;k-1}(\D')$, where $u^{\delta^{\sigma}_{u}}$ is 1 if $\sigma\neq u$ and $u$ otherwise.

    Using as bootstrap assumptions $W_v,W_\omega\in \Hb^{1/2-\delta;k-1}(\D'),W_u\in\Hb^{3/2-\delta;k-1}(\D')$ for $0<\delta\ll1$,
    we can  recover
    \begin{nalign}
       Y:=2g^{\mu\nu}W_\mu\partial_\nu=\Hb^{1/2-\delta;k-1}(\D')\pu+\Hb^{3/2-\delta;k-1}(\D')\pv+\Hb^{3/2-\delta;k-1}(\D')\partial_{\omega_a}\\
        u^{\delta^{\sigma}_{u}}\partial_\sigma g^{v\mu}W_\mu \in \Hb^{2-\delta;k-1}(\D'),\quad u^{\delta^{\sigma}_{u}}\partial_\sigma g^{\mu\nu}W_\mu W_\nu \in \Hb^{2-2\delta;k-1}(\D').
    \end{nalign}

    Hence, by a limiting procedure and a bootstrap argument, we can find $u^{\delta^{\sigma}_{u}}W_\sigma\in\Hb^{1/2-\delta;k-1}(\D')$ solving \cref{eq:EVE:eikonal}.
    Integrating from $\scrim$ with $w|_{\scrim}=0$ yields $w\in\Hb^{1/2-\delta;k-1}(\D')$ solving $g\big(\dd(v+w),\dd(v+w)\big)=0$.
    This shows that the level sets of $v+w$ are light cones terminating at $\scrim$.
    
    \emph{b)} Solving the transport equation for $W_\mu$, we immediately get $W_v,W_\omega,uW_u\in\Hbt^{3/2-;k-1}(\D')$.
    Integrating this yields the desired expression for $w$.
\end{proof}

\subsubsection{Proof of scattering in the region \texorpdfstring{$\D^-$}{D-}}\label{sec:EVE:proof}
We are now ready to prove our main scattering result:

\begin{thm}[EVE scattering]\label{thm:EVE:scatteringpast}
    Let $k\geq300$, and $h_{\mathcal{UU}}|_{\D^{v_0+1}_{v_0}}$, $\{a^{\mathcal{UU}}_-\}$ be as in \cref{ass:EVE:data}. Fix $a_-=\min\{a^{\mathcal{UU}}_-\}$ and some $a_0\in(-1,a_-)$.
    Let 
	$\mathfrak{h}_{\mathcal{UU}}\in\Hb^{a_0;k}(\scrim)$ be scattering data on $\scrim$ with \cref{eq:EVE:restriction_on_scattering_data} as in \cref{def:EVE:data}.
    \begin{enumerate}[label={\alph*)}]
        \item
            Fix $c_-=\min(a^{\Lbar\Lbar}_-,2a_-+1,a_-+1,1-)>0$ and $c_0=\min(c_--, a_0)$.\footnote{In particular, $c_0$ is always less than 1! \label{footnote:lindblad}}
        	Then, for $-u_0$ sufficiently large, \cref{eq:EVE:EVE_schematic_past} admits a unique scattering solution in $\D^-$ in the sense of \cref{def:EVE:data}, i.e.~$h$ solves \cref{eq:EVE:EVE_schematic_past}, agrees with $h_{\mathcal{UU}}$ in $\mathcal{D}^{v_0+1}_{v_0}$ and 
            \begin{equation}\label{eq:EVE:solution_past}
    rh_{\Lbar\Lbar}\in\Hbt^{a^{\Lbar\Lbar},a_0;k}(\D^-)+\Hb^{c_-,c_0;k-3}(\D^-),\quad  rh_{\mu\nu}\in\Hbt^{a_-,a_0;k}(\D^-)+\A{phg,b}^{\overline{(0,1)},a_0;k-1}(\D^-)+\Hb^{c_-,c_0;k-3}(\D^-).
            \end{equation}
            Uniqueness is in the class $rh_{\mathcal{UU}}$ satisfying \cref{eq:EVE:solution_past} and \cref{eq:EVE:scattering_def} for $k=10$ and $k=1$ respectively.
        	Moreover, if $h$ is a solution to the Einstein vacuum equations in harmonic gauge in $\D^{v_0+1}_{v_0}$, then $h$ is a solution to the Einstein vacuum equations in $\D^-$ in harmonic gauge.
        \item (No incoming radiation) If $\mathfrak{h}=0$, then $rh_{\mathcal{UU}}\in\Hbt^{a_-,a_--;k-3}(\D^-)$.
        \item (Propagation of polyhomogeneity)
            If, in addition,  $\mathfrak{h}_{\mathcal{UU}}\in\A{phg}^{\E^{\scrim}}(\scrim)$ and $rh_{\mathcal{UU}}\in \A{phg}^{\E^{\mathcal{UU}}}(\D^{v_0+1}_{v_0})$, then the scattering solution satisfies  $rh_{\mathcal{UU}}\in\A{phg}^{\vec{\E}}(\D^-)$ for some $\vec{\E}$.
            
    \end{enumerate}
	
\end{thm}

\begin{proof}
  \textit{a)}  We first show the scattering statement \textit{a)}.
The proof proceeds in three steps: In the first, we incorporate the data in the initial slab and along $\scrim$ into an approximate solution~$h^-$. 
By subtracting $h^-$ and using a cutoff, we then rewrite the equations as equations with trivial data in the second step (where the data have been put into the inhomogeneity).
In the third step, we show that the arising system of equations falls into our usual framework of admissible perturbations.

\emph{Step 1: Peeling off an approximate solution.}
    Let us write $rh^{-,1}_{\mathcal{UU}}(u,v,\omega)=rh_{\mathcal{UU}}(u,v_0,\omega)\in\Hbt^{a^{\mathcal{UU}}_-,a^{\mathcal{UU}}_-;k}(\D^-)$.
    It holds that $\bar{\Box}_0 h^{-,1}_{\mathcal{UU}}\in\Hbt^{a^{\mathcal{UU}}_-+3,a^{\mathcal{UU}}_-+3;k-2}(\D^-)$.
    We can integrate from $\incone$ to obtain that $r(h-h^{-,1})_{\mathcal{UU}}\in \Hbt^{a^{\mathcal{UU}}_-+1;k-1}(\D^{v_0+1}_{v_0})$.
    
    Next, we perform a single iterate to be able to peel off the leading order polyhomogeneous part of $h_{LL}$.
    Recall the definition of $\mathfrak{h}^{(0,1)}_{LL}:=\int\dd v \big(rF_{LL}[r^{-1}\mathfrak{h}_{\mathcal{UT}}]\big)|_{\scrim}$.
    Write $rh^{-,2}_{\mathcal{T}\mathcal{U}}(u,v,\omega)=\mathfrak{h}_{\mathcal{TU}}(v,\omega)\in\A{phg}^{\mindex{0},a_0;k}(\D^-)$
    and
    \begin{equation}
        rh_{LL}^{-,2}(u,v,\omega):=\log(-u/v)\mathfrak{h}_{LL}^{0,1}(v,\omega)+\mathfrak{h}_{LL}(v,\omega)\in\A{phg}^{\overline{(0,1)},a_0;k-1}(\D^-).
    \end{equation}
    We compute that \cref{eq:EVE:scattering_def} holds for $h^{-}=h^{-,1}+h^{-,2}$ and that
    \begin{equation}
        \bar{\Box}_0h^{-,2}_{\mathcal{UT}}\in\Hb^{3-,a_0+3;k-2}(\D^-),\quad \bar{\Box}_0h^{-,2}_{LL}-F_{LL}[h^{-,2}]\in\Hb^{3-,a_0+3;k-3}(\D^-).
    \end{equation}
    Using \ref{item:EVE:H_past} from \cref{thm:EVE:form_of_EVE_past}, we obtain that the inverse metric satisfies $rH[h^{-}]^{LL}\in\Hb^{\min(a^{\Lbar\Lbar}_-,2a_-+1),a_0;k-2}(\D^-)$, $rH[h^{-}]^{\mathcal{UU}}\in\Hb^{a_-,a_0;k-2}(\D^-)$.
     
    \emph{Step 2: Converting into a forced problem.}
    Let $\chi(v)$ be a smooth cutoff function with  $\chi|_{v>v_0+1}=0=(1-\chi)|_{v<v_0+1/2}$.
    Define the error term $h^+$ via $h=h^{-,1}+\chi(h-h^{-,1})+h^{-,2}+h^{+}=h^-+h^+$. 
    In view of our assumption that $\mathfrak{h}$ is supported on $v\geq v_0+1$, we have $h^+=0$ in $\D^{v_0+1/2}_{v_0}$ (this is exactly the reason why we have to include the cutoff---we don't want to pick up any boundary terms in the initial layer when applying our weighted energy estimates to $h^+$.)

    Using \cref{thm:EVE:form_of_EVE_past} and inserting the above bounds on $h-h^{-,1}$ as well as $h^{-,2}$, we may compute the nonlinear terms from \cref{eq:EVE:EVE_schematic_past} to be
    \begin{nalign}\label{eq:EVE:past_proof1}
		Q_{\mathcal{UU}}[h]&=Q_{\mathcal{UU}}[h^+]+V^+_{\mathcal{UU}}[h^+]+\Hb^{\min(4+2a_-,3-),4+2a_0;k-2}(\D^-)\\
		G_{\mathcal{UU}}[h]&=G_{\mathcal{UU}}[h^+]+Q^+_{\mathcal{UU}}[h^+]+V^+_{\mathcal{UU}}[h^+]+\Hb^{3+a_-,5+3a_0;k-2}(\D^-)\\
		F_{LL}[h_{\mathcal{UT}}]&=F_{LL}[h^+_{\mathcal{UT}}]+V^+_{\mathcal{UU}}[h^+_{\mathcal{UT}}]+F_{LL}[h^-_{\mathcal{UT}}]\\
		\bar{\Box}_h h_{\mathcal{UU}}&=\bar{\Box}_{h}h^+_{\mathcal{UU}}+V^+_{\mathcal{UU}}[h^+]+Q^+_{\mathcal{UU}}[h^+]+\Hb^{\min(a^{\Lbar\Lbar}_-,2a_-+1,a_-+1)+2,a_0+3;k-3}(\D^-)
	\end{nalign}
    where $V^{+}_{\mathcal{TU}}$ are potential perturbations with weights $\vec{a}^V:=(\min(3/2+a_-,1+{a^{{\Lbar\Lbar}}_-}),3+a_0)$; where $V^+_{LL}$ is a potential perturbation with weight $(1-,3+a_0)$ depending on $h^+_{\mathcal{UT}}$ and with weight $\vec{a}^V$ depending on $h^+_{\mathcal{UU}}$; and where $Q^+_{\mathcal{UU}}$ is a nonlinear perturbation of order 2 and weight $(2+a_-,4+a_0)$.
    In conclusion, we get that $h^+$ satisfies 	\begin{nalign}\label{eq:EVE:EVE_schematic_past_0}
		\bar{\Box}_{h}h^+_{\mathcal{T}\mathcal{U}}&=A_{\mathcal{T}\mathcal{U}}[h^+]+Q^+_{\mathcal{T}\mathcal{U}}[h^+]+G^+_{\mathcal{T}\mathcal{U}}[h^+]+V^+_{\mathcal{T}\mathcal{U}}[h^+_{\mathcal{U}\mathcal{U}}]+f_{\mathcal{T}\mathcal{U}}\\
		\bar{\Box}_hh^+_{LL}&=A_{LL}[h^+]+F_{LL}[h^+_{\mathcal{T}\mathcal{U}}]+Q^+_{LL}[h^+]+G^+_{LL}[h^+]+V^+_{LL}[h^+_{\mathcal{UU}}]+f_{LL}
	\end{nalign}
	with modified perturbations ($G^+,Q^+$) of the same decay rate as in \cref{eq:EVE:EVE_schematic_past};   linear perturbations $V^+$ as in \cref{eq:EVE:past_proof1}; and inhomogeneities $rf_{\mathcal{UU}}\in\Hb^{c_-+1,a_0+2;k-3}(\D^-)$. 
    Let us note that under the assumption $a_->-1/4, a^{\Lbar\Lbar}>1/2$, there exists $\delta>0$ such that $rf_{\mathcal{UU}}\in\Hb^{3/2+\delta,2+a_0;k-3}(\D^-)$.
    Furthermore, $f_{\mathcal{UU}}$ vanishes for $v<v_0+1/2$.
    We can thus consider this as a system of equations for $h^+$, with $h^+$ vanishing in $\D^{v_0+1/2}_{v_0}$, so we won't pick up boundary terms in the energy estimates. We will next show that the system of equations is admissible:

    \emph{Step 3: Testing admissibility.}
	Finally, we show that \cref{eq:EVE:EVE_schematic_past_0} is an admissible perturbation for $rh^{+}_{\mathcal{UU}}\in\Hb^{\vec{c}^{\mathcal{UU};k-4}}(\D^-)$ for 
	\begin{equation}\label{eq:EVE:past_proof2}
		c^{\mathcal{UU}}_0=\min(a_0,c_--),\quad c^{\mu\nu}_-=\begin{cases}
			c_--\epsilon & \text{if } \mu=\nu=L,\\
			c_- & \text{else},
		\end{cases}
	\end{equation}
	where  $0<\epsilon\ll c_-$. 
    More precisely, we need to confirm that $h$ corresponds to an admissible metric perturbations, that $A$ is an admissible long-range potential, that $Q,G,V$ are admissible short-range perturbations, and that the inhomogeneities $f$ are admissible:
	That the metric \cref{eq:EVE:g_past} is admissible follows from $a^{\Lbar\Lbar}_->0$ and $a_->-1/2$.
    Clearly, $A$ is an admissible long-range potential. 
	The inhomogeneity in \cref{eq:EVE:EVE_schematic_past_0} is also admissible for the weights in \cref{eq:EVE:past_proof2}.
    
	We are left with the proving that the remaining terms are short-range perturbations: Since $c^{\mathcal{UU}}_->1/2$, admissibility is equivalent to the perturbations on the right hand side \cref{eq:EVE:EVE_schematic_past_0} having weights larger than $\vec{c}^{\mathcal{UU}}+(1,2)$ .
	
	We compute that the condition for the nonlinearities forcing  $h^{+}_{\Lbar\Lbar}$ to be perturbative is given by
	\begin{equation}
		\min\{2(c^{LL}_-,c_0)+(3/2,3),3(c^{LL}_-,c_0)+(2,4),c_-^{LL}+(a^{V}_-,3+a_0)\}>(1,2)+(c^{\Lbar\Lbar}_-,c_0,),
	\end{equation}
    which indeed holds. 
	All other components $h_{\mathcal{UT}}$ can be treated in the same way. 
	
	Finally, for $h_{LL}$, the admissibility condition is given by
	\begin{equation}
		\min\{2(c_-,c_0)+(1,3),3(c^{LL}_-,c_0)+(2,4),2(c^{LL}_-,c_0)+(3/2,3),c_-+(1-,3+a_0)\}>(1,2)+(c^{LL}_-,c_0).
	\end{equation}
	Again, this clearly holds.
	
	Since we require no incoming data for $h^{+}$, the result of scattering and uniqueness, together with the estimates, now follows from \cref{thm:current:quasilinear_scattering} together with \cref{lemma:EVE:eikonal}.

    Note that, as explained below \cref{eq:EVE:restriction_on_scattering_data}, the scattering data $\mathfrak{h}$ is chosen precisely so that the wave equation for $\Upsilon[h]$ has no data at $\scrim$, so provided $\Upsilon[h]=0$ in $\D^{v_0+1}_{v_0}$ we obtain that $\Upsilon[h]=0$ in $\D^-$.
    
    \emph{b)} We turn our attention to solutions with no incoming radiation. The previous construction doesn't quite work here because if for $h^{-,2}$, the appearance of the cut-off in the definition of $\chi$ generates artificial incoming radiation (where $\pv\chi\neq0$) and we thus don't have compatibility with no incoming radiation.
    We therefore define instead $h^{+,1}:=h-h^{-,1}$, with $h^{-,1}$ as before. 
   This now has support even in $\D^{v_0+1}_{v_0}$.
   In order to still avoid boundary terms in the energy estimates from \cref{lemma:current:perturbative_h}, we multiply all our multipliers with the same cut-off function $\chi$ as before. 
    This will generate additional bulk terms with a bad sign in the region $\D^{v_0+1}_{v_0}$; but these pose no problem because we already have a priori control on $h$ in this initial slab.
    We can therefore apply the same proof as before, but with nonlinearities and inhomogeneities now compatible with no incoming radiation. The result then follows, again, from \cref{thm:current:quasilinear_scattering}.

    \emph{c)} The propagation of polyhomogeneity follows as in \cref{sec:app}.
    \end{proof}

\subsubsection{An example construction of solutions with fast convergence to Schwarzschild}\label{sec:EVE:fast_decaying_scattering}
In forward problems, it is often assumed that the data at $t=0$ are given by the Schwarzschild initial data plus terms that decay faster, i.e.~plus terms of order $r^{-1-\epsilon}$ for some $\epsilon>0$. In particular, this is assumed in \cite{Lindblad2010,hintz_stability_2020}. However, as pointed out in \cite{kehrberger_case_2024-1}, general scattering solutions cannot be expected to exhibit this fast decay to Schwarzschild towards spacelike infinity.
In particular, note that, even for compactly supported data $\pv\mathfrak{h}$ along $\scrim$ and trivial data in the slab $\D^{v_0+1}_{v_0}$, the solutions generated by \cref{thm:EVE:scatteringpast} only decay like $h\sim r^{-1}$ towards $I^0$ (cf.~\cref{footnote:lindblad}), which one can show to be sharp (and supported on all angular modes) with a bit more work.
This is, in fact, analogous to the observations from \cref{sec:prop} that compactly supported scattering data $\pv\psi^{\scrim}$ for $\Box_\eta\phi=0$ lead to $\phi\sim r^{-1}$ decay towards spacelike infinity, cf.~\cref{rem:poly:improved_decay_at_I0,obsidobsi}. 

The latter observations also state that the $1/r$ coefficient towards $I^0$ vanishes if the $v$-integral over $\pv\psi^{\scrim}$ along~$\scrim$ vanishes. 
The purpose of this subsection is to show that this still holds for the Einstein vacuum equations. More precisely, we will show that for compactly supported data with vanishing integral, the solution converges to Schwarzschild plus $r^{-2}$-terms near spacelike infinity.
For simpler notation, we work in the smooth category ($k=\infty$).

\begin{prop}[Solutions with fast convergence to Schwarzschild]\label{lemma:EVE:converging_to_Sch}
    There exist scattering data $\mathfrak{h}$ (as in \cref{def:EVE:data}) such that the corresponding solution $g=\eta+h$ to $\cref{eq:EVE:rEVE}$ (with trivial data in $\D^{v_0+1}_{v_0}$) up to a diffeomorphism satisfies $\big(\eta+h-g_{\mathrm{Sch}})_{\mathcal{UU}}\in\Hb^{2-,2-}(\D^-)$.
    In particular, the solution is exactly Minkowski before an incoming cone $\incone$ and on a Cauchy hypersurface approaches a Schwarzschild exterior at a rate faster than required in \cite{Lindblad2010} (namely at rate $g-\eta\sim r^{-2+}$).
\end{prop}
\begin{proof}
    We take the solution to be exact Minkowski space before the cone $\incone$, and prescribe scattering data with $\mathfrak{h}=0$ in $\D^{v_0+1}_{v_0}$ and
\begin{equation}
   \pv \mathfrak{h}_{\mu\nu}=\begin{cases}
        \chi_{ab},\qquad \mu,\nu\notin\{L,\Lbar\}\\
        0,\qquad \text{else},
    \end{cases}
\end{equation}
such that $\mathrm{tr}_{g_{S^2}} \chi=0$ and such that the $v$-integrals over the components $\chi_{ab}$, $\int_{v_0}^v \chi_{ab}\dd v' \in C^\infty_c(\scrim\cap\{v\in(v_0+1,v_0+2)\})$ are compactly supported.
Then, observing the (time-reversed version of the) explicit form of $F_{LL}[h_{\mathcal{UT}}]$ from \cite[Eq.~(3.19),(5.5)]{lindblad_global_2005}, we get that $\mathfrak{h}_{LL}^{(0,1)}(v,\omega)=\frac12\int\dd v (g_{S^2}^{-1})^{ca}(g_{S^2}^{-1})^{db}\chi_{cd}\chi_{ab}$.
Write $8m=\mathfrak{h}_{LL}^{(0,1)}(v_0+2,\omega)\in C^\infty(S^2)$ and note that $\mathfrak{h}_{LL}^{(0,1)}(v,\omega)=\mathfrak{h}_{LL}^{(0,1)}(v_0+2,\omega)$ $\forall v>v_0+2$. 
From \cref{thm:EVE:main_scattering}, we get that the corresponding solution $h$ satisfies in $\D^{v_0+3}_{v_0+2}$ 
\begin{equation}\label{eq:EVE:particular_proof}
    rh_{\mu\nu}=\Hb^{1-}(\D^{v_0+3}_{v_0+2})+\begin{cases}
        8m\log(u),&\mu=\nu=L\\
        0,&\text{else}.
    \end{cases}
\end{equation}
In fact, we can upgrade \cref{eq:EVE:particular_proof} to $\Hbt^{1-}(\D^{v_0+4}_{v_0+3})$ using that there is no longer any incoming radiation for $v\geq v_0+3$.

We now split $m=m_0+\slashed{m}$, with $\int_{S^2}\slashed{m}=0$.
Using \cite[Lemma~3.2]{kehle_event_2023} we get that there exists $\chi_{ab}$ such that $\slashed{m}=0$. We will restrict to such $\chi$ from now on. In particular, $m$ is spherically symmetric.\footnote{We give a brief construction here:  consider $\chi^{(1)}=\dd \theta^2-\sin^2(\theta)\dd\varphi^2$ defined in $U_1=\{\abs{\hat{z}}=\abs{\cos\theta}<0.9\}$.
We have $\tr_{g_{S^2}}\chi^{(1)}=0$ and $\abs{\chi^{(1)}}^2_{g_{S^2}}=2$.
Set $\chi^{(2)}$ to be similar with respect to the $y$ axis defined on $U_2=\{\abs{\hat{y}}<0.9\}$.
For $i\in\{1,2\}$, fix $f_i\in C_c^\infty(U_1)$ such that $f_1^2+f_2^2=1$.
Then $\chi=\mathfrak{x}(v-v_0)f_1\chi^{(1)}+\mathfrak{x}(v-v_0-1)f_1\chi^{(2)}$ satisfies the requirement for $\supp\mathfrak{x}\subset[0,1]$.}

At this point, we want to construct coordinates such that the solution looks to leading order like Schwarzschild.
Consider $\Cbar$ a null cone for $g[h]$ in $\D^{v_0+4}_{v_0+3}$, the existence of which follows from \cref{lemma:EVE:eikonal}.
We define $\bar{x}_\mu$ to solve
\begin{equation}
    \bar{\Box}_h (\bar{x}_\mu-x_\mu)=0, \qquad \bar{x}_\mu|_{\Cbar}=x_\mu\big(1+r^{-1}2m\log r\big),\quad {\pv(r(\bar{x}_\mu-x_\mu))|_{\scrim}=0.}
\end{equation}
Using the fast decay of the metric, a basic argument gives that in $\D^{v_0+4}_{v_0+3}$, $\bar{x}_0=x_0-2m\frac{-u}{r}\log(-u)+\Hbt^{1-}(\D^{v_0+4}_{v_0+3})$, $\bar{x}_i=x_i+\frac{x_{i}}{r}2m\frac{-u}{r}\log(-u)+\Hbt^{1-}(\D^{v_0+4}_{v_0+3})$. In particular, to leading order, we have: $\bar{v}=v$, $\bar{u}=u-2m\frac{-u\log(-u)}{r}$.
Setting $\Phi:\bar{x}\mapsto x$, we compute the pull-back:
  \begin{equation}
      \Phi^{*}(\eta+h)=\underbrace{-4\dd\bar{u}\dd\bar{v}+\bar{r}^2 g_{S^2}}_{:=\bar{\eta}}+\frac{8m}{\bar{r}}\dd \bar{u}\dd \bar{v} - 4\bar{r}m\log \bar{r} g_{S^2}\qquad \mod \Hbt^{2-}(\D^{v_0+4}_{v_0+3}),
\end{equation}
where the error term is expressed with respect to the old $\mathcal{UU}$ frame. We will write $\Phi^{\ast}(\eta+h)=\bar{\eta}+\bar{h}$---note that $\bar{\eta}$ is not the pull-back of $\eta$!

All that's left to do now is to notice that the leading order term  above, $r^{-1}8m\dd u\dd v-4r m \log r g_{S^2}$,is the same as the leading order term in the Schwarzschild metric written in harmonic coordinates: 
Recall that we can write $g_{\mathrm{Sch}}=-4(1-2m/r)\dd u\dd v+r^2g_{S^2}$ with respect to coordinates $u,v,\omega$, where $r(r^\star)=r^\star-2m\log r^\star+\Hb^{1-}(\R_r)$ and $r^\star=v-u$, yielding the required leading order form.
Here $\Hb^{1-}(\R_r)$ is understood with respect to the compactification of $\R_r$ given by $1/r$.
The corresponding rectangular coordinates $x_\mu$ are not harmonic with respect to $g_{\mathrm{Sch}}$, but we can construct out of $x$ harmonic coordinates $\tilde{x}$ by solving the wave equations
\begin{equation}
    \Box_{g_{\mathrm{Sch}}}(\tilde{x}_\mu-x_\mu)=-\Box_{g_{\mathrm{Sch}}}x_\mu,\quad (\tilde{x}_\mu-x_\mu)|_{\incone}=0,\, \pv(r(\tilde{x}_\mu-x_\mu))|_{\scri-}=0.
\end{equation}
In $\tilde{x}_\mu$ coordinates, we still have the correct leading order expression, i.e., for $\tilde{\Phi}: \tilde{x}\mapsto x$
\begin{equation}
    \tilde{\Phi}^{\ast}g_{\mathrm{Sch}}=-4(1-2m/\tilde{r})\dd \tilde{u}\dd \tilde{v}+\tilde{r}^2 g_{S^2}\quad \mathrm{mod}\quad  \Hb^{2-,2-}(\D^-),
\end{equation}
where $r$ is now defined via $r(\tilde{v}-\tilde{u})$.
 We now define a Schwarzschild metric $\bar{g}_{\mathrm{Sch}}$ to be the metric above but with $\tilde{x}$ replaced by $\bar{x}$. 

We may then conclude that, in $\D^{v_0+4}_{v_0+3}$, $\bar{\eta}_{\mathcal{UU}}+\bar{h}_{\mathcal{UU}}=\Phi^{\ast}(\eta+h)_{\bar{\mathcal{U}}\bar{\mathcal{U}}}=(\bar{g}_{\mathrm{Sch}})_{\bar{\mathcal{U}}\bar{\mathcal{U}}}+\Hb^{2-}(\D^{v_0+4}_{v_0+3})$.  

Using that $\bar{g}_{\mathrm{Sch}}$ solves the Einstein equations in harmonic coordinates, we finally obtain that $\bar{h}_{\bar{\mathcal{U}}\bar{\mathcal{U}}}=(\bar{g}_{Sch}-\bar{\eta})_{\bar{\mathcal{U}}\bar{\mathcal{U}}}+\Hbt^{2-}(\D^-\setminus \D^{v_0+4}_{v_0})$, working now with coordinates $\bar{x}$.
\end{proof}

\subsection{Construction of solutions in the future region \texorpdfstring{$\D^+$}{D+} }\label{sec:EVE:future}
The solutions constructed in the previous subsection are not compatible with the decay assumed in the standard stability results \cite{Lindblad2010,hintz_stability_2020}.
In this section, we therefore show that the results of \cite{hintz_stability_2020}  (including the propagation of polyhomogeneity) still hold when the data have slower decay. More precisely, we will consider solutions that, near spacelike infinity, decay like $g-\eta\sim r^{-1+a_0}$ with $a_0<0$.\footnote{We re-emphasise the necessity of considering such data in the context of scattering problems, as, in general, one expects that the metric will have $a_0=0$ deviation from the Minkowski metric with perturbation supported on all $\ell$ modes, see \cite[Eq.~(1.5)]{kehrberger_case_2024-1}. We also recall that he existence of solutions for such data is of course known from the works \cite{bieri_extension_2010,shen_stability_2023}, see also \cite{leflochNonlinearStabilitySelfgravitating2024}, where a similar problem is studied in harmonic gauge.} 
We note that the treatment of any $a_0<0$ is already sufficient for the physically important scenario captured in \cref{ass:EVE:data}.

The only novelty compared to \cite{hintz_stability_2020} is that we must correct for the bending of the lightcones imposed by harmonic gauge \textit{without} having access to the explicit Schwarzschild solution.
This is done in \cref{prop:EVE:ansatz} by finding an ansatz to correct for this bending.
As we only do one iteration, we restrict to $a_0>-1/2$, which yields that the second iterate will have decay towards $I^0$ as $h^2_{\mathcal{UU}}\sim r^{-(1+a_1)}$ for $a_1=1+2a_0>0$, thus falling under the same setting as considered in \cite{Lindblad2010,hintz_stability_2020}. With a bit more work, one could generalise this to $a_0>-1$.

In this section, it will be convenient to introduce the analogue of $\Hbt$ near $\scrip$.
Indeed, in this section \emph{exclusively}, we write
\begin{equation}
	\Hbt^{\vec{a};k}(\D^+)=\{f\in\Hb^{\vec{a};k}(\D^+):\{r\pu,\Vb\}^kf\in\Hb^{\vec{a};0}(\D^+)\}.
\end{equation} 
We will use the following characterisation of solutions with weakly decaying data

\begin{lemma}[Weakly decaying data]\label{lemma:EVE:weak_data_at_I0} Fix some $k\in\N$ sufficiently large.

    \textit{a)} Let $a_0\in(-1,0)$.
    Let $\phi$ be the forward solution to $\Box_\eta\phi=f\in\Hb^{a_0+3;k}(\D^+\cap\D^-)$, i.e.~ $\phi|_{\D^-\setminus\D^+}=0$.
    Then we have $r\phi\in\Hbt^{a_0-,a_0-;k-2}(\D^+)+\Hb^{a_0-,0-;k-4}(\D^+)$.
    The same holds for the initial value problem with $\phi|_{t=0}\in\Hb^{a_0+1;k+1}$ and $\partial_t\phi|_{t=0}\in\Hb^{a_0+2;k}$.

    \textit{b)} For $a_0\in(-2,-1)$ the above conclusion holds with extra derivative losses $r\phi\in\Hbt^{a_0-,a_0-;k-4}(\D^+)+\Hb^{a_0-,0-;k-6}(\D^+)$.

    \textit{c)} For $f_H,\vec{a}^H$, $\delta=1-$ and quasilinear perturbations as in \cref{lemma:current:perturbative_h}, the same result holds for $\tilde{\Box}_H\phi=f$ provided that the coefficients are at least $f_H\in\Hbt^{\vec{a}^H;\min(k+2,10)}(\D^+)$ regular.
\end{lemma}
\begin{proof}
    a)
	Using \cref{prop:en:futureestimate} we already have $r\phi\in\Hb^{a_0,a_0-;k}(\D^+)$.
	Treating $r^{-2}\Dl$ as a perturbation and using \cref{corr:ODE:du_dv}, we have $\pu r\phi\in\Hb^{a_0+1,a_0+1-;k-2}(\D^-)+\A{b,phg}^{a_0+1,\mindex{0};k-2}(\D^-)$.
	In conclusion, we get that $ r\phi$ has a finite radiation field.
	We also have that $r\phi|_{\outcone{}}\in\Hb^{a_0-;k}(\outcone{})$.
    Next, we view the solution as arising from backwards scattering and split it into two parts, one with trivial data along $\outcone{}$ and $\pu  \psi|_{\scrip}$ along $\scrip$, the other with $\psi|_{\outcone{}}$ along $\outcone{}$ and with trivial data along $\scrip$.

    For the solution with nontrivial data only on $\outcone{}$, the result follows from \cref{thm:scat:scat_general}.

    For the solution arising from nontrivial data on $\scrip$, we use the trick as in \cref{thm:scat:scat_incoming} to write the solution as $\psi|_{\scrip}(v,\omega)+\bar{\psi}(u,v,\omega)$.
    Then $\bar{\psi}$ solves $\Box_\eta\bar{\phi}=-r^{-3}\Dl\psi|_{\scrip}(u,\omega)\in\Hb^{3-,a_0+2-;k-4}(\D^+)$.
    The result for $\bar\psi$ follows from \cref{thm:scat:scat_general}.

    For nontrivial initial data, we may use the energy estimates with boundary terms.

    b) For $T\phi$ part a) applies. Then we recover $\phi$ by integrating the previous two components with data along $\outcone{}$ and on $\scrip$ separately using \cref{corr:ODE:du_dv}.

    c) Note that throughout the proof, we  treated $r^{-2}\Dl\in \rho_+^2\rho_0^2\Diff_{\b}^2$ as a perturbation for constructing the solution.
    Therefore, the proof still holds for $\tilde{\Box}_H\phi=f$ when replacing \cref{prop:en:futureestimate,thm:scat:scat_general} with \cref{thm:current:metric_scattering}.
\end{proof}

To study the evolution of \cref{eq:EVE:EVE_schematic} from $I^0$ towards $\scrip$, we must first create an initial ansatz. 
In fact, the construction of this ansatz is the main novelty of this section compared to previous works.
In hopes of making \cref{prop:EVE:ansatz} easier to follow, we use the following remark to discuss some of the main ideas going into the construction of the ansatz at the example of the Schwarzschild solution. 
\begin{rem}[Constructing an ansatz for Schwarzschild]\label{rem:EVE:schwex}
Suppose we are already given the Schwarzschild solution in some harmonic coordinates $\tilde{x}$, namely
\begin{equation}\label{eq:EVE:sch_in_harmonic_1}
    \tilde{g}=\tilde{g}_{\mathrm{Sch}}=-\dd \tilde{t}^2 \frac{\tilde{r}-m}{\tilde{r}+m}+\dd \tilde{r}^2\frac{\tilde{r}+m}{\tilde{r}-m}+(\tilde{r}+m)^2 \tilde{g}_{S^2}.
\end{equation}
If we use these coordinates to compare against Minkowski ($\tilde{\eta}=-\dd \tilde{t}^2+\dd \tilde{r}^2+\tilde{r}^2g_{S^2}$), the asymptotic geometry is not captured well:\footnote{Of course, one can simply write \cref{eq:EVE:sch_in_harmonic_1} with respect to tortoise coordinates and then find harmonic coordinates the keep the light cones fixed, but the point here is to mimic the proof of \cref{prop:EVE:ansatz} below.}
Indeed, note that to leading order, 
\begin{equation}\label{eq:EVE:ex:gsch}
    \tilde{g}=\tilde{\eta}+\tilde{h}=\tilde{\eta}+\underbrace{\frac{2m}{\tilde{r}}(\dd \tilde{t}^2+\delta^{ij}\dd\tilde{x}_i\dd\tilde{x}_j)}_{:=\tilde{h}_{\mathrm{lin}}}+\underbrace{\text{ higher order terms}}_{:=\tilde{h}^{\mathfrak{Err}}}.
\end{equation}
While $\tilde{h}_{\mathrm{lin}}$ is a solution to the linearised Einstein vacuum equations and constraint equation that removes the leading-order terms near $I^0$ (so of the data), we have $\tilde{h}_{\mathrm{lin},\tilde{L}\tilde{L}}\sim 4m/\tilde{r}$, so $\tilde{h}_{\mathrm{lin}}$ is not a short-range metric perturbation near future null infinity. 
(Note that in the proof below, we will have to work a bit harder to find $\tilde{h}_{\mathrm{lin}}$; there, it corresponds to $\tilde{\chi}h+\tilde{h}^{(1)}$.)

The goal of the construction below is to find a change of coordinates $x\mapsto \tilde{x}$ that removes this problematic term.
Of course, given any other coordinate set $x:\D\to\R^{3+1}$, we can still define $g$ and ${h}_{\mathrm{lin}}$ in analogy to $\tilde{g}$ and $\tilde{h}_{\mathrm{lin}}$ above (simply replacing all $\tilde{x}$'s by $x$'s)---note that these are not pull-backs!\footnote{For instance, given a diffeomorphism $\Phi:x\mapsto \tilde{x}$ between two charts we write $\eta=-\dd t^2+\dd x^2$ and $\tilde{\eta}=-\dd \tilde{t}^2+\dd \tilde{x}^2$, but of course $\Phi^\star \tilde{\eta}\neq\eta$.}

Now, as in \cref{item:current:metric_coordiante_change}, we can find a coordinate change correcting this bad behaviour: Indeed, we find a  map $u\mapsto\tilde{u}=F(u)+u$ with $\pu\pv F=\pu ({h}_{\mathrm{lin},LL})$, $v\mapsto \tilde{v}=v$, which in rectangular coordinates reads
\begin{equation}\label{eq:EVE:ex:prelimcoords}
    \Phi_{\mathrm{prelim}}:(t,x_i)\mapsto (\tilde{t},\tilde{x}_i)=(t+8m\frac{v\log v}{r},x_i).
\end{equation}
Note that the new coordinates $\tilde{x}$ are harmonic with respect to $\eta$.
We compute the pull-back of $\tilde{\eta}$,
\begin{equation}
     \Phi_{\mathrm{prelim}}^{\ast}\tilde{\eta}=\underbrace{-4 \dd u \dd v+r^2 g_{S^2}}_{:=\eta}-\frac{4m}{r}\dd v^2+\text{ higher order terms},
\end{equation}
and note that it precisely accounts for the bad behaviour of ${h}_{\mathrm{lin},LL}$; indeed:
\begin{equation}
 (\Phi_{\mathrm{prelim}}^{\ast}\tilde{\eta})_{LL}-\eta_{LL}=-\frac{4m}{r}+\rho_+\Hbt^{1-,1-}(\D^+).
\end{equation}
We define $h^1_{\mathrm{prelim}}:=h_{\mathrm{lin}}+\Phi_{\mathrm{prelim}}^{\ast}\tilde\eta -\eta$.

Finally, we relabel the thus found coordinates $x_{\mu}$ to $x_{\mathrm{prelim,\mu}}$, because we now want to correct $ \Phi_{\mathrm{prelim}}^{-1}$ to be a harmonic map with respect to the actual $\tilde{g}$ in a neighbourhood of spacelike infinity, i.e.~in an open neighbourhood $\tilde{U}\subset\D^+\cap\D^-$, say $\tilde{U}:=\D^+\cap\D^-\cap\{\tilde{\rho}_+\in(-1/10,1/10)\}$. 
To  this end, we let $\tilde{\chi}$ be a cut-off localising to $\tilde{U}$ and define the metric $\underline{g}=(1-\tilde{\chi})\eta+\tilde{\chi}\tilde{g}_{\mathrm{Sch}}$, which can be evaluated in the $\tilde{x}$ ($x$) chart by a pushforward (pullback) via $\Phi_{\mathrm{prelim}}$,
Define $y_{\mu}$ to solve
\begin{equation}
    \Box_{(1-\tilde{\chi})\eta+\tilde{\chi}\tilde{g}_{\mathrm{Sch}}} y_\mu=0,
\end{equation} 
with data $y_{\mu}=\Phi^{-1}_{\mathrm{prelim}}(\tilde{x})_\mu$ in $\tilde{\chi}=0$.
We now define $\Phi^{-1}(\tilde{x})_\mu=y_\mu$.

Finally, we define
\begin{equation}
    h^1:= h_{\mathrm{lin}}+\Phi^{\ast}\tilde\eta-\eta,\quad \text{and}\quad h^2:=\Phi^{\ast}(\tilde{g})-h^1.
\end{equation}
One then checks (for details see the proof below) that $h^1$ satisfies
    \begin{equation}
        h^1_{\mathcal{UU}}|_{U}=\Big(\Phi^{*}( \tilde{g}_{\mathrm{Sch}})-\eta\Big)_{\mathcal{UU}}+\Hb^{2-}(U),
    \end{equation}
    where we used that $(\Phi^{\ast}\tilde{h}_{\mathrm{lin}})_{\mathcal{UU}}=({h}_{\mathrm{lin}})_{\mathcal{UU}}+\Hb^{2-}(U)$ in $U$.
    That is, after applying a  change of coordinates $\Phi$ that is exactly harmonic \emph{within} the region $U$ (and approximately harmonic away from $U$), the corresponding leading order (in terms of decay) solution to \cref{eq:EVE:rEVE} is $h^1$, a \emph{short range metric perturbation} satisfying $rh^1_{LL}\sim \rho_+$ in $\D^+$.
   
    Finally, one checks that $h^2=\Phi^*(\tilde{g}_{\mathrm{Sch}})-h^1$ satisfies a \cref{eq:EVE:h2} type equation with $a_0=0$, i.e.~an equation that we can solve for towards the future with $h^2\sim r^{-2}$ in $U$ (and we no longer have to worry about the bad $1/r$-behaviour of~${h}^{2}_{LL}$).  $\blacksquare$
\end{rem}

We now want to generalise the ideas above to our setting. By \cref{thm:EVE:scatteringpast}, we may suppose we are already given 
 $r\tilde{h}\in\Hb^{a_0;\infty}(\tilde{\D}^+\cap\tilde{\D}^-)$ solving \cref{eq:EVE:rEVE,eq:EVE:EVE} in harmonic rectangular coordinates $\tilde{x}$.
 We let $\tilde{\chi}$ be cutoff function localising to $\tilde{U}:=\D^+\cap\D^-\cap\{\tilde{\rho}_+\in(-1/10,1/10)\}$, a neighbourhood of the interior of $I^0$.
We write $\tilde{h}^{-}=\tilde{\chi} \tilde{h}$.
As in the example above, we can analogously define $h^-$ in $x$-coordinates.
In the proposition below, we will find a diffeomorphism relating the two coordinate systems ${x}$ and $\tilde{x}$, following more or less the ideas above. 

\begin{prop}\label{prop:EVE:ansatz}
    Let $k\geq30$ and $\tilde{h}, \tilde{U}$ be as above with $\norm{r\tilde{h}_{\mathcal{UU} } }_{\Hb^{a_0;k}(\tilde{U})}\ll 1$. Let $a_0\in(-1/2,0]$, and define $a_1:=1+2a_0>0$.
    There exists
    \begin{itemize}
        \item $\Phi:\D^+\to\R^{3+1}$, a diffeomorphism onto its image, such that $\Phi_\mu(x)-x_\mu\in\Hbt^{a_0-,a_0-;k-13}(\D^+)$ and $\bar{\Box}_h \Phi_\mu(x)=0$ in $\{\chi=1\}$ for $\mu\in\{0,1,2,3\}$;
        \item $h^1_{\mathcal{UU}}$, a metric perturbation ansatz, satisfying
        \begin{subequations}
        \begin{align}
            \begin{aligned}
                rh^1_{\mathcal{UU}}&\in\Hbt^{a_0-,a_0-;k-13}(\D^+)+\Hb^{a_0-,0-;k-13}(\D^+),\\
                rh^1_{LL}&\in\Hb^{a_0-,1-;k-13}(\D^+)+\rho_+\Hbt^{a_0-,a_0-;k-13}(\D^+),\label{eq:EVE:future_h1_size}
            \end{aligned}\\
            -\chi h^2:= \chi h^1_\mathcal{UU}-\chi \big(\Phi^*(\tilde{\eta}+\tilde{h})-\eta\big)_{\mathcal{UU}}\in r^{-1}\Hb^{a_1-;k-13}(U);\label{eq:EVE:h2_size}
        \end{align}
        \end{subequations}
        \item furthermore, we can write 
     $\Phi^{\ast}(\tilde{\eta}+\tilde{h})-\eta:=h^1+h^2$ with 
    \begin{nalign}\label{eq:EVE:h2}
        \bar{\Box}_{h}h^2_{\mathcal{UT}}=\mathfrak{F}^2_{\mathcal{UT}}[h^2]+\Hb^{4+2a_0-,3+3a_0-;k-15}(\D^+),\\
        \bar{\Box}_{h}h^2_{\Lbar\Lbar}=\mathfrak{F}^2_{\Lbar\Lbar}[h^2_{\mathcal{UT}}]+\Hb^{4+2a_0-,2-;k-15}(\D^+),\\
        \pu h^2_{LL}=A^2_{\mathrm{t}}[h^2]+\mathfrak{G}_{\mathrm{t}}[h^2]+\Hb^{3+2a_0-,2+2a_0-;k-15}(\D^+),
    \end{nalign}
    where
    \begin{itemize}
        \item all terms in $\mathfrak{F}^2_{\mathcal{UU}}[h^2]-\mathfrak{F}_{\mathcal{UU}}[h^2]$  are short-range perturbations with coefficients of $k-13$ differentiability (and size arbitrarily small depending on $\norm{r\tilde{h}_{\mathcal{UU}}}_{\Hb^{a_0;k}(\tilde{U})}$ ) for $rh^2_{\mathcal{UT}}\in\Hb^{a_1-\delta,-\delta}(\D^+)$ and $rh^2_{\Lbar\Lbar}\in\Hb^{a_1-\delta,-2\delta}(\D^+)$ for $\delta>0$ sufficiently small; 
        \item $A_\mathrm{t}^2\in\rho_0\rho_+^{1+a_0}\Diff^1_b$ is a linear perturbation with weight $(1,1+a_0)$, and $\mathfrak{G}_\mathrm{t}$ is an order 2 nonlinearity with weight~$(2,1)$.
    \end{itemize}
    \end{itemize}
    Finally, all memberships in the previous inclusions can be made sufficiently small depending only on $\norm{r\tilde{h}_{\mathcal{UU}}}_{\Hb^{a_0;k}(\tilde{U})}$.
\end{prop}

The proof closely follows \cref{rem:EVE:schwex} above ($a_0$ is given by $0-$ for $g_{\mathrm{Sch}}$), though there are various additional difficulties: For instance, because we are of course not given $g$ globally (as in \cref{eq:EVE:ex:gsch}), we need to work a bit harder to construct the analogue of $h_{\mathrm{lin}}$. 
Furthermore, our preliminary choice of coordinates \cref{eq:EVE:ex:prelimcoords} will in general not be harmonic with respect to $\eta$ so we will have to account for this as well.

\begin{proof}
    \textbf{Step 1: Constructing a linearised solution.}
    
    For the argument below, we recall from the discussion in \cref{sec:EVE:equations} that, by the linearised Bianchi equations ($\delta_\eta\circ G_\eta\circ D_\eta \mathrm{Ric}=0$), any solution of $D_\eta (\mathrm{Ric}-\delta_g^\star\Upsilon)=\Box_\eta h=0$, i.e.~the covariant wave equation on a 2-tensor $h$, also satisfies $\Box_\eta(D_\eta\Upsilon)=0$.
    In words, the linearised gauge conditions, $D_\eta\Upsilon[h]=\delta_\eta G_\eta h=0$, are propagated.
    
    For the first step, let us work in rectangular coordinates, that is, we express tensors in the frame $\dd x_i,\dd t$ as opposed to $\mathcal{U}$.
    Now, for $h^{-}=\chi h$ from above and the form of the nonlinear terms in \cref{thm:EVE:form_of_EVE}, we compute that on $\{\chi=1\}$, we have
    \begin{equation}
        f^\Box_{\mu\nu}:=(\Box_\eta h^{-})_{\mu\nu}\in \Hb^{4+2a_0;k-2}(U),\quad f^\Upsilon_\mu:=(D_\eta\Upsilon[h^{-}])_{\mu}\in\Hb^{3+2a_0;k-1}(U).
    \end{equation}
    Next, we define $h^{(2)}$ in $\D^+$ via the tensorial wave equation
    \begin{equation}\label{eq:EVE:hlinequation}
        \Box_\eta\big(h^{(1)}+h^{-}\big)=0,\qquad h^{(1)}_{\mu\nu}|_{t=0}=\gamma_{\mu\nu},\quad \partial_t h^{(1)}_{\mu\nu}|_{t=0}=k_{\mu\nu},
    \end{equation}
    for $\gamma,k$ to be specified. Write $\underline{\mathcal{T}}=G_\eta \mathcal{T}$ for all two tensors $\mathcal{T}\in\{h^{(1)},h^-,f^{\Box},k,\gamma\}$.
    We choose $\gamma,k$ such that $\delta_\eta G_\eta ( h^{(1)}+h^-)=D_\eta\Upsilon[ h^{(1)}+h^-]$ has zero initial data at $\{t=0\}$, i.e.
    \begin{nalign}
        -\partial_t \underline{h}^{(1)}_{0\mu}|_{t=0}+\partial_{x_i}\underline h^{(1)}_{i\mu}|_{t=0}&=-\underline k_{0\mu}+\partial_{x_i}\underline \gamma_{i\mu}=-f^{\Upsilon}_\mu|_{t=0},\\
        -\partial_t^2 \underline h^{(1)}_{0\mu}|_{t=0}+\partial_{x_i}\partial_t \underline h^{(1)}_{i\mu}|_{t=0}&=-\Delta\underline \gamma_{0\mu}+\partial_{x_i}\underline k_{i\mu}-\underline f^\Box_{0\mu}|_{t=0}=-\partial_t f^{\Upsilon}_j|_{t=0},
    \end{nalign}
    where $\Delta$ is the Euclidean Laplacian and we used the convention, that $\mu=0$ respectively $\mu=i$ denotes the $\dd t$ and $\dd x^i$ components.
    On $\Sigma=U\cap\{t=0\}$, setting $\underline \gamma_{i\mu}=0$, we obtain $\underline k_{0\mu}=f^\Upsilon_\mu|_\Sigma$ and 
    \begin{equation}\label{eq:EVE:data_construction}
        \Delta \underline \gamma_{00}=\partial_{x_i}f^{\Upsilon}_i|_{\Sigma}+\partial_t f^{\Upsilon}_{0}|_{\Sigma}-\underline f^{\Box}_{00}|_{\Sigma}\in\Hb^{4+2a_0;k-3}(\Sigma),\quad \partial_{x_i}\underline{k}_{ij}=\underline f^{\Box}_{0j}|_{\Sigma}-\partial_t f^{\Upsilon}_j|_{\Sigma}\in\Hb^{4+2a_0;k-3}(\Sigma).
    \end{equation}
    By standard elliptic theory, we have $\underline{\gamma}_{00}\in \Hb^{2+2a_0;k-1}(\Sigma)$,  $\underline{k}_{ij}\in\Hb^{3+2a_0;k-2}(\Sigma)$.
    Indeed, this follows from \cite[Corollary 2.6]{hintz_underdetermined_2023} applied to extensions of the the right hand sides of \cref{eq:EVE:data_construction} to $\R^3$, such that these extensions are orthogonal to a finite dimensional subspace.
    In particular, we require the extensions to satisfy $\int_{\R^3}\Delta \underline\gamma_{00}=0$ and $\int_{\R^3}\partial_i\underline{k}_{ij}=0$ for all $j$.
    
    Using \cref{lemma:EVE:weak_data_at_I0} with nontrivial initial data, we already have that $rh^{(1)}_{\mathcal{UU}}\in\Hbt^{a_0-,a_0-;k-8}(\D^+)+\Hb^{a_0-,0-;k-8}(\D^+)$ and also from \cref{prop:en:futureestimate} also $h^{(1)}|_{\chi=1}\in\Hb^{a_1;k-3}(\{\chi=1\})$.

    Next, by our observation from the beginning of the proof, we have that $h^{(1)}+h^{-}$ solves the linearised gauge constraints, i.e.~$D_\eta\Upsilon[h^{(1)}+h^{-}]=0$. 
    Therefore, by \cref{eq:EVE:constraints:puhLL},  we in fact have
    $\pu h^{(1)}_{LL}\in \Hbt^{2+a_0-,2+a_0-;k-9}(\D^+)+\Hb^{2+a_0-,2-;k-9}(\D^+)$.
    Integrating backwards from an outgoing cone $r h^{(1)}_{LL}|_{\outcone{}}\in\Hb^{a_0-}(\outcone{})$ with \cref{corr:ODE:du_dv} yields $rh^{(1)}_{LL}=rh^{(1,\mathrm{w})}_{LL}+rh^{(1,\mathrm{s})}_{LL}\in \Hbt^{a_0-,a_0-;k-9}(\D^+)+\Hb^{a_0-,1-;k-9}(\D^+)$, where the slow decay in the first $\Hbt$-space comes from the slow decay along~$\outcone{}$.
    We will next find a diffeomorphism to improve the first of these spaces to $\rho_+\Hbt^{a_0-,a_0-}(\D^+)$, thereby fixing the null geometry.
    
    \textbf{Step 2: Preliminary change of coordinates.}
    
    As in \cref{item:current:metric_coordiante_change}, we can find a change of coordinates function {$F\in\Hbt^{a_0-,a_0-;k-10}(\D^+)$} solving $\pv F=h^{(1,\mathrm{w})}_{LL}$:
    We consider the equation $\pu\pv F=\pu h^{(1,\mathrm{w})}_{LL} \in\Hbt^{a_0+2-,a_0+2-;k-10}(\D^+)$ and first integrate in $v$ with trivial data at $\scrip$ to write $\pu F(u,v,\omega)=\int_v^{\infty} \pu h^{(1,\mathrm{w})}_{LL}(u,v',\omega)\dd v'\in \Hbt^{a_0+1-,a_0+1-;k-10}(\D^+)$.
   Secondly, we integrate from a cone $\outcone{0}$ with initial data $ F|_{\outcone{0}}=\int_{v_0}^v h_{LL}^{(1,\mathrm{w})}\dd v'|_{\outcone{0}}\in \Hb^{a_0-}(\outcone{0})$. 
   The $F$ thus constructed indeed satisfies $F\in\Hbt^{a_0-,a_0-;k-10}(\D^+)$ and $\pv F=h^{(1,\mathrm{w})}_{LL}$.

    Let us define $\bar{\Phi}_{\mathrm{prelim}}:(u,v,\omega)\mapsto (\bar{u},\bar v,\bar\omega)=(u+F(u,v,\omega),v,\omega)$, which is a diffeomorphism from $\D^+$ to its image for $h$ sufficiently small.
    Let us write $\eta,\,\bar\eta_{\mathrm{prelim}}$ for the Minkowski metric with respect to the coordinates $x,\bar{x}=\bar\Phi_{\mathrm{prelim}}(x)$ respectively, i.e.~$\eta=-\dd t^2+\delta^{ij}\dd x_i\dd x_j$, $\bar{\eta}_{\mathrm{prelim}}=-\dd \bar{t}^2+\delta^{ij}\dd \bar{x}_i\dd \bar{x}_j$.
    By construction, we have $\bar\Phi_{\mathrm{prelim}}^*(\bar{\eta}_{\mathrm{prelim}})=\eta-h^{(1,\mathrm{w})}_{LL}\dd v\dd v+h^{(2)}$, where $rh^{(2)}_{\mathcal{UU}}\in\Hbt^{a_0-,a_0-;k-11}(\D^+)$ and, importantly $h^{(2)}_{LL}\in\Hbt^{2+a_0-,2+a_0-;k-11}(\D^+)$.
    We would like to use these coordinates in $\D^+$, but they are of course not harmonic even for $\eta$, we merely have
    $\Box_\eta (\bar\Phi_{\mathrm{prelim},\mu}(x))\in\Hbt^{2+a_0-,2+a_0-;k-12}(\D^+)$.

    \textbf{Step 3a: Making the coordinates harmonic w.r.t~$\eta$ in $\{\chi=0\}$.}
    
    We first perform a harmonic (with respect to $\eta$) change of coordinates in the region $\{\chi=0\}$. Later, we will correct these coordinates to coordinates that are harmonic with respect to $g$ in the support of $\chi$.
    
   For $\bar{x}_\mu=\bar{\Phi}_{\mathrm{prelim}}(x)_\mu$ as above, we define the scalar functions $\tilde{x}_{\mu}$, for $\mu=0,1,2,3$ to be the backwards scattering solutions to
    \begin{equation}\label{eq:EVE:proof_harmonic}
        \Box_{\eta}(\tilde{x}_\mu-\bar{x}_\mu)=-\Box_\eta\bar{x}_\mu,\qquad (\tilde{x}_\mu-\bar{x}_\mu)|_{\outcone{0}}=0,\quad {\pu(r(\tilde{x}_\mu-\bar{x}_\mu))|_{\scrip}=0.}
    \end{equation}
    From \cref{thm:scat:scat_general}, we have  $\tilde{x}_\mu-\bar{x}_\mu\in\rho_+\Hbt^{a_0-,a_0-;k-12}(\D^+)$, and hence $\tilde{u}-\bar{u}\in\rho_+\Hbt^{a_0-,a_0-;k-12}(\D^+)$.
    Let's define the diffeomorphisms $\tilde{\Phi}_{\mathrm{prelim}}:\bar{x}_\mu\mapsto \tilde{x}_\mu$ and $\Phi_{\mathrm{prelim}}=\tilde{\Phi}_{\mathrm{prelim}}\circ\bar{\Phi}_{\mathrm{prelim}}: x\mapsto\tilde{x}$. Let's also define as before $\tilde{\eta}_{\mathrm{prelim}}=-\dd \tilde{t}^2+\delta^{ij}\dd \tilde{x}_i\dd \tilde{x}_j$.

    Using that $a_0>-1/2$, along with the decay of $\tilde{u}-\bar{u}$, we have that $\Phi_{\mathrm{prelim}}^*(\tilde\eta_{\mathrm{prelim}})=\eta-h_{LL}^{(1,\mathrm{w})}\dd v\dd v+h^{(3)}$, where $rh^{(3)}_{\mathcal{UU}}\in\Hbt^{a_0-,a_0-;k-13}(\D^+)$ and $rh^{(3)}_{LL}\in\rho_+\Hbt^{a_0-,a_0-;k-13}(\D^+)$.
    
    In rectangular coordinates, we can write $\Phi_{\mathrm{prelim}}^*(\tilde\eta_{\mathrm{prelim}})_{\mu\nu}=\eta_{\mu\nu}+h^{\mathrm{g}}_{\mu\nu}+h^{(4)}_{\mu\nu}$  where $2h^{\mathrm{g}}_{\mu\nu}=\partial_\nu \Phi_{\mathrm{prelim},\mu}({x})+\partial_\mu \Phi_{\mathrm{prelim},\nu}({x})$ and $rh^{(4)}_{\mathcal{UU}}\in\Hbt^{3+2a_0-,3+2a_0-;k-13}(\D^+)$.
    Using that $\Box_\eta\Phi_{\mathrm{prelim},\mu}({x})=0$ for all $\mu\in\{0,1,2,3\}$, we have $\Box_\eta h^{\mathrm{g}}_{\mu\nu}=0$ for all $\nu,\mu\in\{0,1,2,3\}$ and $\Box_\eta \Phi^{*}(\tilde\eta)\in\Hbt^{6+2a_0-,6+2a_0-;k-15}(\D^+)$.
    
    Let us define $h^1_{\mathrm{prelim}}=h^{(1)}+(\Phi^*_{\mathrm{prelim}}(\tilde\eta_{\mathrm{prelim}})-\eta)$.
    Then we can summarise the above computations as
    \begin{subequations}\label{eq:EVE:ansatz}
        \begin{align}
            \begin{aligned}
                (rh^{1}_{\mathrm{prelim}})_{\mathcal{UU}}|_{\{\chi=0\}}&\in \Hbt^{a_0-,a_0-;k-15}(\D^+)+\Hb^{a_0-,0-;k-15}(\D^+),\\
                (rh^{1}_{\mathrm{prelim}})_{LL}|_{\{\chi=0\}}&\in \rho_+\Hbt^{a_0-,a_0-;k-15}(\D^+)+\Hb^{a_0-,1-;k-15}(\D^+);\label{eq:EVE:ansatz_decay}
            \end{aligned}\\
            (\Box_\eta h^{1}_{\mathrm{prelim}})_{\mathcal{UU}}|_{\{\chi=0\}}\in\Hbt^{6+2a_0-,6+2a_0-;k-15}(\D^+)\label{eq:EVE:ansatz_error}.
        \end{align}
    \end{subequations}

    \textbf{Step 3b: Making the coordinates harmonic w.r.t.~$\tilde{g}$ in the region $U$.}
    
    Next, we want to extend $\Phi_{\mathrm{prelim}}^{-1}$ to a map $\Phi^{-1}$ that is harmonic with respect to $\tilde{g}$ in the region region $\supp\tilde\chi$.
    We define $y_\mu$ coordinates via
    \begin{equation}
       \Box_{(1-\tilde{\chi})\eta+\tilde{\chi}(\tilde{\eta}+\tilde{h}^-)} y_\mu=0\qquad  (y_\mu -x_{\mu})|_{\{\chi=0\}}=0.
    \end{equation}
    We then finally define the diffeomorphism $\Phi^{-1}:\tilde{x}_\mu\mapsto y_{\mu}$, together with
    \begin{equation}
        h^1:=h^{(1)}+(\Phi^{*}(\tilde\eta)-\eta)+h^-.
    \end{equation}
    Clearly, in $U$, we have $\Phi^{-1}(\tilde{x})_\mu-{x}_\mu\in \Hb^{a_0;k-13}(U)$. Therefore, the metric $h_{\mathrm{lin}}:=h^{(1)}+h^-$ and the corresponding metric $\tilde{h}_{\mathrm{lin}}$ defined with respect to $\tilde{x}$-coordinates are close in the sense that $(\Phi^{\ast}\tilde{h}_{\mathrm{lin}}-h_{\mathrm{lin}})_{\mathcal{UU}}\in \Hb^{2a_0+2;k-13}(U)=r^{-1}\Hb^{a_1;k-13}(U)$.
    We thus deduce that $h^1$ as defined above satisfies
    \begin{equation}
       r h^1_{\mathcal{UU}}|_{\{\chi=1\}}=r(\Phi^{\ast}(\tilde{g})-\eta)_{\mathcal{UU}}+\Hb^{a_1-}(U).
    \end{equation}
    In particular, \cref{eq:EVE:ansatz_decay} still holds for $h_1$, and \cref{eq:EVE:ansatz_error} also still holds in view of \cref{eq:EVE:hlinequation}.
    Finally, since $h_{\mathrm{lin}}$ solves the linearised constraints $D_\eta \Upsilon[h_{\mathrm{lin}}]=0$, we have
    \begin{equation}\label{eq:EVE:linearisedconstraints:h1}
         D_\eta \Upsilon[h^1]=D_\eta \Upsilon[\Phi^{\ast}\tilde{\eta}]=D_\eta \Upsilon[h^{(4)}]\in\Hbt^{5+2a_0-,5+2a_0-;k-14}(\D^+)
    \end{equation}

Let us recap what we have done here: Both $\tilde{x}$ as well as $x$ are harmonic coordinates for $\tilde{g}$ (at least restricted to $\tilde{U}$). However, solving the leading order error term of \cref{eq:EVE:rEVE} with respect to the coordinates $\tilde{x}$ yields a non-short range metric perturbation, whereas, solving in the coordinates $x_\mu$, we get a short-range metric perturbation given by $h^1$, up to a faster-decaying error term.

    \textbf{Step 4: The equations satisfied by the error term $h^2$.}
    
    Next, we compute the equations satisfied by the error term
    \begin{equation} 
    h^2:=\Phi^{\ast}\tilde{g}-\eta-h^1.
    \end{equation}
    For this, we always only need to compute the nonlinear terms in, say, \cref{eq:EVE:eve_L,eq:EVE:eve_Lbar}, as $h^1$ already solves the equations to linear order up to $\Hbt^{(6+2a_0-,6+2a_0)}(\D^+)$ using \cref{eq:EVE:ansatz}.
    In particular, we compute that, away from the support of $h^-$, we have:\footnote{All the inhomogeneities here follow from the weights and order of the operators, except for $G_{\mathcal{UU}}$, where we used the extra property that $G_{\mathcal{UU}}$ has two differentiated terms yielding extra decay for the $\Hbt$ part of $h^1$ improving $\Hb^{5+3a_0-,3+3a_0-}(\D^+)$ to $\Hb^{5+3a_0-,3+a_0-}(\D^+)$}
    \begin{subequations}\label{eq:EVE:proof_future_h2}
        \begin{align}
        Q_{\mathcal{UU}}[h]&=Q_{\mathcal{UU}}[h^2]+V^{(1)}_\mathcal{UU}[h^2]+\Hb^{4+2a_0-,3+2a_0-;k-15}(\D^+)\label{eq:EVE:proof_future_Q}\\
        G_{\mathcal{UU}}[h]&=G_{\mathcal{UU}}[h^2]+Q^{(1)}_{\mathcal{UU}}[h^2]+V^{(2)}[h^2]+\Hb^{5+3a_0-,3+a_0-;k-15}(\D^+)\label{eq:EVE:proof_future_G}\\
        F_{\Lbar\Lbar}[h_{\mathcal{UT}}]&=F_{\Lbar\Lbar}[h^2_{\mathcal{UT}}]+V^{(3)}[h^2_{\mathcal{UT}}]+\Hb^{4+2a_0-,2-;k-15}(\D^+);\label{eq:EVE:proof_future_F}\\
        \bar{\Box}_{h}h&=\bar{\Box}_h h^2+V^{(4)}[h^2]+Q^{(2)}[h^2]+\Hb^{4+2a_0-,3+2a_0-;k-15}(\D^+),\label{eq:EVE:proof_future_box}
        \end{align}
    \end{subequations}
    where we introduced new linear ($V^{(i)}$) and semilinear ($Q^{(i)}$) perturbations with weights and nonlinear orders as below:
    \begin{nalign}
        a^{V^{(1)}}=(3+a_0,2+a_0),\quad a^{V^{(2)}}=(4+2a_0,2+2a_0),&\quad a^{V^{(3)}}=(3+a_0,1-), \quad a^{V^{(4)}}=(3+a_0,3/2+a_0)\\
        a^{Q^{(1)}}=(3+a_0,2+a_0),&\quad a^{Q^{(2)}}=(3+a_0,2),\qquad\qquad n^{Q^{(1)}}=n^{Q^{(2)}}=2.
    \end{nalign}

    Next, we compute that for $rh^2_\mathcal{UU}\in\Hb^{1+2a_0,-\delta}(\D^+)$ and $rh^2_{\Lbar\Lbar}\in\Hb^{1+2a_0,-2\delta}(\D^+)$ all these perturbations are short range for $\delta>0$ sufficiently small.
    For $V^{(1)},V^{(2)},V^{(4)},Q^{(1)},Q_{\mathcal{UU}},G_{\mathcal{UU}}$, this follows directly from the weights. 
   On the other hand, for $V^{(3)},F_{\Lbar\Lbar}$, we exploit that they only depend on $h^2_{\mathcal{UT}}$.
    Finally, we observe that the inhomogeneities (the error terms) in \cref{eq:EVE:proof_future_h2} are also consistent with the weights.

    Finally, we compute the perturbations to \cref{eq:EVE:eve_Ltransport}. In view of \cref{eq:EVE:linearisedconstraints:h1}, we again only need to study the nonlinear terms: We get
    \begin{nalign}
        G_{\mathrm{t}}[h]=G_{\mathrm{t}}[h^2]+V^{(5)}[h^2]+\Hbt^{3+2a_0-,2+2a_0-;k-14}(\D^+),
    \end{nalign}
    where $V^{(5)}$ is a linear perturbation with weight $a^{V^{(5)}}=(2+a_0,1+a_0)$.
\end{proof}

We finish the paper by proving that we can indeed solve for $h^2$ in \cref{eq:EVE:h2} in $\D^+$.
Following 
\cite{Lindblad2010}, this is a standard result.
We prove it here for completeness.
Now, the reason why \cref{eq:EVE:h2}  does not fall under the construction in \cref{thm:scat:scat_general} is that a naive estimate using the wave equation for $h^2_{LL}$ would produce an estimate $rh^2_{LL}\in\Hb^{a_1,0-}(\D^+)$, which in turn yields a non-short range metric perturbation in \cref{eq:EVE:g_original}.
This is remedied by the transport equation in~\cref{eq:EVE:h2}.
The proof therefore consists essentially of a standard energy estimate at top order, together with a linear transport estimate at lower orders.

\begin{lemma}[High energy estimate]\label{lemma:EVE:future_high_energy}
    Let $h^1$ be as in \cref{eq:EVE:future_h1_size}.
    Let $a_1>0$ and $\delta>0$ be sufficiently small.
    Assume that 
    \begin{equation}
        \norm{rh^2_{\mathcal{UU}}}_{\Hb^{a_1-\delta,-2\delta;5}(\D^+)}+\norm{rh^2_{LL}}_{\Hb^{a_1-\delta,a_1-2\delta;5}(\D^+)}\leq \epsilon_1\ll1.
    \end{equation}
    Then for $c_0,c_+\in\R$ and $c_0^\delta=c_0-\delta,c_+^\delta=c_+-2\delta$, we have for any $k\in\mathbb{N}$:
    \begin{nalign}
        \norm{\big(\Vb^k((g[h]-\eta)^{\mu\nu})\nabla_\nu\nabla_\mu\big)h^2_{LL}}_{\Hb^{c_0,c_+;0}(\D^+)}&\lesssim \epsilon_1\norm{h^2_{LL}}_{\Hb^{c_0^\delta-(3+a_1),c_+^\delta-(1+a_1);k}(\D^+)}+\epsilon_1\norm{h^2_{\mathcal{UU}}}_{\Hb^{c_0^\delta-(3+a_1),c_+^\delta-(5/2+2a_0);k}(\D^+)},\\
        \norm{\big(\Vb^k((g[h]-\eta)^{\mu\nu})\nabla_\nu\nabla_\mu\big)h^2_{\mathcal{UU}}}_{\Hb^{c_0,c_+;0}(\D^+)}&\lesssim \epsilon_1\norm{h^2_{LL}}_{\Hb^{c_0^\delta-(3+a_1),c_+^\delta-1;k}(\D^+)}+\epsilon_1\norm{h^2_{\mathcal{UU}}}_{\Hb^{c_0^\delta-(3+a_1),c_+^\delta-(3/2);k}(\D^+)}.\\
    \end{nalign}
\end{lemma}
\begin{proof}
    This is a direct consequence of Sobolev embedding and the form of $g$ in \cref{eq:EVE:g_original} and the estimates on $h^1$ and~$h^2$.
\end{proof}

\begin{lemma}\label{lemma:EVE:future_nonlinear}
    Fix $k= 35$, $-u_0$ sufficiently large, let $h$ be a sufficiently small solution to \cref{eq:EVE:EVE_schematic} in $\D^+\cap\D^-$ with $rh_{\mathcal{UU}}\in\Hb^{a_0;k}(\D^+\cap\D^-)$, and let $h^1$ be as in \cref{prop:EVE:ansatz}.
    In particular, let us use the coordinates given by $\Phi$ in \cref{prop:EVE:ansatz}.
    There exists $\epsilon$ small enough such that if the error terms in \cref{eq:EVE:h2_size,eq:EVE:h2} have size $\leq\epsilon$ and
    \begin{equation}
        \norm{rh^1_{\mathcal{UU}}}_{\Hb^{a_0-\delta/10,a_0-\delta/10;4}(\D^+)}+\norm{rh^1_{LL}}_{\Hb^{a_0-\delta/10,1+a_0-\delta/10;4}(\D^+)}\leq \epsilon
    \end{equation}
    (both of these smallness assumptions hold if $h$ is sufficiently small in $\D^+\cap \D^-$ by \cref{prop:EVE:ansatz}),
    then the system in \cref{eq:EVE:h2} has a solution in $\D\cap\{u<u_0-10\}$.
    
    Furthermore, higher regularity is propagated: More precisely, the estimates \cref{eq:EVE:proof_bootstrap} hold for any $l$ provided \cref{eq:EVE:h2_size,eq:EVE:h2} hold for $k\geq l+20$.
\end{lemma}
\begin{proof}
    Fix $a_1=1+2a_0>0$, $\delta\in(0,a_1/10)$ and let us consider the hypersurface $\Sigma=\{u=u_0+v^{-\delta}\}$.
    Under the bootstrap assumptions below, \cref{eq:EVE:proof_bootstrap}, $\Sigma$ is a spacelike hypersurface for $g[h]$.
    We will study the solution $h$ in $\D^+_T=\D^+\cap\{t<T\}\cap\{u\leq u_0+v^{-\delta}\}$ for $u_0$ sufficiently small, which in turn has two future directed boundaries on which we can ignore boundary terms coming from energy estimates.

    Fix $k-20 =l\geq15$ and $\delta$ as in \cref{prop:EVE:ansatz}.
    We claim as a bootstrap assumption the following bounds
    \begin{subequations}\label{eq:EVE:proof_bootstrap}
        \begin{align}
        \norm{rh^2_{\mathcal{UT}}}_{\Hb^{a_1-\delta,-\delta;l}(\D^+_T)}+\norm{rh^2_{\Lbar\Lbar}}_{\Hb^{a_1-\delta,-2\delta;l}(\D^+_T)}\label{eq:EVE:proof_bootstrap_low}\qquad\qquad\qquad\qquad\qquad\qquad\\
        +\norm{rh^2_{\mathcal{UU}}}_{\Hb^{a_1-\delta,-4\delta;l+5}(\D^+_T)}+\norm{rh^2_{LL}}_{\Hb^{a_1-\delta,-2\delta;l+5}(\D^+_T)} \leq \epsilon_1 \label{eq:EVE:proof_bootstrap_high},\\
        \norm{rh^2_{LL}}_{\Hb^{a_1-\delta,a_1-\delta;l+2}(\D^+_T)}\leq \epsilon_1^{4/5} ,\label{eq:EVE:proof_bootstrap_transport}
        \end{align}
    \end{subequations}
    where $\epsilon\ll \epsilon_1\ll 1$.
    We will show that, assuming \cref{eq:EVE:proof_bootstrap} holds in $\D^+_T$, we can replace $\epsilon_1\mapsto \epsilon_1/2$.
    Using local existence for quasilinear wave equations in turn implies existence for $h^2$ in $\D^+_{\infty}$ via a continuity argument.

    We begin with the terms in \cref{eq:EVE:proof_bootstrap_low}.
    Using \cref{eq:EVE:proof_bootstrap_transport,eq:EVE:proof_bootstrap_high}, we know that up to regularity $k$, $\bar{\Box}_h$ is a short-range metric perturbation.
    Using the short range nature of the right hand side of \cref{eq:EVE:h2} together with the smallness of the coefficients of the linear terms given by $-u_0\gg 1$ and the energy estimate \cref{prop:current:tame_estimate}, we already have 
    \begin{equation}
        \norm{rh^2_{\mathcal{UT}}}_{\Hb^{a_1-\delta,-\delta;l}(\D^+_T)}+\norm{rh^2_{\Lbar \Lbar}}_{\Hb^{a_1-\delta,-2\delta;l}(\D^+_T)}\leq C(\epsilon_1^{9/5}+\epsilon)\leq \epsilon_1/2
    \end{equation}
    for some fixed constant $C$ depending only on the order of regularity.
    
    For \cref{eq:EVE:proof_bootstrap_transport}, we use the transport equation in \cref{eq:EVE:h2} together with \cref{corr:ODE:du_dv} to obtain
    \begin{equation}
        \norm{rh^2_{LL}}_{\Hb^{a_1-\delta,a_1-\delta;l+2}(\D^+_T)}\leq C (\epsilon_1+\epsilon)\leq (\epsilon_1/2)^{4/5}.
    \end{equation}

    Finally, for \cref{eq:EVE:proof_bootstrap_high}, we can commute $\bar{\Box}_h h^2_{\mathcal{UU}}$ to obtain that $\bar{\Box}_h \Vb^{l+5} h^2_{\mathcal{UU}}$ is given by the schematic form 
    \begin{equation}
        (\Vb^{l-5}(g[h]^{\mu\nu})) \nabla_\mu\nabla_\nu \Vb^{l+4} h^2_{\mathcal{UU}}  +(\Vb^{l+5}(g[h]^{\mu\nu})) \nabla_\mu\nabla_\nu \Vb^{10} h^2_{\mathcal{UU}} +\Vb^{l+5}\mathfrak{F}_{\mathcal{UU}}^2[h^2]+\Hb^{4+2a_0,3+3a_0;k-l-20}(\D^+),
    \end{equation}
    where the first term represents a sum of terms with at most $l-5$ derivatives acting on $g[h]$ and at most $l+4$ derivatives acting on $h^2$, and similarly for the other ones.
    
    Using the a priori control \cref{eq:EVE:proof_bootstrap_low,eq:EVE:proof_bootstrap_transport} we get that the first term is a linear short-range perturbation for $\Vb^{l+4} h^2$.
    Similarly, the $\mathfrak{F}^2$ term is a combination of short range and long range perturbations for $\Vb^{l+4} h^2$, with coefficients depending only on $\Vb^{l-5} h^2$.
    Both of these can be bounded using \cref{prop:current:tame_estimate}.
    For the second term, we use \cref{lemma:EVE:future_high_energy} to bound the corresponding error by $C\epsilon_1^{9/5}$.
    The existence follows.

    Higher order regularity is propagated by repeating the previous estimates without having to close a nonlinear inequality with respect to the size of top order norms.
\end{proof}

\begin{lemma}[Polyhomogeneity]
    Provided that $rh_{\mathcal{UU}}\in\A{phg}^{\E}(U)$, then $\Phi,h^1$ from \cref{prop:EVE:ansatz} are polyhomogeneous on $\D^+$ and so is $h^2$ from \cref{lemma:EVE:future_nonlinear}.
\end{lemma}
\begin{proof}
    This follows the same iteration approach as in \cref{thm:app:general}, once a solution is shown to exist and satisfy $\rho_+^{-a_1}\rho_0^{-a_0} rh_{LL},\rho_+^{-a_0}\rho_0^{-a_0} r h_\mathcal{UU}\in\Hb^{0,0;\infty}(\D^+)$, where $a_1>0$ and $a_0>-1/2$.
\end{proof}

\subsection{Scattering solutions in \texorpdfstring{$\D$}{D}}

In this section, we summarise the results of \cref{thm:EVE:scatteringpast,prop:EVE:ansatz,lemma:EVE:future_nonlinear} in all of $\D$ to make more precise the asymptotics towards $\scrip$ in our construction.
We also include an improved conformal regularity statement in the case of compactly supported incoming radiation.

\begin{cor}\label{cor:EVE:final_form}
    Let $k,h_{\mathcal{UU}}|_{\D^{v_0+1}_{v_0}},\{a^{\mathcal{UU}}_-\},a_-,a_0,\mathfrak{h}_{\mathcal{UU}}$ be as in \cref{thm:EVE:scatteringpast}.

    a) Then there exists
    \begin{itemize}
        \item $\eta^-+h^-$, a solution of \cref{eq:EVE:EVE,eq:EVE:rEVE} with $\Box_{\eta^-+h^-}x_\mu=0$, satisfying \cref{eq:EVE:solution_past} in $\D^-$;
        \item $\eta^++h^+$, a solution of \cref{eq:EVE:EVE,eq:EVE:rEVE} with $\Box_{\eta^++h^+}x_\mu=0$, where $h^+$ is a sum of terms satisfying \cref{eq:EVE:future_h1_size,eq:EVE:proof_bootstrap} in $\D^+$;
        \item a diffeomorphism $\Phi_\mu:\D^+\to\R^{3+1}$ such that  $\Phi_\mu(x)-x_\mu\in\Hb^{a_0-;k-13}(\D^+\cap\D^-)$ and $\eta^++h^+=\Phi^* (\eta^-+h^-)$.
    \end{itemize}

    b) Furthermore, for compactly supported incoming radiation $\supp \mathfrak{h}_{\mathcal{UU}}\subset\scrim_{v_0+1,v_0+2}$ (see~\cref{rem:appB:compactsupport}), it holds that $rh_{\mathcal{UU}}\in \Hbt^{0-}(\D^-\setminus\D_{v_0}^{v_0+10})$ and
    \begin{equation}\label{eq:EVE:final_form+}
        rh^+_{\mathcal{UT}}\in \Hb^{0-;k-25-2k'}(\scrip)+\Hb^{0-,1-;k-25-2k'}(\D^+),\quad rh^+_{\Lbar\Lbar}\in \Hb^{0-,0-;k-25-2k'}(\D^+),\quad rh^+_{LL}\in\Hb^{0-,1-;k-20}(\D^+),
    \end{equation}
    for some $k'<\infty$ as in \cref{rem:prop:finite_loss}, provided that $k-2k'\gg 1$.
\end{cor}
\begin{proof}

    \textit{a)} Follows immediately from the previous results.

  \textit{b)} Let us take $a_-=a_0=0$.
    In the case of compactly supported incoming radiation, we note that we have $rh^-|_{\D^{v_0+11}_{v_0+10}}\in \A{phg}^{\overline{(0,1)};k-4}(\D^{v_0+11}_{v_0+10})+\Hbt^{1-;k-4}(\D^{v_0+11}_{v_0+10})$ and $rh^-|_{\D^-\setminus\D_{v_0}^{v_0+10}}\in\tilde{H}^{0-,0-;k-2}(\D^-\setminus\D_{v_0}^{v_0+10})$ from b)-c) of \cref{thm:EVE:main_scattering}.

   Note, that in all the above, we did not work in a single global gauge, instead we i) constructed a solution in $\D^-$, ii) changed coordinates in $\D^+\cap\D^-$, iii) constructed a solution in $\D^+$. In this way, we obtained two different metrics $\eta^{\pm}+h^{\pm}$ in the same coordinate system.
   We can alternatively view the two metrics $\eta^\pm+h^\pm$ as a single metric on a manifold with two coordinate charts.
   Below we extend the gauge from $\D^+$ to the region $\D^-$ to obtain a single metric in a single global harmonic coordinate system.
   
    Consider $\mathfrak{D}$ as a manifold with two coordinate charts 
    $\Phi^\pm:\D^\pm\to\mathfrak{D}$ and a transition function $\Phi^+\circ(\Phi^-)^{-1}=\Phi$.
    Let us write $x^\pm:=(\Phi^\pm)^{-1}$ for coordinate functions on $\mathfrak{D}$ defined on open sets.
    We define the metric $g=\Phi^\pm_*(\eta^\pm+h^\pm)$ on $\mathfrak{D}$ coming from $\eta^\pm+h^\pm$, which is well defined due to $\Phi^*(\eta^-+h^-)=\eta^++h^+$.

    Let us write $y_\mu=x^+_\mu-x^-_\mu\in\Hb^{a_0;k-15}(\Phi^+(\D^+)\cap\Phi^-(\D^-))$ and note that $\Box_{g}y_\mu=0$, since both coordinate systems are harmonic in this region.
    We extend $y_\mu$ to $\D^-$ keeping it harmonic via \cref{lemma:EVE:weak_data_at_I0} to obtain that in the $x^-$ coordinate system $y_\mu\in \Hbt^{0-,0-;k-19}(\D^-)+\Hb^{1-,0-;k-21}(\D^-)$.
    Finally, we extend the coordinate system $x^+=(\tilde{\Phi}^+)^{-1}:\mathfrak{D}\to\R^{3+1}$ to $\D^-$ via $x^+=x^-+y$, and define $h^+:=(\tilde{\Phi}^+)^*\Phi^-_*(\eta^-+h^-)$ in $\D^-$ via the pullback of $h^-$.
    Using that $\frac{\partial x^+_\mu}{\partial x^-_\mu}=\delta^\mu_\nu+\Hb^{1-,1-;k-22}(\D^-)$, it follows that 
    $\partial_{x^-} (rh^+)\in \Hb^{1-,1-;k-23}(\D^-\setminus\D_{v_0}^{v_0+10})$.
    
    We stress that the metric $\eta^+ +h^+$ is no longer asymptotically flat in the sense that $h^+_{\Lbar\Lbar}$ possibly will not decay faster than $1/r$ towards $\scrim_{\mathfrak{D}}=\Phi^-(\scrim)$ and $v^+:=x^+_0-\sqrt{x_i^+x_j^+\delta^{ij}}$ may not be bounded towards $\scrim_{\mathfrak{D}}$.
   
    Let us now write $h^+=h^{+,1}+h^{+,2}$ where the rectangular coordinates of $h^{+,1}$ satisfy scalar wave equations
    \begin{equation}
        \bar{\Box}_{h^+} h^{+,1}_{\mu\nu}=0,\qquad h^{+,1}|_{\Cbar}=h^+|_{\Cbar}, \, \partial_vrh^{+,1}|_{\scrim}=0.
    \end{equation}
    By writing $\bar{\Box}_{h^+}$ with respect to $x^\pm$ in $\Phi^\pm(\D^\pm)$, or with respect to some glued coordinate $x:\mathfrak{D}\to\D$ satisfying  $x=x^+\chi+(1-\chi)x^-$,
    we see that it is a short-range admissible linear perturbation of the scalar wave equation with gap $\delta=1-$.
    We claim that $rh^{+,1}_{\mu\nu}\in\A{b,phg}^{0-,\overline{(0,1)};k-23-k'}(\D^+)+\Hb^{0-,1-;k-23-k'}(\D^+)$ in $\D^-$ for some $k'$ as in \red{\cref{eq:prop:finite_reg}}.
    Indeed, using that the gap is $\delta=1-$, the result follows from \cref{thm:app:general}.
    Combining with \cref{cor:app:minkowski_with_error} it follows that $rh^{+,1}\in\A{b,phg}^{0-,\overline{(0,0)};k-23-k'}(\D^+)+\Hb^{0-,1-;k-23-k'}(\D^+)$.

    Next, we write \cref{eq:EVE:rEVE} for the rectangular components, $h^{+,2}_{\mu\nu}$, in the $x^+$ coordinates 
    \begin{equation}
        \bar{\Box}_{h^+} h^{+,2}_{\mathcal{\mu\nu}}=\bar{\Box}_{h^+} h^+_{\mu\nu}\in \Hb^{4-,4-,2-;k-25}(\D),\qquad h^{+,1}|_{\Cbar^+}=0,
    \end{equation}
    where the improvement of weights towards $\scrim$ follows from the no incoming radiation condition, i.e.~$\partial_{x^-} rh^+\in \Hb^{1-,1-;k-23}(\D^-\setminus \D^{v_0+10}_{v_0})$.
    Using \cref{eq:EVE:eve_L,eq:EVE:eve_Lbar} and \cref{thm:app:general}, we may upgrade this to 
    \begin{equation}
        rh^{+,2}_{\mathcal{UT}}\in \A{b,phg}^{1-,\overline{(0,0)};k-25-k'}(\D^+)+\Hb^{1-,1-;k-25-k'}(\D^+),\quad rh^{+,2}_{LL}\in \A{b,phg}^{1-,\overline{(0,1)};k-25-2k'}(\D^+)+\Hb^{1-,1-;k-25-2k'}(\D^+).
    \end{equation}
    The result follows.
\end{proof}

We may put finally put out metric into a double null gauge in a small neighbourhood of null infinity as follows:
\begin{lemma}\label{lem:EVE:doublenullconversion}
    Let $h$ be a metric perturbation as in \cref{eq:EVE:final_form+} with $k=\infty$.
    Then, we can introduce coordinates $\Phi^{\mathrm{dn}}(u,v,\omega)\mapsto (v^1, u^1, \omega^1_A)$ such that locally around a point $p\in \D^+\cap\scrip$ the metric takes the double null form
    \begin{equation}\label{eq:EVE:finaldoublenullform}
        (\eta+h)^{-1}=-4\Omega^2\partial_{
        {u}^1}\otimes_s\partial_{v^1}+(\Phi^{\mathrm{dn}}_{\ast}(r))^{-1}b^A\partial_{{v}^1}\otimes_s\partial_{\omega_A^1}+(\Phi^{\mathrm{dn}}_{\ast}(r))^{-2}\slashed{g}^{AB}\partial_{\omega_A^1}\otimes_{s}\partial_{\omega_B^1},
    \end{equation}
    where $\Phi_{\ast}^{\mathrm{dn}}(r)$ denotes the push-forward of $r$ along $\Phi^{\mathrm{dn}}$; we will suppress writing this out in the sequel and simply write~$r$. 
  Furthermore, for $X=H^{;\infty}(\scrip_{u_0-1,u_0})/r+\Hb^{2-;\infty}(\D_{u_0}^{u_0-1})$ it holds componentwise that
    \begin{equation}
       -4 \Omega^2=-1+X,\quad b^A\in X,\quad \slashed{g}^{AB}=(\slashed{g}_{S^2}^{-1})^{AB}+X.
    \end{equation}
\end{lemma}
\begin{proof}
    We solve 3 nonlinear transport equations to remove the following terms from the inverse metric: $\partial_u^2$, $\partial_v^2$ and $\partial_u\otimes\partial_{\omega_A}$.

    Let us work around a point $p\in\scrip$ and use local coordinates $(u,v,\omega)$.
    We write $X=H^{;\infty}(\scrip_{u_0-1,u_0})/r+\Hb^{2-;\infty}(\D_{u_0}^{u_0-1})$ and $Y=\Hb^{1-;\infty}(\D_{u_0}^{u_0-1})$. In an abuse of notation, we also write $X$ and $Y$ for elements in $X$ and $Y$, respectively,  and observe that $X* X\subset X* Y\subset Y* Y\subset X$.
    Recall the matrix inversion formula for $A,B\in \R^{4\times 4}$ with $A$ invertible:
    \begin{equation}
        (A+\epsilon B)^{-1}=A^{-1}-\epsilon A^{-1}BA^{-1}+\mathcal{O}(\epsilon^2).
    \end{equation}
    Using this, we invert \cref{eq:EVE:final_form+} in the frame $(\dd v,\dd u,r \dd\omega)$ and the corresponding dual frame to obtain
    \begin{equation}
        \eta+h=\begin{pmatrix}
            r^{-1}Y & -2+X & X\\
            -2+X & Y & X\\
            X & X & \slashed{g}_{S^2}+X
        \end{pmatrix}
        \implies (\eta+h)^{-1}=\begin{pmatrix}
            Y & -1/2+X & X\\
            -1/2+X & r^{-1}Y & X\\
            X & X & \slashed{g}^{-1}_{S^2}+X
        \end{pmatrix}.
    \end{equation}
    Rescaling the angular frame by $r$, 
    this implies in the $(\dd v,\dd u,\dd\omega)$ and coresponding dual frames
    \begin{equation}
        \eta+h=\begin{pmatrix}
            r^{-1}Y & -2+X & rX\\
            -2+X & Y & rX\\
            rX & rX & r^2\slashed{g}_{S^2}+r^2X
        \end{pmatrix}
        \implies (\eta+h)^{-1}=\begin{pmatrix}
            Y & -1/2+X & r^{-1}X\\
            -1/2+X & r^{-1}Y & r^{-1}X\\
            r^{-1}X & r^{-1}X & r^{-2}\slashed{g}^{-1}_{S^2}+r^{-2}X
        \end{pmatrix},
    \end{equation}
    We now find  $u^1=u+f$ for a function $f$ such that $u^1$ is null and $f|_{\scrip}=0$. We are thus led to solve:
    \begin{multline}
        0=(\eta+h)^{-1}(\dd u+\dd f,\dd u+\dd f)=\boxed{r^{-1}Y}+r^{-1}Y(\partial_u f)^2\boxed{-\partial_vf}+X\partial_vf+Y(\partial_v f)^2+(-1+X)\partial_vf\partial_uf\\
        +r^{-2}\partial_A f\partial_Bf+r^{-1}X\partial_uf\partial_A f+r^{-1}X\partial_vf\partial_A f.
    \end{multline}
    The highlighted terms are the leading-order parts of the equation.
    Differentiating the equation, we see that derivatives of $f$ are transported along a vector-field that is to leading order $\partial_v$.
    It follows with the same proof as in \cref{lemma:EVE:eikonal} that $f\in Y$.
    Therefore, in the $(v,u^1,\omega)$ coordinate system, we have the following form for the inverse metric
    \begin{equation}
        (\eta+h)^{-1}=_{(v,u_1,\omega)}\begin{pmatrix}
            Y & -1/2+X & r^{-1}X\\
            -1/2+X & 0 & r^{-1}X\\
            r^{-1}X & r^{-1}X & r^{-2}\slashed{g}^{-1}_{S^2}+r^{-2}X
        \end{pmatrix}
    \end{equation}
    Next, we find  $\omega_A^1=\omega_A+f_A$ for some functions $f_A$ such that $f_A|_{\scrip}=0$ and
    \begin{equation}
        0=(\eta^1+h^1)^{-1}(\dd u^1,\dd \omega_A+\dd f_A)=(-1/2+X)\partial_v f_A+r^{-1}X \partial_B f_A +r^{-1}X.
    \end{equation}
    Now $f_A$ solves a linear transport equation, and we get that $f_A\in X$. Thus, in the coordinates $(v,u^1,\omega^1)$, the metric takes the form
    \begin{equation}
        (\eta+h)^{-1}=_{(v,u^1,\omega)}\begin{pmatrix}
            Y & -1/2+X & r^{-1}X\\
            -1/2+X & 0 & 0\\
            r^{-1}X & 0 & r^{-2}\slashed{g}^{-1}_{S^2}+r^{-2}X
        \end{pmatrix}
    \end{equation}
    
    Let us construct another eikonal function $v^1=v+f$ for a new function $f$ that solves the transport equation 
    \begin{equation}
        0=(\eta+h)^{-1}(\dd v+\dd f,\dd v+\dd f)=\boxed{Y}+(\partial_vf)^2+\boxed{(-1+X)\partial_uf}+Xr^{-1}\partial_uf\partial_A f+r^{-2}\partial_Af\partial_Bf
    \end{equation}
    where we highlighted the leading order terms.
    Since $f$ solves a nonlinear transport with leading order term $\partial_u$, we initialise it as $f|_{\{u^1=\tau\}}=0$ for some $\tau$.
   Again, arguing as in \cref{lemma:EVE:eikonal}, it follows that $f\in Y$.
    And so in the coordinate system $(v^1,u^1,\omega)$ the inverse metric takes the form 
    \begin{equation}
        (\eta+h)^{-1}=\begin{pmatrix}
            0 & -1/2+X & r^{-1}X\\
            -1/2+X & 0 & 0\\
            r^{-1}X & 0 & r^{-2}\slashed{g}_{S^2}^{-1}+r^{-2}X
        \end{pmatrix}
    \end{equation}
    This concludes the proof.
    \end{proof}
\begin{rem}[Relation to peeling]\label{rem:EVE:peeling}
Combining part \textit{b)} of \cref{cor:EVE:final_form} with \cref{lem:EVE:doublenullconversion} above, we conclude that smooth, compactly supported data along $\scrim$ (in the sense of \cref{rem:appB:compactsupport}) lead to a solution to the Einstein vacuum equations that, when recast in a double null gauge, takes the form \cref{eq:EVE:finaldoublenullform}. 
The reader familiar with the double null gauge (see e.g.~\cite{dafermos_non-linear_2021}[Section~1.2]) may readily infer from this that the curvature quantity $\alpha$ (whose linearised analogue $\al$ featured in \cref{sec:Sch:ling}) satisfies $\alpha\in\Hb^{4-;\infty}(\D_{u_0}^{u_0-1})$ component-wise in an orthonormal frame.
This is almost sharp in view of the expansion $r^5\al\sim M a(\omega)r$ (cf.~\cref{eq:Schw:thmlingrav:al}) that we prove for the linearised system in this setting.

It is somewhat surprising that we can, with relatively little effort, prove an almost-sharp upper bound for $\alpha$ in this case. Notice that, in contrast, doing actual asymptotic computations for $\alpha$ would be significantly more challenging. Similarly, proving almost-sharp upper bounds for $\alpha$ for the setting of the specific data constructed in \cref{lemma:EVE:converging_to_Sch} (which we expect to be $r^{-5+}$), since we would now need to track all the non-conformally smooth terms of $h$ towards $\scrip$ with much more care.
\end{rem}
\begin{rem}[Further restricting the gauge]
    We note that one can further change the gauge (while keeping a double null form) such that, for instance, the trace of $\slashed{g}-r^2\slashed{g}_{S^2}$ is in $\Hb^{0-;\infty}(\D^{u_0-1}_{u_0})$.
    Similarly, the $1/r$-term in the $\Omega^2$ can also be fixed.
    This is related to Bondi-normalising the gauge at $\scrip$.
\end{rem}

	\printbibliography
 
\end{document}